\documentclass[openany, 11pt]{book}
\usepackage{amssymb} 
\usepackage{amsthm} 
\usepackage{amsmath} 
\usepackage{float} 
\usepackage{fancyhdr} 
\usepackage[english]{babel}
\usepackage[utf8]{inputenc}
\usepackage[T1]{fontenc}
\usepackage[margin=1.0in]{geometry} 
\usepackage[pdftex]{graphicx,color} 
\usepackage{wasysym}
\newtheorem{definition}{Definition}[section]
\newtheorem{theorem}{Theorem}[section]
\newtheorem{corollary}{Corollary}[theorem]

\newtheorem{remark}{Remark}
\usepackage[osf,sc]{mathpazo}
\usepackage{textcomp}
\usepackage{microtype}
\usepackage{graphicx, color}
\usepackage{amsfonts}
\usepackage{booktabs}
\usepackage{nicefrac}
\usepackage{pifont}
\usepackage{tikz}
\usepackage{graphicx}
\usepackage{float}
\usepackage{wrapfig}
\usepackage[all]{xy}
\usepackage{appendix}

\usepackage[center]{titlesec}
\usepackage{lipsum}



\begin{document}

\frontmatter
\begin{titlepage}
\title{Euler and the Gamma Function} 
\author{Alexander Aycock}
\maketitle
\end{titlepage}

\setcounter{tocdepth}{4}

\tableofcontents

\setcounter{secnumdepth}{4}

\newpage
\listoffigures

\mainmatter
 \lhead[\thepage]{}
\chead[Main Thesis]{Euler and the $\Gamma$-Function}
\rhead[]{\thepage}
\chapter{Euler and the $\Gamma$-Function}

This is the main body of the work. We present and discuss Euler's results on the $\Gamma$-function and will explain how Euler obtained them and how Euler's ideas anticipate more modern approaches and theories.

\section{Introduction} \label{sec:Introduction}

\subsection{Motivation} \label{sec: Motivation}

According to Leibniz, there are two arts in mathematics: The "Ars Inveniendi" (Art of Finding) and the "Ars Demonstrandi" (Art of Proving). Nowadays, the latter is dominating the mathematical education, whereas the first is often neglected.\\
Leonhard Euler's mathematical works are special not only for their quality and quantity, but also for his pedagogical style rendering them easily understandable. This is summarized by two famous quotes, the first due to Laplace: "Read Euler, read Euler, he is the Master of us all"\footnote{W. Dunham even wrote a book with a title alluding to this quote \cite{Du99}.}, the other due to Gau\ss: "The study of Euler's works cannot be replaced by anything else."\\[2mm] 
Moreover, Euler does not only provide the proofs of certain theorems, but also tells us how he found the theorem in the first place. In other words, Euler's work gives us both, the Ars Demonstrandi\footnote{Being written in the 18th century, Euler's works do not meet the modern standards of mathematical rigor, but in most cases it is not a lot of work to formulate Euler's proofs rigorously.} and the Ars Inveniendi\footnote{His famous book on the calculus of variations \cite{E65}, the first ever written on the subject, bearing the title "Methodus inveniendi lineas curvas maximi minimive proprietate gaudentes, sive solutio problematis isoperimetrici lattissimo sensu accepti", even contains Methodus inveniendi in the title.}.\\[2mm]
In this article we want to review and discuss some of the properties of the $\Gamma$-function, defined as\footnote{Most books devoted to the $\Gamma$-function start from the integral representation. We mention \cite{Ar15}, \cite{Fr06} as an example. \cite{Ni05} is an exception. It starts the investigation from the difference equation satisfied by $\Gamma(x)$, i.e. $\Gamma(x+1)=x\Gamma(x)$.}

\begin{equation} \label{eq:1}
\Gamma(x):= \int\limits_{0}^{\infty} t^{x-1}e^{-t}dt \quad \text{for} \quad \operatorname{Re}(x) > 0,
\end{equation} 
that were already discovered by Euler himself. More precisely, we will explain how Euler, the discoverer of \eqref{eq:1} \cite{E19}\footnote{To be completely precise, in \cite{E19} Euler introduced the integral $\int\limits_{0}^{1} \left(\log \frac{1}{t}\right)^{x-1}dt$ which transforms  into the above representation by the substitution $t=e^{-u}$.}, derived several different expressions for the $\Gamma$-function which are usually attributed to others.

\subsection{Euler's Idea concerning the $\Gamma$-function}
\label{subsec: Euler's Idea concerning the Gamma-function}

Euler found his expressions for the $\Gamma$-function basically from two different sources.\\[2mm]
1. Interpolation theory\\
2. The functional equation $\Gamma(x+1)=x\cdot \Gamma(x)$.\\[2mm]
But eventually they all boil down to the solution of the functional equation by different methods\footnote{This is the more natural approach, since the functional equation is one of the characteristic properties of the factorial.}. The first approach, outlined in \cite{E212} (Chapter 16 and 17 of the second part) and \cite{E613}, is based on difference calculus and led him to the Weiersta\ss{} product expansion of $\frac{1}{\Gamma(x)}$ and also to the Taylor series expansion of $\log \Gamma(1+x)$. Indeed, Euler implicitly uses the functional equation even in this approach. But we want to separate it from the other approaches, in which he tried to solve the functional equation explicitly and said so.\\[2mm]
Concerning the direct solution of the functional equation, first, Euler solved the functional equation satisfied by $\Gamma(x)$,

\begin{equation*}
\Gamma(x+1) = x \Gamma(x),
\end{equation*}
or, equivalently, the difference equation

\begin{equation*}
\log \Gamma(x+1) - \log \Gamma(x) = \log x
\end{equation*}
 by applying the Euler-Maclaurin formula, which he had  discovered in \cite{E25} and proved in \cite{E47}. We will discuss this below in section in section \ref{sec: Solution of log Gamma(x+1)- log Gamma (x)= log x via the Euler-Maclaurin Formula}. Later, he tried a solution\footnote{We say "tried", since the solution Euler found in this way is incorrect. He even realized this. But he tried to argue it away by another incorrect argument. This will be discussed in more detail in the corresponding section, i.e. in section \ref{subsubsec: An Application - Derivation of the Stirling Formula for the Factorial}.} by conversion of the difference equation into a differential equation of infinite order with constant coefficients, see \cite{E189}, which he then wanted to solve by the methods he had developed in \cite{E62}, \cite{E188}. Both ideas led him to the Stirling-formula for the factorial, i.e.,

\begin{equation*}
x! = \Gamma(x+1) = \dfrac{x^x}{e^x}\sqrt{2 \pi x} \quad \text{for} \quad x \rightarrow \infty.
\end{equation*}
The sign $=$ is to understood as follows here: We have

\begin{equation*}
    \lim_{x \rightarrow \infty} \dfrac{\Gamma(x+1)}{\frac{x^x}{e^x}\sqrt{2 \pi x}}=1.
\end{equation*}
In section \ref{sec:  Solution of the Equation log Gamma(x+1)- log  Gamma (x)= log x by Conversion into a Differential Equation of infinite Order} we will also explain that this formula follows from an asymptotic series.\\
In \cite{E123}, he explained a method how to solve the functional equation by an educated guess. Euler applied it to the factorial in $\S 13$ of \cite{E594}, basically adding some more examples to those of \cite{E123}; we will see in  section \ref{sec: Euler's direct Solution of the Equation Gamma (x+1)=xGamma(x) - The Moment-Ansatz} that his ideas lead to the integral representation  (\ref{eq:1}).\\
Finally, in \cite{E652}, he uses the functional equation to derive a product representation he first stated without proof in \cite{E19}.

\subsection{Organisation of the Paper}
\label{subsec: Organisation of the Paper}

\subsubsection{General Overview}
\label{subsubsec: General Overview}
This paper is organised as follows:\\
It can  be subdivided into three parts, referred to as part I., part II. and part III., respectively.\\[2mm]
Part I. will be a brief historical overview on the Gamma function, explicitly stressing Euler's contributions in section \ref{sec: Historical Overview over the Gamma-function}.\\[2mm]
Part II. contains the modern introduction of the $\Gamma$-function and the classification theorems in section \ref{sec: Short modern Introduction to the Gamma-Function}. Furthermore, we dedicate a whole section to the rigorous solution of the difference equation $F(x+1)=xF(x)$ in section \ref{sec: Solution of the difference equation F(x+1)=xF(x)}.\\[2mm]
Finally, part III. is  devoted to Euler's several approaches he used or could have used to arrive at the $\Gamma$-function (most of the time, he intended something entirely different and the $\Gamma$-function just was a special case). We will discuss Euler's idea to solve general homogeneous difference equations with linear coefficients by an idea we will refer to as "moment ansatz" in section \ref{sec: Euler's direct Solution of the Equation Gamma (x+1)=xGamma(x) - The Moment-Ansatz}, and we will discuss how he solved the general difference equation by converting it into a differential equation of infinite order in section \ref{sec: Solution of the Equation log Gamma(x+1)- log  Gamma (x)= log x by Conversion into a Differential Equation of infinite Order}. Another idea of Euler was to derive solutions of the general difference equation by difference calculus, which we will also discuss in detail in section \ref{sec: Interpolation Theory and Difference Calculus}. We also devoted a complete section to the relation among the $\Gamma$- and $B$-function and how Euler found those connections in section \ref{sec: Relation between Gamma and B}. After this we will conclude and try to summarize Euler's vast output in section \ref{sec: Summary}.\\[2mm]
Given the time in which Euler wrote his papers, some of his arguments are not completely rigorous in the modern sense and some are even incorrect. Therefore, when necessary or appropriate, we will show, how his ideas can be formulated in a modern setting and at some points also give the  rigorous proof. Sometimes, we will not give all the details, since this would simply take too long and would carry us too far away from our actual intention.\\
We will always try to put Euler's results and ideas in contrast to these modern ideas and it will turn out that Euler actually anticipated a lot that came after him (if understood in the modern context). Thus, we also added some sections just containing some historical notes. Furthermore, since this is mainly a paper on Euler's works, we also included some quotes from his papers, translated from his  original Latin into English. They will help to understand better Euler's way of thinking and his persona.

\subsubsection{Notation}
\label{subsubsec: Notation}

Euler invented many of the modern notation, e.g., $\sum$ for a sum, $f:x$ for a function of $x$ etc. Nevertheless, most of the times he did not use the compact notation, but wrote things out explicitly, e.g., for $\sum_{k=1}^{x}\frac{1}{k}$ he wrote $1+\frac{1}{2}+\cdots +\frac{1}{x}$. When referring to Euler's papers, we will retain his notation as closely as possible and only resort to a modern compact notation, if things become  clearer that way. Euler also never used the symbol $\Gamma(x)$ - This notation was introduced by Legendre in \cite{Le26a} p. 365 - to denote the factorial nor did he write $x!$, his notation varies from paper to paper. We will always use the modern notation.\\
Furthermore, the notion of limits as understood today did not exist at that time. Euler often speaks of infinitely small or infinitely large numbers. In this case, we will use the modern symbol $\lim_{n\rightarrow 0}$ and $\lim_{n \rightarrow \infty}$, respectively. At some occasions we also included scans of Euler's writings (all available at the Euler Archive) such that the reader can compare Euler's notation to the modern one.\\
Finally, a reader going through the whole text will inevitably encounter some  repetitions, which could not be avoided in the attempt to present each part independently from the previous ones.\\
Finally,  we do not always give a complete rigorous proof, but just explain the general idea how to prove a certain theorem. For, otherwise this would carry us too far away from our actual intention, presenting Euler's contributions to the theory of the $\Gamma$-function.

\newpage

\section{Historical Overview of the $\Gamma$-function} 
\label{sec: Historical Overview over the Gamma-function}

It will be convenient to give a short overview of the history of the $\Gamma$-function from its first appearance in Euler's 1738 paper \cite{E19} to the axiomatic introduction in Bourbaki more than 200 years later \cite{Bo51}. While doing this, we will mainly follow the history as described by Davis in his 1959 article on the $\Gamma$-function \cite{Da59} and Sandifer's articles in his book \cite{Sa07}. But \cite{Du91} and \cite{Gr03} are also recommended. The first goes into more detail on the investigations on the factorials before the $\Gamma$-function was actually introduced. The second discusses the correspondence between Euler, Goldbach and Bernoulli on the matter. We will talk about this after the general overview in section \ref{subsec: More detailed Discussion on the Discovery of the Gamma-function}.\\[2mm]
We will especially stress Euler's contributions and add some information, if necessary.

\subsection{Brief historical Overview}
\label{subsec: Brief historical Overview}

In his 1959 article \cite{Da59}, Davis described the history of the $\Gamma$-function from its discovery by Euler in 1729 \cite{Eu29} in his correspondence with Goldbach\footnote{The first ever definition is due to Daniel Bernoulli, who in a letter written to Goldbach  on the 6th of October 1729 \cite{Be29} defines the factorial as:

\begin{equation*}
    n! =\left(A+\frac{x}{2}\right)^{x-1}\left(\frac{2}{1+x} \cdot \frac{3}{2+x}\cdot \frac{4}{3+x}\cdots \frac{A}{A-1+x}\right)
\end{equation*}
for an infinitely large number $A$. Euler's definition, also stated in \cite{E19}, appeared  also in a letter to Goldbach \cite{Eu29}, written one week later than Bernoulli's. However, Bernoulli's definition did not catch on.} to the axiomatic introduction in Bourbaki in 1951 \cite{Bo51}. His article can be divided into the following phases:\\[2mm]
1. Discovery by Euler in 1729 (Davis seems to be unaware of Bernoulli's priority on the first definition. We will discuss this in more detail in the next section in section \ref{subsec: More detailed Discussion on the Discovery of the Gamma-function}.)\\
2. Gau\ss's investigations in 1812 on the hypergeometric series and the factorials in 1812 \cite{Ga28}.\\
3. Development of the notion of an analytic function by Riemann and Weierstra\ss{} around 1850. \\
4. Application of analytic function theory to the $\Gamma$-function. \\
5. Classification as a transcendentally transcendent function by Hölder in 1887. \\
6. Uniqueness theorems, especially the Bohr-Mollerup theorem. We will also prove his theorem later in section \ref{subsubsec: Bohr-Mollerup Theorem}.\\
7. Transcendence questions. \\[2mm]

\subsubsection{1. Phase: Discovery by Euler}
\label{subsubsec: 1. Phase: Discovery by Euler} 

As already mentioned, the $\Gamma$-function appeared in print the first time in 1737 in \cite{E19}. In that paper, Euler gave two expressions for the $\Gamma$-function. On the one hand, he found

\begin{equation*}
    \left[\left(\frac{2}{1}\right)^n\dfrac{1}{n+1}\right] \cdot \left[\left(\frac{3}{2}\right)^n\dfrac{2}{n+2}\right] \cdot
    \left[\left(\frac{4}{3}\right)^n\dfrac{3}{n+3}\right] \cdots = n!.
\end{equation*}
This is equation (1) in \cite{Da59}. In \cite{E19}, Euler does not give a proof for the formula, but just shows its correctness for $n=1,2,3$. A proof had to wait until his 1755 book \cite{E212}, a completely elaborated version then also appeared in \cite{E613}. This proof will be discussed  in section \ref{subsubsec: Gammafunction}. Therefore, we will not discuss this any further here. But we want to mention that it is more illustrative to write the product as follows:

    \begin{equation*}
    x! = \Gamma(x+1) = 1^{1-x} \cdot \dfrac{1^{1-x}2^x}{x+1} \cdot \dfrac{2^{1-x}3^x}{x+2} \cdot \dfrac{3^{1-x}4^x}{x+3} \cdot \text{etc.} = \prod_{k=1}^{\infty} \dfrac{k^{1-x}(k+1)^x}{x+k}.
\end{equation*}
This is also how Euler gave it already in \cite{E19} and in his letter to Goldbach \cite{Eu29}, for reasons of convergence, which he then elaborated on in \cite{E613}.\\[2mm]
In \cite{E19}, Euler also found the famous integral representation:

\begin{equation*}
    n! = \int\limits_{0}^{1} (-\log x)^ndx \quad \text{for} \quad \operatorname{Re}(x)>0.
\end{equation*}
He arrived at it starting from the integral:

\begin{equation*}
    \int\limits_{0}^{1}x^e(1-x)^ndx = \dfrac{1 \cdot 2 \cdots n}{(e+1)(e+2)\cdots (e+n+1)}
\end{equation*} 
by some clever substitutions and application of L'Hospital's rule. We will discuss his ideas in more detail below in section \ref{sec: Relation between Gamma and B}. Davis \cite{Da59} and Sandifer \cite{Sa07} also give Euler's original proof, i.e. also sticking to his notation. We will present it, when discussing the $B$-function in section \ref{subsubsec: Using the B-function}.

\subsubsection{2. Phase: Gau\ss's Investigations}
\label{subsubsec: 2. Phase: Gauss's Investigations}

In 1812 Gau\ss{} wrote the influential paper \cite{Ga28}, mainly on the hypergeometric series, but the last few paragraphs contain a discussion of the factorial. Gau\ss{}, in contrast to Euler, did not set the task for himself to find a formula for the factorial, but he simply started his  investigations from the product:

\begin{equation*}
    n! = \lim_{m \rightarrow \infty} \dfrac{m!(m+1)^n}{(n+1)(n+2)\cdots (n+m)},
\end{equation*}
written in Davis's notation; it is equation (3) in \cite{Da59}. Although immediate from the above formula, Euler did not represent the $\Gamma$-function in this form in \cite{E19}, but arrived at this formula by a different reasoning in \cite{E652}. We will see this below in section \ref{subsubsec: Using the B-function}, when discussing the connection of the $\Gamma$-function to the $B$-function, i.e. the integral

\begin{equation*}
    B(x,y)= \int\limits_{0}^{1}t^{x-1}(1-t)^{y-1}dt \quad \text{for} \quad  \operatorname{Re}(x), \operatorname{Re}(y)>0.
\end{equation*}
Anyway, Gau\ss{} realized that the above product formula is easier to work with. First, since it is defined for all non-negative integer numbers, whereas the integral is valid only for $\operatorname{Re}(y)>0$. Second, since it allows to prove some special identities of the $\Gamma$-function more easily than by using the integral representation. We mention the reflection formula:

\begin{equation*}
    \Gamma(x)\Gamma(1-x)= \dfrac{\pi}{\sin (\pi x)}
\end{equation*}
and the multiplication formula

\begin{equation*}
\Gamma \left(\frac{x}{n}\right)\Gamma \left(\frac{x+1}{n}\right) \Gamma \left(\frac{x+2}{n}\right) \cdots \Gamma \left(\frac{x+n-1}{n}\right)=  n^{1-x} \Gamma(x) \sqrt{\dfrac{(2\pi)^{n-1}}{n}},
\end{equation*}
as it is presented in modern form. Gau\ss{} proved this formula in \cite{Ga28} using the product formula representation, which hence also often has its name attached to it. But in his 1772 paper \cite{E421}, Euler already heuristically arrived at a formula that can be shown to be equivalent to the Gau\ss{} multiplication formula. This will also be discussed below in section \ref{subsec: Multiplication Formula}.

\subsubsection{3. Phase: Development of complex Analysis}
\label{subsubsec: 3. Phase: Development of complex Analysis}

We do not want to go into much detail here. We just mention  that in this phase functions of a complex variable became the focus of the studies and replaced the older notion of an analytical expression in the sense Euler introduced in his 1748 book \cite{E101}.\\[2mm]
One of the most important results in complex analysis, due to the works of Riemann \cite{Ri51} and Weiersta\ss{} \cite{We78}, is the concept of analytic continuation, i.e. the extension of a domain of the analytic function. \\
Below we will see that the function defined via the integral of the $\Gamma$-function is an analytic function, if negative integer numbers are excluded, as is the function defined from the product representation. But on first sight, those expressions do not seem to have to do anything with one another at all. But using the functional equation of the $\Gamma$-function, the integral can be extended to the same domain as the product formula and thus the product formula can be seen to be just an analytic continuation of the integral representation. One might argue, that this is what Euler did in \cite{E652} to obtain the product formula. This will also be discussed in more detail below in section \ref{subsec: Euler on Weierstrass's Condition}.\\
Furthermore, general results from complex analysis allow simple proofs of otherwise very difficult to prove identities.\\[2mm]

\subsubsection{4. Phase: Formulas following from complex Analysis}
\label{subsubsec: 4. Phase: Formulas following from complex Analysis}

As already indicated, by extending the domain of a function from the real to complex numbers, many interesting identities emerge and moreover they are proved more easily. As an example, let us again mention the reflection formula

\begin{equation*}
    \Gamma(z)\Gamma(1-z)= \dfrac{\pi}{\sin (\pi z)}.
\end{equation*}
Euler gave a proof in his 1772 paper \cite{E421}. We will present his proof in section \ref{subsubsec: Euler's Proof of the Reflection Formula} and another one, he could have given, below in section \ref{subsubsec: Proof Euler could have given}. Euler's proof is based on the product formula and the product for the sine, which he discovered in \cite{E41}\footnote{This is the paper, in which he solved the Basel problem, i.e. summed the series $1+\frac{1}{1^2}+\frac{1}{3^2}+\cdots =\frac{\pi^2}{6}$.} and proved in \cite{E61} for the first time. Using complex analysis, this identity is easily proved applying Liouville's theorem. \\
Another, more prominent  example, is the functional equation of the Riemann $\zeta$-function:

\begin{equation*}
    \zeta(z) = \zeta(1-z)\Gamma(1-z)2^z \pi^{z-1}\sin \frac{1}{2}\pi z
\end{equation*}
with

\begin{equation*}
    \zeta(z)= 1+ \dfrac{1}{2^z}+\dfrac{1}{3^z}+\text{etc.}
\end{equation*}
This formula is attributed to Riemann who proved it in his famous 1859 paper \cite{Ri59} using complex analysis. But Euler, using his definition of the sum of a divergent series which he gave in \cite{E212}, arrived at an equivalent identity in his 1768 paper \cite{E352}. An equivalent formula also appears in the 1773 paper \cite{E432}. Euler's arguments are heuristic, as he also admitted. The story of Euler and the $\zeta$-function is an interesting subject for itself, but we will not discuss it here. The reader is referred to \cite{Ha48}, who devoted a section to this discussion, and Ayoub's 1974 article \cite{Ay74}, also reprinted in \cite{Du07}. Let us mention that the functional equation for the $\zeta$-function can easily be deduced from the results of Malmsten's 1846 paper \cite{Ma46}, confer, e.g., \cite{Ay13}.\\[2mm]
Returning to formulas for the $\Gamma$-function, let us mention the following formula found by Newman in 1848.
\begin{equation*}
    \dfrac{1}{\Gamma(z)}= ze^{\gamma z}\left[\left(1+z\right)e^{-z}\right]\left[\left(1+\dfrac{z}{2}\right)e^{-\frac{z}{2}}\right]\cdots, \quad \text{where} \quad \gamma = 0.5772156649\cdots
\end{equation*}
$\gamma$ is the Euler-Mascheroni constant. Confer the appendix, section \ref{subsec: gamma meets Gamma - Euler on the Euler-Mascheroni Constant}, for more information on this constant.\\
This formula is often referred to as Weierstra\ss{} product representation of the $\Gamma$-function, since it is a special case Weierstra\ss{}'s factorisation theorem proved in 1878 in \cite{We78}. But this formula was  found even earlier by Schlömilch in 1844 \cite{Sc44}. We will review Weierstra\ss's ideas in section \ref{subsec: Modern Idea - Weierstrass product} and show that it was in some sense anticipated by Euler in \cite{E613} in section \ref{subsubsec: Comparison to Euler's Idea}. Thus, it comes as no surprise that Newman's or Schlömilch's formula follows already from Euler's.\\
Finally, in this context we mention the recent paper \cite{Pe20}, in which another definition for the $\Gamma$-function is derived using complex analysis. This definition originates by generalising Laplace's  definition of the $\Gamma$-function given in \cite{La82}, i.e.

\begin{equation*}
    \dfrac{1}{\Gamma(s)}= \dfrac{1}{2\pi}\int\limits_{-\infty}^{\infty}\dfrac{e^{x+iy}}{(x+iy)^s} \quad \text{for} \quad \operatorname{Re}(s)>0,~x>0
\end{equation*}
to Hankel's contour integral definition \cite{Ha64}, i.e.

\begin{equation*}
    \dfrac{1}{\Gamma(s)} =\dfrac{1}{2\pi i} \int_{\eta} z^{-s}e^zdz.
\end{equation*}
Here, $\eta$ is a curve from $-\infty$ to $+\infty$ surrounding the negative real axis $\mathbb{R}_{-}= ]-\infty, 0]$ and where for $z \in \mathbb{C}\setminus \mathbb{R}_{-}$ in $z^{-s}= e^{-s \log z}$ the principal branch of the logarithm is taken. This integral converges for all $s \in \mathbb{C}$ and thus defines an entire functions. Both formulas, Laplace's and Hankel's, also appear in \cite{Ni05}, but there Hankel's formula is attributed to Weierstraß, who actually discovered  it later than Hankel.

\subsubsection{5. Phase: Classification as transcendentally transcendental Function}
\label{subsubsec: 5. Phase: Classification as transcendentally transcendental Function}

As the title of \cite{E19} already suggests, the $\Gamma$-function is not an algebraic function. An algebraic function, already at the time of \cite{E19}, is a function that can be obtained by finitely many algebraic operations. Algebraic operations, according to Euler, were addition, subtraction, multiplication, division, raising to a natural power and taking root of integer powers.\\
The simplest, or rather most familiar, non-algebraic or transcendental functions are exponentials, logarithms, sines and cosines. They cannot be constructed from finitely many algebraic operations but have to defined by a limit procedure somehow. For example,

\begin{equation*}
    e^z: = \lim_{n \rightarrow \infty} \left(1+\dfrac{z}{n}\right)^n.
\end{equation*}
But all of the above functions share the property that they satisfy at least an algebraic differential equation.\\
In 1887 \cite{Hoe87}, Hölder showed that the $\Gamma$-function does not even satisfy an algebraic differential equation. Thus, the $\Gamma$-function is not only not an algebraic function, but  has an even higher degree of transcendence than, e.g., $\sin(x)$, $\log x$ etc. The $\Gamma$-function is a so-called transcendentally transcendent function.

\subsubsection{6. Phase: Towards the axiomatic Introduction}
\label{subsubsec: 6. Phase: Towards the axiomatic Introduction}

Up to this point we already mentioned several different expressions for the $\Gamma$-function, and initially there is no reason for them to be identical for all values. Euler claimed that all the different expressions he found while solving the interpolation problem of the factorial are identical. At least, it seems that he never addressed the issue, confer, e.g., \cite{E368}.\\
On the other hand, he was well aware that, given  points of a function only at the natural numbers, this does not determine the function. On the contrary, there are still infinitely many solutions. He stated this explicitly in his 1753 paper \cite{E189}. We will come back to some of the results of this paper below, when we will discuss the solution of difference equations by converting them into solvable differential equations of infinite order.\\[2mm]
The non-uniqueness of the interpolation theorem for the $\Gamma$-function raises the question, what makes it unique among all other solutions for  the same problem. In other words: Which properties determine the $\Gamma$-function uniquely? The way to answer this question is an example for the Methodus inveniendi, whence we want to describe a possible line of thought to get to the characteristics of the function.\\[2mm]
As mentioned above, the sole condition

\begin{equation*}
    f(n)=(n-1)! \quad \forall ~ n \in \mathbb{N}
\end{equation*}
is not enough to conclude $f(x)=\Gamma(x)$ for all $x \in \mathbb{C}$. For,

\begin{equation*}
    f(x)= \Gamma(x)p(x),
\end{equation*}
with $p(x)$ being a periodic function with period $1$ and $p(1)=1$, is also a solution to the same question. Moreover, Hadamard in 1894 \cite{Ha94} found the following nice solution:

\begin{equation*}
    f(x) = \dfrac{1}{\Gamma(1-x)} \dfrac{d}{dx} \log \left[\dfrac{\Gamma\left(\frac{1-x}{2}\right)}{\Gamma\left(1-\frac{x}{2}\right)}\right].
\end{equation*}
One can check that

\begin{equation*}
    f(x+1)=xf(x) +\dfrac{1}{\Gamma(1-x)} \quad \text{and} \quad f(1)=1.
\end{equation*}
This also $  f(n)=(n-1)! \quad \forall ~ n \in \mathbb{N}$.\\[2mm]
Those two examples show that we need additional conditions. First, it seems necessary to demand $f(1)=1$ and $f(x+1)=xf(x)$ for all $x \in \mathbb{C}$. This excludes Hadamard's solution, as it does not satisfy the recurrence relation. Moreover, one has to exclude solutions just being a product of the $\Gamma$-function by a function with period $1$. One possible way to do this, was found by Weierstra\ss{}, who in his 1856 paper \cite{We56} introduced the $\Gamma-$function as the unique function satisfying: $f(1)=1$ and $f(x+1)=xf(x)$ and

\begin{equation*}
    \lim_{n \rightarrow \infty} \dfrac{f(x+n)}{(n-1)!n^x}=1 \quad \forall x \in \mathbb{C}.
\end{equation*}
In his paper, he proves that the function satisfying those conditions is unique and mentioned that he was using Gau\ss's 1812 paper \cite{Ga28} and the product representation for the $\Gamma$-function as motivation for this definition. Although this formulation solves the uniqueness question, as we will also show below in section \ref{sec: Solution of the difference equation F(x+1)=xF(x)}, mathematicians kept looking for other, more "natural" properties, as Davis calls it.\\
In other words, we want to replace the third of Weierstra\ss's conditions by another one, which is somehow more intuitive. Trying so, one immediately notices the $\Gamma$-function to be convex for $x>0$.  But this is still not enough, as the following example, equation (35) in \cite{Da59}, shows. Define

\begin{equation*}
    f(x)= \left\lbrace
    \renewcommand{\arraystretch}{2,0}
\setlength{\arraycolsep}{0.0mm}
    \begin{array}{lll}
         \dfrac{1}{x} & \quad \text{for} \quad 0 < x \leq 1, \\
          1 & \quad \text{for} \quad 1 \leq x \leq 2, \\
          (x-1) & \quad \text{for} \quad 2 \leq x \leq 3, \\
           (x-1)(x-2) & \quad \text{for} \quad 3 \leq x \leq 4, \\
           etc.
    \end{array} \right\rbrace.
\end{equation*}
This function is convex and has the two other two properties.\\[2mm]
The correct formulation had to wait until 1922 \cite{Mo22}, when the Bohr-Mollerup theorem was formulated. We will meet it again below several times, but let us already formulate it here:

\textit{The $\Gamma$-function is the only function defined for $x>0$ which is positive, satisfies $\Gamma(1)=1$ and the functional equation $x\Gamma(x)=\Gamma(x+1)$ and is logarithmically convex.}\\[2mm]
In \cite{Bo51}, the $\Gamma$-function is introduced as the function satisfying all the above properties. The uniqueness is then proved afterwards.

\subsubsection{7. Phase: Transcendence Questions}
\label{subsubsec: 7. Phase: Transcendence Questions}

Having established the higher transcendence of the $\Gamma$-function, it is only natural to ask, whether the values of the $\Gamma-$function at rational values are always transcendental and if so, if  they can be expressed in terms of already known constants. We will meet the Euler-Mascheroni constant $\gamma$, defined as

\begin{equation*}
    \gamma := \lim_{n \rightarrow \infty} \left[\sum_{k=1}^{n}\dfrac{1}{k} -\log(n)\right]
\end{equation*}
again  in section \ref{subsubsec: Finding the constant K}, and discuss Euler's contribution to the constant in the appendix in section \ref{subsec: gamma meets Gamma - Euler on the Euler-Mascheroni Constant}. It was discovered by Euler in his 1740 paper \cite{E43} and appeared in Newman's infinite product formula for $\Gamma(x)$. To this day, it is not known whether $\gamma$ is rational or not. \\
Concerning values of the $\Gamma$-function itself, we know that $\Gamma(n+1)=n!$ and thus a natural number. From the reflection formula, we conclude

\begin{equation*}
    \Gamma \left(\frac{1}{2}\right)=\sqrt{\pi}
\end{equation*}
and thus transcendental. Euler already found this value in \cite{E19} from his product formula which reduces to the Wallis product formula for $\pi$. And this encouraged him to look for an integral representation for $\Gamma(x)$. More on this below in section \ref{subsubsec: Euler's Thought Process}.\\
Using results on periods as defined in \cite{Ko01} and the relation of the $\Gamma-$function to the $B$-function, which we will also study below and mention Euler's contributions to this, one can conclude that $\Gamma\left(\frac{1}{4}\right)\sqrt[4]{\pi}$ is a transcendental number, whereas it is not known whether only $\Gamma\left(\frac{1}{5}\right)$ is rational or not. But in 1975 Chudnovsky \cite{Ch84} showed that $\Gamma \left(\frac{1}{6}\right)$, $\Gamma \left(\frac{1}{4}\right)$, $\Gamma \left(\frac{1}{3}\right)$, $\Gamma \left(\frac{3}{4}\right)$, $\Gamma \left(\frac{5}{6}\right)$ are transcendental and algebraically independent from $\pi$.

\subsection{More detailed Discussion on the Discovery of the $\Gamma$-function} 
\label{subsec: More detailed Discussion on the Discovery of the Gamma-function}

First, we want to stress here again that the first representation for the factorial was given by Daniel Bernoulli, not by Euler. Confer \cite{Dut91}, \cite{Gr03}. In \cite{Be29} Bernoulli gave the formula:

\begin{equation*}
    x! =\left(A+\frac{x}{2}\right)^{x-1}\left(\frac{2}{1+x} \cdot \frac{3}{2+x}\cdot \frac{4}{3+x}\cdots \frac{A}{A-1+x}\right)
\end{equation*}
for an infinitely large number $A$. Euler's wrote the letter containing his definition to Goldbach one week later. Euler's definition is just the first that appeared in print.\\[2mm]
Euler and Daniel Bernoulli were friends since their childhood and communicated about math via letters and even both had positions in the  Imperial Russian Academy of Sciences in Saint Petersburg at 1729 and met daily, probably discussing various things, including the problem of  interpolating the factorial.  Thus, it is natural to ask, whether and, if so, how they influenced each other concerning the interpolation problem of the factorial. This is rather difficult to answer, since they are not explicit about this in their respective letter to Goldbach. But, for the sake of completeness, we mention, that in \cite{Ju65}, a book containing the letters exchanged between Euler and Goldbach from 1729-1764, in the footnote on page 21 it is stated  that Euler knew the contents of the letters from Bernoulli to Goldbach and thus also Bernoulli's formula.\\[2mm]
Bernoulli  just stated  the formula at the end of \cite{Be29} without any explanation.

\begin{center}
\begin{figure}
\centering
   \includegraphics[scale=0.9]{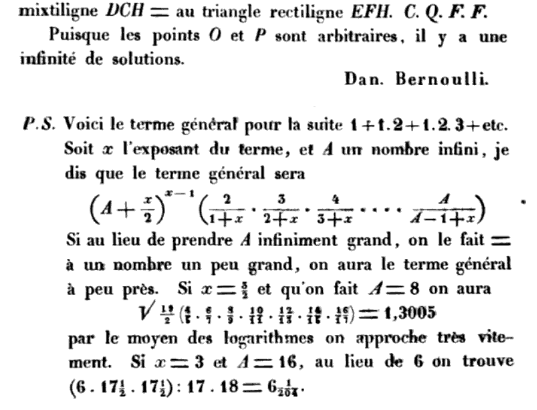}
    \caption{Letter from Bernoulli to Goldbach}
    Part of Bernoulli's letter \cite{Be29} written 6th October 1729 to Goldbach. It contains the first ever definition of the $\Gamma$-function.
    \end{figure}
\end{center}
Euler opened his letter with the formula 

\begin{equation*}
    x! = \dfrac{1}{1+x} \left(\dfrac{2}{1}\right)^x \cdot  \dfrac{2}{2+x} \left(\dfrac{3}{2}\right)^x \cdot  \dfrac{3}{3+x} \left(\dfrac{4}{3}\right)^x \cdot  \dfrac{4}{4+x} \left(\dfrac{5}{4}\right)^x \cdots
\end{equation*}
but dis also not explain how he arrived at this formula. In fact, a proof had to wait until his 1755 book \cite{E212}.

\begin{center}
\begin{figure}
\centering
    \includegraphics[scale=0.9]{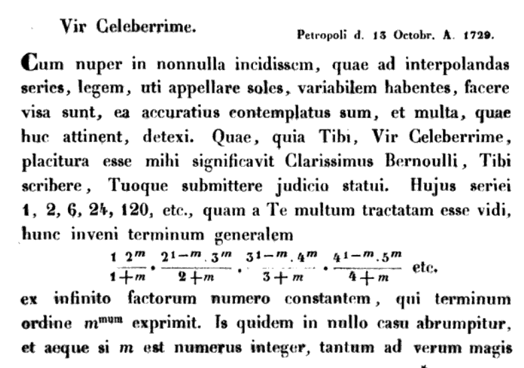}
      \caption{Letter from Euler to Goldbach - first Part}
       First page of Euler's letter \cite{Eu29} written to Goldbach 13th October 1729 - one week letter than Bernoulli's. The first page already contains the product representation. 
    \end{figure}
\end{center}

In the same letter, he also said that:

\begin{equation*}
    m! = \dfrac{1 \cdot 2 \cdot 3 \cdots n}{(1+m)(2+m)(3+m)\cdots (n+m)}(n+1)^{m}
\end{equation*}
for infinite $m$.

\begin{center}
\begin{figure}
\centering
    \includegraphics[scale=0.9]{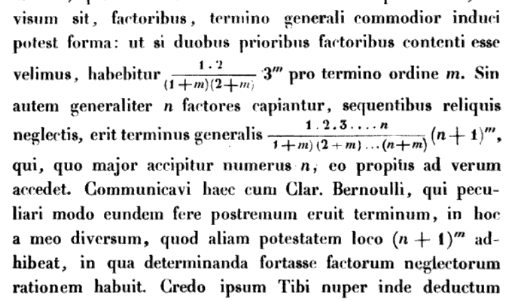}
      \caption{Letter from Euler to Goldbach - second Part}
    Part of the second page of Euler's letter \cite{Eu29} written to Goldbach 13th October 1729. It already contains what would become known  as the Gau\ss's product representation of the factorial.
    \end{figure}
\end{center}
He proved this later in his 1769 paper \cite{E368} and again in his 1793 paper \cite{E652}. Nevertheless, as mentioned, this expression is attributed to Gau\ss, who used it as a definition for the factorial in his influential 1812 paper \cite{Ga28}.\\
In his letter, concerning his and Bernoulli's expressions, Euler wrote, translated from the Latin original:\\[2mm]
\textit{I communicated these discoveries to Bernoulli, who in a peculiar way found almost the same last term\footnote{Euler means the factor $\left(A+\frac{x}{2}\right)^{x-1}$ in Bernoulli's formula, in contrast to the factor $(n+1)^m$ in his formula.}, different from mine in that regard that he uses another power instead of $(n+1)^{m}$, determining which he maybe took into account the neglected factors.}\\[2mm]
This sentence reveals two things: First, Euler knew about Bernoulli's formula, most probably from private communication. For, the subject was never mentioned in letters they exchanged. Second, he did not know how Bernoulli arrived at his result.\\
It is speculation, whether Euler found his result influenced by Bernoulli's formula and, not knowing how the latter arrived at his formula, found his expression or whether he had  found it independently. The letter seems to indicate the first option. Dutka in his overview article \cite{Dut91} argues in favour of this possibility, stating the opinion that Bernoulli obtained his formula first and talked to Euler about it and then Euler, after this, found his solution. This is also endorsed by \cite{Gr03}. But one might argue that Euler's second formula is more natural, since it can be easily found from the recursive property of the factorial and taking the limit. This is also how Euler then proved it in \cite{E652}, as we will see below in section \ref{subsec: Euler on Weierstrass's Condition}. Unfortunately, the question how Bernoulli actually found his formula, remains unsettled. But, having already stated the Bohr-Mollerup theorem above, it is easily checked that also Bernoulli's expression is another expression for the $\Gamma$-function; $\Gamma(x+1)$, to be precise. Applying the Bohr-Mollerup theorem, we also see that we get another expression for the $\Gamma$-function:

\begin{equation*}
    \Gamma(x+1)  = \lim_{n \rightarrow \infty} A(x,n)^{x-1} \left(\frac{2}{1+x} \cdot \frac{3}{2+x}\cdot \frac{4}{3+x}\cdots \frac{n}{n-1+x}\right)
\end{equation*}
if just $A(x,n)$ is logarithmically convex for $x >1$ and for $n \in \mathbb{N}$ and 

\begin{equation*}
    \lim_{n\rightarrow \infty} A(n,0)\cdot n =1.
\end{equation*}
And Bernoulli's choice $A(n,x)=\left(n+\frac{x}{2}\right)$ obviously satisfies those requirements. But, e.g., also $\left(n+\frac{x^k}{2}\right)^{x-1}$ for a natural number $k$ works as well.\\
Bernoulli probably preferred his choice, since it converges rapidly to its limit. Thus, the particular choice for $A(n,x)$ is probably a result of extensive calculations trying out several similar formulas.

\subsection{Overview on Euler's Contributions}
\label{subsec: Overview on Euler's Contributions}

Having now described the history of the $\Gamma$-function in general, we want to give an overview more focused on Euler's work. To this end, let us first mention that most historical overviews, e.g., \cite{Da59} mostly consider only Euler's first paper on the $\Gamma-$function \cite{E19} and discuss mainly the two presentations of the $\Gamma$-function Euler gave in that same paper, i.e. his product representation and the integral representation.\\
But the $\Gamma-$function appears in Euler's work at several places throughout his  career and he devoted several papers to it. Additionally, as we will show below in section \ref{sec: Euler's direct Solution of the Equation Gamma (x+1)=xGamma(x) - The Moment-Ansatz}, some of his results he discovered while investigating other subjects can also be applied to the $\Gamma$-function, but Euler did not always do so. To give an example, in his 1750 paper \cite{E123}, he considered difference equations of the form:

\begin{equation*}
    (\alpha x + a)f(x+1)= (\beta x +b)f(x) +(\gamma x +c)f(x-1)
\end{equation*}
but not does not apply it to the functional equation of the $\Gamma-$function, being a special case of the above equation. He did this only in the 1785 paper \cite{E594} on the same subject in section \ref{subsubsec: Application to the Gamma-function - Finding the Integral Representation}. Both papers are actually devoted to continued fractions and the results on the $\Gamma$-function are just a byproduct.\\[2mm]
Thus, it will be convenient to give a comprehensive list of Euler's papers on the $\Gamma$-function and group them into categories, whether they are directly devoted to the $\Gamma$-function or the results are just byproducts of other investigations. Furthermore, in the second part, we want to distinguish, whether Euler realized the consequences from his investigations or not, since on some occasions he did, on others he did not.\\[2mm]

\subsubsection{Euler's Papers on the $\Gamma$-function}
\label{subsubsec: Euler's Papers on the Gamma-function}

In total, Euler wrote 13 papers either directly devoted to the $\Gamma$-function or  containing insights on the $\Gamma-$function as byproducts. The papers devoted to the $\Gamma$-functions are: \cite{E19}, \cite{E652}, \cite{E661}, \cite{E421}, \cite{E662}, \cite{E816}, \cite{E368}.\\
The papers containing results on the $\Gamma$-function as a byproduct are: \cite{E122}, \cite{E123}, \cite{E189}, \cite{E613}, \cite{E594}, chapter 4, 16 and 17 of the second part of the book \cite{E212}.\\
For the sake of convenience, we want to give a brief summary of each paper and its title already here, although we will mention them  in more detail in the corresponding sections. On this occasion, it will be convenient to mention  Eneström's 1910 book \cite{En10}, listing all of Euler's writings in chronological order with a short description of the content. Eneström attributed a number  from $1$ to $866$ to each of Euler's papers or books, referred to a the Eneström number. The smaller the Eneström number is, the earlier the paper was published. We always cite Euler's papers by their Eneström number.\\

1. \cite{E19}: \textit{De progressionibus transcendentibus seu quarum termini generales algebraice dari nequeunt}, published in 1738.\\
This is Euler's first paper on the $\Gamma$-function. As mentioned above in section \ref{subsec: Brief historical Overview}, Euler found  the integral representation and the product representation. \\[2mm]

2. \cite{E122}: \textit{De productis ex infinitis factoribus ortis}, published in 1750.\\
Euler studied $B$-function integrals and derives several relations among them. \\[2mm]

3. \cite{E123}: \textit{De fractionibus continuis observationes}, published in 1750.\\
This is Euler's longest paper on continued fractions he ever wrote. It contains many topics, including the transformation of series into continued fractions and the solution of the Riccati differential equation via continued fractions. Additionally, Euler described a method to solve homogeneous difference equations with linear coefficients. See section \ref{sec: Euler's direct Solution of the Equation Gamma (x+1)=xGamma(x) - The Moment-Ansatz}.\\[2mm]

4. \cite{E189}: \textit{De serierum determinatione seu nova methodus inveniendi terminos generales serierum}, published in 1753.\\
Euler solved several difference equations by converting them into a differential equation of infinite order and solving those differential equations then. The factorial is one example he considered. We devote a whole section to the explanation of his ideas. See section \ref{sec: Solution of the Equation log Gamma(x+1)- log  Gamma (x)= log x by Conversion into a Differential Equation of infinite Order}. \\[2mm]

5. \cite{E212}: \textit{Institutiones calculi differentialis cum eius usu in analysi finitorum ac doctrina serierum}, published 1755:\\
This is Euler's book on differential calculus. It contains many topics. We mention the introduction of differential quotients, the Euler-Maclaurin summation formula, and many transformation formulas for series as examples. We will mainly focus on his interpolation theory based on difference calculus. See section \ref{sec: Interpolation Theory and Difference Calculus}.\\[2mm]

 6. \cite{E368}: \textit{De curva hypergeometrica hac aequatione expressa 
$y=1 \cdot 2 \cdot 3 \cdots x$}, published 1769.\\
This is an overview article on the $\Gamma$-function. Euler stated all formulas he ever found for $x!$ - most of them without out proof -, evaluated the factorial and its derivative for special values. \\[2mm]

7. \cite{E421}: \textit{Evolutio formulae integralis $\int x^{f-1}dx \left(\log (x)\right)^{\frac{m}{n}}$ integratione a valore $x=0$ ad $x=1$ extensa}, published 1772.\\
This is an overview paper on the integral representation of the $\Gamma$-function. It contains many results - with proofs - e.g., the relation among the $\Gamma$- and $B$-function, the derivation of the sine product formula, even formulas attributed to Gau\ss{}. All this will be discussed in \ref{sec: Relation between Gamma and B}.\\[2mm]

8. \cite{E594}: \textit{Methodus inveniendi formulas integrales, quae certis casibus datam inter se teneant rationem, ubi sumul methodus traditur fractiones continuas summandi}, published 1785.\\
This paper can be considered as an addendum to \cite{E123}. It contains several examples not discussed in the latter, including the $\Gamma$-function. Thus, \cite{E594} contains a derivation of the integral representation of the $\Gamma$-function from its functional equation.\\[2mm]

9. \cite{E613}: \textit{Dilucidationes in capita postrema calculi mei differentalis de functionibus inexplicabilibus}, published in 1787.\\
This is an addendum to chapter 16 and 17 in \cite{E212}. It is a more detailed discussion of his solution of the interpolation  problem of a sum using difference calculus. He applied his general formulas to the factorial this time.\\[2mm]

10. \cite{E652}: \textit{De termino generali serierum hypergeometricarum}, published in 1793.\\
This paper contains the proof of the Gau\ss{}ian representation of the factorial. The proof is based on the functional equation. We discuss this in \ref{subsec: Euler on Weierstrass's Condition}.\\[2mm]

11. \cite{E661}: \textit{Variae considerationes circa series hypergeometricas}, published in 1794.\\
Euler uses the Euler-Maclaurin summation formula to study the factorial.\\[2mm]

12. \cite{E662}: \textit{De vero valore formulae integralis $\int \partial x\left(\log \frac{1}{x}\right)^n$ a termino $x = 0$ usque ad terminum $x = 1$ extensae}, published in 1894, but presented already in 1776.\\
In this paper, Euler tried to evaluate the integral for several rational numbers. Due to the transcendence of the $\Gamma$-function, his efforts are eventually in vain. Nevertheless, he was able to express products of several values in terms of familiar quantities, like $\pi$ and $\sin \pi x$. \\[2mm]

13. \cite{E816}: \textit{Considerations sur quelques formules integrales dont les valeurs peuvent etre exprimees, en certains cas, par la quadrature du cercle}, published in 1862. \\
This paper was written in the same spirit as \cite{E662}.\\[3mm]
Having described all papers briefly, we want to note that the papers \cite{E19}, \cite{E368} and \cite{E421} together contain all formulas and insights on the $\Gamma$-function that Euler ever made and had. Thus, we provided a translation of the Latin originals into English of all three papers\footnote{Meanwhile, all  12 (of 13) papers, which were written in Latin, have been translated into English by the author of this thesis. The translations are available at online the Euler-Kreis Mainz.}. They are found in the appendix. As mentioned, the 1738 paper \cite{E19} was the first appearance of the $\Gamma$-function in print ever. The  1772 paper \cite{E421} and the 1769 paper \cite{E368} are rather overview papers. The first is devoted to the integral representation, whereas the second focuses more on the product representations. Thus, they are elaborations on both the representations that appeared already in \cite{E19}. Unfortunately, the proofs of the formulas are not always found in those same papers, but in others. In in some cases, Euler even only gave a heuristic proof.\\[2mm]

\subsubsection{Euler's Results on the $\Gamma$-function}
\label{subsubsec: Euler's Results on the Gamma-function}

Here, we want to give a list of Euler's results on the $\Gamma$-function together with the paper, in which it can be found.  How he obtained his result, will be explained below in the corresponding sections in more detail.\\[3mm]

1. \textit{Integral Representation}: \\ Euler found the function defined as:

\begin{equation*}
    \int\limits_{0}^{1} \left(\log \dfrac{1}{x}\right)^ndx
\end{equation*}
as a solution to the problem of interpolating of the factorial. In \cite{E19} he obtained it rather indirectly from a $B$-function integral. We will see this below in section \ref{subsubsec: Using the B-function}. But confer also \cite{Da59}, \cite{Sa07}. The paper \cite{E123}, actually devoted to the continued fractions, contains the tools to obtain the integral formula directly. He did this in the paper \cite{E594}, also a paper on continued fractions. \\[2mm]

2. \textit{Euler Product Representation}: \\ By this we mean the formula

\begin{equation*}
    x! = \dfrac{1}{1+x} \left(\dfrac{2}{1}\right)^x \cdot  \dfrac{2}{2+x} \left(\dfrac{3}{2}\right)^x \cdot  \dfrac{3}{3+x} \left(\dfrac{4}{3}\right)^x \cdot  \dfrac{4}{4+x} \left(\dfrac{5}{4}\right)^x \cdots
\end{equation*}
which he first stated without proof in \cite{E19}. It is also stated  in the first paragraph of \cite{E368}. But neither of those papers contains a proof. A proof had to wait until 1755 in chapter 16 and 17 of the second part of his book from 1755. There, he solved interpolation problems via difference calculus. The idea is presented in more detail in \cite{E613}. The above formula is just a corollary of more general formulas. Below in section \ref{subsec: Difference Calculus according to Euler} we will discuss Euler's ideas from \cite{E613}  and show, how they anticipated the fundamental ideas of Weierstra\ss-products in section \ref{subsec: Modern Idea - Weierstrass product}. \\[2mm]

3. \textit{Gau\ss ian  product representation}, i.e. 

\begin{equation*}
    x! = \lim_{n \rightarrow \infty} \dfrac{1}{1+x} \cdot \dfrac{2}{2+x} \cdots \dfrac{n}{n+x}(n+1)^x.
\end{equation*}
This formula is often attributed to Gau\ss{}, since it was the starting point of his investigations on the factorial in the influential paper \cite{Ga28}. But Euler, arguing via infinitely large numbers, proved this formula in \cite{E652}. We will discuss later in section \ref{subsec: Euler on Weierstrass's Condition} how he found this product. Additionally, Euler gave this definition in his 1729 letter to Goldbach \cite{Eu29}, as we have already seen in section \ref{subsec: More detailed Discussion on the Discovery of the Gamma-function}. \\[2mm]

4. \textit{Stirling formula for $\Gamma(x+1)$:}\\
The Stirling formula says

    \begin{equation*}
     x! =
\Gamma (x+1)= \sqrt{2 \pi x}\dfrac{x^x}{e^x} \quad \text{for} \quad x \rightarrow \infty.
\end{equation*}
This formula is attributed to Stirling for his 1730 investigations \cite{St30}. But we will have to say several things concerning priorities below in section \ref{subsubsec: An Application - Derivation of the Stirling Formula for the Factorial}. Anyhow, Euler proved this formula on two occasions. First, in his 1755 book \cite{E212}. There, the formula is derived as a consequence from his more general results on the Euler-Maclaurin summation formula. We will represent his proof, when discussing the Stirling formula in section \ref{subsubsec: An Application - Derivation of the Stirling Formula for the Factorial}. \cite{E661} is then solely devoted to investigations of slight generalisations of the factorial functions via the Euler-Maclaurin formula. He also stated the Stirling formula in \cite{E368}.\\
He also tried to give another proof in \cite{E189}. In this paper, he solved difference equations including $f(x+1)-f(x)=\log(x)$, which is solved by $\log \Gamma (x)$ of course, by converting them into differential equations of infinite order first. But his general investigations contain a conceptual error and thus his proof is incorrect. We will elaborate on this below in section \ref{subsubsec: Mistake in Euler's Approach}, discuss his error and give a method how to correct his arguments. \\[2mm]

5. \textit{Series expansions for $\log \Gamma(x+1)$}:\\
Euler found several series expansions for the logarithm of the $\Gamma-$function, he gave a list in \cite{E368}. Let us mention two striking examples:

\begin{alignat*}{9}
& \log \Gamma(x+1) && =  - \gamma x +x +\dfrac{1}{2}x+\dfrac{1}{3}x+\dfrac{1}{4}x +\text{etc.}\\
&  && -\log \left(1+\dfrac{x}{1}\right)-\log \left(1+\dfrac{x}{2}\right) -\log \left(1+\dfrac{x}{3}\right)-\log \left(1+\dfrac{x}{4}\right)+\text{etc.}
\end{alignat*}
Taking the exponentials and applying the functional equation $\Gamma(x+1)=x\Gamma(x)$, we arrive at Newman's formula which we saw above already

\begin{equation*}
    \dfrac{1}{\Gamma(z)}= ze^{\gamma z}\left[\left(1+z\right)e^{-z}\right]\left[\left(1+\dfrac{z}{2}\right)e^{-\frac{z}{2}}\right]\cdots.
\end{equation*}
As a second example, we mention the Taylor series expansion of $\log \Gamma(x+1)$:

\begin{equation*}
    \log \Gamma(x+1) = - \gamma x + \dfrac{\zeta(2)}{2}x^2 -  \dfrac{\zeta(3)}{3}x^3 +  \dfrac{\zeta(4)}{4}x^4 - \text{etc.}
\end{equation*}
where

\begin{equation*}
    \zeta(z)= \sum_{k=1}^{\infty} \dfrac{1}{k^z}.
\end{equation*}
A proof by Euler can be found in chapter 16 and 17 of the second part of \cite{E212}.\\[2mm]

6. \textit{Reflection formula} 

\begin{equation*}
    \dfrac{\pi}{\sin (\pi x)}= \Gamma(x)\Gamma(1-x).
\end{equation*}
He used it in his investigations on the $\zeta$-function in \cite{E352}, but provided a proof based on the product formula of the $\Gamma-$function and the sine product formula, discovered in \cite{E41} but proved later in \cite{E61}, just in \cite{E421}. Below, in section \ref{subsubsec: Euler's Proof of the Reflection Formula}, we will present this proof  and also discuss Euler's proof of the sine product formula in section \ref{subsubsec: Euler's Proof of the Sine Product Formula}.\\[2mm]

7. \textit{Relation among the $\Gamma-$ and $B-$function}:

\begin{equation*}
    B(x,y) = \dfrac{\Gamma(x)\Gamma(y)}{\Gamma(x+y)}.
\end{equation*}
Euler stated this formula on several occasions, e.g. \cite{E421}, \cite{E816}, \cite{E662}. But he never gave a satisfactory proof. \\[2mm]

8. \textit{Gau\ss ian multiplication formula}:

\begin{equation*}
    \Gamma \left(\dfrac{x}{n}\right)\Gamma \left(\dfrac{x+1}{n}\right)\cdots \Gamma \left(\dfrac{x+n-1}{n}\right)= \dfrac{(2\pi)^{\frac{n-1}{2}}}{n^{x-\frac{1}{2}}}\cdot \Gamma(x)
\end{equation*}
Euler did not state it in the above form, but, in \cite{E421}, he gave the formula:

\begin{equation*}
\left[ \frac{m}{n} \right] = \frac{m}{n} \sqrt[n]{n^{n-m}\cdot 1 \cdot 2 \cdot 3 \cdots (m-1) \left(\frac{1}{m}\right)\left(\frac{2}{m}\right)\left(\frac{3}{m}\right)\cdots \left(\frac{n-1}{m}\right)}.
\end{equation*}
Here is the explanation of Euler's notation from \cite{E421}:

\begin{equation*}
    [x]=x!= \Gamma(x+1) \quad \text{and} \quad  \left(\dfrac{p}{q}\right)= \dfrac{1}{n}\int\limits_{0}^{1}dx x^{\frac{p}{n}-1}(1-x)^{\frac{q}{n}-1}= \frac{1}{n}B \left(\frac{p}{n}, \frac{q}{n}\right). 
\end{equation*}
In section \ref{subsec: Multiplication Formula}, we will show that this formula is equivalent to the Gau\ss ian formula. Euler did not prove this formula and only obtains it heuristically by pattern recognition. \\[2mm]

9. \textit{Expression of $\Gamma \left(\frac{p}{q}\right)$ in terms of integrals of algebraic functions}: \\ In \cite{E19} and \cite{E122}, Euler gave the formula

\begin{equation*}
\int\limits_{0}^{1} \left(-\log x\right)^{\frac{p}{q}}dx =\sqrt[q]{1 \cdot 2 \cdot 3 \cdots p\left(\dfrac{2p}{q}+1\right)\left(\dfrac{3p}{q}+1\right)\left(\dfrac{4p}{q}+1\right)\cdots \left(\dfrac{qp}{q}+1\right)}
\end{equation*}
\begin{equation*}
\times \sqrt[q]{\int\limits_{0}^{1} dx(x-xx)^{\frac{p}{q}} \cdot \int\limits_{0}^{1} dx(x^2-x^3)^{\frac{p}{q}} \cdot \int\limits_{0}^{1} dx(x^3-x^4)^{\frac{p}{q}} \cdot \int\limits_{0}^{1} dx(x^4-x^5)^{\frac{p}{q}} \cdots \int\limits_{0}^{1} dx(x^{q-1}-x^q)^{\frac{p}{q}}}.
\end{equation*}

\newpage

\section{Short modern Introduction to the $\Gamma$-Function}
\label{sec: Short modern Introduction to the Gamma-Function}

We briefly mention the modern definition of $\Gamma(x)$ following  \cite{Fr06} (pp. 194-197). We start from the integral representation and derive the characteristic properties from it. Furthermore, we will obtain two equivalent characterisations of the $\Gamma$-function, based on Wielandt's theorem in section \ref{subsubsec: Wielandt's Theorem} and the Bohr-Mollerup theorem in section \ref{subsubsec: Bohr-Mollerup Theorem}.

\subsection{Definition and simple Properties}
\label{subsec: Definition and simple Properties}

\subsubsection{Definition}
\label{subsubsec: Definition}

\begin{definition}[$\Gamma$-integral]
We define the $\Gamma$-function as a function in the complex plane as the following integral

\begin{equation*}
    \Gamma(z) := \int\limits_{0}^{\infty} t^{z-1}e^{-t}dt.
\end{equation*}
Here $t^{z-1}:= e^{(z-1)\log (t)}$, $\log t \in \mathbb{R}$, $\operatorname{Re}(z) >0$.
\end{definition}
We have the following simple theorem:

\begin{theorem}
The $\Gamma$-integral

\begin{equation*}
    \Gamma(z) := \int\limits_{0}^{\infty} t^{z-1}e^{-t}dt
\end{equation*}
converges absolutely for $\operatorname{Re}(z) >0$ and represents an analytic function on the domain. The derivatives are given (for $k \in \mathbb{N}$) by

\begin{equation*}
    \Gamma^{(k)}(z) = \int\limits_{0}^{\infty} t^{z-1}(\log t)^{k}e^{-t}dt.
\end{equation*}
\end{theorem}
\begin{proof}
We split the $\Gamma$ integral into the two integrals

\begin{equation*}
    \Gamma(z) = \int\limits_{0}^{1} t^{z-1}e^{-t}dt +\int\limits_{1}^{\infty} t^{z-1}e^{-t}dt
\end{equation*}
and use the relation

\begin{equation*}
    \left|t^{z-1}e^{-t}\right| = t^{x-1}e^{-t}
\end{equation*}
where we wrote $x$ for $\operatorname{Re}(z)$. Let us consider both integrals separately.  In general, for each $x_0 >0$ there is a number $C>0$ with

\begin{equation*}
    t^{x-1} \leq C e^{\frac{t}{2}} \quad \forall ~ x \quad \text{with} \quad 0 < x \leq x_0 \quad ~\text{and}~ t\geq 1. 
\end{equation*}
Thus, the integral

\begin{equation*}
    \int\limits_{1}^{\infty} t^{z-1}e^{-t}dt
\end{equation*}
converges absolutely for all $z \in \mathbb{C}$.\\[2mm]
For the other integral, we use the estimate

\begin{equation*}
    \left|t^{z-1}e^{-t}\right|  < t^{x-1} \quad \text{for} ~ t >0
\end{equation*}
and the existence of the integral

\begin{equation*}
    \int\limits_{0}^{1} \dfrac{1}{t^s}dt \quad \text{for} \quad s <1.
\end{equation*}
From these estimates it follows that the sequence of functions

\begin{equation*}
    f_n(z) := \int\limits_{\frac{1}{n}}^{n} t^{z-1}e^{-t}dt
\end{equation*}
converges uniformly to $\Gamma$ for $n \rightarrow \infty$. Therefore, $\Gamma$ is an analytic function.\\
The formula for the $k-$th derivative follows from the application of the Leibniz rule (for differentiation) and then taking the limit $n \rightarrow \infty.$
\end{proof}

\subsubsection{Simple Properties}
\label{subsubsec: Simple Properties}

We have

\begin{theorem}[Elementary properties of the $\Gamma$-integral]
The $\Gamma$-function can be analytically continued to the whole complex plane except at the points

\begin{equation*}
    z \in S := \lbrace 0, -1, -2, -3, \cdots \rbrace
\end{equation*}
and at $\mathbb{C}\setminus S$ satisfies the functional equation:

\begin{equation*}
    \Gamma(z+1)=z\Gamma(z).
\end{equation*}
All singularities are poles of first order with the residues:

\begin{equation*}
    \operatorname{Res}(\Gamma, -n) = \dfrac{(-1)^n}{n!}
\end{equation*}
\end{theorem}
\begin{proof}
We show the functional equation first. Obviously, we have

\begin{equation*}
    \Gamma(1) = \int\limits_{0}^{\infty} e^{-t}dt = \left[-e^{-t}\right]_{0}^{\infty}=1.
\end{equation*}
By integration by parts one arrives at the functional equation 

\begin{equation*}
    \Gamma(z+1) = z \Gamma(z) \quad \text{for} \quad \operatorname{Re}(z) >0.
\end{equation*}
Using the functional equation iteratively, we find

\begin{equation*}
    \Gamma(z) = \dfrac{\Gamma(z+n+1)}{z \cdot (z+1)\cdots (z+n)}.
\end{equation*}
The right-hand side of the equation has a large domain where it can be defined, i.e.

\begin{equation*}
    \operatorname{Re}(z)> -(n+1) \quad \text{and} \quad z \neq 0, -1, -2, -3, \cdots ,-n.
\end{equation*}
Therefore, the above equation is an analytic continuation of $\Gamma$ into a larger domain. \\[2mm]
Finally, let us consider the residues. Using the functional equation, we have

\begin{equation*}
    \operatorname{Res}(\Gamma; - n) = \lim_{z \rightarrow -n}(z+n)\Gamma(z) = \dfrac{\Gamma(1)}{(-n)(-n+1)\cdots (-1)}= \dfrac{(-1)^n}{n!}.
\end{equation*}
\end{proof}

\subsection{Classification Theorems}
\label{subsec: Classification Theorems}

The $\Gamma$-function was invented by Euler to interpolate the factorial in 1738 \cite{E19}. The integral representation obviously fulfills this task, since $\Gamma(n+1)=n!$. The factorial has these two properties $0!=1$ and $n!=n(n-1)!$. Therefore, this automatically raises the question, whether the $\Gamma$-function is the only holomorphic function with $\Gamma(z+1)=z\Gamma(z)$ and $\Gamma(1)=1$. As already mentioned in section \ref{subsubsec: 6. Phase: Towards the axiomatic Introduction}, the answer to this question is no, since, e.g., 

\begin{equation*}
    f(z):= (1+\sin(2 \pi z))\Gamma(z)
\end{equation*}
also has these two properties. \\
Below we will encounter several other expressions also satisfying the functional equation and $f(1)=1$ and above we already did in section \ref{subsubsec: 6. Phase: Towards the axiomatic Introduction}. Therefore, it will be useful to have theorems that tell us immediately that the new expression is indeed the $\Gamma$-function without showing the equality to the integral representation directly.\\
This is provided by classification theorems. They state that the $\Gamma$-function can be uniquely defined by the two obvious properties $\Gamma(1)=1$ and $\Gamma(z+1)=z \Gamma(z)$ and an additional third one. We will present two theorems, Wielandt's theorem and the Bohr-Mollerup theorem. In Bourbaki \cite{Bo51}, the Bohr-Mollerup theorem is the starting point for theory of the $\Gamma$-function. There, one does not start from a specific representation.

\subsubsection{Wielandt's Theorem}
\label{subsubsec: Wielandt's Theorem}

Wielandt's Theorem is one possible characterisation of the $\Gamma$-function. Wielandt's original proof can be found in his collected papers \cite{Wi96}. Other proofs can be found, e.g., in the books \cite{Kn41} (pp. 47-49) and \cite{Fr06} (pp. 198-199) which we will present here, and in the paper \cite{Re96}. \\ 
We have:

\begin{theorem}[Wielandt's Theorem]
 Let $D \subseteq \mathbb{C}$ be a domain containing the vertical strip
 
 \begin{equation*}
     1 \leq x < 2.
 \end{equation*}
 Let $f:D \rightarrow \mathbb{C}$ be a function with the following properties:\\
 1) $f$ is bounded in the vertical strip \\
 2) We have 
 
 \begin{equation*}
     f(z+1)=zf(z) \quad \text{for} \quad z, z+1 \in D
 \end{equation*}
 Then we have:
 
 \begin{equation*}
     f(z)= f(1)\Gamma(z) \quad \text{for} \quad z \in D.
 \end{equation*}
\end{theorem}
\begin{proof}
Applying the functional equation, it is easily seen that the function $f$ can be analytically continued to the whole complex plane except at the points:

\begin{equation*}
z \in S = \lbrace 0,-1,-2, -3, \cdots\rbrace    
\end{equation*}
and satisfies

\begin{equation*}
    f(z+1)=zf(z).
\end{equation*}
All $z \in S$ are either poles of first order or removable singularities, and we have:

\begin{equation*}
    \operatorname{Res}(f;-n)= \dfrac{(-1)^n}{n!}f(1).
\end{equation*}
Therefore, the function $h(z):= f(z)-f(1)\Gamma(z)$ is an entire function. Furthermore, it is bounded in the vertical strip $0 \leq x \leq 1$, which follows immediately from the boundedness in the strip $1 \leq x <2$ and the functional equation for $|\operatorname{Im}(z)| \geq 1$. The domain $|\operatorname{Im}(z)| \leq 1$, $0 \leq \operatorname{Re}(z) \leq 1$ is compact.\\
We want to use Liouville's theorem and observe that from the functional equation for $h$, i.e. $h(z)z = h(z+1)$, if we define

\begin{equation*}
H(z) := h(z)h(1-z),    
\end{equation*}
we find $H(z+1)=-H(z)$. But the strip $0 \leq x \leq 1$ is not changed under the transformation $z \rightarrow 1-z$. Thus, $H$ is bounded on this strip and, because of the periodicity, it is bounded on $\mathbb{C}$. Therefore, Liouville's theorem implies that $H$ is constant.  But $h(1)=0$, so $H=0$ and hence also $h=0$ for all $z \in \mathbb{C}$.
\end{proof}
The $\Gamma$-integral obviously satisfies all three properties. We will see this below, when we find the integral representation from the moment ansatz in section \ref{sec: Euler's direct Solution of the Equation Gamma (x+1)=xGamma(x) - The Moment-Ansatz}.

\subsubsection{Bohr-Mollerup Theorem} 
\label{subsubsec: Bohr-Mollerup Theorem}

The Bohr-Mollerup Theorem, first proved in the 1922 book \cite{Mo22}, also states that the Gamma function can be uniquely classified by three properties. In other words, aside from the two obvious ones $\Gamma(x+1)=x\Gamma(x)$, $\Gamma(1)=1$, we, as in the case of Wielandt's theorem, need one additional property. This is the so-called logarithmic convexity. For the sake of completeness, let us define convexity first and show that $\Gamma(x)$ has the property of logarithmic convexity, before we get to the theorem.

\begin{definition}[Logarithmic Convexity]
Let $X,Y$ be open subsets of the real numbers $\mathbb{R}$. Further, let $f: X \rightarrow Y$ be a function. Then, $f$ is called {convex}, if the following inequality holds:

\begin{equation*}
    f(tx +(1-t)y) \leq tf(x)+ (1-t)f(y) \quad \forall ~x,y \in X, \quad \forall t \in \left[0,1\right]
\end{equation*}
Furthermore, $f$ is called {logarithmically convex} , if $\log f(x)$ is convex.  
\end{definition}
Let us state a theorem which can be used if $f$ additionally is twice continuously differentiable. 

\begin{theorem}
 If the second derivative of a twice continuously differentiable function is always $\geq 0$ in the interval $(a,b)$, then the function $f$ is convex in this interval. The converse of this theorem is also true.
\end{theorem}
The proof can be found in every book on analysis of one variable, one can also find a proof in \cite{Ar15} (pp. 6-7). We will need the following corollary.

\begin{corollary}
If $f: \mathbb{R} \rightarrow \mathbb{R}$ is twice continuously differentiable and the following inequalities are satisfied for all $x\in (a,b)$

\begin{equation*}
    f(x) >0, \quad f(x)f''(x) -(f'(x))^2 \geq 0,
    \end{equation*}
    then $\log f$ is convex, i.e. $f$ is logarithmically convex, in this interval.
\end{corollary}
For a proof one just has to apply the previous theorem to $\log f$. Further, we have

\begin{corollary}
The sum of two logarithmically convex function is also logarithmically convex.
\end{corollary}
We will not prove this statement here. For a proof  the reader is referred to \cite{Ar15}. Instead, we want to go over to the logarithmic convexity of $\Gamma$. For this, consider $f(t,x)$, continuous in both variables $x$ and $t$. Let $a \leq t \leq b$ and $x$ live in another interval. If $f(t,x)$ now is  logarithmically convex for all $t$ and twice continuously differentiable with respect to $x$, define:

\begin{equation*}
    F_n(x) = h \left\lbrace f(a,x)+f(a+h,x)+f(a+2h,x)+ \cdots + f(a+(n-1)h,x), \quad h= \dfrac{b-a}{n} \right\rbrace
\end{equation*}
Then, $F_n(x)$ is also logarithmically convex for all $n \in \mathbb{N}$. Therefore, also

\begin{equation*}
    \lim_{n \rightarrow \infty}F_n(x) = \int\limits_{a}^{b}f(t,x)dt
\end{equation*}
is logarithmically convex. This also holds for improper integrals, if the integral exists. Therefore, we  have:

\begin{theorem}
The $\Gamma$- function, given as

\begin{equation*}
    \int\limits_{0}^{\infty} t^{x-1}e^{-t}dt,
\end{equation*}
is logarithmically convex for $x >0$.
\end{theorem}

Now, having mentioned all this in advance, we can finally state the Bohr-Mollerup theorem.

\begin{theorem}[Bohr-Mollerup Theorem]
If a function $f: \mathbb{R}^+ \rightarrow \mathbb{R}$ satisfies the  three properties\\
1) $f(x+1)=xf(x)$\\
2) $f$ is logarithmically convex on the whole domain where it is defined \\
3) $f(1)=1$,\\
it is identical to the $\Gamma$-function in the region where it is defined.
\end{theorem}
\begin{proof}
We have shown that $\Gamma$ satisfies all conditions. Therefore, let $f$ be another function with the above properties. From the functional equation we find

\begin{equation*}
    f(x+n) =(x+n-1)(x+n-2) \cdots (x+1)x f(x).
\end{equation*}
Since $f(1)=1$ we have $f(n)=\Gamma(n)~ \forall n \in \mathbb{N}$. We only need to show $f=\Gamma$ for the interval $0 < x \leq 1$, because of the functional equation. Thus, let $x$ be a number in that  interval and $n$ a natural number $\geq 2$. Then, we have the following inequality

\begin{equation*}
    \dfrac{\log (f(-1+n))-\log (f(n))}{(-1+n)-n} \leq \dfrac{\log (f(x+n))- \log (f(n))}{(x+n)-n} \leq \dfrac{\log (f(1+n))- \log (f(n))}{(1+n)-n}
\end{equation*}
which follows from the logarithmic convexity. We can simplify the last equation:

\begin{equation*}
    \log (n-1) \leq \dfrac{\log (f(x+n))- \log (f(n))}{(x+n)-n} \leq \log n
\end{equation*}
or

\begin{equation*}
  \log ((n-1)^x(n-1)!) \leq f(x+n) \leq \log (n^x(n-1)!).  
\end{equation*}
Using the above equation for $f(x+n)$:

\begin{equation*}
    \dfrac{(n-1)^x(n-1)!}{x(x+1)\cdots (x+n-1)} \leq f(x) \leq \dfrac{n^x(n-1)!}{x(x+1)\cdots (x+n-1)}= \dfrac{n^xn!}{x(x+1)\cdots (x+n)}\cdot \dfrac{x+n}{n}. 
\end{equation*}
Since we assumed $n \geq 2$, we can replace $n$ by $n+1$ and find:

\begin{equation*}
    \dfrac{n^xn!}{x(x+1)\cdots (x+n)}\leq f(x) \leq \dfrac{n^x n!}{x(x+1) \cdots (x+n)}\cdot \dfrac{x+n}{n}.
\end{equation*}
Therefore,

\begin{equation*}
    f(x)\dfrac{n}{n+x} \leq \dfrac{n^x n!}{x(x+1) \cdots (x+n)} \leq f(x).
\end{equation*}
Taking the limit $n \rightarrow \infty$:

\begin{equation*}
    f(x) = \lim_{n \rightarrow \infty} \dfrac{n^x n!}{x(x+1) \cdots (x+n)}.
\end{equation*}
Since the function $f$ only had to satisfy the three conditions in the theorem and was arbitrary otherwise, we conclude $f(x)= \Gamma(x)$.
\end{proof}
We have the following corollary:

\begin{corollary}
\begin{equation*}
    \Gamma(x) = \lim_{n \rightarrow \infty} \dfrac{n^x n!}{x(x+1) \cdots (x+n)}.
\end{equation*}
\end{corollary}We will find other ways to get to this product representation below. It is interesting that it follows directly from the proof. Additionally, we already pointed out in section \ref{subsubsec: 2. Phase: Gauss's Investigations} that Gau\ss{} in 1812 \cite{Ga28} used the last corollary as a definition for the $\Gamma$-function.  However, we already saw in section \ref{subsec: More detailed Discussion on the Discovery of the Gamma-function} that this product formula had already been discovered by Euler in \cite{Eu29}.

\newpage

\section{Solution of the Difference Equation $F(x+1)=xF(x)$}
\label{sec: Solution of the difference equation F(x+1)=xF(x)}

In this section we will solve the functional equation in general and from there descend to the $\Gamma$-function. This will lead us to the Weierstra\ss{} product expansion of the $\Gamma$-function. Our exposition follows \cite{Ni05}.

\subsection{Weierstra\ss's Definition of the $\Gamma$-function}
\label{subsubsec: Weierstrass's Definition of the Gamma-function}

It was Weierstra\ss's \cite{We56} idea to define the $\Gamma$-function as solution of the difference equation

\begin{equation*}
F(x+1)=xF(x)    
\end{equation*}
with the additional condition\footnote{Weierstra\ss{} added the condition $F(1)=1$ which, however, is not necessary to define the $\Gamma$-function uniquely.}

\begin{equation*}
    \lim_{n \rightarrow \infty} \dfrac{F(x+n)}{(n-1)!n^x}=1.
\end{equation*}
The above condition, as we will see soon,  excludes solutions of the form $\Gamma(x)p(x)$ with $p$ a periodic function with period $1$.\\[2mm]
Anyhow, in this section we want to solve the difference equation in general and want to show that it indeed defines the $\Gamma$-function as claimed.

\subsection{A Remark concerning the Solution of the Difference Equation}
\label{subsec: A Remark concerning the Solution of the Difference Equation}

Let us begin with the following remark:

\begin{remark}
In order to solve  the difference equation $F(x+1)=xF(x)$, we essentially only need one particular solution.
\end{remark}

\begin{proof}
For, let $F_1(x)$ and $F_2(x)$ be two particular solutions of the difference equation.  Then, one has

\begin{equation*}
    \dfrac{F_2(x+1)}{F_1(x+1)}=\dfrac{xF_2(x)}{xF_1(x)}=\dfrac{F_2(x)}{F_1(x)},
\end{equation*}
i.e. the quotient of the two solutions is  a periodic function with period $+1$. In other words, if $F_1(x)$ is a solution of the difference equation, then every other solution $F_2(x)$ is connected to it by

\begin{equation*}
    F_2(x)=\omega(x)F_1(x) \quad \text{with} \quad \omega(x+1)=\omega(x).
\end{equation*}
\end{proof}

\subsection{General Solution of the Equation $F(x+1)=xF(x)$}
\label{subsec: General Solution of the Equation F(x+1)=xF(x)}

\subsubsection{Introduction of an Auxiliary Function}
\label{subsubsec: Introduction of an Auxiliary Function}

Keeping the remark of the previous section in mind, we can now proceed to find a solution of the difference equation and show that the $\Gamma$-function is actually the only one. For this aim, let us introduce the following function:

\begin{definition}
We define a function\footnote{We will meet this function again in section \ref{subsubsec: Harmonic Series}, when we talk about Euler's ideas on interpolation of so-called inexplicable functions, a term he coined in 1755 in chapter 16 of \cite{E212}.} $\Sigma: \mathbb{C}\setminus\lbrace 0,-1,-2, -3, \cdots \rbrace \rightarrow \mathbb{C}$ by the sum

\begin{equation*}
    \Sigma(x) := \sum_{s=0}^{\infty} \left(\dfrac{1}{s+1}-\dfrac{1}{x+s}\right).
\end{equation*}
\end{definition}
 This series is easily seen to converge uniformly. Thus, we are allowed to integrate it term by term with respect to $x$, provided the path of integration is of finite length and does not pass through any of the poles.\\[2mm]
Furthermore, we have

\begin{theorem}
$\Sigma$  satisfies the functional equation

\begin{equation*}
    \Sigma(x+1)= \Sigma(x)+\dfrac{1}{x}.
\end{equation*}
\end{theorem}
\begin{proof}
Consider the difference $\Sigma(x+1)-\Sigma(x)$; it reads

\begin{equation*}
\renewcommand{\arraystretch}{2,5}
\setlength{\arraycolsep}{0.0mm}
\begin{array}{llllll}
    &\sum_{s=0}^{\infty} \left(\dfrac{1}{s+1}-\dfrac{1}{x+1+s}\right) - \sum_{s=0}^{\infty} \left(\dfrac{1}{s+1}-\dfrac{1}{x+s}\right) \\ ~=~&\sum_{s=0}^{\infty} \left(\dfrac{1}{s+1}-\dfrac{1}{x+1+s}\right)-(1-\dfrac{1}{x})- \sum_{s=0}^{\infty} \left(\dfrac{1}{s+2}-\dfrac{1}{x+1+s}\right)\\
    ~=~& 1-1+\dfrac{1}{x}=\dfrac{1}{x},
    \end{array}
\end{equation*}
since the sums involving $x$ cancel and the sums involving only $s$ are telescoping sums.
\end{proof}

\subsubsection{Product Representation of the Function $F$}
\label{subsubsec: Product Representation of the Function F}

\begin{theorem}[]
Every meromorphic function $F$ satisfying the equation $F(x+1)=xF(x)$ has a product expansion of the form

\begin{equation*}
    F(x+1)= \omega(x) \cdot \dfrac{e^{Kx}}{x}\cdot \prod_{s=1}^{\infty} \dfrac{e^{\frac{x}{s}}}{1+\frac{x}{s}},
\end{equation*}
where $K$ is a constant and $\omega : \mathbb{C} \rightarrow \mathbb{C}$ is a integrable and periodic function with period $+1$.
\end{theorem}
\begin{proof}
Let $\omega(x)$ be as above, and let it be integrable on $(0,x)$ with $x>0$. Define

\begin{equation*}
    \omega_1(x):= \int\limits_{0}^{x}\omega(x)dx.
\end{equation*}
Then $\omega_1(x)$ satisfies the functional equation:

\begin{equation*}
    \omega_1(x+1)= \omega_1(x)+K,
\end{equation*}
$K$ being a constant. Now, recalling the definition of the function $\Sigma$ from section \ref{subsubsec: Introduction of an Auxiliary Function}, by integration we find

\begin{equation*}
    \log F(x+1) = \int\limits_{0}^{x} \Sigma(x+1)dx+K\cdot (x+1)
\end{equation*}
or, equivalently substituting the series for $\Sigma$ and integrating it term by term, we have

\begin{equation*}
    \log F(x+1) = \sum_{s=0}^{\infty} \left(\dfrac{x}{s+1}-\log \left(1+\dfrac{x}{s+1}\right)\right)+ K \cdot (x+1).
\end{equation*}
Thus, taking the exponentials, we find

\begin{equation*}
    F(x) = \omega(x) \cdot \dfrac{e^{Kx}}{x} \cdot \prod_{s=1}^{\infty} \dfrac{e^{\frac{x}{s}}}{1+\frac{x}{s}}.
\end{equation*}
\end{proof}
Thus, to summarize the proof: We basically solved the simpler equation $f(x+1)-f(x)=\frac{1}{x}$ first. A particular solution is given by our function $\Sigma$. Thus, by integrating, we can then deduce the solution of $g(x+1)-g(x)=\log(x)$ and by taking the exponentials we arrive at the functional equation for $F(x)$. It is helpful to keep this in mind, since this is also basically what Euler did in 1780 in \cite{E613} and in 1755 in \cite{E212} to find the product representation of the $\Gamma$-function. In other words, this proof can easily be constructed from Euler's ideas in that paper.

\subsubsection{Finding the Constant $K$}
\label{subsubsec: Finding the constant K}

Finally, we need to find the constant $K$, which was introduced by an integration in the previous section. Hence let us introduce the sequence of functions $G_n: \mathbb{C}\setminus \lbrace{0,-1,-2, \cdots -(n-1)\rbrace} \rightarrow \mathbb{C}$ defined by

\begin{equation*}
    G_n(x):= \dfrac{e^{Kx}}{x}\prod_{s=1}^{n-1} \dfrac{e^{\frac{x}{s}}}{1+\frac{x}{s}}.
\end{equation*}
From this

\begin{equation*}
    G_n(x+1)= \dfrac{e^{Kx}\cdot e^K}{x+1}\cdot \prod_{s=1}^{n-1} \dfrac{e^{\frac{x}{s}}\cdot e^{\frac{1}{s}}}{1+\frac{x}{s+1}}\cdot \dfrac{s}{s+1}
\end{equation*}
We want to rewrite this as

\begin{equation*}
   G_n(x+1)= \dfrac{e^{Kx}\cdot e^K}{x+1}\cdot \prod_{s=1}^{n-1} \dfrac{e^{\frac{x}{s}}}{1+\frac{x}{s+1}}\cdot e^{\frac{1}{s}-\log \left(1+\frac{1}{s}\right)}.
\end{equation*}
Now using the well-known result that

\begin{equation*}
    \lim_{n \rightarrow \infty}\sum_{s=1}^{n-1} \left(\frac{1}{s}-\log \left(1+\frac{1}{s}\right)\right) =\gamma
\end{equation*}
where $\gamma$ is the Euler-Mascheroni constant, we know that the limit exists\footnote{This constant is discussed in the appendix in section \ref{subsec: gamma meets Gamma - Euler on the Euler-Mascheroni Constant}. There it is also shown that the limit exists.}. Define

\begin{equation*}
    \gamma_n = \dfrac{1}{1}+\dfrac{1}{2}+\cdots +\dfrac{1}{n}- \log n
\end{equation*}
such that also $\gamma = \lim_{n \rightarrow \infty} \gamma_n$. Then, we have

\begin{equation*}
    G_n(x)=\dfrac{e^{-\gamma_nx +\frac{x}{n}}}{x}\prod_{s=1}^{n-1}\dfrac{e^{\frac{x}{s}}}{1+\frac{x}{1+\frac{x}{s}}}
\end{equation*}
and the functional equation

\begin{equation*}
G_n(x+1)= x \cdot G_n(x) \cdot \dfrac{n}{n+x}.    
\end{equation*}
Therefore, we have:

\begin{theorem}
Let $\gamma$ be the Euler-Mascheroni constant. Further, let $\omega: \mathbb{C} \rightarrow \mathbb{C}$ be an arbitrary function satisfying $\omega(x+1)=\omega(x)$; then, the most general solution of the difference equation $F(x+1)=xF(x)$ is given by

\begin{equation*}
F(x) = \omega(x)\cdot \dfrac{e^{-\gamma x}}{x}\cdot \prod_{s=1}^{\infty} \dfrac{e^{\frac{x}{s}}}{1+\frac{x}{s}}.    
\end{equation*}
\end{theorem}

\subsection{Application to $\Gamma(x)$- The Weierstra\ss{} Product Representation}
\label{subsec: Application to Gamma(x)- The Weierstrass Product Representation}

Now that we found the most general solution of the difference equation $F(x+1)=xF(x)$, we want to descend to $\Gamma(x)$ from this. This is, e.g., possible by an application of the Bohr-Mollerup theorem.\\
Doing so, we will arrive at the following theorem

\begin{theorem}[Weierstra\ss{} Product Expansion of $\Gamma(x)$]
The $\Gamma$-function has the following product expansion:

\begin{equation*}
    \Gamma(x)= \dfrac{e^{-\gamma x}}{x}\cdot \prod_{s=1}^{\infty} \dfrac{e^{\frac{x}{s}}}{1+\frac{x}{s}}.
\end{equation*}
\end{theorem}
\begin{proof}
We need to check whether the three conditions in the  Bohr-Mollerup theorem are fulfilled. Therefore, let us check $\Gamma(1)=1$ first.\\[2mm]
We have

\begin{equation*}
    \Gamma(1)= e^{-\gamma} \cdot\prod_{s=1}^{\infty}\dfrac{e^{\frac{1}{s}}}{1+\frac{1}{s}}
\end{equation*}
or, by taking logarithms,

\begin{equation*}
    \log \Gamma(1) = - \gamma + \sum_{s+1}^{\infty} \left(\dfrac{1}{s}-\log \left(1+\dfrac{1}{s}\right)\right)
\end{equation*}
But, as we have seen above, the sum evaluates to $\gamma$, whence $\log \Gamma (1)=1$ or $\Gamma(1)=1$. Hence the first condition is satisfied.\\[2mm]
 The second condition of the Bohr-Mollerup Theorem, i.e. $\Gamma(x+1)=x\Gamma(x)$ is satisfied, since it solves the general difference equation $F(x+1)=xF(x)$.\\[2mm]
 Finally, let us check logarithmic convexity. Obviously, $\log \Gamma(x)$ is twice continuously differentiable, since the resulting sum converges uniformly. We find
 
 \begin{equation*}
     \dfrac{d}{dx}\log \Gamma(x) = -\gamma - \dfrac{1}{x} +\sum_{s=1}^{\infty}\left(\dfrac{1}{s}-\dfrac{1}{1+\frac{x}{s}}\cdot \frac{1}{s}\right)
 \end{equation*}
 and
 
 \begin{equation*}
     \dfrac{d^2}{dx^2} \log \Gamma(x)=\dfrac{1}{x^2}+\sum_{s=1}^{\infty} \dfrac{1}{s^2}\cdot \dfrac{1}{(1+\frac{x}{s})^2},
 \end{equation*}
 Therefore, obviously $ \dfrac{d^2}{dx^2} \log \Gamma(x) > 0 ~ \forall x>0$.\\[2mm]
 Hence the Bohr-Mollerup theorem applies and the function defined by the infinite product is indeed the familiar $\Gamma$-function.
\end{proof}

\subsection{Euler on Weierstra\ss's Condition}
\label{subsec: Euler on Weierstrass's Condition}

We want to show how Euler already arrived at the condition for the $\Gamma$-function that Weierstra\ss{} used to introduce it. Note that since we obtained that, if in the above theorem, we chose the periodic function $\omega(x)$ to be $=1$, $F(x)=\Gamma(x)$, we proved that:
\begin{theorem}

The $\Gamma$-function can also be defined by Weierstra\ss's conditions, i.e. the $\Gamma$ function is the unique meromorphic function satisfying
\begin{equation*}
    \Gamma(x+1)=x\Gamma(x)
\end{equation*}
and

\begin{equation*}
     \lim_{n \rightarrow \infty} \dfrac{\Gamma(x+n)}{(n-1)!n^x}=1. \quad \forall~x \in \mathbf{C}\setminus \left\lbrace 0,-1,-2, -3, \cdots\right\rbrace
\end{equation*}
\end{theorem}
We will  present how Euler obtained this condition in \cite{E652}. Weierstra\ss{} in 1856 in \cite{We56} attributed it to Gau\ss{} who introduced this condition in \cite{Ga28}. \\[2mm]
We will show it for $\Gamma(x+1)$, since in 1793 paper \cite{E652}\footnote{It was published after Euler's death in 1783.} Euler also did it for $x!$.
Euler's idea was to consider $n$ as a very large natural number and $x$ as a fixed finite natural number with $x \ll n$ and evaluate $\Gamma(x+n+1)$ in two ways. 
Using the functional equation $x$ times, we have

\begin{equation*}
    \Gamma(x+1+n)= (x+n)(x+n-1)\cdots (n+1)\Gamma(n+1).
\end{equation*}
But, since $x \ll n$, the finite parts added to $n$ in each factor can be ignored such that

\begin{equation*}
    \Gamma(x+1+n) \approx n^x \Gamma(n+1).
\end{equation*}
On the other hand, we can use the functional equation $n$ times, to find:

\begin{equation*}
    \Gamma(x+n+1) = (n+x)(n+x-1) \cdots (x+1)\Gamma(x+1).
\end{equation*}
Therefore, dividing both expressions expressions for $\Gamma(x+n+1)$, using $\Gamma(n+1)=n!$ and solving for $\Gamma(x+1)$, we arrive at the formula:

\begin{equation*}
    \Gamma(x+1) \approx \dfrac{n^x n!}{(x+1)\cdots (n+x-1)(n+x)} \quad \text{for} \quad x \ll n. 
\end{equation*}
In modern notation, this is  the  condition in the theorem. One just has to use the functional equation on $\Gamma(x+n)$ $n$ times to arrive at Euler's and Gau\ss's formula. Note that although the proof required $x$ and $n$ to be natural numbers, the right-hand side does not require $x$ to be a natural number. Therefore, it can also be used to interpolate $x!=\Gamma(x+1)$. Additionally, we point out again that Gau\ss{} \cite{Ga28} used this formula to introduce the $\Gamma$-function, but he did not motivate it at all.\\[2mm]
Euler's reasoning, using infinitely large numbers, is obviously not rigorous enough for modern times. But, in possession of the Bohr-Mollerup theorem, one could start from this expression and check, whether all conditions are satisfied or not. But since we already did it for the Weierstra\ss{} product and know this expression to be equivalent to it, we do not want to repeat this here. 

\newpage

\section{Euler's direct Solution of the Equation $\Gamma (x+1)=x\Gamma(x)$ - The Moment Ansatz}
\label{sec: Euler's direct Solution of the Equation Gamma (x+1)=xGamma(x) - The Moment-Ansatz}

We now go over to Euler's different approaches leading him to an explicit formula of the $\Gamma$-function. We will start with the "moment ansatz", a name that will become clear later. First, we want to explain briefly, where the method actually originated. 

\begin{figure}
\centering
    \includegraphics[scale=1.3]{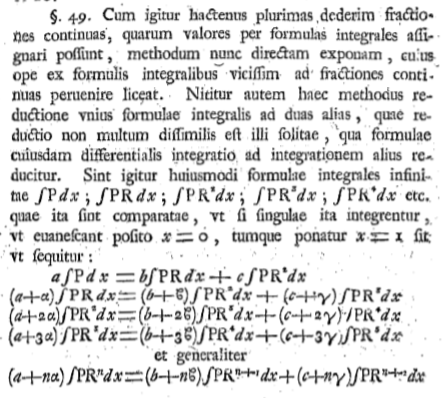}
    \caption{Moment Method}
     One page of Euler's paper \cite{E123}. Euler explains his idea how to solve the difference equation $(a +\alpha x)f(x)= (b +\beta x)f(x+1)+(c +\gamma x)f(x+2)$ by assuming it to be a certain integral ($\S$ 49). In the following paragraphs he applies it to continued fractions, which was his actual intention in that paper. 
     \end{figure}

\subsection{Origin of the Idea}
\label{subsec: Origin of the Idea}

Euler uses a technique, which we will refer here to as \textit{moment ansatz}, to solve difference equations of the kind:

\begin{equation*}
    (a +\alpha x)f(x)= (b +\beta x)f(x+1)+(c +\gamma x)f(x+2),
\end{equation*}
where $\alpha, \beta, \gamma ~ \in \mathbb{R}\setminus \lbrace 0 \rbrace$ and $a,b,c~ \in \mathbb{R}$\footnote{Euler, of course did not state the condition on $\alpha$, $\beta$, $\gamma$ explicitly. But the condition can be inferred from the following calculations in \cite{E123}.}, in his papers \cite{E123} and \cite{E594}. His actual intention was to derive continued fractions from this. For, dividing the above equation by $(a+\alpha x)$ and $f(x+1)$, one will find

\begin{equation*}
    \dfrac{f(x)}{f(x+1)}= \dfrac{b+ \beta x}{a +\alpha x}+ \dfrac{c + \gamma x}{a + \alpha x} \dfrac{f(x+2)}{f(x+1)}
\end{equation*}
or

\begin{equation*}
    \dfrac{f(x)}{f(x+1)}= \dfrac{b+ \beta x}{a +\alpha x}+ \dfrac{c + \gamma x}{a + \alpha x} \dfrac{1}{\frac{f(x+1)}{f(x+2)}}.
\end{equation*}
Replacing $x$ by $x+1$ one will get a similar equation for the quotient $\frac{f(x+1)}{f(x+2)}$ which can be inserted in the above equation. Repeating this procedure infinitely often, one will get a continued fraction for $\frac{f(x)}{f(x+1)}$.\\
Euler was interested in the continued fraction arising from this and he  tried to solve the difference equation. In the following, we will explain how he did this.

\begin{equation*}
    A(x)f(x)= B(x)f(x+1)+C(x)f(x+2)
\end{equation*}
with more general functions $A(x)$, $B(x)$, $C(x)$ etc. also leads to continued fractions. But Euler only considered the case in which those functions are linear functions in his papers \cite{E123} and \cite{E594}. Indeed, his investigations do not go beyond the case of linear functions in any of his papers.

\subsection{Euler's Idea}
\label{subsec: Euler's Idea}

Let us discuss his idea on the concrete example of the above difference equation. The generalisation to the general difference equation with linear coefficients is immediate. Euler  assumed that the solution is given as an integral of the form

\begin{equation*}
    \int\limits_{a}^{b}t^{x-1}P(t)dt,
\end{equation*}
whence we have to determine the limits of the integration and the function $P(t)$.  Euler, in modern formulation, assumed the solution of the difference equation to be the $x$-th moment of the function $P$. This is why we gave the method  the name moment ansatz. In order to do so, Euler considered the auxiliary equation

\begin{equation*}
    (a+\alpha x)\int\limits_{}^{t}t^{x-1}P(t)dt =  (b+\beta x)\int\limits_{}^{t}t^{x}P(t)dt + (c+\gamma x)\int\limits_{}^{t}t^{x+1}P(t)dt + t^{x}Q(t);
\end{equation*}
here, $\int\limits_{}^{t}$ is supposed to denote the indefinite integral over $t$ and $Q(t)$ is another function we have to determine; the use of this function will become clear in a moment.\\[2mm]
Euler then differentiated the auxiliary equation with respect to $t$:

\begin{equation*}
  (a+\alpha x)t^{x-1}P(t)= (b+\beta x)t^{x}P(t)+ (c +\gamma x)t^{x+1}P(t)  + xt^{x-1}Q(t)+t^xQ'(t).
\end{equation*}
Now divide by $t^{x-1}$:

\begin{equation*}
      (a+\alpha x)P(t)= (b+\beta x)tP(t)+ (c +\gamma x)t^{2}P(t)  + xQ(t)+tQ'(t).
\end{equation*}
Comparing the coefficients of the powers of $x$, we will get the following systems of coupled equations:

\begin{equation*}
    \renewcommand{\arraystretch}{1,5}
\setlength{\arraycolsep}{0.0mm}
\begin{array}{llllllll}
1. \quad    &aP(t) &~=~& btP(t)  & ~+~& ct^2P(t) & ~+~ &tQ'(t) \\
2. \quad     &\alpha P(t) & ~=~ & \beta tP(t) & ~+~& \gamma t^2P(t) & ~+~ &Q(t)
\end{array}
\end{equation*}
Solving both equations for $P$, we find

\begin{equation*}
    \renewcommand{\arraystretch}{2,5}
\setlength{\arraycolsep}{0.0mm}
\begin{array}{llllllll}
1. \quad    &P(t) &~=~& \dfrac{tQ'(t)}{a-bt-ct^2}\\
2. \quad     & P(t) & ~=~ & \dfrac{Q(t)}{\alpha -\beta t -\gamma t^2}
\end{array}
\end{equation*}
Therefore, we obtain the following equation for $Q(t)$

\begin{equation*}
    \dfrac{tQ'(t)}{Q(t)}= \dfrac{a-bt-ct^2}{\alpha - \beta t - \gamma t^2}
\end{equation*}
Although this equation can be solved in general, we will not do this here, because it will be more illustrative to consider examples. Anyhow, having found $Q(t)$, we can also find $P(t)$ substituting the value of $Q(t)$ in one of the above equations.\\[2mm]
Finally, we need the term $t^xQ(t)$ to vanish in the auxiliary equation. Hence the limits of integration are found from the solutions of the equation $t^xQ(t)=0$.

\subsubsection{Application to the $\Gamma$-function - Finding the Integral Representation}
\label{subsubsec: Application to the Gamma-function - Finding the Integral Representation}

As mentioned, everything becomes a lot clearer in certain examples. Therefore, let us consider the $\Gamma$-function, i.e. the functional equation $f(x+1)=xf(x)$. Euler considered the factorial explicitly in $\S 13$ of \cite{E594}, a paper published in 1785.
We make the ansatz

\begin{equation*}
    f(x) = \int\limits_{a}^{b}t^{x-1}P(t)dt.
\end{equation*}
Hence we need to determine $P(t)$ and the limits of integration $a$ and $b$. Let us introduce the auxiliary equation:

\begin{equation*}
    \int\limits_{}^{t}t^{x}P(t)dt = x \int\limits_{}^{t}t^{x-1}P(t)dt + t^x Q(t)
\end{equation*}
Differentiating with respect to $t$ gives

\begin{equation*}
    t^{x}P(t) = x t^{x-1}P(t)+ xt^{x-1}Q(t)+ t^xQ'(t).
\end{equation*}
    \begin{figure}
        \centering
        \includegraphics[scale=1.3]{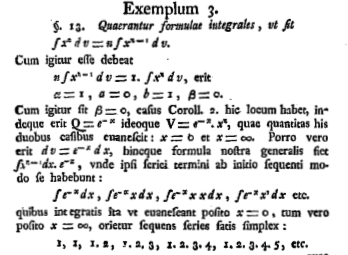}
        \caption{Euler finds the Integral Representation of $\Gamma(x)$ via the Moment Method} 
        This is one of the examples that Euler considered. It led him to the familiar integral representation of the the $\Gamma$-function. Euler called the sequence of the factorial "hypergeometric series". The scan is taken from \cite{E594}.
    \end{figure}
Division by $t^{x-1}$ gives

\begin{equation*}
    tP(t) = x P(t)+ xQ(t)+ tQ'(t).
\end{equation*}
Therefore, comparing the coefficients of $x$:

\begin{equation*}
    \renewcommand{\arraystretch}{1,5}
\setlength{\arraycolsep}{0.0mm}
\begin{array}{lrlllll}
1. \quad    &tP(t) &~=~&tQ'(t) \\
2. \quad     &0 & ~=~ & P(t) +Q(t).
\end{array}
\end{equation*}
Solving both equations for $P(t)$:

\begin{equation*}
    \renewcommand{\arraystretch}{1,5}
\setlength{\arraycolsep}{0.0mm}
\begin{array}{llllllll}
1. \quad    &P(t) &~=~&Q'(t) \\
2. \quad     &P(t) & ~=~ & - Q(t).
\end{array}    
\end{equation*}
Hence we obtain the following differential equation for $Q(t)$:

\begin{equation*}
    \dfrac{Q'(t)}{Q(t)}= -1.
\end{equation*}
This equation is easily integrated and gives

\begin{equation*}
    \log (Q(t)) =C - t \quad \text{or} \quad Q(t)=  Ce^{-t}, 
\end{equation*}
where $C \neq 0$ is an arbitrary constant of integration. From this $P$ is found to be

\begin{equation*}
   P(t)= -e^{-t}. 
\end{equation*}
Finally, we need to find the limits of integration. For this, we consider the equation $t^xQ(t)=Ct^{x}e^{-t}=0$. For $x>0$\footnote{Note that this is precisely the condition on $x$ we need for the integral to converge!} we find the two solution $t=0$ and $t = \infty$. Therefore, the term $t^xQ(t)$ in the auxiliary equation vanishes in these cases and we find:

\begin{equation*}
    C\int\limits_{0}^{\infty}t^xe^{-t}dt =x C\int\limits_{0}^{\infty}t^{x-1}e^{-t}dt.
\end{equation*}
In other words, the equation $f(x+1)=xf(x)$ is satisfied by:

\begin{equation*}
    f(x)=C\int\limits_{0}^{\infty}t^{x-1}e^{-t}dt.
\end{equation*}
This is, of course, almost the famous integral representation of the $\Gamma$-function. (There the constant $C$ is one.)\\

\subsubsection*{Finding the Integral Representation of $\Gamma(x)$}

We can force the function $f$ to be the $\Gamma$ by demanding it to satisfy all conditions of Wielandt's theorem. The condition $\Gamma(1)=1$ forces $C=1$. More precisely, we have the theorem:

\begin{theorem}[Integral Representation of $\Gamma(x)$]

The $\Gamma$-function is given by the following integral:

\begin{equation*}
    \Gamma(x) = \int_{0}^{\infty} t^{x-1}e^{-t}dt \quad \text{for} \quad \operatorname{Re}x>0.
\end{equation*}
\end{theorem}
\begin{proof}
 We have to check all conditions of  Wielandt's theorem\footnote{Therefore, at this point, we basically prove that the integral representation is indeed a correct definition for $\Gamma(x)$.}. First, find $\Gamma(1)$.

\begin{equation*}
    \Gamma(1)= \int_{0}^{\infty} e^{-t}dt = \left[-e^{t}\right]_0^{\infty} = 0 -(-1)=1.
\end{equation*}
Secondly, the functional equation is satisfied, as demonstrated in the last section.\\
Finally, we have to check holomorphy (which is obvious) and that $\Gamma(x)$ is bounded in the strip $S:= \lbrace x| 1 \leq x < 2\rbrace$. Hence consider

\begin{equation*}
    |\Gamma(x)| = \left|  \int_{0}^{\infty} t^{x-1}e^{-t}dt\right| \leq  \int_{0}^{\infty} \left| t^{x-1}\right|e^{-t} dt 
\end{equation*}
In other words, we have

\begin{equation*}
    |\Gamma(x)| \leq \operatorname{Re}(\Gamma(x)) \quad \text{for} \quad \operatorname{Re}x > 0.
\end{equation*}
For checking Wielandt's theorem we have to consider

\begin{equation*}
      \int_{0}^{\infty}  t^{x-1}e^{-t} dt \quad \text{for} \quad 1 \leq x <2. 
\end{equation*}
But these integrals are obviously bounded, whence Wieldlandt's theorem applies.
\end{proof}

\subsubsection{Some Remarks on the Ansatz}
\label{subsubsec: Some Remarks on the Ansatz}

We assume the solution to have the form $\int t^{x-1}P(t)dt$, which explains the name moment ansatz\footnote{A moment is defined as $M_n:=\int\limits_{a}^{b}t^n d\mu $, $\mu$ being some integration measure.}. But one can, of course, make other choices for the integrand. For the sake of an example, one can set $\int (R(t))^{x-1}$. By the same procedure, one would then arrive at the equation

\begin{equation*}
    \Gamma(x)= \int\limits_{0}^{1} \left(\log \dfrac{1}{t}\right)^{x-1}dt.
\end{equation*}
This was Euler's preferred integral representation and actually the first he found in 1738 in \cite{E19}. It follows from the representation just found by setting $e^{-t}=u$.\\[2mm]
Furthermore, one can even generalize the ansatz to

\begin{equation*}
    f(x) = \int R(t)^{x-1}P(t)dt,
\end{equation*}
where $R(t)$ is another function to be determined\footnote{It is indeed convenient to use this ansatz in the case of hypergeometric series, for example.}. Carrying out the procedure as above, one would arrive at certain conditions on the function $R(t)$ which are trivially satisfied by $R(t)=t$. Indeed, Euler tried this most general ansatz in his 1750 paper \cite{E123}, but realizing that $R(t)=t$ meets all requirements, he quickly focused on that special case. 

\subsection{Examples of other Equations which can be solved by this Method}
\label{subsec: Examples of other Equations which can be found by this Method}

Having found the integral representation of the $\Gamma$-function from its difference equation, let us apply Euler's method to more complicated but still familiar difference equations in order to find some interesting integral representations.

\subsubsection{1. Example: Legendre Polynomials}
\label{subsubsec: 1. Example: Legendre Polynomials}

The Legendre polynomials satisfy the following difference equation

\begin{equation*}
(n+1)P_{n+1}(x)=(2n+1)x P_{n}(x)-nP_{n-1}(x).
\end{equation*}
Together with the conditions $P_0(x)=1$ and $P_1(x)=x$, this difference equation determines them completely. The moment ansatz can be used to find an explicit formula for $P_n(x)$. More precisely, we have the theorem:

\begin{theorem}[Integral representation for the $n$-th Legendre Polynomial]
We have

\begin{equation*}
     P_n(x)=\dfrac{1}{ \log (-1)}\int\limits_{x-\sqrt{x^2-1}}^{x+\sqrt{x^2-1}} \dfrac{t^n}{\sqrt{1-2xt+t^2}}dt.
\end{equation*}
where the principal branch of $\log (-1)$ is to be taken.
\end{theorem}
\begin{proof}
This expression can be found by the moment ansatz. Since this is our first concrete example of a second order difference equation and one has to be more careful than in the case of the $\Gamma$-function, let us present the calculation in detail. We start with the auxiliary equation again which reads

\begin{equation*}
(n+1)\int\limits_{}^{t}t^nR(x,t)dt = (2n+1)x\int\limits_{}^{t}t^{n-1}R(x,t)dt - n\int\limits_{}^{t}t^{n-2}R(x,t)dt + Q(x,t)t^n.
\end{equation*}
We wish to find $R(x,t)$\footnote{Although we wrote $R(x,t)=R(t)$ instead of $R(t)$, this does not alter the procedure at all. It will just turn out that $R$ depends also on $x$ which is to be considered as a parameter in the difference equation. The same goes for $Q$.} and $Q(x,t)$ and the limits of integration. 
Let us differentiate that equation with respect to $t$, we find:

\begin{equation*}
(n+1)t^nR(x,t)=(2n+1)xt^{n-1}R(x,t)-nt^{n-2}R(x,t)+ {Q'}(x,t)t^n+nt^{n-1}Q(x,t).
\end{equation*}
Dividing by $t^{n-2}$ and comparing the coefficients of the powers of $n$, we obtain the following system of equations

\begin{equation*}
\begin{array}{lll}
1. \quad t^2R(x,t) & = & 2tR(x,t)x -R(x,t) + t Q(x,t) \\
2.\quad t^2R(x,t) & = & tR(x,t)x + Q'(x,t)t^2
\end{array}
\end{equation*}
Solving both for $R(x,t)$

\begin{equation*}
 \renewcommand{\arraystretch}{2,5}
\setlength{\arraycolsep}{0.0mm}
\begin{array}{lll}
1. \quad R(x,t) & = & \dfrac{tQ(x,t)}{t^2-2tx+1}\\
2. \quad R(x,t) & = & \dfrac{t^2 {Q'}(x,t)}{t^2-xt} 
\end{array}
\end{equation*}
Therefore,

\begin{equation*}
\dfrac{tQ(t)}{t^2-2xt+1}= \dfrac{t^2 {Q'}(x,t)}{t^2-xt}
\end{equation*}
whence we find

\begin{equation*}
Q(t)= C(x)\sqrt{t^2-2xt+1}.
\end{equation*}
$C(x)$ being an arbitrary function of $x$ that entered via integration with respect to $t$. Therefore, 

\begin{equation*}
R(x,t)= C(x)\dfrac{t}{\sqrt{1-2xt+t^2}}.
\end{equation*}
Integrating the differentiated auxiliary equation again from $a$ to $b$, we would have

\begin{equation*}
P_n(x) = C(x)\int\limits_{a}^{b} \dfrac{t^n}{\sqrt{1-2xt+t^2}}dt
\end{equation*}
if we determine $a$ and $b$ in such a way that $Q(x,t)t^n$ vanishes for $a$ and $b$ for all $n$. Since $t^0=1$ has no zeros, we have to put $Q(x,t)=0$. This gives 

\begin{equation*}
a= x-\sqrt{x^2-1} \quad \text{and} \quad b = x+\sqrt{x^2-1}.
\end{equation*}
Therefore, it remains to find $C(x)$. For this we use the special case $P_0(x)=1$. We calculate

\begin{small}
\begin{equation*}
 C(x)\int\limits_{x-\sqrt{x^2-1}}^{x+\sqrt{x^2-1}} \dfrac{t^0}{\sqrt{1-2xt+t^2}}dt= C(x)\left[\log\left(\sqrt{1-2xt+t^2}+t-x\right)\right]_{x-\sqrt{x^2-1}}^{x+\sqrt{x^2-1}}
\end{equation*}
\end{small} Therefore, 

\begin{equation*}
C(x)= \dfrac{1}{\log(-1)},
\end{equation*}
where the principal branch of the logarithm is to be taken, of course. The explicit formula for the $n$-th Legendre polynomial hence reads

\begin{equation*}
P_n(x)=\dfrac{1}{\log(-1)}\int\limits_{x-\sqrt{x^2-1}}^{x+\sqrt{x^2-1}} \dfrac{t^n}{\sqrt{1-2xt+t^2}}dt.
\end{equation*}
It is easily checked that the explicit formula also gives $P_1(x)=x$. Therefore, both initial conditions and the functional equations are satisfied and hence the above formula gives the $n-$th Legendre polynomial.\\
The formula for the $n$-th Legendre polynomial seems to be ambiguous, but one always arrives at the same value for $P_n(x)$ as long as the same branch of the logarithm is taken.
\end{proof}

\subsubsection{Historical Remark on Legendre Polynomials}
\label{subsubsec: Historical Remark on Legendre Polynomials}

The Legendre polynomials were named after Legendre because of his 1785 paper \cite{Le85}; he discovered them in his investigations on the gravitational potential. Nowadays, they are important in electrodynamics, more precisely, the multipole expansion.  But they were in fact already discovered by Euler in his 1783 paper \cite{E551} in a completely different context. In his 1782 paper \cite{E606}, Euler even gave the explicit formula for the $n-$th Legendre polynomial we derived above. But since in that work he was mainly interested in continued fractions for the quotients of two consecutive integrals, he did not find the constant $C(x)=\frac{1}{\log (-1)}$.

\begin{center}
\begin{figure}
\centering
    \includegraphics[scale=1.0]{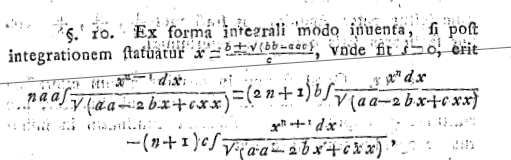}
    \caption{Euler and the Legendre Polynomials}
     Taken from \cite{E606}. Euler solved the difference equation in terms of an integral. He was interested in the continued fraction of two consecutive Legendre polynomials. Thus, he did not find any of their more interesting properties, e.g., their orthogonality relation. This had to wait until Legendre's paper \cite{Le85}.
    \end{figure}
\end{center}
Further, it seems that he did not notice the connection between the findings of \cite{E551} published in 1783  and \cite{E606} published in 1782. In other words, he was not aware that he already obtained an explicit formula for $P_n(x)$. This is even more interesting, because in \cite{E551} he said that it is not possible for him to find such an explicit formula, although he provided all necessary tools in his earlier papers in 1750 in \cite{E123} and in 1785 in \cite{E594}\footnote{Although \cite{E594} was published later than \cite{E606} and \cite{E551}, they were chronologically written according to their Eneström numbers, i.e. \cite{E551} was written first, \cite{E606} last.}. In summary, it seems that Euler was not aware that in those papers he basically discovered a general method to find a particular solution of the general homogeneous difference equation with linear coefficients.

\subsubsection{2. Example: Hermite Polynomials}
\label{subsubsec: 2. Example: Hermite Polynomials}

The Hermite polynomials satisfy the recurrence relation

\begin{equation*}
    H_{n+1}(x)= 2x H_n(x)-2nH_{n-1}(x)
\end{equation*}
with the additional conditions $H_0(x)=1$, $H_1(x)=2x$. We then have the formula

\begin{theorem}[Explicit Formula for the $n$-th Hermite Polynomial]
The following formula holds:

\begin{equation*}
   H_n(x)= \dfrac{i^n e^{x^2}}{2\sqrt{\pi}} \int\limits_{-\infty}^{\infty} t^n e^{\frac{1}{2}\left(-\frac{t^2}{2}-2xit\right)}dt
\end{equation*}
\end{theorem}
\begin{proof}
We use the moment ansatz. We will not carry out the calculation since it is similar to the case of the Legendre polynomials. We will only state the intermediate results.\\
Of course, we start from the auxiliary equation:

\begin{small}
\begin{equation*}
    \int\limits_{}^{t}t^{n+1}P(t)dt = 2x \int\limits_{}^{t}t^{n}P(t)dt -2n \int\limits_{}^{t}t^{n-1}P(t)dt + t^nQ(t).
\end{equation*}
From this we derive the following equations for $P(t)$ and $Q(t)$:

\begin{equation*}
       \renewcommand{\arraystretch}{2,5}
\setlength{\arraycolsep}{0.0mm}
\begin{array}{llllllllll}
1. \quad     &P(t) &~=~& \dfrac{Q'(t)t}{t^2-2xt}  \\
2. \quad     &P(t) &~=~& \dfrac{Q(t)}{2},  
\end{array}
\end{equation*}
\end{small}whence

\begin{equation*}
    Q(t) = C(x) e^{\frac{1}{2}\left(\frac{t^2}{2}-2xt\right)} \quad \text{and} \quad P(t)= \dfrac{1}{2} C(x) e^{\frac{1}{2}\left(\frac{t^2}{2}-2xt\right)}
\end{equation*}
In order to find the limits of integration, we need to solve $t^n e^{\frac{1}{2}\left(\frac{t^2}{2}-2xt\right)} =0$, which leads to $t=\pm  i \infty$, if we want $n$ to be an arbitrary integer number. Therefore, up to this point we have:

\begin{equation*}
    H_n(x)=\dfrac{C(x)}{2}\int\limits_{-i\infty}^{i\infty}t^n e^{\frac{1}{2}\left(\frac{t^2}{2}-2xt\right)}dt.
\end{equation*}
It is convenient to get rid of the imaginary limits by the substitution $t = iy$. This gives

\begin{equation*}
    H_n(x)= \dfrac{C(x)}{2}i^{n+1} \int\limits_{-\infty}^{\infty}y^n e^{\frac{1}{2}\left(-\frac{y^2}{2}-2xiy\right)}dy.
\end{equation*}
From the initial condition $H_0(x)=1$ we find

\begin{equation*}
    \dfrac{C(x)}{2} = \dfrac{e^{x^2}}{2i\sqrt{\pi}}.
\end{equation*}
Therefore, we arrive at

\begin{equation*}
    H_n(x)= \dfrac{e^{x^2}i^n}{2 \sqrt{\pi}} \int\limits_{-\infty}^{\infty} t^n e^{\frac{1}{2}\left(-\frac{t^2}{2}-2itx\right)}dt.
\end{equation*}
The condition $H_1(x)=2x$ is easily checked to be satisfied by the explicit formula. It is obvious that one can find similar explicit formulas for other orthogonal polynomials defined by second order homogeneous difference equations with linear coefficients, like, e.g., the Laguerre and Chebyshev polynomials. 
\end{proof}

\subsubsection{3. Example: Beta Function}
\label{subsubsec: 3. Example: Beta-Function}

Let us consider the $B$-function, which is defined as 

\begin{definition}[$B$-function]
The $B$-function, also referred to as Eulerian integral of the first kind, is defined as:

\begin{equation*}
    B(x,y) := \int\limits_{0}^{1} t^{x-1}(1-t)^{y-1}dt \quad \text{for} \quad \operatorname{Re}x, \operatorname{Re}y > 0
\end{equation*}
\end{definition}
From its definition it is immediate that $B$ satisfies the functional equation 

\begin{equation*}
    B(x+1,y)=\dfrac{x}{x+y}B(x,y) \quad (*).
\end{equation*}
And one could start from this functional equation to obtain the integral representation via the moment ansatz. Euler did this in 1785 in $\S$ 17 of \cite{E594}\footnote{He even considered a slightly more general example.}. Having already given several examples of this method, we do not want to do this here. \\
Here we want to use the results obtained up to this point to show:

\begin{theorem}
We have

\begin{equation*}
    B(x,y)= \dfrac{\Gamma(x)\Gamma(y)}{\Gamma(x+y)}.
\end{equation*}
\end{theorem}
\begin{proof}
We start from the functional equation $(*)$, of course. Further, we assume that $B$ can be written as product of two functions $B_1$, $B_2$. We demand those to satisfy the equations:

\begin{equation*}
    B_1(x+1,y)=xB_1(x,y) \quad \text{and} \quad (x+y)B_2(x+1,y)=B_2(x,y).
\end{equation*}
For the sake of brevity, we will drop the $y$ in the argument in the following; we consider the functional equation only in $x$, and $y$ can be seen as a parameter.\\
It is easily seen that $B_1(x)\cdot B_2(x)$ satisfies the functional equation for $B(x,y)$. Therefore, we need to solve the equations for $B_1$ and $B_2$ to find an expression for $B(x,y)$.\\[2mm]
Let us consider $B_1$ first. It satisfies the functional equation of the $\Gamma$-function in $x$. Therefore, we immediately have
\begin{equation*}
    B_1(x)= C_1(y)\Gamma(x)\omega_1(x),
\end{equation*}
$C_1(y)$ being an arbitrary function of $y$, $\omega_1(x)$ being a  periodic function with period $+1$.\\
To solve the functional equation for $B_2(x)$, let us introduce $D(x)=\frac{1}{B_2(x)}$. Then, $D(x)$ satisfies the functional equation:

\begin{equation*}
    D(x+1)=(x+y)D(x).
\end{equation*}
This equation is easily seen to be solved by

\begin{equation*}
    D(x)= C_2(y)\Gamma(x+y)\cdot\omega_2(x).
\end{equation*}
$C_2(y)$ is an arbitrary function of $y$, $\omega_2(x)$ being a periodic function of period $1$. And hence

\begin{equation*}
    B_2(x)=\dfrac{1}{C_2(y) \Gamma(x+y)}.
\end{equation*}
Combining the results, we have found

\begin{equation*}
    B(x,y)= \omega_1(x)\omega_2(x)C(y)\dfrac{\Gamma(x)}{\Gamma(x+y)},
\end{equation*}
where $C(y)=\frac{C_1(y)}{C_2(y)}$.\\[2mm]
We can omit the periodic factor in this solution, since otherwise we would have:

\begin{equation*}
    B(x,y)= \int\limits_{0}^{1}dtt^{x-1}(1-t)^{y-1}= \Omega_1(x)\dfrac{\Gamma(x)\Gamma(y)}{\Gamma(x+y)},
\end{equation*}
where $\Omega_1(x)$ is a function with period $1$.
But $\Omega_1(x)$ can be found from the special case $B(x,1)$. For, in this case we have on the one hand

\begin{equation*}
    B(x,1)= \int\limits_{0}^{1}dtt^{x-1}=\dfrac{1}{x}.
\end{equation*}
But on the other hand

\begin{equation*}
    B(x,1) = \dfrac{\Gamma(x)}{\Gamma(x+1)}\Omega_1(x)=\dfrac{\Omega_1(x)}{x}.
\end{equation*}
Here we used the functional equation of the $\Gamma$-function and $\Gamma(1)=1$. This already implies $\Omega_1(x)=1$ for all $x$. Hence the periodic function is simply $=1$.\\
It remains to determine the function $C(y)$. From the definition of $B(x,y)$ we find:

\begin{equation*}
    B(1,y) = \int\limits_{0}^{1} (1-t)^{y-1}dt = \dfrac{1}{y}.
\end{equation*}
First, from our solution we find

\begin{equation*}
    B(1,y)= C(y)\dfrac{\Gamma(1)}{\Gamma(y+1)}=C(y)\dfrac{1}{y\cdot \Gamma(y)},
\end{equation*}
where we used $\Gamma(1)=1$ and $\Gamma(y+1)=y\Gamma(y)$. Therefore, $C(y)$ must satisfy:

\begin{equation*}
    C(y)\dfrac{1}{y\Gamma(y)}=\dfrac{1}{y} \quad \text{or} \quad C(y)=\dfrac{1}{\Gamma(y)}.
\end{equation*}
Therefore, we finally arrived at the formula:

\begin{equation*}
    B(x,y)=\dfrac{\Gamma(x)\Gamma(y)}{\Gamma(x+y)}.
\end{equation*}
\end{proof}

\subsubsection{Some Remarks on the Relation of $\Gamma$ and $B$.}
\label{subsubsec: Some Remarks}

The relation among the $B$- and $\Gamma$-function was already discovered by Euler essentially in 1738 in \cite{E19}\footnote{We say "essentially" here, because Euler did not state it explicitly in that paper.}. He states it explicitly, e.g., in 1772 in \cite{E421}. But the argument he gave there is not  a rigorous proof, as we mentioned above in section \ref{subsec: Overview on Euler's Contributions},  he only proved the formula for integers $x$ and $y$ and then, without any further explanation, replaced the factorials by the integral representation of $\Gamma(x)$.\\
Rigorous proofs were first given by Jacobi in 1834 in \cite{Ja34} and Dirichlet in 1839 in \cite{Di39}, but they both use the theory of double integrals, which Euler did not know.  We will discuss Euler's argument, Dirichlet's and Jacobi's proof below in section \ref{subsubsec: Euler's Proof}, in section \ref{subsubsec: Dirichlet's Proof} and in section \ref{subsubsec: Jacobi's Proof}, respectively.\\[2mm]
 Our reasoning to obtain this fundamental relation  does not require double integrals and our proof certainly was within Euler's grasp. In \cite{E594}, he even considered similar questions, but never made the connection to the $B$- and $\Gamma$-functions.\\[2mm]

\subsubsection{4. Example: Hypergeometric Series}
\label{subsubsec: 4. Example: Hypergeometric Series}

Finally, let us mention the hypergeometric series. It was first defined by Euler in \cite{E710}, a paper  published only in 1801,

\begin{definition}[Hypergeometric Series]
For $a,b,c \in \mathbb{C}\setminus \lbrace{0,-1,-2, -3, \cdots \rbrace}$ the hypergeometric series is defined by

\begin{equation*}
    _2F_1(a,b,c;z) = 1 +\dfrac{ab}{c}\dfrac{z}{1!}+\dfrac{a(a+1)b(b+1)}{c(c+1)}\dfrac{z^2}{2!}+ \cdots \quad \text{for} \quad |z|<1.
\end{equation*}
\end{definition}
We will drop the subscripts $2$ and $1$, and write simply $F$, if there is no chance for confusion.\\
The first systematic study was done by Gau\ss{} in 1812 in \cite{Ga28}, whence the above series is often referred to as Gaussian hypergeometric series. Many people contributed to the nowadays highly developed theory of this function. We mention Kummer \cite{Ku36} and Riemann \cite{Ri57} as some of the contributors. The Gau\ss{}ian hypergeometric series has been generalized in several ways. The number of parameters in the coefficients has been increased, leading to the Tomae functions $_pF_q$, see, e.g., \cite{Sl09}. The number of variables has been increased, e.g., leading to Lauricella functions, which are also discussed in \cite{Sl09}. Furthermore, we want to mention the GKZ systems and the modern text \cite{Yo97}. Here we are mainly interested in hypergeometric integrals. A modern treatise  on the subject is \cite{Ao11}. What is of interest for us is that one can derive the  integral representation of the hypergeometric series which is usually attributed to Euler\footnote{We have to say some things about that later.} from a certain difference equation which is satisfied by the the hypergeometric series. Gau\ss{} in his paper called them contiguity relations and gave a complete list of 15 of such equations. We will need one equation that follows from those he gave in the mentioned paper. 

\begin{theorem}
We have the following equation:
\begin{equation*}
 \renewcommand{\arraystretch}{2,5}
\setlength{\arraycolsep}{0.0mm}
\begin{array}{lll}
    B(b+2,c-b){}F(a,b+2,c+2;x)&~=~& \left(\dfrac{b}{x(a-c-1)}\right)
    B(b,c-b)F(a,b,c;x) \\
    &~+~ &\left(\dfrac{(b-a+1)x+c}{x(c-a+1)}\right)B(b+1,c-b)F(a,b+1,c+1;x)
    \end{array}
\end{equation*}
\end{theorem}
The proof is simply done by expanding each hypergeometric function into a power series and comparing coefficients. We will not do this here\footnote{Below, in section \ref{subsec: Modern Idea - Intersection Theory}, we will arrive precisely at this relation starting from the integral representation.}. We will consider the above equation as an equation in $b$. Note that by dividing both sides by $B(b,c-b)$ and applying the relation to the $\Gamma$-function and its functional equation, the coefficients become linear functions in $b$. Then, it is a homogeneous difference equation with linear coefficients: $a,c,x$ are considered as parameters. Thus, we can solve this equation by the moment ansatz. Indeed, proceeding as in the previous cases (with the condition $F(a,0,c;x)=1$), after a long and tedious calculation we arrive at:

\begin{equation*}
    F(a,b,c;x)= \dfrac{\Gamma(c)}{\Gamma(a)\Gamma(c-b)}\int\limits_{0}^{1}t^{b-1}(1-t)^{c-b-1}(1-xt)^{-a}dt.
\end{equation*}
This is the  Eulerian integral representation of the hypergeometric series.\\[2mm]

Finally, let us mention one drawback of the moment ansatz. In Gau\ss's paper one also finds contiguous relations, relating $F(a,b,c;x)$, $F(a+1,b,c;x)$ and $F(a-1,b,c;x)$ (see, e.g. equation [1] in $\S 7$ of \cite{Ga28}). The coefficients are also linear in $a$. But in this case the moment ansatz does not produce a nice solution, if one just uses the ansatz

\begin{equation*}
    \int t^{a-1}P(t)dt,
\end{equation*}
since $P(t)$, as we have seen, also  depends on $a$. 

\subsubsection{Historical Note on the Integral Representation}
\label{subsubsec: Historical Note on the Integral Representation}

We wish to make some remarks on the origin of the the integral representation of the hypergeometric function. It is often ascribed to Euler that he found this representation, see, e.g., \cite{An10}, but it is actually not that simple. What can be said for certain is that we do not find the above equality explicitly in any of Euler's publications. Therefore, let us briefly discuss, what Euler actually did. \\[2mm]
First, as we already mentioned in the previous section, Euler  studied the hypergeometric series in its most general form and  defined it as the power series above in \cite{E710}, a paper written in 1778, but just published in 1801. \\[2mm]
He proceeded to find the differential equation satisfied by it and found a transformation, now bearing his name, of the hypergeometric series in that paper. But he did NOT state the above integral representation anywhere. Nevertheless, on several instances, he did more general investigations, from which the formula would easily follow. Those investigations were mainly concerned with differential equations. We mention his 1763 paper \cite{E274}\footnote{Unfortunately, the second half of his paper is lost. But reading the first few paragraphs, it is clear that Euler's investigations would have led him to a formula containing the integral representation of the hypergeometric series as a special case} and especially chapter 12 of his second book on integral calculus \cite{E366} published in 1769. In both works, he derived differential equations for parameter integrals, depending on one ore more variables, by differentiating them under the integral sign.\\[2mm]
His paper on the hypergeometric series, \cite{E710}, was written later, but nevertheless he did not make the connection to his earlier investigations. In conclusion, Euler could have written down the above equation, but he did not do so in any of his works.\\
Therefore, let us turn to the people, who actually stated the integral representation. The first to write down the integral representation appears to be Legendre in 1817 in \cite{Le17}, although Abel in 1827 in \cite{Ab27} is often credited for it, confer, e.g., \cite{Ni05}\footnote{In some sense this is true, since Abel was the first to prove that uniformly convergent series can be integrated term-by term.}. Kummer also found it in 1837 in \cite{Ku37a} and \cite{Ku37b}. In those papers, he gave a general method to convert certain integrals into series and vice versa. \\
Having mentioned all this, it seems to be up to personal preference to call the integral formula the Eulerian integral representation or not.

\begin{figure}
\centering
    \includegraphics[scale=0.8]{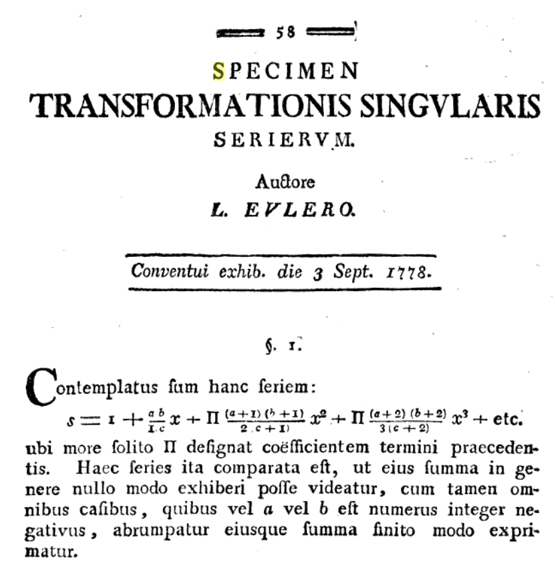}
     \caption{Euler and the hypergeometric Series}
     Definition of the hypergeometric series by Euler in \cite{E710}.
    \end{figure}

\begin{figure}
\centering
    \includegraphics[scale=0.8]{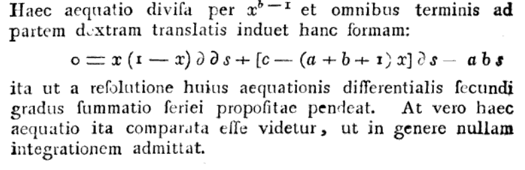}
    \caption{Euler and the hypergeometric Differential Equation}
    Euler finds the differential equation for the hypergeometric series in \cite{E710}.
    \end{figure}

\begin{figure}
\centering
    \includegraphics[scale=0.9]{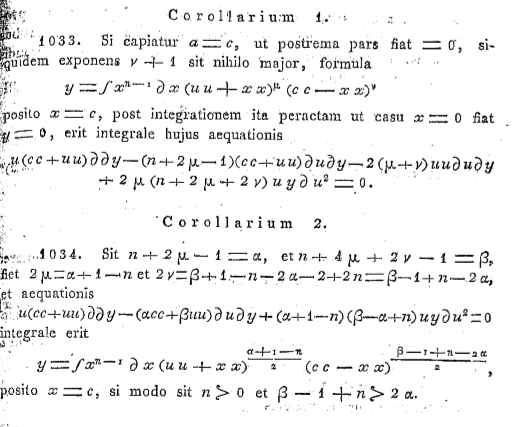}
    \caption{Integral Representation of the hypergeometric Series}
     Euler got close to the integral representation of the hypergeometric series, taken from chapter 10 of \cite{E366}. For a suitable choice of variables and after some simple substitutions, one arrives at the desired formula. Euler never related this result to the findings in \cite{E710}, the paper containing the first ever definition of the hypergeometric series as a power series.
     \end{figure}

\subsection{Euler and the Mellin-Transform}
\label{subsec: Euler and the Mellin Transform}

\subsubsection{Definition of the Mellin-Transform and inverse Mellin-Transform}

Euler's ansatz

\begin{equation*}
    F(x) = \int\limits_{a}^{b}t^{x-1}P(t)dt
\end{equation*}
for the solution of a homogeneous difference equation with linear coefficients bears quite a resemblance to the Mellin-Transform of a function $f(x)$. The Mellin-Transform $M(f)(s)$ is defined as

\begin{equation*}
    M(f)(s):= \int\limits_{0}^{\infty} t^{s-1}f(t)dt,
\end{equation*}
provided the integral exists\footnote{We will give the precise definition in a moment.}.
Although in Euler's approach we had to calculate the limits of integration, we will argue that Euler's method can be considered as a precursor of the Mellin-Transform, a concept introduced by Mellin in his 1895 paper \cite{Me95}. To this end, let us introduce the Mellin-Transform properly.

\begin{definition}[Mellin-Transform]
Let $t^{s-1}f(t) \in L^1(0,\infty)$, then the Mellin-Transform of the function $f$ is defined as:

\begin{equation*}
    F(s):= \int\limits_{0}^{\infty}t^{s-1}f(t)dt.
\end{equation*}
\end{definition}
The inverse Mellin-Transform allows to recover the function $f$ from its Mellin-Transform. We have the

\begin{theorem}[Inverse Mellin-Transform]
If $F(s)=\int\limits_{0}^{\infty}t^{s-1}f(t)dt$ exist, then we have

\begin{equation*}
    f(t)= \dfrac{1}{2\pi i} \int\limits_{c -i \infty}^{c+\infty} F(s)(t)^{-s}ds
\end{equation*}
with $c>0$, provided the integral exists.
\end{theorem}
We omit the proof here and refer the interested reader, e.g., to \cite{Ti48}. We will not use this formula, but just stated it for the sake of completeness. Obviously, one can apply the inverse Mellin-Transform to find the function $P(t)$ in his ansatz. Thus, Euler's method to find the function $P(t)$ is also a method to find the inverse Mellin-Transform.\\[2mm]
In this section, we want to show that  one can also start from the differential equation satisfied by $P(t)$ and using Euler's ansatz one can recover the difference equation. We want to do this in Eulerian fashion and consider several examples of increasing complexity. \\
Before doing so, let us state the following theorem in advance:

\begin{theorem}
Let $P(t)t^{s-1-n} \in L^1(0,\infty)$, $P \in \mathcal{C}^n(\mathbb{R})$ and $n \in \mathbb{N}$ a fixed natural number. If

\begin{equation*}
    F(s)= \int\limits_{0}^{\infty}t^{s-1}P(t)dt,
\end{equation*}
then
\begin{equation*}
   \int\limits_{0}^{\infty}t^{s-1}P^{(n)}(t)dt= (-1)(s-1)(s-2)\cdots (s-n)F(s-n).
\end{equation*}
\end{theorem}
The proof is by induction and integration by parts.

\subsubsection{1. Example: $\Gamma$-function}

Let us start with the most familiar example of a Mellin-Transfrom. We want to solve the equation

\begin{equation*}
    \Gamma(x+1)=x\Gamma(x)  \quad \text{with} \quad \Gamma(1)=1 \quad \text{for} \quad x>0
\end{equation*}
again. To this end, we start from the differential equation

\begin{equation*}
    P(t)=-P'(t)
\end{equation*}
at which we arrived using Euler's method in section \ref{subsubsec:  Application to the Gamma-function - Finding the Integral Representation}. Thus, multiplying this equation by $t^x$ and integrating from $0$ to infinity, formally

\begin{equation*}
    \int\limits_{0}^{\infty}t^xP'(t)dt =- \int\limits_{0}^{\infty}t^xP(t)dt,
\end{equation*}
whence by theorem 2.5.7

\begin{equation*}
    -x\int\limits_{0}^{\infty}t^{x-1}P(t)dt= - \int\limits_{0}^{\infty}t^xP(t)dt.
\end{equation*}
Therefore, calling $\Gamma(x)=\int\limits_{0}^{\infty}t^{x-1}P(t)dt$, we recover the equation

\begin{equation*}
    \Gamma(x+1)=x\Gamma(x).
\end{equation*}
Since the differential equation for $P$ is solved by $P(t)=Ce^{-t}$ with $C\neq 0$, as we also saw above, using the initial condition $\Gamma(1)=1$, we arrive at the integral representation of the $\Gamma$-function again.

\subsubsection{Remark}

Therefore, we used the differential equation, we found for $P(t)$ via the moment ansatz to reconstruct the difference equation we wanted to solve.\\[2mm]
Generally speaking, we used that the Mellin-Transform converts a differential equation into a difference equation. Thus, vice versa the inverse Mellin-Transform converts a difference equation into a differential equation, which turned out to be the differential equation equation that we found using Euler's method. The latter even works for general limits of integration as we will see in the following examples, whereas the Mellin-Transform is actually just defined for  the limits $0$ and $\infty$.

\subsubsection{2. Example: Generalized $\Gamma$-function}
\label{subsubsec: 2. Example: Generalized Gamma function}

We now want to solve the difference equation
\begin{equation*}
    f(x+1)= (\alpha x +\beta)f(x) \quad \text{with}, \quad f(1)=1, \alpha > 0, \beta > 0 \quad \text{for} \quad x>0.
\end{equation*}
One can use Euler's method again to solve this difference equation or simply guess the solution. Either way, the solution is given by

\begin{equation*}
    f(x) = C\int\limits_{0}^{\infty} t^{x-1+\frac{\beta}{\alpha}}e^{-\frac{t}{\alpha}}dt. 
\end{equation*}
Thus, one finds that $P(t)= Ct^{\frac{\beta}{\alpha}}e^{-\frac{t}{\alpha}}$ in the moment ansatz. The constant $C$ is not important at the moment. We will determine it later in section \ref{subsec: Generalized Factorials}. $P$ satisfies the following the differential equation\footnote{We do not mention the initial condition, since we do not care about the constant $C$ here.}:

\begin{equation*}
    P'(t)\cdot t \alpha = \beta P(t) -t P(t). 
\end{equation*}
Now assume that we did not know about the origin of that equation and we wanted to derive the corresponding difference equation. To this end, we multiply both sides by $t^{x-1}$ and integrate from $0$ to $\infty$:

\begin{equation*}
    \alpha \int\limits_{0}^{\infty} t^xP'(t)dt = \beta  \int\limits_{0}^{\infty} t^{x-1}P(t)dt -  \int\limits_{0}^{\infty}t^xP(t)dt.
\end{equation*}
By theorem 2.5.7, this simplifies to

\begin{equation*}
    - x \alpha  \int\limits_{0}^{\infty} t^{x-1}P(t)dt = \beta  \int\limits_{0}^{\infty} t^{x-1}P(t)dt-  \int\limits_{0}^{\infty} t^xP(t)dt.
\end{equation*}
Introducing $f(x)=  \int\limits_{0}^{\infty} t^{x-1}P(t)dt$, the equation reads.

\begin{equation*}
    -x \alpha f(x)= \beta f(x)- \alpha f(x+1)
\end{equation*}
or, equivalently

\begin{equation*}
    f(x+1)=(\alpha x+ \beta)f(x),
\end{equation*}
which is the initial difference equation again.\\[2mm]
In the next example we want to see how one can derive the limits of integration, if they are not known in advance.

\subsubsection{3. Example: Hermite Polynomials}

We start from the differential equation

\begin{equation*}
    P(t)(t-2x)=2P'(t),
\end{equation*}
 multiply it by $t^n$, and integrate over $t$ from $a$ and $b$:

\begin{equation*}
     \int\limits_{a}^{b}P(t)t^{n+1}dt -  2x \int\limits_{a}^{b} P(t)t^n dt = 2  \int\limits_{a}^{b} t^n P'(t)dt.
\end{equation*}
Since theorem 2.5.7 only holds for $a=0$ and $b=\infty$, we have to keep the absolute part while integrating the right-hand side by parts:

\begin{equation*}
    \int\limits_{a}^{b}P(t)t^{n+1}dt -  2x \int\limits_{a}^{b} P(t)t^n dt =  \left. 2t^n P(t)\right|_{a}^{b} -2n \int\limits_{a}^{b} t^{n-1}P(t)dt.
\end{equation*}
Introducing the notation $H_n(x) = \int\limits_{a}^{b}P(t)t^{n+1}dt$, we have:

\begin{equation*}
    H_{n+1}(x)-2xH_n(x) =  \left. 2t^n P(t)\right|_{a}^{b} -2nH_{n-1}(x).
\end{equation*}
Therefore, to recover the difference equation for the Hermite polynomials, see section \ref{subsubsec: 2. Example: Hermite Polynomials}, it has to be:

\begin{equation*}
    0 =  \left. 2t^n P(t)\right|_{a}^{b} \quad \text{or} \quad 0=b^nP(b)-a^nP(a) \quad n \in \mathbb{N}.
\end{equation*}
Furthermore, we require $a \neq b$, since this solution is trivial. As we also saw in \ref{subsubsec: 2. Example: Hermite Polynomials}, $P(t)$ is given by

\begin{equation*}
    P(t)= C(x) e^{\frac{t^2}{4}-tx}.
\end{equation*}
$C(x)$ is just the integration constant. Therefore, a solution of the  above equation is given by

\begin{equation*}
    (a,b)= (-i \infty, +i \infty),
\end{equation*}
whence, proceeding with this result, we recover the solution found in \ref{subsubsec: 2. Example: Hermite Polynomials}.\\[2mm]
Therefore, we actually do not need Euler's auxiliary function $Q(t)$ to find the limits of integration, but can also use the function $P(t)$ to do so.

\subsubsection{4. Example: Legendre Polynomials}

Finally, let us reconsider the Legendre polynomials already discussed in \ref{subsubsec: 1. Example: Legendre Polynomials}. We start from the differential equation\footnote{This equation can be derived from the explicit formula for the function $R$ in \ref{subsubsec: 1. Example: Legendre Polynomials}.}

\begin{equation*}
    R'(t)(t-2xt^2+t^3)=(1-tx)R(t),
\end{equation*}
 multiply it by $t^{n-2}$, and then integrate it from $a$ to $b$, whence:

\begin{equation*}
    \int\limits_{a}^{b}R'(t)(t^{n-1}-2xt^n+t^{n+1})dt = \int (t^{n-2}-t^{n-1}x)R(t) dt. 
\end{equation*}
Integration by parts on the left-hand side yields:

\begin{equation*}
    \left. R(t)t^{n-1}(t^2-2xt+1)\right|_{a}^{b}-   \int\limits_{a}^{b} R(t) ((n-1)t^{n-2}-2xnt^{n-1}+(n+1)t^n)dt =   \int\limits_{a}^{b} R(t) (t^{n-2}-t^{n-1}x)dt.
\end{equation*}
We want the absolute terms to vanish; using the explicit formula for $R(t)$ from \ref{subsubsec: 1. Example: Legendre Polynomials}, i.e. $\frac{t}{\sqrt{1-2xt+t^2}}$, this implies:

\begin{equation*}
    0 = b^n \sqrt{1-2xb+b^2}-a^n \sqrt{1-2xa+a^2} \quad \forall n \in \mathbb{N}.
\end{equation*}
A solution, aside from $a=b$, is given by

\begin{equation*}
    (a,b) = \left(x-\sqrt{x^2-1}, x +\sqrt{x^2+1}\right),
\end{equation*}
which is found, if each term in the above condition is set $=0$. Finally, let us introduce the notation $P_n(x) = \int \limits_{a}^{b} t^{n-1}P(t)dt$ with the solutions we just found for $a$ and $b$. Hence we arrive at

\begin{equation*}
    -(n-1)P_{n-1}(x) +2xn P_n(x)-(n+1)P_(n+1)(x) = P_{n-1}(x)-xP_n(x),
\end{equation*}
which, after some rearrangement, reads as:

\begin{equation*}
    (n+1)P_{n+1}(x) = (2n+1)xP_n(x)-nP_{n-1}(x)
\end{equation*}
which is the difference equation from which we started in \ref{subsubsec: 1. Example: Legendre Polynomials}.

\subsubsection{Final Results}

The examples we reconsidered illustrate lucidly that Euler, in some sense, anticipated the Mellin-Transform and found a way to calculate the inverse Mellin-Transform without using contour integration. As we already indicated, Euler's approach is more general then Mellin's, since Euler's method works for general limits of integration.\\
Moreover, it is now evident that Euler's results obtained via the moment ansatz can be justified rigorously using the theory of the Mellin-Transform and therefore, Euler's results are indeed correct. And Euler's momentum method is not just a precursor of the Mellin-Transform, but it is actually the same idea. Euler just discovered it in a completely different context, i.e. the solution of difference equations, see section \ref{subsec: Euler's Idea}, whereas Mellin introduced it in his studies of the hypergeometric differential equation. see, e.g., Mellin's paper \cite{Me95}.

\subsection{Modern Idea - Intersection Theory}
\label{subsec: Modern Idea - Intersection Theory}

Euler's approach, referred to as moment method in section \ref{subsec: Euler's Idea}, started from the propounded difference equation and reduced it to a differential equation, whereas the Mellin-Transform approach outlined in section \ref{subsec: Euler and the Mellin Transform} started from a differential equation and derived the initially propounded difference equation. In this section, we want to introduce another method; it starts from a hypergeometric integral of the Eulerian type, i.e.

\begin{equation*}
    I = \int_{\mathcal{C}}t^{n-1}(1-t)^{\alpha}(1-xt)^{\beta}dt, \quad n,\alpha, \beta \in \mathbb{C}\setminus \mathbb{Z},
\end{equation*}
and derives a difference equation for it. We, following the treatise \cite{Ao11}, will briefly mention the fundamental theorems and ideas of the theory and give some examples, partially also found in \cite{Fr19}. Before doing so, let us summarize the main ingredient of Euler's approach.

\subsubsection{Main Ingredient of Euler's Approach}
\label{subsubsec: Main Ingredient of Euler's Approach}

As mentioned in section \ref{subsec: Euler's Idea}, Euler explained his idea to solve homogeneous difference equations with linear coefficients in his 1750 paper \cite{E123} $\S\S 49 - 53$, mainly focusing on continued fractions. But his main idea, aside from the ansatz

\begin{equation*}
    \int\limits_{a}^{b}t^{x-1}P(t)dt,
\end{equation*}
is the assumption of an auxiliary equation, i.e. adding the extra term $t^{x}Q(t)$. This was necessary to find the limits of integration\footnote{It will turn out that it is very difficult to find the limits using intersection theory, since that theory actually starts from the integrals, i.e. the integration domain is given from the start.}. Then, we differentiated this auxiliary equation just to integrate it again later. Therefore, what we have essentially done, is to express the integral

\begin{equation*}
    \int\limits_{a}^{b}t^{x+1}P(t)dt
\end{equation*}
as a linear combination of one or two others of the same kind. More precisely, we found an equation

\begin{equation*}
    \int\limits_{a}^{b}t^{x+1}P(t)dt =  C_1(x) \int\limits_{a}^{b}t^{x}P(t)dt + C_2(x)   \int\limits_{a}^{b}t^{x-1}P(t)dt.
\end{equation*}
$C_1(x)$ and $C_2(x)$ are functions of $x$, linear in $x$ in our case. In physics, see, e.g., \cite{Fr19}, such equations are  referred to as integration-by-parts-identities, or IBPs for short, and are frequently employed in the calculation of Feynman  integrals of Feynman diagrams. The name IBP stems from the fact that the exponent of $t$ is lowered by $1$ by integrating by parts once. \\
Therefore, the extra term $t^xQ(t)$ can be understood as a term that has to be added to ensure that we arrive at the IBP given by the propounded difference equation. Intersection theory now turns this on its head and starts from an integral like the one above. It tells how many integrals of the same type one needs to express the given integral (e.g., it would immediately tell us that one requires two integrals to express $\int\limits_{a}^{b}t^{x+1}P(t)dt$) and furthermore it provides us with a tool to calculate the coefficients ($C_1(x)$ and $C_2(x)$ in our example).\\
But we will use it here to solve difference equations again, which will once more lead us to differential equations for the function $P(t)$. Thus, let us explain the origin of intersection theory first.

\subsubsection{Overview}

In the above integral, products of powers appear. Since we assume the powers to be non-integer numbers in general, we will modify de Rham's theory presented in section \ref{subsubsec: de Rham Theory}, formalizing the ordinary theory of integrals of single-valued functions, to the multi-valued case. This modified de Rham theory, presented in section \ref{subsubsec: Twisted de Rham Theory}, is referred to as twisted de Rham theory. Since the key to ordinary de Rham theory is Stokes' theorem, will we explain the analogue of this theorem for integrals of multi-valued functions first. We will do this in section \ref{subsubsec: Towards Stokes' Theorem for Integrals of multi-valued Functions}.

\subsubsection{Necessity for the Concept of a Twist}
\label{subsubsec: Necessity for the Concept of a Twist}

In this section, we want to motivate the concept of a twist.  To this end, let $M$ be a $n$-dimensional affine variety obtained as the complement, in $\mathbb{C}^n$, of zeros $D_j := \lbrace P_j(u)=0 \rbrace$ of a finite number of polynomials $P_j(u)=P_j(u_1, \cdots, u_n)$, $1 \leq j \leq m$; hence

\begin{equation*}
    M:= \mathbb{C}^n\setminus D, \quad D := \bigcup\limits_{j=1}^{m}D_j, 
\end{equation*}
and consider the multi-valued function on $M$:

\begin{equation*}
    U(u)=\prod_{j=1}^{m}P_{j}(u)^{\alpha_j}, \quad \alpha_j \in \mathbb{C}\setminus \mathbb{Z}, ~~ 1 \leq j\leq m.
\end{equation*}
Treating this $U(u)$, one way to get rid of its multi-valuedness such that the concepts of ordinary analysis can be applied to it, is to lift this $U$ to a  covering manifold $\tilde{M}$ of $M$ and consider it on this $\tilde{M}$. But, in general the relation among $\tilde{M}$ and $M$ is so complicated that this approach is not preferable for practical calculations as we intend to perform them here.\\
Thus, instead of lifting $U$ to $\tilde{M}$, we  introduce some quantities introducing a twist arising from the multi-valuedness, and then analyze $U$. In order to explain the key concepts as simply as possible, let us take a smooth triangulation $K$ on $M$. As mentioned earlier, we want to extend Stokes' theorem to multi-valued differential forms $U\varphi$ with $\varphi \in \mathcal{A}^{\bullet}(M)$ and  find the twisted version of de Rham's theory afterwards.

\subsubsection{Towards Stokes' Theorem for Integrals of multi-valued Functions}
\label{subsubsec: Towards Stokes' Theorem for Integrals of multi-valued Functions}

We will argue rather intuitively, since the explanation of every mathematical detail is beyond the scope of this article. Let $\Delta$ be one of the $p$-simplexes of the smooth triangulation $K$ of $M$ and let $\varphi$ be a smooth $p$-form on $M$. To determine the integral of a multi-valued $p$-form we have to fix the branch of $U$ on $\Delta$. Then, we get to the following definition:

\begin{definition}
Let the symbol $\Delta \otimes U_{\Delta}$ denote the pair of $\Delta$ and one of the branches $U_{\Delta}$ of $U$. Then,  for $\varphi \in \mathcal{A}^{p}(M)$, the integral $\int\limits_{\Delta \otimes U_{\Delta}} U \cdot \varphi $ is defined as:
\begin{equation*}
    \int\limits_{\Delta \otimes U_{\Delta}} U \cdot \varphi := \int\limits_{\Delta} \left[\text{the fixed branch }U_{\Delta} \text{ of }U \text{ on } \Delta\right] \cdot \varphi.
\end{equation*}
\end{definition}
Since $U_{\Delta}$ can be continued analytically on a sufficiently small neighborhood of $\Delta$, on this neighborhood, for a single-valued $p$-form $U_{\Delta}\cdot \varphi$ and a $p$-simplex $\Delta$, the ordinary Stokes theorem holds and reads:

\begin{theorem}[Stokes's Theorem on $\Delta$]
For $\varphi \in \mathcal{A}^{p-1}(M)$ one has

\begin{equation*}
    \int\limits_{\Delta} d \left(U_{\Delta} \cdot \varphi\right) = \int\limits_{\partial \Delta} U_{\Delta}\cdot \varphi.
\end{equation*}
\end{theorem}
On the other hand, since

\begin{equation*}
    d\left(U_{\Delta}\cdot \varphi\right) = dU_{\Delta} \wedge \varphi + U_{\Delta}d \varphi = U_{\Delta} \left(d \varphi+\dfrac{dU_{\Delta}}{U_{\Delta}} \wedge \varphi\right)
\end{equation*}
on $\Delta$ and $\omega = \frac{dU}{U}$ is a single-value holomorphic $1$-form, setting

\begin{equation*}
    \nabla_{\omega}\varphi := d \varphi +\omega \wedge \varphi,
\end{equation*}
we have:

\begin{equation*}
     d\left(U_{\Delta}\cdot \varphi\right) = U_{\Delta} \nabla_{\omega} \varphi.
\end{equation*}
$\nabla_{\omega} \varphi$ can be regarded as the defining equation of the covariant differential operator $\nabla_{\omega}$ associated to the connection form $\omega$. It satisfies $\nabla_{\omega}\cdot \nabla_{\omega}=0$ and in such a case $\nabla_{\omega}$ is said to be integrable or to define a Gau\ss{}-Manin connection of rank $1$ according to \cite{De70}. Thus, we can rewrite Stokes' theorem, using the symbol $\Delta \otimes U_{\Delta}$, as:

\begin{theorem}[Stokes' Theorem rewritten]
For $\varphi \in \mathcal{A}^{p-1}(M)$ we have:

\begin{equation*}
    \int\limits_{\Delta}d\left(U_{\Delta}\cdot \varphi\right)= \int\limits_{\Delta \otimes U_{\Delta}} U \cdot \nabla_{\omega}\varphi,
\end{equation*}
where

\begin{equation*}
     \nabla_{\omega}\varphi := d \varphi +\omega \wedge \varphi \quad \text{and} \quad \omega =\dfrac{dU}{U}.
\end{equation*}
\end{theorem}
Next, we want to rewrite the right-hand side of the ordinary Stokes' theorem (2.5.8). We will illustrate the idea by considering examples in low dimension.

\subsubsection{Towards the twisted Homology Group}
\label{subsubsec: Towards the twisted Homology Group}

We start with the one-dimensional case. For an oriented $1$-simplex, with starting point $p$ and ending point $q$, and a smooth function (since $0$-forms are identified with functions) $\varphi$ on $M$, Stokes' theorem (2.5.8.) becomes:

\begin{equation*}
    \int\limits_{\Delta} d\left(U_{\Delta}\cdot \varphi\right)= U_{\Delta}(q)\varphi(q)-U_{\Delta}(p)\varphi(p).
\end{equation*}
Here we have $\partial \Delta = \left\langle q \right\rangle - \left\langle p \right\rangle$ and $U_{\Delta}(q)$ is the value of the germ of the function $U_q$, determined by the branch $U_{\Delta}$ at the boundary $q$ of $\Delta$, at the point $q$.\\
Therefore, the symbol $\left\langle q\right\rangle \otimes U_q$ is determined and we have:

\begin{equation*}
    U_{\Delta}(q)\varphi(q)= \text{the integral of }U\cdot \varphi \text{ on } \left\langle q\right\rangle \otimes U_q,
\end{equation*}
whence we conclude the formula:

\begin{equation*}
    \int\limits_{\Delta \otimes U_{\Delta}} U \cdot \nabla_{\omega}\varphi = \int\limits_{\left\langle q\right\rangle \otimes U_q- \left\langle p\right\rangle \otimes U_p} U \cdot \varphi.
\end{equation*}
Since, in the context of the rewritten Stokes theorem, the right-hand side is supposed to be the boundary of $\Delta \otimes U_{\Delta}$, we can define a boundary operator by:

\begin{definition}[Boundary Operator for one dimensional $1$-simplex]
The boundary operator for this case can be defined as:
\begin{equation*}
    \partial_{\omega}(\Delta \otimes U_{\Delta}): = \left\langle q\right\rangle \otimes U_q- \left\langle p\right\rangle \otimes U_p.
\end{equation*}
\end{definition}Stokes' theorem then reads:

\begin{equation*}
    \int\limits_{\Delta \otimes U_{\Delta}} U \cdot \nabla_{\omega}\varphi = \int\limits_{ \partial_{\omega}(\Delta \otimes U_{\Delta})} U \cdot \varphi.
\end{equation*}
Let us also consider the two-dimensional case and let $\Delta$ be an oriented $2$-simplex with vertices $\left\langle 1 \right\rangle$,  $\left\langle 2 \right\rangle$,  $\left\langle 3 \right\rangle$. 

\begin{center}
    \begin{figure}
        \centering
        \includegraphics[scale=0.7]{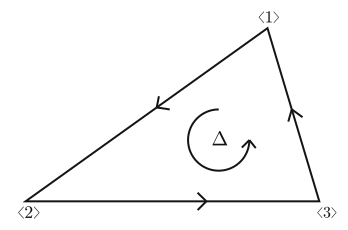}
        \caption{2-simplex with Orientation in the two-dimensional Case} Auxiliary Figure for the determination of the boundary operator in the two-dimensional case. This is Figure 2.2 from \cite{Ao11}.
    \end{figure}
\end{center}
For a smooth $1$-form $\varphi$ on $M$, the right-hand side of the ordinary Stokes theorem reads:

\begin{equation*}
    \int\limits_{\partial \Delta} U_{\Delta} \cdot \varphi = \int\limits_{\left\langle 12\right\rangle} U_{\Delta}\cdot \varphi +\int\limits_{\left\langle 23\right\rangle} U_{\Delta}\cdot \varphi +
    \int\limits_{\left\langle 31\right\rangle} U_{\Delta}\cdot \varphi.
\end{equation*}
Let $U_{\left\langle 12 \right\rangle}$ be the branch on the boundary $\left\langle 12 \right\rangle$ of $\Delta$ determined by $U_{\Delta}$. Using the symbol $\left\langle 12 \right\rangle \otimes U_{\left\langle 12 \right\rangle}$, the  above formula can be rewritten as follows:

\begin{equation*}
    \int_{\partial \Delta} U_{\Delta}\cdot \varphi = \int\limits_{\left\langle 12 \right\rangle \otimes U_{\left\langle 12 \right\rangle}}U \cdot \varphi +
    \int\limits_{\left\langle 23 \right\rangle \otimes U_{\left\langle 23 \right\rangle}}U \cdot \varphi +
    \int\limits_{\left\langle 31 \right\rangle \otimes U_{\left\langle 31 \right\rangle}}U \cdot \varphi. 
\end{equation*}
This leads to the following definition of the boundary operator in this case:

\begin{definition}[Boundary Operator for $2$-simplex in two-dimensional case]
The boundary operator in this case reads:

\begin{equation*}
    \partial_{\omega}(\Delta \otimes U_{\Delta}):= \left\langle 12 \right\rangle \otimes U_{\left\langle 12 \right\rangle}+ \left\langle 23 \right\rangle \otimes U_{\left\langle 23 \right\rangle}+ \left\langle 31 \right\rangle \otimes U_{\left\langle 31 \right\rangle}.
\end{equation*}
\end{definition}
Hence Stokes' theorem  becomes:

\begin{equation*}
    \int\limits_{\Delta \otimes U_{\Delta}} U \cdot \nabla_{\omega} \varphi = \int\limits_{\partial(\Delta \otimes U_{\Delta})} U\varphi.
\end{equation*}
After these introductory examples, we can generalize them to the higher-dimensional case. Note that the boundary operators defined  above are  precisely those appearing in algebraic topology when one defines homology with coefficients in a local system. We will explain this below.\\
For now, consider the following differential equation on $M$

\begin{equation*}
    \nabla_{\omega}h = dh +\sum_{j=1}^{m}\alpha_j \dfrac{dP_j}{P_j}h=0
\end{equation*}
which has a general solution, formally expressible as:

\begin{equation*}
    h= c \prod_{j=1}^{m}P_j^{-\alpha_j}, \quad c \in \mathbb{C}.
\end{equation*}
The space generated by the local solutions of the differential equation has dimension $1$. Cover the manifold $M$ by a sufficiently fine  locally finite open cover $M= \cup U_{\nu}$ and  fix a single-valued non-zero solution $h_{\nu}$ on each $U_{\nu}$. Since, provided $\mu \neq \nu$, $h_{\mu}$ and $h_{\nu}$ are solutions of the same differential function on the non-empty set $U_{\mu}\cap U_{\nu}$, setting

\begin{equation*}
    h_{\mu}(u)= g_{\mu \nu}(u)h_{\nu}, \quad u \in U_{\mu}\cap U_{\nu}
\end{equation*}
the transition function $g_{\mu \nu}$ is a constant on $U_{\mu}\cap U_{\nu}$. Since a solution $h(u)$ on $U_{\mu}\cap U_{\nu}$ is expressed in two ways, i.e. as $h= \xi_{\mu}h_{\mu}=h= \xi_{\nu}h_{\nu}$ with $\xi_{\mu}, \xi_{\nu} \in \mathbb{C}$, we find $\xi_{\mu}= g_{\mu \nu}^{-1}\xi_{\nu}$. \\
Therefore, the set of all local solutions of the given differential equation defines a flat line bundle $\mathcal{L}_{\omega}$ obtained by gluing the fibers $\mathbb{C}$ by the transition functions $\left\lbrace g_{\mu \nu}^{-1}\right\rbrace$. Let us denote the  flat line bundle obtained from the transition functions $\left\lbrace g_{\mu \nu}\right\rbrace$ by $\mathcal{L}_{\omega}^{\vee}$ and call it the dual line bundle of the line bundle $\mathcal{L}_{\omega}$.  By the above equation relating $h_{\nu}$ and $h_{\mu}$, $h_{\mu}^{-1}$ becomes a local section of $\mathcal{L}_{\omega}^{\vee}$. Moreover, from the initial differential equation, $\mathcal{L}_{\omega}^{\vee}$ can be considered as the flat line bundle generated by the set of all local solutions of

\begin{equation*}
    \nabla_{-\omega} h=dh- \sum_{j=1}^{m}\alpha_{j}\dfrac{dP_j}{P_j}h=0,
\end{equation*}
and $U(u)= \prod P_j^{\alpha_j}$ is its local section.\\[2mm]
Often, e.g., in \cite{Ao11}, the term "flat line bundle" is used in the same sense as "system of rank $1$". Therefore, we call $\mathcal{L}_{\omega}^{\vee}$ the dual local system of $\mathcal{L}_{\omega}$. Hence we see that the boundary operators defined in the first two examples are those of the chain groups $C_1(K, \mathcal{L}_{\omega}^{\vee})$ and $C_2(K, \mathcal{L}_{\omega}^{\vee})$ with coefficients in $\mathcal{L}_{\omega}^{\vee}$. Naturally, if using $M$ instead of the simplicial complex $K$, we denote the chain group by $C_q(M, \mathcal{L}_{\omega}^{\vee})$ for the $q$-dimensional case. More precisely, we were lead to the definition:

\begin{definition}[$p$-dimensional twisted Chain Group]
We call

\begin{equation*}
    C_p(M, \mathcal{L}_{\omega}^{\vee}) = \left\lbrace \begin{array}{l}
         \text{for a $p$-simplex $\Delta$ of $M$}, \\
         \text{the complex vector space} \\
         \text{with basis $\Delta \otimes U_{\Delta}$}
    \end{array} \right\rbrace
\end{equation*}
the $p$-dimensional twisted chain  group.
\end{definition}Moreover, as a generalization of the above examples, we have:

\begin{definition}[Boundary Operator of the $p$-simplex]
We define the boundary operator

\begin{equation*}
    \partial_{\omega}: C_{p}(M, \mathcal{L}_{\omega}^{\vee}) \rightarrow C_{p-1}(M, \mathcal{L}_{\omega}^{\vee})
\end{equation*}
for the $p$-simplex $\Delta = \left\langle 012\cdots p\right\rangle$ as

\begin{equation*}
    \partial_{\omega}(\Delta \otimes U_{\Delta}):= \sum_{j=0}^{p} (-1)^{j} \left\langle 01 \cdots \widehat{j}\cdots p\right\rangle \otimes U_{\left\langle 01 \cdots \widehat{j}\cdots p\right\rangle}.
\end{equation*}
where $\widehat{j}$ indicates that this term is omitted.
\end{definition}
Similarly to the ordinary homology theory, one has:

\begin{theorem}
The boundary operator $\partial_{\omega}$ satisfies:

\begin{equation*}
    \partial_{\omega}\circ \partial_{\omega}=0.
\end{equation*}
\end{theorem}
Having given this theorem, we can define the twisted homology group:

\begin{definition}[Twisted Homology Group]
The qoutient vector space

\begin{equation*}
    H_{p}(M, \mathcal{L}_{\omega}^{\vee}):= \left\lbrace \text{Ker } \partial_{\omega}:   C_{p}(M, \mathcal{L}_{\omega}^{\vee}) \rightarrow C_{p-1}(M, \mathcal{L}_{\omega}^{\vee}) \right\rbrace / \partial_{\omega} C_{p+1}(M, \mathcal{L}_{\omega}^{\vee})
\end{equation*}
is called the $p$-dimensional twisted homology group, and an element of Ker $\partial_{\omega}$ is called a twisted cycle.
\end{definition}

\subsubsection{Locally finite Twisted Holonomy Group}
\label{subsubsec: Locally finite Twisted Holonomy Group}

Having arrived at the definition of the twisted homology group, we can now, in analogy to the ordinary homology group, introduce some related concepts. We start with the locally finite twisted homology group. \\
Since in general $M$ is not compact, it would require infinitely many simplices for its triangulation, hence it is natural to consider an infinite chain group which  is locally finite. More precisely, we set:

\begin{definition}[Locally finite $p$-dimensional twisted chain group]
The group

\begin{equation*}
    C_{p}^{lf}(M, \mathcal{L}_{\omega}^{\vee}) := \left\lbrace  \Sigma c_{\Delta}\Delta \otimes U_{\Delta}| \text{the $\Delta$'s are locally finite}\right\rbrace
\end{equation*}
is called the locally finite Twisted chain Group.
\end{definition}
Since $\partial_{\omega} \circ \partial_{\omega}=0$, $(C_{\bullet}^{lf}(M,\mathcal{L}_{\omega}^{\vee}))$ forms a complex, which in turn leads to the next

\begin{definition}[Locally finite Twisted  Holonomy Group]
The homology group obtained from the complex $(C_{\bullet}^{lf}(M,\mathcal{L}_{\omega}^{\vee}))$ is denoted by $(H_{\bullet}^{lf}(M,\mathcal{L}_{\omega}^{\vee}))$ and called the locally finite Twisted Holonomy Group.
\end{definition}
Let us now consider Stokes' theorem. Thus, let $\Delta$ be a $p$-simplex of $M$, $U_{\Delta}$ a branch of $U$ on $\Delta$, $\varphi \in \mathcal{A}^{p-1}(M)$, then the right-hand side of Stokes' theorem becomes:

\begin{equation*}
        \renewcommand{\arraystretch}{1,5}
\setlength{\arraycolsep}{0.0mm}
\begin{array}{lll}
     \int\limits_{\partial \Delta} U_{\Delta} \cdot \varphi &~=~& \sum_{i=0}^{p} (-1)^i \int\limits_{\left\langle 0 \cdots \widehat{i}\cdots p\right\rangle} U_{\Delta}\cdot \varphi \\
     &~=~& \sum_{i=0}^{p} (-1)^i \int\limits_{\left\langle 0 \cdots \widehat{i}\cdots p\right\rangle \otimes U_{\left\langle 0 \cdots \widehat{i} \cdots p \right\rangle}} U \cdot \varphi \\
     &~=~& \int\limits_{\partial(\Delta \otimes U_{\Delta})}U \cdot \varphi.
\end{array}
\end{equation*}
For a general twisted chain, one can just extend the definition $\mathbb{C}$-linearly and arrives at the following  form of Stokes' theorem:

\begin{theorem}[Stokes' Theorem for $C_p(M, \mathcal{L}_{\omega}^{\vee})$ and $C_p^{lf}(M, \mathcal{L}_{\omega}^{\vee})$]{~}\\

\begin{itemize}
\item[1.] For $\sigma \in C_p(M, \mathcal{L}_{\omega}^{\vee}) $ and $\varphi \in \mathcal{A}^{p-1}(M)$ we have:

\begin{equation*}
    \int\limits_{\sigma}U \cdot \nabla_{\omega} \varphi = \int\limits_{\partial_{\omega}\sigma} U \cdot \varphi.
\end{equation*}
\item[2.]  For $\tau \in C_p^{lf}(M, \mathcal{L}_{\omega}^{\vee}) $ and $\varphi \in \mathcal{A}_c^{p-1}(M)$ we have:

\begin{equation*}
    \int\limits_{\tau}U \cdot \nabla_{\omega} \varphi = \int\limits_{\partial_{\omega}\tau} U \cdot \varphi.
\end{equation*}
\end{itemize}
\end{theorem}
Next, since we already see the similarity to the de Rham theory here, we want to construct the twisted version of the theory. To this end, using the last theorem, we have to establish relations among the cohomology of

\begin{equation*}
    \left(A^{\bullet}(M), \nabla_{\omega}\right), \quad \left(A_p^{\bullet}(M), \nabla_{\omega}\right)
\end{equation*}
and the twisted holonomy groups 

\begin{equation*}
    H_{\bullet}\left(M, \mathcal{L}_{\omega}^{\vee}\right), \quad  H_{\bullet}^{lf}\left(M, \mathcal{L}_{\omega}^{\vee}\right).
\end{equation*}

\subsubsection{de Rham Theory}
\label{subsubsec: de Rham Theory}

As mentioned in the preceding section, we want to establish that the cohomology of the complex $ \left(A^{\bullet}(M), \nabla_{\omega}\right)$ is dual to the twisted homology   $H_{\bullet}\left(M, \mathcal{L}_{\omega}^{\vee}\right)$. Since the proof of this fact is too long to present it here, we will simply state the theorem. The reader interested in the proof is, e.g., referred to \cite{Ao11}.\\
Furthermore, let us point out that ordinary de Rham theory corresponds to the case $\omega =0$ and thus twisted de Rham theory will be its natural generalization. \\[2mm]
De Rham theory is summarized by de Rham's theorem and duality theorems.
\begin{theorem}[de Rham's Theorem]
The following isomorphisms exist:

\begin{equation*}
     \renewcommand{\arraystretch}{2,0}
\setlength{\arraycolsep}{0.0mm}
\begin{array}{ll}
   (1) \quad & H^p(M, \mathbb{C}) \simeq  \left \lbrace
     \renewcommand{\arraystretch}{1,0}
\setlength{\arraycolsep}{0.0mm}
    \begin{array}{c}
         H^{p}(\mathcal{A}^{\bullet}(M),d) \\
         \uparrow \wr\\
         H^{p}(\Omega^{\bullet}(M),d)
    \end{array}
    \right\rbrace \xrightarrow{\sim} H^{p}(\mathcal{K}^{\bullet}(M),d) \xleftarrow{\sim} H_{2n-p}^{lf}(M, \mathbb{C}) \\
    (2) \quad & H_c^{p}(M, \mathbb{C}) \simeq H^p(\mathcal{A}_{c}^{\bullet}(M),d) \xrightarrow{\sim} H^p(\mathcal{K}_{c}^{\bullet}(M),d) \xleftarrow{\sim} H_{2n-p}(M,\mathbb{C}).
    \end{array}
\end{equation*}
\end{theorem}
Here, $\mathcal{K}_c^p$ is the sheaf of germs of currents of degree $p$ with compact support on $M$. A current is an element of the space containing both $\mathcal{C}^{\infty}$ differential forms and smooth singular chains.\\
Furthermore, we have the following 

\begin{theorem}[Duality Theorems]
The following bilinear forms are non-degenerate
\begin{equation*}
       \renewcommand{\arraystretch}{2,0}
\setlength{\arraycolsep}{0.0mm}
\begin{array}{cccccccccccccc}
    (1) \quad & H_p(M,\mathbb{C}) &~\times~ & H^{p}(\mathcal{A}^{\bullet}(M),d) &~\rightarrow ~& \mathbb{C} \\
              &                  &         & (\left[ \sigma\right], \left[\varphi\right]) &~\longmapsto~ & \int\limits_{\sigma}\varphi, \\
              (2) \quad & H^{2n-p}(\mathcal{A}_{c}^{\bullet}M,\mathbb{C}) &~\times~ & H^{p}(\mathcal{A}^{\bullet}(M),d) &~\rightarrow ~& \mathbb{C} \\
              &                  &         & (\left[ \alpha\right], \left[\beta\right]) &~\longmapsto~ & \int\limits_{M}\alpha \wedge \beta. \\
\end{array}
\end{equation*}
\end{theorem}
These are the fundamental theorems in de Rham theory. These results can be extended to twisted de Rham theory.

\subsubsection{Twisted de Rham Theory}
\label{subsubsec: Twisted de Rham Theory}

We state the most important theorems necessary to develop twisted de Rham theory

\begin{theorem}
There is an natural isomorphism
\begin{equation*}
    H^{p}(M, \mathcal{L}_{\omega}) \simeq \operatorname{Hom}_{\mathbb{C}}(H_p(M, \mathcal{L}_{\omega}^{\vee}), \mathbb{C})
\end{equation*}
\end{theorem}
Via this isomorphism, the cohomology group $H^{p}(M, \mathcal{L}_{\omega})$ can be considered as the dual space of the homology group $H_{p}(M,\mathcal{L}_{\omega}^{\vee})$. Therefore,  the value $\left\langle \left[\sigma\right], \left[ \varphi\right] \right\rangle$ of a cohomology class $\left[\varphi \right] \in H^p(M, \mathcal{L}_{\omega})$ at a homology class $\left[\sigma\right] \in H_p(M, \mathcal{L}_{\omega})$ is well-defined and we have the following 
\begin{theorem}
The following bilinear form is non-degenerate
\begin{equation*}
      \renewcommand{\arraystretch}{1,0}
\setlength{\arraycolsep}{0.0mm}
\begin{array}{ccccccccc}
     H_p(M, \mathcal{L}_{\omega}^{\vee})&~\times ~& H^p(M, \mathcal{L}_{\omega})&~\rightarrow & \mathbb{C} \\
      & & \left(\left[\sigma\right], \left[\varphi\right]\right) &~ \rightarrow ~& \left\langle \left[\sigma\right], \left[\varphi\right] \right\rangle.
\end{array}
\end{equation*}
\end{theorem}
In analogy to de Rham's Theorem, we have the

\begin{theorem}
The following isomorphisms exist

\begin{equation*}
     \renewcommand{\arraystretch}{2,0}
\setlength{\arraycolsep}{0.0mm}
\begin{array}{ll}
   (1) \quad & H^p(M, \mathcal{L}_{\omega}) \simeq  \left \lbrace
     \renewcommand{\arraystretch}{1,0}
\setlength{\arraycolsep}{0.0mm}
    \begin{array}{c}
         H^{p}(\mathcal{A}^{\bullet}(M),\nabla_{\omega}) \\
         \uparrow \wr\\
         H^{p}(\Omega^{\bullet}(M),\nabla_{\omega})
    \end{array}
    \right\rbrace \xrightarrow{\sim} H^{p}(\mathcal{K}^{\bullet}(M),\nabla_{\omega}) \xleftarrow{\sim} H_{2n-p}^{lf}(M, \mathcal{L}_{\omega}) \\
    (2) \quad & H_c^{p}(M, \mathcal{L}_{\omega}) \simeq H^p(\mathcal{A}_{c}^{\bullet}(M),\nabla_{\omega}) \xrightarrow{\sim} H^p(\mathcal{K}_{c}^{\bullet}(M),\nabla_{\omega}) \xleftarrow{\sim} H_{2n-p}(M,\mathcal{L}_{\omega}).
    \end{array}
\end{equation*}
\end{theorem}
Finally, we can give the most important theorem for our investigations.

\begin{theorem}
  The following bilinear forms are non-degenerate
  
  \begin{equation*}
         \renewcommand{\arraystretch}{2,0}
\setlength{\arraycolsep}{0.0mm}
\begin{array}{crccclcc}
     (1) \quad & H_p(M, \mathcal{L}_{\omega}^{\vee})&~\times ~& H^{p}(\mathcal{A}^{\bullet}(M), \nabla_{\omega})&~\rightarrow ~& \mathbb{ C } \\
      &   & & (\left[\sigma\right], \left[\varphi\right]) &~\mapsto ~& \int\limits_{\sigma}U \cdot \varphi. \\
      (2) \quad & H^p(\mathcal{A}_c^{\bullet}(M), \nabla_{\omega})&~\times ~& H_{p}^{lf}(M), \mathcal{L}_{\omega}^{\vee})&~\rightarrow ~& \mathbb{ C } \\
      &   & & (\left[\psi\right], \left[\tau\right]) &~\mapsto ~& \int\limits_{\tau}U \cdot \psi. \\
      (3) \quad & H^{2n-p}(\mathcal{A}_c^{\bullet}(M), \nabla_{-\omega})&~\times ~& H^{p}(\mathcal{A}^{\bullet}(M), \nabla_{\omega})&~\rightarrow ~& \mathbb{ C } \\
      &   & & (\left[\alpha\right], \left[\beta\right]) &~\mapsto ~& \int\limits_{M} \alpha \wedge \beta. \\
     (4) \quad & H_p(M, \mathcal{L}_{\omega}^{\vee})&~\times ~& H_{2n-p}^{lf}(M, \mathcal{L}_{\omega})&~\rightarrow ~& \mathbb{ C } \\
      &   & & (\left[\sigma\right], \left[\tau\right]) &~\mapsto ~& \text{intersection number } \left[\sigma\right] \cdot [\tau].
\end{array}
  \end{equation*}
\end{theorem}
In the next sections we will explain the meaning of equation (4), i.e. the intersection number which is crucial for investigations of hypergeometric integrals.

\subsubsection{Construction of twisted Cycles}
\label{subsubsec: Construction of twisted Cycles}

In this section, we want to construct some twisted cycles explicitly. Starting from the one-dimensional case and the twisted cycle around one point, we will generalize this and will arrive at an explicit formula for $H_1(M, \mathcal{L}_{\omega}^{\vee})$.\\[2mm]
We will explicitly construct $m-1$ independent twisted cycles associated to the multi-valued function

\begin{equation*}
    U(u)= \prod_{j=1}^{m}(u-x_j)^{\alpha_j}, \quad \alpha_j \notin \mathbb{Z}
\end{equation*}
which is defined $M=\mathbb{C}\setminus \lbrace x_1, \cdots, x_m\rbrace$. Since later we will also be mainly interested in this case, we assume all the points $x_j,~ 1 \leq j \leq m$ to lie on the real axis such that $x_1 <\cdots< x_m$. We set $\omega =\frac{dU}{U}$ such that this $U$, as we explained in section \ref{subsubsec: Towards the twisted Homology Group},  defines the rank $1$ local system $\mathcal{L}_{\omega}^{\vee}$.  We choose the following branch of $U(u)$: we take the one that is single-valued on the lower half plane and that satisfies

\begin{equation*}
    \operatorname{arg}(u-x_j)= \left\lbrace 
    \begin{array}{rl}
         0 & (1 \leq j \leq p) \\
         - \pi & (p+1 \leq j \leq m)
    \end{array}
    \right\rbrace
\end{equation*}
on each interval $\Delta_p:= (x_p,x_{p+1})$, $1 \leq p \leq m-1$. From basic algebraic topology we know that\\[2mm]
(1) $H_{\bullet}(M, \mathcal{L}_{\omega}^{\vee})$ is homotopy invariant \\
(2) The Mayer-Vietoris sequence holds\\[2mm]
Thus, using (1) it suffices to calculate $H_{\bullet}(K, \mathcal{L}_{\omega}^{\vee})$ for the one-dimensional simplicial complex $K$ seen in the figure above. Note that, for the sake of simplicity, we also write $\mathcal{L}_{\omega}^{\vee}$ for the restriction of $\mathcal{L}_{\omega}^{\vee}$ on $M$ to $K$. To apply the Mayer-Vietoris sequence (2), let us write $K=K_1 \cup K_2$, i.e. as the union of two subcomplexes

\begin{equation*}
    K_1= \bigcup_{j=1}^{m} S_{\varepsilon}^1(x_j), \quad K_2 = \bigcup_{j=1}^{m-1}[x_j+\varepsilon, x_{j+1}-\varepsilon].
\end{equation*}

\begin{center}
    \begin{figure}
    \centering
        \includegraphics[scale=0.7]{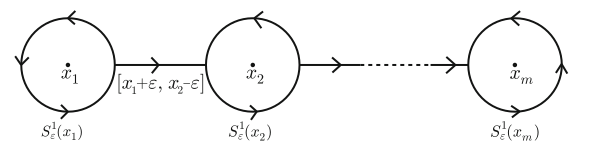}
        \caption{Construction of twisted Cycles in the one-dimensional Case} Auxiliary figure for the construction of a twisted cycles for the function $U(u)=\prod_{j=1}^m (u-x_j)^{\alpha_j}, ~\alpha_j \notin \mathbb{Z}$. This is figure 2.5 from \cite{Ao11}.
    \end{figure}
\end{center}
In the first step, we calculate the  twisted homology  for each circle $S_{\varepsilon}^1(x_j)$, i.e. we prove:

\begin{theorem}[Twisted homology for a circle]
For $\alpha_j \notin \mathbb{Z}$, $1 \leq j \leq m$, we have

\begin{equation*}
    H_q(S_{\varepsilon}^1(x_j), \mathcal{L}_{\omega}^{\vee})=0, \quad q=0,1, \cdots
\end{equation*}
\end{theorem}
\begin{proof}
We triangulate the circle as seen in the figure above

\begin{center}
    \begin{figure}
    \centering
        \includegraphics[scale=0.7]{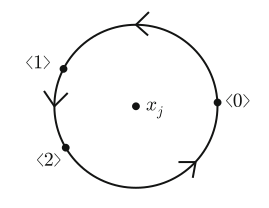}
        \caption{Triangulation of the Circle} Auxiliary figure for the calculation of $H_q(S_{\varepsilon}^1(x_j), \mathcal{L}_{\omega}^{\vee})$. This is figure 2.6 from \cite{Ao11}.
    \end{figure}
\end{center}
Thus, the chain group $C_q(S_{\varepsilon}^1(x_j), \mathcal{L}_{\omega}^{\vee})$ and the boundary operator $\partial_{\omega}$ are:

\begin{equation*}
   \renewcommand{\arraystretch}{2,0}
\setlength{\arraycolsep}{0.0mm}
    \begin{array}{l}
        C_0(S_{\varepsilon}^1(x_j), \mathcal{L}_{\omega}^{\vee})= \mathbb{C}\langle 0 \rangle +\mathbb{C}\langle 1 \rangle +\mathbb{C}\langle 2 \rangle, \\
     C_1(S_{\varepsilon}^1(x_j), \mathcal{L}_{\omega}^{\vee})=  \mathbb{C}\langle 01 \rangle +\mathbb{C}\langle 12 \rangle +\mathbb{C}\langle 20 \rangle, \\
     \partial_{\omega}\langle 01\rangle =\langle 1 \rangle - \langle 0 \rangle, \quad \partial_{\omega}\langle 12\rangle =\langle 2 \rangle - \langle 0 \rangle, \quad \partial_{\omega}\langle 20\rangle =c_j\langle 0 \rangle - \langle 2 \rangle,
    \end{array}
\end{equation*}
with $c_j= \exp(2\pi i \alpha_j)$. Writing these relations as a vectorial equation, we have

\begin{equation*}
    \partial_{\omega} \begin{pmatrix} \langle 01 \rangle \\
    \langle 12 \rangle \\
    \langle 20 \rangle
    \end{pmatrix} = \begin{pmatrix} -1 & 1 & 0 \\
    0 & -1 & 1 \\
    c_j & 0 & -1
    \end{pmatrix}
    \begin{pmatrix}
    \langle 0 \rangle \\
    \langle 1 \rangle \\
    \langle 2 \rangle
    \end{pmatrix}
\end{equation*}
and $\operatorname{det}(\partial_{\omega})=c_j-1$. This is not zero, since $\alpha_j \notin \mathbb{Z}$. Therefore, the map

\begin{equation*}
    \partial_{\omega} C_1(S_{\varepsilon}^1(x_j), \mathcal{L}_{\omega}^{\vee}) \rightarrow  C_0(S_{\varepsilon}^1(x_j), \mathcal{L}_{\omega}^{\vee})
\end{equation*}
is an isomorphism and hence the theorem is proven.
\end{proof}
Next, we have the 

\begin{theorem}
Provided $\alpha_j \notin \mathbb{Z}$, $1 \leq j \leq m$, we have

  \begin{equation*}
        \renewcommand{\arraystretch}{2,0}
\setlength{\arraycolsep}{0.0mm}
\begin{array}{lcl}
     H_0(M, \mathcal{L}_{\omega}^{\vee})&~=~& 0,\\
     H_1(M, \mathcal{L}_{\omega}^{\vee})&\simeq~& \mathbb{C} \cdot \Delta_{j}(\omega)
\end{array}
  \end{equation*}
  with the twisted cycle $\Delta_j(\omega)$
  
  \begin{equation*}
   \Delta_j(\omega):= \dfrac{1}{c_j-1}S_{\varepsilon}^{1}(x_j)+[x_j+\varepsilon, x_{j+1}-\varepsilon]-\dfrac{1}{c_{j+1}-1}S_{\varepsilon}^1 (x_{j+1}).
  \end{equation*}
\end{theorem}
We sketch the idea of the proof. One applies the Mayer-Vietoris sequence to the union $K=K_1 \cup K_2$ to find the exact sequence

\begin{equation*}
    0 \rightarrow H_1(K_1, \mathcal{L}_{\omega}^{\vee}) \oplus H_1(K_2, \mathcal{L}_{\omega}^{\vee}) \rightarrow H_1(K, \mathcal{L}_{\omega}^{\vee}) \xrightarrow{\delta}
\end{equation*}
\begin{equation*}
    \rightarrow H_0(K_1 \cap K_2, \mathcal{L}_{\omega}^{\vee}) \rightarrow H_0(K_1, \mathcal{L}_{\omega}^{\vee}) \oplus H_0(K_2, \mathcal{L}_{\omega}^{\vee}) \rightarrow H_0(K, \mathcal{L}_{\omega}^{\vee}) \rightarrow 0.
\end{equation*}
Therefore, the task is to find $\delta$ such that one can determine $H_1(K,\mathcal{L}_{\omega}^{\vee})$. To this end, note that $K_2$ is homotopically equivalent to $m-1$ points and $K_1 \cap K_2$ is homotopically equivalent to $2m-2$ points. Furthermore, using the last theorem, many terms in the sequence are seen to vanish, and we are left with the exact sequence:

\begin{equation*}
        \renewcommand{\arraystretch}{1,0}
\setlength{\arraycolsep}{0.0mm}
\begin{array}{cccccccccccc}
&    0 & ~\rightarrow ~& H_1(K, \mathcal{L}_{\omega}^{\vee}) &~\xrightarrow{\delta}~&  H_0(K_1 \cap K_2, \mathcal{L}_{\omega}^{\vee}) &~\xrightarrow{i_{*}}~ & H_0( K_2, \mathcal{L}_{\omega}^{\vee})   
     & ~ \rightarrow \\
    & &  &  & &||\wr & & ||\wr \\
     & &  &  & & \mathbb{C}^{2m-2} & & \mathbb{C}^{m-1} \\
  & & ~\rightarrow ~ &  H_0(K_, \mathcal{L}_{\omega}^{\vee}) & ~\rightarrow ~& 0   
\end{array}
\end{equation*}
The map $i_{*}$ is induced from he inclusion $i: K_1 \cap K_2 \rightarrow K_2$ and we find

\begin{equation*}
    i_{*}\left(\sum_{j=1}^{m-1}a_j\langle x_j+\varepsilon\rangle +b_j\langle x_{j+1}-\varepsilon\rangle\right)= \sum_{j=1}^{m-1}(a_j+b_j)\langle x_j+\varepsilon\rangle, \quad a_j, b_j \in \mathbb{C}.
\end{equation*}
Thus, $i_{*}$ is surjective and we find:

\begin{equation*}
        \renewcommand{\arraystretch}{2,0}
\setlength{\arraycolsep}{0.0mm}
\begin{array}{llll}
     H_0(K, \mathcal{L}_{\omega}^{\vee})=0 \\
     \operatorname{Im } \delta = \operatorname{Ker }i_{*} =\bigoplus_{j=1}^{m-1}\mathbb{C}(\langle x_j-\varepsilon\rangle -\langle x_{j+1}-\varepsilon\rangle).
\end{array}
\end{equation*}
One finds $\delta$, which is necessary to find $H_1(K,\mathcal{L}_{\omega}^{\vee})$, in the same way; we state the result. Defining

\begin{equation*}
   \Delta_j(\omega):= \dfrac{1}{c_j-1}S_{\varepsilon}^{1}(x_j)+[x_j+\varepsilon, x_{j+1}-\varepsilon]-\dfrac{1}{c_{j+1}-1}S_{\varepsilon}^1 (x_{j+1}),
\end{equation*}
one has 

\begin{equation*}
    \delta(\Delta_j(\omega))=\langle x_j+\varepsilon \rangle -\langle x_{j+1}-\varepsilon\rangle
\end{equation*}
Therefore,

\begin{equation*}
    H_1(K, \mathcal{K}_{\omega}^{\vee})= \bigoplus_{j=1}^{m-1}\mathbb{C}\cdot \Delta_{j}(\omega).
\end{equation*}

\subsubsection{Intersection Number for Twisted Cycles}
\label{subsubsec: Intersection Number for Twisted Cycles}

We will explain the intersection number by means of a simple example which can then be generalized to the general case. First, we mention the 

\begin{theorem}
Under the assumption $\sum_{j=1}^{m} \alpha_j \notin \mathbb{Z}$, twisted cycles $\Delta_{\nu}\otimes U_{\Delta_{\nu}}$ defined by bounded chambers $\Delta_{\nu}$, $1 \leq \nu \leq \binom{m-1}{n}$, form a basis of $H_{n}^{lf}(M, \mathcal{L}_{\omega}^{\vee})$:

\begin{equation*}
    H_{n}^{lf}(M, \mathcal{L}_{\omega}^{\vee}) \simeq \bigoplus\limits_{\nu=1}^{\binom{m-1}{n}}\mathbb{C}\left[\Delta_{\nu}\otimes U_{\Delta_{\nu}}\right].
\end{equation*}
\end{theorem}
Concerning bounded champers: The arrangement of $m$ real hyperplanes in $\mathbb{R}^n$ decomposes that space into several chambers, which are referred to as bounded chambers.\\
Having stated this theorem in advance, let us explain what the intersection number between $H_1(M, \mathcal{L}_{\omega}^{\vee})$ and $H_{1}^{lf}(M, \mathcal{L}_{\omega}^{\vee})$ means. We will consider the case $m=2$, the general case can then be understood similarly.\\
Let $M=\mathbb{C}\setminus \lbrace 0,1\rbrace$ and choose a branch of the multi-valued function $U(u)=u^{\alpha}(u-1)^{\beta}$, $\alpha, \beta, \alpha +\beta \notin \mathbb{Z}$. $H_1(M,\mathcal{L}_{\omega}^{\vee})$ is one dimensional and is basis is given by

\begin{equation*}
 \renewcommand{\arraystretch}{2,0}
\setlength{\arraycolsep}{0.0mm}
    \begin{array}{l}
        \Delta(\omega) = \dfrac{1}{c_1-1}S_{\varepsilon}^{1}(0) \otimes U_{S_{\varepsilon}^1(0)} +[\varepsilon, 1- \varepsilon]\otimes U_{[\varepsilon, 1- \varepsilon]}-\dfrac{1}{c_2-1}S_{\varepsilon}^1(1) \otimes U_{S_{\varepsilon}^1(1)}, \\
        c_1= \exp(2\pi i \sqrt{\alpha}), \quad c_2= \exp(2\pi i \beta)
    \end{array}
\end{equation*}
$S_{\varepsilon}^{1}(0)$ is a circle around $0$ with radius $p$ with starting point $\varepsilon$, which turns in anticlockwise direction, stopping at an infinitely near point before $\varepsilon$, and $U_{S_{\varepsilon}^1(0)}$ is the branch obtained by analytic continuation along $S_{\varepsilon}^{1}(0)$, and similarly for the rest.\\
Now, using the above theorem, $H_1^{lf}(M, \mathcal{L}_{\omega})$ is the one dimensional vector space with basis $(0,1)\otimes U_{(0,1)}^{-1}$, where $U_{(0,1)}$ is the restriction of the branch determined above to $(0,1)$. Define $\Delta := \gamma \otimes U_{\gamma}^{-1}$, obtained by deforming the  interval $(0,1)$ as in the figure above, this is homologous to $(0,1)\otimes U_{(0,1)}^{-1}$ in $H_1^{lf}(M, \mathcal{L}_{\omega})$.\\[2mm]

The geometric intersections of $\Delta(\omega)$ and $\Delta$ are the three points $p$, $\frac{1}{2}$, $q$ and the signatures are $-$, $-$, $+$. Since the difference  of the branches is cancelled by $U$ and $U^{-1}$, we finally arrive at:

\begin{equation*}
    \renewcommand{\arraystretch}{2,0}
\setlength{\arraycolsep}{0.0mm}
\begin{array}{lll}
     \Delta(\omega) \cdot \Delta &~=~& \dfrac{1}{c_1-1}\times (-1)+(-1)+ \left(-\dfrac{1}{c_2-1}\right)\times (+1) \\
     &~=~& -\dfrac{c_1c_2 -1}{(c_1-1)(c_2-1)}
\end{array}
\end{equation*}
for the intersection number.\\[2mm]

As already indicated, one can compute the intersection matrix for a general $m$ in a similar way. The result will look as follows: We set $\Delta_j := (x_j,x_{j+1}) \otimes U_{(x_j, x_{j+1})}^{-1} \in H_1^{lf}(M, \mathcal{L}_{\omega})$, $1 \leq j \leq m-1$, $c_j = \exp(2\pi i \alpha_j)$, $d_j:= c_j-1$, $d_{jk}=c_jc_k-1$, then we have:

\begin{equation*}
\Delta_j(\omega)\cdot \Delta_k = \left\lbrace
    \renewcommand{\arraystretch}{2,0}
\setlength{\arraycolsep}{0.0mm}
\begin{array}{lll}
     \dfrac{c_j}{d_j}, \quad & j=k+1, \\
     -\dfrac{d_{j,j+1}}{d_jd_{j+1}}, \quad & j=k, \\
     -\dfrac{1}{d_k}, \quad & j+1 =k \\
     0, \quad & \text{otherwise}
\end{array}
\right.
\end{equation*}

\begin{center}
    \begin{figure}
        \centering
        \includegraphics[scale=0.7]{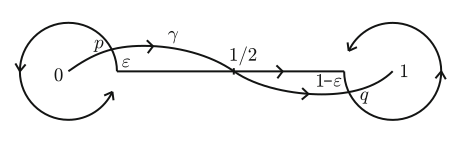}
        \caption{Intersections of $\Delta(\omega)$ and $\Delta$} Figure showing the three intersection of $\Delta(\omega)$ and $\Delta$ at the three points $p$, $\frac{1}{2}$, $q$. This is Figure 2.7 from \cite{Ao11}.
    \end{figure}
\end{center}

Therefore, the intersection matrix, consisting of each intersection number, reads:

\begin{equation*}
    \begin{pmatrix}
    \Delta_1(\omega) \cdot \Delta_1 & \cdots & \Delta_1(\omega) \cdot \Delta_{m-1} \\
    \vdots & & \vdots \\
    \Delta_{m-1}(\omega) \cdot \Delta_1 & \cdots & \Delta_{m-1}(\omega) \cdot \Delta_{m-1}
     \end{pmatrix}
\end{equation*}
\begin{equation*}
    =- \begin{pmatrix}
    \dfrac{d_{12}}{d_1d_2} & \dfrac{-1}{d_2} &  & 0 & \cdots & 0 & 0 \\
    \dfrac{-c_2}{d_2} & \dfrac{d_{23}}{d_2d_3} & &  &        & 0 & 0 \\
                                               & &  &        &  & \vdots \\
      0 & 0 & & & &\dfrac{d_{m-2,m-1}}{d_{m-2}d_{m-1}} & \dfrac{-1}{d_{m-1}} \\
      0 & 0 & \cdots & & 0 & \dfrac{-c_{m-1}}{d_{m-1}} & \dfrac{d_{m-1,m}}{d_{m-1}d_m} 
    \end{pmatrix}
\end{equation*}

\subsubsection{Intersection Numbers for Twisted Cocycles}
\label{subsubsec: Intersection Numbers for Twisted Cocycles}

Although the intersection number of twisted cycles is of great importance in the general theory of hypergeometric functions, for us, more interested in difference equations, we need the  intersection number of  twisted \textit{cocycles}. Above, in the main theorem of section \ref{subsubsec: Twisted de Rham Theory}, equation (3),

\begin{equation*}
    H^{n}(M, \mathcal{L}_{\omega})\times H_c^n(M, \mathcal{L}_{\omega}^{\vee}) \rightarrow \mathbb{C}
\end{equation*}
can be considered to define an intersection number of twisted cocycles.
Indeed, the calculation works similarly to the case of twisted cycles and was carried out explicitly in \cite{Ch95} such that it will enough to mention the  necessary definitions briefly and consider examples after this.\\
Here, we set $M=\mathbb{P}^1\setminus \lbrace x_0, \cdots, x_m\rbrace$, where the case $x_i=\infty$ is not excluded. We define

\begin{equation*}
    \omega := \sum_{j=0}^{m}\alpha_{j}\dfrac{du}{u-x_j}, ~~ \nabla_{\omega}:= d+\omega \wedge, ~~ \nabla_{\omega}^{\vee}:= d -\omega \wedge
\end{equation*}
and let $D$ be the divisor of $\mathbb{P}^1$ determined by the $m+1$ points $x_0, \cdots, x_m$, and denote the sheaf of logarithmic $p$-forms that may have logarithmic poles only along $D$ by $\Omega_{\mathbb{P}^1}^p(\log D)$ and the sub-sheaf of $\Omega_{\mathbb{P}^1}^p(\log D)$ of $p$-forms which have poles along $D$ of at least order $1$ by $\Omega_{\mathbb{P}^1}^p(\log D)(-D)$, then we have the 

\begin{theorem}
The following bilinear form is  non-degenerate
\begin{equation*}
    I: \Gamma \left(\mathbb{P}^1,\Omega_{\mathbb{P}^1}^{1}(\log D)\right) \times H^1(\mathbb{P}^1, \mathcal{O}_{\mathbb{P}^1}(-D)) \rightarrow \mathbb{C}.
\end{equation*}
\end{theorem}
This defines the intersection number

\begin{definition}[Intersection Number for twisted Cocycles]
We define the intersection number of $\varphi^{+}$, $\psi^{-}$ by

\begin{equation*}
    I_c(\varphi^{+}, \psi^{-}):= I(\varphi, i^{-1}(\psi^{-})),
\end{equation*}
where

\begin{equation*}
    \varphi, \psi \in \Gamma \left(\mathbb{P}^1,\Omega_{\mathbb{P}^1}^{1}(\log D)\right) \quad \text{and} \quad  \varphi^{+} = \varphi ~\text{mod} ~\mathbb{C}\cdot \omega, \quad \psi^{-} = \psi ~\text{mod} ~\mathbb{C}\cdot (-\omega)
\end{equation*}
and $i$ is the inclusion map

\begin{equation*}
    \left(\Omega^{\bullet}(\log D)(-D), \nabla_{\omega}^{\vee}\right) \xrightarrow{i}  \left(\Omega^{\bullet}(\log D), \nabla_{\omega}^{\vee}\right)
\end{equation*}
\end{definition}

\subsubsection{Calculating Examples - Using a different Notation}

Having given all necessary concepts to understand the notion of a twist, we finally want to calculate some examples. To do this as conveniently as possible, we want to introduce a more compact notation, i.e. the one used in the application-oriented paper \cite{Fr19}. There, the concept of  twisted cohomology etc. is not explicitly introduced. But having introduced all the necessary concepts and theorems, the reader can now always add the corresponding mathematical concepts to the more physicist-oriented paper.\\[2mm]
Let us mention in advance that in the following calculations the  manifold of interest is always $M= \mathbb{P}^1\setminus \mathcal{P}$, where $\mathcal{P}$ is the set of poles of the $1$-form $\omega$. Now, let us go through the most important definitions of \cite{Fr19}. We will use the notation and symbols of that paper such that it is easier for the reader to compare.\\
First they pair $\left\langle \varphi \right|$, which is a cocylce, and $\left| \mathcal{C} \right]$, which is a domain in $\mathbb{P}$,  to obtain an integral of the above form. More precisely, they define

\begin{equation*}
    \left\langle \varphi \right| \mathcal{C}] := \int_{\mathcal{C}} u \varphi.
\end{equation*}
 This defines a bilinear form in the spirit of the main theorem of section \ref{subsubsec: Twisted de Rham Theory} and hence can be used to find linear relations among hypergeometric functions.  To see this, let $\nu$ be the number of linearly independent cocycles, and denote an arbitrary basis of differential forms by

\begin{equation*}
    \left\langle e_1 \right|, \quad \left\langle e_2 \right|,  \cdots, \left\langle e_{\nu} \right|.
\end{equation*}
Note that we know from our considerations that the dimension $\nu$ of the space of cocylcles is finite.
Then, as known from linear algebra, a decomposition is achieved by expressing the arbitrary cocycle $\left\langle \varphi \right|$ as a linear combination of the base elements. This can be done as follows. We need a dual space of twisted cocycles, whose basis we want to denote by $\left| h_i \right\rangle$ with $i \in \lbrace 1,2,\cdots \nu \rbrace$. This will lead us to the definition of the intersection number from \cite{Fr19}\footnote{This is seen to agree with the definition of the intersection number of twisted cocylces we have given in \ref{subsubsec: Intersection Numbers for Twisted Cocycles}. }:

\begin{definition}[Intersection Number in this new Notation]
Let $\left\langle e_i \right|$ with $i \in \lbrace 1,2,\cdots, \nu \rbrace$ be a base element of the space of cocycles and $\left| h_i \right\rangle$ with $i \in \lbrace 1,2,\cdots,\nu \rbrace$ a basis element of a second space  dual to the first of the space of cocycles\footnote{From twisted de Rham Theory we know this dual space to exist.}, then we define the intersection numbers

\begin{equation*}
    \mathbf{C}_{ij}:= \left\langle e_i | h_j \right\rangle.
\end{equation*}
\end{definition}
These are the entries of a $\nu \times \nu$ matrix.

\subsubsection{Derivation of the Decomposition Formula}
\label{subsubsec: Derivation of the Decomposition Formula}

Having defined the intersection numbers, our next task is to find the decomposition formula, i.e. how to find the coefficients in the linear combination of $\left\langle \varphi \right|$ in terms of the basis elements. This is the most important formula in \cite{Fr19}. It is used in all examples there.\\ 
To derive the decomposition formula, let us define a $(\nu+1)\times (\nu+1)$ matrix $\mathbf{M}$  as follows:

\begin{equation*}
    \mathbf{M}=
    \begin{pmatrix}
    \left\langle \varphi | \psi \right\rangle &   \left\langle \varphi | h_1 \right\rangle &   \left\langle \varphi | h_2 \right\rangle & \cdots &   \left\langle \varphi | h_{\nu} \right\rangle \\
    \left\langle e_1 | \psi \right\rangle &   \left\langle e_1| h_1 \right\rangle &   \left\langle e_1 | h_2 \right\rangle & \cdots &   \left\langle e_1 | h_{\nu} \right\rangle \\
     \left\langle e_2 | \psi \right\rangle &   \left\langle e_2| h_1 \right\rangle &   \left\langle e_2 | h_2 \right\rangle & \cdots &   \left\langle e_2 | h_{\nu} \right\rangle \\
     \vdots & \vdots & \vdots & \ddots  & \vdots \\
      \left\langle e_{\nu} | \psi \right\rangle &   \left\langle e_{\nu}| h_1 \right\rangle &   \left\langle e_{\nu} | h_2 \right\rangle & \cdots &   \left\langle e_{\nu} | h_{\nu} \right\rangle \\
    \end{pmatrix} \equiv
    \begin{pmatrix}
     \left\langle \varphi | \psi \right\rangle & \mathbf{A}^{\intercal} \\
     \mathbf{B} & \mathbf{C}
    \end{pmatrix}
\end{equation*}
Each entry is given by a pairing i.e. a bilinear. The matrix $\mathbf{C}$ is  a submatrix of $\mathbf{M}$. $\mathbf{B}$ is columnvector, $\mathbf{A}^{\intercal}$ is a row vector.\\[2mm]
Since we have $\nu+1$ cocycles labelling the rows and columns and the corresponding vector spaces have dimension $\nu$, the determinant of $\mathbf{M}$ must vanish. Therefore, from the common formula for the determinant of a matrix with block matrices:

\begin{equation*}
    \operatorname{det}(\mathbf{M})= \operatorname{det}(\mathbf{C})\cdot \left( \left\langle \varphi | \psi \right\rangle - \mathbf{A}^{\intercal}\mathbf{C}^{-1}\mathbf{B}\right)=0.
\end{equation*}
But by definition the matrix $\mathbf{C}$ cannot be zero, whence the other factor must be zero. From this we conclude

\begin{equation*}
     \left\langle \varphi | \psi \right\rangle = \mathbf{A}^{\intercal}\mathbf{C}^{-1}\mathbf{B}
\end{equation*}
or writing out the products

\begin{equation*}
      \left\langle \varphi | \psi \right\rangle = \sum_{i,j =1}^{\nu} \left\langle \varphi | h_j\right\rangle \left(\mathbf{C}^{-1}\right)_{ji}\left\langle e_i | \psi\right\rangle
\end{equation*}
Therefore,  since $\left|\psi\right\rangle$ is arbitrary, we arrive at the the decomposition formula:

\begin{theorem}[Decomposition Formula]

Let $\left\langle e_i \right|$ with $i \in \lbrace 1,2,\cdots \nu \rbrace$ be a basis element of the space of cocycles and $\left\langle e_i \right|$ with $i \in \lbrace 1,2,\cdots \nu \rbrace$ a basis element of the dual space of the space of cocycles, then we have the decomposition formula
\begin{equation*}
      \left\langle \varphi \right| = \sum_{i,j =1}^{\nu} \left\langle \varphi | h_j\right\rangle \left(\mathbf{C}^{-1}\right)_{ji}\left\langle e_i \right|
\end{equation*}
\end{theorem}

The decomposition formula provides us with a projection of $\left\langle \varphi \right|$ onto the basis elements $\left\langle e_i \right|$. Contracting both sides with the twisted cocycle $\left|\mathcal{C}\right]$ (this means multiplying by $u$ and integrating over $\mathcal{C}$), we find the formula

\begin{corollary}
\begin{equation*}
    \int_{\mathcal{C}}u \varphi = \sum_{i,j =1}^{\nu} \left\langle \varphi | h_j\right\rangle \left(\mathbf{C}^{-1}\right)_{ji} \int_{\mathcal{C}} ue_i.
\end{equation*}
\end{corollary}
Therefore, we find the coefficients of the decomposition

\begin{corollary}
Given
\begin{equation*}
    I  = K\left\langle \varphi | \mathcal{C} \right]= \sum_{i=1}^{\nu} c_iJ_i,
\end{equation*}
where

\begin{equation*}
    J_i \equiv KE_i \quad \text{with} \quad E_i \equiv \left\langle e_i | \mathcal{C}\right],
\end{equation*}
we have

\begin{equation*}
    c_i = \sum_{j=1}^{\nu} \left\langle \varphi | h_j \right\rangle \left(\mathbf{C}^{-1}\right)_{ji}.
\end{equation*}
\end{corollary}

\subsubsection{The case of $1$-forms}
\label{subsubsec: The case of 1-forms}

In this section, we descend to $1$-forms and show how to calculate intersection numbers and to apply the decomposition formula. First, we need to find the dimension $\nu$ of the vector space under consideration. This is addressed by the following theorem.

\begin{theorem}
Let $\omega$ be a differential $1$-form defined on $X=\mathbb{P}^1\setminus \mathcal{P}$, where $\mathcal{P}$ is the set of poles of $\omega$. Then, the dimension $\nu$ of the vector space of cocycles is 
\begin{equation*}
    \nu = \lbrace \text{number of solutions of }\omega =0\rbrace.
\end{equation*}
\end{theorem}
Following \cite{Ch95} and \cite{Ma98}, we define the intersection number $\left\langle \varphi_L|\varphi_R\right\rangle_{\omega}$ as follows:

\begin{definition}[Intersection Number $\left\langle \varphi_L|\varphi_R\right\rangle_{\omega}$]
Let $\mathcal{P}$ be the set of poles of $\omega$, a meromphic $1$-form on $\mathbb{P}$, i.e

\begin{equation*}
    \mathcal{P}= \left\lbrace z | z \text{ is a pole of }\omega \right\rbrace
\end{equation*}
Then, we have:

\begin{equation*}
    \left\langle \varphi_L|\varphi_R\right\rangle_{\omega}=  \sum_{p \in \mathcal{P}} \operatorname{Res}_{z=p}\left(\psi_p \varphi_R\right),
\end{equation*}
where $\psi_p$ is a function and solves the differential equation $\nabla_{\omega}\psi = \psi_L$ around $p$, i.e. 

\begin{equation*}
    \nabla_{\omega_p}\psi_p = \varphi_{L,p}.
\end{equation*}
In general, $f_p$ denotes the Laurent expansion of the function $f$ around $p$.
\end{definition}

\subsubsection{1. Example: $B$-function}
\label{subsubsec: 1. Example: B-function}

 We will start with the simplest hypergeometric integral - the $B$-function, defined as the integral

\begin{equation*}
    B(x,y)= \int\limits_{0}^{1}dz z^{x-1}(1-z)^{y-1}.
\end{equation*}
Moreover, in section \ref{subsubsec: 3. Example: Beta-Function} we found the relation:

\begin{equation*}
    B(x,y)= \dfrac{\Gamma(x)\Gamma(y)}{\Gamma(x+y)}.
\end{equation*}
We will discuss the $B$-function in more detail to see how the method works in contrast to other more familiar methods discussed in this chapter.

\subsubsection*{1. Approach: Direct Integration}

Let us put

\begin{equation*}
    I_n = \int_{\mathcal{C}}uz^ndz, \quad u= z^{\gamma}(1-z)^{\gamma}, \quad \mathcal{C} =[0,1].
\end{equation*}
We can use the relation among $\Gamma$ and $B$ to express $I_n$ directly:

\begin{equation*}
    I_n= \dfrac{\Gamma(1+\gamma)\Gamma(1+\gamma +n)}{\Gamma(2+2 \gamma +n)}.
\end{equation*}
Therefore,

\begin{equation*}
    I_n=\dfrac{\Gamma(1+\gamma+n)\Gamma(2+2\gamma)}{\Gamma(1+\gamma)\Gamma(2+2\gamma +n)}I_0.
\end{equation*}
For $n=1$, this reduces to

\begin{equation*}
    I_1 = \dfrac{1}{2}I_0.
\end{equation*}

\subsubsection*{2. Approach: Integration-by-Parts}

Let us see, whether we can find the last relation among $I_1$ and $I_0$ by integration by parts. For this aim, note that

\begin{equation*}
    \int_{\mathcal{C}}d\left((z(1-z))^{\gamma +1}z^{n-1}\right) =0.
\end{equation*}
Expanding the integrand gives:

\begin{equation*}
    (\gamma +n)I_{n-1}-(1+2 \gamma +n)I_n=0.
\end{equation*}
Hence

\begin{equation*}
    I_n= \dfrac{(\gamma+n)}{(1+2\gamma +n)}I_{n-1},
\end{equation*}
which for $n=1$ gives

\begin{equation*}
    I_1= \dfrac{1}{2}I_0.
\end{equation*}

\subsubsection*{3. Approach: Intersection Numbers}

We want to find the relation among the $B$-integrals again. We define

\begin{equation*}
    I_n = \int u \phi_{n+1} \equiv _{\omega}\left\langle \phi_{n+1} |\mathcal{C}\right], \quad \phi_{n+1} \equiv z^ndz.
\end{equation*}
Additionally,

\begin{equation*}
    u=z^{\gamma}(1-z)^{\gamma}, \quad \omega = d \log u = \gamma \left(\dfrac{1}{z}+\dfrac{1}{z-1}\right)dz.
\end{equation*}
The equation $\omega = 0$ obviously has only $1$ solution. Hence the dimension of the space under consideration is $\nu=1$. Further, we need to find the poles of $\omega$. The poles $z=0$ and $z=1$ are immediate. To consider $z=\infty$, perform the substitution $t=\frac{1}{z}$, whence $dt=-\frac{1}{z^2}$. Therefore,

\begin{equation*}
    \omega = -\gamma \left(t+\dfrac{t}{1-t}\right)\dfrac{dt}{t^2}.
\end{equation*}
Since this tends to $\infty$ for $t \rightarrow 0$, the set of poles is:

\begin{equation*}
    \mathcal{P}= \lbrace 0,1,\infty \rbrace.
\end{equation*}
we want to express $I_1$ in terms of some integrals of the same class. Since $\nu =1$, we only need one integral and choose it to be $I_0$. Therefore, we know

\begin{equation*}
    I_1 =c_1I_0 \quad \Leftrightarrow \quad _{\omega}\left\langle \phi_2| \mathcal{C}\right] = c_1{}\cdot{}_{\omega}\left\langle \phi_1| \mathcal{C}\right]
\end{equation*}
and we need to find $c_1$.\\
Since $\nu=1$, the matrix $\mathbf{C}_{ij}$ is simply a number, i.e. $\left\langle \phi_1|\phi_1 \right\rangle$. Hence, in total we need to evaluate the numbers $\left\langle \phi_1|\phi_1 \right\rangle$, $\left\langle \phi_2|\phi_1 \right\rangle$.\\
For this we need the residues. For each pole $p \in \mathcal{P}$ we need $\phi_{i,p}$ - the series expansion of $\phi_i$ about $z=p$ - and $\psi_{i,p}$ - the series expansion of $\psi_i$ around $z=p$, which is found from the differential equation

\begin{equation*}
    \nabla_{\omega_p} \psi_{i,p}= \phi_{i,p}.
\end{equation*}
Note that the following Laurent expansions are known

\begin{equation*}
    \phi_{i,p}= \sum_{k = \operatorname{min}-1} \phi_{i,p}^{(k)}z^k \quad \text{and} \quad \omega_{p}= \sum_{k=-1} \omega_p^{k}z^k
\end{equation*}
and we need to find

\begin{equation*}
    \psi_{p} = \sum_{k=\operatorname{min}}^{\operatorname{max}}\alpha_k z^k.
\end{equation*}
In other words, the coefficients $\alpha_k$ are to be found. $\operatorname{max}(\phi_i)= \operatorname{ord}_p(\phi_i)+1$ and $\operatorname{min}(\phi_i)= -\operatorname{ord}_p(\phi_i)-1$. We introduced $\operatorname{min}$ and $\operatorname{max}$, since those determine the range of coefficients necessary to determine the residue we need. We do not need the rest of the Laurent series.\\
Writing out the above differential equation for differential forms, we arrive at the following simpler differential equation for functions:

\begin{equation*}
    \dfrac{d}{dz}\psi_p + \omega_p \psi_p = \phi_{i,p}.
\end{equation*}
Now we can begin with the calculation of the intersection numbers. We start with $\left\langle \phi_1|\phi_1 \right\rangle$. For $p=0$ and $p=1$ we find that $\operatorname{min}=1 > \operatorname{max}=-1$. Therefore, the resulting residues are zero. For $p=\infty$ on the other hand, we find $\operatorname{min}=-1$ and $\operatorname{min}=1$. Therefore, our series for $\psi_{\infty}$ has the following three terms:

\begin{equation*}
    \psi_{\infty}= \alpha_1\dfrac{1}{z}+\alpha_0 +\alpha_1 z^1.
\end{equation*}
Substituting this series in the above differential equations together with the series expansion of $\phi_{1,\infty}$ and $\omega_{\infty}$ and comparing coefficients, we find

\begin{equation*}
    \alpha_{-1}=\dfrac{1}{2\gamma +1}, \quad \alpha_0 =-\dfrac{1}{2(2\gamma +1)}, \quad \alpha_1=\dfrac{\gamma}{2(2\gamma-1)(2\gamma +1)}.
\end{equation*}
Therefore, the intersection number is 

\begin{equation*}
    \left\langle \phi_1|\phi_1 \right\rangle = \operatorname{Res}_{z=\infty}(\psi_{\infty}\phi_1)= \dfrac{\gamma}{2(2\gamma -1)(2\gamma +1)}.
\end{equation*}

Let us go over to $\left\langle \phi_2|\phi_1 \right\rangle$. As in the first case, the differential equation has no solution for $p=0$ and $p=1$. For $p=\infty$ we  find that the series must look as follows:

\begin{equation*}
    \psi_{\infty}= \alpha_{-2}\dfrac{1}{z^2}+ \alpha_{-1}\dfrac{1}{z}+ \alpha_{0}+ \alpha_1 z.
\end{equation*}
The coefficients are found by the same procedure as above and are calculated to be

\begin{equation*}
    \alpha_{-2}=\dfrac{1}{2(\gamma +1)}, \quad \alpha_{-1}=\dfrac{\gamma}{2(\gamma +1)(\gamma +2)}, \quad \alpha_0 =\dfrac{1}{4(2\gamma+1)}, \quad \alpha_1=-\dfrac{\gamma}{4(2\gamma-1)(2 \gamma+1)}.
\end{equation*}
Thus,

\begin{equation*}
    \left\langle \phi_2|\phi_1 \right\rangle = \operatorname{Res}_{z=\infty}(\psi_{\infty}\phi_1)=\dfrac{\gamma}{4(2\gamma-1)(2\gamma +1)}.
\end{equation*}
Now we can apply the decomposition formula and find

\begin{equation*}
    c_1= \left\langle \phi_2|\phi_1 \right\rangle (\left\langle \phi_1|\phi_1 \right\rangle)^{-1}=\dfrac{1}{2},
\end{equation*}
i.e.

\begin{equation*}
    I_1 = \dfrac{1}{2}I_0
\end{equation*}
in agreement with the approaches above.

\subsubsection{2. Example: Hypergeometric Function}
\label{subsubsec: 2. Example: Hypergeometric Function}

Above  in section \ref{subsubsec: 4. Example: Hypergeometric Series} we saw that we can express the Gaussian hypergeometric series as

\begin{equation*}
    B(b,c-b){}_2F_1(a,b,c;x) = \int\limits_{0}^{1}z^{b-1}(1-z)^{c-b-1}(1-xz)^{-a}dz.
\end{equation*}
Hence the integration contour is $\mathcal{C}=[0,1]$. $B(a,b)$ is the $B$-function. We want to use intersection theory to find a difference equation for hypergeometric series. We write:

\begin{equation*}
    B(b,c-b)_2F_{1}(a,b,c;x) = \int_{\mathcal{C}} u\varphi = {}_{\omega}\left\langle \varphi | \mathcal{C}\right],
\end{equation*}
with the letters $u$, $\omega$ and $\varphi$ being

\begin{equation*}
     \renewcommand{\arraystretch}{2,5}
\setlength{\arraycolsep}{0.0mm}
\begin{array}{lll}
    u &~=~& z^{b-1}(1-xz)^{-a}(1-z)^{-b+c+1} \\
    \omega &~=~& d \log u = \dfrac{xz^2(c-a-2)+z(ax-c+x+2)-bxz+b-1}{(z-1)z(xz-1)}dz \\
    \varphi &~=~& dz.
\end{array}
\end{equation*}
Therefore, we have

\begin{equation*}
    \nu =2 \quad \text{and} \quad \mathcal{P}=\left\lbrace 0,1,\frac{1}{x}, \infty\right\rbrace.
\end{equation*}
This indicates that we can express one hypergeometric integrals by two others. This is not surprising, since we know Gau\ss's contiguous  relations. But let us find a specific one by using intersection theory. \\[2mm]
We chose the basis $\lbrace \left\langle \phi_{1}\right|, \left\langle \phi_{2}\right|\rbrace$, where, as above $\phi_{i+1}=z^{i}dz$.\\
In this example, the matrix $\mathbf{C}$ is a $2\times 2$ matrix and looks as follows:

\begin{equation*}
    \mathbf{C}= \begin{pmatrix}
    \left\langle \phi_1 | \phi_1 \right\rangle &   \left\langle \phi_1 | \phi_2 \right\rangle \\
      \left\langle \phi_2 | \phi_1 \right\rangle &   \left\langle \phi_2 | \phi_2 \right\rangle
    \end{pmatrix}
\end{equation*}
The intersection numbers are calculated as above, we just list the results:

\begin{equation*}
         \renewcommand{\arraystretch}{1,5}
\setlength{\arraycolsep}{0.0mm}
\begin{array}{lll}
 \left\langle \phi_1 | \phi_1 \right\rangle & ~=~ & \bigg(x^2(-(a-b+1))(b-c+1)-2ax(-b+c-1)+a(c-2)\bigg)/\bigg(x^2(a\\
 &~-~&c+1)(a-c+2)(a-c+3) \\
  \left\langle \phi_1 | \phi_2 \right\rangle &~=~& \bigg(x^3(-(a-b+1)(a-b+2)(b-c+1))-ax^2(-b+c+1)(2a-3b \\
   &~+~& c+2)+ax(a+2c-5)(-b+c+1)-a(c-3)(c-2)\bigg)/\bigg(x^3(a-c+1) \\
  & & (a-c+2)(a-c+3)(a-c+4)\bigg), \\
    \left\langle \phi_2 | \phi_1 \right\rangle &~=~ & \bigg(x^3(-(a-b))(a-b+1)(b-c+1)-ax^2(-b+c-1)(2a-3b+c)\\
    &~+~ & ax(a+2c-3)(-b+c+1)-a(c-2)(c-1)\bigg)/\bigg( x^3(a-c)(a-c+1) \\
    & &(a-c+2)(a-c+3)\bigg),\\
     \left\langle \phi_2 | \phi_2 \right\rangle &~=~ & \bigg(-ax^2(a^2b-a^2c+a^2-3ab^2+7abc-8ab-4ac^2 +9ac-5a-3b^2c \\
     &~+~& 6b^2+4bc^2-10bc+6b-c^3+2c^2-c)+x^4(-(a^3-3a^2b +3a^2 +3ab^2 \\
     &~-~& 6ab+2a-b^3+3b^2-2b))(b-c+1)+2ax^3(a-b+1)(ab-ac+a \\
     &~-~& 2b^2+3bc-2b-c^2 +c)+2a(c-2)x(a+c-2)(b-c+1)+a(c^3-6c^2 \\
     &~+~& 11c-6)\bigg)/\bigg(x^4(a-c)(a-c+1)(a-c+2)(a-c+3)(a-c+4)\bigg).
\end{array}
\end{equation*}
Now we found everything necessary for the application of the decomposition formula, which in this case reads

\begin{equation*}
    \left\langle \phi_{n} \right| = \sum_{i,j=1}^{2}  \left\langle \phi_n | \phi_j \right\rangle \left(\mathbf{C}^{-1}\right)_{ji} \left\langle \phi_i \right|.
 \end{equation*}
As an example, let us take $\left\langle \phi_3 | \mathcal{C}\right]= B(b+2,c-b){}_2F_1(a,b+2,c+2;x)$ in terns of $B(b,c-b){}_2F_1(a,b,c;x)$ and $B(b+1,c-b){}_2F_1(a,b+1,c+1;x)$. Then, we find

\begin{equation*}
 \renewcommand{\arraystretch}{2,5}
\setlength{\arraycolsep}{0.0mm}
\begin{array}{lll}
    B(b+2,c-b){}_2F_1(a,b+2,c+2;x)&~=~& \left(\dfrac{b}{x(a-c-1)}\right)B(b,c-b){}_2F_1(a,b,c;x) \\
    &~+~ &\left(\dfrac{(b-a+1)x+c}{x(c-a+1)}\right)B(b+1,c-b){}_2F_1(a,b+1,c+1;x)
    \end{array}
\end{equation*}
which relation can also be derived from the contiguous relations for hypergeometric functions.

\subsubsection{3. Example: $\Gamma$-function}
\label{subsubsec: 3. Example: Gamma-function}

Having explained the idea, let us apply the idea to the $\Gamma$-function and recover its integral representation from its functional equation one more time. We still assume that it is possible to write $\Gamma(x)$ as an integral of the form

\begin{equation*}
    \Gamma(x) = \int\limits_{0}^{\infty}t^{x-1}P(t)dt
\end{equation*}
and we have to determine $P(t)$. We want to solve $\Gamma(x+1)= x\Gamma(x)$. In the language of intersection theory, we have

\begin{equation*}
    u = t^{x}P(t) \quad \text{and hence} \quad \omega = d \log u = \left(\dfrac{x}{t}+\dfrac{P'(t)}{P(t)}\right)dt.
\end{equation*}
Since we seek to express $\Gamma(x+1)$ by one other integral, this forces us to arrange that

\begin{equation*}
    \dfrac{x}{t}+\dfrac{P'(t)}{P(t)} =0
\end{equation*}
has precisely one solution for $t$. This in turn implies

\begin{equation*}
    \dfrac{P'(t)}{P(t)}=C \quad \text{or} \quad P(t)=Be^{Ct}.
\end{equation*}
for some constants $B$ and $C\neq 0$ we have to determine from the remaining conditions. We need to find the poles of $\omega$, i.e. the $1$-form defined above\footnote{Of course, at this point we could use Wielandt's theorem directly to force the integral to become $\Gamma(x)$ - it would imply $B=1$ and $C=-1$. But we will use intersection theory to the end and see how we have to find $B$ and $C$ his way.}. Inserting the result we found for $P(t)$, we have

\begin{equation*}
    \omega = \left(\dfrac{x}{t}+C\right)dt.
\end{equation*}
One pole is obviously given by $t=0$.  But we also have to consider $t=\infty$. Therefore, put $u=\frac{1}{t}$. Hence locally around $t= \infty$

\begin{equation*}
    \omega =  -\left(xu+C\right)\dfrac{du}{u^2}
\end{equation*}
which is infinite for $u=0$ and hence indicates a pole at $t=\infty$. Therefore,

\begin{equation*}
    \mathcal{P}= \lbrace 0, \infty\rbrace.
\end{equation*}
As in the first example, we need the intersection numbers $\left\langle \phi_1| \phi_1 \right\rangle$ and $\left\langle \phi_2| \phi_1 \right\rangle$. Since we forced $\nu=1$, the intersection matrix is just $\mathcal{C}_{11}=\left\langle \phi_1| \phi_1 \right\rangle$, i.e a $1\times 1$-matrix. The two necessary intersection numbers are calculated as in the preceding examples. As in the first example only the point $p=\infty$ actually contributes and we find:

\begin{equation*}
    \renewcommand{\arraystretch}{2,5}
\setlength{\arraycolsep}{0.0mm}
\begin{array}{lllllll}
     \left\langle \phi_1| \phi_1 \right\rangle  &~=~ & \dfrac{x}{C^2} \\
     \left\langle \phi_2| \phi_1 \right\rangle  &~=~ & -\dfrac{x^2}{C^3} \\
\end{array}
\end{equation*}
Therefore, the decomposition formula yields

\begin{equation*}
    c_1= \left\langle \phi_2| \phi_1 \right\rangle   \left\langle \phi_1| \phi_1 \right\rangle^{-1} = -\dfrac{x}{C}.  
\end{equation*}
Recall that we want to solve the equation

\begin{equation*}
    \Gamma(x+1)=x\Gamma(x).
\end{equation*}
Therefore, this gives the condition

\begin{equation*}
    -\dfrac{x}{C}=x \quad \Rightarrow \quad C=-1.
\end{equation*}
Hence, we finally arrive at the solution:

\begin{equation*}
    \Gamma(x) = B \int\limits_{0}^{\infty}t^{x-1}e^{-t}dt.
\end{equation*}
$B$ is a constant different from $0$. Using the initial condition $\Gamma(1)=1$, one finds the integral representation of $\Gamma(x)$ again.

\subsubsection{Remarks}
\label{subsubsec: Remarks}

This last example shows that intersection theory can also be used backwards to solve certain difference equations, although it is quite a lot of work compared to Euler's method. But note the difference, how we found the function $P(t)$ here in contrast to the other methods. We recovered the differential equation for $P$ by requiring an equation to have only one solution in $t$. This is quite different from Euler's method or the inverse Mellin-Transform. Furthermore, one has to insert the integration limits from the beginning in contrast to Euler's method, where you calculate them explicitly. But since we showed above in section \ref{subsec: Euler and the Mellin Transform} how to circumvent this problem, this is not to be considered an defect of intersection theory in this regard. \\
Therefore, we have now presented three methods to solve homogeneous difference equations with linear difference equations, where intersection theory is clearly the most general.

\newpage

\section{Solution of the Equation $\log \Gamma(x+1)- \log  \Gamma (x)= \log x$ by Conversion into a Differential Equation of infinite Order}
\label{sec: Solution of the Equation log Gamma(x+1)- log  Gamma (x)= log x by Conversion into a Differential Equation of infinite Order}

\subsection{Overview}
\label{subsec: Overview}
We solve the general difference equation

\begin{equation*}
    f(x+1)-f(x)=g(x)
\end{equation*}
by converting it into a differential equation of infinite order via Taylor's theorem. This was done by Euler in 1753 in \cite{E189}\footnote{He considered other examples of differential equations of infinite order in \cite{E62} and his second book on integral calculus \cite{E366}.}. Unfortunately, there is an error in Euler's approach that we will explain in section \ref{subsubsec: Mistake in Euler's Approach} and correct in section \ref{subsec: Correction of Euler's Approach}, before we solve the equation by the more modern approach of Fourier analysis in section \ref{subsubsec: Solution Algorithm for the general difference equation - Finding a particular solution}. But we will also present a solution that Euler could have given in section \ref{subsec: A Solution Euler could have given}, and present Euler's, corrected, derivation of the Stirling-Formula given in \cite{E189} in section \ref{subsubsec: An Application - Derivation of the Stirling Formula for the Factorial}.

\subsection{Euler's Idea}
\label{subsec: Eulers Idea}

Euler's reasoning involves some purely formal operations. Therefore, we will not try to give the conditions under which the operations are valid here. Nevertheless, it is  a beautiful example of the "Ars Inveniendi" (Art of Finding).

\subsubsection{Presentation of his Idea}
\label{subsec: Presentation of his Idea}

As mentioned, the main source for this section is  \cite{E189}. Here, Euler actually intended to solve various interpolation problems. One of them is to interpolate the function $f$ defined for positive integers:

\begin{equation*}
    f(n) := \sum_{k=0}^{n-1} g(k),
\end{equation*}
$g$ being an arbitrary function. Obviously, $f$ satisfies the functional equation:

\begin{equation*}
    f(n+1)-f(n)= g(n) \quad \forall n \in \mathbb{N}.
\end{equation*}
One possible way to interpolate $f$ is to solve the above difference equation for general $n \in \mathbf{C}$  which was Euler's intention  \cite{E189}\footnote{Euler actually assumed $n$ to be real, but this restriction is not necessary at all.}. For this he used Taylor's theorem to write, now for $x\in \mathbb{C}$:

\begin{equation*}
    f(x+1) = \sum_{n=0}^{\infty} \dfrac{f^{(n)}(x)}{n!}.
\end{equation*}
Therefore, the above difference equation becomes:

\begin{equation*}
     \sum_{n=1}^{\infty} \dfrac{f^{(n)}(x)}{n!} =g(x).
\end{equation*}
This can be seen as an inhomogeneous ordinary differential equation of infinite order with constant coefficients.\\
In \cite{E62} and \cite{E188}, Euler explained a method to solve such differential equations, if the order is finite\footnote{\cite{E62} considers the homogeneous case, whereas \cite{E188} treats the inhomogeneous case}. That method is still the same we use today and can be found in any modern textbook on differential equations.

\subsubsection{Actual Solution}
\label{subsubsec: Actual Solution}

Let us present Euler's solution. This first step to solve an equation of such a kind (at least in the case of finite order) is to consider the characteristic polynomial, i.e., the polynomial resulting by substituting $\frac{d^n}{dx^n}$ for $z^n$ in the differential operator acting on $f$. This "polynomial" in our case reads:

\begin{equation*}
    P(z)= e^z-1.
\end{equation*}
Next, Euler wants to find the zeros of $P$. Using the theory of complex logarithms, which  he had developed in \cite{E168} and \cite{E807}, he found the zeros to be

\begin{equation*}
    z_k = 2 k \pi i, \quad k \in \mathbb{Z}.
\end{equation*}
All zeros are easily seen to be simple. From this Euler (assuming that the case of infinite order can be treated as the case of finite order) concluded that each zero will lead to a term 

\begin{equation*}
    e^{z_k x}\int\limits_{}^{x} e^{-z_k t}g(t)dt.
\end{equation*}
$\int\limits_{}^{x}$ means that we have to put $x=t$ after the integration. Therefore, Euler claimed that the solution of the general difference equation reads

\begin{equation*}
    f(x) = \sum_{k= -\infty}^{\infty} e^{2 k \pi i x}\int\limits_{}^{x} e^{-2 k \pi i t}g(t)dt.
\end{equation*}
Note that each integral leads to an integration constant $c_k$, whence this solution is not  a particular but the complete solution.

\subsubsection{Mistake in Euler's Approach}
\label{subsubsec: Mistake in Euler's Approach}

Unfortunately, Euler's solution is incorrect. It gives the correct solution only in the homogeneous case. One finds

\begin{equation*}
    f(x) = \sum_{k= -\infty}^{\infty}c_k e^{2 k \pi ix}
\end{equation*}
as solution of $f(x+1)=f(x)$. $c_k$ are arbitrary constants of integration. Interestingly, Euler found that each periodic function has Fourier series. But Euler did not realize what a broad field of mathematics he had entered and did not pursue this any further. \\
But his solution formula already fails to give the correct result in the case $g(t)=1$. The correct formula, as we will prove in the following, reads:

\begin{equation*}
    f(x) =- \dfrac{1}{2}g(x)+ \sum_{k= -\infty}^{\infty}c_k e^{2 k \pi ix},
\end{equation*}
or, presented in the more convenient form,

\begin{equation*}
    f(x) = \int\limits_{}^{x}g(t)dt -\dfrac{1}{2}g(x)+ \sum_{k \in \mathbb{Z}\setminus \lbrace 0\rbrace} e^{2 k \pi ix}\int\limits_{}^{x} e^{-2 k \pi i t}g(t)dt.
\end{equation*}
Furthermore, Euler's mistake is not a computational but a conceptual one. His approach to construct the solution from the zeros of the characteristic "polynomial" (in analogy to the finite case where this is possible) simply does not work in the case of infinite order. Instead of the zeros one has to use the partial fraction decomposition\footnote{The following partial fraction decomposition is proved in the appendix in section \ref{subsec: Euler and the Partial Fraction Decomposition of Transcendental Functions}}:

\begin{equation*}
\dfrac{1}{e^z-1}    = -\dfrac{1}{2}+\dfrac{1}{z} + \sum_{k \in \mathbb{Z}\setminus \lbrace 0\rbrace} \dfrac{1}{z - 2 k \pi i}.
\end{equation*}
Comparing this to the solution it is easily seen that each term $\frac{1}{z -2 k \pi i}$ leads to a term $ e^{2 k \pi ix}\int\limits_{}^{x} e^{-2 k \pi i t}g(t)dt$ in the solution, whereas $-\frac{1}{2}$ explains the term $-\frac{1}{2}g(x)$. 
We will prove this in the following section, but need to give some definitions and state some auxiliary theorems in advance.

\subsection{Correction of Euler's Approach}
\label{subsec: Correction of Euler's Approach}

\subsubsection{Preparations - Lemmata and Definitions}
\label{subsubsec: Preparations - Lemmata and Definitions}

\begin{theorem}[Fundamental Theorem of Algebra] Let $P(z)=a_0+a_1z+\cdots +a_nz^n$ be a non-constant polynomial of degree $n$ with complex coefficients, then $P(z)$ has exactly $n$ zeros in $\mathbb{C}$, where the zeros have to be counted with multiplicity.
\end{theorem}
This theorem is usually proved by applying Liouville's theorem. We refer the reader to any modern book on complex analysis for a proof and mention \cite{Fr06} as an example.\\[2mm]
As a historical note, we add that Euler also tried to prove the fundamental theorem of Algebra in \cite{E170}. But his proof is incomplete, as pointed out by Gau\ss{} in his 1799 dissertation \cite{Ga99}. We refer to \cite{Du91} for a review of Euler's  paper \cite{E170} from the modern perspective.\newline \\[2mm]
But  let us go over to the correction of Euler's approach concerning the solution of the difference equation. We need to introduce some definitions.
\begin{definition}[Schwartz Space]

 We denote by $\mathcal{S}$ the set of all functions $f \in C^{\infty}(\mathbb{R}^n)$ such that

\[
\sup_{x\in \mathbb{R}^n} \max_{|\alpha|, |\beta| < N \in \mathbb{N}} |x^{\alpha}D^{\beta}f(x)  | < \infty
\]
holds. The linear space $\mathcal{S}$ endowed with the convergence
\[
f_j \rightarrow 0: \Leftrightarrow \sup_{x\in \mathbb{R}^n} \max_{|\alpha|, |\beta| < N \in \mathbb{N}} |x^{\alpha}D^{\beta}f_j(x)  | \rightarrow 0
\]
is called the Schwartz space. $D$ is the differentiation operator and multi-index notation is used.
\end{definition}
Next, we introduce the Fourier transform and state the Fourier inversion formula.

\begin{definition}[Fourier Transform and Fourier Inversion Formula]
 
We define the Fourier transformation of a function $f \in L_1(\mathbb{R}^n)$ as:

\[
\widehat{f}(p):= \int\limits_{\mathbb{R}^n}e^{-ix \cdot p}f(x)dx,
\]
where $x\cdot p= \sum_{i=1}^n x_kp_k$.\newline
If also $\widehat{f}$ is integrable, the following formula, the Fourier inversion formula holds:

\[
f(x)=\frac{1}{(2\pi)^n} \int\limits_{\mathbb{R}^n}e^{ix \cdot p}\widehat{f}(p)dp.
\]
\end{definition}
Next, we state a theorem, often very useful in calculations involving Fourier transforms.

\begin{theorem}[Convolution Theorem]

Let $\mathcal{F}$ be the operator of the Fourier transform, that $\mathcal{F}\{f\}=\widehat{f}$ and $\mathcal{F}\left\lbrace g\right\rbrace =\widehat{g}$ are the Fourier transforms of the functions $f$ and $g \in \mathcal{S}$, then we have

\[
\mathcal{F}\left\lbrace f*g\right\rbrace =(2\pi)^{\frac{n}{2}}\mathcal{F}\left\lbrace f\right\rbrace \cdot \mathcal{F}\left\lbrace g\right\rbrace \quad \text{and} \quad (2\pi)^{\frac{n}{2}}\mathcal{F}\left\lbrace f \cdot g\right\rbrace =\mathcal{F}\left\lbrace f\right\rbrace * \mathcal{F}\left\lbrace g\right\rbrace,
\]
where $(f*g)(x)$ means the convolution product defined via

\[
(f*g)(x):= \int_{-\infty}^{\infty}f(\tau)g(x-\tau)d\tau =\int_{-\infty}^{\infty}f(x-\tau)g(\tau)d\tau
\]
\end{theorem}
The theorem is straight-forward by a simple application of Fubini's theorem, so that we omit it here. One finds a proof in most textbook that covers the Fourier transform. We refer to \cite{Ti48}.

\begin{theorem}[Convolution Theorem for the inverse Fourier Transform]
Let $\mathcal{F}^{-1}$ be the operator of the inverse Fourier transform, so that $\mathcal{F}^{-1}\left\lbrace f\right\rbrace$ and $\mathcal{F}^{-1}\left\lbrace g\right\rbrace$ the inverse Fourier transforms $f$ and $g \in \mathcal{S}$, then we have

\[
f*g=(2\pi)^{\frac{n}{2}}\mathcal{F}^{-1}\left\lbrace\mathcal{F}\left\lbrace f\right\rbrace \cdot \mathcal{F}\left\lbrace g\right\rbrace\right\rbrace \quad \text{and} \quad (2\pi)^{\frac{n}{2}}f \cdot g=\mathcal{F}^{-1}\left\lbrace\mathcal{F}\left\lbrace f\right\rbrace * \mathcal{F}\left\lbrace g\right\rbrace\right\rbrace
\]
\end{theorem}
 The proof is just by applying the inverse Fourier transform to the equations in the convolution theorem. \newline \\[2mm]
Before we  present the solution algorithm for the general difference equation, we need two more properties of the Fourier Transform. First that the Fourier transform changes differential operators $D$ into $-ip$, i.e. an algebraic object. More precisely, we have:

\begin{theorem}
Let $f \in \mathcal{S}(\mathbb{R}^n)$, then

\begin{equation*}
    \mathcal{F}(D^{\alpha}f)(x) = (ip)^{\alpha}\widehat{f}(p).
\end{equation*}
$\alpha$ denotes a multi-index again.
\end{theorem}
The proof is just by integration by parts. Furthermore, we have

\begin{theorem}
If $f \in \mathcal{S}(\mathbb{R}^n)$, then $\widehat{f} \in \mathcal{S}(\mathcal{R}^n)$.
\end{theorem}
This is an easy application of the fact that the Fourier transform interchanges differentiation and multiplication. Finally, we have the following property of the inverse Fourier transform:

\begin{theorem}
The Fourier transform is a bijective mapping on the Schwartz space.
\end{theorem}
 The proof of this (and the two previous theorems) can be found, e.g., in \cite{St02}.

\subsubsection{Solution Algorithm for the general difference equation - Finding a particular solution}
\label{subsubsec: Solution Algorithm for the general difference equation - Finding a particular solution}

Having stated all the definitions and theorems, we can now finally give the solution for the general difference equation with constant coefficients.

\begin{theorem}
Let the following equation be propounded:

\[
\sum_{k=0}^{N}a_kf(x+k)=g(x),
\]
with $f$ and $g \in \mathcal{S}(\mathbb{R})$ and let denote $z_k$ the zeros, which we, for the sake of brevity, assume to be simple, of the polynomial \footnote{The fundamental theorem of algebra guarantees that we always have exactly $N$ solutions.}

\[
P(z)=\sum_{k=0}^{N}a_kz^k.
\]
And let $p_l$ be a solutions of the equation

\[
e^{i p_l}=z_k.
\]
Further, find the partial fraction decomposition:

\[
\sum_{k=0}^{N}\left[a_k e^{+i p \cdot k}\right]^{-1}=C+\sum_{l\in \mathbb{Z}}\frac{b_l}{i(p-p_l)}.
\]
Then, a particular solution of the difference equation is given by

\[
f(x)=Cg(x)+\sum_{l\in \mathbb{Z}}b_l \int\limits_{}^{x} e^{-ip_l(x-\tau)}g(\tau)d\tau,
\] 
where $\int\limits_{}^{x}f(\tau)d\tau$ means that we integrate with respect to $\tau$ and then write  $x$ for $\tau$ after the integration and the constant of integration is to be kept. This solution is unique up to the addition of a function, satisfying the equation:

\[
\sum_{k=0}^{N}a_kf(x+k)=0.
\]
\end{theorem}
\begin{proof}

To give a proof, it is easier to start from the solution and derive the difference equation. We do not present the calculation in detail. All steps are justified, by using the theorems in the preparations. With $p_k$ defined as above, consider

\[
f(x):= Cg(x)+\sum_{l \in \mathbb{Z}}b_le^{ip_l x}\int\limits_{}^{x}e^{-ip_l\tau}g(\tau)d\tau.
\]
Taking the Fourier transform of this expression, simplifying it by using the convolution theorem and the inverse convolution theorem\footnote{In section \ref{subsubsec: Explicit Solution of the general Difference Equation via Fourier-Transform} we will also present a solution of the simple difference equation $f(x+1)-f(x)=g(x)$ without resorting to the convolution theorem.}, we arrive at the following expression:

\[
\widehat{f}(p)=\widehat{g}(p)\left(C+\sum_{l \in \mathbb{Z}} b_l\frac{1}{i(p-p_l)}\right).\\
\]
The expression in brackets is just the partial fraction decomposition of the polynomial in $e^{i pk}$, hence we have

\[
\widehat{f}(p)=\widehat{g}(p)\left[\sum_{k=0}^{N}a_k e^{ipk}\right]^{-1}.
\]
Solving for $\widehat{g}(p)$ gives

\[
\widehat{g}(p)=\widehat{f}(p)\left[\sum_{k=0}^{N}a_k e^{-ipk}\right].
\]
Finally, taking the inverse Fourier transform, we get

\[
g(x)=\sum_{k=0}^Na_kf(x+k),
\]
which is the difference equation propounded and completes our proof. That we can add a function satisfying 

\[
\sum_{k=0}^{N}a_kf(x+k)=0,
\]
is obvious.
\end{proof}

\subsubsection{Case of multiple Roots}
\label{subsubsec: Case of multiple Roots}

 Although we only proved the theorem for simple zeros of the polynomial, we can directly generalize it to multiple zeros. Suppose, that $z_k$ is a multiple zero of order $m$ of the polynomial $P(z)$ defined above. Then in the formula

\[
\sum_{k=0}^{N}\left[a_ke^{i p\cdot k}\right]^{-1}=C+\sum_{l\in \mathbb{Z}}\frac{b_l}{i(p-p_l)}
\]
we just have to replace

\[
\frac{b_l}{p-p_l} \quad \text{by} \quad \sum_{j=1}^m \frac{b_j}{(p-p_l)^j}.
\]
In general, we have that

\[
\frac{j!}{(ip-ip_l)^{j+1}}\widehat{g(p)} 
\]
leads to the term

\[
\int\limits_{}^{x}(\tau-x)^{j}e^{ip_l(\tau-x)}g(\tau)d \tau
\]
in the final solution.

\subsubsection{Application to the difference equation $f(x+1)-f(x)=g(x)$.}
\label{subsubsec: Application to the difference equation f(x+1)-f(x)=g(x)}

The difference equation we are interested in is a special case of the theorem we just proved. Here, 

\begin{equation*}
    P(z)= z-1.
\end{equation*}
Considering the solutions of

\begin{equation*}
    e^{i p_l}=z_k=1,
\end{equation*}
$p_l$ is found to satisfy:

\begin{equation*}
    ip_l = 2 k \pi i \quad k \in \mathbb{Z}.
\end{equation*}
Therefore, we need to find the partial fraction decomposition of

\begin{equation*}
    \dfrac{1}{e^{ip}-1}
\end{equation*}
which (comparing to the result we mentioned above in \ref{subsubsec: Mistake in Euler's Approach}) is found to be

\begin{equation*}
    -\dfrac{1}{2}+ \sum_{l = \infty}^{\infty} \dfrac{1}{ip-2l \pi i}
\end{equation*}
or in terms of $p_l$, see also section \ref{subsec: Euler and the Partial Fraction Decomposition of Transcendental Functions}:

\begin{equation*}
     \dfrac{1}{e^{ip}-1} =  -\dfrac{1}{2}+ \sum_{l = \infty}^{\infty} \dfrac{1}{i(p-p_l)}.
\end{equation*}
Therefore, the solution of the difference equation is concluded to be

\begin{equation*}
    f(x) = -\dfrac{1}{2} + \sum_{l = \infty}^{\infty} e^{2 l \pi i x}\int\limits_{}^{x} e^{-2 \pi i l t}g(t)dt.
\end{equation*}
Precisely, as we stated the solution above in section \ref{subsubsec: Mistake in Euler's Approach}.

\subsubsection{Explicit Solution of the general Difference Equation via Fourier-Transform}
\label{subsubsec: Explicit Solution of the general Difference Equation via Fourier-Transform}

Above in section \ref{subsubsec: Solution Algorithm for the general difference equation - Finding a particular solution} we just explained how to proceed in general. Therefore, we want to add the explicit calculation in the example of the simple difference equation, i.e.

\begin{equation*}
    f(x+1)-f(x)=g(x).
\end{equation*}
Assuming the function $f,g$ have all the necessary properties, i.e. are $\in \mathcal{S}(\mathbb{R})$, we now present the solution of the above equation. First, we rewrite this equation as

\begin{equation*}
    \sum_{n=1}^{\infty} \dfrac{f^{(n)}(x)}{n!}=g(x)
\end{equation*}
using Taylor's theorem and apply the Fourier-Transform; then, our equation becomes

\begin{equation*}
    \left(e^{ip}-1\right)\widehat{f}(p)= \widehat{g}(p).
\end{equation*}
Thus,

\begin{equation*}
    \widehat{f}(p)= \dfrac{1}{e^{ip}-1}\widehat{g}(p).
\end{equation*}
Using the partial fraction decomposition of $\frac{1}{e^z-1}$ given in \ref{subsubsec: Mistake in Euler's Approach}, we find

\begin{equation*}
    \widehat{f}(p)=\left(\dfrac{1}{ip}-\dfrac{1}{2}+\sum_{k \in \mathbb{Z}\setminus \lbrace 0\rbrace} \dfrac{1}{ip-2k \pi i}\right)\widehat{g}(p).
\end{equation*}
Next, we want to take the inverse Fourier transform of the last equation. The inverse Fourier transform of $-\frac{1}{2}\widehat{g}(p)$ is simply $-\frac{1}{2}g(x)$. Therefore, let us consider

\begin{equation*}
    \mathcal{F}^{-1}\left(\dfrac{1}{ip}\widehat{g}(p)\right)=\dfrac{1}{2\pi} \int\limits_{-\infty}^{\infty} \dfrac{1}{ip}\widehat{g}(p)e^{ixp}dp
\end{equation*}
Now note that one can write
\begin{equation*}
    \int\limits_{i\infty}^{x}e^{iyp}dy =\left. \dfrac{e^{iyp}}{ip} \right|_{i\infty}^{x}= \dfrac{e^{ixp}}{ip}.
\end{equation*}
Inserting this into the above formula, we have

\begin{equation*}
     \mathcal{F}^{-1}\left(\dfrac{1}{ip}\widehat{g}(p)\right)= \dfrac{1}{2\pi} \int\limits_{-\infty}^{\infty} \left(\int\limits_{i\infty}^{x} e^{iyp}dy \right)\widehat{g}(p)dp.
\end{equation*}
Applying Fubini's theorem, this becomes

\begin{equation*}
    =\dfrac{1}{2\pi} \int\limits_{i \infty}^{x} \left(\int\limits_{-\infty}^{\infty}\widehat{g}(p)e^{iyp}dp\right)dy.
\end{equation*}
The inner integral times $\frac{1}{2\pi}$ is just the inverse Fourier transform of $\widehat{g}(p)$. Therefore, in total

\begin{equation*}
     \mathcal{F}^{-1}\left(\dfrac{1}{ip}\widehat{g}(p)\right)= \int\limits_{i\infty}^{x}g(y)dy.
\end{equation*}
Let us now consider the more general case

\begin{equation*}
      \mathcal{F}^{-1}\left(\dfrac{1}{ip-2 k \pi i}\widehat{g}(p)\right) = \dfrac{1}{2\pi} \int\limits_{-\infty}^{\infty}\dfrac{\widehat{g}(p)}{i(p-2k \pi)}e^{ixp}dp.
\end{equation*}
Substituting $p-2k \pi =q$, this becomes

\begin{equation*}
   = \dfrac{1}{2\pi}\int\limits_{-\infty}^{\infty}\dfrac{\widehat{g}(q+2k \pi)}{iq}e^{ix(q+2k \pi)}dq.
\end{equation*}
Rewriting $\frac{e^{ixq}}{iq}$ as above and applying Fubini's theorem, we find 

\begin{equation*}
      \mathcal{F}^{-1}\left(\dfrac{1}{ip}\widehat{g}(p)\right)=\dfrac{1}{2\pi}e^{2\pi ix} \int\limits_{i \infty}^{x}\left(\int\limits_{-\infty}^{\infty}e^{iyq}\widehat{g}(q+2k\pi) dq\right)dy.
\end{equation*}
The inner integral times $\frac{1}{2\pi}$ is just the inverse Fourier transform of $\widehat{g}(q+2k\pi)$. To find this integral,  one just has to set $p=q+2k \pi $ again. More precisely,

\begin{equation*}
    \setlength{\arraycolsep}{0mm}
\renewcommand{\arraystretch}{2,5}
\begin{array}{llll}
    \dfrac{1}{2\pi}\int\limits_{-\infty}^{\infty} e^{iyq}\widehat{g}(q+2 k \pi) &~=~& \dfrac{1}{2\pi} \int\limits_{-\infty}^{\infty}e^{iy(p-2k \pi)}\widehat{g}(p)dp \\  
     &~=~& =\dfrac{1}{2\pi}e^{-2k \pi i y} \int\limits_{-\infty}^{\infty}e^{iyp}\widehat{g}(p)dp \\
     &~=~& \dfrac{1}{2\pi}e^{-2k \pi i y}g(y),
\end{array}
\end{equation*}
since the last integral times $\frac{1}{2\pi}$ is the inverse Fourier transform of $\widehat{g}(p)$.  Therefore, in total, we found

\begin{equation*}
    \mathcal{F}^{-1}\left(\dfrac{\widehat{g}(p)}{ip -2 k \pi i}\right)=e^{2k \pi ix} \int\limits_{i\infty}^{x} e^{-2k \pi y}g(y).
\end{equation*}
Finally, inserting everything in the complete formula

\begin{equation*}
    f(x) = \int\limits_{i \infty}^{x} g(y)dy -\dfrac{1}{2}g(x)+\sum_{k\in \mathbb{Z}\in \lbrace 0 \rbrace} e^{+2 k \pi ix}\int\limits_{i\infty}^{x}e^{-2 k \pi iy}g(y)dy.
\end{equation*}
So, there are no integration constants in this formula and the solution we found is a particular solution, as we claimed. But since the full solution is obtained by adding an arbitrary function of period $1$, we hence derive the general solution

\begin{equation*}
    f(x) =\int\limits^{x}g(y)dy -\dfrac{1}{2}g(x)+\sum_{k\in \mathbb{Z}\setminus \lbrace 0\rbrace}e^{2 k \pi i x}\int\limits^{x}e^{-2 k \pi i y}g(y)dy,
\end{equation*}
as we stated it above in \ref{subsubsec: Mistake in Euler's Approach}.

\subsection{A Solution Euler could have given}
\label{subsec: A Solution Euler could have given}

Fourier analysis came after Euler\footnote{Fourier published his book, in which he introduced Fourier series, in 1822 \cite{Fo22}.}, i.e. he had no access to it. But here we argue that he could have derived the correct solution from his results,  if only we overlook some details of mathematical rigor.\\[2mm]
The first idea is to consider $\frac{d}{dx}$ and the higher derivatives as  operators acting on the function $f(x)$. Then, solving a differential equation is equivalent to finding the inverse operator to the operator acting on $f$ and applying it to both sides of the equation. Let us again write $\frac{d}{dx}=z$. The fundamental theorem of calculus implies

\begin{equation*}
\dfrac{1}{z}f = \int f dx.
\end{equation*}
And the difference equation we want to solve then becomes

\begin{equation*}
f(x)=\dfrac{1}{e^z-1}X.
\end{equation*}
We can now simplify this equation inserting the partial fraction decomposition of $(e^z-1)^{-1}$. We obtain

\begin{equation*}
f(x)= \left(\dfrac{1}{z}-\dfrac{1}{2}+\sum_{k \in \mathbb{Z}\setminus \lbrace 0\rbrace}\dfrac{1}{z-2k \pi i}\right)X.
\end{equation*}
We only need to find out what $\frac{1}{z-\alpha}X$ is. For this, let us write

\begin{equation*}
\renewcommand{\arraystretch}{2,5}
\begin{array}{ll}
\dfrac{1}{z-\alpha}X & =\dfrac{1}{z \left(1-\frac{\alpha}{z}\right)}X \\
 &= \dfrac{1}{z}\sum_{n=0}^{\infty} \dfrac{\alpha^n}{z^n}X,
\end{array}
\end{equation*}
where we used the geometric series in the second step. $\frac{1}{z^{n+1}}X$ is the $n+1$ times iterated integral of $X$. Euler considered  integrals of this kind in \cite{E679}, a paper written in 1778. His formulas yield 

\begin{equation*}
\int\limits_{}^{n}Xdx = \int\limits_{}^{x}\dfrac{(x-t)^{n-1}}{(n-1)!}X(t)dt.
\end{equation*}
Here, $\int\limits_{}^{n}$ denotes the $n$ times iterated integral, $\int\limits_{}^{x}$ indicates that one has to put $x=t$ after the integration. \\
Substituting this formula and using the Taylor series for $e^x$, we find

\begin{equation*}
\dfrac{1}{z-\alpha}X= \sum_{n=0}^{\infty} \dfrac{\alpha^n}{n!}\int\limits_{}^{x}(x-t)^n X(t)dt = \int\limits_{}^{x} e^{\alpha (x-t)}X(t)dt.
\end{equation*}
Therefore, we can write our solution as

\begin{equation*}
\setlength{\arraycolsep}{0mm}
\renewcommand{\arraystretch}{2,5}
\begin{array}{llllllllllll}
f(x) & ~=~ &\int\limits_{}^{x}Xdx &~-~&\dfrac{1}{2}X(x)&~+~& \sum_{k \in \mathbb{Z}\setminus \lbrace 0\rbrace} \int\limits_{}^{x}e^{2 k \pi i (x-t)}X(t)dt \\
 &~=~& \int\limits_{}^{x}Xdx &~-~&\dfrac{1}{2}X(x)&~+~& 2\sum_{k=1}^{\infty} \int\limits_{}^{x}\cos\left(2 k \pi  (x-t)\right)X(t)dt,
\end{array}
\end{equation*}
where we used the formula $\cos x = \frac{e^{ix}+e^{-ix}}{2}$ in the second line.\\
Note that this is almost the formula Euler gave, in Euler's formula only the term $-\frac{1}{2}X$ is missing, making his result incorrect.\\[2mm]
Although we have operated completely non-rigorously, the formula, as we found it, is correct, see \cite{We14}.\\
This formal calculus is a  beautiful example of the "Ars inveniendi" and is validated by the techniques from Fourier analysis.\\[2mm]

\subsubsection{An Application - Derivation of the Stirling Formula for the Factorial}
\label{subsubsec: An Application - Derivation of the Stirling Formula for the Factorial}

In the last paragraphs of \cite{E189}, Euler tried to derive the Stirling formula for $\Gamma(x)$ from his general solution of the general difference equation. But since he missed the term $-\frac{1}{2}g(x)$ in the solution of the difference equation, his formula is incorrect\footnote{He was even aware of this but tried to argue it away by another incorrect argument. More precisely, he arrived at the asymptotic formula $x!=x \log x-x +P$, i.e. the term $-\frac{1}{2}\log x$ is missing. But he argued that his formula is nevertheless correct, since  $P$ has to be a periodic function - which statement is incorrect - and he just found $P=\frac{1}{2}\log (2\pi)$ instead of the correct value $P=\frac{1}{2}\log (2\pi x)$.}. He gave a correct derivation of the formula in \cite{E212}. We will present his idea from \cite{E189} and derive the Stirling formula from the solution of the difference equation:

\begin{equation*}
    \log \Gamma (x+1)- \log \Gamma(x)= \log(x).
\end{equation*}
More precisely, we prove the formula

\begin{theorem}[Stirling Formula for $\Gamma(x+1)$]
\begin{equation*}
\Gamma (x+1) \sim \sqrt{2 \pi x}\dfrac{x^x}{e^x} \quad \text{for} \quad x \rightarrow \infty.
\end{equation*}
\end{theorem}
Before we start with the proof, we want to mention that we will, presenting Euler's argument, argue formally. For, we will write $=$, although we will encounter a series which   actually is an asymptotic series and thus $=$ is not the correct sign to use here. For definition of asymptotic series, the reader is, e.g., referred to \cite{Ha48}, \cite{Ba09} and the chapter on divergent series in \cite{Va06}. Having said this in advance, let us go over to Euler's proof.
\begin{proof}
His idea was to simplify the solution of the difference equation we derived; applied to $\log \Gamma(x)$ this solution reads

\begin{equation*}
    \log \Gamma(x) = -\dfrac{1}{2}\log x + \int\limits_{}^{x} \log (t) dt + \sum_{k \in \mathbb{Z}\setminus \lbrace 0\rbrace} e^{2 k \pi i x} \int\limits_{}^{x}e^{-2 \pi i t}\log (t)dt,
\end{equation*}
We simplify this general solution following Euler in \cite{E189}. First, we have

\begin{equation*}
\int\limits_{}^{x} \log t dt = x \log x - x +C,
\end{equation*}
$C$ being the constant of integration. Secondly, by iterated partial integration, we obtain

\begin{equation*}
\int\limits_{}^{x} \log t e^{-2 k \pi i t}dt =C_k -\dfrac{\log x e^{-2 k \pi i x}}{2 k \pi i}+ \sum_{n=1}^{\infty}\dfrac{(-1)^n}{(2k \pi i)^{n+1}}\dfrac{(n-1)!}{x^n}e^{-2 k  \pi i x},
\end{equation*}
$C_k$ being the constant of integration. Therefore,
\begin{equation*}
\log \Gamma(x) = x \log x - x + \Pi(x) - \dfrac{1}{2}\log x  +\sum_{k \in \mathbb{Z}\setminus \lbrace 0\rbrace} \left[-\dfrac{\log x }{2 k \pi i}+ \sum_{n=1}^{\infty}\dfrac{(-1)^n}{(2k \pi i)^{n+1}}\dfrac{(n-1)!}{x^n}\right],
\end{equation*}
where $\Pi(x)$ is an arbitrary periodic function with period $1$, i.e., a solution of the homogeneous equation $y(x+1)-y(x)=0$. Now note that for a natural number $m$

\begin{equation*}
\sum_{k \in \mathbb{Z}\setminus \lbrace 0\rbrace} \dfrac{1}{k^{2m-1}}=0 \quad \text{and} \quad\sum_{k \in \mathbb{Z}\setminus \lbrace 0\rbrace} \dfrac{1}{k^{2m}}= 2 \sum_{k=1}^{\infty} \dfrac{1}{k^{2m}}. 
\end{equation*}
Therefore,

\begin{equation*}
\log \Gamma(x)= x\log x-x+\Pi(x)-\frac{1}{2}\log x+ \sum_{k=1}^{\infty}\sum_{m=1}^{\infty} \dfrac{(-1)^{2m-1}}{x^{2m-1}}\dfrac{2(2m-2)!}{(2k \pi i)^{2m}}.
\end{equation*}
Now recall Euler's famous formula for the even $\zeta$-values\footnote{He never wrote it down like this explicitly but gave  a list of the first 13 values, e.g., in \cite{E130} and \cite{E212}. He certainly was aware of this general formula.}

\begin{equation*}
\sum_{k=1}^{\infty} \dfrac{1}{k^{2m}}= \dfrac{(-1)^{m-1}B_{2m}(2\pi)^{2m}}{2(2m)!},
\end{equation*}
where $B_n$ are the Bernoulli numbers. \\
Hence, having substituted those explicit values for the sums, we finally arrive at

\begin{equation*}
\log \Gamma(x)= x \log x-x- \frac{1}{2}\log x + \Pi(x)+ \sum_{k=1}^{\infty} \dfrac{B_{2m}}{2m(2m-1)x^{2m-1}}.
\end{equation*}
This is where Euler stopped.  But in his case, the term $-\frac{1}{2}\log x$ was missing, as we mentioned in section \ref{subsubsec: Mistake in Euler's Approach}.\\[2mm]
But having simplified the solution this far, let us now check the conditions of the Bohr-Mollerup theorem. Demanding them to be satisfied by the solution, we can prove that indeed $\Pi(x)=\frac{1}{2}\log (2\pi)$. We start with logarithmic convexity.\\[2mm]
One way of arguing that $\Pi(x)=A$, i.e., a constant, is as follows. Recall the following condition for convexity.

\begin{definition}[Convexity for a differentiable function]

A differentiable function $y(x)$ is convex on the interval $I$ if and only if

\begin{equation*}
\dfrac{y(a)-y(b)}{a-b} \geq y'(b) \quad \forall~ a,b \in I.
\end{equation*}
\end{definition}
Let us denote the convex part of the right-hand side in the final equation, i.e., the sum of every function except the periodic function $\Pi(x)$ since all of them are easily seen to be convex on the positive real axis, by $C(x)$. \\
Then, since $y(x)=\Pi(x)+C(x)$ and $y(x)$ is convex by assumption, we have

\begin{equation*}
\dfrac{\Pi(a)-\Pi(b)+C(a)-C(b)}{a-b}\geq \dfrac{\Pi(a)-\Pi(b)}{a-b} + C'(b)
\end{equation*}
and 

\begin{equation*}
\dfrac{\Pi(a)-\Pi(b)+C(a)-C(b)}{a-b} \geq C'(b)+\Pi'(b)
\end{equation*}
Therefore, we must have either

\begin{equation*}
\dfrac{\Pi(a)-\Pi(b)}{a-b} +C'(b) \geq C'(b) + \Pi'(b),
\end{equation*}
in which case $\Pi(x)$ would already be convex and everything works out nicely, since the only differentiable periodic function that is convex on the whole positive real axis is the constant function, or we have

\begin{equation*}
C'(b)+\Pi'(b) \geq C'(b)+ \dfrac{\Pi(a)-\Pi(b)}{a-b},
\end{equation*}
or equivalently

\begin{equation*}
\Pi'(b) \geq \dfrac{\Pi(a)-\Pi(b)}{a-b}.
\end{equation*}
Since this inequality must hold for all positive $a$ and $b$, let us put $b=a+1$. Then, since $\Pi(x)$ and hence also $\Pi'(x)$ is periodic,

\begin{equation*}
\Pi'(a) \geq  0.
\end{equation*}
This condition must hold for all positive $a$. Hence $\Pi(x)$ is a monotonically increasing function on the positive real axis. But the only differentiable  periodic function that is monotonically increasing on the whole positive real axis is, again, the constant function. This completes the proof that $\Pi(x)$ is indeed a constant. \\[2mm]
In the next step, we have to determine the constant explicitly. This is done using the condition $y(1)= 0$. Therefore,

\begin{equation*}
0= y(1)= -1 +A + \sum_{k=1}^{\infty}\dfrac{B_{2m}}{2m(2m-1)}.
\end{equation*}
This equation, at least in principle, allows  to find the constant $A$. We say "in principle", since the sum  diverges because of the rapid growth of the Bernoulli numbers. Nevertheless, applying techniques to sum divergent series, see, e.g., \cite{Ha48} written in 1948, the sum can be evaluated and we find $A=\log \sqrt{2\pi}$.\\
 The evaluation of this constant was Stirling's great contribution to the Stirling-formula - he gave a more general formula in 1730 in \cite{St30} and essentially used the Wallis product formula for $\pi$ to find the constant $A$. But the formula we want to prove and was named after him, was actually discovered by de Moivre in 1718 in \cite{dM18}, who could at that time only evaluate the constant numerically. See, e.g, \cite{Du91}, \cite{Le86}, \cite{Pe24} for a discussion on this matter.\\[2mm]
In order to avoid the use of divergent series here, let us present Euler's strategy to find $A$, taken from \cite{E212} $\S 158$ chapter 6 of the second part. Euler argues that, since $A$ is constant, we can determine it from any case. First, let us assume that $x$ is an integer and $x \gg 1$ so that the term $\sum_{k=1}^{\infty} \frac{B_{2m}}{2m(2m-1)x^{2m-1}}$ is negligible. Then, we have the summation

\begin{equation*}
\sum_{k=1}^{x} \log k = A + \left(x+\dfrac{1}{2}\right)\log x -x
\end{equation*}

\begin{center}
\begin{figure}
\centering
    \includegraphics[scale=1.0]{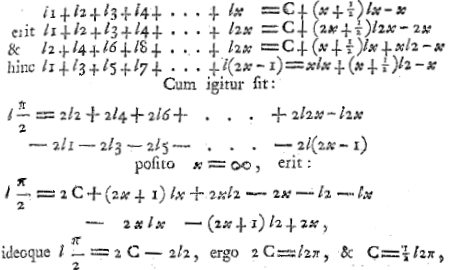}
    \caption{Euler and Stirling's Constant}
    Taken from \cite{E212}. Euler uses the Wallis product formula for $\pi$ to find the Stirling constant, $\log \sqrt{2 \pi}$, in the Stirling formula. Euler wrote $l$ for $\log$.
    \end{figure}
\end{center}
and hence

\begin{equation*}
\sum_{k=1}^{2x} \log k = A + \left(2x+\dfrac{1}{2}\right)\log (2x) -2x.
\end{equation*}
Moreover,

\begin{equation*}
\sum_{k=1}^{x} \log (2k) = \sum_{k=1}^{x} \log k +\sum_{k=1}^{x} \log 2 = x \log 2 + A + \left(x+\dfrac{1}{2}\right)\log x -x.
\end{equation*}
Therefore, finally

\begin{equation*}
\sum_{k=1}^{x} \log(2k-1) = \sum_{k=1}^{2x} \log k - \sum_{k=1}^{x} \log (2k) = x \log x + \left(x+\dfrac{1}{2}\right)\log 2 -x.
\end{equation*}
Now, Euler's idea also was to use the Wallis product formula for $\pi$, i.e.

\begin{equation*}
\dfrac{\pi}{2}= \dfrac{2}{1} \cdot \dfrac{2 \cdot 4}{3 \cdot 3} \cdot \dfrac{4 \cdot 6}{5 \cdot 5}\cdot \dfrac{6 \cdot 8}{7 \cdot 7} \cdot \dfrac{8 \cdot 10}{9 \cdot 9} \cdot \text{etc.}
\end{equation*}
Therefore,

\begin{equation*}
\log \dfrac{\pi}{2}= \lim_{x \rightarrow \infty} \left(2 \sum_{k=1}^{x}\log(2k) - \log(2x)-2\sum_{k=1}^{x}\log(2k-1)\right).
\end{equation*}
Thus, combining the corresponding equations and taking the limit, we find

\begin{equation*}
\log \dfrac{\pi}{2}= 2A - 2 \log 2
\end{equation*}
and hence

\begin{equation*}
A = \log \sqrt{2 \pi }.
\end{equation*}
Finally, we can write down our expression for $y(x)= \log \Gamma(x)$ in its final form:

\begin{equation*}
\log \Gamma (x) = x\log x -x-\dfrac{1}{2}\log x + \log \sqrt{2 \pi} + \sum_{k=1}^{\infty}\dfrac{B_{2m}}{2m(2m-1)x^{2m-1}}.
\end{equation*}
Finally, taking exponentials we arrive at the limit formula. Concerning this formula, we stress here again that this is an asymptotic expansion and thus one rather should write:

\begin{equation*}
\log \Gamma (x) \sim x\log x -x-\dfrac{1}{2}\log x + \log \sqrt{2 \pi} + \sum_{k=1}^{\infty}\dfrac{B_{2m}}{2m(2m-1)x^{2m-1}}.
\end{equation*}

\end{proof}

\subsection{Generalized Factorials}
\label{subsec: Generalized Factorials}

\subsubsection{Euler's Results on generalized factorials}
\label{subsubsec: Euler's results on generalized factorials}
In \cite{E661}, as the title \textit{Variae considerationes circa series hypergeometricas}\footnote{The term "hypergeometric series" that Euler used can be translated as "factorial series" in this context.} suggests, Euler studied more general factorials. The generalisation is that the difference equation $f(n+1)=nf(n)$ is now replaced by the slightly more general one $f(n+1)=(a+bn)f(n)$ with positive real numbers $a, b$. Having discussed the moment ansatz, we can easily solve such equations, but Euler had other ideas. He wanted to find Stirling-like formulas and used the Euler-Maclaurin summation formula\footnote{The Euler-Maclaurin summation formula will be discussed in the following section \ref{sec:  Solution of log Gamma(x+1)- log Gamma (x)= log x via the Euler-Maclaurin Formula}, but it will turn out that the summation formula is just a special case of the formula we found for the solution of the difference equation in this section.}. Otherwise, the ideas in this paper are not new. Regardless, we  want to state his results. He introduced the following functions:

\begin{equation*}
 \renewcommand{\arraystretch}{1,5}
\setlength{\arraycolsep}{0.0mm}
\begin{array}{lll}
    \Gamma(i) & ~=~ & a(a+b)(a+2b)(a+3b)\cdots (a+(i-1)b) \\ 
    \Delta(i) & ~=~ & a(a+2b)(a+4b) \cdots (a+(2i-2)b) \\
    \theta(i) & ~=~ & (a+b)(a+3b) \cdots (a+(2i-1)b).
\end{array}
\end{equation*}
He established several relations among them and found the Stirling-like formulas, i.e. asymptotic expansions,

\begin{equation*}
    \renewcommand{\arraystretch}{1,5}
\setlength{\arraycolsep}{0.0mm}
\begin{array}{lll}
\Gamma(i) &~\sim~& Ae^{-i}(a-b+bi)^{\frac{a}{b}+i-\frac{1}{2}} \\  
\Delta(i) &~\sim~& Be^{-i}(a-2b+2bi)^{\frac{a}{2b}+i-\frac{1}{2}} \\ 
\theta(i) &~\sim~& Ce^{-i}(a-b+2bi)^{\frac{a}{2b}+i}. 
\end{array}
\end{equation*}
All equations hold only for $i \rightarrow \infty$.
In these formulas, $\Gamma(i)$ is not to be confused with our $\Gamma$-function. We simply adopted Euler's notation from \cite{E661}.
Euler could not determine the constants $A$, $B$, $C$ as in the case of Stirling's formula, since he did not have access to a Wallis-like product and could not argue analogously as we showed in section \ref{subsubsec: An Application - Derivation of the Stirling Formula for the Factorial}. But we, in possession of all necessary tools, will evaluate these constants in the next sections. 
\begin{center}
    \begin{figure}
        \centering
        \includegraphics[scale=1.0]{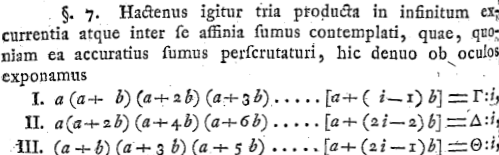}
        \caption{Euler's generalized Factorials} Euler's definition of the generalized factorials he considered in \cite{E661}. The scan is taken from the same paper.
    \end{figure}
\end{center}

\subsubsection{Evaluation of the Constant $A$}
\label{subsubsec: Evaluation of the constant A}

We will present the evaluation of the constant $A$ in the  first of the three asymptotic expansions. We claim that the following calculations could all have been done by Euler himself, since all steps involved concepts or techniques that he used elsewhere in a similar way. We have:

\begin{theorem}
Let $F$ be a meromorphic function that is logarithmically convex\footnote{Of course, Euler did not know the definition of a logarithmically convex function. But we state the result as general as possible.} on the positive real axis. Furthermore, let $F$ satisfy:

\begin{equation*}
    F(x+1)=(a+bx)F(x) \quad \forall ~ a,b,x>0 \quad \text{and} \quad F(1)=1.
\end{equation*}
Then $F$ has the following asymptotic expansion:

\begin{equation*}
    F(x) \sim A\left(a-b+bx\right)^{\frac{a}{b}-\frac{1}{2}+x}e^{-x} \quad \text{for} \quad x \rightarrow \infty
\end{equation*}
with

\begin{equation*}
    A= \dfrac{b^{-\frac{a}{b}-\frac{1}{2}}\sqrt{2\pi}e^{1-\frac{a}{b}}}{\Gamma\left(1+\frac{a}{b}\right)}.
\end{equation*}
\end{theorem}
\begin{proof}
That the general asymptotic formula is correct follows from an  application of the Euler-Maclaurin summation formula, which will be discussed in section \ref{sec:  Solution of log Gamma(x+1)- log Gamma (x)= log x via the Euler-Maclaurin Formula}, and was already shown by Euler in \cite{E661}, as we mentioned above. Thus, we will only determine the constant $A$.\\
To this end, recall the solution for the propounded difference equation we found in section \ref{subsubsec: 2. Example: Generalized Gamma function}, i.e.

\begin{equation*}
    F(x) = C \int\limits_{0}^{\infty}t^{x-1+\frac{a}{b}}e^{-\frac{t}{b}}dt \quad \text{for} \quad x-1+\frac{a}{b}>0.
\end{equation*}
$C$ is a constant that we did not need in the previous investigations.  But  we will determine it here. Before doing so, let us express $F$ in terms of the familiar $\Gamma$-function. Substituting $\frac{t}{b}=y$ in the above integral, we find:

\begin{equation*}
    F(x) =C \int\limits_{0}^{\infty} b^{x+\frac{a}{b}}y^{x-1+\frac{a}{b}}e^{-y}dy = C b^{x+\frac{a}{b}}\Gamma\left(x+\dfrac{a}{b}\right),
\end{equation*}
where we used the integral representation of the $\Gamma$-function in the last step. Therefore, from $F(1)=1$:

\begin{equation*}
    C= \dfrac{b^{-\frac{a}{b}-1}}{\Gamma\left(1+\frac{a}{b}\right)}.
\end{equation*}
Finally, we  have to use the Stirling-formula from section \ref{subsubsec: An Application - Derivation of the Stirling Formula for the Factorial} yielding the asymptotic expansion of $\Gamma(x+1)$ for large $x$; we just have to replace $x$ by $x+\frac{a}{b}$ in that formula and find the asymptotic expansion:

\begin{equation*}
    F(x+1) \sim \dfrac{b^{-\frac{a}{b}-1}}{\Gamma\left(1+\frac{a}{b}\right)}b^{x+1+\frac{a}{b}}\sqrt{2\pi}\left(x+\dfrac{a}{b}\right)^{x+\frac{a}{b}+\frac{1}{2}}e^{-x-\frac{a}{b}} \quad \text{for} \quad x \rightarrow \infty
\end{equation*}
Therefore, 

\begin{equation*}
    F(x)\sim \dfrac{b^{-\frac{a}{b}-1}}{\Gamma\left(1+\frac{a}{b}\right)}b^{x+\frac{a}{b}}\sqrt{2\pi}\left(x-1+\dfrac{a}{b}\right)^{x+\frac{a}{b}-\frac{1}{2}}e^{-x-\frac{a}{b}+1}  \quad x \rightarrow \infty.
\end{equation*}
Massaging this into a form resembling Euler's:

\begin{equation*}
    F(x) \sim \dfrac{b^{-\frac{a}{b}-\frac{1}{2}}}{\Gamma\left(1+\frac{a}{b}\right)}\sqrt{2\pi}\left(bx-b+a\right)^{x+\frac{a}{b}-\frac{1}{2}}e^{-x}\cdot e^{1-\frac{a}{b}}  \quad x \rightarrow \infty.
\end{equation*}
Therefore, comparing this to Euler's result we find:

\begin{equation*}
    A= \dfrac{\sqrt{2\pi}\cdot e^{1-\frac{a}{b}}}{\Gamma\left(1+\frac{a}{b}\right)}\cdot b^{-\frac{a}{b}-\frac{1}{2}}.
\end{equation*}
\end{proof}

\subsubsection{Finding the other Constants}
\label{subsubsec: Finding the other constants}

Having found $A$ in this way, one can proceed analogously to find the others. Concerning the function $\Delta(x)$, satisfying

\begin{equation*}
    \Delta(x+1)=(a+2bx)\Delta(x), \quad a,b>0 \quad \text{with} \quad \Delta(1) =1,
\end{equation*}
we can derive the asymptotic expansion just replacing $b$ by $2b$ in the expansion found for $F$. Hence we find:

\begin{equation*}
    \Delta(x) \sim \dfrac{(2b)^{-\frac{a}{2b}-\frac{1}{2}}}{\Gamma\left(1+\frac{a}{2b}\right)}\sqrt{2\pi}\left(2bx-2b+a\right)^{x+\frac{a}{2b}-\frac{1}{2}}e^{-x}\cdot e^{1-\frac{a}{2b}}  \quad x \rightarrow \infty.
\end{equation*}
Thus, the constant $B$ in Euler's formulas from section \ref{subsubsec: Euler's results on generalized factorials} is:

\begin{equation*}
    B= \dfrac{(2b)^{-\frac{a}{2b}-\frac{1}{2}}}{\Gamma \left(1+\frac{a}{2b}\right)}\sqrt{2\pi}e^{1-\frac{a}{2b}}.
\end{equation*}
Finally, Euler introduced the function $\theta(x)$ satisfying

\begin{equation*}
    \theta(x+1) =(a+(2b+1)x)\theta(x), \quad a,b >0 \quad \text{with} \quad \theta(0)=1.
\end{equation*}
Additionally, it is logarithmically convex, as it follows from Euler's definition given in \ref{subsubsec: Euler's results on generalized factorials}. Here, one is inclined to derive the asymptotic expansion from the one of $F$ again. But this does not lead to a form resembling Euler's. To get to such a form, note that

\begin{equation*}
    \Delta \left(x+\dfrac{1}{2}\right)=k \theta(x)
\end{equation*}
for a constant $k$, since both sides satisfy the same functional equation and are logarithmically convex on the positive real axis.  From the special case $x=0$ the constant $k$ is found to be $=\frac{1}{\Delta(\frac{1}{2})}$. Let us evaluate $k$. First, we have:

\begin{equation*}
    \Delta(x) = \dfrac{(2b)^{-\frac{a}{2b}-1}}{\Gamma\left(1+\frac{a}{2b}\right)} \cdot (2b)^{x+\frac{a}{2b}} \cdot \Gamma\left(x+\dfrac{a}{2b}\right).
\end{equation*}
Hence

\begin{equation*}
    \Delta \left(\dfrac{1}{2}\right)= (2b)^{-\frac{1}{2}}\cdot \dfrac{\Gamma\left(\frac{1}{2}+\frac{a}{2b}\right)}{\Gamma\left(1+\frac{a}{2b}\right)}.
\end{equation*}
Therefore, we can now use the asymptotic expansion of $\Delta$ to find the one for $\theta:$

\begin{equation*}
    \theta(x) \sim \dfrac{1}{\Delta \left(\frac{1}{2}\right)}\cdot B (2bx-b+a)^{x+\frac{a}{2b}}e^{-x-\frac{1}{2}}.
\end{equation*}
Thus, the constant $C$ in Euler's formulas from \ref{subsubsec: Euler's results on generalized factorials} can be expressed in terms of $k= \frac{1}{\Delta \left(\frac{1}{2}\right)}$ and $B$, which we already found:

\begin{equation*}
    C = k\cdot B e^{-\frac{1}{2}} 
\end{equation*}

\subsubsection{Euler's Relations among the Constants}
\label{subsubsec: Euler's relations among the constants}

As mentioned in section \ref{subsubsec: Euler's results on generalized factorials}, in $\S 17$ of \cite{E661}, Euler also proved some relations among the constants $A$, $B$ and $C$ in the asymptotic expansions. More precisely, he proved:

\begin{equation*}
    B = \Delta \left(\frac{1}{2}\right) e^{\frac{1}{2}}C \quad \text{and} \quad B = \sqrt{A\Delta \left(\frac{1}{2}\right)e}.
\end{equation*}
The relation among $B$ follows directly from our calculations so that we focus on the derivation of the other one, among $A$ and $B$. To prove it from our explicit formulas for $A$ and $B$, we need the following 
\begin{theorem}[Legendre's duplication formula for $\Gamma$]
The following equation holds for all $z \in \mathbb{C}$:

\begin{equation*}
    \Gamma(z)\Gamma \left(z+\dfrac{1}{2}\right)= 2^{1-2z}\sqrt{\pi}\Gamma(2z).
\end{equation*}
\end{theorem}
This formula is a special case of the Gau\ss{}ian multiplication formula for the $\Gamma$-function, which will be discussed in quite some detail below in section \ref{subsec: Multiplication Formula}.\\
But, for now, let us consider $A \cdot \Delta \left(\frac{1}{2}\right) \cdot e$; inserting the values for $A$ and $\Delta \left(\frac{1}{2}\right)$:

\begin{equation*}
    A \Delta \left(\dfrac{1}{2}\right)\cdot e = \dfrac{b^{-\frac{a}{b}-\frac{1}{2}}\sqrt{2\pi}e^{1-\frac{a}{b}}}{\Gamma \left(1+\frac{a}{b}\right)}\cdot (2b)^{-\frac{1}{2}}\dfrac{\Gamma\left(\frac{1}{2}+\frac{a}{2b}\right)}{\Gamma\left(1+\frac{a}{2b}\right)}\cdot e.
\end{equation*}
Applying Legendre's duplication formula for $z=\frac{1}{2}+\frac{a}{2b}$, this equation becomes

\begin{equation*}
    A \Delta \left(\dfrac{1}{2}\right)\cdot e = \dfrac{b^{-\frac{a}{b}-\frac{1}{2}}\sqrt{2\pi}\cdot e^{2-\frac{a}{b}}(2b)^{-\frac{1}{2}}\cdot \sqrt{\pi}\cdot 2^{-\frac{a}{b}}}{\Gamma\left(\frac{1}{2}+\frac{a}{2b}\right)\left(\Gamma\left(1+\frac{a}{2b}\right)\right)^2}\Gamma \left(\dfrac{1}{2}+\dfrac{a}{2b}\right).
\end{equation*}
This expression can be simplified to

\begin{equation*}
     A \Delta \left(\dfrac{1}{2}\right)\cdot e = \dfrac{(2b)^{-\frac{a}{b}-1}\cdot 2 \pi \cdot e^{2-\frac{a}{b}}}{\left(\Gamma\left(1+\frac{a}{2b}\right)\right)^2}
\end{equation*}
The right-hand side is just $B^2$; thus, the relation  is proven.\\[2mm]
Finally, we want to point out that Euler in \cite{E661} found the relations among $A$, $B$, $C$ from the relations among the functions $\Delta$, $F$ and $\theta$. Therefore, Legendre's duplication formula for the $\Gamma$-function follows from the relations among $A$, $B$ and $C$. Euler did not notice this, of course. But below in \ref{subsubsec: Euler's Version of the Multiplication Formula} we will show that later in his career Euler  found a formula equivalent to the multiplication formula for the $\Gamma$-function containing Legendre's formula as a special case.

\newpage

\section{Solution of $\log \Gamma(x+1)- \log \Gamma (x)= \log x$ via the Euler-Maclaurin Formula}
\label{sec: Solution of log Gamma(x+1)- log Gamma (x)= log x via the Euler-Maclaurin Formula}

\subsection{Overview on the Section}
\label{subsec: Overview on the Section}

We will present Euler's derivation in section \ref{subsec: Euler's Derivation of the Formula} and the modern derivation of the Euler-Maclaurin summation formula in section \ref{subsec: Modern Derivation}. Euler also saw the formula as the solution of the difference equation, i.e., as a tool to calculate finite sums.

\subsection{Euler's Derivation of the Formula}
\label{subsec: Euler's Derivation of the Formula}

Euler derived the Euler-Maclaurin summation formula on various occasions. The first occasion was in 1738 in \cite{E25}, but he also gave derivations in 1741 in \cite{E47} and in his book \cite{E212} published 1755. \\
He saw it as a tool to calculate finite sums and this helps to understand his derivation. For, we already mentioned that

\begin{equation*}
    f(n) := \sum_{k=1}^{n} g(k)
\end{equation*}
satisfies the functional equation

\begin{equation*}
    f(n)-f(n-1)=g(n) \quad \forall n \in \mathbb{N}.
\end{equation*}
And the idea to find the function $f$ is still the same as in section \ref{subsec: Eulers Idea} (or in \cite{E189}), i.e. to solve the above difference equation. Using Taylor's theorem, he wrote

\begin{equation*}
    f(n-1)= \sum_{k=0}^{\infty} \dfrac{(-1)^k f^{(k)}(n)}{k!}.
\end{equation*}
This, as above, led him to a differential equation of infinite order, i.e.

\begin{equation*}
    \sum_{k=1}^{\infty} \dfrac{(-1)^k f^{(k)}(n)}{k!} =g(n).
\end{equation*}
Now, it is important to note that at the time of his first proof in 1738 in \cite{E25} and later in 1741 in \cite{E47}, Euler had not developed the theory how to solve differential equations with constant coefficients yet. But Euler had another idea: He made an educated guess. More precisely, he assumed the solution to be of the following form:

\begin{equation*}
    f(n) = \alpha \int\limits_{}^{n}g(k)dk + \beta\dfrac{dg(n)}{dn}+ \gamma\dfrac{d^2g(n)}{dn^2}+ \delta\dfrac{d^3g(n)}{dn^3}+\varepsilon\dfrac{d^4g(n)}{dn^4}+\text{etc.}
\end{equation*}
Inserting this ansatz into the differential equation of infinite order gives recursive relations to define the coefficients $\alpha, \beta, \gamma$ etc. \\

\begin{center}
\begin{figure}
\centering
    \includegraphics[scale=1.2]{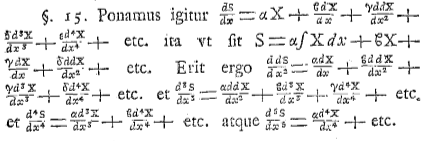}
    \caption{Finding the Euler-Maclaurin Summation Formula}
     Scan taken from \cite{E47}. Euler makes the ansatz with undetermined coefficients $\alpha$, $\beta$, $\gamma$ etc. to solve the differential equation of infinite order (derived from the difference equation $S(x+1)-S(x)=X$) for $S$. In the following paragraphs, he found the recursive relation for the coefficients. This led him to the Euler-Maclaurin summation formula.
    \end{figure}
\end{center}
In his 1755 book \cite{E212}, Euler then also proved that the coefficients are generated by:

\begin{equation*}
    \dfrac{z}{e^{-z}-1}
\end{equation*}
if it is expanded into a Taylor series around the origin\footnote{In his earlier papers on the summation formula, he did not realize this.}. This is almost the modern definition of the Bernoulli numbers\footnote{Indeed, Euler introduced that name for those numbers in 1755 in \cite{E212}. Furthermore, in \cite{E746}, a paper just published in 1815, he arrived at the generating function $\frac{z}{e^z-1}$, which is used nowadays to introduce the Bernoulli numbers.}.

\begin{definition}[Bernoulli Numbers]
The Bernoulli numbers are defined via a generating function. More precisely, we define the $n$-th Bernoulli number $B_n$ via

\begin{equation*}
    \dfrac{z}{e^z-1}= \sum_{n= 0}^{\infty} \dfrac{B_n}{n!}z^n.
\end{equation*}
\end{definition}
Euler calculated many of the Bernoulli numbers, which were defined by Jacob Bernoulli in the context of combinatorics in his 1713 book \cite{Be13}, and proved some elementary properties about them, e.g., that all odd Bernoulli numbers except $B_1= - \frac{1}{2}$ vanish. He was mainly interested in them, since they appear in his formula for $\zeta(2n)= \frac{1}{1^{2n}}+\frac{1}{2^{2n}}+\frac{1}{3^{2n}}+\cdots$, a connection he realized for the first time in 1750 in \cite{E130}. See also Sandifer's article in \cite{Br07} (pp. 279 - 303).\\
Using the Bernoulli numbers, one can write the Euler-Maclaurin formula as

\begin{equation*}
    f(n+1)=\sum_{k=0}^{n}g(k) = \int\limits_{}^{n}g(k)dk + \sum_{k=0}^{\infty} \dfrac{B_{k+1}}{(k+1)!}\dfrac{d^k g(n)}{dn^k}.
\end{equation*}
The constant of integration must be determined in such a way that the condition $f(1)=g(0)$ is satisfied. This is at least, how Euler used the formula. The modern formula avoids this problem, essentially by subtracting two sums from each other. Furthermore, we stress that the infinite sum does in general not converge for any value of $n$ and is to be understood as an asymptotic series. We addressed this issue already, when we discussed the Stirling formula \ref{subsubsec: An Application - Derivation of the Stirling Formula for the Factorial}.

\begin{center}
\begin{figure}
\centering
    \includegraphics[scale=1.1]{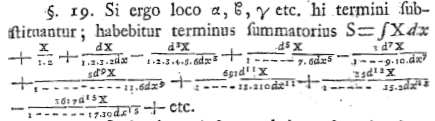}
    \caption{The Euler-Maclaurin Sum Formula}
    Taken from \cite{E47}. Euler states the Euler-Maclaurin formula explicitly for the first time. $\alpha$, $\beta$, $\gamma$ etc. are defined recursively in the same paper. They are the Bernoulli numbers, a fact Euler did not realize at that time.
    \end{figure}
\end{center}

\subsubsection{Several Remarks}
\label{subsubsec: Several Remarks}

The derivation of the above series is purely formal and the convergence of the sum is not guaranteed. Indeed, since the Bernoulli numbers increase rapidly (roughly as $(2n)!$, which follows from Euler's formula for $\zeta(2n)$.), the series actually converges very rarely. Euler was aware of this and only used it for numerical calculations truncating the sum after a certain number of terms. The Euler-Maclaurin summation formula leads to the notion of a semi-convergent series, a term coined by Gau\ss{} in 1812 in \cite{Ga28}. $B_n$ is small for the first few $n$, whence the series seems to converge taking only a few terms\footnote{Indeed, the value found that way is often very accurate.}, although the sum if continued to infinity must ultimately diverge, if the derivatives of $g$ do not vanish. An explicit formula for the remainder term, if the sum is truncated at some point, was given by Jacobi in 1834 in \cite{Ja34a} and Poisson.\\
The issue of semi-convergence troubled Euler and many others, including Gau\ss{} in \cite{Ga28}. Nowadays, the right-hand side of the above series is understood as an asymptotic expansion of the sum on the right. We mention \cite{Ha48}, \cite{Ba09} and \cite{Va06} again as references discussing the notion of an asymptotic expansion.\\[2mm]
Leaving the issues of convergence aside, let us discuss the nature of the formula. Recalling its origin,  the solution of a differential equation of infinite order, it has to be a particular solution of that equation. In other words, it has to be a special case of the solution given above in section \ref{sec: Solution of the Equation log Gamma(x+1)- log  Gamma (x)= log x by Conversion into a Differential Equation of infinite Order}. Staying purely formal, it is easier to understand this connection, what we will do in the following section.

\subsubsection{Purely formal Derivation}
\label{subsubsec: Purely formal Derivation}

We will use the idea that we can replace $\frac{d}{dx}$ by $z$ and an integral by $\frac{1}{z}$ and vice versa\footnote{Note that using the language of Fourier analysis this can be made completely rigorous, as we have seen above.}. Above we saw that a solution of

\begin{equation*}
    f(x+1)-f(x) =g(x)
\end{equation*}
is given by

\begin{equation*}
    f(x) = \dfrac{1}{e^z-1}g(x).
\end{equation*}
And next, we expanded $\frac{1}{e^z-1}$ into partial fractions and derived the complete solution of the simple difference equation. But we can also expand $\frac{1}{e^z-1}$ differently, i.e. into a Laurent series around $z=0$. Using the definition of the Bernoulli numbers above, we find

\begin{equation*}
    \dfrac{1}{e^z-1}= \sum_{n=0}^{\infty} \dfrac{B_n}{n!}z^{n-1} = \dfrac{B_0}{z} +  \sum_{n=0}^{\infty} \dfrac{B_{n+1}}{(n+1)!}z^{n}.
\end{equation*}
Therefore, replacing $\frac{1}{e^z-1}$ by the right-hand side of this equation and then replacing $z^n$ by $\frac{d^n}{dx^n}$ and $\frac{1}{z}$ by $\int\limits_{}^{x}$ in the above equation, we have

\begin{equation*}
    f(x) = B_0\int\limits_{}^{x} g(t)dt + \sum_{n=0}^{\infty} \dfrac{B_{n+1}}{(n+1)!}\dfrac{d^n g(x)}{dx^n}.
\end{equation*}
It is easy to see that this is the Euler-Maclaurin summation formula again. Therefore, the naive and formal derivation easily reproduces both, the general solution and the particular solution provided by the Euler-Maclaurin series, from the general difference equation.

\subsection{Historical Overview}
\label{subsec: Historical Overview}

For the sake of completeness, it will be convenient to give a least a short overview about the results in the theory of differential equations of infinite order. The overview can be subdivided in formal results and rigorously proven results. Naturally, the formal results extend much further, but will we see that the general formulas, derived in purely formal manner, do not hold in every case. And the study of the first rigorous results will reveal, why the the formal approaches are not correct in the most general case.

\subsubsection{Formal Approaches - From Euler to Bourlet}
\label{subsec: Formal Approaches - From Euler to Bourlet}

Although many people contributed to the formal theory of differential equations of infinite order, including Lagrange \cite{La72}, Laplace \cite{La20} and many others, the most general result was derived by Bourlet \cite{Bo97} and \cite{Bo99}, who basically treated the problem of solving a differential equation as a problem to find the left-inverse operator of the corresponding differential operator as we did in the case of the general difference equation. He wrote $z=\frac{d}{dx}$, and considered the equation

\[
F(x,z)f(x)=\sum_{n=0}^{\infty} a_n(x)\frac{d^n}{dx^n}f(x)=g(x),
\]
where Bourlet, like Euler, did not specify the functions $a_n(x),f(x),g(x)$. Now Bourlet's simple idea was that, since $F(x,z)$ is an operator, it has (in modern language), a left inverse $X(x,z)$. And treating $z$ as a variable quantity, he derived a partial differential equation that determines $X(x,z)$. It reads as follows:

\[
\sum_{n=0}^{\infty} \frac{1}{n!}\frac{\partial^n X}{\partial z^n}\cdot \frac{\partial^n F}{\partial x^n}=1.
\]
As appealing as this formula might look, it is not true in general. To see this, this consider the equation:

\[
\sum_{n=0}^{\infty} \frac{(-x)^n}{n!}\frac{d^n}{dx^n}f(x)= g(x).
\]
On the one hand we see, that the left hand side is by Taylor's theorem equal to:

\[
f(0).
\]
Thus, $g(x)$ cannot be chosen arbitrarily. On the other hand, Bourlet's formula would lead to an inverse function, because the corresponding differential equation can be solved. So it gives a solution to an ill-defined question.\newline

\subsubsection{Rigorous Results - From Carmichel to today}
\label{subsubsec: Rigorous Results - From Carmichel to today}

Therefore, formally the theory is established by Bourlet's formula, a beautiful account of this is given in \cite{Da36}, where the main focus is put on solving the equations by mostly formal means.  Because of the problems with the formal procedure,  mathematicians considered more special problems and considered more special classes of functions, for which rigorous results can be proven. One of the first overview papers on the rigorous results was \cite{Ca36}. Most of the results described in that paper only consider  the existence of a solution, but do not give explicit solution formulas. One exception is the theorem we want to quote here.\newline \\[2mm]

\begin{theorem}
 In the linear differential equation of infinite order

\[
a_0y+a_1y^{\prime}+\cdots =\phi(x)
\]
let the constants $a_{\nu}$ be such constants that the function

\[
F(z)=a_0+a_1z+a_2z^2+\cdots
\]
is analytic in the region $|z| \le q$, i.e. inside a disc\footnote{To be precise, we should rather write $|z|<q$, because for differentiation we need open domains, but Carmichel's paper states the theorem with $\le$, which we simply adopted here.}, where $q$ is a given positive constant or zero, and let $\phi(x)$ be a function of exponential type not exceeding $q$. If $F(z)$ vanishes at least once in the region $|z| \le q$, let $n$ be the number of its zeros in this region (each counted according to its multiplicity) and let $P(z)$ be the polynomial of degree $n$ with leading coefficient unity, that $\frac{F(z)}{P(z)}$ does not vanish in the region. If $F(z)$ does not vanish in the region, let $P(z)$ be identically equal to $1$. When $P(z) \equiv 1$, let $P_{n-1}(z)$ be identically equal to zero; otherwise, let it be an arbitrary polynomial of degree $n-1$ (including the case of an arbitrary constant when $n=1$). Then the general solution $y(x)$, subject to the condition that it shall be a function of exponential type not exceeding $q$, may be written in the form

\[
y(x)=\frac{1}{2 \pi i}\int_{C_{\rho}}\frac{\Psi(s)}{F(s)}ds+\frac{1}{2 \pi i}\int_{C_{\rho}}\frac{P_{n-1}(s)}{P_n(s)}ds
\]
where

\[
\Psi(s)=\sum_{\nu=0}^{\infty}\frac{\phi^{(\nu)}(0)}{s^{\nu+1}},
\]
and where $C_{\rho}$ is a circle of radius $\rho$ about $0$ as  center, $\rho$ being greater than $q$ and such that $F(z)$ is analytic in the region $q>|z|\le \rho$ and does not vanish there.\newline
If $\phi(x)$ is exactly of exponential type $q$, then the named solution $y(x)$ is also of exponential type $q$.
\end{theorem}
Having stated this theorem, we see, how carefully it is formulated and that the class of functions is quite restricted. But we have to keep in mind that in the time of this paper, i.e. 1936, the notion and theory of distributions did not exist, making the theorems quite difficult to state.  The more recent treatments, e.g., \cite{Du10} start with constructing the appropriate function spaces, before solving any differential equations.

\subsection{Modern Derivation}
\label{subsec: Modern Derivation}

It will be illustrative to compare a modern proof of the Euler-Maclaurin summation formula to Euler's idea. Most modern proofs in modern introductory textbooks are similar. We will present it as it is found in \cite{Koe00} (pp. 223-226). The proof in \cite{Va06} is the same.
The proof of the general formula proceeds in several steps, slowly ascending from special cases  to the general formula.

\subsubsection{Euler Maclaurin Formula for $\mathcal{C}^{1}-$functions}
\label{subsubsec: Euler Maclaurin Formula for C1-functions}

We first have to define an auxiliary function

\begin{definition}
We define a function $H:\mathbb{R}\rightarrow \mathbb{R}$ as follows

\begin{equation*}
    H(x) := \left\lbrace
    \begin{array}{ll}
      x -[x]-\frac{1}{2}   & \quad \text{for} \quad x \in \mathbb{R}\setminus \mathbb{Z}  \\
        0 & \quad \text{for} \quad x \in \mathbb{Z} 
    \end{array} \right\rbrace
\end{equation*}
$[x]$ is the Gau\ss{} bracket of $x$ and expresses the integer part of $x$.
\end{definition}
This function obviously has period $1$. Having introduced $H$ in this way, we can state the elementary form of the summation formula

\begin{theorem}[Euler summation formula (simple version)]
Let $f:[1,n]\rightarrow \mathbb{C}$, $n\in \mathbb{N}$ be a once continuously differentiable function, then

\begin{equation*}
    \sum_{k=1}^{n}f(k) = \int\limits_{1}^{n}f(x)dx+\dfrac{1}{2}(f(1)+f(n))+\int\limits_{1}^{n}H(x)f'(x)dx.
\end{equation*}
\end{theorem}
\begin{proof}
By integration by parts over the interval $[k, k+1]$, one has

\begin{equation*}
    \int\limits_{k}^{k+1}1 \cdot f(x)dx = \left[ \left(x-k-\dfrac{1}{2}\right)f(x)\right]_{k}^{k+1} -     \int\limits_{k}^{k+1} \left(x-k-\dfrac{1}{2}\right)f'(x)dx.
\end{equation*}
Since the function $(x-k-\frac{1}{2})f'$ is  identical to $Hf'$ in the interval $[k;k+1]$\footnote{Maybe not at the end points of the intervals, but this does not matter in the following.}, their integrals over the same intervals are identical, thus,

\begin{equation*}
       \int\limits_{k}^{k+1} 1 \cdot f(x)dx = \dfrac{1}{2}(f(k+1)-f(k))-   \int\limits_{k}^{k+1}H(x)f'(x)dx.
\end{equation*}
Summation over $k$ from $1$ to $n-1$ and addition of $\frac{1}{2}(f(1)+f(n))$ then gives the formula.
\end{proof}

\subsubsection{General Euler-Maclaurin Summation Formula}
\label{subsubsec: General Euler-Maclaurin Summation Formula}

As in the simple case, we need to introduce some auxiliary functions

\begin{definition}
We define functions $H_k:\mathbb{R}\rightarrow \mathbb{R}$ recursively as follows: \\
1) $H_k$ is a primitive of $H_{k-1}$, $k\geq 2 \in \mathbb{N}$ and $H_1:=H$\\
2) $\int\limits_{0}^{1} H_k(x)dx=0$.
\end{definition}
Now we can state the general Euler-Maclaurin summation formula

\begin{theorem}[Euler-Maclaurin Summation Formula]
Let $f:[1,n]\rightarrow \mathbb{C}$ be a $\mathcal{C}^{2k+1}$-function and $k\geq 1$. Then we have

\begin{equation*}
    \sum_{\nu=1}^{n} = \int\limits_{1}^{n}f(x) +\dfrac{1}{2}\left(f(1)+f(n)\right)+\left[\sum_{\kappa =1}^{k}H_{\kappa}(0)f^{(2\kappa -1)}\right]_{1}^{n}+R(f);
\end{equation*}
with

\begin{equation*}
    R(f)= \int_{1}^{n}H_{2k+1}f^{(2k+1)}dx.
\end{equation*}
\end{theorem}
\begin{proof}
One just has to note that $H_{k}$ is a periodic function for all $k$, which is seen as follows by induction. We know that $H_1=H$ is periodic. Thus, consider

\begin{equation*}
    H_{k+1}(x+1)-H_{k+1}(x) = \int\limits_{x}^{x+1}H_k(t)dt= \int\limits_{0}^{1}H_k(t)dt=0,
\end{equation*}
where we used the defining properties of $H_k$ and the induction assumption that $H_k$ is periodic. Having noticed this, one proceeds just as in the case of the simple Euler-Maclaurin formula, but integrates by parts $2k+1$-times to arrive at the formula.\\[2mm]
\end{proof}

\subsubsection{Comparison to Euler's Idea}
\label{subsubsec: Comparison to Euler's Idea 1}

First, we want to mention that this idea of the modern proof is basically Jacobi's, see \cite{Ja34a}. Furthermore, we want to point out that we have in general

\begin{equation*}
    H_{k}(0)=\dfrac{1}{k!}B_{k}.
\end{equation*}
Thus, Euler's result and the modern result agree. Anyhow, the modern proof does not obtain the formula from the solution of a difference equation but rather starts from the results and proves it to be correct. In Euler's case, the formula resulted as an answer to a more general question.\\
Finally, we used the procedure of iterated integration by parts above in our derivation of the Stirling formula  for $n!$ in section \ref{subsubsec: An Application - Derivation of the Stirling Formula for the Factorial} and simply ignored $R(f)$ in the calculation. As already mentioned in the same section, this is not justified. A rigorous reasoning requires the theory of asymptotic expansions which we did not want to discuss at that point and hence operated on a formal level. Nevertheless, if understood as an asymptotic expansion, dropping $R(f)$ in the Stirling-formula is justified.

\newpage

\section{Interpolation Theory and Difference Calculus}
\label{sec: Interpolation Theory and Difference Calculus}

This section mainly discusses chapter 16 and 17 of \cite{E212} and the paper \cite{E613}, which Euler stated to be an elaboration of the mentioned chapters.

\subsection{Overview}
\label{subsec: Overview and Content}

Euler again tried to solve the difference equation

\begin{equation*}
    f(x+1)-f(x) =g(x)
\end{equation*}
in order to find an explicit formula for

\begin{equation*}
    \sum_{k=1}^{x} g(k-1).
\end{equation*}
The sum, only defined for integer $x$, is interpolated by the solution of the difference equation. But this time, Euler used the rules of difference calculus which  he had developed in \cite{E212} to solve the difference equation. Interestingly, addressing issues of convergence, this led him to the concept of Weierstra\ss{} products in section \ref{subsec: Modern Idea - Weierstrass product}.

\subsection{Euler's Idea}
\label{subsubsec: Euler's Idea 3}

Euler's idea, outlined in \cite{E613}, is best explained by an example. Euler tried to find the sum

\begin{equation*}
    \sum_{k=1}^{x}g(k).
\end{equation*}
This time, he simply added $0$ in a clever way. For, we formally have

\begin{equation*}
    \renewcommand{\arraystretch}{1,5}
\setlength{\arraycolsep}{0.0mm}
\begin{array}{cccccccccccccccccccccccccc}
     \sum_{k=1}^{x}g(k) &~=~& g(1) & ~+~& g(2) &~+~ & g(3) & ~+~& g(4) & ~+~& \text{etc.}   \\
       &~-~& g(x+1) & ~-~& g(x+2) &~-~ & g(x+3) & ~-~& g(x+4) & ~-~& \text{etc.}   \\ 
\end{array}
\end{equation*}
Or, in short notation

\begin{equation*}
     \sum_{k=1}^{x}g(k)=\sum_{k=1}^{\infty}\left(g(k)-g(x+k)\right).
\end{equation*}
As Euler observed, this can only work, if $g(k)-g(x+k)$ converges to zero for $k \rightarrow \infty$.\\
If the series does not converge, one can again add  $0$ in a clever way. If the resulting series does still not converge, one can do it again. Indeed, one can repeat the process arbitrarily often to increase the convergence as much as one wants.\\[2mm]
In \cite{E613} Euler divided sums into classes according to the behaviour of their infinitesimal terms. More precisely, the first class contains those series, in which we have $\lim_{k \rightarrow \infty} g(k)=0$. The second class contains those series whose differences of infinitesimal terms vanish, i.e. $\lim_{k\rightarrow \infty}g(x+k+1)-g(x+k)= 0$. The third class contains the series, whose second differences vanish etc. From this definition it immediately follows that the series of the $i+1$-th class are a subset of the series of the $i$-th class.\\[2mm]
Euler then gave formulas for the first, second and third class and explains how the general formula for the $i-$th class can be constructed. We will only need the first and second class for our discussion. The formula for the first class was stated above. Therefore, we will give the formula for the second class. It reads

\begin{equation*}
      \renewcommand{\arraystretch}{1,5}
\setlength{\arraycolsep}{0.0mm}
\begin{array}{cccccccccccccccccccccccc}
    \sum_{k=1}^x g(k) &~=~ & (1-x)g(1) &~+~&  (1-x)g(2) &~+~& (1-x)g(3) &~+~& \text{etc.}     \\
 +   xg(1)   & ~+~ & xg(2)  & ~+~ & xg(3)  & ~+~ & xg(4)  & ~+~ & \text{etc.} \\
           & ~-~ & g(x+1)  & ~-~ & g(x+2)  & ~-~ & g(x+3)  & ~-~ & \text{etc.}   
\end{array}
\end{equation*}
Or in short notation:

\begin{equation*}
     \sum_{k=1}^x g(k) = xg(1) + \sum_{k=1}^{\infty} \left((1-x)g(k)+xg(k+2)-g(x+k)\right).
\end{equation*}
Whereas the left-hand side only makes sense for integer $x$, this restriction is not necessary on the right-hand side. Therefore, we can interpolate the sum in this way. But let us consider some examples first.

\subsection{Examples}
\label{subsec: Examples}

\subsubsection{Harmonic Series}
\label{subsubsec: Harmonic Series}

Euler chose the harmonic series as his first example, i.e.

\begin{equation*}
    f(x) = \sum_{k=1}^{x} \dfrac{1}{k}= 1 +\dfrac{1}{2}+\dfrac{1}{3}+\dfrac{1}{4}+\cdots +\dfrac{1}{x}
\end{equation*}
This series belongs to the first class and hence we can write

\begin{equation*}
       \renewcommand{\arraystretch}{2,5}
\setlength{\arraycolsep}{0.0mm}
\begin{array}{cccccccccccccc}
    f(x) &  ~=~ & \dfrac{1}{1} &~+~ & \dfrac{1}{2} &~+~ & \dfrac{1}{3}  &~+~ & \dfrac{1}{4} & ~+~ & \text{etc.} \\
      &  ~-~ & \dfrac{1}{x+1} &~-~ & \dfrac{1}{x+2} &~-~ & \dfrac{1}{x+3}  &~-~ & \dfrac{1}{x+4} & ~-~ & \text{etc.} \\
\end{array}
\end{equation*}
And adding each two terms written above each other, we find:

\begin{equation*}
    f(x) = \sum_{k=1}^{\infty} \dfrac{x}{k(x+k)} = \dfrac{x}{1(x+1)}+ \dfrac{x}{2(x+2)}+ \dfrac{x}{3(x+3)} + \dfrac{x}{4(x+4)}+ \text{etc.},
\end{equation*}
Trying to find the derivative of the function, in his book \cite{E212}, Euler arrived at the Taylor series expansion of this series. For this, he just expanded each term $\frac{x}{k(x+k)}$ into a power series via the geometric series and summed the resulting series columnwise. 

\subsubsection{$\Gamma(x)$-function}
\label{subsubsec: Gammafunction}

But we are mainly interested in the expression for the $\Gamma$-function arising from this idea. Since we have $\Gamma(x+1)= x\Gamma(x)$, we also have $\log \Gamma(x+1)-\log \Gamma(x)= \log x$. Therefore, we can, as we did above, consider $\log \Gamma(x)$ as the solution of the simple difference equation, and for integer $x$ we have

\begin{equation*}
    \log \Gamma(x+1) = \sum_{k=1}^{x} \log (k).
\end{equation*}
Since $\log(k+1)-\log(k)= \log \left(1+\frac{1}{k}\right)$ tends to zero for infinite $k$, the series belongs to the second class. Applying the corresponding formula, we obtain

\begin{equation*}
      \renewcommand{\arraystretch}{1,5}
\setlength{\arraycolsep}{0.0mm}
\begin{array}{cccccccccccccccccccccccc}
    \sum_{k=1}^x \log (k) &~=~ & (1-x)\log(1) &~+~&  (1-x)\log (2) &~+~& (1-x)\log(3) &~+~& \text{etc.}     \\
 +   x\log(1)   & ~+~ & x\log(2)  & ~+~ & x\log(3)  & ~+~ & x\log(4)  & ~+~ & \text{etc.} \\
           & ~-~ & \log(x+1)  & ~-~ & \log(x+2)  & ~-~ & \log(x+3)  & ~-~ & \text{etc.}   
\end{array}
\end{equation*}
Collecting the columns, we arrive at:

\begin{equation*}
    \log (\Gamma(x+1))= \log (1^{1-x}) + \log \left(\dfrac{1^{1-x}2^x}{x+1}\right) + \log \left(\dfrac{2^{1-x}3^x}{x+2}\right) + \log \left(\dfrac{3^{1-x}4^x}{x+3}\right) + \text{etc.}
\end{equation*}
Taking exponentials, we arrive at Euler's first formula for $x!$ in \cite{E19}\footnote{There, he did not prove the formula, but just gives an heuristic argument why it is true.}.

\begin{equation*}
    x! = \Gamma(x+1) = 1^{1-x} \cdot \dfrac{1^{1-x}2^x}{x+1} \cdot \dfrac{2^{1-x}3^x}{x+2} \cdot \dfrac{3^{1-x}4^x}{x+3} \cdot \text{etc.} = \prod_{k=1}^{\infty} \dfrac{k^{1-x}(k+1)^x}{x+k}
\end{equation*}

\subsection{Difference Calculus according to Euler}
\label{subsec: Difference Calculus according to Euler}

Here, we will explain in a bit more detail how Euler arrived at his formulas. To this end, we need to explain his results on difference calculus. He outlined his ideas in his book \cite{E212} and also in \cite{E613}. We want to state and prove his main formula from \cite{E613} for the finite sum of $x$ terms.

\begin{center}
\begin{figure}
\centering
    \includegraphics[scale=0.9]{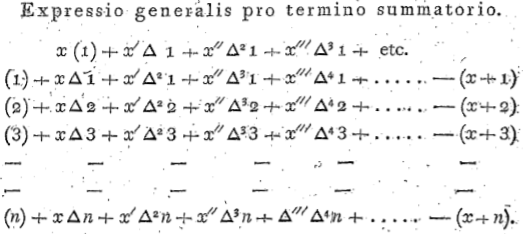} 
    \caption{Euler's Key Formula in the Calculus of Differences}
  Taken from \cite{E613}. Euler's main formula from the paper. He used the word "summatory" to describe the sum of a series up to $n$ terms and writes $(x)$ to denote the sum from $1$ to $x$. $\Delta^k$ denotes the $k$-th difference. Furthermore, Euler did not use a separate letter to denote a function. This is suppressed in his notation, e.g., $(x+3)$ is to be understood a $(f(x+3))$ for a general function $f$.
    \end{figure}
\end{center}
For this, we need to introduce some notation.

\begin{definition}[$n-$th Difference]
We denote the $n$-th difference by $\Delta^n g(k)$ and define it recursively by

\begin{equation*}
    \Delta^n g(k) : = \Delta^{n-1} g(k+1) - \Delta^{n-1}g(k)
\end{equation*}
$k$ being an arbitrary number\footnote{In most cases $k$ will be a natural number.} and 

\begin{equation*}
    \Delta^{0}g(k) : = g(k).
\end{equation*}
Additionally, we sometimes write simply $\Delta$ for $\Delta^1$.
\end{definition}
It is easily seen that we can express each term $g(k)$ using only the differences of $g(1)$. More precisely, we have:

\begin{theorem}
We have for a natural number $k$:

\begin{equation*}
    g(k) = g(1)+ \dfrac{k-1}{1}\Delta g(1) + \dfrac{k-1}{1} \cdot \dfrac{k-2}{2}\Delta ^2 g(1) + \dfrac{k-1}{1} \cdot \dfrac{k-2}{2} \cdot \dfrac{k-3}{3}\Delta ^3 g(1) + \text{etc.}
    \end{equation*}
    or in compact notation

    \begin{equation*}
        g(k) = \sum_{n=0}^{\infty} \binom{k-1}{n}\Delta^{n}g(1).
    \end{equation*}
    where $\binom{n}{k}$ is the binomial coefficient.
\end{theorem}

\begin{proof}
Since the binomial coefficients vanish if $n > k-1$, the sum is finite and hence converges. The proof  by induction. The formula is obviously true for $k=1$. So let us assume that it is true for a fixed $k \in \mathbb{N}$. Consider 

\begin{equation*}
    g(k+1) = g(k) + \Delta g(k) = g(k) + \sum_{n=0}^{\infty} \binom{k-1}{n}\Delta^{n}g(k),
\end{equation*}
where we used the definition of $\Delta$ and the induction assumption. Using the formula \footnote{The formula follows by taking $\Delta$ on both sides of the equation for $g(k)$.}

\begin{equation*}
    \Delta g(k) = \sum_{n=1}^{\infty} \binom{k-1}{n-1}\Delta^n g(1),
\end{equation*}
we arrive at

\begin{equation*}
    g(k+1) = \sum_{n=0}^{\infty} \left(\binom{k-1}{n}\Delta^{n}g(1)+ \sum_{n=1}^{\infty}\binom{k-1}{n-1}\Delta^n g(1)\right),
\end{equation*}
or

\begin{equation*}
    g(k+1) = \sum_{n=0}^{k-1} \binom{k-1}{n}\Delta^{n}g(1)+ \sum_{n=1}^{k}\binom{k-1}{n-1}\Delta^n g(1).
\end{equation*}
Let us  extract the first term of the first sum and the last term of the second sum such that

\begin{equation*}
    g(k+1) = g(1) + \sum_{n=1}^{k-1} \binom{k-1}{n}\Delta^{n}g(1)+ \sum_{n=1}^{k-1}\binom{k-1}{n-1}\Delta^n g(1) + \Delta^{k}g(k).
\end{equation*}
Contracting the sums and using the well-known identity $\binom{k-1}{n}+\binom{k-1}{n-1}=\binom{k}{n}$, we have

\begin{equation*}
    g(k+1) = g(1)+ \sum_{n=1}^{k-1} \binom{k}{n}\Delta^{n}g(1) + \Delta^k g(k).
\end{equation*}
Finally, absorbing the two isolated terms into the sum and using the vanishing of the binomial coefficients for $n >k$ we arrive at

\begin{equation*}
    g(k+1) = \sum_{n=0}^{\infty} \binom{k}{n}\Delta^{n}g(1).
\end{equation*}
\end{proof}
Next, we want to determine the sum $\sum_{k=1}^{x}g(k)$. This is done in the following theorem.

\begin{theorem}
We have

\begin{equation*}
    \sum_{k=1}^{x}g(k) = \sum_{k=1}^{\infty} \binom{x}{k}\Delta^kg(k).
\end{equation*}
\end{theorem}
The proof is by induction again along the same lines as the last. Therefore, we omit it here.  we want to make some remarks instead.\\[2mm]
Although the proof explicitly assumes $x$ to be a natural number, the right-hand side of the above equation does not require $x$ to be an integer\footnote{The binomial coefficients can be defined for non-integer numbers replacing the factorials by $\Gamma$-functions. Indeed, this is a consequence of Newton's generalized binomial theorem proved by Euler in \cite{E465}.} and hence interpolates the sum on the left-hand side. In other words, the right-hand side solves the functional equation satified by the finite sum, i.e. $\sum_{k=1}^{x} g(k) - \sum_{k=1}^{x-1}g(k)=g(x)$ with the initial condition $\sum_{k=1}^{1}g(k)=g(1)$. \\
And Euler precisely did this replacement, even addressing the issue of convergence of the then infinite series. \\[2mm]
But let us now go over to Euler's idea of adding zero in a clever way to enforce convergence. First note that from the above theorem we also have:

\begin{equation*}
    g(x+1)= \sum_{k=0}^{\infty} \binom{x}{k}\Delta^k g(1). 
\end{equation*}
But we can find similar expressions for $g(x+2)$ using $g(2)$. We have:

\begin{equation*}
    g(x+2)= \sum_{k=0}^{\infty} \binom{x}{k}\Delta^k g(2). 
\end{equation*}
And in general for $n$:

\begin{equation*}
    g(x+n)= \sum_{k=0}^{\infty} \binom{x}{k}\Delta^k g(n). 
\end{equation*}
Therefore, using the sum representation we found above, we arrive at the following equation:

\begin{theorem}[Equation for sum in terms of differences]
\begin{equation*}
\begin{array}{c}
    \sum_{k=1}^{x}g(k) = \\
     \renewcommand{\arraystretch}{2,5}
\setlength{\arraycolsep}{0.0mm}
\begin{array}{rrrrrrrrrrrrrrrrrrrrrrrrrrrrr}
  &  &\binom{x}{1} g(1) &~+~& \binom{x}{2}\Delta^1 g(1)  &~+~& \binom{x}{3}\Delta^2 g(1)    &~+~& \cdots \\ 
    +g(1)  &~+~& \binom{x}{1}\Delta^1 g(1)  &~+~& \binom{x}{2}\Delta^2 g(1) &~+~& \binom{x}{3}\Delta^3 g(1)   &~+~& \cdots &~-~& g(x+1) \\ 
     +g(2)  &~+~& \binom{x}{1}\Delta^1 g(2)  &~+~& \binom{x}{2}\Delta^2 g(2) &~+~& \binom{x}{3}\Delta^3 g(2)   &~+~& \cdots &~-~& g(x+2) \\
      +g(3)  &~+~& \binom{x}{1}\Delta^1 g(3)  &~+~& \binom{x}{2}\Delta^2 g(3) &~+~& \binom{x}{3}\Delta^3 g(3)   &~+~& \cdots &~-~& g(x+3) \\
      & \vdots & & \vdots & & \vdots & & \vdots & & \vdots & &  \\
        +g(n)  &~+~& \binom{x}{1}\Delta^1 g(n)  &~+~& \binom{x}{2}\Delta^2 g(n) &~+~& \binom{x}{3}\Delta^3 g(n)   &~+~& \cdots&~-~& g(x+n)  
\end{array}
\end{array}
\end{equation*}
\end{theorem}
\begin{proof}
The proof is immediate from the preceding. The first row is just the alternate representation of the sum. Each following row is  simply $=0$ by the results we stated. 
\end{proof}
And this is Euler's fundamental formula. For, he proceeded to sum the series column by column from $n=1$ to $n=\infty$. And hence it is easily seen that the definition of the classes we mentioned above makes sense. For,  if any iterated difference vanishes, one only has a finite number of columns to sum and the series converges\footnote{Under some mild additional assumptions.}. It is easily seen, how the examples we gave (from the first and second class) result from this formula. In the next section, we will explain how this idea actually anticipates the idea of Weierstra\ss{} factors.

\subsection{Modern Idea - Weierstra\ss{} Product}
\label{subsec: Modern Idea - Weierstrass product}

We introduce, but not prove, Weierstra\ss's product theorem and show the connection to Euler's ideas outlined in the last sections. The exposition of Weierstra\ss's theory follows \cite{Fr06} (pp. 213 -217).

\subsubsection{Introduction to the Problem}
\label{subsubsec: Introduction to the Problem}

Weierstra\ss{} considered the following problem in his 1878 paper \cite{We78}: Given a domain $D \subset \mathbb{C}$ and a discrete subset $S$ in $D$, can one construct an analytic function $f : D \rightarrow \mathbb{C}$ with zeros of a given order $m_s$ precisely in $S$?. The answer to this question is yes and the task can be solved by Weierstra\ss{} products. Let us see how we arrive at the concept. \\[2mm]
For the sake of simplicity, let us take $D= \mathbb{C}$. First, we note that closed disks are compact sets and hence there are only finitely many $s \in S$ with $|s| \leq N \in \mathbb{N}$. Thus, $S$ is a countable set and  the elements can be ordered with respect to their magnitude

\begin{equation*}
    S = \lbrace s_1, s_2, \cdots\rbrace \quad |s_1| \leq |s_2| \leq |s_3| \leq \cdots.
\end{equation*}
If $S$ is a finite set, we know how to solve the problem, the solution is given by the polynomial

\begin{equation*}
    \prod_{s \in S}(z-s)^{m_s}.
\end{equation*}
For infinite sets on the other hand, the product obtained in this way cannot converge in general. But we can assume $S$ to not contain zero, since we can multiply by $z^{m_0}$ at the end. We want to do this, since we can focus on products of the form

\begin{equation*}
    \prod_{n=1}^{\infty} \left(1-\dfrac{z}{s_n}\right)^{m_n}, \quad m_n := m_{s_n}.
\end{equation*}
Indeed, Euler already had this idea and it led him to the discovery of the sine product in \cite{E41} and its proof in \cite{E61}\footnote{The discussion of his proof can be found in section \ref{subsubsec: Euler's Proof of the Sine Product Formula}}.\\
This product still does not always converge (it converges, e.g., for $s_n =n^2$, $m_n=1$ but diverges for $s_n$, $m_n=1$).\\
Weierstra\ss{} had the idea to multiply it by factors not changing the zeros but forcing the product to converge. He made the ansatz:

\begin{equation*}
    f(z):=  \prod_{n=1}^{\infty} \left(1-\dfrac{z}{s_n}\right)^{m_n} \cdot e^{P_n(z)}.
\end{equation*}
$P_n(z)$ is a polynomial still to be determined. We have to ensure that

\begin{equation*}
    \lim_{n \rightarrow \infty} \left(1-\dfrac{z}{s_n}\right)^{m_n}e^{P_n(z)}=1 \quad \forall z \in \mathbb{C}.
\end{equation*}
This is possible, since it is easily seen that we can find an analytic function $A_n(z)$

\begin{equation*}
    \left(1-\dfrac{z}{s_n}\right)^{m_n}e^{A_n(z)}=1 \quad \forall z \in U_{\left|s_n\right|}(0)
\end{equation*}
with $A_n(0)=0$. The power series $A_n$ converges uniformly in each compact subset of the disk $U_{\left|s_n\right|}(0)$. Therefore, truncating this power series for $A_n$, we can easily find a polynomial $P_n(z)$ with

\begin{equation*}
    \left|1 - \left(1-\dfrac{z}{s_n}\right)^{m_n}e^{P_n(z)}\right|\leq \dfrac{1}{n^2} \quad \text{for all $z$ with } |z| \leq \dfrac{1}{2}\left|s_n\right|.
\end{equation*}
Since the series $1+\frac{1}{4}+\frac{1}{9}+\cdots$ converges, we arrive at the theorem:

\begin{theorem}
The series

\begin{equation*}
    \sum_{n=1}^{\infty} \left|1 - \left(1-\dfrac{z}{s_n}\right)^{m_n}e^{P_n(z)}\right|\leq \dfrac{1}{n^2} \quad \text{for all $z$ with } |z| \leq \dfrac{1}{2}\left|s_n\right|
\end{equation*}
converges normally.
\end{theorem}
Concerning the problem propounded initially, we can now formulate Weierstra\ss's factorisation theorem.

\begin{theorem}[Weierstra\ss's factorisation theorem]
Let $S \in \mathbb{C}$ be a discrete subset. Further, let the following map be given

\begin{equation*}
    m:S \rightarrow \mathbb{N}, \quad s \mapsto m_s.
\end{equation*}
Then, there exists an analytic function

\begin{equation*}
    f: \mathbb{C} \rightarrow \mathbb{C}
\end{equation*}
with the properties:\\
1) $S:=\lbrace z \in \mathbb{C}| f(z)=0\rbrace$ \\
2) $m_s= \operatorname{ord}(f;s)$. (Order of the zero.)
\end{theorem}

\subsubsection{Comparison to Euler's Idea}
\label{subsubsec: Comparison to Euler's Idea}

Let us compare both, Weierstra\ss's idea from 1878 in \cite{We78} and Euler's idea from \cite{E613}, a paper written in 1780 and published in 1813. In his paper, Euler added zero and expressed the series under consideration in an alternative way to obtain a more convergent series, whereas Weierstra\ss{} told us to  multiply by additional exponential of polynomials to ensure the convergence. But it is easily seen that both ideas are actually equivalent. The following quote from the Introduction written by G. Faber in  Volume 16,2 of the first series of Euler's Opera Omnia confirms this. See p. XLIII.\\[2mm]

    \textit{Tatsächlich hat Euler nicht nur die Produktdarstellung (12) [this means the product expansion of the $\Gamma$-function], sondern sogar den Gedanken der Konvergenz erzeugenden Faktoren von Weierstra\ss{} vorweggenommen. Denn es bedeuted keinen Unterschied, ob man den Gliedern des divergenten Produktes $\prod_{\nu=1}^{\infty}\left(1+\frac{x}{\nu}\right)$ die Konvergenz erzeugenden Fakoren $e^{-\frac{x}{\nu}}$ oder den Gliedern der divergenten unendlichen Reihe $\sum_{\nu=1}^{\infty}\log \left(1+\frac{x}{\nu}\right)$ die Konvergenz erzeugenden Summanden $-\frac{x}{\nu}$ oder auch $-x \log \left(1+\frac{1}{\nu}\right)$ beifügt. Das tat aber Euler mit voller Absicht in der Abhandlung 613.}\\

For, Euler's idea translates into the one of Weierstra\ss{} by considering sums of logarithms. Euler even told us how to find those factors you need to enforce convergence\footnote{Indeed, there is even a prescription how to find those factors for the Weierstra\ss{} product. Confer, e.g. \cite{Fr06}.}. Therefore, Euler actually anticipated the idea of Weierstra\ss{} factors without actually intending it. His intention, as we saw above, was the interpolation of a sum, which is defined for positive integer numbers, to all numbers.\\
Nevertheless, Weierstra\ss's name is attached to the idea, since he constructed a rigorous theory of infinite products and he, by solving the problem propounded above \ref{subsec: Modern Idea - Weierstrass product}, also provided the mathematical community with a large class of analytic functions. Euler's contribution was maybe overlooked, since he was interested in something completely different and did not point out the generality of his method clearly enough.
\newpage

\section{Relation among $\Gamma$ and $B$}
\label{sec: Relation between Gamma and B}

This section is entirely devoted to the connection between the $\Gamma$- and $B$-function.

\subsection{From $B$ to $\Gamma$ - Euler's first Way to the Integral Representation}
\label{subsec: From B to Gamma - Euler's first Way to the Integral Representation}

\subsubsection{Euler's Thought Process}
\label{subsubsec: Euler's Thought Process}

The 1738 paper \cite{E19} is interesting, since Euler described his thought process how he got the idea that the $\Gamma$-function  can be expressed as an integral. This provides us with a beautiful example of the Ars inveniendi (Art of Finding). He explains his thoughts in $\S\S 3-7$. He wrote:\\[2mm]

{\em I had believed before that the general term of the series $1,2,6,24$ etc., if not algebraic, is nevertheless given as an exponential. But after I had understood that certain terms depend on the quadrature of the circle, I realized that neither algebraic nor exponential quantities suffice to express it. [...].\\
But after I had considered that among differential quantities there are formulas of such a kind, which admit an integration in certain cases and then yield algebraic quantities, but in others do not admit an integration and then exhibited quantities depending on quadratures, it came to mind that maybe formulas of this kind are apt to  express the general terms of the mentioned and other progressions.[...].\\
But the differential formula must contain a certain variable quantity.[...]\\
For the sake of clarity, I say that $\int pdx$ is the  general term of the progression to be found as follows from it; but let $p$ denote a function of $x$ and constants, amongst which here $n$ must be contained\footnote{Euler wanted the parameter integral to express the $n$-th term of the progression. Thus, the parameter integral must contain $n$, which is not to be integrated over.}. Imagine $pdx$ to be integrated and such a constant to be added that for $x=0$ the whole integral vanishes; then set $x$ equal to a certain known quantity. Having done this, if in the found integral only quantities extending to the progression remain, it will express the term, whose index is $n$. In other words, the integral determined like this will be the general term.[...]\\
Therefore, I considered many differential formulas only admitting an integration, if one takes $n$ to be a positive integer number, so that the principal terms become algebraic, and hence formed progressions.}\\[2mm]
To summarize Euler's idea in modern formulation: He propounded that there is a function $p(x,n)$ such that

\begin{equation*}
    n ! = \int\limits_{a}^{b}p(x,n)dx \quad \text{for} \quad n \in \mathbb{N}.
    \end{equation*}
In the following paragraphs, he really tried out several different functions and integrals (which essentially all boil down to the $B$-function) and eventually arrived at the integral representation of the $\Gamma$-function. We will discuss his proof in the following sections, in section \ref{subsubsec: Euler's Mathematical Argument} and in section \ref{subsubsec: Using the B-function}.\\     
But here we want to stress  that the theory of parameter integrals did not really exist at the time. Euler was the first to develop this theory. Furthermore, this is one, if not the first paper, in which parameter integrals or functions defined through integrals have been discussed.

\subsubsection{Euler's  Mathematical Argument}
\label{subsubsec: Euler's Mathematical Argument}

It is interesting, how Euler arrived at the integral representation of $\Gamma(x)$ for the first time in his 1738 paper \cite{E19}. He repeated his argument in 1772 in \cite{E421}. For a review of the following argument, using Euler's notation, confer, e.g., the article on the $\Gamma$-function in \cite{Sa15}, \cite{Va06} or \cite{Du99}. First, he showed that

\begin{equation*}
    \dfrac{1 \cdot 2 \cdot 3 \cdots n}{(f+g)(f+2g)\cdots (f+ng)}= \dfrac{f+(n+1)g}{g^{n+1}}\int\limits_{0}^{1} x^{\frac{f}{g}}dx(1-x)^n.
\end{equation*}
This is easily proved by induction. Now it is easy to see that one arrives at an expression for $1\cdot 2 \cdots n =n!$, if one sets $g=0$, i.e.

\begin{equation*}
    \dfrac{n!}{f^n} = f\lim_{g \rightarrow 0} \int\limits_{0}^{1}\dfrac{x^{\frac{f}{g}}(1-x)^n}{g^{n+1}}dx.
\end{equation*}
To get rid of the $g$ in the denominator, Euler set $x=y^{\frac{g}{f+g}}$ (and uses then $y=x$ again) and arrived at:

\begin{equation*}
   \dfrac{n!}{f^n} = f\lim_{g \rightarrow 0} \int\limits_{0}^{1}\dfrac{g}{f+g}\dfrac{\left(1-x^{\frac{g}{f+g}}\right)^n}{g^{n+1}}dx.
\end{equation*}
We are mainly interested in the case $f=1$, i.e.

\begin{equation*}
    n! =\lim_{g \rightarrow 0} \int\limits_{0}^{1}\dfrac{g}{1+g}\dfrac{\left(1-x^{\frac{g}{1+g}}\right)^n}{g^{n+1}}dx.
\end{equation*}
But this is the same as:

\begin{equation*}
    n! = \lim_{g \rightarrow 0} \int\limits_{0}^{1}\dfrac{\left(1-x^{\frac{g}{1+g}}\right)^n}{g^{n}}dx.
\end{equation*}
 We can pull the limit into the integral, i.e.

\begin{equation*}
    n! =  \int\limits_{0}^{1}\lim_{g \rightarrow 0}\dfrac{\left(1-x^{g}\right)^n}{g^{n}}dx.
\end{equation*}
Note that we already took the limit\footnote{This is allowed since all functions are continuous.} in the denominator  of the power of $x$. Let us rewrite this as

\begin{equation*}
    n! =  \int\limits_{0}^{1}\left(\lim_{g \rightarrow 0}\dfrac{(1-x^{g})}{g}\right)^ndx.
\end{equation*}
This is allowed, since $x^n$ is a continuous function. But for natural $n$ this limit can be found by L'Hospital's rule. And one finds:

\begin{equation*}
    n ! = \int\limits_{0}^{1} \left(\log \dfrac{1}{x}\right)^ndx.
\end{equation*}
This is the integral representation of $\Gamma(n+1)$. Hence it is easily understood why Euler preferred to work with this representation. He was led naturally to it.

\subsubsection{Using the $B$-function}
\label{subsubsec: Using the B-function}

Euler did not consider the $B$-function as an independent function at the time he wrote \cite{E19} (in 1738) and \cite{E421}  is devoted to the integral representation of the $\Gamma$-function. Therefore, let us see how Euler's argument can be formulated using the known properties of $B$.\\
We also start from

\begin{equation*}
    {n!} = \lim_{g \rightarrow 0} \int\limits_{0}^{1}\dfrac{x^{\frac{1}{g}}(1-x)^n}{g^{n+1}}dx.
\end{equation*}
Let us rewrite the integral as a $B$-function:

\begin{equation*}
    {n!} = \lim_{g \rightarrow 0} \dfrac{B\left(\frac{1}{g}+1, n+1\right)}{g^{n+1}}dx.
\end{equation*}
Using the functional equation $B(x+1,y)= \frac{x}{x+y}B(x,y)$, we have

\begin{equation*}
    {n!} = \lim_{g \rightarrow 0} \dfrac{\frac{1}{g}}{\frac{1}{g}+n+1}\dfrac{B\left(\frac{1}{g}, n+1\right)}{g^{n+1}}.
\end{equation*}
It is more convenient to put $\frac{1}{g}=h$ and consider the following limit

\begin{equation*}
    {n!} = \lim_{h \rightarrow \infty} \dfrac{h}{n+h+1} h^{n+1}B\left(h, n+1\right).
\end{equation*}
This is the same limit as

\begin{equation*}
    {n!} = \lim_{h \rightarrow \infty}  h^{n+1}B\left(h, n+1\right).
\end{equation*}
Let us use the relation $B(x,y)=\frac{\Gamma(x)\Gamma(y)}{\Gamma(x+y)}$:

\begin{equation*}
    {n!} = \lim_{h \rightarrow \infty}  h^{n+1} \dfrac{\Gamma(h)\Gamma(n+1)}{\Gamma(n+h+1)}.
\end{equation*}
Using the functional equation of the $\Gamma$-function $h$ times in the denominator, we find

\begin{equation*}
     {n!} = \lim_{h \rightarrow \infty}  h^{n+1} \dfrac{\Gamma(h)\Gamma(n+1)}{\Gamma(n+1)(n+h)(n+h-1)\cdots (n+1)}.
\end{equation*}
Or equivalently

\begin{equation*}
     {n!} = \lim_{h \rightarrow \infty}   \dfrac{h^{n+1}\Gamma(h)}{(n+h)(n+h-1)\cdots (n+1)}.
\end{equation*}
Interestingly, we arrived at the condition Weierstra\ss{}  used to define $\Gamma$.

\subsubsection{Gau\ss's Idea}
\label{subsubsec: Gauss's Idea}

Now that we have seen that it is possible to get to the $\Gamma$-function from the $B$-function, for the sake of completeness, let us also briefly mention Gau\ss's idea, which he presented in his 1828 paper \cite{Ga28}. His idea is basically the same as Euler's. We already mentioned that Gau\ss{} defined the $\Gamma$-function as Weierstra\ss{} did later\footnote{ Weiersta\ss{} followed Gau\ss's example in \cite{Ga28} and explicitly said so in \cite{We78}.} as the above limit, i.e. he defined

\begin{equation*}
    \Gamma(x)= \lim_{n \rightarrow \infty} \dfrac{n^x n!}{(x+1)\cdot (x+2)\cdots (x+n)}.
\end{equation*}
And Gau\ss{} observed, essentially as Euler did in 1738 in \cite{E19}, that

\begin{equation*}
    \dfrac{n^x n!}{(x+1)\cdot (x+2)\cdots (x+n)} = \int\limits_{0}^{n} t^{x-1}\left(1-\dfrac{t}{n}\right)^{n-1}dt
\end{equation*}
and hence

\begin{equation*}
     \Gamma(x)= \lim_{n \rightarrow \infty} \dfrac{n^x n!}{(x+1)\cdot (x+2)\cdots (x+n)} = \int\limits_{0}^{n} t^{x-1}\left(1-\dfrac{t}{n}\right)^{n-1}dt \quad \text{for} \quad \operatorname{Re}(x)>0.
\end{equation*}
Hence he concluded:

\begin{equation*}
    \Gamma(x) = \int\limits_{0}^{\infty} e^{-t}t^{x-1}dt \quad \text{for} \quad \operatorname{Re}(x)>0.
\end{equation*}
It is interesting that, although starting from the same idea, the same identity even, Gau\ss{} and Euler got to the integral representation in such different ways. Euler's proof can even be considered rigorous by today's standards\footnote{Euler did not know how to reason rigorously that the limit can be pulled inside the integral, he simply did it without any reasoning. But it can be justified by means of the Lebesgue integral.}, whereas Gau\ss's approach is harder to make it rigorous. A rigorous proof was given by Schl\"omilch \cite{Sc79} and is also presented in Nielsen's book \cite{Ni05}.

\subsection{Expressing the $B$-function via $\Gamma$-functions - Fundamental Relation}
\label{subsec: Expressing the B-function via Gamma-functions - Fundamental Relation}

We want to start with the formula, already proved above, relating the $\Gamma$- and $B$-function, i.e. the formula

\begin{equation*}
    B(x,y)=\dfrac{\Gamma(x)\Gamma(y)}{\Gamma(x+y)}.
\end{equation*}

\subsubsection{Euler's Proof}
\label{subsubsec: Euler's Proof}

As already indicated above in section \ref{subsubsec: Remarks}, Euler's proof is not  rigorous by modern standards, and based on the extension of the validity of the formula for natural numbers $x$ and $y$ to all numbers. Nevertheless, we present Euler's arguments here, since it is interesting to see how he discovered the  result. In \cite{E421} ($\S 26$), he stated the following equation

\begin{equation*}
    \dfrac{\int\limits_{0}^{1}dx\left(\log \frac{1}{x}\right)^{n-1}\cdot \int\limits_{0}^{1} dx \left(\log \frac{1}{x}\right)^{m-1}}{\int\limits_{0}^{1}dx \left(\log \frac{1}{x}\right)^{m+n-1}}= k \int\limits_{0}^{1}x^{mk-1}dx(1-x^k)^{n-1},
\end{equation*}
which reduces to the desired relation for $m=1$. But in his paper, he only established this equation for natural numbers $n$ and $m$ and not in general. In the following paragraphs, he then established the formula for some more fractional numbers, but not for the general case. \\
In the next two sections, we will consider Dirichlet's proof from \cite{Di39} and Jacobi's proof from \cite{Ja34} and see why it was difficult for Euler to prove the identity in general. 

\subsubsection{Jacobi's Proof}
\label{subsubsec: Jacobi's Proof}

We present Jacobi's proof of the fundamental relation from this 1834 paper \cite{Ja34}, the proof is also given in \cite{Ni05}.

\begin{proof}
Let $x$, $y$ $>0$ and consider:

\begin{equation*}
    \Gamma(x)\Gamma(y) = \int\limits_{0}^{\infty} t^{x-1}e^{-t}dt \cdot \int\limits_{0}^{\infty} u^{y-1}e^{-u}du =  \int\limits_{0}^{\infty} \int\limits_{0}^{\infty} e^{-(t+u)}t^{x-1}u^{y-1}dtdu.
\end{equation*}
Make the substitution:

\begin{equation*}
    t+u = \alpha(u,t), \quad u=\alpha(u,t)\beta(u,t).
\end{equation*}
This gives:

\begin{equation*}
     \renewcommand{\arraystretch}{1,5}
\setlength{\arraycolsep}{0.0mm}
\begin{array}{llllll}
    t =\alpha(1-\beta), & \quad u(1-\beta)=t \beta, \\ 
    dt = (1- \beta)d \alpha, & \quad (1-\beta)du = \alpha d \beta, 
\end{array}
\end{equation*}
which immediately leads to the formula:

\begin{equation*}
    \Gamma(x)\Gamma(y) =  \int\limits_{0}^{\infty} e^{-\alpha}\alpha^{x+y-1}d \alpha \cdot \int\limits_{0}^{1}\beta^{x-1}(1- \beta)^{y-1}d \beta.
\end{equation*}
This implies the desired formula.
\end{proof}
Obviously, the hard part is to find the substitution and actually it is only possible, if you know in advance how the final result looks like.

\subsubsection{Dirichlet's Proof}
\label{subsubsec: Dirichlet's Proof}

Dirichlet's proof \cite{Di39} is a bit more straight-forward.

\begin{proof}
We start from the following expression of $B$

\begin{equation*}
    B(x,y) = \int\limits_{0}^{\infty} \dfrac{t^{x-1}dt}{(1+t)^{x+y}}.
\end{equation*}
It is obtained from $B(x,y)= \int\limits_{0}^{1}dt t^{x-1}(1-t)^{y-1}$ by setting $t=\frac{1}{z}$ and then $z=u+1$.
Furthermore, setting $t=ky$, provided $k>0$, in the integral representation of the $\Gamma$-function, we have

\begin{equation*}
    \int\limits_{0}^{\infty}e^{-tk}t^{x-1}dt = \dfrac{\Gamma(x)}{k^x}.
\end{equation*}
Therefore, from the above representation of the $B$-function

\begin{equation*}
    B(x,y)= \dfrac{1}{\Gamma(x+y)}\cdot \int\limits_{0}^{\infty} t^{x-1}dt  \int\limits_{0}^{\infty} e^{-(1+t)u}\cdot u^{x+y-1}du.
\end{equation*}
We can exchange the order of integration here by Fubini's theorem such that

\begin{equation*}
    B(x,y) = \dfrac{1}{\Gamma(x+y)}\cdot \int\limits_{0}^{\infty}e^{-u}u^{x-y-1}du \cdot \int\limits_{0}^{\infty} e^{-tu}t^{x-1}dt.
\end{equation*}
Performing the integral over $t$:

\begin{equation*}
    B(x,y) = \dfrac{1}{\Gamma(x+y)} \int\limits_{0}^{\infty} e^{-u}u^{x+y-1} \cdot \dfrac{\Gamma(x)}{u^x}.
\end{equation*}
Finally, we arrive at:

\begin{equation*}
    B(x,y) = \dfrac{\Gamma(x)\Gamma(y)}{\Gamma(x+y)}.
\end{equation*}
\end{proof}

\subsubsection{Discussion}
\label{subsubsec: Discussion}

Considering these proofs, Jacobi's proof was out of Euler's reach. For, Euler did not know how to perform a substitution in the case of several variables; in other words, he did not know the concept of the Jacobi determinant. This was only explained in 1841 by Jacobi in \cite{Ja41}, as we already mentioned in section \ref{subsubsec: Some Remarks}. Additionally, the  anticommutativity of the wedge product puzzled him in his only paper devoted explicitly to double integrals \cite{E391}\footnote{The concept of the wedge product was introduced much later by Cartan. Indeed, Euler, could not explain why his results seemed to indicate that $dxdy= -dydx$ and tried to argue it away. In hindsight, we might say that he encountered the wedge product for the first time.}. See also Katz's  article in \cite{Du07} on the subject.\\[2mm]
Dirichlet's proof on the other hand was  definitely within Euler's reach. Especially, since Euler also knew the formula 

\begin{equation*}
    \dfrac{\Gamma(x)}{k^x} = \int\limits_{0}^{\infty}e^{-kt}t^{x-1}dt
\end{equation*}
and derived many extraordinary integrals from it in \cite{E675}, including the Fresnel integrals. Although Fubini's theorem was proven only a lot later, this would not have  troubled Euler, since he basically considered integrals as ordinary sums, just over infinitesimally small numbers. In general, the exchange of limits did not bother him at all. The necessity to prove the validity of such a procedure only came after Abel's proof of what we now call Abel's limit theorem for series, i.e. more than 30 years after Euler's death.\\
Interestingly, although Euler obviously knew the formula

\begin{equation*}
    B(x,y) = \int\limits_{0}^{1} t^{x-1}(1-t)^{y-1}dt
\end{equation*}
and devoted some of his papers to examine it, e.g., \cite{E321}, \cite{E640} and even a chapter in his second book on integral calculus \cite{E366}, he never explicitly stated the formula:

\begin{equation*}
    B(x,y) = \int\limits_{0}^{1} t^{x-1}(1-t)^{y-1}dt = \int\limits_{0}^{\infty}\dfrac{t^{x-1}}{(1+t)^{x+y}},
\end{equation*}
which would have simplified his investigations and results on definite integrals a lot. Indeed, almost all his papers on definite integrals (contained in the Opera Omnia, Series 1, Volumes 17-19) can be understood very easily using the $B$- and $\Gamma$-functions and their derivatives and the fundamental relation connecting them that we considered in this section. But it seems that Euler did not realize this, at least he does not mention it in any of his papers or books.

\subsection{Relation to the Beta Function - Expressing $\Gamma$ via $B$}
\label{subsec: Relation to the Beta Function - Expressing Gamma via B}

In \cite{E19} and \cite{E122}, Euler stated, but not proved, a formula which expresses $\Gamma(\frac{p}{q})$ in terms of a product of several $B$-functions\footnote{Euler wanted $p$ and $q$ to be natural numbers, but we will see that this restriction is not necessary.}. This is interesting, since the $\Gamma$-function  involves a transcendental integrand, whereas $B$ is an integral over an algebraic function (if the variables are rational numbers) and hence is a period in the sense of Zagier and Kontsevich \cite{Ko01}. Euler stated the formula as follows in \cite{E19}

\begin{equation*}
\int\limits_{0}^{1} \left(-\log x\right)^{\frac{p}{q}}dx =\sqrt[q]{1 \cdot 2 \cdot 3 \cdots p\left(\dfrac{2p}{q}+1\right)\left(\dfrac{3p}{q}+1\right)\left(\dfrac{4p}{q}+1\right)\cdots \left(\dfrac{qp}{q}+1\right)}
\end{equation*}
\begin{equation*}
\times \sqrt[q]{\int\limits_{0}^{1} dx(x-xx)^{\frac{p}{q}} \cdot \int\limits_{0}^{1} dx(x^2-x^3)^{\frac{p}{q}} \cdot \int\limits_{0}^{1} dx(x^3-x^4)^{\frac{p}{q}} \cdot \int\limits_{0}^{1} dx(x^4-x^5)^{\frac{p}{q}} \cdots \int\limits_{0}^{1} dx(x^{q-1}-x^q)^{\frac{p}{q}}}.
\end{equation*}
Replacing the integral $\int\limits_{0}^{1} \left(-\log x\right)^{\frac{p}{q}}dx$ by $\Gamma \left(\frac{p}{q}+1\right)$ and $1 \cdot 2 \cdot 3 \cdots p$ by $\Gamma(p+1)$, we can formulate the theorem as follows:

\begin{theorem}
Let $p>0$ and $q \in \mathbb{N}$. Then, the following formula holds:

\begin{equation*}
    \Gamma \left(\dfrac{p}{q}+1\right) = \sqrt[q]{\Gamma (p+1) \times \prod_{k=2}^{q}\left(\dfrac{kp}{q}+1\right) \times \prod_{i=1}^{q-1} \int\limits_{0}^{1} dx(x^{i-1}-x^i)^{\frac{p}{q}}}.
\end{equation*}
\end{theorem}
\begin{proof}
The idea is to express the integrals as $B$-functions, and then rewrite them in terms of $\Gamma$-functions. Let us consider the $q$-th power of the right-hand side and call it $G(p,q)$. Then

\begin{equation*}
    G(p,q) = \Gamma (p+1) \times \prod_{k=2}^{q}\left(\dfrac{kp}{q}+1\right) \times \prod_{i=2}^{q} \int\limits_{0}^{1} x^{(i-1)\frac{p}{q}}dx(1-x)^{\frac{p}{q}}.
\end{equation*}
The integrals can  be rewritten in terms of $B$:

\begin{equation*}
 G(p,q) = \Gamma (p+1) \times \prod_{k=2}^{q}\left(\dfrac{kp}{q}+1\right) \times \prod_{i=2}^{q} B \left(\dfrac{(i-1)p}{q}+1, \dfrac{p}{q}+1\right).    
\end{equation*}
Next, express the $B$-functions via $\Gamma$-functions:

\begin{equation*}
 G(p,q) = \Gamma (p+1) \times \prod_{k=2}^{q}\left(\dfrac{kp}{q}+1\right) \times \prod_{i=2}^{q}\dfrac{\Gamma \left(\frac{(i-1)p}{q}+1\right)\Gamma\left(\frac{p}{q}+1\right)}{\Gamma \left(\frac{ip}{q}+2\right)}.
\end{equation*}
$\Gamma \left(\frac{p}{q}\right)$ does not depend on the multiplication index. Therefore, we can pull it out of the product. Additionally, let us use the functional equation of the $\Gamma$-function to rewrite the denominator. Then, we arrive at:

\begin{equation*}
    G(p,q) = \Gamma(p+1) \cdot \left(\Gamma \left(\frac{p}{q}+1\right)\right)^{q-1} \times \prod_{k=2}^{q}\left(\dfrac{kp}{q}+1\right) \times \prod_{i=2}^{q}\dfrac{\Gamma \left(\frac{(i-1)p}{q}+1\right)}{\left(\frac{ip}{q}+1\right)\Gamma \left(\frac{ip}{q}+1\right)}.
\end{equation*}
Finally, writing everything as one product

\begin{equation*}
    G(p,q) = \Gamma (p+1) \cdot \left(\Gamma \left(\frac{p}{q}+1\right)\right)^{q-1} \times \prod_{i=2}^{q} \dfrac{\left(\frac{(ip}{q}+1\right)\cdot \Gamma \left(\frac{(i-1)p}{q}+1\right)}{\left(\frac{ip}{q}+1\right)\cdot\Gamma \left(\frac{ip}{q}+1\right)}.
\end{equation*}
Almost every factor in the product cancels. Indeed, it just remains

\begin{equation*}
    \dfrac{\Gamma\left(\frac{p}{q}+1\right)}{\left(\frac{p}{q}+1\right)\cdot \Gamma \left(\frac{pq}{q}+1\right)} = \dfrac{\Gamma \left(\frac{p}{q}+1\right)}{\Gamma(p+1)}.
\end{equation*}
Hence we arrive at
\begin{equation*}
    G(p,q) =  \left(\Gamma \left(\frac{p}{q}+1\right)\right)^q.
\end{equation*}
Therefore, we arrived at the desired result.
\end{proof}

\subsection{Reflection Formula}
\label{subsec: Reflection Formula}

\begin{theorem}[Reflection formula for the $\Gamma$-function]
We have:

\begin{equation*}
    {\Gamma(x)\Gamma(1-x)}=\dfrac{\pi}{\sin (\pi x)}.
\end{equation*}
\end{theorem}

We want to discuss the proof of the theorem that Euler gave in \cite{E421} and one he could have given. For Euler's proof, we need to prove the product formula for the sine first.

\subsubsection{Euler's Proof of the Sine Product Formula}
\label{subsubsec: Euler's Proof of the Sine Product Formula}

The product formula for the sine, i.e.

\begin{equation*}
    \sin (\pi x)= x \pi \prod_{k=1}^{\infty}\left(1-\dfrac{x^2}{k^2}\right),
\end{equation*}
is one of Euler's most beautiful discoveries. He discovered it in 1740 in \cite{E41} and used it in the solution of the Basel problem, i.e. the summation over the reciprocals of the squares in the same paper.

\begin{center}
\begin{figure}
\centering
    \includegraphics[scale=1.1]{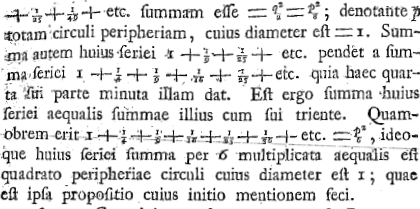}
    \caption{Solution of the Basel Problem}
     Taken from \cite{E41} $\S 11$. Euler arrives at the famous formula $1+\frac{1}{2^2}+\frac{1}{3^2}+\cdots = \frac{\pi^2}{6}$ for the first time. He writes simply $p$ for $\pi$.
     \end{figure}
\end{center}
Due to  criticism from Bernoulli and Cramer, he gave a proof in 1741 in \cite{E61}, which he repeated in \cite{E101}, a book published in 1748. 

\begin{center}
\begin{figure}
\centering
    \includegraphics[scale=0.9]{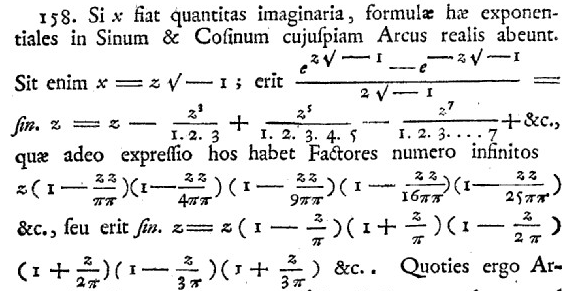}
    \caption{Sine Product Formula}
     Taken from the book \cite{E101}. Euler arrives at the sine product formula. His proof given there is rigorous even by modern standards.
    \end{figure}
\end{center}
Concerning the criticism, Euler in \cite{E61} $\S 6$ wrote:

\begin{center}
    {\em  This almost completely oppressed worry\footnote{The worry concerns the correctness of his product formula for the sine and the solution of the Basel problem derived from it.} has recently been renewed by  letters from Daniel Bernoulli, in which he shared the same worries concerning my method and also mentioned that Cramer  has the same doubts whether my method is right.}
\end{center}
The main point of criticism was that in \cite{E41} Euler did not prove that the sine only has the integer numbers times $\pi$ as roots. In other words, Euler could not rule out complex roots at that time.\\[2mm]
But in \cite{E61}, one year after \cite{E41}, he then offered a proof for the sine product formula. For this, he considered the expression

\begin{equation*}
a^n - b^n,
\end{equation*}
$n$ being a natural number, and its factorization into real factors. The  complex factors are easily found by de Moivre's theorem. They appear in complex conjugate pairs and hence, combing each into a real factor, each real factor reads

\begin{equation*}
a^2 -2 ab \cos \dfrac{2k}{n}\pi +b^2,
\end{equation*}
$k$ being a natural number with $2k<n$. If $n$ is odd, one has the additional factor $a-b$.\\[2mm]
In the next step, Euler used the famous identity named after him to write

\begin{equation*}
\sin s = \dfrac{e^{is}-e^{-is}}{2i}
\end{equation*}
and the definition of $e^s$

\begin{equation*}
e^s := \lim_{n \rightarrow \infty} \left(1+\dfrac{s}{n}\right)^n.
\end{equation*}
Therefore, Euler considered the expression

\begin{equation*}
\dfrac{\left(1+\frac{is}{n}\right)^n - \left(1-\frac{is}{n}\right)^n}{2i}
\end{equation*}
and noted that for infinite $n$ this expression goes over into $\sin s$. \\[2mm]
This expression has the form of the general expression above and hence each factor is given by the form we ascribed to it above. Here, $a= 1+\frac{is}{n}$, $b=1-\frac{si}{n}$ and hence each factor has the form

\begin{equation*}
2-\dfrac{2s^2}{n^2}-2\left(1+\dfrac{s^2}{n^2}\right)\cos \dfrac{2k\pi}{n}.
\end{equation*}
Therefore, all the factors of the sine are obtained, if all natural numbers are substituted for $k$.\\
But since now $n$ is infinite we have

\begin{equation*}
\cos \dfrac{2k\pi}{n}= 1- \dfrac{2k^2 \pi^2}{n^2},
\end{equation*}
at least approximately, whence our factor  becomes

\begin{equation*}
-\dfrac{4s^2}{n^2}+\dfrac{4k^2\pi^2}{n^2},
\end{equation*}
or, if we want each factor to have the form $1-a_k$, the general factor will be

\begin{equation*}
1-\dfrac{s^2}{k^2\pi^2}.
\end{equation*}
Therefore, we already have

\begin{equation*}
\sin s = As \prod_{k=1}^{\infty}\left(1-\dfrac{s^2}{k^2\pi^2}\right).
\end{equation*}
The constant $A$ is easily seen to be $=1$ from the well-known limit $\lim_{x\rightarrow 0} \frac{\sin x}{x}=1$. This completes Euler's proof of the sine product formula.\\[2mm]
Although some arguments are not completely rigorous, the general idea is correct and one can work out a proof meeting today's standard of rigor. Confer, e.g., \cite{Va06} for a modern proof based on Euler's ideas. Let us mention that in \cite{E664} Euler started from the right-hand side of the last equation\footnote{More precisely, he used the corresponding expression of $\cos x$, since the proof is easier working with the product of $\cos x$.} and proved that it is equal to $\sin x$.

\subsubsection{Euler's Proof of the Reflection Formula}
\label{subsubsec: Euler's Proof of the Reflection Formula}

 The proof, we are about to give, is the one from \cite{E421}. 
 
 \begin{center}
 \begin{figure}
 \centering
     \includegraphics[scale=1.1]{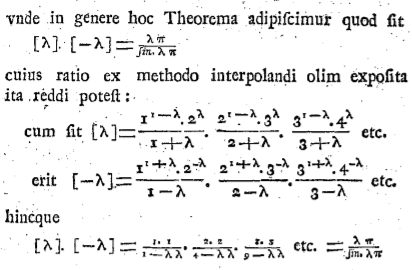}
     \caption{Reflection Formula for the $\Gamma$-function}
      Taken from E421. Euler proves the reflection formula. He uses the symbol $[\lambda]$ to denote $\lambda !$. Thus, in modern notation $[\lambda] = \Gamma(\lambda +1)$.
     \end{figure}
 \end{center}

\begin{proof}
Above we had the formula 

\begin{equation*}
   \Gamma(x+1) = 1^{x} \cdot \dfrac{1^{1-x}2^x}{x+1} \cdot \dfrac{2^{1-x}3^x}{x+2} \cdot \dfrac{3^{1-x}4^x}{x+3} \cdot \text{etc.} = \prod_{k=1}^{\infty} \dfrac{k^{1-x}(k+1)^x}{x+k}.
\end{equation*}
Therefore, replacing $x$  by $-x$:

\begin{equation*}
   \Gamma(1-x) = 1^{1+x} \cdot \dfrac{1^{1+x}2^{-x}}{1-x} \cdot \dfrac{2^{1+x}3^{-x}}{2-x} \cdot \dfrac{3^{1+x}4^{-x}}{3-x} \cdot \text{etc.} = \prod_{k=1}^{\infty} \dfrac{k^{1+x}(k+1)^{-x}}{k-x}.
\end{equation*}
Therefore, 

\begin{equation*}
    \Gamma(1+x)\Gamma(1-x) = \dfrac{1^2}{x^2-1^2} \cdot \dfrac{2^2}{x^2-2^2} \cdot \dfrac{3^2}{3^2-x^2}\cdots
\end{equation*}
or in compact notation:

\begin{equation*}
      \Gamma(1+x)\Gamma(1-x) = \prod_{k=1}^{\infty} \dfrac{k^2}{k^2-x^2} = \prod_{k=1}^{\infty} \dfrac{1}{1- \frac{x^2}{k^2}}
\end{equation*}
The product is the well-known product formula for the sine such that

\begin{equation*}
    \Gamma(x+1)\Gamma(1-x) = \dfrac{x\pi}{\sin (\pi x)}.
\end{equation*}
Using the functional equation of the $\Gamma$-function, one arrives at the desired formula.

\end{proof}

\subsubsection{Proof Euler could have given}
\label{subsubsec: Proof Euler could have given}

In \cite{E59}, Euler stated and, in \cite{E60}, \cite{E61} and \cite{E462}, then proved the following identity:

\begin{equation*}
\int\limits_0^1 \dfrac{t^{x-1}-t^{-x}}{1+t}dt = \dfrac{\pi}{\sin \pi x} \quad 0<x<1.
\end{equation*}
He did this by considering integrals over rational fractions and solving them by partial fraction decomposition. But using the partial fraction decomposition of $\frac{\pi}{\sin \pi x}$, one will easily see the identity to be true by expanding the denominator into a geometric series and integrating term by term. Therefore, we will not give the proof here.\\[2mm]
Instead, we want to prove the reflection formula in the most simple way using only formulas at Euler's disposal. The theorem to be proved is still the same, i.e that

\begin{equation*}
    \Gamma(x)\Gamma(1-x) =\dfrac{ \pi }{\sin (\pi x)}.
\end{equation*}
\begin{proof}
We have

\begin{equation*}
    \Gamma(x)\Gamma(1-x)= B(x,1-x)= \int\limits_{0}^{\infty}\dfrac{t^{x-1}}{1+t}dt
\end{equation*}
by using the relation among the $\Gamma$- and $B$-function discussed in section \ref{subsec: Expressing the B-function via Gamma-functions - Fundamental Relation} and using the integral representation with limits $0$ and $\infty$ for $B$. Next, we split the integral into two integrals $\int\limits_{0}^{1} + \int\limits_{1}^{\infty}$ and put $t =\frac{1}{u}$ in the second. Then, calling $u$ $t$ again, we arrive at:

\begin{equation*}
    \Gamma(x)\Gamma(1-x) = \int\limits_0^1 \dfrac{t^{x-1}-t^{-x}}{1+t}dt = \dfrac{\pi}{\sin \pi x}.
\end{equation*}
\end{proof}

\subsection{A slight Detour: Gamma Class Conjectures}

At this point, it seems appropriate to mention a rather recent application of the $\Gamma$-function, more precisely the reflection formula, which was the subject of the last section. We will follow \cite{Ga18}.\\
We want to mention briefly the so-called \textit{Gamma Conjecture} and how the $\Gamma$-function is connected to it. This will also show that there is a recent interest in the $\Gamma$-function.

\subsubsection{Gamma Class}

Let us  describe briefly what a Gamma class is. The description follows \cite{Ga16}. The \textit{Gamma class} of a complex manifold $X$ is the cohomology class

\begin{equation*}
    \widehat{\Gamma}_X= \prod_{i=1}^{n}\Gamma(1+ \delta_i) \in H^{\bullet}(X,\mathbb{R})
\end{equation*}
where $\delta_1$, $\delta_2$, $\cdots$, $\delta_n$ are the Chern roots of the tangent bundle $TX$ and $\Gamma(x)$ is our $\Gamma$-function. Above we have seen the Taylor series for $\log(\Gamma(1+x))$, i.e.

\begin{equation*}
    \log (\Gamma(1+x))= -\gamma x + \sum_{k=2}^{\infty} (-1)^k\zeta(k)x^k.
\end{equation*}
$\gamma$ is the Euler-Mascheroni constant, $\zeta(s)$ denotes the Riemann zeta function. This Taylor series, which we gave in section \ref{subsubsec: Euler's Results on the Gamma-function}, implies the following formula for the Gamma Class:

\begin{equation*}
    \widehat{\Gamma}_X= \exp \left(-\gamma c_1(X)+ \sum_{k=2}^{\infty}(-1)^k (k-1)! \zeta(k) \operatorname{ch}_k(TX)\right)
\end{equation*}
$\operatorname{ch}_k$ is the $k$-th Chern character.

\subsubsection{Gamma Conjectures}

The \textit{Gamma Conjecture},  stated, e.g., in \cite{Ga16} and \cite{Ga18}, is a conjecture relating quantum cohomology of a Fano manifold with its topology. The small quantum cohomology of $F$ defines a flat connection, often referred to as a quantum connection, over $\mathbb{C}^{\times}$ and its solution is a cohomology-valued function $J_F(t)$ called \textit{J-function}. Under certain conditions, not to be discussed here, the limit of the $J$-function 

\begin{equation*}
    A_F : = \lim_{t \rightarrow \infty} \dfrac{J_F(t)}{\left\langle \left[\operatorname{pt}\right],J_F(t) \right\rangle} \in H^{\bullet}(F)
\end{equation*}
exists and defines  the \textit{principal asymptotic class} $A_F$ of $F$. \textit{Gamma Conjecture I} states that $A_F$ equals the Gamma class $\widehat{\Gamma}_F= \widehat{\Gamma}(TF)$ of the tangent bundle of $F$, i.e. that

\begin{equation*}
    A_F = \widehat{\Gamma}_F.
\end{equation*}
Now suppose that the quantum cohomology of $F$ is semisimple. Then, one can define \textit{higher} asymptotic classes $A_{F,i}$ with $1 \leq i \leq N =\operatorname{dim}H^{\bullet}(F)$ from exponential asymptotics of flat sections of the quantum connection. \\
\textit{Gamma Conjecture II} states that there exists a full exceptional collection $E_1, E_2, \cdots, E_n$ of $D_{\text{coh}^b}(F)$ such that

\begin{equation*}
    A_{F,i}= \widehat{\Gamma}_F\cdot \operatorname{Ch}(E_i) \quad i=1, \cdots, N.
\end{equation*}
Here,

\begin{equation*}
    \operatorname{Ch}(E):= (2\pi i)^{\frac{deg}{2}}\operatorname{ch}(E)= \sum_{p=0}^{\operatorname{dim}F}(2\pi i)^p \operatorname{ch}_p (E)
\end{equation*}
is the Chern character. The principal asymptotic class $A_F$ corresponds to the exceptional object $E=O_F$.\\[2mm]
Now let us see how the $\Gamma$-function enters into this: One can say that the \textit{Gamma conjecture} is something like the \textit{square root} of the index theorem. To explain this, first recall the Hirzebuch-Riemann-Roch formula:

\begin{equation*}
    \chi(E_1,E_2)= \int_F \operatorname{ch}\left(E_1^{\vee}\right)\cdot \operatorname{ch}(E_2) \cdot \operatorname{td}_F
\end{equation*}
for vector bundles $E_1, E_2$ of $F$. $\chi(E_1,E_2)= \sum_{i=0}^{\operatorname{dim}F}(-1)^i \operatorname{dim}\operatorname{Ext}^{i}(E_1,E_2) $ is the Euler-Pairing and $\operatorname{td}_F= \operatorname{td}(TF)$ is the Todd class of $F$. We have seen in \ref{subsec: Reflection Formula} that Euler proved

\begin{equation*}
    \Gamma(1-x)\Gamma(1+x)= \dfrac{\pi x}{\sin (\pi x)}.
\end{equation*}
Rewriting the sines in terms of exponential functions and then writing $\frac{x}{2\pi i}$ instead of $x$, one arrives at

\begin{equation*}
    \dfrac{x}{1-e^{-x}} = e^{\frac{x}{2}} \Gamma\left(1-\dfrac{x}{2\pi i}\right)\Gamma\left(1+\dfrac{x}{2\pi i}\right).
\end{equation*}
Recall that we have seen the left-hand side of this equation while discussing  the Bernoulli numbers in section \ref{subsubsec: Purely formal Derivation}. \\
Anyhow, the factorisation of the $\Gamma$-functions  gives the factorisation of the Todd class:

\begin{equation*}
    (2\pi i)^{\frac{\operatorname{deg}}{2}}\operatorname{td}_F =e^{\pi i c_1(F)}\cdot \widehat{\Gamma}_F \cdot \widehat{\Gamma}^{\ast}_{{F}}.
\end{equation*}
This in turn, factorises the Hirzebuch-Riemann-Roch formula:

\begin{equation*}
    \chi(E_1,E_2)= \left[\operatorname{Ch}(E_1)\cdot \widehat{\Gamma}_F, \operatorname{Ch}(E_2)\cdot \widehat{\Gamma}_F\right).
\end{equation*}
In these formulas: $\widehat{\Gamma}^{\ast}_{{F}}= (-1)^{\frac{\operatorname{deg}}{2}} \widehat{\Gamma}_F = \sum_{p=0}^{\operatorname{dim}F}(-1)^p \gamma_p$, if one writes $\widehat{\Gamma}_F = \sum_{p=0}^{\operatorname{dim}F}\gamma_p$ with $\gamma_p \in H^{2p}(F)$, and $\left[\cdot, \cdot\right)$ is a non-symmetric pairing of $H^{\bullet}(F)$ given by

\begin{equation*}
    \left[\alpha, \beta\right)= \dfrac{1}{(2 \pi)^{\operatorname{dim}F}}\int_F \left(e^{\pi i c_1(X)}e^{\pi i \mu} \alpha\right) \cup \beta
\end{equation*}
with $\mu \in \operatorname{End}H^{\bullet}(F)$ defined by $\mu (\phi)= \left(p-\dfrac{\operatorname{dim}F}{2}\right)\phi$ for $\phi \in H^{2p}(F)$.

\subsection{Multiplication Formula}
\label{subsec: Multiplication Formula}

One of the fundamental properties of the $\Gamma$-function is the so-called {\em multiplication formula} that reads in  modern notation
\begin{equation*}
\Gamma \left(\dfrac{x}{n}\right)\Gamma \left(\dfrac{x+1}{n}\right)\cdots \Gamma \left(\dfrac{x+n-1}{n}\right)= \dfrac{(2\pi)^{\frac{n-1}{2}}}{n^{x-\frac{1}{2}}}\cdot \Gamma(x).
\end{equation*}
For $n=2$ one obtains the {\em duplication formula} that is usually ascribed to Legendre, who proved it in \cite{Le26}; we met this formula already in \ref{subsubsec: Finding the other constants}.\\[2mm]
The multiplication formula was first proven by Gau\ss{} in  his influential 1828 paper \cite{Ga28} on the hypergeometric series, in which he also gave a complete account of the factorial function $\Pi(x):=\Gamma(x+1)=x!$.  
Gau\ss{} cited Euler's results very often, but apparently he was not aware of the lesser-known paper \cite{E421} of Euler. At least, he did not cite it.

\begin{center}
\begin{figure}
\centering
    \includegraphics[scale=1.1]{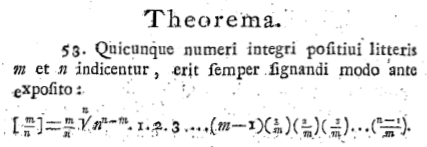}
    \caption{Euler's Version of the Multiplication Formula for the $\Gamma$-function}
    This formula, found in $\S 53$ of \cite{E421}, is equivalent to the Gau\ss ian multiplication formula above. The restriction to natural numbers $m$ is not necessary. Although Euler states it as a theorem, his proof is not rigorous.
    \end{figure}
\end{center}
In that paper Euler presented a formula that is equivalent to the above equation (see figure), as we will explain now. We need to introduce some notation.

\subsubsection{The Function $\left(\frac{p}{q}\right)$}
\label{subsubsec: The Function pq)}

In $\S 3$ of \cite{E321} and $\S 44$ of \cite{E421}, Euler studied properties of the function 
\begin{equation*}
\left(\frac{p}{q}\right):= \int\limits_{0}^{1} \dfrac{x^{p-1}dx}{(1-x^n)^{\frac{n-q}{n}}}.
\end{equation*}
In his notation, the variable $n$ is left implicit, and Euler showed the nice symmetry property
\[ \left(\frac{p}{q}\right)=\left(\frac{q}{p}\right) .\]
By the substitution $x^n=y$ it is clear that this function is just the Beta function in disguise: 
\begin{equation*}
\left(\frac{p}{q}\right) = \frac{1}{n} \int\limits_{0}^1 y^{\frac{p}{n}-1}dy(1-y)^{\frac{q}{n}-1}=\frac{1}{n} B \left(\frac{p}{n},\frac{q}{n}\right),
\end{equation*}
where the Beta function still is defined as
\begin{equation*}
B(x,y)= \int\limits_{0}^1 t^{x-1}dt(1-t)^{y-1} \quad \text{for} \quad \text{Re}(x),\text{Re}(y) > 0.
\end{equation*}
Euler implicitly assumed $p$ and $q$ to be natural numbers,  but
this restriction is not necessary.\\
 As mentioned several times, see, e.g. section \ref{subsec: Expressing the B-function via Gamma-functions - Fundamental Relation}, Euler already knew the relation among Beta integral and the $\Gamma$-function:
\begin{equation*}
B(x,y) = \dfrac{\Gamma(x) \cdot \Gamma(y)}{\Gamma(x+y)}.
\end{equation*}

\subsubsection{Applying the Reflection Formula}
\label{subsubsec: Applying the Reflection Formula}

Euler's version  of the reflection formula for the $\Gamma$-function, 
\begin{equation*}
\dfrac{\pi}{\sin \pi x}= \Gamma (x)\Gamma (1-x),
\end{equation*}
can be found in  $\S 43$ of \cite{E421} and reads
\[ [\lambda]\cdot [-\lambda]=\frac{\pi \lambda}{\sin \pi \lambda} ,\]
where, in Euler's notation, $[\lambda]$ stands for $\lambda !$, that is $\Gamma(1+\lambda)$.\\
If one applies the reflection formula for $x=\frac{i}{n}$, $i=1,2,\cdots,n-1$, one obtains
$$
\begin{array}{rcl}
\Gamma \left(\dfrac{1}{n}\right)  \Gamma \left(\dfrac{n-1}{n}\right)&=&\dfrac{\pi}{\sin \frac{\pi}{n}},\\
\Gamma \left(\dfrac{2}{n}\right)  \Gamma \left(\dfrac{n-2}{n}\right)&=&\dfrac{\pi}{\sin \frac{2\pi}{n}},\\
\Gamma \left(\dfrac{3}{n}\right)  \Gamma \left(\dfrac{n-3}{n}\right)&=&\dfrac{\pi}{\sin \frac{3\pi}{n}},\\
\ldots&=&\ldots\\
\Gamma \left(\dfrac{n-1}{n}\right)  \Gamma \left(\dfrac{1}{n}\right)&=&\dfrac{\pi}{\sin \frac{(n-1)\pi}{n}} .\\
\end{array}
$$
Multiplying these equations  gives our first auxiliary formula

\begin{equation*}
 \prod_{i=1}^{n-1}\Gamma \left(\frac{i}{n}\right)^2= \frac{\pi^{n-1}}{\prod_{i=1}^{n-1} \sin \left(\frac{i \pi}{n}\right)} .
\end{equation*}
Our second auxiliary formula is
\begin{equation*}
\prod_{i=1}^{n-1} \sin \left(\frac{i \pi}{n} \right) = \frac{n}{2^{n-1}} ,
\end{equation*}
which was certainly known to Euler.
For example, in $\S 7$ of \cite{E562} and in $\S 240$ of \cite{E101}, he stated the more general formula

\begin{equation*}
\sin n \varphi = 2^{n-1} \sin \varphi \sin \left(\dfrac{\pi}{n}- \varphi\right) \sin \left(\dfrac{\pi}{n}+ \varphi\right)
\end{equation*}
\begin{equation*}
 \sin \left(\dfrac{2\pi}{n}- \varphi\right) \sin \left(\dfrac{2\pi}{n}+ \varphi\right)\cdot\text{etc.}
\end{equation*}
The product has $n$ factors in total. If we divide by $2^{n-1}\sin \varphi$, 
use $\sin \left(\frac{\pi(n-i)}{n}\right)= \sin \left(\frac{i\pi}{n}\right)$
and take the limit $\varphi \rightarrow 0$, we obtain the second auxiliary 
formula.\\[2mm]
The first and the second auxiliary formula were also given by Gau\ss{} in \cite{Ga28} and are used in his proof of the multiplication formula\footnote{Gau\ss's proof follows the same lines as ours concerning the use of the auxiliary formulas.}.\\[2mm]
Combining them and taking the square root, we obtain the beautiful formula
\begin{equation*}
\Gamma \left(\frac{1}{n}\right) \Gamma \left( \frac{2}{n}\right) \cdots \Gamma \left(\frac{n-1}{n}\right)=\sqrt{\frac{(2\pi)^{n-1}}{n}}.
\end{equation*}
This formula was also found by Euler in $\S 46$ of \cite{E816}, where he stated it in the form

\begin{equation*}
\int\limits_{0}^{1}dx \left(\log \dfrac{1}{x}\right)^{\frac{1}{n}}\int\limits_{0}^{1}dx \left(\log \dfrac{1}{x}\right)^{\frac{2}{n}}\cdots \int\limits_{0}^{1}dx \left(\log \dfrac{1}{x}\right)^{\frac{n-1}{n}}= \dfrac{1 \cdot 2 \cdot 3 \cdots (n-1)}{n^{n-1}}\sqrt{\dfrac{2^{n-1}\pi^{n-1}}{n}}.
\end{equation*}

\subsubsection{Euler's Version of the Multiplication Formula}
\label{subsubsec: Euler's Version of the Multiplication Formula}

In $\S 53$ of \cite{E421}, Euler gave the formula

\begin{equation*}
\left[ \frac{m}{n} \right] = \frac{m}{n} \sqrt[n]{n^{n-m}\cdot 1 \cdot 2 \cdot 3 \cdots (m-1) \left(\frac{1}{m}\right)\left(\frac{2}{m}\right)\left(\frac{3}{m}\right)\cdots \left(\frac{n-1}{m}\right)}.
\end{equation*}
As before, $[\lambda]$ is Euler's notation for the factorial of $\lambda$ such that $\left[ \frac{m}{n} \right] = \Gamma \left(\frac{m}{n}+1\right)$. Euler assumed $m$ and $n$ to be natural numbers, but it is easily seen that we can interpolate $1 \cdot 2 \cdot 3 \cdots (m-1)$ by $\Gamma(m)$. Therefore, if we assume $x$ to be real and positive and write $x$ instead of $m$ in the above formula and
express it in terms of the Beta function, Euler's formula becomes

\begin{equation*}
\Gamma \left(\frac{x}{n}\right)= \sqrt[n]{n^{n-x} \Gamma(x) \frac{1}{n^{n-1}}B \left(\frac{1}{n},\frac{x}{n}\right)B \left(\frac{2}{n},\frac{x}{n}\right)\cdots B \left(\frac{n-1}{n},\frac{x}{n}\right)}.
\end{equation*}
Expressing the $B$-function in terms of the $\Gamma$-function, then after  some rearrangement under the $\sqrt[n]{}$-sign we obtain

\begin{equation*}
\Gamma \left(\frac{x}{n}\right)= \sqrt[n]{n^{1-x}\Gamma(x) \dfrac{\Gamma\left(\frac{1}{n}\right)\Gamma \left(\frac{x}{n}\right)}{\Gamma \left(\frac{x+1}{n}\right)} \cdot \dfrac{\Gamma\left(\frac{2}{n}\right)\Gamma \left(\frac{x}{n}\right)}{\Gamma \left(\frac{x+2}{n}\right)} \cdots \dfrac{\Gamma\left(\frac{n-1}{n}\right)\Gamma \left(\frac{x}{n}\right)}{\Gamma \left(\frac{x+n-1}{n}\right)}}.
\end{equation*}
Bringing all $\Gamma$-functions of fractional argument to the left-hand side, the expression simplifies to

\begin{equation*}
\Gamma \left(\frac{x}{n}\right)\Gamma \left(\frac{x+1}{n}\right) \Gamma \left(\frac{x+2}{n}\right) \cdots \Gamma \left(\frac{x+n-1}{n}\right)= n^{1-x} \Gamma (x) \Gamma \left(\frac{1}{n}\right) \cdots \Gamma \left(\frac{n-1}{n}\right).
\end{equation*}
The product on the right-hand side, $\Gamma \left(\frac{1}{n}\right) \cdots \Gamma \left(\frac{n-1}{n}\right)$, was evaluated above and thus we obtain 
\begin{equation*}
\Gamma \left(\frac{x}{n}\right)\Gamma \left(\frac{x+1}{n}\right) \Gamma \left(\frac{x+2}{n}\right) \cdots \Gamma \left(\frac{x+n-1}{n}\right)=  n^{1-x} \Gamma(x) \sqrt{\dfrac{(2\pi)^{n-1}}{n}}.
\end{equation*}
Thus, we arrived at the multiplication formula.

\subsubsection{Euler's Proof of the Multiplication Formula}

As mentioned, Euler never had a complete proof for the multiplication formula. Basically, this was because he did not prove the fundamental relation to the $B$-function for all real numbers. Nevertheless, for the sake of completeness, we present Euler's proof which he offered in the  first supplement to \cite{E421}. His proof is based on the formula that he wrote as

\begin{equation*}
    \dfrac{\left[\frac{m}{n}\right]\left[\frac{\lambda}{n}\right]}{\left[\frac{\lambda+m}{n}\right]}= \dfrac{\lambda m}{\lambda +m} \left(\dfrac{\lambda}{m}\right).
\end{equation*}
This is just the fundamental relation among the $B$- and $\Gamma$-function. In modern notation, the formula reads:

\begin{equation*}
    \dfrac{\Gamma\left(\frac{m}{n}+1\right)\Gamma \left(\frac{\lambda}{n}+1\right)}{\Gamma \left(\frac{\lambda+m}{n}+1\right)}= \dfrac{\lambda m  }{\lambda +m}\cdot \dfrac{1}{n} B\left(\dfrac{\lambda}{n}, \dfrac{m}{n}\right).
\end{equation*}

His idea was to substitute all natural numbers from $1$ to $n$ for $\lambda$ and multiply the resulting expressions. Then, the multiplication formula will follow from simple manipulation of the product. On the right-hand side he finds, in modern notation:

\begin{equation*}
    \dfrac{m}{1+m}\cdot \dfrac{2m}{2+m} \cdots \dfrac{nm}{n+m}\cdot \dfrac{1}{n^n} B\left(\dfrac{1}{n}, \dfrac{m}{n}\right) B\left(\dfrac{1}{n}, \dfrac{m}{n}\right)\cdots B\left(\dfrac{n}{n}, \dfrac{m}{n}\right)
\end{equation*}
which simplifies to

\begin{equation*}
    \dfrac{m^n}{n^n}\cdot \dfrac{m!}{(n+1)(n+2)\cdots (n+m)} B\left(\dfrac{1}{n}, \dfrac{m}{n}\right) B\left(\dfrac{1}{n}, \dfrac{m}{n}\right)\cdots B\left(\dfrac{n}{n}, \dfrac{m}{n}\right)
\end{equation*}
On the left-hand side on the other hand:

\begin{equation*}
    \dfrac{\Gamma \left(\frac{m}{n}+1\right)\Gamma \left(\frac{1}{n}+1\right)}{\Gamma \left(\frac{1+m}{n}+1\right)}\cdot  \dfrac{\Gamma \left(\frac{m}{n}+1\right)\Gamma \left(\frac{2}{n}+1\right)}{\Gamma \left(\frac{2+m}{n}+1\right)}\cdots  \dfrac{\Gamma \left(\frac{m}{n}+1\right)\Gamma \left(\frac{m}{n}+1\right)}{\Gamma \left(\frac{n+m}{n}+1\right)}
\end{equation*}
\begin{equation*}
    = \left(\Gamma\left(\dfrac{m}{n}+1\right)\right)^n \cdot \dfrac{\Gamma \left(\frac{1}{n}+1\right)\Gamma \left(\frac{2}{n}+1\right)\cdots \Gamma \left(\frac{n}{n}+1\right)}{\Gamma \left(\frac{1+m}{n}+1\right)\cdot \Gamma \left(\frac{2+m}{n}+1\right)\cdots \Gamma \left(\frac{m+n}{n}+1\right)}.
\end{equation*}
This simplifies to

\begin{equation*}
     \left(\Gamma\left(\dfrac{m}{n}+1\right)\right)^n \cdot \dfrac{\Gamma \left(\frac{1}{n}+1\right)\Gamma \left(\frac{2}{n}+1\right)\cdots \Gamma \left(\frac{m}{n}+1\right)}{\Gamma \left(\frac{1+n}{n}+1\right)\cdot \Gamma \left(\frac{2+n}{n}+1\right)\cdots \Gamma \left(\frac{m+n}{n}+1\right)}
\end{equation*}
Finally, using the recursive relation in the denominator and cancelling all the $\Gamma$-functions:

\begin{equation*}
    =  \left(\Gamma\left(\dfrac{m}{n}+1\right)\right)^n \cdot \dfrac{n^m}{(n+1)(n+2)\cdots (n+m)}.
\end{equation*}
Therefore, equating the final result of the left-hand and right-hand side, we arrive at the equation:

\begin{equation*}
     \dfrac{m^n}{n^n}\cdot \dfrac{m!}{(n+1)(n+2)\cdots (n+m)} B\left(\dfrac{1}{n}, \dfrac{m}{n}\right) B\left(\dfrac{1}{n}, \dfrac{m}{n}\right)\cdots B\left(\dfrac{n}{n}, \dfrac{m}{n}\right)
\end{equation*}
\begin{equation*}
    =\left(\Gamma\left(\dfrac{m}{n}+1\right)\right)^n \cdot \dfrac{n^m}{(n+1)(n+2)\cdots (n+m)}.
\end{equation*}
Therefore, replacing also $m!$ by $\Gamma(m+1)$, we find

\begin{equation*}
    \left(\Gamma\left(\dfrac{m}{n}+1\right)\right)^n = m^n \Gamma(m+1) n^{-n-m} B\left(\dfrac{1}{n}, \dfrac{m}{n}\right)\cdots B\left(\dfrac{n}{n}, \dfrac{m}{n}\right).
\end{equation*}
This is just Euler's version of the multiplication formula in modern notation and thus finishes his proof.\\

\subsubsection{Discussion of Euler's Result}
\label{subsubsec: Discussion of Euler's Result}

From the above sketch it is apparent that in \cite{E421} Euler had a result that is  equivalent to the multiplication formula for the $\Gamma$-function. He expressed it in terms of the symbol $\left(\frac{p}{q}\right)$, which  in modern notation is the Beta function. One may wonder why Euler did not express his result in terms of the $\Gamma$-function itself. Reading his paper, it becomes clear that his main motivation was to express the factorial of rational numbers in terms of integrals of \textit{algebraic functions}, and the formula given by Euler fulfills this purpose. For the same reason, he probably did not replace $1 \cdot 2 \cdot 3 \cdots (m-1)$ by $\Gamma(m)$.\\[2mm]
Euler also expressed $\Gamma(\frac{p}{q})$, as we saw in \ref{subsec: Relation to the Beta Function - Expressing Gamma via B} in terms of integrals of algebraic algebraic functions in $\S 23$ \cite{E19} and $\S 5$ of \cite{E122}. That formula reads

\begin{equation*}
\int\limits_{0}^{1} \left(-\log x\right)^{\frac{p}{q}}dx =\sqrt[q]{1 \cdot 2 \cdot 3 \cdots p\left(\dfrac{2p}{q}+1\right)\left(\dfrac{3p}{q}+1\right)\left(\dfrac{4p}{q}+1\right)\cdots \left(\dfrac{qp}{q}+1\right)}
\end{equation*}
\begin{equation*}
\times \sqrt[q]{\int\limits_{0}^{1} dx(x-xx)^{\frac{p}{q}} \cdot \int\limits_{0}^{1} dx(x^2-x^3)^{\frac{p}{q}} \cdot \int\limits_{0}^{1} dx(x^3-x^4)^{\frac{p}{q}} \cdot \int\limits_{0}^{1} dx(x^4-x^5)^{\frac{p}{q}} \cdots \int\limits_{0}^{1} dx(x^{q-1}-x^q)^{\frac{p}{q}}}.
\end{equation*}
Despite the similarity to the first formula expressing $\Gamma$ via $B$, that formula is not as general as the multiplication formula\footnote{We want to mention here that in the foreword of the Opera Omnia, series 1, volume 19, p. LXI A. Krazer and G. Faber claim that these two formulas are equivalent and both are a special case of the multiplication formula. This is incorrect, as it was shown in the preceding sections. The formula given in section \ref{subsec: Relation to the Beta Function - Expressing Gamma via B} does not lead to the multiplication formula, it only interpolates $\Gamma \left(\frac{p}{q}\right)$ in terms of algebraic integrals.}.\\[2mm]
It appears that Euler was aware that the proofs he indicated in \cite{E421} were not completely convincing. 
He expressed that  with characteristic honesty in a concluding SCHOLIUM:\\

{\em  Hence infinitely many relations among the integral formulas of the form
\[ \int \dfrac{x^{p-1}dx}{(1-x^n)^{\frac{n-q}{n}}}=\left(\frac{p}{q} \right)\]
follow, which are even more remarkable, because we were led to them by a
completely singular method. And if anyone does not believe them to be true,
he or she should consult my observations on these integral formulas\footnote{Here Euler refers to his paper \cite{E321}.} 
and will then hence easily be convinced of their truth for any case. But even if
this consideration provides some confirmation, the relations found here are
nevertheless of even greater importance, because a certain structure is noticed
in them and they are easily generalized to all classes, whatever number was
assumed for the exponent $n$, whereas in the first treatment the calculation for the higher classes becomes continuously more cumbersome and intricate.}\\[2mm]

\newpage

\section{Summary}
\label{sec: Summary}

\subsection{Overview of Euler's Results on the $\Gamma$-Function}
\label{subsec: Overview over Euler's Results on the Gamma-Function}

\subsubsection{General Overview}
\label{subsubsec: General Overview summary}

As we have seen, Euler basically already found all common representations of the $\Gamma$-function ranging from the integral representation to the product representation, although he never got credit for the latter. The common denominator of all his ways to these representations was an attempt to solve the functional equation $\Gamma(x+1)=x\Gamma(x)$. He basically had four different approaches: The moment method (section \ref{sec:  Euler's direct Solution of the Equation Gamma (x+1)=xGamma(x) - The Moment-Ansatz}), which gave the integral represantation, solution by conversion into a differential equation (section \ref{sec:  Solution of the Equation log Gamma(x+1)- log  Gamma (x)= log x by Conversion into a Differential Equation of infinite Order}), which led to the Euler-Maclaurin formula (section \ref{sec: Solution of log Gamma(x+1)- log Gamma (x)= log x via the Euler-Maclaurin Formula}) and from it to the Stirling formula (section \ref{subsubsec: An Application - Derivation of the Stirling Formula for the Factorial}),  difference calculus (section \ref{sec: Interpolation Theory and Difference Calculus}), which led him to Weierstra\ss{} product formula (section \ref{subsec: Difference Calculus according to Euler}),  and direct iterative application of the functional equation, which led him to the Gau\ss{} product formula (section \ref{subsec: Euler's Idea concerning the Gamma-function}).\\
Furthermore, he discovered several special properties like the reflection formula (section \ref{subsec: Reflection Formula}), the multiplication formula \ref{subsec: Multiplication Formula}, and the relation to the $B$-function (section \ref{subsec: Expressing the B-function via Gamma-functions - Fundamental Relation}). Thus, it is safe to say that throughout his career he discovered all basic properties of the $\Gamma$-function. \\[3mm]

He could have found some more results that were then later discovered by others. As the foremost example, we want to mention the Fourier series expansion of $\log(\Gamma(x))$. This is usually attributed to Kummer \cite{Ku47}, but an equivalent result was obtained one year earlier by Malmsten \cite{Ma46} (p. 25)\footnote{Malmsten's paper was published later than Kummer's, but written earlier in 1846. Both authors clearly obtained their results independently, since the  techniques applied are very different.}. Fourier series were introduced later in Fourier's monumental treatise on heat \cite{Fo22} in 1822. But certain examples appear at seemingly random places in Euler's work and in \cite{E189} he actually showed that every periodic function has a Fourier expansion. But at that time, he did not pay any further attention to it. He gave the formula for the Fourier coefficients then later in \cite{E704} but did not make the connection to his earlier findings and did not realize the importance of those findings as Fourier later did.

\begin{center}
\begin{figure}
\centering
    \includegraphics[scale=0.9]{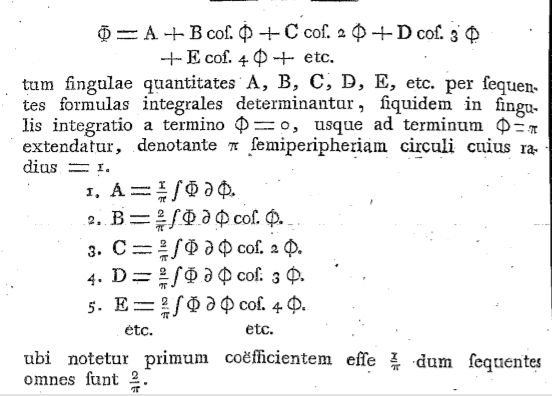}
    \caption{Formula for the Fourier Coefficients}
     Taken from \cite{E704}. Euler states the formula for the Fourier series of a symmetric function. He proves it formally in the following paragraphs of this paper.
    \end{figure}
\end{center}

But there  are also results that were out of his reach. Those concern all results which require the  theory of complex functions, like the classification theorems discussed in section \ref{subsec: Classification Theorems}. Those are necessary in order to see, whether two different expressions for $\Gamma$ are actually identical. Euler never felt the need to do so and never even addressed that issue. Indeed, it is quite difficult to show directly that all the different expressions are identical.\\
Another property that Euler rarely talked about is that the $\Gamma$-function is a transcendental function and even such values as $\Gamma\left(\frac{1}{2}\right)$ are transcendental. Although Euler had a notion of what a transcendental number is, he never actually gave a definition.\\[3mm]
Unfortunately, Euler never pointed out the connections between all his findings concerning the $\Gamma$-function and never organized all his results with the exception of \cite{E421}. But that paper mainly focused on the integral representation and the other representations are not discussed. Probably, the closest to an overview article might be the paper \cite{E368}, in which he listed almost all of the formulas\footnote{Compare also Figure 2.24 and 2.25}, which we discussed throughout the text.\\
But it seems that in this  paper he was more interested in the evaluation of other values of the $\Gamma$-function than $\Gamma\left(\frac{1}{2}\right)$ which can be expressed in terms of known constants. His attemps were fruitless at the end due to the higher transcendality of the $\Gamma$-function pointed out in \ref{subsubsec: 7. Phase: Transcendence Questions}. It would have been really intriguing to see, what Euler had to say about all the connections we pointed out.\\
But despite all this, it was interesting to see what Euler actually already knew about the $\Gamma$-function and how many ideas he actually anticipated and that this is not generally known. Thus, we have added some interesting details to the history of the $\Gamma$-function, wonderfully told in \cite{Da59}. Therefore, I hope, our article on Euler and the $\Gamma$-function  provides a motivation to go through other historical texts and see what treasures they might contain.

\begin{center}
\begin{figure}
\centering
    \includegraphics[scale=1.1]{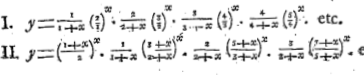}
    \caption{Euler's Formulas for $\Gamma(x)$ - Part 1}
    First two of Euler's formulas for the factorial from his paper \cite{E368}.
    \end{figure}
\end{center}

\begin{center}
\begin{figure}
\centering
    \includegraphics[scale=1.0]{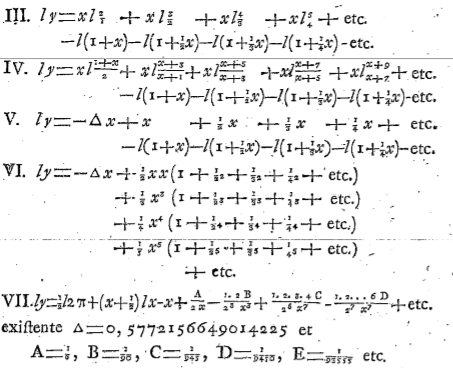}
    \caption{Euler's Formulas for $\Gamma(x)$ - Part 2}
     Remainder of Euler's formulas for the factorial from his paper \cite{E368}. It contains all the formulas he ever found for the factorial. He writes simply $l$ for $\log$, $\Delta$ is the Euler-Mascheroni constant $\gamma$, $A, B, C$ denote the rational part of $\zeta(2n)$, i.e. $\zeta(2)=A\pi^2$, $\zeta(4)=B\pi^4$ etc.
    \end{figure}
\end{center}

\subsubsection{Results not generally associated with Euler}
\label{subsubsec: Results not generally associated with Euler}

In this section, we wish to give an overview of the results that Euler already obtained, but that are not usually attributed to him. Or in other words, some of his lesser-known results:\\[2mm]

\textit{1. Moment Method}:\\
In \cite{E123} and \cite{E594}, two papers actually devoted to continued fractions, Euler finds a method which can be used to solve homogeneous difference equations with linear coefficients in terms of definite integrals (section \ref{sec: Euler's direct Solution of the Equation Gamma (x+1)=xGamma(x) - The Moment-Ansatz}).\\
In this context, we mention that in \cite{E551} he already found the recursive definition of the Legendre polynomials (see  section \ref{subsubsec: Historical Remark on Legendre Polynomials}). Using the moment ansatz, one can find an explicit formula for the $n$-th Legendre polynomial in terms of definite integrals. This formula was found (up to a multiplicative constant) by Euler in \cite{E606}. But it seems that Euler was not aware of the connections between \cite{E551} and \cite{E606}.
Moreover, we explained that the moment method can be considered as the first appearance of the Mellin-Transform \ref{subsec: Euler and the Mellin Transform}.\\[2mm]

\textit{2. Anticipation of convergence inducing factors:}\\
In \cite{E613}, a paper actually devoted to interpolation problems, Euler anticipated Weierstra\ss's idea of introducing factors forcing a product to converge (see section \ref{subsec: Difference Calculus according to Euler}). \\[2mm]

\textit{3. Gau\ss's definition of the factorial and multiplication formula}\\
Gau\ss's definition of the factorial \cite{Ga28} now bearing his name had already been found by Euler in \cite{Eu29} (see also in section \ref{subsec: More detailed Discussion on the Discovery of the Gamma-function}). He proved it in \cite{E652} (see section \ref{subsec: Euler on Weierstrass's Condition}). Additionally, he also found a formula equivalent to the Gau\ss{}ian multiplication formula (see section  \ref{subsec: Multiplication Formula}).\\[2mm]

\textit{4. Solution of differential equations of infinite order:}\\
In \cite{E189}, Euler considered and solved differential equations of infinite order with constant coefficients (see section \ref{sec: Solution of the Equation log Gamma(x+1)- log  Gamma (x)= log x by Conversion into a Differential Equation of infinite Order}). Unfortunately, his  formulas are incorrect in general due to a conceptual error (see section \ref{subsubsec: Mistake in Euler's Approach}).\\[2mm]

\textit{5. Cauchy's definition for convergence of infinite series:}\\
In \cite{E43}, Euler gave the necessary and sufficient criterion for the convergence of an infinite series usually attributed to Cauchy (see section \ref{subsec: gamma meets Gamma - Euler on the Euler-Mascheroni Constant}).\\[2mm]

\textit{6. Partial Fraction Decomposition of transcendental functions:}\\
In \cite{E592}, Euler explained  his method to find the partial fraction decomposition of simple transcendental functions as, e.g., $\frac{1}{\sin(x)}$, $\tan (x)$ (see section in section \ref{subsec: Euler and the Partial Fraction Decomposition of Transcendental Functions}). Anyhow, although all examples given in that paper are correct, his method does not work in general (see section \ref{subsubsec: Example 2}).

\subsubsection{Results erroneously ascribed to Euler}
\label{subsubsec: Results wrongly ascribed to Euler}

There are also some results that are wrongly ascribed to Euler that we learnt about. We also list them here.\\[3mm]

\textit{1. Discovery of the $\Gamma$-function:}\\
Euler was not the first to discover an expression for the $\Gamma$-function. The first was Daniel Bernoulli \cite{Be29}. But Euler was indeed the first to discover the integral representation one year later. Anyhow, his paper \cite{E19} is the first paper published in a scientific journal containing an expression for $\Gamma$-function. See section \ref{subsec: More detailed Discussion on the Discovery of the Gamma-function} for a more detailed discussion on the subject.\\[2mm]

\textit{2. Eulerian Integral Representation for the hypergeometric series:}\\
It seems that Euler never wrote down the integral representation for the hypergeometric series in his papers, although it can easily be derived from his ideas and results in \cite{E254} and \cite{E366}. Legendre seems to be the first to write it down explicitly  in his 1812 book \cite{Le17}. See section \ref{subsubsec: Historical Note on the Integral Representation}.

\newpage

\section{Appendix}
\label{sec: Appendix}

\subsection{A Solution of the Basel Problem that Euler missed (but could have found)}
\label{subsec: A Solution of the Basel Problem that Euler missed (but could have found)}
The solution of the Basel problem, i.e. the summation of the series

\begin{equation*}
    \sum_{k=1}^{\infty}\dfrac{1}{k^2} =1 +\dfrac{1}{4}+\dfrac{1}{9}+\cdots
\end{equation*}
was Euler's first great claim to fame. He gave the correct result the first time in \cite{E41}. But his proof is based on the sine-product that he could not prove at that time. Therefore, the first complete solution of the problem he gave was in \cite{E61} one year later. In his career, he gave at least 5 different proofs of this formula. Here, we want to show, how Euler could have given another proof, i.e. we give a proof using only Euler's results.\\[2mm]
We showed that Euler had an explicit solution of the difference equation

\begin{equation*}
f(x+1)-f(x)=X(x).
\end{equation*}
But the same equation, for integer $x$, is also solved by

\begin{equation*}
f(x) = \sum_{n=1}^{x}X(n-1).
\end{equation*}
To solve the Basel problem, let us consider the special case $X=x^2$, i.e. the equation

\begin{equation*}
f(x+1)-f(x)= x^2.
\end{equation*}
For integer $x$ the solution is

\begin{equation*}
f(x) = \sum_{n=1}^{x}(n-1)^2 =\sum_{n=1}^{x-1}n^2.
\end{equation*}
But applying the Euler-Maclaurin summation formula this sum is easily found  to be

\begin{equation*}
f(x) = \dfrac{1}{6}x(x-1)(2x-1)=\dfrac{x^3}{3}-\dfrac{x^2}{2}+\dfrac{x}{6}.
\end{equation*}
But this solution is seen to satisfy the propounded difference equation for all $x \in \mathbb{C}$! (The general solution  is obtained by adding an arbitrary periodic function to that solution as we discussed in section \ref{subsubsec:  Application to the difference equation f(x+1)-f(x)=g(x)}.).\\[2mm]
But now let us also use the explicit solution of the general difference equation that we found in \ref{subsubsec: Application to the difference equation f(x+1)-f(x)=g(x)}. We find

\begin{equation*}
\setlength{\arraycolsep}{0mm}
\renewcommand{\arraystretch}{2,5}
\begin{array}{llllllllllll}
f(x) & ~=~& \int\limits_{}^{x}t^2dt & ~-~ & \dfrac{x^2}{2} & ~+~ &  \sum_{k \in \mathbb{Z}}^{\prime} e^{2\pi ix}\int\limits_{}^{x}e^{-2 \pi i t}t^2dt \\
& ~=~& \dfrac{x^3}{3}+C & ~ -~ &\dfrac{x^2}{2} &~+~ & \dfrac{x}{\pi^2} \sum_{k=1}^{\infty} \dfrac{1}{k^2}+\sum_{k \in \mathbb{Z}\setminus \left\lbrace 0\right\rbrace} C_ke^{2\pi ix}\\
& ~=~&\dfrac{x^3}{3} & ~-~ &\dfrac{x^2}{2} & ~+~ & \dfrac{x}{\pi^2} \sum_{k=1}^{\infty} \dfrac{1}{k^2} +A(x).
\end{array}
\end{equation*}
The $'$ indicates that the term for $k=0$ is left out in the summation. The $C_k$ are integration constants. $A$ is just an arbitrary periodic function with period one, i.e., a solution of the homogeneous difference equation $f(x+1)-f(x)=0$. \\[2mm]
  Since the solution of a difference equation must be unique (if the solution of the homogeneous equation is subtracted), comparing the coefficients of $x$  and the last equations we find

\begin{equation*}
\dfrac{1}{\pi^2}\sum_{k=1}^{\infty}\dfrac{1}{k^2} = \dfrac{1}{6}.
\end{equation*}
And this is the Basel sum again.\\[2mm]
Although the way in which we obtained this result was rather non-straightforward, it nevertheless only used  results that Euler knew; hence we  claim that he also could have given that solution. \\
Anyhow, this solution seems to be  new, since it does not appear in \cite{Ch03} which is a list of several different solutions for the Basel problem.

\subsection{Euler and the Partial Fraction Decomposition of Transcendental Functions}
\label{subsec: Euler and the Partial Fraction Decomposition of Transcendental Functions}

Euler devoted the whole paper \cite{E592} to the expansion of transcendental functions into partial fractions. Although his method is not correct in general, all formulas he stated in his paper are correct. The flaw in his method can easily be fixed.
We will explain his method in the example of $\frac{1}{\sin x}$, where his method works, but then will also give an example in which his method fails and will explain the reason.

\subsubsection{An Example: $\frac{\pi}{\sin \pi x}$}
\label{subsubsec: Example 1}

In \cite{E592}, Euler wrote down the formula

\begin{equation*}
\dfrac{\pi}{2 \lambda \sin \lambda \pi}-\dfrac{1}{2\lambda^2}=\dfrac{1}{1-\lambda^2}-\dfrac{1}{4-\lambda^2}+\dfrac{1}{9-\lambda^2}-\dfrac{1}{16-\lambda^2}+\text{etc.} \quad \text{valid for} \quad \lambda \in \mathbb{C}\setminus \mathbb{Z}.
\end{equation*}
To find the partial fraction decomposition of $\frac{1}{\sin \varphi}$ and other similar transcendental functions, Euler argued as if they were rational functions. For those he outlined the procedure in \cite{E101}, \cite{E162}, \cite{E163} and gave an improved version in the last chapter of \cite{E212}. Sandifer wrote a nice paper on this method in \cite{Sa07}.\\[2mm]
The first step is, as usual, to find the zeros of the denominator, i.e., $\sin \varphi$. They are, as it was demonstrated in section \ref{subsubsec: Euler's Proof of the Sine Product Formula}, $\varphi_i = i \pi$ with $i \in \mathbb{Z}$. Furthermore, all of those zeros are simple.\\[2mm]
Hence Euler  made the ansatz

\begin{equation*}
\dfrac{1}{\sin \varphi} = \dfrac{A_i}{\varphi -i \pi}+R_i(\varphi),
\end{equation*}
$A_i$ being a constant to be determined, $R_i(\varphi)$ being a function that does not contain the factor $\varphi -i \pi$ or any positive or negative powers of it. To determine the constant $A_i$,  he wrote

\begin{equation*}
A_i = \dfrac{\varphi - i \pi}{\sin \varphi}-R_i(\varphi)(\varphi - i \pi).
\end{equation*}
Since $R_i(\varphi)$ does not contain $\varphi - i \pi$ or any power of it, we  have

\begin{equation*}
A_i = \lim_{\varphi \rightarrow i \pi} \dfrac{\varphi - i \pi}{\sin \varphi} = \lim_{\varphi \rightarrow i \pi} \dfrac{1}{\cos \varphi} = (-1)^i,
\end{equation*}
where L'Hospital's rule was used in the second step.\\[2mm]
Since this method works for all zeros Euler then concluded

\begin{equation*}
\dfrac{1}{\sin \varphi}= +\dfrac{1}{\varphi}-\dfrac{1}{\varphi - \pi}-\dfrac{1}{\varphi + \pi}+\dfrac{1}{\varphi - 2\pi}+\dfrac{1}{\varphi + 2\pi}-\dfrac{1}{\varphi - 3\pi}-\dfrac{1}{\varphi +3\pi}+\cdots
\end{equation*}
This follows from the integral representation we used above to prove the sine product formula in section \ref{subsubsec: Proof Euler could have given} and can vice versa be used to show that the integral formula is correct. Although this formula turns out to be right, Euler's reasoning is not quite correct. We will elaborate on this in the next section. For now, let us simplify the result.\\
Adding each two terms we find

\begin{equation*}
\dfrac{1}{\sin \varphi}-\dfrac{1}{\varphi}= \dfrac{2 \varphi}{\pi^2 - \varphi^2}- \dfrac{2 \varphi}{4\pi^2 - \varphi^2}+ \dfrac{2 \varphi}{9\pi^2 - \varphi^2}- \dfrac{2 \varphi}{16\pi^2 - \varphi^2}+\cdots
\end{equation*}
Diving by $2\varphi$ and then setting $\varphi = \lambda \pi$ we arrive at

\begin{equation*}
\dfrac{\pi}{\lambda \sin \lambda \pi}-\dfrac{1}{2 \lambda^2}=\dfrac{1}{1-\lambda^2}-\dfrac{1}{4-\lambda^2}+\dfrac{1}{9-\lambda^2}-\dfrac{1}{16-\lambda^2}+ \cdots
\end{equation*}
This is the claimed formula and hence completes the proof.

\subsubsection{An Example in which Euler's Method fails - $\frac{1}{e^z-1}$}
\label{subsubsec: Example 2}

In the preceding section, we already mentioned that Euler's method to find the partial fraction decomposition of a transcendental function does not work in general.
We want to illustrate this with an example that we actually already needed above in section \ref{subsec: Correction of Euler's Approach}.\\[2mm]
Let us try to find the partial fraction decomposition of $(e^z-1)^{-1}$.\\
The first step is again to find all the zeros of the denominator. They are $z_k=2 k \pi i$ with an integer number $k$. Furthermore, they are all simple.\\[2mm]
Using Euler's ansatz let, us set

\begin{equation*}
\dfrac{1}{e^z-1}= \dfrac{A_k}{z-2 k \pi i}+ R_k(z),
\end{equation*}
where $A_k$ is a constant we want to determined and $R_k$ is a function not containing $z-2 k \pi i$ or any powers of it. Proceeding as above we find $A_k =1$.\\
Therefore, we can write

\begin{equation*}
\dfrac{1}{e^z-1}= \sum_{k \in \mathbb{Z}}\dfrac{1}{z -2 k \pi i}+ R(z),
\end{equation*}
and we have to determine $R(z)$. Euler simply would have assumed $R(z)$ to be zero according to the procedure outlined in \cite{E592}; this will turn out to be not true. To see this, consider

\begin{equation*}
R(z) = \dfrac{1}{e^z -1} - \sum_{k \in \mathbb{Z}}\dfrac{1}{z -2 k \pi i}.
\end{equation*}
From this we find that $R(z)$ is a bounded holomorphic function and hence constant by Liouville's theorem.\\
To find the value of $R$ consider the Laurent series of $(e^z-1)^{-1}$ around $z=0$. This series turns out to be

\begin{equation*}
\dfrac{1}{e^z -1} = \dfrac{1}{z}\sum_{n=0}^{\infty} \dfrac{B_n}{n!}z^n,
\end{equation*}
where $B_n$ are the Bernoulli numbers. As we mentioned in section \ref{subsec: Euler's Derivation of the Formula}, Euler essentially found this generating function for the Bernoulli number already in \cite{E25}, but  especially pointed it out in his studies concerning the Euler-Maclaurin summation formula, see, e.g.,  \cite{E47}, \cite{E55} and in his book \cite{E212}.\\[2mm]
Using this series we have 

\begin{equation*}
R= \dfrac{B_0}{z}+B_1+\dfrac{1}{z}\sum_{n=2}^{\infty} \dfrac{B_n}{n!}z^n - \dfrac{1}{z}-\sum_{k \in \mathbb{Z}\setminus \lbrace 0\rbrace}^{\prime}\dfrac{1}{z -2 k \pi i}.
\end{equation*}
 But from the definition of the Bernoulli numbers one easily finds $B_0=1$ and $B_1 =-\frac{1}{2}$. Hence putting $z=0$ in the equation for $R$ we obtain

\begin{equation*}
R(z)=R(0)=R=B_1=-\frac{1}{2}.
\end{equation*}
Therefore, we obtain the partial fraction decomposition

\begin{equation*}
\dfrac{1}{e^z-1}=-\dfrac{1}{2}+\sum_{k \in \mathbb{Z}}\dfrac{1}{z -2 k \pi i}.
\end{equation*}
As we saw above in section \ref{subsec: Correction of Euler's Approach}, the extra term of $-\frac{1}{2}$  turns out to be of major importance in the derivation of the solution of the simple difference equation.

\subsection{Euler on the Euler-Mascheroni Constant}
\label{subsec: gamma meets Gamma - Euler on the Euler-Mascheroni Constant}

We needed the Euler-Mascheroni\footnote{Mascheroni's name was added to the constant, since  in his elaborations to Euler's book on integral calculus he found a new formula for $\gamma$ and used it to calculate several digits of it. Additionally, he found several other functions, in whose Taylor series expansion it appears. Mascheroni's ideas are reprinted in the Opera Omnia version of Euler's book \cite{E366}.} constant $\gamma$ in the derivation of the Weierstra\ss{} product formula for $\Gamma(x)$ in section \ref{subsubsec: Finding the constant K}; hence we want to  take the opportunity to consider  Euler's contribution to the constant. The $\gamma$ constant appears frequently at various places in mathematics. Therefore, it is natural that it also appears at various places in Euler's works. But we want to focus mainly on three papers \cite{E43}, the first occurrence of $\gamma$, \cite{E583}, a paper devoted to $\gamma$, and \cite{E629} considering one singular expression for $\gamma$ in more detail\footnote{For an interesting historical account, confer also Sandifer's article in \cite{Sa15}.}.

\subsubsection{Euler's Discovery of $\gamma$}
\label{subsubsec: Euler's Discovery of gamma}

Euler discovered the constant, nowadays called $\gamma$, in 1740 in \cite{E43}. That paper is also interesting for a another reason. Euler stated the modern convergence criterion for the convergence of an infinite series in $\S$ 2; he wrote:

\begin{center}
    {\em A series, which continued to infinity has a finite sum, even though it is continued twice as far, will  never gain any increment, but that what is added after infinity, will actually be infinitely small. For, if it would not be like this, the sum, even though it is continued to infinity, would not be defined and hence not finite. Hence it follows, if that, what results beyond the infinitesimal term, is of finite magnitude, that the sum of the series is necessarily finite.} 
\end{center}
This is essentially Cauchy's definition for the convergence of a series. But let us discuss how Euler discovered $\gamma$. He encounters it the first time in $\S 6$, where he noted, if 

\begin{equation*}
    s := \sum_{n=1}^{i} \dfrac{c}{a+(n-1)b}
\end{equation*}
and $i$ is a very large number,

\begin{center}
\begin{figure}
\centering
    \includegraphics[scale=0.9]{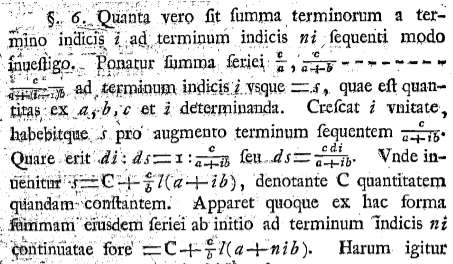}
    \caption{Discovery of the Euler-Mascheroni Constant}
     First appearance of the Euler-Mascheroni constant in print.  In this paper (\cite{E43}) Euler denotes it by $C$. He uses $i$ to denote an infinitely large number. 
    \end{figure}
\end{center}
that

\begin{equation*}
    \dfrac{ds}{di}= \dfrac{c}{a+bi} \quad \text{and hence} \quad s = C +\dfrac{c}{b}\log(a+ib).
\end{equation*}
The constant $C$ will turn out to be $\gamma$ in the case $a=b=c=1$. By this reasoning, Euler found that 

\begin{equation*}
    \lim_{n \rightarrow \infty} \left(\sum_{k=1}^{n}\dfrac{1}{k}-\log(n+1)\right)
\end{equation*}
is a finite number, although both the harmonic series and the logarithm become infinite for $n\rightarrow \infty$. Euler was certainly intrigued by this and tried to find $\gamma$. For this,  he  noted that:

\begin{equation*}
    \log \left(\dfrac{2}{1}\right)+ \log \left(\dfrac{3}{2}\right)+ \cdots + \log \left(\dfrac{n+1}{n}\right)= \log(n+1).
\end{equation*}
Next, he noted the series expansion:

\begin{equation*}
    \log \left(\dfrac{n+1}{n}\right)= \log \left(1+\dfrac{1}{n}\right)= \dfrac{1}{n}-\dfrac{1}{2n^2}+\dfrac{1}{3n^3}- \text{etc.}
\end{equation*}
which is is just the Taylor series expansion of $\log(1+x)$ around $x=0$ applied for $x=\frac{1}{n}$. Since $n \in \mathbb{N}$, the convergence condition $|x|\leq 1$ is met. Therefore, we have

\begin{equation*}
        \renewcommand{\arraystretch}{2,5}
\setlength{\arraycolsep}{0.0mm}
\begin{array}{ccccccccccccccccccccc}
     1 &~=~ & \log 2 &~+~& \dfrac{1}{2} & ~-~ & \dfrac{1}{3} &~+~ & \dfrac{1}{4} &~-~& \text{etc.} \\
      \dfrac{1}{2} &~=~ & \log \left(\dfrac{3}{2}\right) &~+~& \dfrac{1}{2 \cdot 2^2} & ~-~ & \dfrac{1}{3 \cdot 2^3} &~+~ & \dfrac{1}{4\cdot 2^4} &~-~& \text{etc.} \\
       \dfrac{1}{3} &~=~ & \log \left(\dfrac{4}{3}\right) &~+~& \dfrac{1}{2 \cdot 3^2} & ~-~ & \dfrac{1}{3 \cdot 3^3} &~+~ & \dfrac{1}{4\cdot 3^4} &~-~& \text{etc.} \\
          \dfrac{1}{4} &~=~ & \log \left(\dfrac{5}{4}\right) &~+~& \dfrac{1}{2 \cdot 4^2} & ~-~ & \dfrac{1}{3 \cdot 4^3} &~+~ & \dfrac{1}{4\cdot 4^4} &~-~& \text{etc.} \\
          \vdots & & \\
             \dfrac{1}{n} &~=~ & \log \left(\dfrac{n+1}{n}\right) &~+~& \dfrac{1}{2 \cdot n^2} & ~-~ & \dfrac{1}{3 \cdot n^3} &~+~ &\dfrac{1}{4\cdot n^4} &~-~& \text{etc.} \\
\end{array}
\end{equation*}
Therefore, adding the columns, we arrive at:

\begin{equation*}
       \renewcommand{\arraystretch}{2,5}
\setlength{\arraycolsep}{0.0mm}
\begin{array}{cccccccccccccccccc}
1+\dfrac{1}{2}+\dfrac{1}{3}+ \cdots + \dfrac{1}{n}     &~=~ & \log(n+1) &~+~& \dfrac{1}{2}\bigg(1 &~+~& \dfrac{1}{2^2}&~+~&\dfrac{1}{3^2}&~+~&\dfrac{1}{4^2}&~+~&\text{etc.}\bigg) \\  
         &  &  &~-~& \dfrac{1}{3}\bigg(1 &~+~& \dfrac{1}{2^3}&~+~&\dfrac{1}{3^3}&~+~&\dfrac{1}{4^3}&~+~&\text{etc.}\bigg)\\
           &  &  &~-~& \dfrac{1}{4}\bigg(1 &~+~& \dfrac{1}{2^4}&~+~&\dfrac{1}{3^4}&~+~&\dfrac{1}{4^4}&~+~&\text{etc.}\bigg)\\
           &  & \text{etc.}
\end{array}
\end{equation*}
Therefore, if we call $\zeta(m) = \sum_{n=1}^{\infty} \frac{1}{n^m}$, we conclude:

\begin{equation*}
    \gamma = \sum_{m=2}^{\infty}\dfrac{(-1)^m}{m}\zeta(m).
\end{equation*}
By the Leibniz criterion, this series is seen to converge and hence Euler proved that $\gamma$ indeed exists and is finite.

\subsubsection{Euler's Formulas for $\gamma$}
\label{subsubsec: Euler's formulas for gamma}

As mentioned above \cite{E583} is solely devoted to the determination of alternate expressions for $\gamma$.

\begin{center}
\begin{figure}
\centering
    \includegraphics[scale=0.9]{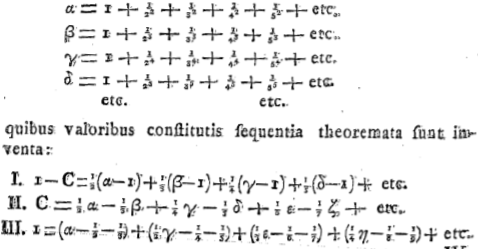}
    \caption{Formulas for the Euler-Mascheroni Constant - Part 1}
    First part of Euler's formulas for the Euler-Mascheroni constant from \cite{E583}. He wrote $\alpha$, $\beta$, $\gamma$ etc. for $\zeta(2)$, $\zeta(3)$, $\zeta(4)$ etc.  and $C$ for the Euler-Mascheroni constant $\gamma$.
    \end{figure}
\end{center}
We do not want to prove them here, but want to mention some of the results and ideas. Probably the most interesting formula is (formula I. in the figure):

\begin{equation*}
1-    \gamma = \sum_{m=2}^{\infty} \dfrac{1}{m}\left(\zeta(m)-1\right).
\end{equation*}
Euler listed all other formulas, involving series, at the end of this paper.
\begin{center}
\begin{figure}
\centering
    \includegraphics[scale=0.9]{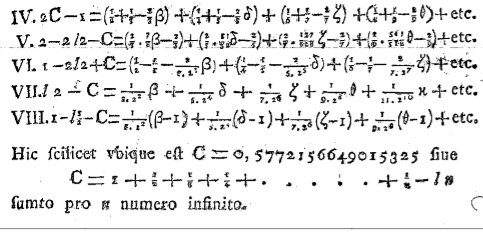}
      \caption{Formulas for the Euler-Mascheroni Constant - Part 2}
     Second part of Euler's formulas for the Euler-Mascheroni constant from \cite{E583}.  He writes $\alpha$, $\beta$, $\gamma$ etc. for $\zeta(2)$, $\zeta(3)$, $\zeta(4)$ etc.  and $C$ for the Euler-Mascheroni constant $\gamma$.
    \end{figure}
\end{center}
He also gave an integral formula\footnote{\cite{E629} is completely devoted to this one formula.}, namely

\begin{equation*}
    \gamma = \int\limits_{0}^{1} \left(\dfrac{1}{1-z}+\dfrac{1}{\log z}\right)dz.
\end{equation*}
He found it in true Eulerian fashion. First, he notes that

\begin{equation*}
\lim_{x \rightarrow 1}    \log \dfrac{1-x^n}{1-x} = n 
\end{equation*}
and

\begin{equation*}
    \log \left(1-x^n\right)=-n \int \dfrac{x^{n-1}dx}{1-x^n};
\end{equation*}
in particular, for $n=1$

\begin{equation*}
    \log \left(1-x\right)=- \int \dfrac{dx}{1-x}.
\end{equation*}
Moreover,

\begin{equation*}
    \int\limits_{0}^{1}\dfrac{(1-x^n)}{1-x}dx = 1 +\dfrac{1}{2} + \cdots + \dfrac{1}{n}
\end{equation*}
which follows integrating  the geometric sum,

\begin{equation*}
    \dfrac{1-x^n}{1-x} =1 +x +x^2 + \cdots + x^{n-1},
\end{equation*}
term by term. Therefore, we have:

\begin{equation*}
    \gamma = \lim_{n \rightarrow \infty}\left( - \int\limits_{0}^{1} \dfrac{(1-x^n)dx}{1-x} + n \int\limits_{0}^{1}\dfrac{x^{n-1}dx}{1-x^n} -\int\limits_{0}^{1}\dfrac{dx}{1-x}\right),
\end{equation*}
or

\begin{equation*}
\gamma = \lim_{n \rightarrow \infty}\left( - \int\limits_{0}^{1} \dfrac{x^ndx}{1-x} + n \int\limits_{0}^{1}\dfrac{x^{n-1}dx}{1-x^n} \right).
\end{equation*}
Next, Euler set $x^n=z$ to arrive at

\begin{equation*}
\gamma = \lim_{n \rightarrow \infty}\left( -\dfrac{1}{n} \int\limits_{0}^{1} \dfrac{z^{\frac{1}{n}}dz}{1-z^{\frac{1}{n}}}  +\int\limits_{0}^{1}\dfrac{dz}{1-z} \right).
\end{equation*}
Note the known limit

\begin{equation*}
    \log (z) = \lim_{n \rightarrow \infty} n \left(z^{\frac{1}{n}}-1\right).
\end{equation*}
Pull the limit inside the integral and use the above limit and $z^{\frac{1}{n}} \rightarrow 0$ for $n \rightarrow \infty$ and we finally arrive at

\begin{equation*}
    \gamma = \int\limits_{0}^{1} \left(\dfrac{1}{1-z}+\dfrac{1}{\log z}\right)dz.
\end{equation*}
Therefore, we arrived at the formula, which is subject of \cite{E629}. But the formulas Euler presented in that paper are not as useful as those in \cite{E583}. So we will not discuss this here.\\[2mm]
Finally, we want to add Euler's conjecture that $\gamma$ is a logarithm of an important number (confer $\S 2$ of \cite{E583}). Unfortunately, Euler did not further explain what "important" means in this context. In general, not much is known about the Euler-Mascheroni constant. Not even whether it is irrational or not. See, e.g., \cite{Ha17} for an entertaining account on the history and current state of the art concerning Euler's constant $\gamma$.

\appendix
\appendixpage
\addappheadtotoc

\renewcommand\thechapter{\Roman{chapter}}
\lhead[\thepage]{}
\chead[Translations]{Appendix}
\rhead[]{\thepage}
\chapter{On the Translations of Euler's Papers}

\section{Possible Questions about the Translations}

We answer some possible question the reader might ask about the translations provided in the following.

\subsection{Where can I find the original Papers?}

Scanned Versions of the original papers  can be found at  the Euler Archive online. Moreover, they are all reprinted in the Opera Omnia, Euler's collected works published by Birkh\"auser\footnote{Detailed information on \cite{E19}: Original title: "De progressionibus transcendentibus seu quarum termini generales algebraice dari nequeunt", first published in {Commentarii academiae scientiarum Petropolitanae} 5 (1730/31), 1738, p. 36-57, reprint in: \textit{Opera Omnia}: Series 1, Volume 14, pp. 1 - 24, Eneström-Number E19}\footnote{Detailed information on \cite{E368}: "De curva hypergeometrica hac aequatione expressa $y = 1 \cdot 2 \cdot 3 \cdots x$",\\
first published in Novi Commentarii academiae scientiarum Petropolitanae 13, 1769, pp. 3-66,\\
reprint in: Opera Omnia: Series 1, Volume 28, pp. 41 - 98, Eneström-Number E368}\footnote{Detailed information on \cite{E421}: Original title: "Evolutio formulae integralis $\int x^{f-1} dx (\log x)^\frac{m}{n}$ integratione a valore $x = 0$ ad $x = 1$ extensa", first published {Novi Commentarii academiae scientiarum Petropolitanae 16}, 1772, pp. 91-139, reprint in Opera Omnia: Series 1, Volume 17, pp. 316 - 357, Eneström-Number E421}.

\subsection{Why these three Papers?}

As mentioned in the main text, the three papers $E19$, $E368$, $E421$ contain all formulas that Euler contributed to the $\Gamma$-Function. Moreover, $E19$ was the first published paper containing the definition of the $\Gamma$-function. In the papers $E368$ and $E421$ Euler summarised everything he had discovered on the theory of the $\Gamma$-function throughout the years.

\subsection{What was changed compared to Euler's Originals?}

In the translations we stayed as close as possible to the original, concerning both the translation and the notation. Euler tends to write rather long sentences, which is rather unusual in the English language. Thus, at some instance we split such sentences into two shorter ones. 
Notationwise, we did not change much. Most of the modern notation we use today traces back to Euler so that we had to change only minor things. Euler simply wrote $lx$ for the logarithm of $x$, in the translation we always write $\log x$, i.e we made the change $\operatorname{l}\mapsto \log$ Figure 2.14 provides an example for Euler's notation of the logarithm. Furthermore, Euler wrote $\text{sin.}x$, $\text{cos.}x$, $\text{tan.}x$ instead of $\sin x$, $\cos x$, $\tan x$, since he abbreviated "\textit{sinus}", "\textit{cosinus}", "\textit{tangent}" by it. Thus, we made the changes $\text{tan.}\mapsto \tan$, $\text{cos.} \mapsto \cos$, $\text{sin.} \mapsto \sin$.  Euler's notation is seen, e.g., in figure 2.20. Additionally, Euler did not draw the square root sign over the complete expression it is supposed to contain but just in front of it. For example, Euler would write $\sqrt{}(1-xx)$ instead of the modern notation $\sqrt{1-xx}$. See figure 2.22 for an example of Euler's use of the square root sign. Since Euler's notation tends to be confusing at times, this was replaced by the modern notation. 

\subsection{Why were certain Things not changed?}

We did not add limits of integration when Euler talked about integrals\footnote{He did so, e.g., in \cite{E19} and \cite{E421}.}, since Euler always described them in the text or it is is clear form the context what he meant. Furthermore, we did not write long formulas in compact form, e.g., using the summation sign to denote an infinite sum, since written out matters are often more clear than in the compact notation.

\subsection{What about the Footnotes?}

In the translations one will find several footnotes. They were all included by the translator to provide additional information not given by Euler. Indeed, Euler did not make any footnotes in any of his works. At Euler's time, mathematicians did not always provide the reader with the precise source of theorem they applied in the respective paper. Thus, some of the footnotes state the missing information, i.e. the precise source Euler actually means.
In this context, we want to mention that footnotes about references Euler made to other works were also included by the respective editor in the Opera Omnia version of the Latin original.\\

\lhead[\thepage]{}
\chead[Appendix]{Translation of E19}
\rhead[]{\thepage}
\chapter{Translation of E19 - On transcendental Progressions or those whose general terms cannot be given algebraically}

\paragraph*{§1}

After  on the occasion of the observations on series Goldbach had discussed in the Society\footnote{Euler refers to Goldbach's work on series the latter presented at the St. Petersburg Academy of Sciences. Euler spent a whole paper on Goldbach's findings, namely \cite{E72}.}, I recently tried to find a general expression, which would give all terms of this progression

\begin{equation*}
1+1\cdot 2 +1 \cdot 2 \cdot 3 + 1 \cdot 2 \cdot 3 \cdot 4 +\text{etc.};
\end{equation*}
considering, if it is continued to infinity, that it is eventually confounded with the geometric series, I discovered the following expression

\begin{equation*}
\frac{1 \cdot 2^n}{1+n}\cdot \frac{2^{1-n} \cdot 3^n}{2+n}\cdot \frac{3^{1-n} \cdot 4^n}{3+n} \cdot \frac{4^{1-n} \cdot 5^n}{4+n} \cdot \text{etc.},
\end{equation*}
which expresses the term of order $n$ in the before-mentioned progression. It certainly does not terminate in either case, neither if $n$ is an integer number nor if it is a fraction;  but it only yields approximations to find the respective term, if the cases $n=0$ and $n=1$ are excluded, in which cases the formula is immediately seen to be $1$. Set $n=2$; one will then have

\[
\dfrac{2\cdot 2}{1\cdot 3}\cdot\dfrac{3\cdot 3}{2\cdot 4}\cdot\dfrac{4 \cdot 4}{3\cdot 5}\cdot\dfrac{5\cdot 5}{4\cdot 6} \cdot \text{etc.} = \text{ the second term $2$}
\]
If it is $n=3$, one will have

\[
\dfrac{2\cdot 2\cdot 2}{1\cdot 1\cdot 4}\cdot\dfrac{3\cdot 3\cdot 3}{2\cdot 2\cdot 5}\cdot\dfrac{4\cdot 4\cdot 4}{3\cdot 3\cdot 6}\cdot\dfrac{5\cdot 5\cdot 5}{4\cdot 4\cdot 7}\cdot\text{etc} = \text{the third term $6$}
\]
\paragraph*{§2}
But although this expression seems to have no use for finding the terms, it can nevertheless be applied  to interpolate the series or the terms, whose indices are fractional numbers. But I decided to explain nothing about this here, since  more appropriate  ways  to achieve the same will be exhibited below. I only want to mention  everything that is necessary to get to the things to follow about the general term. I tried to find the term corresponding to the index $n=\frac{1}{2}$ or the term which falls in the middle between the first, $1$, and the preceding one, which is also $1$. But having put $n=\frac{1}{2}$ I obtain the product

\[\sqrt{\dfrac{2\cdot 4}{3\cdot 3}\cdot\dfrac{4\cdot 6}{5\cdot 5}\cdot\dfrac{6\cdot 8}{7\cdot 7}\cdot\dfrac{8\cdot 10}{9\cdot 9}\cdot\text{etc.}},\]
which expresses the term in question. But this series, to me, seemed to be similar to the one I remembered to have seen in Wallis's works on the area of the circle\footnote{Euler refers to Wallis' book "{}Arithmetica infintorum sive nova method inquirendi in curvilineorum quadraturam, aliaque difficiliora Matheseos problemata{}", published 1655}. For, Wallis found that the area of the circle has a ratio to the diameter of the circle as of

\[
2\cdot 4\cdot 4\cdot 6\cdot 6\cdot 8\cdot 8\cdot 10\cdot\text{etc.} \quad \text{to} \quad  3\cdot 3\cdot 5\cdot 5\cdot 7\cdot 7\cdot 9\cdot 9\cdot\text{etc.}
\]
Therefore, if the diameter was $=1$, the area of the circle will be

\[
= \dfrac{2\cdot 4}{3\cdot 3}\cdot\dfrac{4\cdot 6}{5\cdot 5}\cdot\dfrac{6\cdot 8}{7\cdot 7}\cdot\text{etc.}
\]
Therefore, from the agreement of this series with mine it is possible to conclude that the term corresponding to the index $\frac{1}{2}$ is equal to the square root of the area of the circle, whose diameter is $=1$.
\paragraph*{§3}

I had believed before that the general term of the series $1$, $2$, $6$, $24$ etc., if not algebraically, is  at least given exponentially\footnote{Euler means as an exponential, i.e. as $a^b$ for some real numbers $a$ and $b$.}. But after I had understood that certain intermediate terms depend on the quadrature of the circle, I realized that neither algebraic nor exponential quantities are sufficient to express it. For, the general term of the progression must be of such a nature that  one time it contains algebraic quantities but another time quantities depending on the quadrature of the circle and even on other quadratures; and no algebraic nor exponential formula fulfills this condition.
\paragraph*{§4}

But after I had remembered that among differential  formulas  quantities  exist, which in certain cases admit an integration and then yield algebraic quantities, but in other cases  do not admit an integration and then exhibit quantities depending on the quadratures of curves, it came to mind that maybe formulas of this kind are appropriate to express the general terms of the mentioned progression and other similar ones. And in the following, I will call progressions requiring such general terms, which cannot be given algebraically, \textit{transcendental}; as Geometers used to call everything exceeding the power of common Algebra, transcendental.
\paragraph*{§5}

Therefore, I thought about what properties differential formulas should have in order to express general terms of progressions in the best possible way. But the general term is a formula, into which both constant as well certain other non constant quantities such as $n$ enter, which number $n$ gives the order of the terms or the index such that, if the third term is in question, instead of $n$ one has to put $3$. But  also a certain variable quantity must be contained in the differential formula. For this quantity it is not advisable to use $n$, since this variable is not to be integrated over, but rather, after the formula has been integrated or when it is assumed that it has been integrated,   this variable $n$ is just used to form the progression. Therefore,  a certain variable quantity $x$ must be contained in the differential formula, which after the integration must be set equal to another number concerning the progression\footnote{By this Euler means that the upper limit of the integral has to be chosen appropriately}; and hence the term, whose index is $n$, results.

\paragraph*{§6}

In order that this is understood more clearly, I say that $\int pdx$ is the general term of the progression which is to be found in the following from it; but let $p$ denote an arbitrary function of $x$ and of constants, one of which constants must be $n$. Now assume $pdx$ to be integrated and augmented by such a constant such that having put $x=0$ the whole integral vanishes; then put $x$ equal to a certain known quantity. Having done this  only quantities related to the progression will remain in the found integral, and that integral will express the term corresponding to the index $=n$.  In other words, the integral determined in this way will be the general term. If this integration is actually possible, the differential formula is not necessary and the progression formed from this will have a general algebraic term; but things change, if the integration only succeeds for certain numbers  $n$.

\paragraph*{§7}

Therefore, I considered and tried many differential formulas of this kind only admitting an integration,  if  a positive integer is substituted for $n$ such that the principal terms become algebraic, and hence formed progressions. Therefore, their general terms were immediately clear, and it will be possible to define the quadrature describing the intermediate terms.  I will certainly not go through many formulas of this kind here, but will only treat a general one  which extends very far and which can be accommodated to all progressions, whose the arbitrary terms are  products consisting of a certain number of terms depending on the index; the factors of these products are fractions, whose numerators and denominators proceed in an arbitrary arithmetic progression, as, e.g.,

\[
\dfrac{2}{3} + \dfrac{2\cdot 4}{3\cdot 5} + \dfrac{2\cdot 4\cdot 6}{3\cdot 5\cdot 7} + \dfrac{2\cdot 4\cdot 6\cdot 8}{3\cdot 5\cdot 7\cdot 9} + \text{ etc.}
\]
\paragraph*{§8}
Let this formula be given

\[
\int x^{e}dx(1 - x)^{n}
\]
describing the general term; this formula, integrated in such a way that it becomes $=0$, if  $x=0$, and then having set  $x=1$,  expresses the term of order $n$ of the progression resulting from it. Therefore, let us see, which progression it actually describes. We have
\[
(1 - x)^{n} = 1 - \dfrac{n}{1}x + \dfrac{n(n-1)}{1\cdot 2}x^{2} - \dfrac{n(n-1)(n-2)}{1\cdot 2\cdot 3}x^{3} + \text{ etc.}
\]
and hence
\[
x^edx(1 - x)^{n} = x^{e}dx - \dfrac{n}{1}x^{e + 1}dx + \dfrac{n(n-1)}{1\cdot 2}x^{e + 2}dx - \dfrac{n(n-1)(n-2)}{1\cdot 2\cdot 3}x^{e + 3}dx + \text{ etc.}\]
Hence
\[
\int x^{e}dx(1 - x)^{n} = \dfrac{x^{e + 1}}{e + 1} - \dfrac{nx^{e + 2}}{1\cdot (e + 2)} + \dfrac{n(n-1)x^{e + 3}}{1\cdot 2\cdot (e + 3)} - \dfrac{n(n-1)(n-2)x^{e + 4}}{1\cdot 2\cdot 3\cdot (e + 4)} + \text{ etc.}
\]
Set $x=1$, since the addition of the constant is not necessary; and one will have
\[
\dfrac{1}{e + 1} - \dfrac{n}{1\cdot (e + 2)} + \dfrac{n(n-1)}{1\cdot 2\cdot (e + 3)} - \dfrac{n(n-1)(n-2)}{1\cdot 2\cdot 3\cdot (e + 4)} + \text{ etc.}\]
as the general term in question of the series. The series will be of such a nature that, if  $n=0$, the corresponding term results to be $= \frac{1}{e + 1}$; if $n = 1$, the term $\frac{1}{(e+1)(e+2)}$; if $n=2$, the corresponding  term is $=\frac{1 \cdot 2}{(e+1)(e+2)(e+3)}$, if  $n=3$, then the corresponding term is  $=\frac{1 \cdot 2 \cdot 3 \cdot 4}{(e+1)(e+2)(e+3)(e+4)}$; the rule describing how  these terms proceed, is manifest.
\paragraph*{§9}
Therefore, I obtained this progression
\[
\dfrac{1}{(e + 1)(e + 2)} + \dfrac{1\cdot 2}{(e + 1)(e + 2)(e + 3)} + \dfrac{1\cdot 2\cdot 3}{(e + 1)(e + 2)(e + 3)(e + 4)} + \text{ etc.},
\]
whose general term is
\[
\int x^{e}dx(1 - x)^{n}.
\]
On the other hand, the term of order $n$ will be  expressed by this form

\[
\dfrac{1\cdot 2\cdot 3\cdot 4\cdots n}{(e + 1)(e + 2)\cdots (e + n +1)}.
\]
This form certainly suffices to find the terms of  integer indices, but if the indices were no integers, this form cannot be used to find the corresponding terms. This following series will be helpful to find  them approximately
\[
\dfrac{1}{e + 1} - \dfrac{n}{1\cdot (e + 2)} + \dfrac{n(n-1)}{1\cdot 2\cdot (e + 3)} - \dfrac{n(n-1)(n-2)}{1\cdot 2\cdot 3\cdot (e + 4)} + \text{ etc.}\]
If $\int x^e dx(1-x)^n$ is multiplied by $e+n+1$, one will have a progression, whose term of order $n$ has this form

\[
\dfrac{1\cdot 2\cdot 3\cdots n}{(e + 1)(e + 2)\cdots (e + n)},
\]
whose general term  will therefore be
\[
(e + n + 1)\int x^{e}dx(1 - x)^{n}.
\]
Here it is to be noted that the progression always becomes algebraic, whenever  a positive number is assumed  for $e$. For the sake of an example, set $e=2$; then the $n$-th term of the progression will be
\[
\dfrac{1\cdot 2\cdot 3\cdots n}{3\cdot 4\cdot 5\cdots (n + 2)} \quad\text{or } \quad \dfrac{1\cdot 2}{(n + 1)(n + 2)}.
\]
The general term itself also indicates this; since this term will be
\[
(n + 3)\int xxdx(1 - x)^{n}.
\]
For, its integral is
\[
\left(C - \dfrac{(1 - x)^{n + 1}}{n + 1} + \dfrac{2(1 - x)^{n + 2}}{n + 2} - \dfrac{(1 - x)^{n + 3}}{n + 3}\right)(n + 3);
\]
in order for this to become $=0$, if one has $x=0$, it will be

\[
C = \dfrac{1}{n + 1} - \dfrac{2}{n - 2} + \dfrac{1}{n + 3}.
\]
Set $x=1$; the general term will be

\[
\dfrac{n + 3}{n + 1} - \dfrac{2(n + 3)}{n + 2} + 1= \dfrac{2}{(n + 1)(n + 2)}.
\]
\paragraph*{§10}

Therefore, in order to obtain transcendental progressions, set $e$ equal to the fraction $\frac{f}{g}$. The term of order $n$ of the progression will be
\[
= \dfrac{1\cdot 2\cdot 3\cdots n}{(f + g)(f + 2g)(f + 3g)\cdots (f + ng)}g^{n}
\]
or 
\[
\dfrac{g\cdot 2g\cdot 3g\cdots ng}{(f + g)(f + 2g)(f + 3g)\cdots (f + ng)}.
\]
On the other hand, the general term  will be

\[
= \dfrac{f + (n + 1)g}{g}\int x^{\frac{f}{g}}dx(1 - x)^{n}.
\]
If this is divided by $g^n$, we will obtain the general term for the progression
\[
\dfrac{1}{f + g} + \dfrac{1\cdot 2}{(f + g)(f + 2g)} + \dfrac{1\cdot 2\cdot 3}{(f + g)(f + 2g)(f + 3g)} + \text{ etc.},
\]
whose term of $n$-th order is

\[
= \dfrac{1\cdot 2\cdot 3\cdots n}{(f + g)(f + 2g)\cdots (f + ng)}.
\]
Therefore, the general term of the progression will be

\[
= \dfrac{f + (n + 1)g}{g^{n + 1}}\int x^{\frac{f}{g}}dx(1 - x)^{n}.
\]
If the fraction $\frac{f}{g}$ is not equal to an integer number, or if  $f$ is no multiple of $g$, the progression will be transcendental and the intermediate terms will depend on quadratures.

\paragraph*{§11}

In order to display the general term  more clearly, I want to mention a certain example. In the first progression of the preceding paragraph, let $f=1$ and $g=2$; the term of order $n$ will be

\[
= \dfrac{2\cdot 4\cdot 6\cdot 8\cdots 2n}{3\cdot 5\cdot 7\cdot 9\cdots (2n + 1)},
\]
the progression itself on the other hand will be this one

\[
\dfrac{2}{3} + \dfrac{2\cdot 4}{3\cdot 5} + \dfrac{2\cdot 4\cdot 6}{3\cdot 5\cdot 7} + \text{ etc.},
\]
whose general term will hence be
\[
\dfrac{2n + 3}{2}\int dx(1 - x)^{n}\sqrt{x}.
\]
Let the term corresponding to the index $\frac{1}{2}$ be in question; therefore, it will be $n=\frac{1}{2}$ and one will find the general term in question to be

\[
= 2\int dx\sqrt{x - xx}.
\]
Since this denotes the element of the circular area, it is perspicuous that the term in question is the area of the circle, whose diameter is $=1$.\\
Further, let this series be put forth

\[
1 + \dfrac{r}{1} + \dfrac{r(r - 1)}{1\cdot 2} + \dfrac{r(r - 1)(r - 2)}{1\cdot 2\cdot 3} + \text{ etc.},
\]
which is the  coefficient of the binomial raised to the power $r$. Therefore, the term of order $n$ is

\[
\dfrac{r(r- 1)(r - 2)\cdots (r - n + 2)}{1\cdot 2\cdot 3\cdots (n - 1)}.
\]
In the preceding paragraph on the other hand,  we found  this expression

\[
\dfrac{1\cdot 2\cdot 3\cdots n}{(f + g)(f + 2g)\cdots (f + ng)}.
\]
In order to compare these two expressions, invert the second one such that one has
\[
\dfrac{(f + g)(f + 2g)\cdots (f + ng)}{1\cdot 2\cdots n};
\]
now multiply it  by $\frac{n}{f+ng}$ and it will be

\[
= \dfrac{(f + g)(f + 2g)\cdots (f + (n - 1)g)}{1\cdot 2\cdots (n - 1)};
\]
therefore, it must be $f+g=r$ and $f+2g=r-1$, whence it will be $g=-1$ and $f=r+1$. Treat the following general term in the same way
\[
\dfrac{f + (n + 1)g}{g^{n + 1}}\int x^{\frac{f}{g}}dx(1 - x)^{n}.
\]
For the propounded progression 
\[
1 + \dfrac{r}{1} + \dfrac{r(r - 1)}{1\cdot 2} + \text{ etc.}
\]
this general term will result

\[
\dfrac{n(-1)^{n + 1}}{(r - n)(r - n + 1)\int x^{-r - 1}dx(1 - x)^{n}}.
\]
Let $r=2$; the general term of this progression
\[
1, \quad 2, \quad 1,\quad 0,\quad 0, \quad 0 \text{etc.}
\]
will be
\[
\dfrac{n(-1)^{n + 1}}{(2 - n)(3 - n)\int x^{-3}dx(1 - x)^{n}}.
\]
Here it must be noted that this  and other cases, in which $e+1$ is a negative number, cannot be deduced from the general expression; for, in these cases,  the integral does not become $=0$, if  $x=0$. But in order to treat even these cases, it is convenient to integrate

\[
\int x^{e}dx(1 - x)^{n}
\]
in a peculiar way; for, after the integration an infinite constant is to be added. But whenever $e+1$ is a positive number, as I assumed in  § 8, the addition of the constant is not necessary. But having considered the progression, whose term of order $n$ was the following

\[
\dfrac{r(r - 1)(r - 2)\cdots (r - n + 2)}{1\cdot 2\cdot 3\cdots (n - 1)},
\]
that form of the term corresponding to the index $n$ can be transformed into this one

\[
\dfrac{r(r - 1)\cdots 1}{(1\cdot 2\cdot 3\cdots (n-1))(1\cdot 2\cdots (r - n + 1))}.
\]
But by means of § 14, 
\[
r(r - 1)\cdots 1 = \int dx(-\log(x))^{r}
\]
and

\[
1\cdot 2\cdot 3\cdots (n - 1) = \int dx(-\log(x))^{n - 1}
\]
and

\[
1\cdot 2\cdots (r - n + 1) = \int dx(-\log(x))^{r - n + 1}.
\]
Therefore, the progression considered there
\[
1 + \dfrac{r}{1} + \dfrac{r(r - 1)}{1\cdot 2} + \dfrac{r(r - 1)(r - 2)}{1\cdot 2\cdot 3} + \text{ etc.}
\]
has this general term

\[
\dfrac{\int dx(-\log(x))^{r}}{\int dx(-\log(x))^{n - 1}\int dx(-\log(x))^{r - n + 1}}-
\]
If it was $r=2$, the general term will be
\[
\dfrac{2}{\int dx(-\log(x))^{n - 1}\int dx(-\log(x))^{3 - n}},
\]
to which this progression corresponds
\[
1,\quad 2,\quad 1,\quad 0,\quad 0,\quad 0 \quad \text{etc.};
\]
and if the term of the index $\frac{3}{2}$ is in question, it will be

\[
\dfrac{2}{\int dx(-\log(x))^{\frac{1}{2}}\int dx(-\log(x))^{\frac{3}{2}}}.
\]
Therefore, having called the area of the circle, whose diameter is $=1$, $A$, since 
\[
\int dx(-\log(x))^{\frac{1}{2}} = \sqrt{A}\quad\text{and} \quad \int dx(-\log(x))^{\frac{3}{2}} = \dfrac{3}{2}\sqrt{A},
\]
the term falling in the middle between the first two terms of the progression $1$, $2$, $1$, $0$, $0$, $0$ etc. will be  $\frac{4}{3A}$, this means approximately $\frac{5}{3}$.

\paragraph*{§12}
Now I proceed to the progression I talked about in the beginning,

\[
1 + 1\cdot 2 + 1\cdot 2\cdot 3 + \text{etc.}
\]
and in which the term corresponding to the index $n$ is $1 \cdot 2 \cdot 3 \cdot 4 \cdots n$. This progression is contained in our general one, but the general term must be derived from it in a peculiar way. Until now  I considered the general term, if the term of order $n$ is
 
\[
\dfrac{1\cdot 2\cdot 3\cdots n}{(f + g)(f + 2g)\cdots (f + ng)},
\]
which, if one sets $f=1$ and $g=0$, goes over into $1 \cdot 2 \cdot 3 \cdots n$, which is in question; therefore, in the general term

\[
\dfrac{f + (n + 1)g}{g^{n + 1}}\int x^{\frac{f}{g}}dx(1 - x)^{n}
\]
substitute these values for $f$ and $g$; the general term in question will be
\[
\int\dfrac{x^{\frac{1}{0}}dx(1 - x)^{n}}{0^{n + 1}}.
\]
And I will investigate the value of this general term as follows.

\paragraph*{§13}

Considering the condition that general terms of this kind must fulfill in order to be useful, it is understood that instead of $x$ other functions can be assumed, as long as they were of such a kind that they are $=0$, if  $x=0$, and $=1$, if $x=1$. For, if  a function of this kind is substituted for $x$, the general term will  fulfill the same condition as before. Therefore, write $x^{\frac{p}{f+g}}$ instead of $x$ and as a logical consequence $\frac{g}{f+g}x^{\frac{-f}{g+f}}dx$ instead of $dx$; having done this one will have

\[
\dfrac{f + (n + 1)g}{g^{n + 1}}\int\dfrac{g}{f + g}dx(1 - x^{\frac{g}{f + g}})^{n}.
\]
Now set $f=1$ and $g=0$ here; one will have

\begin{equation*}
\int \frac{dx(1-x^0)^n}{0^n}.
\end{equation*}
But because  $x^0=1$, here we have a case, in which the numerator $(1-x^0)^n$ and the denominator $0^n$ vanish. Therefore, applying the known rule\footnote{By this Euler means what we refer to as L'Hospital's rule today. Euler derived the rule in book \cite{E212} written in 1755. He most likely knew the rule from Jacob Bernoulli, his mentor in his teenage years. Jacob Bernoulli on the other hand also was a teacher of L'Hospital.} let us find the value of the fraction $\frac{1-x^0}{0}$. This will by achieved by finding the value of the fraction $\frac{1-x^z}{z}$ for vanishing $z$; therefore, differentiate the numerator and the denominator with respect to that variable $z$; one will find $\frac{-x^z dz \log (x)}{dz}$ or $- x^z \log (x)$; if now one puts $z=0$, $- \log (x)$ will result. Therefore, 

\begin{equation*}
\frac{1-x^0}{0}=- \log (x).
\end{equation*}

\paragraph*{§14}
Therefore, because 
\[\dfrac{1 - x^{0}}{0} = -\log(x),
\]
it will be
\[
\dfrac{(1 - x^{0})^{n}}{0^{n}} = (-\log(x))^{n}
\]
and therefore the general term in question $\int \frac{dx(1-x^0)^n}{0^n}$ is transformed into $\int dx(- \log (x))^n$. Its value can be expressed by means of quadratures. Therefore, the general term of this progression
\[
1, \quad 2,\quad 6, \quad 24, \quad 120, \quad 720 \quad \text{etc.}
\]
is 

\[
\int dx(-\log(x))^{n},
\]
and it is to be used in the same way as it was prescribed above. That this really is the general term of the propounded progression is also seen from this, that it indeed yields the correct terms for the cases, in which the indices are positive integers. For the sake of an example, let $n=3$; it will be

\[
\int dx(-\log(x))^{3} = \int -dx(\log(x))^{3} = -x(\log(x))^{3} + 3x(\log(x))^{2} - 6x\log(x) + 6x;
\]
the addition of a constant is not necessary, since for $x=0$ everything vanishes; therefore, put $x=1$; since $\log(1)=0$, all terms containing  logarithms will vanish and $6$ will remain, which is the third term.
\paragraph*{§15}

It is clear that this method to find the terms of this series is too laborious,  even for those terms corresponding to integer indices; they are certainly obtained more easily by just continuing the progression. But this expression is nevertheless more than appropriate to find the terms corresponding to rational indices; for, it was not even possible to define these terms using the most laborious methods. If one sets $x=\frac{1}{2}$,  the corresponding term will be $= \int dx \sqrt{- \log (x)}$, whose value is given by quadratures. But at the beginning  [§ 11] I showed that this term is equal to the square root of the area the circle, whose diameter is $1$. Hence it is certainly not possible to conclude the same because of the missing Analysis; but below a method will be explained to reduce the same intermediate terms to quadratures of algebraic curves. And comparing this method to the one already explained it will perhaps be possible to derive many results which can be used to develop the whole field of Analysis even further.
\paragraph*{§16}

The general term of the progression, whose term of order $n$ is given by
 
\[
\dfrac{1\cdot 2\cdot 3\cdots n}{(f + g)(f + 2g)(f + 3g)\cdots (f + ng)},
\]
is, by means of § 10,
\[
\dfrac{f + (n + 1)g}{g^{n + 1}}\int x^{\frac{f}{g}}dx(1 - x)^{n}.
\]
But if the term of order $n$ was

\[
1\cdot 2\cdot 3\cdots n,
\]
then the general term is

\[
\int dx(-\log(x))^{n}.
\]
If this formula is substituted for $1 \cdot 2 \cdot 3 \cdots n$, one will have

\[
\dfrac{\int dx(-\log(x))^{n}}{(f + g)(f + 2g)(f + 3g)\cdots (f + ng)} = \dfrac{f + (n + 1)g}{g^{n + 1}}\int x^{\frac{f}{g}}dx(1 - x)^{n}.
\]
Hence  

\[
(f + g)(f + 2g)(f + 3g)\cdots (f + ng) = \dfrac{g^{n + 1}\int dx(-\log(x))^{n}}{(f + (n + 1)g)\int x^{\frac{f}{g}}dx(1 - x)^{n}}.
\]
Therefore, this expression is the general term of this general progression

\[
(f + g),\quad (f + g)(f + 2g), \quad (f + g)(f + 2g)(f + 3g) \quad \text{etc.}
\]
Therefore, by means of the general term, all terms corresponding to any arbitrary index of all progressions of this kind are defined. What will follow below on the reduction of $\int dx (-\log (x))^n$ to more familiar quadratures or quadratures of algebraic curves, will also be useful here.

\paragraph*{§17}
Let $f+g=1$ and $f+2g=3$; it will be $g=2$ and $f=-1$. Hence this particular  progression will result

\[
1, \quad 1\cdot 3, \quad 1\cdot 3\cdot 5, \quad 1\cdot 3\cdot 5\cdot 7 \quad \text{etc.}
\]
Therefore, its general term is

\[
\dfrac{2^{n + 1}\int dx(-\log(x))^{n}}{(2n + 1)\int x^{-\frac{1}{2}}dx(1-x)^{n}}.
\]
Although here one exponent of $x$ is negative, nevertheless the inconvenience addressed above does not occur here, since it is greater than $-1$. Put $n=\frac{1}{2}$ to find the term corresponding to the index $\frac{1}{2}$; that term  will be

\[
= \dfrac{2^{\frac{3}{2}}\int dx\sqrt{-\log(x)}}{2\int x^{-\frac{1}{2}}dx \sqrt{1 - x}} = \dfrac{\sqrt{2} \cdot\int dx \sqrt{-\log(x)}}{\int\frac{dx - xdx}{\sqrt{x - xx}}}.
\]
But from § 15 it is known that $\int dx \sqrt{- \log (x)}$ gives the square root of the circle, whose diameter is $=1$; let the circumference of that circle be $p$; the area will be $=\frac{1}{4}p$ and hence $\int dx \sqrt{- \log (x)}$ gives $\frac{1}{2}\sqrt{p}$. Further, 

\[
\int \dfrac{dx - xdx}{2 \sqrt{x - xx}} = \int\dfrac{dx}{2\sqrt{x - xx}} + \sqrt{x - xx},
\]
but $\int\frac{dx}{2\sqrt{x - xx}}$ gives the arc, whose sinus versus is $x$, of the circle. Therefore, having put $x=1$ $\frac{1}{2}p$ will result. Therefore, the term in question will be

\[
= \sqrt{\dfrac{2}{p}}.
\]
\paragraph*{§18}
 Since the general term of the progression, whose term of order $n$ is given by
 
\[
(f + g)(f + 2g)\cdots (f + ng),
\]
by means of § 16 reads

\[
\dfrac{g^{n + 1}\int dx(-\log(x))^{n}}{(f + (n + 1)g)\int x^{\frac{f}{g}} dx(1 - x)^{n}},
\]
in like manner, if the term of order $n$ was

\[
(h + k)(h + 2k)\cdots (h + nk),
\]
the general term  will be

\[
\dfrac{k^{n + 1}\int dx(-\log(x))^{n}}{(h + (n + 1)k)\int x^{\frac{h}{k}}dx(1 - x)^{n}}.
\]
Divide the first progression by the second, namely the first term by the first, the second by the second and so forth; in this way one will arrive at a new progression, whose term of order $n$ will be

\[
\dfrac{(f + g)(f + 2g)\cdots (f + ng)}{(h + k)(h + 2k)\cdots (h + nk)}.
\]
And the general term of this progression composed of these two will be

\[
\dfrac{g^{n + 1}(h + (n + 1)k)\int x^{\frac{h}{k}}dx(1 - x)^{n}}{k^{n + 1}(f + (n + 1)g)\int x^{\frac{f}{g}}dx(1 - x)^{n}}.
\]
This term does not contain the logarithmic integral $\int dx(-\log(x))^{n}$.

\paragraph*{§19}

In all general terms of this kind, one has to note that not even for $f$, $g$, $h$, $k$ one has to put constant numbers, but they can be assumed to depend on $n$ arbitrarily. For, in the integration these letters are treated in the same way as $n$, namely as constants. Therefore, let the term of order $n$ be this one

\[
(f +g)(f + 2g)\cdots (g + ng);
\]
set $g = 1$, but $f = \frac{nn - n}{2}$. Since the progression itself is

\[
f + g, \quad  (f + g)(f + 2g), \quad (f + g)(f + 2g)(f + 3g) \quad
\text{etc.}
\]
 sett $1$ instead of $g$ everywhere; the progression will be

\[
f + 1, \quad  (f + 1)(f + 2), \quad (f + 1)(f + 2)(f + 3) \quad \text{etc.}\]
But  one has to write $0$ in the first term instead of $f$, $1$ in the second, $3$ in the third, $6$ in the fourth and so forth; then this progression will result

\[
1, \quad 2\cdot 3, \quad 4\cdot 5\cdot 6, \quad 7\cdot 8\cdot 9\cdot 10 \quad \text{etc.},
\]
whose general term is
\[\dfrac{2\int dx(-\log(x))^{n}}{(nn + n + 2)\int x^{\frac{nn - n}{2}}dx(1 - x)^{n}} 
=\dfrac{2\int dx(-\log(x))^{n}}{(nn + n +2)\int dx(x^{\frac{n - 1}{2}} - x^{\frac{n + 1}{2}})^{n}}.\]
\paragraph*{§20}

Now I proceed to the progressions, which led me to the invention of the artifice to define the intermediate of this progression more easily

\[
1, \quad 2,\quad 6, \quad 24, \quad 120 \quad \text{etc.}
\]
For, this artifice extends  further than only to this progression, since its general term

\[
\int dx(-\log(x))^{n}
\]
also enters the general terms of infinitely many other progressions.
I assume this general term
\[
\dfrac{f + (n + 1)g}{g^{n + 1}}\int x^{\frac{f}{g}}dx(1 - x)^{n},
\]
to which  this term of order $n$  corresponds

\[
\dfrac{1\cdot 2\cdot 3\cdots n}{(f + g)(f + 2g)(f + 3g)\cdots (f + ng)}.
\]
Here I set $f=n$, $g=1$; hence the general term will  be

\[
(2n + 1)\int x^{n}dx(1 - x)^{n} \quad \text{or} \quad (2n + 1)\int dx(x - xx)^{n}
\]
and its form of order $n$ will be

\[
\dfrac{1\cdot 2\cdot 3\cdots n}{(f + g)(f + 2g)(f + 3g)\cdots 2n}.
\]
The progression itself on the other hand is this one
\[
\dfrac{1}{2}, \quad \dfrac{1\cdot 2}{3\cdot 4}, \quad \dfrac{1\cdot 2\cdot 3}{4\cdot 5\cdot 6}\quad \text{etc.}
\]
or this one
\[
\dfrac{1\cdot 1}{1\cdot 1}, \quad \dfrac{1\cdot 2\cdot 1\cdot 2}{1\cdot 2\cdot 3\cdot 4}, \quad \dfrac{1\cdot 2\cdot 3\cdot 1\cdot 2\cdot 3}{1\cdot 2\cdot 3\cdot 4\cdot 5\cdot 6}.
\]
In this expression, the numerators are the squares of the progression $1$, $2$, $6$, $24$ etc.; and now it is easy to find the terms corresponding to rational indices whose denominator is $2$. For,  in the progression $1$, $2$, $6$, $24$ etc. let the term corresponding to the index  $\frac{1}{2}$ be $A$; the term of order $\frac{1}{2}$ of the propounded progression will then be $=\frac{AA}{1}$.

\paragraph*{§21}
In the general term
\[
(2n + 1)\int x^{n}dx(1 - x)^{n}
\]
set $n = \frac{1}{2}$; the term corresponding to that exponent will be
 
\[
2\int dx\sqrt{x - xx} = \dfrac{AA}{1},
\]
whence
\[
A = \sqrt{1\cdot 2\int dx\sqrt{x - xx}}
\]
$=$ to the term of the progression $1$, $2$, $6$, $24$ etc. corresponding to the index $\frac{1}{2}$; hence this term, as   is clear from the integral, is the square root of the area of the circle, whose diameter is $1$. Now call the term of order $\frac{3}{2}$ of this progression $A$; the corresponding term in the assumed progression will be

\[
\dfrac{A\cdot A}{1\cdot 2\cdot 3} = 4\int dx(x - xx)^{\frac{3}{2}},
\]
therefore,
\[
A = \sqrt{1\cdot 2\cdot 3\cdot 4\int dx(x - xx)^{\frac{3}{2}}}.
\]
In like  manner, the term of order $\frac{5}{2}$ will be found to be

\[
= \sqrt{1\cdot 2\cdot 3\cdot 4\cdot 5\cdot 6\int dx(x - xx)^{\frac{5}{2}}}.
\]
From these result I conclude in general that the term of order $\frac{p}{2}$ will be

\[
= \sqrt{1\cdot 2\cdot 3\cdot 4\cdots (p + 1)\int dx(x - xx)^{\frac{p}{2}}}.
\]
Therefore, in this way one finds all terms of the progression $1$, $2$, $6$, $24$ etc., whose indices are fractions, while the denominator is $2$.

\paragraph*{§22}

Further, in the general term

\[
\dfrac{f + (n +1)g}{g^{n + 1}}\int x^{\frac{f}{g}}dx(1 - x)^{n}
\]
I set $f = 2n$, while $g$ remains $=1$; then this expression results

\[
(3n + 1)\int dx(xx - x^{3})^{n}
\]
as general term of the progression

\[
\dfrac{1}{3}, \quad \dfrac{1\cdot 2}{5\cdot 6}, \quad \dfrac{1\cdot 2\cdot 3}{7\cdot 8\cdot 9} \quad \text{etc.}
\]
Multiply that one by the preceding $(2n+1) \int dx (x-xx)^n$; then this expression will result

\[
(2n + 1)(3n + 1)\int dx(x - xx)^{n}\int dx(xx - x^{3})^{n}\int dx(x^3 - x^{4})^{n}.
\]
This will give this progression
\[
\dfrac{1\cdot 1\cdot 1}{1\cdot 2\cdot 3}, \quad \dfrac{1\cdot 2\cdot 1\cdot 2\cdot 1\cdot 2}{1\cdot 2\cdot 3\cdot 4\cdot 5\cdot 6}\quad 
\text{etc.},
\]
where the numerators are the cubes of the corresponding terms of the progression $1$, $2$, $6$, $24$ etc. Let the term of order $\frac{1}{3}$ of this progression be $A$; the corresponding term of that progression will be

\[
\dfrac{A^{3}}{1} = 2\left(\dfrac{2}{3} + 1\right)\int dx(x - xx)^{\frac{1}{3}}\int dx(xx - x^{3})^{\frac{1}{3}},
\]
therefore, the term of order $\frac{1}{3}$ is

\[
\sqrt[3]{1\cdot 2\cdot\dfrac{5}{3}\int dx(x - xx)^{\frac{1}{3}}\int dx(xx - x^{3})^{\frac{1}{3}}};
\]
similarly, the term of order $\frac{2}{3}$ is

\[
\sqrt[3]{1\cdot 2\cdot 3\cdot\dfrac{7}{3}\int dx(x - xx)^{\frac{2}{3}}\int dx(xx - x^{3})^{\frac{2}{3}}}.
\]
And the term of order $\frac{4}{3}$ is

\[
\sqrt[3]{1\cdot 2\cdot 3\cdot 4\cdot 5\cdot\dfrac{11}{3}\int dx(x - xx)^{\frac{4}{3}}\int dx(xx - x^{3})^{\frac{4}{3}}}
\]
and in general the term of order $\frac{p}{3}$ is

\[
\sqrt[3]{1\cdot 2\cdots p\cdot \dfrac{2p + 3}{3}\cdot (p + 1)\int dx(x - xx)^{\frac{p}{3}}\int dx(xx - x^{3})^{\frac{p}{3}}}.
\]
\paragraph*{§23}

If we want to proceed further by putting $f=3n$, it will be necessary to multiply the general term

\[
(4n + 1)\int dx(x^{3} - x^{4})^{n}
\]
by the preceding ones, whence one has

\[
(2n +1)(3n + 1)(4n + 1)\int dx(x - xx)^{n}\int dx(x^{2} - x^{3})^{n},
\]
which is the general term of this series

\[
\dfrac{1\cdot 1\cdot 1\cdot 1}{1\cdot 2\cdot 3\cdot 4}, \quad \dfrac{1\cdot 2\cdot 1\cdot 2\cdot 1\cdot 2\cdot 1\cdot 2}{1\cdot 2\cdot 3\cdot 4\cdot 5\cdot 6\cdot 7\cdot 8} \quad 
\text{etc.}
\]
Using this expression, the terms of the progression $1$, $2$, $6$, $24$ etc. will be defined, whose indices are fractions having the denominator $4$. For, the term, whose index is $\frac{p}{4}$, will be found to be

\[=\sqrt[4]{1\cdot 2\cdot 3\cdots p\left(\dfrac{2p}{4} + 1\right)\left(\dfrac{3p}{4} + 1\right)\left(p + 1\right)}
\]
\[
\times\int dx(x - xx)^{\frac{p}{4}}\int dx(xx - x^{3})^{\frac{p}{4}}\int dx(x^{3} - x^{4})^{\frac{p}{4}}.
\]
Hence it is possible to conclude in general that the term of order $\frac{p}{q}$ is

\[
= \sqrt[q]{1\cdot 2\cdot 3\cdots p\left(\dfrac{2p}{q} + 1\right)\left(\dfrac{3p}{q} + 1\right)\left(\dfrac{4p}{q} + 1\right)\cdots (p + 1)}\]
\[
\times\int dx(x -xx)^{\frac{p}{q}}\int dx(x^{2} - x^{3})^{\frac{p}{q}}\int dx(x^{3} - x^{4})^{\frac{p}{q}}\cdots \int  dx(x^{q - 1} - x^{q})^{\frac{p}{q}}.
\]
Therefore, from this formula the terms corresponding to any arbitrary fractional indices are found by means of a quadrature of algebraic curves; but for this $1 \cdot 2 \cdot 3 \cdots p$ is required, the term, whose index is the numerator of the propounded fraction.

\paragraph*{§24}
In the same way it is possible to proceed further to higher composed progressions by assuming higher composited numbers, but I will not follow this line any further here.  It is also possible to use multiple integrals so that the general term is

\[
\int qdx\int pdx;
\]
for, the integral of $pdx$ must be multiplied by  $qdx$ and what results from the integration must be then integrated again; and the result of this second integration will just then, having put $x=1$, give the general term of the series. But in each of the two integrations, in order for it to be well-defined, a constant of such a kind has to be added that having put $x=0$ the integral likewise becomes $=0$.\\[2mm]
In like manner, general terms are to be treated, which are expressed using several integrals, i.e.

\[
\int rdx\int qdx\int pdx.
\]
But nevertheless   functions of such a kind are always to be taken for $p$, $q$, $r$ that, as often as $n$ is a positive integer, at least algebraic terms result.

\paragraph*{§25}
Let the general term be

\[
\int \dfrac{dx}{x}\int x^{e}dx(1 - x)^{n};
\]
this expression converted into a series gives

\[
\dfrac{x^{e + 1}}{(e + 1)^{2}} - \dfrac{nx^{e + 2}}{1\cdot (e + 2)^{2}} + \dfrac{n(n - 1)x^{e + 3}}{1\cdot 2\cdot (e + 3)^{2}} - \text{etc.}
\]
Having put $x=1$, one will have the term of order $n$ expressed by this series

\[
\dfrac{1}{(e + 1)^{2}} - \dfrac{n}{1\cdot (e + 2)^{2}} + \dfrac{n(n - 1)}{1\cdot 2\cdot (e + 3)^{2}} - \text{etc.}
\]
The progression, beginning from the term corresponding to the index $0$, will be 
\begin{small}
\[
\dfrac{1}{(e + 1)^{2}}, \quad  \dfrac{(e + 2)^{2} - (e + 1)^{2}}{(e + 2)^{2}(e + 1)^{2}}, \quad \dfrac{(e + 3)^{2}(e + 2)^{2} - 2(e + 3)^{2}(e + 1)^{2} + (e + 2)^{2}(e + 1)^{2}}{(e + 3)^{2}(e + 2)^{2}(e + 1)^{2}}
\]
\[\dfrac{(e + 4)^{2}(e + 3)^{2}(e + 2)^{2} - 3(e + 4)^{2}(e + 3)^{2}(e + 1)^{2}+ 3(e + 4)^{2}(e + 2)^{2}(e + 1)^{2} - (e + 3)^{2}(e + 2)^{2}(e + 1)^{2}}{(e + 4)^{2}(e + 3)^{2}(e + 2)^{2}(e + 1)^{2}}\]
\begin{center}
etc.
\end{center}
\end{small}The structure of this progression is manifest and does not require any explanation. Let $e=0$; it will be

\[
\int dx(1 - x)^{n} = \dfrac{1 - (1 - x)^{n + 1}}{n + 1};
\]
therefore, the general term is

\[
\int\dfrac{dx - dx(1 - x)^{n + 1}}{(n + 1)x},
\]
the progression on the other hand will be 

\[
\dfrac{1}{1}, \quad \dfrac{4 - 1}{4\cdot 1}, \quad \dfrac{9\cdot 4 - 2\cdot 9\cdot 1 + 4\cdot 1}{9\cdot 4\cdot 1}, \quad \dfrac{16\cdot 9\cdot 4 - 3\cdot 16\cdot 9\cdot 1 + 3\cdot 16\cdot 4\cdot 1 - 9\cdot 4\cdot 1}{16\cdot 9\cdot 4\cdot 1}.
\]
The difference will constitute this progression

\[
\dfrac{-1}{4\cdot 1}, \quad \dfrac{-9 + 4}{9\cdot 4\cdot 1}, \quad\dfrac{-16\cdot 9 + 2\cdot 16\cdot 4 - 9\cdot 4}{16\cdot 9\cdot 4\cdot 1} \quad 
\text{etc.}
\]

\paragraph*{§26}

Therefore, in this dissertation I achieved, what I mainly intended, namely to  find the general terms of all progressions, each term of which is a product of  factors proceeding in an arithmetic progression, and in which the number of factors depends on the index in an arbitrary manner. But although here the number of factors is always set equal to the index, if the number of factors is desired to depend on it in another way, this will not cause any difficulties. The index is denoted by the letter $n$; if now anyone would require that the number of factors is $\frac{nn+n}{2}$, it is only necessary to substitute $\frac{nn+n}{2}$ for $n$ everywhere.

\paragraph*{§27}

Instead of ending the dissertation here, I want to add something more curious than useful. It is known that by $d^nx$ the differential of order $n$ of $x$ is understood and $d^np$, if $p$ denotes a certain function of $x$ and $dx$ is put constant, is proportional to $dx^n$; and, if $n$ is a positive integer number, the ratio of $d^np$ to $d^nx$ can always be expressed algebraically; consider, e.g., the case $n=2$ and $p=x^3$, the ratio of $d^2(x^3)$  to $dx^2$ will be the same as $6x$ to $1$. Now it is in question, what that ratio will be, if $n$ is a rational number.  The difficulty is easily understood in these cases; for, if $n$ is a positive integer number, $d^n$ is found by iterated differentiation; but this is not possible, if $n$ is a fractional number. But nevertheless by means of interpolations of the progressions I considered in this dissertation, it will be possible to answer this question.

\paragraph*{§28}

Let the ratio of $d^n(z^e)$ to $dz^n$ to be found for constant $dz$, or let the value of the fraction $\frac{d^n(z^e)}{dz^n}$ be in question. First, let us see, what its values are, if $n$ is an integer number; after this we will then proceed to the non-integer cases. If  $n=1$, its value will be

\[ez^{e - 1} = \dfrac{1\cdot 2\cdot 3\cdots e}{1\cdot 2\cdot 3\cdots (e - 1)}z^{e - 1};
\]
I express $e$ in this way such that later on the results we found in this paper can be applied more easily here. \\
If  $n=2$, the value will be

\[
e(e - 1)z^{e - 2} = \dfrac{1\cdot 2\cdot 3\cdots e}{1\cdot 2\cdot 3\cdots (e - 2)}z^{e - 2}.
\]
If  $n = 3$, one will have
\[
e(e - 1)(e - 2)z^{e - 3} = \dfrac{1\cdot 2\cdot 3\cdots e}{1\cdot 2\cdot 3\cdots (e - 3)}z^{e - 3}.
\]
Hence, I generally infer, whatever $n$ is, that it will always be

\[
\dfrac{d^{n}(z^{e})}{dz^{n}} = \dfrac{1\cdot 2\cdot 3\cdots e}{1\cdot 2\cdot 3\cdots (e - n)}z^{e - n}.
\]
But by means of § 14 

\[
1\cdot 2\cdot 3\cdots e = \int dx(-\log(x))^{e} \text{ und } 1\cdot 2\cdot 3\cdots (e - n) = \int dx(-\log(x))^{e - n}.
\]
Hence one has

\[
\dfrac{d^{n}(z^{e})}{dz^{n}} = z^{e - n}\dfrac{\int dx(-\log(x))^{e}}{\int dx(-\log(x))^{e - n}}
\]
or
\[
d^{n}(z^{e}) = z^{e - n}dz^{n}\dfrac{\int dx(-\log(x))^{e}}{\int dx(-\log(x))^{e - n}}.
\]
Here $dz$ is set constant and $\int dx (- \log (x))^s$ and $\int dx (-\log (x))^{s-n}$ must be integrated in such a way as it was prescribed above. 

\paragraph*{§29}

It is not necessary to show how the true value is found; this will become clear by substituting an arbitrary integer number for $n$. But let it be  questioned what $d^{\frac{1}{2}}z$ is, if $dz$ is constant. Therefore, it will be $e=1$ and $n=\frac{1}{2}$ in our expression. Therefore, one will have

\[
d^{\frac{1}{2}}z = \dfrac{\int dx(-\log(x))}{\int dx\sqrt{-\log(x)}}\sqrt{zdz}.
\]
But 
\[
\int dx(-\log(x)) = 1
\]
and having called the area of the circle, whose diameter is $1$, $A$, it will be

\[
\int dx\sqrt{-\log(x)} = \sqrt{A},
\]
whence 

\[
d^{\frac{1}{2}}z = \sqrt{\dfrac{zdz}{A}}.
\]
Therefore, let this equation for a  curve be propounded

\[
yd^{\frac{1}{2}}z = z\sqrt{dy},
\]
where $dz$ is set constant, and let it be questioned what kind of curve this is. Since  $d^{\frac{1}{2}}z= \sqrt{\frac{zdz}{A}}$, the equation will go over into this one

\[
y\sqrt{\dfrac{zdz}{A}} = z\sqrt{dz},
\]
which, having squared it, gives

\[
\dfrac{yydz}{A} = zdy,
\]
whence one finds

\[
\dfrac{1}{A}\log(z) = c - \dfrac{1}{y}
\]
or
\[
y\log(z) = cAy - A,
\]
which is the equation for the curve in question.

\lhead[\thepage]{}
\chead[Appendix]{Translation of E368}
\rhead[]{\thepage}
\chapter{Translation of E368 - On the hypergeometric curve expressed by the equation $y = 1 \cdot 2 \cdot 3 \cdots x$}

\paragraph*{1.} 

While  the letter $x$ denotes the abscissa and $y$ the ordinate,  this equation  immediately  indicates the quantity only of those ordinates corresponding to integer numbers; for, if one had  
\begin{equation*}
\renewcommand{\arraystretch}{1.5}
\setlength{\arraycolsep}{1,0mm}
\begin{array}{rrcccccccccccccccccc}
& \text{the abscissas }x\cdots &0, &1, &2, &3, & 4, &5, &6 & \text{etc.} \\
\text{one will have} \quad \\
& \text{the ordinates }y\cdots &1, &1, &2, &6, & 24, &120, &720 & \text{etc.} \\
\end{array}
\end{equation*}
such that, while the abscissas are taken according to the natural numbers, the ordinates proceed according to the Wallisian hypergeometric\footnote{Wallis called such series hypergeometric which are nowadays referred to as factorial series. Wallis reasoning for the name might was as follows: While in a geometric series the ratio of two subsequent terms is always the same, in factorial series the quotient of two subsequent terms increases and is thus more than geometric, i.e. "hypergeometric".} progression; therefore, it will be convenient to call also this curve hypergeometric. But even though  through this equation certainly innumerable, but just a discrete set of, points of this curve are assigned, nevertheless the nature of this curve is to be considered to be determined by this equation such that to each abscissa a certain, and, via this equation, well-defined ordinate corresponds\footnote{From our investigations in the main text we know that this statement is incorrect, confer especially section \ref{subsec: Classification Theorems}.}. For, the  nature of this equation requires, if the ordinate $y=q$ corresponds to a certain abscissa $x=p$, that the ordinate $y=q(p+1)$ corresponds to the abscissa $x=p+1$, but the ordinate $y=\frac{q}{p}$ corresponds to the abscissa $x=p-1$. Therefore, one cannot  draw  a certain curve of  parabolic kind through this infinitely many points arbitrarily, since all its points are determined from the equation.

\paragraph*{2.}

But except for these ordinates corresponding to  abscissas expressed by integer numbers, those are especially noteworthy, which fall into the middle between them from the  equation; and they are  all  determined by the one I once showed to correspond to the abscissa $x=\frac{1}{2}$ and to be equal to $\frac{1}{2}\sqrt{\pi}$. Therefore, since
\begin{equation*}
    \sqrt{\pi}= 1.77245385090548,
\end{equation*}
all these ordinates as well for the positive as for the negative abscissas will be as follows:

\begin{equation*}
\renewcommand{\arraystretch}{1.7}
\setlength{\arraycolsep}{2,0mm}
\begin{array}{c|l||c|l}
 & \text{for the positive abscissas} \quad &  & \text{for the negative abscissas} \\
 x & \text{the ordinate $y$ is} & x & \text{the ordinate $y$ is} \\
 0 & 1 & 0 & +1 \\
 \frac{1}{2} & 0.8862269 & -\frac{1}{2} & +1.7724538 \\
 1 & 1 & -1 & \pm \infty \\
 1\frac{1}{2} & 1.3293404 & -1\frac{1}{2} & -3.5449077 \\
 2 & 2 & -2 & \mp \infty \\
 2\frac{1}{2} & 3.3233509 & -2\frac{1}{2} &+2.3632718 \\
 3 & 6 & -3 & \pm \infty \\
 3\frac{1}{2} & 11.6317284 & -3\frac{1}{2} & -0.9453087 \\
 4 & 24 & -4 & \mp \infty \\
 4\frac{1}{2} & 52.3427777 & -4\frac{1}{2} & +0.2700882 \\
 5 & 120 & -5 & \pm \infty  \\
 5\frac{1}{2} & 287.8852775 & -5\frac{1}{2} & -0.0600196 \\
 6 & 720 & -6 & \mp \infty \\
 6\frac{1}{2} & 1871.2543038 & -6\frac{1}{2} & +0.0109126 \\
 7 & 5040 & -7 & \pm \infty
\end{array}
\end{equation*}
From this I drew the curve seen in figure 1, which extends from the negative abscissa $x=-1$,

\begin{center}
    \includegraphics[scale=0.7]{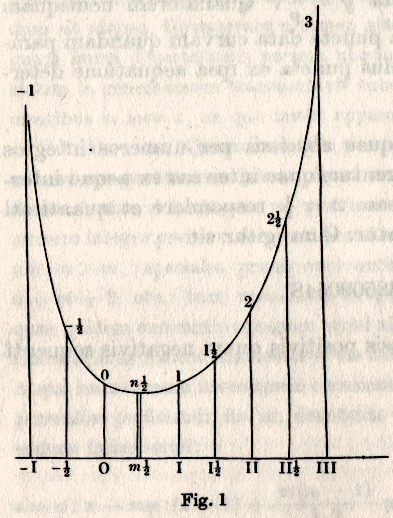}\\
    \begin{footnotesize}
    Figure 1 from E368. It shows the factorial function from $x=-1$ to $x=3$.\\ The scan was taken from the Opera Onmia Version, i.e. p. 44 of volume 28 of the first series.
    \end{footnotesize}
\end{center}
where the ordinate becomes infinite, to $x=3$, where $y=6$, and  from this point on is to be understood to ascend to infinity;  but the values on the left, where for each integer value of the abscissa the ordinates go over into asymptotes, I did not expresses beyond $x=-1$.

\paragraph*{3.}

The consideration of this curve raises many rather curious questions, providing a reason to examine it more accurately; and their solution seem to be even more interesting, since the equation for our curve cannot be expressed in the usual manner. Questions of this kind first concern  the determination of the remaining points of a curve in addition to those which are easily assigned. After this, in each point the tangents require their own investigation such that the behaviour of the whole curve can be defined more easily. But  from the inspection of the figure it is perspicuous that between the abscissas $x=0$ and $x=1$ there must be a smallest ordinate somewhere:  to assign  its abscissa and its value will be worth of one's while.\\
Furthermore,  between each two negative abscissas $-1$, $-2$, $-3$, $-4$, $-5$ etc.,  where the ordinates extend to infinity, there necessarily are smallest ordinates,  which, the more we proceed to the left, become continuously smaller until eventually they  vanish completely. Finally, also the question on the curvature radius in each point deserves our attention, and  the point, where the curvature is the largest, seems to be especially remarkable, since it is manifest that in the elongation of the curve from the axis the  branches come continuously closer to the straight line. Thus, I want to resolve those questions.

\subsection*{First Question}

\textit{To find a continuous equation among the abscissa $x$ and the ordinate $y$ for the hypergeometric curve, which equally holds, no matter whether  an integer or a fractional number is taken for $x$.}

\paragraph*{4.}

Since the propounded equation $y=1 \cdot 2 \cdot 3 \cdots x$ can only hold, if $x$ is an integer number, it must be cast into another form which is not restricted by this condition; this can be achieved in multiple ways  using expressions running to infinity, among which at first this one occurs:

\begin{equation*}
    y= \dfrac{1}{1+x}\left(\dfrac{2}{1}\right)^{x}\cdot \dfrac{2}{2+x}\left(\dfrac{3}{2}\right)^{x} \cdot \dfrac{3}{3+x}\left(\dfrac{4}{3}\right)^{x} \cdot
    \dfrac{4}{4+x}\left(\dfrac{5}{4}\right)^{x} \cdot \text{etc.}
\end{equation*}
 the factors of which must be continued to infinity. The reason for the validity of this expression is obvious, since the more factors are taken the closer the true value is being obtained, and, having taken infinitely many factors, the true value is obtained accurately: For, if  the numbers of factors is $=n$, one has

\begin{equation*}
    y= \dfrac{1}{1+x}\cdot \dfrac{2}{2+x}\cdot \dfrac{3}{3+x} \cdots \dfrac{n}{n+x}(n+1)^x;
\end{equation*}
if its numerator is represented this way:

\begin{equation*}
    1 \cdot 2 \cdot 3 \cdots x(x+1)(x+2)(x+3) \cdots n,
\end{equation*}
the denominator on the other hand this way

\begin{equation*}
    (1+x)(2+x)(3+x)\cdots n(n+1)(n+2) \cdots (n+x),
\end{equation*}
having cancelled the common factors, it results in

\begin{equation*}
    y=\dfrac{1 \cdot 2 \cdot 3 \cdots x}{(n+1)(n+2)(n+3)\cdots (n+x)}(n+1)^x.
\end{equation*}
Hence, if $n$ is an infinite number,  because of the  $n+1$ single factors  and the total amount of $x$ factors of the denominator, the whole denominator is cancelled by the factor $(n+1)^x$ and the propounded equation becomes $y=1 \cdot 2 \cdot 3 \cdots x$. \\

\paragraph*{5.}

This formula can be generalised a bit; for, since the whole task reduces to this that the factor $(n+1)^x$ becomes equal to the last denominator

\begin{equation*}
    (n+1)(n+2)(n+3)\cdots (n+x),
\end{equation*}
in the case, in which $n$ is an infinite number, it is evident that  this condition is also satisfied, if the factor is in general set to $(n+a)^x$, where $a$ is an arbitrary finite number;  but this formula is most  appropriate for our task, if   a certain  mean value of $1$ and $x$, e.g., $a=\frac{1+x}{2}$ or $a=\sqrt{x}$, is attributed to $a$. Now it is necessary that this factor $(n+a)^x$  is resolved into so many factors as $n$ contains units, which is  conveniently achieved using this resolution\footnote{Actually, the resolution Euler states, has $n+1$ factors and not just $n$. But since he is interested in the case $n \rightarrow \infty$ and the resolution is correct, this does not cause any mistakes in the following.}: 

\begin{equation*}
    (n+a)^x= a^x \cdot \left(\dfrac{a+1}{a}\right)^x \cdot \left(\dfrac{a+2}{a+1}\right)^x\cdot \left(\dfrac{a+3}{a+2}\right)^x \cdots \left(\dfrac{a+n}{a+n-1}\right)^x.
\end{equation*}
Therefore, for an arbitrary abscissa $x$ we will have the ordinate:

\begin{equation*}
    y=a^x \cdot \dfrac{1}{1+x} \left(\dfrac{a+1}{a}\right)^x\cdot \dfrac{2}{2+x} \left(\dfrac{a+2}{a+1}\right)^x \cdot \dfrac{3}{3+x} \left(\dfrac{a+3}{a+2}\right)^x \cdot \text{etc. to infinity},
\end{equation*}
which expression is always true, whatever number is chosen for $a$, but leads to the truth most quickly, if one takes $a=\frac{1+x}{2}$, whence it will be:

\begin{equation*}
    y= \left(\dfrac{1+x}{2}\right)^x\cdot \dfrac{1}{1+x} \left(\dfrac{3+x}{1+x}\right)^x \cdot \dfrac{2}{2+x} \left(\dfrac{5+x}{3+x}\right)^x \cdot \dfrac{3}{3+x} \left(\dfrac{7+x}{5+x}\right)^x \cdot \text{etc.},
\end{equation*}
which  expression consists of infinitely many factors of the form

\begin{equation*}
    \dfrac{m}{m+x}\left(\dfrac{a+m}{a-m-1}\right)^x
\end{equation*}
except the first $a^x$, and the more terms are multiplied by each other in a given case, the closer one will get to the truth. But the initial expression results, if one takes $a=1$.

\paragraph*{6.}

 But  this expression is the more useful, the faster the  factors converge to one, which happens by taking $a=\frac{1+x}{2}$;  indeed, then the calculation will become as much easier as smaller numbers are substituted  for $x$; but it always suffices to have investigated ordinates for  abscissas $x$ between one and zero, since from there  the  ordinates  corresponding to $x+1$, $x+2$, $x+3$, $x+4$ etc. are easily  derived. Therefore, let $x=\frac{\alpha}{\beta}$, while $\alpha  < \beta$, and it will be
 
 \begin{equation*}
     y = \left(\dfrac{\alpha +\beta}{2 \beta}\right)^{\frac{\alpha}{\beta}} \cdot \dfrac{\beta}{\alpha +\beta}  \left(\dfrac{3\alpha +\beta}{ \beta +\alpha}\right)^{\frac{\alpha}{\beta}} \cdot \dfrac{2 \beta}{\alpha +2 \beta} \left(\dfrac{5\beta+\alpha}{3 \beta +\alpha}\right)^{\frac{\alpha}{\beta}} \cdot \dfrac{3 \beta}{\alpha + 3 \beta}\left(\dfrac{7 \beta + \alpha}{5\beta+\alpha}\right)^{\frac{\alpha}{\beta}}\cdot \text{etc.},
 \end{equation*}
 whence  the power $y^{\beta}$ of the ordinate results  in this expression:
 
 \begin{equation*}
     y^{\beta}=\left(\dfrac{\alpha +\beta}{2 \beta}\right)^{\alpha} \cdot \dfrac{\beta^{\beta}(3\beta +\alpha)^{\alpha}}{(\beta + \alpha)^{\beta}(\beta +\alpha)^{\alpha}} \cdot \dfrac{(2\beta)^{\beta}(5\beta +\alpha)^{\alpha}}{(2\beta + \alpha)^{\beta}(3\beta +\alpha)^{\alpha}} \cdot \dfrac{(3\beta)^{\beta}(7\beta +\alpha)^{\alpha}}{(3\beta + \alpha)^{\beta}(5\beta +\alpha)^{\alpha}} \cdot \text{etc.}
 \end{equation*}
 But for the abscissa $x=-\frac{\alpha}{\beta}$ the ordinate is hence calculated to be
 
 \begin{equation*}
     y^{\beta}= \left(\dfrac{2 \beta}{\beta - \alpha}\right)^{\alpha} \cdot \dfrac{\beta^{\beta}(\beta-\alpha)^{\alpha}}{(\beta-\alpha)^{\beta}(3\beta - \alpha)^{\alpha}} \cdot \dfrac{(2\beta)^{\beta}(3\beta -\alpha)^{\alpha}}{(2\beta -\alpha)^{\beta}(5\beta-\alpha)^{\alpha}}\cdot \dfrac{(3\beta)^{\beta}(5\beta -\alpha)^{\alpha}}{(3\beta -\alpha)^{\beta}(7\beta-\alpha)^{\alpha}} \cdot \text{etc.}
 \end{equation*}
 For the sake of an example, let us take $x=\frac{1}{2}$ and we will obtain:
 
 \begin{equation*}
     y^2 = \dfrac{3}{4}\cdot \dfrac{2 \cdot 2 \cdot 7}{3 \cdot 3 \cdot 3} \cdot \dfrac{4 \cdot 4 \cdot 11}{5 \cdot 5 \cdot 7} \cdot \dfrac{6 \cdot 6 \cdot 15 }{7 \cdot 7 \cdot 11} \cdot \dfrac{8 \cdot 8 \cdot 19}{9 \cdot 9 \cdot 15} \cdot \text{etc.};
 \end{equation*}
 since a general factor of this is
 
 \begin{equation*}
     \dfrac{2n \cdot 2n (4n+3)}{(2n+1)(2n+1)(4n-1)}= \dfrac{16n^3+12nn}{16n^3+12nn-1}=1+\dfrac{1}{(2n+1)^2(4n-1)},
 \end{equation*}
  it is seen in general how quickly these factors converge to $1$; therefore, it will be:
 
 \begin{equation*}
     y^2 = \dfrac{3}{4}\left(1+\dfrac{1}{3^2 \cdot 3}\right)\left(1+\dfrac{1}{5^2 \cdot 7}\right)\left(1+\dfrac{1}{7^2 \cdot 11}\right)\left(1+\dfrac{1}{9^2 \cdot 15}\right)\left(1+\dfrac{1}{11^2 \cdot 19}\right)\text{etc.},
 \end{equation*}
 where we  know that $y^2=\frac{\pi}{4}$. But if we set $x=-\frac{1}{2}$, which corresponds to $y=\sqrt{\pi}$, from  the other expression it will be
 
 \begin{equation*}
     \pi = 4 \cdot \dfrac{2 \cdot 2 \cdot 1}{1 \cdot 1 \cdot 5}\cdot \dfrac{4 \cdot 4 \cdot 5}{3 \cdot  3 \cdot 9}\cdot \dfrac{6 \cdot 6 \cdot 9}{5 \cdot 5 \cdot 13} \cdot \dfrac{8 \cdot 8 \cdot 13}{7 \cdot 7 \cdot 17}\cdot \text{etc.}
 \end{equation*}
 or
 
 \begin{equation*}
     \pi = 4 \left(1-\dfrac{1}{1^2 \cdot 5}\right)\left(1-\dfrac{1}{3^2 \cdot 9}\right)\left(1-\dfrac{1}{5^2 \cdot 13}\right)\left(1-\dfrac{1}{7^2 \cdot 17}\right)\text{etc.},
 \end{equation*}
 hence
 
 \begin{equation*}
     \pi = 3 \left(1+\dfrac{1}{3^2 \cdot 3}\right) \left(1+\dfrac{1}{5^2 \cdot 7}\right) \left(1+\dfrac{1}{7^2 \cdot 11}\right) \left(1+\dfrac{1}{9^2 \cdot 15}\right)\text{etc.},
 \end{equation*}
 such that  the one expression will get to the truth while it is increasing, the other while it is decreasing.
 
 \paragraph*{7.}
 
 But the calculation is executed more conveniently, if our expression is terminated at each factor; for, then  the following formulas coming continuously closer to the truth will result:
 
 \begin{equation*}
\renewcommand{\arraystretch}{2.5}
\setlength{\arraycolsep}{2,0mm}
\begin{array}{l}
y = \dfrac{1}{1+x}\left(\dfrac{3+x}{2}\right)^x \\
y = \dfrac{1}{1+x}\cdot \dfrac{2}{2+x}\left(\dfrac{5+x}{2}\right)^x \\
y = \dfrac{1}{1+x}\cdot \dfrac{2}{2+x} \cdot \dfrac{3}{3+x}\left(\dfrac{7+x}{2}\right)^x \\
y = \dfrac{1}{1+x}\cdot \dfrac{2}{2+x} \cdot \dfrac{3}{3+x} \cdot \dfrac{4}{4+x}\left(\dfrac{9+x}{2}\right)^x \\
y = \dfrac{1}{1+x}\cdot \dfrac{2}{2+x} \cdot \dfrac{3}{3+x} \cdot \dfrac{4}{4+x} \cdot \dfrac{5}{5+x}\left(\dfrac{11+x}{2}\right)^x.
\end{array}
\end{equation*}
Since, if  one writes $x-1$ instead of $x$, the ordinate $=\dfrac{y}{x}$ results, by similar formulas it will be

\begin{equation*}
    \renewcommand{\arraystretch}{2.5}
\setlength{\arraycolsep}{2,0mm}
\begin{array}{l}
y= \left(\dfrac{2+x}{2}\right)^{x-1} \\
y= \dfrac{2}{1+x}\left(\dfrac{4+x}{2}\right)^{x-1} \\
y= \dfrac{2}{1+x} \cdot \dfrac{3}{2+x}\left(\dfrac{6+x}{2}\right)^{x-1} \\
y= \dfrac{2}{1+x} \cdot \dfrac{3}{2+x} \cdot \dfrac{4    }{3+x}\left(\dfrac{8+x}{2}\right)^{x-1} \\
y= \dfrac{2}{1+x}\cdot \dfrac{3}{2+x} \cdot \dfrac{4    }{3+x} \cdot \dfrac{5}{4+x}\left(\dfrac{10+x}{2}\right)^{x-1}.
\end{array}
\end{equation*}
Hence, having set $x=\frac{1}{2}$, for the ordinate $y=\frac{1}{2}\sqrt{\pi}$ two series of formulas converging to it result:

\begin{equation*}
    \renewcommand{\arraystretch}{2.5}
\setlength{\arraycolsep}{2,0mm}
\begin{array}{l|l}
\dfrac{1}{2}\sqrt{\pi} = \dfrac{2}{3} \sqrt{\dfrac{7}{4}} & \dfrac{1}{2}\sqrt{\pi} = \sqrt{\dfrac{4}{5}} \\
\dfrac{1}{2}\sqrt{\pi} = \dfrac{2 \cdot 4}{3 \cdot 5} \sqrt{\dfrac{11}{4}} & \dfrac{1}{2}\sqrt{\pi} = \dfrac{4}{3}\sqrt{\dfrac{4}{9}} \\
    \dfrac{1}{2}\sqrt{\pi} = \dfrac{2 \cdot 4 \cdot 6}{3 \cdot 5 \cdot 7} \sqrt{\dfrac{15}{4}} & \dfrac{1}{2}\sqrt{\pi} = \dfrac{4 \cdot 6}{3 \cdot 5}\sqrt{\dfrac{4}{13}} \\
     \dfrac{1}{2}\sqrt{\pi} = \dfrac{2 \cdot 4 \cdot 6 \cdot 8}{3 \cdot 5 \cdot 7 \cdot 9} \sqrt{\dfrac{19}{4}} & \dfrac{1}{2}\sqrt{\pi} = \dfrac{4 \cdot 6 \cdot 8}{3 \cdot 5 \cdot 7}\sqrt{\dfrac{4}{17}} \\
 \text{etc.} \quad & \dfrac{1}{2}\sqrt{\pi} = \dfrac{4 \cdot 6 \cdot 8 \cdot 10}{3 \cdot 5 \cdot 7 \cdot 9}\sqrt{\dfrac{4}{21}} \\
 & \text{etc.}
\end{array}
\end{equation*}

\paragraph*{8.}

But products of this kind are  expanded most conveniently using logarithms; and first  from the  general formula involving the  arbitrary number $a$ we obtain:

\begin{equation*}
    \renewcommand{\arraystretch}{2.5}
\setlength{\arraycolsep}{0mm}
\begin{array}{lllllllllllll}
     \log y = x \log a &~+~& x \log \dfrac{a+1}{a} &~+~ & x \log \dfrac{a+2}{a+1} &~+~ & x \log \dfrac{a+3}{a+2} &~+~ & x \log \dfrac{a+4}{a+3} &~+~ & \text{etc.} \\
                       &~-~& \log (1+x) &~-~& \log \left(1+\dfrac{x}{2}\right) &~-~& \log \left(1+\dfrac{x}{3}\right) &~-~& \log \left(1+\dfrac{x}{4}\right)& ~-~& \text{etc.}
\end{array}
\end{equation*}
having taken $a=\frac{1+x}{2}$ this series is rendered most convergent:

\begin{equation*}
    \log y = x \log \dfrac{1+x}{2} + x\log \dfrac{x+3}{x+1}  + x\log \dfrac{x+5}{x+3}  + x\log \dfrac{x+7}{x+5}  + x\log \dfrac{x+9}{x+7}+\text{etc.}
\end{equation*}
\begin{equation*}
    \log (1+x) ~-~ \log \left(1+\dfrac{x}{2}\right) ~-~ \log \left(1+\dfrac{x}{3}\right) ~-~ \log \left(1+\dfrac{x}{4}\right) ~-~ \text{etc.}
\end{equation*}
Therefore, having taken these logarithms and since in general:

\begin{equation*}
    \renewcommand{\arraystretch}{2.5}
\setlength{\arraycolsep}{0mm}
\begin{array}{rlllllllllllllllll}
     x \log \dfrac{x+2m+1}{x+2m-1} &~=~& \dfrac{2x}{x+2m} &~+~& \dfrac{2x}{3(x+2m)^3} &~+~& \dfrac{2x}{5(x+2m)^5} &~+~& \dfrac{2x}{7(x+2m)^7} &~+~& \text{etc.} \\
     \text{and} \log \left(1+\dfrac{x}{m}\right) &~=~& \dfrac{2x}{x+2m} &~+~& \dfrac{2x^3}{3(x+2m)^3} &~+~& \dfrac{2x^5}{5(x+2m)^5} &~+~& \dfrac{2x^7}{7(x+2m)^7} &~+~& \text{etc.} \\
\end{array}
\end{equation*}
we obtain the following formulas consisting of infinitely many series:

\begin{equation*}
    \renewcommand{\arraystretch}{2.5}
\setlength{\arraycolsep}{0mm}
\begin{array}{lllllllllllllllllll}
     \log y = x \log \dfrac{1+x}{2} &~+~& \dfrac{2}{3}x(1-xx)& \bigg(\dfrac{1}{(x+2)^3}&~+~&\dfrac{1}{(x+4)^3}&~+~&\dfrac{1}{(x+6)^3}&~+~&\dfrac{1}{(x+8)^3}&~+~&\text{etc.}\bigg) \\
        &~+~& \dfrac{2}{5}x(1-x^4)& \bigg(\dfrac{1}{(x+2)^5}&~+~&\dfrac{1}{(x+4)^5}&~+~&\dfrac{1}{(x+6)^5}&~+~&\dfrac{1}{(x+8)^5}&~+~&\text{etc.}\bigg) \\
           &~+~& \dfrac{2}{7}x(1-x^6)& \bigg(\dfrac{1}{(x+2)^7}&~+~&\dfrac{1}{(x+4)^7}&~+~&\dfrac{1}{(x+6)^7}&~+~&\dfrac{1}{(x+8)^7}&~+~&\text{etc.}\bigg) \\
   &~+~& \dfrac{2}{9}x(1-x^8)& \bigg(\dfrac{1}{(x+2)^9}&~+~&\dfrac{1}{(x+4)^9}&~+~&\dfrac{1}{(x+6)^9}&~+~&\dfrac{1}{(x+8)^9}&~+~&\text{etc.}\bigg) \\
   & &  & &  & &\text{etc.}
\end{array}
\end{equation*}

\paragraph*{9.}

Let us take a definite number of terms of the first series, which number we want to be $=n$, and since the upper part is reduced to the single term $x\log(a+n)$, it will be

\begin{equation*}
    \log y = x\log(a+x)-\log(1+x)-\log \left(1+\frac{1}{2}x\right)-\log \left(1+\frac{1}{3}x\right)- \cdots - \log \left(1+\frac{1}{n}x\right),
\end{equation*}
which expression comes the closer to the truth, the greater the number $n$ is taken. Therefore, let $n$ be very large, and first we will obviously have

\begin{equation*}
    \log(n+a)= \log n +\dfrac{a}{n}-\dfrac{aa}{2n^2}+\dfrac{a^3}{3n^3}-\text{etc.},
\end{equation*}
where it will be convenient to take $\frac{1+x}{2}$ for $a$; but then, for the sake of brevity, having put the fraction

\begin{equation*}
    0.5772156649015325= \Delta,
\end{equation*}
we know the sum of the harmonic progression to be:

\begin{equation*}
    1+\dfrac{1}{2}+\dfrac{1}{3}+\dfrac{1}{4}+\cdots +\dfrac{1}{n}= \Delta +\log n +\dfrac{1}{2n}-\dfrac{1}{12nn}+\dfrac{1}{120n^4}-\text{etc.},
\end{equation*}
whence, since:

\begin{equation*}
    \log(n+a)=  1+\dfrac{1}{2}+\dfrac{1}{3}+\dfrac{1}{4}+\cdots +\dfrac{1}{n} - \Delta  -\dfrac{1}{2n}+\dfrac{1}{12nn}-\dfrac{1}{120n^4}+\text{etc.}
\end{equation*}
\begin{equation*}
    +\dfrac{a}{n}-\dfrac{aa}{2nn}+\dfrac{a^3}{3n^3}-\text{etc.},
\end{equation*}
having taken $a=\frac{1+x}{2}$, we conclude

\begin{equation*}
    \log y = -\Delta x+x +\dfrac{1}{2}x+\dfrac{1}{3}x+\cdots + \dfrac{1}{n}x
\end{equation*}
\begin{equation*}
    -\log(1+x)-\log \left(1+\dfrac{1}{2}x\right)-\log \left(1+\dfrac{1}{3}x\right)-\cdots -\log \left(1+\dfrac{1}{n}x\right)
\end{equation*}
\begin{equation*}
    +\dfrac{xx}{2n}-\dfrac{x+6xx+3x^3}{24nn}+\text{etc.}
\end{equation*}
Therefore, by increasing the number $n$ to infinity, it will actually be:

\begin{equation*}
     \log y = -\Delta x+x +\dfrac{1}{2}x+\dfrac{1}{3}x+\dfrac{1}{4}x+\text{etc.}
\end{equation*}
\begin{equation*}
    -\log(1+x)-\log \left(1+\dfrac{1}{2}x\right)-\log \left(1+\dfrac{1}{3}x\right) -\log \left(1+\dfrac{1}{4}x\right)-\text{etc.}
\end{equation*}
and, having expanded each logarithm into a series:

\begin{equation*}
       \renewcommand{\arraystretch}{2.5}
\setlength{\arraycolsep}{0mm}
\begin{array}{lllclllllllllllll}
\log y = - \Delta x &~+~& \dfrac{1}{2}&xx&\bigg(1+\dfrac{1}{2^2}+\dfrac{1}{3^2}+\dfrac{1}{4^2}+\text{etc.}\bigg)\\
 &~-~& \dfrac{1}{3}&x^3&\bigg(1+\dfrac{1}{2^3}+\dfrac{1}{3^3}+\dfrac{1}{4^3}+\text{etc.}\bigg)\\
  &~+~& \dfrac{1}{4}&x^4&\bigg(1+\dfrac{1}{2^4}+\dfrac{1}{3^4}+\dfrac{1}{4^4}+\text{etc.}\bigg)\\
   &~-~& \dfrac{1}{5}&x^5&\bigg(1+\dfrac{1}{2^5}+\dfrac{1}{3^5}+\dfrac{1}{4^5}+\text{etc.}\bigg)\\
   & & &\text{etc.}
\end{array}
\end{equation*}

\paragraph*{10.}

But except for those formulas, in which  the ordinate $y$ corresponding to a certain abscissa $x$ is assigned, my method to sum progressions indefinitely\footnote{Euler refers to the Euler-Maclaurin summation formula.}  provides us with an extraordinary expression  accommodated to our purposes.\\
For, since $\log y = \log 1 +\log 2+\log 3+\log 4 +\cdots +\log x$, this progression must be summed indefinitely; but introducing numerical values:

\begin{equation*}
    A=\dfrac{1}{6}, ~~ B=\dfrac{1}{90}, ~~ C=\dfrac{1}{945}, ~~ D=\dfrac{1}{9450}, ~~ E=\dfrac{1}{93555},~~ F=\dfrac{691}{1 \cdot 3 \cdot 5 \cdots 15 \cdot 315}~~ \text{etc.},
\end{equation*}
which progression is of such a nature that

\begin{equation*}
    5B = 2AA,~~ 7C=4AB,~~ 9D=4AC+2BB,~~11E=4AD+4BC ~~\text{etc.},
\end{equation*}
I showed elsewhere\footnote{The following formula is just the Stirling formula for the factorial. Euler showed it, e.g., in \cite{E212}.} that it will be:

\begin{equation*}
    \log y = \dfrac{1}{2}\log 2\pi +\left(x+\dfrac{1}{2}\right)\log x -x +\dfrac{A}{2x}-\dfrac{1 \cdot 2 B}{2^3x^3}+\dfrac{1 \cdot 2 \cdot 3 \cdot 4 C}{2^5x^5}-\dfrac{1 \cdot 2 \cdot 3 \cdot 4 \cdot 5 \cdot 6 D}{2^7x^7}+\text{etc.},
\end{equation*}
which series, compared to the first, has the use that, the greater the abscissas $x$ are taken, the faster it exhibits the true value of the ordinate $y$. Therefore, since, if the ordinate $y$ corresponds to the abscissa $x$, the following ordinate 

\begin{equation*}
    y(x+1)(x+2)(x+3)\cdots(x+n)
\end{equation*}
to the larger abscissa $x+n$, we will have the  following rapidly convergent series\footnote{This series does, in fact, not converge and is to be understood as an asymptotic series. Confer our discussion in section \ref{subsubsec: An Application - Derivation of the Stirling Formula for the Factorial}.}:

\begin{equation*}
    \log y = \dfrac{1}{2}\log 2 \pi -\log(x+1)-\log(x+2)-\log(x+3)-\cdots -\log(x+n)
\end{equation*}
\begin{equation*}
    +\left(x+n+\dfrac{1}{2}\right)\log(x+n)-x-n
\end{equation*}
\begin{equation*}
    +\dfrac{A}{2(x+n)}-\dfrac{1 \cdot 2B}{2^3(x+n)^3}+\dfrac{1\cdot 2 \cdot 3 \cdot 4 C}{2^5(x+n)^5}-\dfrac{1 \cdot 2 \cdot 3 \cdot 4 \cdot 5 \cdot 6 D}{2^7(x+n)^7}+\text{etc.}
\end{equation*}
Therefore, if $e$ denotes the number whose natural logarithm is $=1$, and  if, for the sake of brevity, one sets:

\begin{equation*}
    \dfrac{A}{2(x+n)}-\dfrac{1 \cdot 2B}{2^3(x+n)^3}+\dfrac{1\cdot 2 \cdot 3 \cdot 4 C}{2^5(x+n)^5}-\text{etc.}=s,
\end{equation*}
going back from logarithms to numbers we conclude:

\begin{equation*}
    y=\dfrac{\sqrt{2\pi(x+n)}}{(x+1)(x+2)(x+3)\cdots (x+n)}\left(\dfrac{x+n}{e}\right)^{x+n}e^s,
\end{equation*}
where the integer number $n$ is arbitrary; but the larger it is taken, the easier the true value of $s$ can be found.

\paragraph*{11.}

Finally, the ordinate $y$ can even be exhibited  by an integral formula; for, having  put the abscissa $x=p$ and having introduced the new variable $u$, independent of the quantity $p$, the ordinate will be

\begin{equation*}
    y = \int du \left(\log \dfrac{1}{u}\right)^p,
\end{equation*}
if the integration is extended from the value $u=0$ to the value $u=1$. Or, if one prefers the exponential form, it will also be

\begin{equation*}
    y = \int e^{-v}v^pdv,
\end{equation*}
extending the integration from $v=0$ to $v=\infty$. From those formulas,  if the abscissa $p$ is an integer number,  the integration indeed immediately yields

\begin{equation*}
    y=1 \cdot 2 \cdot 3 \cdots p,
\end{equation*}
but if $p$ was a fractional number, hence it is at the same time  understood  to which class of  transcendental quantities the  value of $y$ is to be referred. Indeed, I showed on another occasion, how the  integral can then be expressed using quadratures of algebraic curves.

\paragraph*{12.}

Therefore, lo and behold the many solutions of our first question, in which for an arbitrary abscissa $x$, even though  it is expressed by a non-integer number, the value of the ordinate $y$ was  sought after; it will be helpful to have  listed up the principal ones, that from here in each case the one which seems to be the most useful can be chosen:

\begin{small}
\begin{equation*}
         \renewcommand{\arraystretch}{2.5}
\setlength{\arraycolsep}{0mm}
\begin{array}{rrlllllllllllllll}
\text{I.}\quad & y &~=~& \dfrac{1}{1+x}\left(\dfrac{2}{1}\right)^x \cdot \dfrac{2}{2+x}\left(\dfrac{3}{2}\right)^x \cdot \dfrac{3}{3+x}\left(\dfrac{4}{3}\right)^x \cdot
\dfrac{4}{4+x}\left(\dfrac{5}{4}\right)^x \cdot \text{etc.} \\
\text{II.} \quad &y&~=~& \left(\dfrac{1+x}{2}\right)^x \cdot \dfrac{1}{1+x}\left(\dfrac{3+x}{2+x}\right)^x\cdot \dfrac{2}{2+x}\left(\dfrac{5+x}{3+x}\right)^x \cdot \dfrac{3}{3+x}\left(\dfrac{7+x}{5+x}\right)^x \cdot \text{etc.} \\
\text{III.} \quad &\log y &~=~& x \log \dfrac{2}{1} +  x \log \dfrac{3}{2} +  x \log \dfrac{4}{3} +  x \log \dfrac{5}{4} + \text{etc.}\\
                 & &~-~& \log(1+x)-\log \left(1+\dfrac{1}{2}x\right)-\log \left(1+\dfrac{1}{3}x\right)-\log \left(1+\dfrac{1}{4}x\right)-\text{etc.} \\
\text{IV.} \quad &\log y &~=~& x \log \dfrac{1+x}{2} +x\log \dfrac{x+3}{x+1} +x\log \dfrac{x+5}{x+3} +x\log \dfrac{x+7}{x+5} +x\log \dfrac{x+9}{x+7}+\text{etc.} \\
 & &~-~& \log(1+x)-\log \left(1+\dfrac{1}{2}x\right)-\log \left(1+\dfrac{1}{3}x\right)-\log \left(1+\dfrac{1}{4}x\right)-\text{etc.} \\
 \text{V.} \quad & \log y &~=~& -\Delta x +x+\dfrac{1}{2}x+\dfrac{1}{3}x+\dfrac{1}{4}x+\text{etc.} \\
  & &~-~& \log(1+x)-\log \left(1+\dfrac{1}{2}x\right)-\log \left(1+\dfrac{1}{3}x\right)-\log \left(1+\dfrac{1}{4}x\right)-\text{etc.} \\
 \text{VI.} \quad & \log y & ~=~& \left\lbrace
        \renewcommand{\arraystretch}{2.5}
\setlength{\arraycolsep}{0mm}
\begin{array}{lllclllllllllllll}
 - \Delta x &~+~& \dfrac{1}{2}&xx&\bigg(1+\dfrac{1}{2^2}+\dfrac{1}{3^2}+\dfrac{1}{4^2}+\text{etc.}\bigg)\\
 &~-~& \dfrac{1}{3}&x^3&\bigg(1+\dfrac{1}{2^3}+\dfrac{1}{3^3}+\dfrac{1}{4^3}+\text{etc.}\bigg)\\
  &~+~& \dfrac{1}{4}&x^4&\bigg(1+\dfrac{1}{2^4}+\dfrac{1}{3^4}+\dfrac{1}{4^4}+\text{etc.}\bigg)\\
   &~-~& \dfrac{1}{5}&x^5&\bigg(1+\dfrac{1}{2^5}+\dfrac{1}{3^5}+\dfrac{1}{4^5}+\text{etc.}\bigg)\\
   &~+~ & &\text{etc.}
\end{array} \right.
\\
\text{VII.} \quad &\log y &~=~& \dfrac{1}{2}\log 2 \pi +\left(x+\dfrac{1}{2}\right)\log x -x +\dfrac{A}{2x}-\dfrac{1 \cdot 2 B}{2^3x^3}+\dfrac{1 \cdot 2 \cdot 3 \cdot 4 C}{2^5 x^5}-\dfrac{1 \cdot 2 \cdots 6D}{2^7x^7}+\text{etc.}
\end{array}
\end{equation*}
\end{small}
while $\Delta = 0.57721556649014225$ and

\begin{equation*}
    A=\dfrac{1}{6}, \quad B=\dfrac{1}{90}, \quad C=\dfrac{1}{945}, \quad D=\dfrac{1}{9450}, \quad E=\dfrac{1}{93555} \quad \text{etc.}
\end{equation*}
Then in the three last forms one has to use natural logarithms.

\subsection*{Second Question}

\textit{To define the direction of its tangent on the hypergeometric curve  for each point.}

\paragraph*{13.}

Therefore, here we assume that for the abscissa $x$ the value of the ordinate $y$ has already been found, and since the direction of the tangent  is defined by the ratio of the differentials $\frac{dy}{dx}$, by which fraction the tangent of the angle, in which the tangent at that point is inclined to the axis, is usually expressed, it is just necessary that we differentiate one of the found formulas. To this end, formula V seems especially suitable, from which we conclude:

\begin{equation*}
         \renewcommand{\arraystretch}{2.5}
\setlength{\arraycolsep}{0mm}
\begin{array}{lllllllllllllllll}
\dfrac{dy}{ydx}=- \Delta &~+~& 1 &~+~& \dfrac{1}{2} &~+~& \dfrac{1}{3} &~+~& \dfrac{1}{4} &~+~& \text{etc.} \\
                         &~-~& \dfrac{1}{1+x} &~-~& \dfrac{1}{2+x} &~-~& \dfrac{1}{3+x} &~-~& \dfrac{1}{4+x} &~-~& \text{etc.}
\end{array}
\end{equation*}
which expression is contracted into this more convenient one:

\begin{equation*}
    \dfrac{dy}{ydx}= -\Delta +\dfrac{x}{1+x}+\dfrac{x}{2(2+x)}+\dfrac{x}{3(3+x)}+\dfrac{x}{4(4+x)}+\text{etc.},
\end{equation*}
whence it is clear at the same time, if $x$ is a negative integer number, that not just the ordinate $y$ but also the formula $\frac{dy}{dx}$ becomes infinite such that in these points the  ordinates, since they are asymptotes, become the tangents. But let us in general put the angle the tangent constitutes with the axis to be $=\varphi$ such that

\begin{equation*}
    \dfrac{dy}{dx}=\tan \varphi.
\end{equation*}

\paragraph*{14.}

Therefore, first let us define the tangents for the abscissas $x$ which are expressed by positive numbers, since the ordinates $y$ are given.\\

\begin{itemize}
    \item[I.]  Therefore, let $x=0$ and, because of $y=1$,
    
    \begin{equation*}
        \dfrac{dy}{dx}= -\Delta =- 0.5772156649= \tan \varphi,
    \end{equation*}
    whence
    
    \begin{equation*}
        \text{the angle }\varphi =-29^{\circ}59'29'',
    \end{equation*}
    where the sign $-$ indicates that the tangent falls to the right of the axis and does not constitute an angle of $ 30^{\circ}$ with it.\\
    \item[II.]  Let $x=1$ and, because of $y=1$,
    
    \begin{equation*}
        \dfrac{dy}{dx}=1-\Delta = 0.422784335 =\tan \varphi,
    \end{equation*}
    and hence
    \begin{equation*}
        \text{the angle }\varphi =22^{\circ}55'.
    \end{equation*}
    \item[III.]  Let $x=2$ and, because of $y=2$, 
    
    \begin{equation*}
        \dfrac{dy}{dx}=2\left(1+\dfrac{1}{2}-\Delta\right)=1.845568670 = \tan \varphi
    \end{equation*}
    and hence
    
    \begin{equation*}
        \text{the angle }\varphi =61^{\circ}33'.
    \end{equation*}
    \item[IV.] Let $x=3$ and, because of $y=6$,
    
    \begin{equation*}
        \dfrac{dy}{dx}=6 \left(1+\dfrac{1}{2}+\dfrac{1}{3}-\Delta\right)=\tan \varphi
    \end{equation*}
    or
    
\begin{equation*}
    \tan \varphi =7.536706010 \quad \text{and} \quad \varphi = 82^{\circ}26'.
\end{equation*}
\item[V.] Let $x=4$ and, because of $y=24$,

\begin{equation*}
    \dfrac{dy}{dx}= 24 \left(1+\dfrac{1}{2}+\dfrac{1}{3}+\dfrac{1}{4}-\Delta\right)
\end{equation*}
and hence

\begin{equation*}
    \tan \varphi =36.146824040 \quad \text{and} \quad \varphi =88^{\circ}25'.
\end{equation*}
\end{itemize}
Therefore, in general, if the abscissa $x$ is equal to an arbitrary integer number, because of $y=1 \cdot 2 \cdots n$, it will be

\begin{equation*}
    \dfrac{dy}{dx}=\tan \varphi = 1 \cdot 2 \cdot 3 \cdots n \left(1+\dfrac{1}{2}+\dfrac{1}{3}+\cdots+\dfrac{1}{n}-\Delta\right).
\end{equation*}

\paragraph*{15.}

Hence let us also define the tangents for the intermediate points, and first certainly for those corresponding to positive abscissas:

\begin{itemize}
    \item[I.] Let $x=\frac{1}{2}$, it will be $y=\frac{1}{2}\sqrt{\pi}$ and
    \begin{equation*}
    \dfrac{dy}{ydx}=-\Delta +1-\dfrac{2}{3}+\dfrac{1}{2}-\dfrac{2}{5}+\dfrac{1}{3}-\dfrac{2}{7}+\text{etc.}
    \end{equation*}
    or
    
    \begin{equation*}
        \dfrac{dy}{ydx}= -\Delta +2 \left(\dfrac{1}{2}-\dfrac{1}{3}+\dfrac{1}{4}-\dfrac{1}{5}+\text{etc.}\right)=-\Delta +2(1-\log 2)
    \end{equation*}
    and hence
    
    \begin{equation*}
        \dfrac{dy}{dx}=\tan \varphi =y(2(1-\log 2)-\Delta)=0.0364899739 \cdot y.
    \end{equation*}
    
    \item[II.] Let $x=\frac{3}{2}$, it will be $y=\dfrac{1 \cdot 3}{2 \cdot 2}\sqrt{\pi}$ and
    
    \begin{equation*}
        \dfrac{dy}{ydx}=-\Delta +2\left(1+\dfrac{1}{3}-\log 2\right),
    \end{equation*}
    whence
    
    \begin{equation*}
        \dfrac{dy}{dx}= \tan \varphi = y \left(2\left(1+\dfrac{1}{3}-\log 2\right)-\Delta\right)= 0.7031566405 \cdot y.
    \end{equation*}
    \item[III.] Let $x=\frac{5}{2}$, it will be $y=\frac{1 \cdot 3 \cdot 5}{2 \cdot 2 \cdot 2}\sqrt{\pi}$ and
    
    \begin{equation*}
        \dfrac{dy}{ydx}= -\Delta +2 \left(1+\dfrac{1}{3}+\dfrac{1}{5}-\log 2\right)
    \end{equation*}
    hence
    
    \begin{equation*}
        \tan \varphi =y \left(2\left(1+\dfrac{1}{3}+\dfrac{1}{5}-\log 2\right)-\Delta\right)=1.1031566405 \cdot y.
    \end{equation*}
\end{itemize}
Since now

\begin{equation*}
    \dfrac{1}{2}\sqrt{\pi \cdot(2(1-\log 2)-\Delta)}=0.0323383973,
\end{equation*}
for these cases it will be:

\begin{equation*}
         \renewcommand{\arraystretch}{2.5}
\setlength{\arraycolsep}{0mm}
\begin{array}{lllrllr}
x =\dfrac{1}{2}, \quad & y &~=~& 0.8862269, \quad & \tan \varphi &~=~& 0.0323384, \\
x =\dfrac{3}{2}, \quad & y &~=~& 1.3293404, \quad & \tan \varphi &~=~& 0.9347345, \\
x =\dfrac{5}{2}, \quad & y &~=~& 3.3233509, \quad & \tan \varphi &~=~& 3.6661767, \\
x =\dfrac{7}{2}, \quad & y &~=~& 11.6317284, \quad & \tan \varphi &~=~& 16.1549694, \\
x =\dfrac{9}{2}, \quad & y &~=~& 52.3427777, \quad & \tan \varphi &~=~& 84.3290907, \\
                       &   &   & \text{etc.}
\end{array}
\end{equation*}

\paragraph*{16.}

Before I proceed, I observe, if for any abscissa it was

\begin{equation*}
    x=p, \quad y=q, \quad \tan \varphi =r,
\end{equation*}
that then for the following abscissa it will be

\begin{equation*}
    x=p+1, \quad y=q(p+1) \quad \text{and} \quad \tan \varphi =r(p+1)+q,
\end{equation*}
but for the preceding one

\begin{equation*}
    x=p-1, \quad y=\frac{q}{p} \quad \text{and} \quad \tan \varphi =\dfrac{r}{p}-\dfrac{q}{pp},
\end{equation*}
whence we can easily continue the above values backwards:

\begin{equation*}
            \renewcommand{\arraystretch}{2.5}
\setlength{\arraycolsep}{0mm}
\begin{array}{lllcllllllll}
x &~=~& & \dfrac{1}{2}, \quad &y &~=~&  &0.8862269, \quad & \tan \varphi &~=~& &0.0323384, \\
x &~=~&~-~ & \dfrac{1}{2}, \quad &y &~=~&  &1.7724538, \quad & \tan \varphi &~=~&~-~ &3.4802308, \\
x &~=~&~-~ & \dfrac{3}{2}, \quad &y &~=~&~-~ &3.5449077, \quad & \tan \varphi &~=~&~-~ &0.1293538, \\
x &~=~&~-~ & \dfrac{5}{2}, \quad &y &~=~& ~+~ &2.3632718, \quad & \tan \varphi &~=~&~+~ &1.6617504, \\
x &~=~&~-~ & \dfrac{7}{2}, \quad &y &~=~& ~-~ &0.9453087, \quad & \tan \varphi &~=~&~-~ &1.0428236, \\
x &~=~&~-~ & \dfrac{9}{2}, \quad &y &~=~& ~+~ &0.2700882, \quad & \tan \varphi &~=~&~+~ &0.3751176, \\
x &~=~&~-~ & \dfrac{11}{2}, \quad &y &~=~& ~-~ &0.0600196, \quad & \tan \varphi &~=~&~-~ &0.0966971, \\
x &~=~&~-~ & \dfrac{13}{2}, \quad &y &~=~& ~+~ &0.0109126, \quad & \tan \varphi &~=~&~+~ &0.0195654, \\
&  &  & & & & &\text{etc.}
\end{array}
\end{equation*}

\paragraph*{17.}

The same differential equation serves for finding  the point $\mu$ of the curve, where the ordinate is the smallest or the tangent is parallel to the axis. Therefore, having put $\frac{dy}{dx}=0$, the corresponding abscissa $x$ must be found from this equation:

\begin{equation*}
    \Delta = \dfrac{x}{1+x}+\dfrac{x}{2(2+x)}+\dfrac{x}{3(3+x)}+\dfrac{x}{4(4+x)}+\dfrac{x}{5(5+x)}+\text{etc.},
\end{equation*}
which is expanded into this one:

\begin{equation*}
      \renewcommand{\arraystretch}{2.5}
\setlength{\arraycolsep}{0mm}
\begin{array}{rll}
     \Delta = ~+~ & x &\left(1+\dfrac{1}{2^2}+\dfrac{1}{3^2}+\dfrac{1}{4^2}+\text{etc.}\right) \\
     ~-~ & x^2 &\left(1+\dfrac{1}{2^3}+\dfrac{1}{3^3}+\dfrac{1}{4^3}+\text{etc.}\right) \\
      ~+~ & x^3 &\left(1+\dfrac{1}{2^4}+\dfrac{1}{3^4}+\dfrac{1}{4^4}+\text{etc.}\right) \\
       ~-~ & x^4 &\left(1+\dfrac{1}{2^5}+\dfrac{1}{3^5}+\dfrac{1}{4^5}+\text{etc.}\right) \\
        & & \text{etc.}
\end{array}
\end{equation*}
But having substituted the proximate sums of these series it will be

\begin{equation*}
          \renewcommand{\arraystretch}{1.5}
\setlength{\arraycolsep}{0mm}
\begin{array}{llllllll}
     0 &~=~& ~+~& 0.5772156649~& &~-~&1.6449340668 ~& x \\
      & & ~+~& 1.2020569032~&x^2 &~-~&1.0823232337 ~& x^3 \\
      & & ~+~& 1.0369277551~&x^4 &~-~&1.0173430620 ~& x^5 \\
     & & ~+~& 1.0083492774~&x^6 &~-~&1.0040773562 ~& x^7 \\
& & ~+~& 1.0020083928~&x^8 &~-~&1.0009945751 ~& x^{9} \\
& & ~+~& 1.0004941886~&x^{10} &~-~&1.0002460866 ~& x^{11} \\
& & ~+~& 1.0001227133~&x^{12} &~-~&1.0000612481 ~& x^{13} \\
& & ~+~& 1.0000305882~&x^{14} &~-~&1.0000152823 ~& x^{15} \\
& &    &              &\text{etc.}
\end{array}
\end{equation*}
But if the first two fractions are kept, the following a lot more convergent series emerges

\begin{equation*}
          \renewcommand{\arraystretch}{1.5}
\setlength{\arraycolsep}{0mm}
\begin{array}{llllllll}
0 &~=~& ~+~& 0.5772156649 &~-~& \frac{x}{1+x}-\frac{x}{2(2+x)} \\
  &   & ~+~& 0.0770569032x^2    &~-~& 0.3949340668x \\
  &   & ~+~& 0.0056777551x^4    &~-~& 0.0198232337x^3 \\
  &   & ~+~& 0.0005367774x^6    &~-~& 0.0017180620x^5 \\
  &   & ~+~& 0.0000552678x^8    &~-~& 0.0001711062x^7 \\
  &   & ~+~& 0.0000059074x^{10} &~-~& 0.0000180126x^9 \\
  &   & ~+~& 0.0000006430x^{12} &~-~& 0.0000019460x^{11} \\
  &   & ~+~& 0.0000000706x^{14} &~-~& 0.0000002130x^{13} \\
  &   & ~+~& 0.0000000078x^{16} &~-~& 0.0000000235x^{15} \\
  &   &    & \text{etc.}
\end{array}
\end{equation*}
Hence one finds approximately $x=\frac{1}{2}$, but this minimal ordinate will be defined more easily by means of the following question.

\subsection*{Third Question}

\textit{Given a point of the hypergeometric curve to investigate the nature of an infinitesimal portion of this curve around this point.}\\

\paragraph*{18.}

Therefore, for the given abscissa $x=p$ let the ordinate $y=q$ have been found; and now one has to find the ordinate, which  corresponds to the abscissa $p+\omega$ differing from that one by just a small amount. Therefore, since according to formula V

\begin{equation*}
            \renewcommand{\arraystretch}{2.5}
\setlength{\arraycolsep}{0mm}
\begin{array}{lllllllllllllllllll}
    \log q = -\Delta p &~+~& p &~+~ &\dfrac{1}{2}p &~+~ &\dfrac{1}{3}p &~+~ &\dfrac{1}{4}p &~+~ &\text{etc.}    \\
     & ~-~& \log (1+p) &~-~ & \log \left(1+\dfrac{1}{2}p\right)  &~-~ & \log \left(1+\dfrac{1}{3}p\right)  &~-~ & \log \left(1+\dfrac{1}{4}p\right) &~-~& \text{etc.},
\end{array}
\end{equation*}
if one writes $p+\omega$ instead of $p$ here, instead of $\log q$ the value of $\log (q+\psi)$ will result, by which the question will be resolved. And if we set $\log q =P$, writing $p +\omega$ instead of $p$, it is known to result

\begin{equation*}
    \log (q+\psi)= P+\dfrac{\omega dP}{1 dp}+\dfrac{\omega^2 ddP}{1 \cdot 2 dp^2}+\dfrac{\omega^3 d^3P}{1 \cdot 2 \cdot 3 dp^3}+\dfrac{\omega^4 d^4 P}{1 \cdot 2 \cdot 3 \cdot 4 dp^4}+\text{etc.}
\end{equation*}
But on the other hand, as we have seen:

\begin{equation*}
    \dfrac{dP}{dp}=-\Delta +\dfrac{p}{1+p}+\dfrac{p}{2(2+p)}+\dfrac{p}{3(3+p)}+\dfrac{p}{4(4+p)}+\text{etc.}
\end{equation*}
and hence further:

\begin{equation*}
               \renewcommand{\arraystretch}{2.5}
\setlength{\arraycolsep}{0mm}
\begin{array}{cllllllllllllllllllllllll}
 \dfrac{ddP}{1 \cdot dp^2}&~=~& & \dfrac{1}{(1+p)^2}&~+~& \dfrac{1}{(2+p)^2}&~+~& \dfrac{1}{(3+p)^2}&~+~& \dfrac{1}{(4+p)^2}&~+~& \text{etc.} \\
 \dfrac{d^3P}{1 \cdot 2 dp^3}&~=~&~-~ & \dfrac{1}{(1+p)^3}&~-~& \dfrac{1}{(2+p)^3}&~-~& \dfrac{1}{(3+p)^3}&~-~& \dfrac{1}{(4+p)^3}&~-~& \text{etc.} \\
 \dfrac{d^4P}{1 \cdot 2 \cdot 3 dp^4}&~=~& & \dfrac{1}{(1+p)^4}&~+~& \dfrac{1}{(2+p)^4}&~+~& \dfrac{1}{(3+p)^4}&~+~& \dfrac{1}{(4+p)^4}&~+~& \text{etc.} \\
   & & & & &\text{etc.}
\end{array}
\end{equation*}
whence, because of $P=\log q$, we conclude:

\begin{equation*}
       \renewcommand{\arraystretch}{2.5}
\setlength{\arraycolsep}{0mm}
\begin{array}{cccccccccccccccccccccccccc}
   \log \left(1+\dfrac{\psi}{q}\right)=- \Delta \omega &~+~&  &\omega &\bigg(\dfrac{p}{1+p}&~+~&\dfrac{p}{2(2+p)}&~+~&\dfrac{p}{3(3+p)}&~+~&\text{etc.}\bigg) \\
    &~+~& \dfrac{1}{2}  &\omega^2 &\bigg(\dfrac{1}{(1+p)^2}&~+~&\dfrac{1}{(2+p)^2}&~+~&\dfrac{1}{(3+p)^2}&~+~&\text{etc.}\bigg) \\
     &~-~& \dfrac{1}{3}  &\omega^3 &\bigg(\dfrac{1}{(1+p)^3}&~+~&\dfrac{1}{(2+p)^3}&~+~&\dfrac{1}{(3+p)^3}&~+~&\text{etc.}\bigg) \\
      &~+~& \dfrac{1}{4}  &\omega^4 &\bigg(\dfrac{1}{(1+p)^4}&~+~&\dfrac{1}{(2+p)^4}&~+~&\dfrac{1}{(3+p)^4}&~+~&\text{etc.}\bigg) \\
       &~-~& \dfrac{1}{5}  &\omega^2 &\bigg(\dfrac{1}{(1+p)^5}&~+~&\dfrac{1}{(2+p)^5}&~+~&\dfrac{1}{(3+p)^5}&~+~&\text{etc.}\bigg) \\
        &  &  &  & & &\text{etc.}
\end{array}
\end{equation*}

\paragraph*{19.}

Here  the coordinates $p$ and $q$ can be considered as constants, since  the letters $\omega$ and $\psi$ denote two new coordinates  taken from a given point of the curve and parallel to the first set; from their relation defined here the nature of the curve around that point is easily investigated. Thus, since we have already assigned innumerable points of the curve, hence the trace of each portion of the curve between two of those conjugated points can be defined approximately. First, from that differentiated equation, as before, the inclination $\varphi$ of the tangent to the axis is calculated and

\begin{equation*}
    \dfrac{d\psi}{d \omega}= \tan \varphi = q \left(-\Delta +\dfrac{p}{1+p}+\dfrac{p}{2(2+p)}+\dfrac{p}{3(3+p)}+\text{etc.}\right).
\end{equation*}
Further, if for the differential equation, for the sake of brevity, we set

\begin{equation*}
    d \psi = Ad\omega +B \omega d \omega +C \omega^2d \omega +\text{etc.},
\end{equation*}
the curvature radius at a given point of the curve will be

\begin{equation*}
    =\dfrac{(1+AA)^{\frac{3}{2}}}{B}= \dfrac{1}{B \cdot \cos^3 \varphi}
\end{equation*}
because of $A= \tan \varphi$. But on the other hand

\begin{equation*}
    B= \tan \varphi \left(-\Delta +\dfrac{p}{1+p}+\dfrac{p}{2(2+p)}+\dfrac{p}{3(3+p)}+\text{etc.}\right)
\end{equation*}
\begin{equation*}
    +q\left(\dfrac{1}{(1+p)^2}+\dfrac{1}{(2+p)^2}+\dfrac{1}{(3+p)^2}+\dfrac{1}{(4+p)^2}+\text{etc.}\right),
\end{equation*}
whence, if the curvature radius is set $=r$, it will be

\begin{equation*}
    \dfrac{1}{r}= \dfrac{\sin^2 \varphi \cos \varphi}{q}+q\left(\dfrac{1}{(1+p)^2}+\dfrac{1}{(2+p)^2}+\dfrac{1}{(3+p)^2}+\text{etc.}\right).
\end{equation*}

\paragraph*{20.}

But in order to extend the investigation of the direction and the curvature from the  principal point defined by the coordinates $p$ and $q$ to the points of the curve, for the sake of brevity, let us set

\begin{equation*}
       \renewcommand{\arraystretch}{2.5}
\setlength{\arraycolsep}{0mm}
\begin{array}{ccccccccccccccccccccccc}
     -\Delta ~+~ &\dfrac{p}{1+p}&~+~& \dfrac{p}{2(2+p)} &~+~& \dfrac{p}{3(3+p)} &~+~& \dfrac{p}{4(4+p)} &~+~& \text{etc.} &~=~& P, \\
       &\dfrac{1}{(1+p)^2}&~+~& \dfrac{1}{(2+p)^2} &~+~& \dfrac{1}{(3+p)^2} &~+~& \dfrac{1}{(4+p)^2} &~+~& \text{etc.} &~=~& Q, \\
           &\dfrac{1}{(1+p)^3}&~+~& \dfrac{1}{(2+p)^3} &~+~& \dfrac{1}{(3+p)^3} &~+~& \dfrac{1}{(4+p)^3} &~+~& \text{etc.} &~=~& R, \\
               &\dfrac{1}{(1+p)^4}&~+~& \dfrac{1}{(2+p)^4} &~+~& \dfrac{1}{(3+p)^4} &~+~& \dfrac{1}{(4+p)^4} &~+~& \text{etc.} &~=~& S \\
               &  &   &  &  & \text{etc.},
\end{array}
\end{equation*}
that

\begin{equation*}
    \log \left(1+\dfrac{\psi}{q}\right)= P \omega +\dfrac{1}{2}Q\omega^2 -\dfrac{1}{3}R \omega^3+\dfrac{1}{4}S\omega^4-\dfrac{1}{5}T \omega^5 +\text{etc.}
\end{equation*}
Hence now differentiating we find: 

\begin{equation*}
    \dfrac{d \psi}{d \omega}= (q+ \psi) (P+ Q \omega -R \omega^2 +S \omega^3 -T \omega^4 +\text{etc.})
\end{equation*}
and differentiating further

\begin{equation*}
    \dfrac{dd\psi}{d\omega^2}=(q+ \psi)(P+ Q \omega -R \omega^2 +S \omega^3 -T \omega^4 +\text{etc.})^2
\end{equation*}
\begin{equation*}
    +(q+\psi)(Q-2 R \omega +3 S\omega^2- 4 T \omega +\text{etc.})
\end{equation*}
\begin{equation*}
    \dfrac{d^3 \psi}{d \omega^3}=3(q+\psi)(Q-2 R \omega +3 S\omega^2 -4 T \omega^3+\text{etc.})(P+Q\omega -R \omega^2 +S \omega^3 -\text{etc.})
\end{equation*}
\begin{equation*}
    +(q+\psi)(P+Q \omega -R \omega^2+S\omega^3-T \omega^4+\text{etc.})^3
\end{equation*}
\begin{equation*}
    -(q+\psi)(2R-6S\omega +12T \omega^2- \text{etc.}).
\end{equation*}
Having covered these calculations, for the point of the curve corresponding to the abscissa $x=p + \omega$ and $y=q+\psi$ the direction of the tangent will be

\begin{equation*}
    \tan \varphi = \dfrac{d \psi}{d \omega}=(q+\psi)(P+Q\omega -R\omega^2+S\omega^3-T \omega^4+\text{etc.}).
\end{equation*}
But then, having put the curvature radius $=r$, we know that it will be:

\begin{equation*}
    r= \left(1+\dfrac{d\psi^2}{d \omega^2}\right)^{\frac{3}{2}}:\dfrac{dd\psi}{d \omega^2}=1: \dfrac{dd \psi}{d \omega^2}\cos^3 \varphi
\end{equation*}
or

\begin{equation*}
    \dfrac{1}{r}= \dfrac{dd\psi}{d \omega^2}\cos^3 \varphi,
\end{equation*}
whence we find for the variability of the curvature:

\begin{equation*}
    -\dfrac{dr}{rrd \omega}= \dfrac{d^3 \psi}{d \omega^3}\cos^3 \varphi - \dfrac{3dd\psi}{d \omega^2}\cdot \dfrac{d \varphi}{d \omega}\sin \varphi \cos^2 \varphi.
\end{equation*}
But on the other hand

\begin{equation*}
    \dfrac{d\varphi}{\cos^2 \varphi}= \dfrac{dd \psi}{d \omega},
\end{equation*}
whence:

\begin{equation*}
    -\dfrac{dr}{rrd\omega}=\dfrac{d^3 \psi}{d \omega^3}\cos^3 \varphi -3 \left(\dfrac{dd\psi}{d\omega^2}\right)^2 \sin \varphi \cos^4 \varphi.
\end{equation*}

\subsection*{Fourth Question}

\textit{To investigate the nature of the hypergeometric curve around its lowest point $\mu$ where the ordinate is the smallest.}

\paragraph*{21.}

Since this point is not far away from the point corresponding to the abscissa $=\frac{1}{2}$ and the ordinate $=\frac{1}{2}\sqrt{\pi}$, let us set $p=\frac{1}{2}$ so that $q=\frac{1}{2}\sqrt{\pi}$, and hence first let us find the values of the letters $P$, $Q$, $R$, $S$ etc., which will result as:

\begin{equation*}
          \renewcommand{\arraystretch}{2.5}
\setlength{\arraycolsep}{0mm}
\begin{array}{lllclcccccccccccccccccccccccc}
     P &~=~&~-~& \Delta &~+~& \dfrac{1}{3}&~+~& \dfrac{1}{2\cdot 5}&~+~& \dfrac{1}{3 \cdot 7}&~+~& \text{etc.} &~=~& 2(1-\log 2)-\Delta &~=~& 0.03648997397857 \\
       Q &~=~& & \dfrac{4}{3^2} &~+~& \dfrac{4}{5^2}&~+~& \dfrac{4}{7^2}&~+~& \dfrac{4}{9^2}&~+~& \text{etc.} & &  &~=~& 0.93480220054468 \\
       R &~=~& & \dfrac{8}{3^3} &~+~& \dfrac{8}{5^3}&~+~& \dfrac{8}{7^3}&~+~& \dfrac{8}{9^3}&~+~& \text{etc.} & &  &~=~& 0.41439832211716 \\
         S &~=~& & \dfrac{16}{3^4} &~+~& \dfrac{16}{5^4}&~+~& \dfrac{16}{7^4}&~+~& \dfrac{16}{9^4}&~+~& \text{etc.} & &  &~=~& 0.0.23484850566707 \\
  T &~=~& & \dfrac{32}{3^5} &~+~& \dfrac{32}{5^5}&~+~& \dfrac{32}{7^5}&~+~& \dfrac{32}{9^5}&~+~& \text{etc.} & &  &~=~& 0.144760040831276 \\
    V &~=~& & \dfrac{64}{3^6} &~+~& \dfrac{64}{5^6}&~+~& \dfrac{64}{7^6}&~+~& \dfrac{64}{9^6}&~+~& \text{etc.} & &  &~=~& 0.09261290502029 \\
      W &~=~& & \dfrac{128}{3^7} &~+~& \dfrac{128}{5^7}&~+~& \dfrac{128}{7^7}&~+~& \dfrac{128}{9^7}&~+~& \text{etc.} & &  &~=~& 0.06035822809843 
\end{array}
\end{equation*}
Further,
\begin{equation*}
    q=\dfrac{1}{2}\sqrt{\pi}= 0.88622692545274.
\end{equation*}

\paragraph*{22.}

Hence let us especially define the point $\mu$, where the ordinate is the smallest  simple approximations shows it to correspond to the abscissa $x=0.4616$,  having set

\begin{equation*}
    p+ \omega = \dfrac{1}{2}+\omega = 0.4616,
\end{equation*}
one   finds approximately

\begin{equation*}
    \omega = -0.0383,
\end{equation*}
which value must be investigated more accurately from the equation $\frac{d \psi}{d\omega}=0$ or

\begin{equation*}
    P+Q\omega -R\omega^2 +S\omega^3 -T\omega^4 +\text{etc.}=0.
\end{equation*}
Therefore, since  approximately $\omega=-\frac{1}{26}$, we set $\omega =-\frac{1}{26}-z$, and after the substitution it has to be

\begin{equation*}
        \renewcommand{\arraystretch}{1.5}
\setlength{\arraycolsep}{0mm}
\begin{array}{rrrrrrrrrrrrr}
    & ~+~&0.03595393079018 &~+~& 0.934802200z \\
    & ~+~&0.00061301526940 &~+~& 0.031876794z &~+~& 0.414398zz \\
    & ~+~&0.00001336188585 &~+~& 0.001042227z &~+~& 0.027097zz \\
    & ~+~&0.00000031677900 &~+~& 0.000032945z &~+~& 0.001285zz \\
    & ~+~&0.00000000779479 &~+~& 0.000001013z &~+~& 0.000053zz \\
    & ~+~&0.00000000019538 &~+~& 0.000000030z &~+~& 0.000002zz \\
    & ~+~&0.00000000000496 &~+~&           2z &~+~& \\ 
    & ~+~&              13 &~+~&  \\ \cline{2-7}
    &    &0.03658063271970 &~+~& 0.967755211z &~+~& 0.442835zz \\
    &    &0.03648997397857 \\ \cline{2-7}
0   &~=~ &0.00009065874113 &~+~& 0.967755211z &~+~& 0.442835zz 
\end{array}
\end{equation*}
whence one finds

\begin{equation*}
    z=-0.00009368323
\end{equation*}
and hence

\begin{equation*}
    \omega = - 0.03836785523.
\end{equation*}
Therefore, the smallest ordinate $m \mu$ corresponds to the abscissa

\begin{equation*}
    Om = 0.46163214477.
\end{equation*}
For the ordinate $m \mu = q+ \psi$ on the other hand one has to expand the equation

\begin{equation*}
    \log \left(1+\dfrac{\psi}{q}\right)= P \omega +\dfrac{1}{2}Q\omega^2 -\dfrac{1}{3}R \omega^3 +\dfrac{1}{4}S\omega^4 -\dfrac{1}{5}T \omega^5 +\text{etc.},
\end{equation*}
from which one concludes

\begin{equation*}
    \log \left(1+\dfrac{\psi}{q}\right)=- 0.000704053
\end{equation*}
and further

\begin{equation*}
    1+\dfrac{\psi}{q}=1-0.000703805,
\end{equation*}
so that the smallest ordinate becomes

\begin{equation*}
    m \mu = q+ \psi = 0.8856031945.
\end{equation*}

\paragraph*{23.}

Now let us in general differentiate to define the value of $\psi$ from the logarithmic equation, and, having done the calculation, we will obtain:

\begin{equation*}
         \renewcommand{\arraystretch}{1.5}
\setlength{\arraycolsep}{0mm}
\begin{array}{lllllll}
     \dfrac{\psi}{q} = &~+~& 0.0364899740 \omega &~+~& 0.468066860 \omega^2 \\
                       &~-~& 0.121069221 \omega^3&~+~& 0.16321479 \omega^4 \\
                       &~-~& 0.09360753 \omega^5 &~+~&\text{etc.}, 
\end{array}
\end{equation*}
which terms suffice, if the value of $\omega$ is very small. But, for the sake of brevity, let us set

\begin{equation*}
    \dfrac{\psi}{q}= \mathfrak{P}\omega +\mathfrak{Q}\omega^2 - \mathfrak{R}\omega^3 +\mathfrak{S}\omega^4 -\mathfrak{T}\omega^5
\end{equation*}
 such that

\begin{equation*}
         \renewcommand{\arraystretch}{1.5}
\setlength{\arraycolsep}{0mm}
\begin{array}{lllllll}
     \mathfrak{P}&~=~& 0.0364899740, \quad &\mathfrak{Q}&~=~&0.468066860, \\
     \mathfrak{R}&~=~& 0.121069221, \quad  &\mathfrak{S}&~=~&0.16321479, \\
     \mathfrak{T}&~=~& 0.09360753,
\end{array}
\end{equation*}
and hence we will have:

\begin{equation*}
         \renewcommand{\arraystretch}{2.5}
\setlength{\arraycolsep}{0mm}
\begin{array}{cll}
     \dfrac{d\psi}{d \omega} &~=~& q \left(\mathfrak{P}+2 \mathfrak{Q}\omega -3 \mathfrak{R}\omega^2 +4 \mathfrak{S}\omega^3 -5 \mathfrak{T}\omega^4\right), \\
     \dfrac{dd\psi}{d\omega^2} &~=~& q(2 \mathfrak{Q}-6 \mathfrak{R}\omega +12 \mathfrak{S}\omega^2 -20 \mathfrak{T}\omega^3).
\end{array}
\end{equation*}
If we now want to find the radius of curvature at the lowest point $\mu$, where

\begin{equation*}
    \omega =-0.03836785523,
\end{equation*}
since there we have $\frac{d \psi}{d \omega}=0$, that radius of curvature will be $=\frac{d \omega^2}{dd\psi}$. Set the curvature radius in this point $=r$, and since

\begin{equation*}
    \dfrac{1}{r}=2q( \mathfrak{Q}-3 \mathfrak{R}\omega +6 \mathfrak{S}\omega^2 -10 \mathfrak{T}\omega^3)=0.9669949,
\end{equation*}
for the point $\mu$ the curvature radius results as

\begin{equation*}
    r= 1.166893.
\end{equation*}

\paragraph*{24.}

I investigated these determinations of the lowest point $\mu$ of the curve with all eagerness such that it cannot without any reason be conjectured, as this point has an extraordinary property,  that  the numbers exhibiting its nature  contain in this way a certain elegance, and  if they cannot be expressed sufficiently simply in terms of a rational or irrational number, that they are at least to be referred to  a certain simpler  kind of transcendental quantities. But against the expectation it happened that such a criterion for elegance appears neither in the abscissa

\begin{equation*}
    Om =0.46163214477
\end{equation*}
nor in the ordinate

\begin{equation*}
    m \mu = 0.8856031945
\end{equation*}
nor in the curvature radius at this point

\begin{equation*}
    =1.166893;
\end{equation*}
since no affinity to simpler rational numbers or irrational numbers or to the quadrature of the circle or to logarithmic or exponential numbers is detected. Since, if the abscissa $Om$ is considered as a logarithm, the number  corresponding to it could seem to promise several things,  I sought after this number and found

\begin{equation*}
    =1.586616,
\end{equation*}
in which no affinity to any known quantities is recognised.

\paragraph*{25.}

Before I end this speculation, it will helpful to have observed that the formula $1\cdot 2 \cdot 3 \cdots x$ can also be expressed indefinitely in terms of the following series

\begin{equation*}
    x^x-x(x-1)^x+\dfrac{x(x-1)}{1 \cdot 2}(x-2)^x-\dfrac{x(x-1)(x-2)}{1 \cdot 2 \cdot 3}(x-3)^x+\text{etc.},
\end{equation*}
which, as often as $x$ is a positive integer number, immediately gives that product as $1 \cdot 2 \cdot 3 \cdots x$. This is indeed also achieved by this further extending expression:

\begin{equation*}
    a^x-x(a-1)^x+\dfrac{x(x-1)}{1 \cdot 2}(a-2)^x-\dfrac{x(x-1)(x-2)}{1 \cdot 2 \cdot 3}(a-3)^x+\text{etc.},
\end{equation*}
for, if for $x$ one successively substitutes the numbers $1, 2, 3$ etc., it will be as follows:
\begin{equation*}
       \renewcommand{\arraystretch}{1.5}
\setlength{\arraycolsep}{0mm}
\begin{array}{l}
     a^0 =1 \\
     a^1 -(a-1)^1=1 \\
     a^2-2(a-1)^2+(a-2)^2=1 \cdot 2 \\
     a^3-3(a-1)^3+3(a-2)^3-(a-3)^3=1 \cdot 2 \cdot 3 \\
     a^4-4(a-1)^4+6(a-2)^4-4(a-3)^4+(a-4)^4=1 \cdot 2 \cdot 3 \cdot 4 \\
     a^5-5(a-1)^5+10(a-2)^5-10(a-3)^5+5(a-4)^5-(a-5)^5= 1 \cdot 2 \cdot 3 \cdot 4 \cdot 5 \\
     \text{etc.}
\end{array}
\end{equation*}

\paragraph*{26.}

These are certainly obvious from the results demonstrated about the difference of each order of algebraic progressions, but nevertheless from the nature of these series the truth is not easily uncovered; thus,  the following proof seems to be in order. Since for smaller exponents $x$ the matter is obvious, I reason as follows, i.e. that, having conceded the truth for the case $x=n$, I will show that it also follows for the case $x=n+1$.\\
Therefore, let

\begin{equation*}
    \text{I.} \quad a^n-n(a-1)^n+\dfrac{n(n-1)}{1\cdot 2}(a-2)^2-\text{etc.} =N=1 \cdot 2 \cdot 3 \cdots n,
\end{equation*}
and since the number $N$ does not depend on $a$, it will also be:

\begin{equation*}
    \text{II.} \quad  (a-1)^n-n(a-2)^n+\dfrac{n(n-1)}{1\cdot 2}(a-3)^2-\text{etc.} =N,
\end{equation*}
which subtracted from the first leaves:

\begin{equation*}
    \text{III.} \quad a^n-\dfrac{(n+1)}{1}(a-1)^n+\dfrac{(n+1)n}{1 \cdot 2}(a-2)^n-\dfrac{(n+1)n(n-1)}{1 \cdot 2 \cdot 3}(a-3)^n+\text{etc.}=0;
\end{equation*}
multiply this by $a$ and it results

\begin{equation*}
    \text{IV.} \quad a^{n+1}-\dfrac{(n+1)}{1}a(a-1)^n+\dfrac{(n+1)n}{1\cdot 2}a(a-2)^n-\dfrac{(n+1)n(n-1)}{1 \cdot 2 \cdot 3}a(a-3)^n +\text{etc.}=0;
\end{equation*}
to this one add equation II multiplied by $n+1$, i.e.:

\begin{equation*}
    \text{V.} \quad  +(n+1)1(a-1)^n-\dfrac{(n+1)n}{1 \cdot 2}2(a-2)^n+\dfrac{(n+1)n(n-1)}{1 \cdot 2 \cdot 3}3(a-3)^n-\text{etc.}=(n+1)N
\end{equation*}
and the aggregate IV$+$V will give

\begin{equation*}
    \text{VI.}~~ a^{n+1}-\dfrac{(n+1)}{1}(a-1)^{n+1}+\dfrac{(n+1)n}{1 \cdot 2}(a-2)^{n+1}-\dfrac{(n+1)n(n-1)}{1 \cdot 2 \cdot 3}(a-3)^{n+1}+\text{etc.}=(n+1)N,
\end{equation*}
where, because of $N=1 \cdot 2 \cdot 3 \cdots N$, it will be $(n+1)N=1 \cdot 2 \cdot 3 \cdots (n+1)$. Therefore, it is proved  that, if our proposition

\begin{equation*}
    a^x-x(a-1)^x+\dfrac{x(x-1)}{1 \cdot 2}(a-2)^x -\dfrac{x(x-1)(x-2)}{1 \cdot 2 \cdot 3}(a-3)^x+\text{etc.}=1 \cdot 2 \cdot 3 \cdots x
\end{equation*}
was true in the case $x=n$, it will also be true in the case $x=n+1$. Therefore, since it is obviously true in the case $x=1$, it follows that it is also true for all positive integer numbers assumed for $x$.

\paragraph*{27.}

But although this expression is sufficiently elegant and worth one's complete attention, it is nevertheless less useful for our task, in which the hypergeometric curve it propounded, since for the cases, in which $x$ is a fractional number, this series not only runs to infinity but it also, if the denominator is an even number, contains imaginary terms such that its value cannot even be calculated using approximations. Therefore, having set $x=\frac{1}{2}$, this infinite series results:

\begin{equation*}
    \sqrt{a}-\dfrac{1}{2}\sqrt{a-1}-\dfrac{1 \cdot 1}{2 \cdot 4}\sqrt{a-2}-\dfrac{1 \cdot 1 \cdot 3}{2 \cdot 4 \cdot 6}\sqrt{a-3}-\dfrac{1 \cdot 1 \cdot 3 \cdot 5}{2 \cdot 4 \cdot 6 \cdot 8}\sqrt{a-4}- \text{etc.},
\end{equation*}
whose value can hardly be shown by anyone to be $=\frac{1}{2}\sqrt{\pi}$. In like manner, taking $x=-\frac{1}{2}$,  we already know from the above results that

\begin{equation*}
    \sqrt{\pi}= \dfrac{1}{\sqrt{a}}+\dfrac{1}{2\sqrt{a-1}}+\dfrac{1 \cdot 3}{2 \cdot 4 \sqrt{a-2}}+\dfrac{1 \cdot 3 \cdot 5}{2 \cdot 4 \cdot 6 \sqrt{a-3}}+\text{etc.}
\end{equation*}
Nevertheless, a further investigation of this series is left to the mathematics, especially if it is extended and represented in this form:

\begin{equation*}
    s=x^n-m(x-1)^n+\dfrac{m(m-1)}{1 \cdot 2}(x-2)^n-\dfrac{m(m-1)(m-2)}{1 \cdot 2 \cdot 3}(x-3)^n+\text{etc.};
\end{equation*}
for, without much effort one soon detects extraordinary properties, whose expansion seems worth our complete attention. I will now  present all extraordinary phenomena I was able to observe about it\footnote{The remainder of this paper is about the last series, which Euler denoted by $s$. It can be read without the first part and can be considered as a separate paper. This might also explain why Euler started the paragraphs from I. again.}.

\section*{Observations on the Series\\[2mm]
$s=x^n-m(x-1)^n+\dfrac{m(m-1)}{1 \cdot 2}(x-2)^n-\dfrac{m(m-1)(m-2)}{1 \cdot 2 \cdot 3}(x-3)^n+\text{etc.}$}

\paragraph*{I.}  In the preceding, I already demonstrated, if the exponent $n$ was $=m$, that the sum of this series will be

\begin{equation*}
    s=1 \cdot 2 \cdot 3 \cdots m
\end{equation*}
such that in this case it does not depend on the number $x$. But  I first conclude, if $n=m-1$, that  $s=0$. For, having taken $n=m$,

\begin{equation*}
    \saturn \cdots 1 \cdot 2 \cdot 3 \cdots m = x^m-m(x-1)^m+\dfrac{m(m-1)}{1\cdot 2}(x-2)^m-\text{etc.}
\end{equation*}
and, writing $x-1$ instead of $x$ and $m-1$ instead of $m$, in like manner, we get:

\begin{equation*}
    \jupiter \cdots 1 \cdot 2 \cdot 3 \cdots (m-1) =(x-1)^{m-1}-(m-1)(x-2)^{m-1}+\dfrac{(m-1)(m-2)}{1\cdot 2}(x-3)^{m-1}-\text{etc.}
\end{equation*}
Now represent that equation $[\saturn]$ in this way:

\begin{equation*}
        \renewcommand{\arraystretch}{2.5}
\setlength{\arraycolsep}{0mm}
\begin{array}{llrrrrrrrrrrr}
     \mars \cdots 1 \cdot 2 \cdot 3 \cdots m &~=~& x \cdot x^{m-1} &~-~& mx(x-1)^{m-1}&~+~& \dfrac{m(m-1)}{1\cdot 2}x(x-2)^{m-1}&~-~&\text{etc.} \\
                                             &  &                  &~+~& (x-1)^{m-1}&~-~& \dfrac{m(m-1)}{1}(x-2)^{m-1} &~+~&\text{etc.};
\end{array}
\end{equation*}
but the equation $\jupiter$ multiplied by $m$ gives:

\begin{equation*}
    \astrosun \cdots 1 \cdot 2 \cdot 3 \cdots m = m(x-1)^{m-1}-\dfrac{m(m-1)}{1}(x-2)^{m-1}+\dfrac{m(m-1)(m-2)}{1\cdot 2}(x-3)^{m-1}-\text{etc.},
\end{equation*}
which subtracted from $\mars$ and divided by $x$ yields:

\begin{equation*}
    \venus \cdots 0= x^{m-1}-\dfrac{m}{1}(x-1)^{m-1}+\dfrac{m(m-1)}{1\cdot 2}(x-2)^{m-1}-\text{etc.},
\end{equation*}
which is the propounded equation for the case $n=m-1$, whose value thus is $=0$.\\

\paragraph*{II.}
In like manner, it is shown that the sum $s$ of the propounded series also vanishes in the case $n=m-2$. For, represent the series $\venus$ in this way:

\begin{equation*}
         \renewcommand{\arraystretch}{2.5}
\setlength{\arraycolsep}{0mm}
\begin{array}{llrrrrrrr}
    \mercury \cdots 0 &~=~& x \cdot x^{m-2} &~-~& \dfrac{m}{1}x(x-1)^{m-2}&~+~& \dfrac{m(m-1)}{1\cdot 2}x(x-2)^{m-2}&~-~&\text{etc.} \\
     & & &~+~& m(x-1)^{m-2} &~-~& \dfrac{m(m-1)}{1\cdot 2}(x-2)^{m-2} &~+~& \text{etc.}
\end{array}
\end{equation*}
and if in the same series $\venus$ one writes $x-1$ instead of $x$ and $m-1$ instead of $m$, but the whole series is multiplied by $m$, it becomes

\begin{equation*}
    \rightmoon \cdots 0 = m(x-1)^{m-2} -\dfrac{m(m-1)}{1}(x-2)^{m-2}+\text{etc.}
\end{equation*}
Having subtracted this one from previous one, divide the remainder by $x$ and it will result:

\begin{equation*}
    0= x^{m-2}-\dfrac{m}{1}(x-1)^{m-2}+\dfrac{m(m-1)}{1 \cdot  2}(x-2)^{m-2}-\text{etc.}
\end{equation*}
And so the sum of the propounded series also vanishes in the case $n=m-2$, and, in like manner, it can be shown that it also vanishes in the cases $n=m-3$, $n=m-4$ etc. and in general $n=m-i$, where $i$ is an arbitrary positive integer number.   Keep in mind that the sum of the series $s$ is $=1 \cdot 2 \cdot 3 \cdots m$ in the case $n=m$, but in the cases, in which the exponent $n$ is smaller than the number $m$, the sum vanishes, if the numbers $m$ and $n$ are integers or at least $n-m$ is a positive integer number, of course. \\

\paragraph*{III.} Therefore, in order to investigate the nature of the remaining cases, let us expand each term of our series and arrange them according to the powers of $x$, having done which we will obtain:

\begin{equation*}
        \renewcommand{\arraystretch}{2.5}
\setlength{\arraycolsep}{0mm}
\begin{array}{lllll}
     s&~=~& x^n \left(1-m+\dfrac{m(m-1)}{1\cdot 2}-\dfrac{m(m-1)(m-2)}{1 \cdot 2 \cdot 3}+\text{etc.}\right) \\
      &~+~& nx^{n-1}\left(m-\dfrac{2m(m-1)}{1\cdot 2}+\dfrac{3m(m-1)(m-2)}{1\cdot 2 \cdot 3}-\text{etc.}\right) \\
      &~-~& \dfrac{n(n-1)}{1\cdot 2}x^{n-2}\left(m-\dfrac{4m(m-1)}{1\cdot 2}+\dfrac{9m(m-1)(m-2)}{1\cdot 2 \cdot 3}-\text{etc.}\right) \\
      &~+~& \dfrac{n(n-1)(n-2)}{1\cdot 2 \cdot 3}x^{n-3}\left(m-\dfrac{8m(m-1)}{1\cdot 2}+\dfrac{27m(m-1)(m-2)}{1\cdot 2 \cdot 3}-\text{etc.}\right)\\
      & &\text{etc.},
\end{array}
\end{equation*}
the sums of which series we will find as follows; first, exhibit them a bit more generally, and since its sum is known:

\begin{equation*}
    1-mu+\dfrac{m(m-1)}{1\cdot 2}u^2-\dfrac{m(m-1)(m-2)}{1\cdot 2 \cdot 3}u^3+\text{etc.}= (1-u)^m,
\end{equation*}
let us differentiate it continuously and always substitute $u$ for $du$ again, and, having changed the signs, it will be:
 
 \begin{equation*}
         \renewcommand{\arraystretch}{2.5}
\setlength{\arraycolsep}{0mm}
\begin{array}{llcllll}
     mu &~-~& \dfrac{2m(m-1)}{1\cdot 2}&u^2&~+ \dfrac{3m(m-1)(m-2)}{1\cdot 2 \cdot 3}u^3 -\text{etc.} = mu(1-u)^{m-1} \\
     mu &~-~& \dfrac{2m(m-1)}{1\cdot 2}&u^2&~ +\text{etc.}= mu(1-u)^{m-1} -m(m-1)u^2(1-u)^{m-2} \\
     mu &~-~& \dfrac{2^3m(m-1)}{1\cdot 2}&u^2&~ +\text{etc.} =mu(1-u)^{m-1}-3m(m-1)u^2 (1-u)^{m-2} \\
        &   &                            &   &~+m(m-1)(m-2)u^3(1-u)^{m-3} \\
     mu&~-~&\dfrac{2^4m(m-1)}{1\cdot 2}&u^2 &~+\text{etc.} = mu(1-u)^{m-1}-7m(m-1)uu(1-u)^{m-2} \\
      & & &  &~+ 6m(m-1)(m-2)u^3(1-u)^{m-3} \\
      & & &  &~-m(m-1)(m-2)(m-3)u^4(1-u)^{m-4}\\
      & & &  &\text{etc.}
\end{array}
 \end{equation*}
 Therefore, here one now has to write $u=1$, after which all terms in each order vanish except for those, where the exponent of $1-u$ becomes $=0$.
 
 \paragraph*{IV.}
 
Now successively attribute the values $1$, $2$, $3$, $4$, $5$ etc. to $m$ and, for the sake of brevity, write $\left(\frac{n-i}{i+1}\right)$ instead of the general coefficient

\begin{equation*}
    \dfrac{n(n-1)(n-2)\cdots (n-i)}{1 \cdot 2 \cdot 3 \cdots (i+1)}
\end{equation*}
 and we obtain the following values:

\begin{center}
\begin{scriptsize}
\begin{equation*}
       \renewcommand{\arraystretch}{2.5}
\setlength{\arraycolsep}{0mm}
\begin{array}{c|crrrrrrrrrrrrrrrrrrrrr}
   \text{if}  & ~~ \text{it will be} \\
    ~ m=1 ~& \dfrac{s}{1} &~=~& \left(\dfrac{n}{1}\right)x^{n-1} &~-~& \left(\dfrac{n-1}{2}\right)x^{n-2} &~+~& \left(\dfrac{n-2}{3}\right)x^{n-3}&~-~& \left(\dfrac{n-3}{4}\right)x^{n-4}&~+~& \left(\dfrac{n-4}{5}\right)x^{n-5}&~-~&\text{etc.} \\
       ~ m=2 ~& \dfrac{s}{1 \cdot 2} &~=~& \left(\dfrac{n-1}{2}\right)x^{n-2} &~-~& 3\left(\dfrac{n-2}{3}\right)x^{n-3} &~+~& 7\left(\dfrac{n-3}{4}\right)x^{n-4}&~-~& 15\left(\dfrac{n-4}{5}\right)x^{n-5}&~+~& 31\left(\dfrac{n-5}{6}\right)x^{n-6}&~-~&\text{etc.} \\
         ~ m=3 ~& \dfrac{s}{1 \cdot 2 \cdot 3} &~=~& \left(\dfrac{n-2}{3}\right)x^{n-3} &~-~& 6\left(\dfrac{n-3}{4}\right)x^{n-4} &~+~& 25\left(\dfrac{n-4}{5}\right)x^{n-5}&~-~& 90\left(\dfrac{n-5}{6}\right)x^{n-6}&~+~& 301\left(\dfrac{n-6}{7}\right)x^{n-7}&~-~&\text{etc.} \\
 ~ m=4 ~& \dfrac{s}{1 \cdot 2 \cdots 4} &~=~& \left(\dfrac{n-3}{4}\right)x^{n-4} &~-~& 10\left(\dfrac{n-4}{5}\right)x^{n-5} &~+~& 65\left(\dfrac{n-5}{6}\right)x^{n-6}&~-~& 350\left(\dfrac{n-6}{7}\right)x^{n-7}&~+~& 1701\left(\dfrac{n-7}{8}\right)x^{n-8}&~-~&\text{etc.} \\
  ~ m=5 ~& \dfrac{s}{1 \cdot 2 \cdots 5} &~=~& \left(\dfrac{n-4}{5}\right)x^{n-5} &~-~& 15\left(\dfrac{n-5}{6}\right)x^{n-6} &~+~& 140\left(\dfrac{n-6}{7}\right)x^{n-7}&~-~& 1050\left(\dfrac{n-7}{8}\right)x^{n-8}&~+~& 6951\left(\dfrac{n-8}{9}\right)x^{n-9}&~-~&\text{etc.} \\
   ~ m=6 ~& \dfrac{s}{1 \cdot 2 \cdots 6} &~=~& \left(\dfrac{n-5}{6}\right)x^{n-6} &~-~& 21\left(\dfrac{n-6}{7}\right)x^{n-7} &~+~& 266\left(\dfrac{n-7}{8}\right)x^{n-8}&~-~& 2646\left(\dfrac{n-8}{9}\right)x^{n-9}&~+~& 22827\left(\dfrac{n-9}{10}\right)x^{n-10}&~-~&\text{etc.} \\
\end{array}
\end{equation*}
\end{scriptsize}
\end{center}
where the formation of each numerical coefficient from the preceding is obvious; hence for the last sixth series:

\begin{equation*}
    21=6 \cdot 1 +15, \quad 266=6 \cdot 21+140, \quad 2646= 6 \cdot 266 +1050 \quad \text{etc.}
\end{equation*}
And hence it is immediately seen, if $m<n$, that the value of $s$ vanishes; for, in the last series, if $n<6$ and hence either $5$ or $4$ or $3$ etc., it will be

\begin{equation*}
    \left(\dfrac{n-5}{6}\right)=0, \quad \left(\dfrac{n-6}{7}\right)=0 \quad \text{etc.}
\end{equation*}
But then on the other hand, if $n=m$, it is also evident that
\begin{equation*}
    \dfrac{s}{1\cdot 2 \cdots m}=1,
\end{equation*}
for, in the lowest series:

\begin{equation*}
    \left(\dfrac{6-5}{6}\right)=1, \quad \left(\dfrac{6-6}{7}\right)=0, \quad \left(\dfrac{6-7}{8}\right)=0, \quad \left(\dfrac{6-8}{9}\right)=0 \quad \text{etc.}
\end{equation*}

\subsection*{Expansion of the Cases $n=m+1$}

\paragraph*{V.} 

Hence let us first expand the cases in which $n=m+1$;  the last form yields

\begin{equation*}
         \renewcommand{\arraystretch}{2.0}
\setlength{\arraycolsep}{0mm}
\begin{array}{c|cllcrr}
    \quad \text{if} \quad & ~~ \text{these sums} ~~\\
    ~~ m=1, \quad n=2 ~~ & \dfrac{s}{1} &~=~& 2x &~-~& 1 \\
    ~~ m=2, \quad n=3 ~~& \dfrac{s}{1\cdot 2} &~=~& 3x &~-~& 3 \\
    ~~ m=3, \quad n=4 ~~& \dfrac{s}{1\cdot 2 \cdot 3} &~=~& 4x &~-~& 6 \\
    ~~ m=4, \quad n=5 ~~& \dfrac{s}{1\cdot 2 \cdot 3 \cdot 4} &~=~& 5x &~-~& 10 \\
    ~~ m=5, \quad n=6 ~~& \dfrac{s}{1\cdot 2 \cdots 6} &~=~& 6x &~-~& 15 \\
                        & \text{etc.},
\end{array}
\end{equation*}
where the first coefficients of $x$ are equal to $n$, but the absolute numbers are equal to the triangular number of $n$; in general, we will have 

\begin{equation*}
        \renewcommand{\arraystretch}{2.0}
\setlength{\arraycolsep}{0mm}
\begin{array}{c|l}
     \text{if} & \quad \text{this equation} \\
\quad     n=m+1 \quad & \quad \dfrac{s}{1 \cdot 2 \cdots m}=(m+1)x-\dfrac{m(m+1)}{1\cdot 2}= (m+1)\left(x-\dfrac{m}{2}\right)
\end{array}
\end{equation*}
such that

\begin{equation*}
     \renewcommand{\arraystretch}{2.5}
\setlength{\arraycolsep}{0mm}
\begin{array}{c}
    x^{m+1}-m(x-1)^{m+1}+\dfrac{m(m-1)}{1\cdot 2}(x-2)^{m+1}-\dfrac{m(m-1)(m-2)}{1\cdot 2 \cdot 3}(x-3)^{m+1}+\text{etc.}\\
    =1 \cdot 2 \cdot 3 \cdots (m+1)\left(x-\dfrac{m}{2}\right).
    \end{array}
\end{equation*}

\subsection*{Expansion of the Cases $n=m+2$}

\paragraph*{VI.} Therefore, for these cases we will have: 

\begin{equation*}
        \renewcommand{\arraystretch}{2.5}
\setlength{\arraycolsep}{0mm}
\begin{array}{c|ccrlllllrrrrrrrrrrrrrr}
     \text{if it was} & ~~ \text{these equations} ~~ \\
     m=1, \quad n=3 \quad & \dfrac{s}{1} &~=~& 3x^2 &~-~& 3 \cdot 1 &x& ~+~& 1 \cdot ~&1 &~=~& 3&\bigg(xx &~-~&x &~+~&\dfrac{2}{6}\bigg) \\
     m=2, \quad n=4 \quad & \dfrac{s}{1 \cdot 2} &~=~& 6x^2 &~-~& 4 \cdot 3 &x& ~+~& 1 \cdot~ & 7 &~=~& 6&\bigg(xx &~-~&2x &~+~&\dfrac{7}{6}\bigg) \\
     m=3, \quad n=5 \quad & \dfrac{s}{1 \cdot 2 \cdot 3} &~=~& 10x^2 &~-~& 5 \cdot 6 &x& ~+~& 1 \cdot~ & 25 &~=~& 10&\bigg(xx &~-~&3x &~+~&\dfrac{15}{6}\bigg) \\
      m=4, \quad n=6 \quad & \dfrac{s}{1 \cdot 2 \cdot 3 \cdot 4} &~=~& 15x^2 &~-~& 6 \cdot 10 &x& ~+~& 1 \cdot~ & 65 &~=~& 15&\bigg(xx &~-~&4x &~+~&\dfrac{26}{6}\bigg) \\
      m=5, \quad n=7 \quad & \dfrac{s}{1 \cdot 2 \cdots 5} &~=~& 21x^2 &~-~& 7 \cdot 15 &x& ~+~& 1 \cdot~ & 140 &~=~& 21&\bigg(xx &~-~&5x &~+~&\dfrac{40}{6}\bigg) \\
        m=6, \quad n=8 \quad & \dfrac{s}{1 \cdot 2 \cdots 6} &~=~& 28x^2 &~-~& 8 \cdot 21 &x& ~+~& 1 \cdot~ & 266 &~=~& 28&\bigg(xx &~-~&6x &~+~&\dfrac{57}{6}\bigg) \\
        &  & & &  & \text{etc.},
\end{array}
\end{equation*}
which forms can be represented as follows:

\begin{equation*}
         \renewcommand{\arraystretch}{2.5}
\setlength{\arraycolsep}{0mm}
\begin{array}{c|cccrrrrc}
     \text{if it was} \quad & ~~ \text{it will be} ~~ \\
     m=1, \quad n=3 \quad & \dfrac{s}{1} &~=~& \dfrac{2 \cdot 3}{1 \cdot 2}&\bigg(xx&~-~& x &~+~& \dfrac{1 \cdot 4}{12}\bigg) \\
      m=2, \quad n=4 \quad & \dfrac{s}{1 \cdot 2} &~=~& \dfrac{3 \cdot 4}{1 \cdot 2}&\bigg(xx&~-~& 2x &~+~& \dfrac{2 \cdot 7}{12}\bigg) \\
      m=3, \quad n=5 \quad & \dfrac{s}{1 \cdot 2 \cdot 3} &~=~& \dfrac{4 \cdot 5}{1 \cdot 2}&\bigg(xx&~-~& 3x &~+~& \dfrac{3 \cdot 10}{12}\bigg) \\
      m=4, \quad n=6 \quad & \dfrac{s}{1 \cdot 2 \cdot 3 \cdot 4} &~=~& \dfrac{5 \cdot 6}{1 \cdot 2}&\bigg(xx&~-~& 4x &~+~& \dfrac{4 \cdot 13}{12}\bigg) \\
      m=5, \quad n=7 \quad & \dfrac{s}{1 \cdot 2 \cdots 5} &~=~& \dfrac{6 \cdot 7}{1 \cdot 2}&\bigg(xx&~-~& 5x &~+~& \dfrac{5 \cdot 16}{12}\bigg) \\
      m=6, \quad n=8 \quad & \dfrac{s}{1 \cdot 2 \cdots 6} &~=~& \dfrac{7 \cdot 8}{1 \cdot 2}&\bigg(xx&~-~& 6x &~+~& \dfrac{6 \cdot 19}{12}\bigg),
\end{array}
\end{equation*}
whence it manifestly follows, if in general $n=m+2$, that it will be

\begin{equation*}
    \dfrac{s}{1\cdot 2 \cdots m}= \dfrac{m+1}{1}\cdot \dfrac{m+2}{2}\left(xx-mx+\dfrac{m(3m+1)}{12}\right)
\end{equation*}
or

\begin{equation*}
    \dfrac{s}{1\cdot 2 \cdots m}=\dfrac{m+1}{1}\cdot \dfrac{m+2}{2} \left(\left(x-\dfrac{m}{2}\right)^2+\dfrac{m}{12}\right).
\end{equation*}
Therefore, one obtains this summation

\begin{equation*}
    x^{m+2}-m(x-1)^{m+2}+\dfrac{m(m-1)}{1 \cdot 2}(x-2)^{m+2}-\dfrac{m(m-1)(m-2)}{1 \cdot 2 \cdot 3}(x-3)^{m+2}+\text{etc.}
\end{equation*}
\begin{equation*}
    =1 \cdot 2 \cdot 3 \cdots (m+2) \left(\dfrac{1}{2}\left(x-\dfrac{m}{2}\right)^2+\dfrac{m}{24}\right).
\end{equation*}

\subsection*{Expansion of the Cases $n=m+3$}

\paragraph*{VII.}

For these cases we will have

\begin{equation*}
          \renewcommand{\arraystretch}{2.5}
\setlength{\arraycolsep}{0mm}
\begin{array}{c|crrlcclrrllll}
    \text{if it was} \quad & \quad \text{these equations} \\
    \quad m=1, \quad n=4 \quad & \dfrac{s}{1} &~=~& 4x^3 &~-~& 6 ~\cdot ~  & 1x^2 &~+~& 4 ~\cdot ~ & 1&x &~-~& 1 \cdot 1 \\ 
     \quad m=2, \quad n=5 \quad & \dfrac{s}{1 \cdot 2} &~=~& 10x^3 &~-~& 10 ~ \cdot  ~ & 3x^2 &~+~& 5 ~\cdot ~ & 7&x &~-~& 1 \cdot 15 \\ 
       \quad m=3, \quad n=6 \quad & \dfrac{s}{1 \cdot 2 \cdot 3} &~=~& 20x^3 &~-~& 15 ~ \cdot  ~ & 6x^2 &~+~& 6 ~ \cdot ~ & 25&x &~-~& 1 \cdot 90 \\ 
        \quad m=4, \quad n=7 \quad & \dfrac{s}{1 \cdot 2 \cdots 4} &~=~& 35x^3 &~-~& 21 ~ \cdot  ~ & 10x^2 &~+~& 7 ~\cdot ~ & 65&x &~-~& 1 \cdot 350 \\ 
          \quad m=5, \quad n=8 \quad & \dfrac{s}{1 \cdot 2 \cdots 5} &~=~& 56x^3 &~-~& 28 ~ \cdot  ~ & 15x^2 &~+~& 8 ~\cdot ~ & 140&x &~-~& 1 \cdot 1050 \\ 
\end{array}
\end{equation*}
which can be represented in this way:

\begin{equation*}
       \renewcommand{\arraystretch}{2.5}
\setlength{\arraycolsep}{0mm}
\begin{array}{cllllcllcllc}
     \dfrac{s}{1} &~=~& \dfrac{2 \cdot 3 \cdot 4}{1 \cdot 2 \cdot 3}&\bigg(x^3 &~-~& \dfrac{3}{2} &x^2 &~+~& \dfrac{1 \cdot 4}{4} &x& ~-~& \dfrac{1 \cdot 1 \cdot 2}{8}\bigg) \\
     \dfrac{s}{1 \cdot 2} &~=~& \dfrac{3 \cdot 4 \cdot 5}{1 \cdot 2 \cdot 3}&\bigg(x^3 &~-~& \dfrac{6}{2} &x^2 &~+~& \dfrac{2 \cdot 7}{4} &x& ~-~& \dfrac{2 \cdot 2 \cdot 3}{8}\bigg) \\
      \dfrac{s}{1 \cdot 2 \cdot 3} &~=~& \dfrac{4 \cdot 5 \cdot 6}{1 \cdot 2 \cdot 3}&\bigg(x^3 &~-~& \dfrac{9}{2} &x^2 &~+~& \dfrac{3 \cdot 10}{4} &x& ~-~& \dfrac{3 \cdot 3 \cdot 4}{8}\bigg) \\
       \dfrac{s}{1 \cdot 2 \cdots 4} &~=~& \dfrac{5 \cdot 6 \cdot 7}{1 \cdot 2 \cdot 3}&\bigg(x^3 &~-~& \dfrac{12}{2} &x^2 &~+~& \dfrac{4 \cdot 13}{4} &x& ~-~& \dfrac{4 \cdot 4 \cdot 5}{8}\bigg) \\
        \dfrac{s}{1 \cdot 2 \cdots 5} &~=~& \dfrac{6 \cdot 7 \cdot 8}{1 \cdot 2 \cdot 3}&\bigg(x^3 &~-~& \dfrac{15}{2} &x^2 &~+~& \dfrac{5 \cdot 16}{4} &x& ~-~& \dfrac{5 \cdot 5 \cdot 6}{8}\bigg) \\
         & & & &\text{etc.},
\end{array}
\end{equation*}
whence in general for the cases $n=m+3$ one concludes

\begin{equation*}
        \renewcommand{\arraystretch}{2.5}
\setlength{\arraycolsep}{0mm}
\begin{array}{lll}
     \dfrac{s}{1 \cdot 2 \cdots m} &~=~& \dfrac{m+1}{1}\cdot \dfrac{m+2}{2}\cdot \dfrac{m+3}{3}\left(x^3-\dfrac{3m}{2}x^2+\dfrac{m(3m+1)}{4}x-\dfrac{mm(m+1)}{8}\right) \\
      &~=~& \dfrac{m+1}{1}\cdot \dfrac{m+2}{2}\cdot \dfrac{m+3}{3} \left(\left(x-\dfrac{m}{2}\right)^2+\dfrac{m}{4}\left(x-\dfrac{m}{2}\right)\right)
\end{array}
\end{equation*}
such that we obtain:

\begin{equation*}
    x^{m+3}-m(x-1)^{m+3}+\dfrac{m(m-1)}{1\cdot 2}(x-2)^{m+3}-\dfrac{m(m-1)(m-2)}{1\cdot 2 \cdot 3}(x-3)^{m+3}+\text{etc.}
\end{equation*}
\begin{equation*}
    =1 \cdot 2 \cdot 3 \cdots (m+3) \left(\dfrac{1}{6}\left(x-\dfrac{m}{2}\right)^3+\dfrac{m}{24}\left(x-\dfrac{m}{2}\right)\right).
\end{equation*}

\subsection*{Preparation for the following Cases}

\paragraph*{VIII.}

Although in paragraph IV. we  gave the formulas only up to the case $m=6$, let us try to find a general formula for them. To this end, let us set $n=m +\lambda$, and to abbreviate  the expression

\begin{equation*}
    \dfrac{k(k-1)(k-2)(k-3)\cdots (k-i+1)}{1 \cdot 2 \cdot 3 \cdot 4 \cdots i}
\end{equation*}
let us write

\begin{equation*}
    \left(\dfrac{k}{i}\right)
\end{equation*}
such that $k$ denotes the first factor of the numerator, $i$ on the other hand the last factor of the denominator. Therefore, let us put  for the case

\begin{center}
\begin{footnotesize}
\begin{equation*}
        \renewcommand{\arraystretch}{2.5}
\setlength{\arraycolsep}{0mm}
\begin{array}{c|clllllllllllllllllll}
     m-1 \quad & \quad \dfrac{s}{1 \cdot 2 \cdots (m-1)} &~=~& \left(\dfrac{m +\lambda}{m-1}\right)x^{\lambda +1}&~-~&A &\left(\dfrac{m +\lambda}{m}\right)x^{\lambda}&~+~& B&\left(\dfrac{m +\lambda}{m+1}\right)x^{\lambda -1}&~-~& C&\left(\dfrac{m +\lambda}{m+2}\right)x^{\lambda -2}&~+~&\text{etc.} \\
       m \quad & \quad \dfrac{s}{1 \cdot 2 \cdot 3 \cdots m} &~=~& \left(\dfrac{m +\lambda}{m}\right)x^{\lambda}&~-~&A^1 &\left(\dfrac{m +\lambda}{m+1}\right)x^{\lambda -1}&~+~& B^1&\left(\dfrac{m +\lambda}{m+2}\right)x^{\lambda -2}&~-~& C^1&\left(\dfrac{m +\lambda}{m+3}\right)x^{\lambda -3}&~+~&\text{etc.} \\
\end{array}
\end{equation*}
\end{footnotesize}
\end{center}
such that $A^1$, $B^1$, $C^1$, $D^1$ etc. are the coefficients which must be investigated. But from the law of these formulas we see that

\begin{equation*}
    A^1 =m \cdot 1+A,~~B^1= mA^1+B,~~C^1=mB^1+C,~~D^1=mC^1+D~~\text{etc.},
\end{equation*}
where it is evident that

\begin{equation*}
    A=\dfrac{m(m-1)}{1 \cdot 2} \quad \text{and} \quad A^1 = \dfrac{(m+1)m}{1\cdot 2}
\end{equation*}
or, in our notation,

\begin{equation*}
    A=\left(\dfrac{m}{2}\right) \quad \text{and} \quad A^1 = \left(\dfrac{m+1}{2}\right).
\end{equation*}
Now for the following operations I observe that:

\begin{equation*}
    \left(\dfrac{m+ \mu +1}{\nu}\right)-\left(\dfrac{m +\mu}{\nu}\right)= \left(\dfrac{m+\mu}{\nu -1}\right),
\end{equation*}
which is clear, since by expanding:

\begin{equation*}
          \renewcommand{\arraystretch}{2.5}
\setlength{\arraycolsep}{0mm}
\begin{array}{rll}
     \left(\dfrac{m+ \mu +1}{\nu}\right)&~=~& \dfrac{(m+\mu+1)(m+\mu)(m+\mu -1)\cdots (m+\mu +2- \nu)}{1 \cdot 2 \cdot 3 \cdots \nu} \\
      \left(\dfrac{m+ \mu}{\nu}\right)&~=~& \dfrac{(m+\mu)(m+\mu -1)\cdots (m +\mu +2- \nu)(m+\mu +1 -\nu)}{1 \cdot 2 \cdots (\nu -1) \nu}, \\
\end{array}
\end{equation*}
whence it is seen that

\begin{equation*}
    \left(\dfrac{m+1}{2}\right)-\left(\dfrac{m}{2}\right)= \left(\dfrac{m}{1}\right)=m.
\end{equation*}

\paragraph*{IX.} Now for it to be

\begin{equation*}
    B^1-B =mA^1 = \left(\dfrac{m+1}{2}\right)m = 3 \left(\dfrac{m+1}{3}\right)+\left(\dfrac{m+1}{2}\right),
\end{equation*}
let us set

\begin{equation*}
    B=  \alpha \left(\dfrac{m+1}{4}\right)+ \beta \left(\dfrac{m+1}{3}\right)
\end{equation*}
and hence

\begin{equation*}
    B^1= \alpha \left(\dfrac{m+2}{4}\right)+ \beta \left(\dfrac{m+2}{3}\right)
\end{equation*}
and it will result

\begin{equation*}
    B^1-B = \alpha \left(\dfrac{m+1}{3}\right)+ \beta \left(\dfrac{m+1}{2}\right),
\end{equation*}
whence $\alpha =3$ and $\beta =1$ such that

\begin{equation*}
    B^1= 3 \left(\dfrac{m+2}{4}\right)+  \left(\dfrac{m+2}{3}\right).
\end{equation*}
But for the following operations note that in general:

\begin{equation*}
    \left(\dfrac{m+\mu}{\nu}\right)m = (\nu +1) \left(\dfrac{m+\mu}{\nu +1}\right)+(\nu -\mu) \left(\dfrac{m+\mu}{\nu}\right),
\end{equation*}
which form results, if the value of $\left(\dfrac{m +\mu}{\nu}\right)$ expanded above is multiplied by

\begin{equation*}
    m= m+\mu - \nu - \nu - \mu = (\nu +1)\cdot \dfrac{m +\mu - \nu}{\nu +1}+(\nu -\mu).
\end{equation*}

\paragraph*{X.}

Since, having observed these things, it must be $C^1-C= mB^1$, because of

\begin{equation*}
    \left(\dfrac{m+2}{4}\right)m = 5 \left(\dfrac{m+2}{5}\right)+2 \left(\dfrac{m+2}{4}\right)
\end{equation*}
and

\begin{equation*}
    \left(\dfrac{m+2}{3}\right)= 4 \left(\dfrac{m+2}{4}\right)+1 \left(\dfrac{m+2}{3}\right),
\end{equation*}
it will be

\begin{equation*}
    mB^1 = 15 \left(\dfrac{m+2}{5}\right)+10 \left(\dfrac{m+2}{4}\right)+1 \left(\dfrac{m+2}{3}\right);
\end{equation*}
therefore, set

\begin{equation*}
    C=15 \left(\dfrac{m+2}{6}\right)+10 \left(\dfrac{m+2}{5}\right)+1 \left(\dfrac{m+2}{4}\right)
\end{equation*}
hence

\begin{equation*}
    C^1=15 \left(\dfrac{m+3}{6}\right)+10 \left(\dfrac{m+3}{5}\right)+1 \left(\dfrac{m+3}{4}\right).
\end{equation*}

\paragraph*{XI.}

Since, in like manner, it has to be $D^1-D=mC^1$, since

\begin{equation*}
            \renewcommand{\arraystretch}{2.5}
\setlength{\arraycolsep}{0mm}
\begin{array}{llllll}
     m \left(\dfrac{m+3}{6}\right)&~=~& 7 \left(\dfrac{m+3}{7}\right)&~+~&3 \left(\dfrac{m+3}{6}\right) \\
     m \left(\dfrac{m+3}{5}\right)&~=~& 6 \left(\dfrac{m+3}{6}\right)&~+~&2 \left(\dfrac{m+3}{5}\right) \\
     m \left(\dfrac{m+3}{4}\right)&~=~& 5 \left(\dfrac{m+3}{5}\right)&~+~&1 \left(\dfrac{m+3}{4}\right), \\
\end{array}
\end{equation*}
it will be

\begin{equation*}
    mC^1=105\left(\dfrac{m+3}{7}\right)+105 \left(\dfrac{m+3}{6}\right)+25\left(\dfrac{m+3}{5}\right)+\left(\dfrac{m+3}{4}\right),
\end{equation*}
whence we conclude:

\begin{equation*}
    D^1 = 105 \left(\dfrac{m+4}{8}\right)+105\left(\dfrac{m+4}{7}\right)+25\left(\dfrac{m+4}{6} \right)+1 \left(\dfrac{m+4}{5}\right).
\end{equation*}

\paragraph*{XII.}

Further, because of $E^1-E=mD^1$, since

\begin{equation*}
            \renewcommand{\arraystretch}{2.5}
\setlength{\arraycolsep}{0mm}
\begin{array}{llllll}
     m \left(\dfrac{m+4}{8}\right)&~=~& 9 \left(\dfrac{m+4}{9}\right)&~+~&4 \left(\dfrac{m+4}{8}\right) \\
       m \left(\dfrac{m+4}{7}\right)&~=~& 8 \left(\dfrac{m+4}{8}\right)&~+~&3 \left(\dfrac{m+4}{7}\right) \\
         m \left(\dfrac{m+4}{6}\right)&~=~& 7 \left(\dfrac{m+4}{7}\right)&~+~&2 \left(\dfrac{m+4}{6}\right) \\
           m \left(\dfrac{m+4}{5}\right)&~=~& 6 \left(\dfrac{m+4}{6}\right)&~+~&1\left(\dfrac{m+4}{5}\right),
     \end{array}
     \end{equation*}
     we conclude
     \begin{equation*}
         mD^1 =945\left(\dfrac{m+4}{9}\right)+1260 \left(\dfrac{m+4}{8}\right)+490 \left(\dfrac{m+4}{7}\right)+56 \left(\dfrac{m+4}{6}\right)+1\left(\dfrac{m+4}{5}\right)
     \end{equation*}
     and hence
     
     \begin{equation*}
         E^1  =945\left(\dfrac{m+5}{10}\right)+1260 \left(\dfrac{m+5}{9}\right)+490 \left(\dfrac{m+5}{8}\right)+56 \left(\dfrac{m+5}{7}\right)+1\left(\dfrac{m+5}{6}\right)
     \end{equation*}
     and, proceeding even further,

    \begin{center} 
 \begin{small}
     \begin{equation*}
         F^1 = 10395\left(\dfrac{m+6}{12}\right)+17325\left(\dfrac{m+6}{11}\right)+9450\left(\dfrac{m+6}{10}\right)+1918 \left(\dfrac{m+6}{9}\right)+119 \left(\dfrac{m+6}{8}\right)+\left(\dfrac{m+6}{7}\right).
     \end{equation*}
\end{small}
\end{center}

\subsection*{Expansion of the case $n=m+\lambda$}
     
\paragraph*{XIII.}
     
     Therefore, for our series in the case $n=m+\lambda$
     
     \begin{equation*}
         s=x^{m+\lambda}-\dfrac{m}{1}(x-1)^{m+\lambda}+\dfrac{m(m-1)}{1\cdot 2}(x-2)^{m+\lambda}-\dfrac{m(m-1)(m-2)}{1 \cdot 2 \cdot 3}(x-3)^{m+\lambda}+\text{etc.},
     \end{equation*}
     if we divide the general equation exhibited above in paragraph VIII by
     
     \begin{equation*}
         \left(\dfrac{m+\lambda}{m}\right)= \dfrac{(m+\lambda)(m+\lambda -1)(m+\lambda -2)\cdots (\lambda +1)}{1 \cdot 2 \cdot 3 \cdots m},
     \end{equation*}
     we will get to this expression
     
     \begin{equation*}
              \renewcommand{\arraystretch}{2.5}
\setlength{\arraycolsep}{0mm}
\begin{array}{llll}
     \dfrac{s}{(\lambda +1)(\lambda+2)\cdots (\lambda +m)} = x^{\lambda}&~-~& \dfrac{\lambda}{m+1}A^1 x^{\lambda -1}+ \dfrac{\lambda (\lambda -1)}{(m+1)(m+2)}B^1 x^{\lambda -2} \\
      &~-~& \dfrac{\lambda (\lambda -1)(\lambda -2)}{(m+1)(m+2)(m+3)}C^1 x^{\lambda -3}+\text{etc.} 
\end{array}
     \end{equation*}
     where one has to substitute the following values for the letters $A^1$, $B^1$, $C^1$, $D^1$ etc.:

\begin{small}     
     \begin{equation*}
        \renewcommand{\arraystretch}{2.5}
\setlength{\arraycolsep}{0mm}
\begin{array}{rrrrrrrrrrrrrrrrrrrrrrr}
 A^1    &~=~& \left(\dfrac{m+1}{2}\right) &~=~& \dfrac{(m+1)m}{1 \cdot 2}  \\
 B^1 &~=~& 3\left(\dfrac{m+2}{4}\right) &~+~& \left(\dfrac{m+2}{3}\right) \\
 C^1 &~=~& 15\left(\dfrac{m+3}{6}\right) &~+~& 10\left(\dfrac{m+3}{5}\right)&~+~& \left(\dfrac{m+3}{4}\right) \\
  D^1 &~=~& 105\left(\dfrac{m+4}{8}\right) &~+~& 15\left(\dfrac{m+4}{7}\right)&~+~& 25\left(\dfrac{m+3}{4}\right)&~+~& \left(\dfrac{m+4}{5}\right) \\
 E^1 &~=~& 945\left(\dfrac{m+5}{10}\right) &~+~& 1260\left(\dfrac{m+5}{9}\right)&~+~& 490\left(\dfrac{m+5}{8}\right)&~+~& 56\left(\dfrac{m+5}{7}\right)&~+~&\left(\dfrac{m+5}{6}\right) \\ 
  F^1 &~=~& 10395\left(\dfrac{m+6}{12}\right) &~+~& 17325\left(\dfrac{m+6}{11}\right)&~+~& 9450\left(\dfrac{m+6}{10}\right)&~+~& 1918\left(\dfrac{m+6}{9}\right)&~+~&119\left(\dfrac{m+6}{7}\right)&~+~&\left(\dfrac{m+6}{7}\right) 
\end{array}
     \end{equation*}
 \end{small}
 where
 
 \begin{equation*}
       \renewcommand{\arraystretch}{1.5}
\setlength{\arraycolsep}{0mm}
\begin{array}{rrrrrrrrrrrrr}
     10395=11 \cdot 945, \quad 17325 &~=~& 10~\cdot ~&1260 &~+~& 5 ~ \cdot ~& 945 \\
      9450 &~=~& 9~\cdot ~&490 &~+~& 4 ~ \cdot ~& 1260 \\
       1918 &~=~& 8~\cdot ~&56 &~+~& 3 ~ \cdot ~& 490 \\
       119 &~=~& 7~\cdot ~&1 &~+~& 2 ~ \cdot ~& 56 \\
         1 &~=~& 6~\cdot ~&0 &~+~& 1 ~ \cdot ~& 1 \\
\end{array}
 \end{equation*}
 hence, if for the following value one sets
 
 \begin{center}
 \begin{small}
 \begin{equation*}
     G^1 = \alpha \left(\dfrac{m+7}{14}\right)+\beta \left(\dfrac{m+7}{13}\right)+\gamma \left(\dfrac{m+7}{12}\right)+\delta \left(\dfrac{m+7}{11}\right)+\varepsilon \left(\dfrac{m+7}{10}\right)+ \zeta \left(\dfrac{m+7}{9}\right)+ \eta\left(\dfrac{m+7}{8}\right),
 \end{equation*}
 \end{small}
 \end{center}
 these  coefficients will be determined as follows:
 
 \begin{equation*}
         \renewcommand{\arraystretch}{1.5}
\setlength{\arraycolsep}{0mm}
\begin{array}{rrrrrrr|rrrrrrrrrrrrr}
     \alpha &~=~& 13 ~\cdot ~& 10395  & & & & \quad \varepsilon &~=~& 9 ~ \cdot ~ & 119 &~+~& 3 ~ \cdot ~ & 1918 \\
     \beta &~=~& 12 ~\cdot ~& 17325  &~+~ & 6 ~\cdot ~ & 10395 \quad & \zeta &~=~& 8 ~ \cdot ~ & 1 &~+~& 2 ~ \cdot ~ & 119 \\
      \gamma &~=~& 11 ~\cdot ~& 9450  &~+~ & 5 ~\cdot ~ & 17325 \quad & \eta &~=~& 7 ~ \cdot ~ & 0 &~+~& 1 ~ \cdot ~ & 1 \\
       \delta &~=~& 10 ~\cdot ~& 1918  &~+~ & 4 ~\cdot ~ & 9450 \quad &  & &   \\
\end{array}
 \end{equation*}
 
 \paragraph*{XIV.}
 
 But the same values are expressed more conveniently in this way:
 
 \begin{footnotesize}
 \begin{equation*}
          \renewcommand{\arraystretch}{2.5}
\setlength{\arraycolsep}{0mm}
\begin{array}{rrllrlrllrrrrrrrrrrrrrrrrrrr}
     A^1 &~=~& \left(\dfrac{m+1}{2}\right)&~\cdot 1 \\
     B^1 &~=~& \left(\dfrac{m+2}{3}\right)&\bigg(1~+~& 3 \cdot \dfrac{m-1}{4}&\bigg)\\
     C^1 &~=~& \left(\dfrac{m+3}{4}\right)&\bigg(1~+~& 10 \cdot \dfrac{m-1}{5} &~+~& 15 \cdot \dfrac{m-1}{5}\cdot \dfrac{m-2}{6}&\bigg) \\
     D^1 &~=~& \left(\dfrac{m+4}{5}\right)&\bigg(1~+~& 25 \cdot \dfrac{m-1}{6} &~+~& 105 \cdot \dfrac{m-1}{6}\cdot \dfrac{m-2}{7} &~+~& 105 \cdot \dfrac{m-1}{6}\cdot \dfrac{m-2}{7}\cdot \dfrac{m-3}{8}\bigg) \\
      E^1 &~=~& \left(\dfrac{m+5}{6}\right)&\bigg(1~+~& 56 \cdot \dfrac{m-1}{7} &~+~& 490 \cdot \dfrac{m-1}{7}\cdot \dfrac{m-2}{8} &~+~& 1260 \cdot \dfrac{m-1}{7}\cdot \dfrac{m-2}{8}\cdot \dfrac{m-3}{9} \\
 & & & & &  & &~+~ &945 \cdot \dfrac{m-1}{7} \cdot \dfrac{m-2}{8} \cdot \dfrac{m-3}{9} \cdot \dfrac{m-4}{10} \bigg) \\
  F^1 &~=~& \left(\dfrac{m+6}{7}\right)&\bigg(1~+~& 119 \cdot \dfrac{m-1}{8} &~+~& 1918 \cdot \dfrac{m-1}{8}\cdot \dfrac{m-2}{9} &~+~& 9450 \cdot \dfrac{m-1}{8}\cdot \dfrac{m-2}{9}\cdot \dfrac{m-3}{10} \\
 & & & & &  & &~+~ &17325 \cdot \dfrac{m-1}{8} \cdot \dfrac{m-2}{9} \cdot \dfrac{m-3}{10} \cdot \dfrac{m-4}{11} \\
 & & & & &  & &~+~ &10395 \cdot \dfrac{m-1}{8} \cdot \dfrac{m-2}{9} \cdot \dfrac{m-3}{10} \cdot \dfrac{m-4}{11} \cdot \dfrac{m-5}{12} \bigg).
\end{array}
 \end{equation*}
 \end{footnotesize}
 To see the  law of this progression more easily, in general let us set

 \begin{equation*}
     M^1 = \left(\dfrac{m+\mu -1}{\mu}\right)\left(1+\alpha \cdot \dfrac{m-1}{\mu +1}+\beta \cdot \dfrac{m-1}{\mu +1}\cdot \dfrac{m-2}{\mu +2}+\gamma \cdot \dfrac{m-1}{\mu +1}\cdot \dfrac{m-2}{\mu +2}\cdot \dfrac{m-3}{\mu +3}+\text{etc.}\right)
 \end{equation*}
 and the following one
 
 \begin{equation*}
     N^1 =\left(\dfrac{m +\mu}{\mu +1}\right)\left(1+\alpha^1 \cdot \dfrac{m-1}{\mu +2}+\beta^1 \dfrac{m-1}{\mu +2}\cdot \dfrac{m-2}{\mu +3}+\gamma^1 \cdot \dfrac{m-1}{\mu +2}\cdot \dfrac{m -2}{\mu+3}\cdot \dfrac{m-3}{\mu +4}+\text{etc.}\right),
 \end{equation*}
 and these coefficients are determined as follows by the preceding ones:
 
 \begin{equation*}
         \renewcommand{\arraystretch}{1.5}
\setlength{\arraycolsep}{0mm}
\begin{array}{llllcll}
     \alpha^1 &~=~& 2 \alpha &~+~& \mu +1 & \\
     \beta^1  &~=~& 3 \beta  &~+~& (\mu +2)& \alpha \\
     \gamma^1  &~=~& 4 \gamma  &~+~& (\mu +3)& \beta \\
     \delta^1  &~=~& 5 \delta  &~+~& (\mu +4)& \gamma \\
      \varepsilon^1  &~=~& 6 \varepsilon  &~+~& (\mu +5) &\delta,
\end{array}
 \end{equation*}
 whence these formulas can easily be continued arbitrarily far.
 
 \paragraph*{XV.}
 
 Let us now substitute these values, and for the sum $s$ of the propounded series, if $n=m+\lambda$, we will obtain the following expression:
 
 \begin{small}
 \begin{equation*}
     \dfrac{s}{(\lambda +1)(\lambda +2)\cdots (\lambda +m)}
 \end{equation*}
 \begin{equation*}
         \renewcommand{\arraystretch}{2.5}
\setlength{\arraycolsep}{0mm}
\begin{array}{llll}
     &~=~& x^{\lambda}-\dfrac{\lambda m}{1 \cdot 2}x^{\lambda -1}+\dfrac{\lambda (\lambda -1)m}{1 \cdot 2 \cdot 3}x^{\lambda -2}\left(1+\dfrac{3(m-1)}{4}\right) \\
     &~-~& \dfrac{\lambda (\lambda -1)(\lambda -2)m}{1 \cdot 2 \cdot 3 \cdot 4}x^{\lambda -3}\left(1+10 \cdot \dfrac{m-1}{5}+15 \cdot \dfrac{m-1}{5}\cdot \dfrac{m-2}{6}\right) \\
     &~+~& \dfrac{\lambda (\lambda -1)(\lambda -2)(\lambda -3)m}{1 \cdot 2 \cdot 3 \cdot 4 \cdot 5}x^{\lambda -4}\left(1+25\cdot \dfrac{m-1}{6}+105 \cdot \dfrac{m-1}{6}\cdot \dfrac{m-2}{7}+105 \cdot \dfrac{m-1}{6}\cdot \dfrac{m-2}{7}\cdot \dfrac{m-3}{8}\right) \\
     &~-~& \dfrac{\lambda \cdots (\lambda -4)m}{1 \cdots 5 \cdot 6}x^{\lambda -5}\left(1+ 56 \cdot \dfrac{m-1}{7}+490 \cdot \dfrac{m-1}{7}\cdot \dfrac{m-2}{8}+1260 \cdot \dfrac{m-1}{7}\cdots \dfrac{m-3}{9}\right.\\
     &~+~&\left. 945\cdot \dfrac{m-1}{7}\cdots \dfrac{m-4}{10}\right)\\
     &~+~& \dfrac{\lambda \cdots (\lambda -5)m}{1 \cdots 6 \cdot 7}x^{\lambda -6} \left(1+119 \cdot \dfrac{m-1}{8}+1918 \cdot \dfrac{m-1}{8}\cdot \dfrac{m-2}{9}+ 9450 \cdot \dfrac{m-1}{8}\cdots \dfrac{m-3}{10}\right.\\
     &~+~& \left.17325 \cdot\dfrac{m-1}{8}\cdots \dfrac{m-4}{8}+10395 \cdot \dfrac{m-1}{8}\cdots \dfrac{m-5}{12}\right)\\
     & & \text{etc.}
     \end{array}
 \end{equation*}
 \end{small}from this  subtract the power

 \begin{equation*}
     \renewcommand{\arraystretch}{2.5}
\setlength{\arraycolsep}{0mm}
\begin{array}{lll}
     \left(x-\dfrac{m}{2}\right)^{\lambda}&~=~&x^{\lambda}-\dfrac{\lambda m}{2}x^{\lambda -1}+\dfrac{\lambda (\lambda -1)m^2}{1 \cdot 2 \cdot 4}x^{\lambda -2}-\dfrac{\lambda (\lambda -1)(\lambda -2)m^3}{1 \cdot 2 \cdot 3 \cdot 8}x^{\lambda} \\
     &~+~& \dfrac{\lambda \cdots (\lambda -3)m^4}{1 \cdots 4 \cdot 16}x^{\lambda -4}-\dfrac{\lambda \cdots (\lambda -4)m^5}{1 \cdots 5 \cdot 32}x^{\lambda -5}+\dfrac{\lambda \cdots (\lambda -5)m^6}{1 \cdots 6 \cdot 64}x^{\lambda -6}-\text{etc.}
     \end{array}
 \end{equation*}
 But here it conveniently happens that
 
 \begin{equation*}
     \dfrac{15}{5 \cdot 6}=\dfrac{4}{8}, \quad \dfrac{105}{6 \cdot 7 \cdot 8}=\dfrac{5}{16}, \quad \dfrac{945}{7 \cdot 8 \cdot 9 \cdot 10}=\dfrac{6}{32}, \quad \dfrac{10395}{8 \cdot 9 \cdot 10 \cdot 11 \cdot 12}=\dfrac{7}{64},
 \end{equation*}
 the reason for which is obvious; therefore, the above  expression, if it is expanded, takes  the following form:
 
 \begin{equation*}
       \renewcommand{\arraystretch}{2.5}
\setlength{\arraycolsep}{0mm}
\begin{array}{llll}
     \left(x-\dfrac{m}{2}\right)^{\lambda}&~+~& \dfrac{\lambda (\lambda -1)m}{1\cdot 2 \cdot 3}x^{\lambda -2}\cdot \dfrac{1}{4}-\dfrac{\lambda(\lambda -1)(\lambda -2)m}{1 \cdot 2 \cdot 3 \cdot 4}x^{\lambda -3}\cdot \dfrac{m}{2} \\
     & ~+~& \dfrac{\lambda \cdots (\lambda -3)m}{1 \cdots 4 \cdot 5}x^{\lambda -4}\left(\dfrac{5}{8}m^2+\dfrac{5}{48}m-\dfrac{1}{24}\right) \\
     &~-~& \dfrac{\lambda \cdots (\lambda -4)m}{1 \cdots 5 \cdot 6}x^{\lambda -5}\left(\dfrac{5}{8}m^3+\dfrac{5}{16}m^2-\dfrac{1}{8}m\right)\\
     &~+~& \dfrac{\lambda \cdots (\lambda -5)m}{1 \cdots 6 \cdot 7}x^{\lambda -6}\left(\dfrac{35}{64}m^4+\dfrac{35}{64}m^3-\dfrac{91}{576}m^2-\dfrac{7}{96}+\dfrac{1}{30}\right)-\text{etc.}
\end{array}
 \end{equation*}
 
 \paragraph*{XVI.} The power of $x-\frac{m}{2}$, more precisely
 
 \begin{equation*}
     \dfrac{\lambda (\lambda -1)m}{2 \cdot 3 \cdots 4}\left(x-\dfrac{m}{2}\right)^{\lambda -2}
 \end{equation*}
 is again detected to be contained in this expression; having separated this power, our expression will be:
 
 \begin{equation*}
     \left(x-\dfrac{m}{2}\right)^{\lambda}+\dfrac{\lambda (\lambda -1)m}{2 \cdot 3 \cdot 4}\left(x-\dfrac{m}{2}\right)^{\lambda -2}+\dfrac{\lambda \cdots (\lambda -3)m}{1 \cdots 4 \cdot 5}x^{\lambda -4}\left(\dfrac{5}{48}m-\dfrac{1}{24}\right)
 \end{equation*}
 \begin{equation*}
     -\dfrac{\lambda \cdots (\lambda -4)m}{1 \cdots 5 \cdot 6}x^{\lambda -5}\left(\dfrac{5}{16}m^2-\dfrac{1}{8}m\right) 
 \end{equation*}
 \begin{equation*}
     +\dfrac{\lambda \cdots (\lambda -5)m}{1 \cdots 6 \cdot 7}x^{\lambda -6}\left(\dfrac{35}{64}m^3-\dfrac{91}{576}m^2-\dfrac{7}{96}m+\dfrac{1}{36}\right)-\text{etc.},
 \end{equation*}
 which still contains
 
 \begin{equation*}
     \dfrac{\lambda \cdots (\lambda -3)m}{1 \cdots 4 \cdot 5}\left(\dfrac{5}{48}m-\dfrac{1}{24}\right)\left(x-\dfrac{m}{2}\right)^{\lambda -4}
 \end{equation*}
 furthermore,  there still is
 
 \begin{equation*}
     \dfrac{\lambda \cdots (\lambda -5)m}{1 \cdots 6 \cdot 7}x^{\lambda -6}\left(\dfrac{35}{576}m^2-\dfrac{7}{96}m+\dfrac{1}{36}\right),
 \end{equation*}
 whence, without any doubt, additionally this power  enters:
 
 \begin{equation*}
     +\dfrac{\lambda \cdots (\lambda -5)m}{1 \cdots 6 \cdot 7}\cdot \dfrac{35m^2-42m +16}{576}\left(x-\dfrac{m}{2}\right)^{\lambda -6}.
 \end{equation*}
 Therefore, our expression will be of this nature:

 \begin{footnotesize}
 \begin{equation*}
     \dfrac{s}{(\lambda +1)(\lambda+2)\cdots (\lambda +m)}= \left(x-\dfrac{m}{2}\right)^{\lambda}+\dfrac{\lambda (\lambda -1)}{1 \cdot 2}\cdot \dfrac{m}{12}\left(x-\dfrac{m}{2}\right)^{\lambda -2}
 \end{equation*}
 \begin{equation*}
     +\dfrac{\lambda (\lambda -1)\cdots (\lambda -3)}{1\cdot 2 \cdots 4}\cdot \dfrac{m(5m-2)}{240}\left(x-\dfrac{m}{2}\right)^{\lambda -4}+\dfrac{\lambda (\lambda -1)\cdots (\lambda -5)}{1 \cdot 2 \cdots 6}\cdot \dfrac{m(35m^2-42m+16)}{4032}\left(x-\dfrac{m}{2}\right)^{\lambda -6}+\text{etc.}
 \end{equation*}
 \end{footnotesize}
 
 \paragraph*{XVII.}
 
 Therefore, lo and behold the extraordinary transformation of  our propounded general series:
 
  \begin{equation*}
     s=x^{m+\lambda}-\dfrac{m}{1}(x-1)^{m+\lambda}+\dfrac{m(m-1)}{1\cdot 2}(x-2)^{m+\lambda}-\dfrac{m(m-1)(m-2)}{1 \cdot 2 \cdot 3}(x-3)^{m+\lambda}+\text{etc.}
 \end{equation*}
 which, since found in such a long-winded way and using such intricate operations, seems so weird that a direct investigation will provide us with useful auxiliary tools for  analysis: To investigate this transformation more easily,  I will represent it in this way:
 
 \begin{equation*}
     \dfrac{s}{(\lambda +1)(\lambda +2)\cdots (\lambda +m)}= \left(x-\dfrac{m}{2}\right)^{\lambda}+\dfrac{\lambda(\lambda -1)}{1\cdot 2}P\left(x-\dfrac{m}{2}\right)^{\lambda -2}
 \end{equation*}
 \begin{equation*}
     +\dfrac{\lambda (\lambda -1)\cdots (\lambda -3)}{1 \cdot 2 \cdots 4}Q\left(x-\dfrac{m}{2}\right)^{\lambda -4}+\dfrac{\lambda (\lambda -1)\cdots (\lambda -6)}{1 \cdot 2\cdots 6}R\left(x-\dfrac{m}{2}\right)^{\lambda -6}
 \end{equation*}
 \begin{equation*}
     +\dfrac{\lambda (\lambda-1)\cdots (\lambda -7)}{1 \cdot 2 \cdots 7}S\left(x-\dfrac{m}{2}\right)^{\lambda -8}+\dfrac{\lambda (\lambda -1)\cdots (\lambda -9)}{1 \cdot 2 \cdots 10}T\left(x-\dfrac{m}{2}\right)^{\lambda -10}
 \end{equation*}
 \begin{equation*}
     \text{etc.},
 \end{equation*}
 for which expression up to this point I have found:
 
 \begin{equation*}
         \renewcommand{\arraystretch}{2.5}
\setlength{\arraycolsep}{0mm}
\begin{array}{lll}
     P &~=~& \dfrac{m}{3 \cdot 4} \\
     Q &~=~& \dfrac{m(5m-2)}{5 \cdot 6 \cdot 8} \\
     R &~=~& \dfrac{m(35mm-42m+16)}{6 \cdot 7 \cdot 96} \\
     S &~=~& \dfrac{m(175m^3-420m^2+404m-144)}{34560}
\end{array}
 \end{equation*}
 but a method to find the values of these letters more quickly is desired. 
 
\paragraph*{XVIII.}

But here it is especially helpful to have noted that our series is transformed into another one which  is a power series in $x-\frac{m}{2}$, the exponents being $\lambda$, $\lambda -2$, $\lambda -4$ etc. continuously decreasing by two;  but then the letters $P$, $Q$, $R$ etc. depend only on the  number $m$ such that  neither the exponent $\lambda$ nor the quantity $x$ enter it; furthermore,  the prefixed coefficients  involve the number $\lambda$ and follow the law of progression resulting from the expansion of the binomial. Having studied this form carefully, it becomes obvious that the values of the letters $P$, $Q$, $R$, $S$ etc.  can be found independently from the propounded series or from its transformed counterpart that we gave in paragraph XV, whose law of progression is known as well, if we set $x=\frac{m}{2}$; for, if one takes $\lambda =2$,

\begin{equation*}
    P=\dfrac{s}{(\lambda +1)(\lambda +2)\cdots (\lambda +m)},
\end{equation*}
but, having set $\lambda=4$,

\begin{equation*}
    Q=\dfrac{s}{(\lambda +1)(\lambda +2)\cdots (\lambda +m)}
\end{equation*}
but, having set $\lambda=6$,

\begin{equation*}
    R=\dfrac{s}{(\lambda +1)(\lambda +2)\cdots (\lambda +m)} \quad \text{etc.}
\end{equation*}

\paragraph*{XIX.}

Therefore, if  we substitute the series found above in paragraph XV. for

\begin{equation*}
    \dfrac{s}{(\lambda +1)(\lambda +2)\cdots (\lambda+m)}
\end{equation*}
here and, for the sake of brevity, we set:
 
 \begin{small}
 \begin{equation*}
         \renewcommand{\arraystretch}{2.2}
\setlength{\arraycolsep}{0mm}
\begin{array}{llllrllrllrllrl}
     \mathfrak{A}&~=~& 1 \\
     \mathfrak{B}&~=~& 1 &~+~& 3 ~\cdot ~& \dfrac{m-1}{4} \\
 \mathfrak{C}&~=~& 1 &~+~& 10 ~\cdot ~& \dfrac{m-1}{5}&~+~& 15 ~\cdot ~&\dfrac{m-1}{5}\cdot \dfrac{m-2}{6}\\
 \mathfrak{D}&~=~& 1 &~+~& 25 ~\cdot ~& \dfrac{m-1}{6}&~+~& 105 ~\cdot ~&\dfrac{m-1}{6}\cdot \dfrac{m-2}{7}&~+~& 105 ~\cdot ~ &\dfrac{m-1}{6}\cdot \dfrac{m-2}{7}\cdot \dfrac{m-3}{8}\\
 \mathfrak{E}&~=~& 1 &~+~& 56 ~\cdot ~& \dfrac{m-1}{7}&~+~& 490 ~\cdot ~&\dfrac{m-1}{7}\cdot \dfrac{m-2}{8}&~+~& 1260 ~\cdot ~ &\dfrac{m-1}{7}\cdots \dfrac{m-3}{9} &~+~& 945 ~\cdot ~&\dfrac{m-1}{7}\cdots \dfrac{m-4}{10}\\
  \mathfrak{F}&~=~& 1 &~+~& 119 ~\cdot ~& \dfrac{m-1}{8}&~+~& 1918 ~\cdot ~&\dfrac{m-1}{8}\cdot \dfrac{m-2}{9}&~+~& 9450 ~\cdot ~ &\dfrac{m-1}{8}\cdots \dfrac{m-3}{10} &~+~& 17325 ~\cdot ~&\dfrac{m-1}{8}\cdots \dfrac{m-4}{11}\\
   & &  & & & & & & & & & &~+~& 10395 ~\cdot ~&\dfrac{m-1}{8}\cdots \dfrac{m-5}{12}\\
   & &  &&  & & & &\text{etc.}
\end{array}
 \end{equation*}
 \end{small}
we obtain the following values 
 
 \begin{equation*}
          \renewcommand{\arraystretch}{2.5}
\setlength{\arraycolsep}{0mm}
\begin{array}{llllrllrllrllrllrllr}
P &~=~& \dfrac{m^2}{2^2}&~-~& 2 \mathfrak{A}~\cdot ~ &\dfrac{m}{2}\cdot \dfrac{m}{2}&~+~& \mathfrak{B} ~\cdot ~ &\dfrac{m}{3}\\
Q &~=~& \dfrac{m^4}{2^4}&~-~& 4 \mathfrak{A}~\cdot ~ &\dfrac{m}{2}\cdot \dfrac{m^3}{2^3}&~+~& 6 \mathfrak{B} ~\cdot ~ &\dfrac{m}{3} \cdot \dfrac{m^2}{2^2}&~-~& 4 \mathfrak{C}~\cdot~& \dfrac{m}{4}\cdot \dfrac{m}{2}&~+~&\mathfrak{D}~\cdot ~&\dfrac{m}{5}\\
R &~=~& \dfrac{m^6}{2^6}&~-~& 6 \mathfrak{A}~\cdot ~ &\dfrac{m}{2}\cdot \dfrac{m^5}{2^5}&~+~& 15 \mathfrak{B} ~\cdot ~ &\dfrac{m}{3} \cdot \dfrac{m^4}{2^4}&~-~& 20 \mathfrak{C}~\cdot~& \dfrac{m}{4}\cdot \dfrac{m^3}{2^3}&~+~&\mathfrak{D}~\cdot~ &\dfrac{m}{5}\cdot \dfrac{m^2}{2^2}-6 \mathfrak{E}\cdot \dfrac{m}{6} \cdot \dfrac{m}{2}+\mathfrak{F}\cdot \dfrac{m}{7}\\
 & & \text{etc.}
\end{array}
 \end{equation*}
 whence it will be convenient to expand the values of those letters $\mathfrak{A}$, $\mathfrak{B}$, $\mathfrak{C}$, $\mathfrak{D}$ etc., whence it results:
 
 \begin{equation*}
             \renewcommand{\arraystretch}{2.5}
\setlength{\arraycolsep}{0mm}
\begin{array}{llllcl}
     &\mathfrak{A} &~=~& 1 \\
     &\mathfrak{B} &~=~& \dfrac{3}{4}m +\dfrac{1}{4}= \dfrac{3}{4}\left(m+\dfrac{1}{3}\right) \\
     &\mathfrak{C} &~=~& \dfrac{1}{2}m^2+ \dfrac{1}{2}m=\dfrac{4}{8}(mm+m)\\
     &\mathfrak{D} &~=~& \dfrac{5}{16}m^3+\dfrac{5}{8}m^2+\dfrac{5}{48}m -\dfrac{1}{24}=\dfrac{5}{16}\left(m^3+2m^2+\dfrac{1}{3}m-\dfrac{2}{15}\right)\\
     &\mathfrak{E}&~=~& \dfrac{3}{16}m^4+\dfrac{5}{8}m^3+\dfrac{5}{16}m^2-\dfrac{1}{8}m = \dfrac{6}{32}\left(m^4+\dfrac{10}{3}m^3+\dfrac{5}{3}m^2-\dfrac{2}{3}m\right) \\
     &\mathfrak{F}&~=~& \dfrac{7}{64}m^5+\dfrac{35}{64}m^4+\dfrac{35}{64}m^3-\dfrac{91}{576}m^2-\dfrac{7}{96}m+\dfrac{1}{36} \\
     \text{or} \quad \\
      &\mathfrak{F}&~=~& \dfrac{7}{64}\left(m^5+5m^4+5m^3-\dfrac{13}{9}m^2-\dfrac{2}{3}m+\dfrac{16}{63}\right);
\end{array}
 \end{equation*}
 but here, aside from the first terms, no structure is seen.
 
 \paragraph*{XX.}
 
  That the transformed series is a power series in $x-\frac{m}{2}$, was based only on induction, but it can be shown to happen necessarily in this way.  Since the propounded progression ends as it begins such that the last two terms will be
 
 \begin{equation*}
     \pm m(x-m+1)^{m+\lambda}\mp (x-m)^{m+\lambda},
 \end{equation*}
 where the upper signs hold, if $m$ is an odd number, the lower on the other hand, if  $m$ is even, let us assume that $m$ is an even number (for, the same conclusion follows, if it was odd) and set $x-\frac{m}{2}=y$, and it will be

 \begin{equation*}
          \renewcommand{\arraystretch}{2.5}
\setlength{\arraycolsep}{0mm}
\begin{array}{llllllllllllll}
     2s & ~=~ & ~+~ & \left(y+\dfrac{1}{2}m\right)^{m+\lambda}&~-~&\dfrac{m}{1}\left(y+\dfrac{1}{2}m-1\right)^{m+\lambda}&~+~&\dfrac{m(m-1)}{1\cdot 2}\left(y+\dfrac{1}{2}m-2\right)^{m+\lambda}&~-~&\text{etc.} \\
       &  & ~+~ & \left(y-\dfrac{1}{2}m\right)^{m+\lambda}&~-~&\dfrac{m}{1}\left(y-\dfrac{1}{2}m-1\right)^{m+\lambda}&~+~&\dfrac{m(m-1)}{1\cdot 2}\left(y-\dfrac{1}{2}m-2\right)^{m+\lambda}&~-~&\text{etc.} \\
\end{array}
 \end{equation*}
 and after the expansion into powers of $y=x-\frac{m}{2}$ one finds:
 
 \begin{equation*}
 \begin{array}{c}
 s= y^{m+\lambda}\left(1-\dfrac{m}{1}+\dfrac{m(m-1)}{1 \cdot 2}-\text{etc.}\right) \\
       \renewcommand{\arraystretch}{2.5}
\setlength{\arraycolsep}{0mm}
\begin{array}{llllllllllllll}
     ~+~& \left(\dfrac{m+\lambda}{2}\right)y^{m+\lambda -2}&\bigg(\left(\dfrac{m}{2}\right)^2&~-~&\dfrac{m}{1}\left(\dfrac{m}{2}-1\right)^2&~+~&\dfrac{m(m-1)}{1 \cdot 2}\left(\dfrac{m}{2}-2\right)^2&~-~&\text{etc.}\bigg) \\
      ~+~& \left(\dfrac{m+\lambda}{4}\right)y^{m+\lambda -4}&\bigg(\left(\dfrac{m}{2}\right)^4&~-~&\dfrac{m}{1}\left(\dfrac{m}{2}-1\right)^4&~+~&\dfrac{m(m-1)}{1 \cdot 2}\left(\dfrac{m}{2}-2\right)^4&~-~&\text{etc.}\bigg) \\
\end{array}\\
\text{etc.}
\end{array}
 \end{equation*}
 But all these series vanish, until one gets to the one in which the exponents are $m$, and we know its sum to be $=1 \cdot 2 \cdot 3 \cdots m$; therefore, having omitted those, whose sum becomes zero, we will obtain:
 
 \begin{equation*}
       \renewcommand{\arraystretch}{2.5}
\setlength{\arraycolsep}{0mm}
\begin{array}{llllllllllllll}
     s &~=~& \left(\dfrac{m+\lambda}{m}\right)y^{\lambda}&\bigg(\left(\dfrac{m}{2}\right)^m &~-~& \dfrac{m}{1}\left(\dfrac{m}{2}-1\right)^m&~+~& \dfrac{m(m-1)}{1 \cdot 2}\left(\dfrac{m}{2}-2\right)^m&~-~&\text{etc.}\bigg)\\
        &~ +~& \left(\dfrac{m+\lambda}{m+2}\right)y^{\lambda -2}&\bigg(\left(\dfrac{m}{2}\right)^{m+2} &~-~& \dfrac{m}{1}\left(\dfrac{m}{2}-1\right)^{m+2}&~+~& \dfrac{m(m-1)}{1 \cdot 2}\left(\dfrac{m}{2}-2\right)^{m+2}&~-~&\text{etc.}\bigg)\\
        & & &  &   &\text{etc.}
\end{array}
 \end{equation*}
 and thus it is manifest, what I tried to demonstrate, that this series  descends in the powers $y^{\lambda}$, $y^{\lambda -2}$, $y^{\lambda -4}$ etc.
 
 \paragraph*{XXI.}
 
 Now let us attribute a form to the series similar to that we had in paragraph XVII., and it will be
 
 \begin{equation*}
        \renewcommand{\arraystretch}{2.5}
\setlength{\arraycolsep}{0mm}
\begin{array}{rlllllllllllll}
     \dfrac{s}{(\lambda +1)(\lambda +2)\cdots (\lambda +m)}= \dfrac{y^{\lambda}}{1 \cdot 2 \cdots m}&\bigg(\left(\dfrac{m}{2}\right)^{m}&~-~&\dfrac{m}{1}\left(\dfrac{m}{2}-1\right)^{m}&~+~&\text{etc.}\bigg)\\
     +\dfrac{1 \cdot 2 y^{\lambda -2}}{1 \cdot 2 \cdots (m+2)}\cdot \dfrac{\lambda (\lambda -1)}{1 \cdot 2}&\bigg(\left(\dfrac{m}{2}\right)^{m+2}&~-~&\dfrac{m}{1}\left(\dfrac{m}{2}-1\right)^{m+2}&~+~&\text{etc.}\bigg)\\
     +\dfrac{1 \cdot 2 \cdot 3 \cdot 4 y^{\lambda -4}}{1 \cdot 2 \cdots (m+4)}\cdot \dfrac{\lambda (\lambda -1)\cdots (\lambda -3)}{1 \cdot 2 \cdot 3 \cdot 4}&\bigg(\left(\dfrac{m}{2}\right)^{m+4}&~-~&\text{etc.}\bigg) \\
      & \text{etc.},
\end{array}
 \end{equation*}
 whence  the values of the letters $P$, $Q$, $R$ etc. can be determined in a new way as follows:
 
 \begin{equation*}
          \renewcommand{\arraystretch}{2.5}
\setlength{\arraycolsep}{0mm}
\begin{array}{llcllllllllllllll}
     P &~=~& \dfrac{1}{3 \cdot 4 \cdots (m+2)}&\bigg(\left(\dfrac{m}{2}\right)^{m+2}&~-~&\dfrac{m}{1}\left(\dfrac{m}{2}-1\right)^{m+2}&~+~&\text{etc.}\bigg)\\
      Q &~=~& \dfrac{1}{5 \cdot 6 \cdots (m+4)}&\bigg(\left(\dfrac{m}{2}\right)^{m+4}&~-~&\dfrac{m}{1}\left(\dfrac{m}{2}-1\right)^{m+4}&~+~&\text{etc.}\bigg)\\
 R &~=~& \dfrac{1}{7 \cdot 8 \cdots (m+6)}&\bigg(\left(\dfrac{m}{2}\right)^{m+6}&~-~&\dfrac{m}{1}\left(\dfrac{m}{2}-1\right)^{m+6}&~+~&\text{etc.}\bigg)\\      
  S &~=~& \dfrac{1}{9 \cdot 10 \cdots (m+8)}&\bigg(\left(\dfrac{m}{2}\right)^{m+8}&~-~&\dfrac{m}{1}\left(\dfrac{m}{2}-1\right)^{m+8}&~+~&\text{etc.}\bigg)\\
   & &\text{etc.}
   \end{array}
 \end{equation*}
 Here, certainly the summation of similar series is necessary; since these only involve the number $m$, our investigation is to be considered to be reduced to a simpler case. Furthermore,  we realise just now that these letters depend only on the number $m$.
 
 \paragraph*{XXII.}
 
 But if  we successively attribute here the definite values $1$, $2$, $3$, $4$, $5$, $6$ etc.  to the letter $m$, we will  obtain as many values for the letters $P$, $Q$, $R$, $S$ etc., knowing which one can easily conclude their general forms. Thus, to find the letter $P$, we will have
 
 \begin{equation*}
     \renewcommand{\arraystretch}{1.3}
\setlength{\arraycolsep}{0mm}
\begin{array}{rlcccccccccccccccccccccccc}
     \text{if } m &~=~& ~0,~ & & ~1,~ & &~2, ~ & & ~3, ~& & ~4, ~ & \text{etc.} \\
     3 \cdot 2^2 P &~=~&~ 0,~ & & ~1,~ & &~2, ~ & &~3, ~ & & ~4, ~ & \text{etc.} \\
 \text{diff} & & &~1, ~ & & ~ 1, ~ & &~1, ~ & & 1
\end{array}
 \end{equation*}
 such that hence $3 \cdot 2^2 P =m$ and $P=\frac{m}{2^2 \cdot 3}$ as before. Further, for the letter $Q$,
 
 \begin{equation*}
     \renewcommand{\arraystretch}{1.3}
\setlength{\arraycolsep}{0mm}
\begin{array}{rcccccccccccccccccccccccccccc}
     \text{if }m &~=~& ~ 0,~ & & ~1,~ & & ~2, ~ & &~3, ~ & &~4, ~ & & ~5, ~ & & ~6 ~ \\
     2^4 \cdot 3 \cdot 5 Q &~=~& ~0, ~ & & ~3, ~ & & ~16, ~ & & ~39, ~ & & ~72, ~ & & ~115, ~ & & ~168 ~ \\
     \text{diff. I.} & &  & ~ 3, ~ & & ~13, ~ & & ~23, ~ & &~33, ~ & & ~ 43, ~ & & ~53 ~ \\
     \text{diff. II.} &  & & & ~ 10,~ & & ~10, ~ & & ~10, ~ & & ~10, ~ & & ~ 10, ~ 
\end{array}
 \end{equation*}
 therefore, it will be
 
 \begin{equation*}
     2^4 \cdot 3 \cdot 5Q + 3m +10 \dfrac{m(m-1)}{1 \cdot 2}+m(5m-2)
 \end{equation*}
 and hence
 
 \begin{equation*}
     Q=\dfrac{m(5m-2)}{2^4 \cdot 3 \cdot 5}.
 \end{equation*}
 In like manner, for the letter $R$,
 
  \begin{equation*}
     \renewcommand{\arraystretch}{1.3}
\setlength{\arraycolsep}{0mm}
\begin{array}{rcccccccccccccccccccccccccccc}
     \text{if }m &~=~& ~ 0,~ & & ~1,~ & & ~2, ~ & &~3, ~ & &~4, ~ & & ~5, ~ & & ~6 ~ \\
     2^6 \cdot 3 \cdot 7 R &~=~& ~0, ~ & & ~3, ~ & & ~48, ~ & & ~205, ~ & & ~544, ~ & & ~1135, ~ & & ~2048 ~ \\
     \text{diff. I.} & &  & ~ 3, ~ & & ~45, ~ & & ~157, ~ & &~339, ~ & & ~ 591, ~ & & ~913 ~ \\
     \text{diff. II.} &  & & & ~ 42,~ & & ~112, ~ & & ~182, ~ & & ~252, ~ & & ~ 322, ~ \\
      \text{diff. III.} & &  &  & & ~70, ~ & & ~70, ~ & &~70, ~ & & ~ 70, ~ & & \\
\end{array}
 \end{equation*}
 whence one concludes
 
 \begin{equation*}
     2^6 \cdot 3 \cdot 7 R =3m +21m(m-1)+\dfrac{35}{3}m(m-1)(m-2)
 \end{equation*}
 and
 
 \begin{equation*}
     R=\dfrac{m(35m^2-42m+16)}{2^6\cdot 3^2 \cdot 7},
 \end{equation*}
 which same values we obtained above already;  therefore, let us apply the same operation to he following letters.
 
 \paragraph*{XXIII.}
 
 Therefore, for the letter $S$ we will have:
 
  \begin{equation*}
     \renewcommand{\arraystretch}{1.3}
\setlength{\arraycolsep}{0mm}
\begin{array}{rcccccccccccccccccccccccccccc}
     \text{if }m &~=~&  0, & & 1, & & 2,  & &3,  & &4,  & & 5,  & & 6  \\
     2^8 \cdot 5 \cdot 9 S &~=~& 0,  & & 5,  & & 256,  & & 2013,  & & 7936,  & & 22085,  & & 49920  \\
     \text{diff. I.} & &  &  5,  & & 251,  & & 1757,  & &~5923,  & & 14149,  & & ~27835 ~ \\
     \text{diff. II.} &  & & &  246, & & 1506,  & & 4166,  & & 8226,  & &  13686,  \\
      \text{diff. III.} & &  &  & & 1260,  & & 2660,  & &4060,  & &  5460,  & &  \\
       \text{diff. IV.} &  & & &  & & 1400,  & & 1400,  & & 1400  & & \\
\end{array}
 \end{equation*}
 whence
 
 \begin{equation*}
     2^8 \cdot 5 \cdot 9 S = 5m+123m(m-1)+210m(m-1)(m-2)+\dfrac{175}{3}m(m-1)(m-2)(m-3)
 \end{equation*}
 Now further, for the letter $T$ we will have:
 
  \begin{equation*}
     \renewcommand{\arraystretch}{1.3}
\setlength{\arraycolsep}{0mm}
\begin{array}{rcccccccccccccccccccccccccccc}
     \text{if }m &~=~&  0, & & 1, & & 2,  & &3,  & &4, & & 5, & & 6 \\
     2^{10} \cdot 3 \cdot 11 T &~=~& 0,  & & 3,  & & 512,  & & 7665,  & & 4680,  & & 174255,  & & 499968  \\
     \text{diff. I.} & &  & 3,  & & 509,  & & 7153,  & &38415,  & &  128175,  & & 325713  \\
     \text{diff. II.} &  & & &  506, & & 6604,  & & 31262,  & & 89760,  & &  197538,  \\
      \text{diff. III.} & &  &  & & 6138,  & & 24618,  & &58498,  & &  107778,  & &  \\
       \text{diff. IV.} &  & & &  & & 18480, & & 33880,  & & 49280,  & &  \\
         \text{diff. V.} & &  &  & &  & & 15400, & & 15400 & &  & &  \\
\end{array}
 \end{equation*}
whence

\begin{equation*}
       \renewcommand{\arraystretch}{2.5}
\setlength{\arraycolsep}{0mm}
\begin{array}{lll}
     2^{10}\cdot 3 \cdot 11 T &~=~& 3m +253m(m-1)+1023m(m-1)(m-2) \\
                              &~+~& 770m(m-1)(m-2)(m-3) \\
                              &~+~&\dfrac{385}{3}m(m-1)(m-2)(m-3)(m-4)
\end{array}
\end{equation*}
and

\begin{equation*}
    T=\dfrac{m(385m^4-1540m^3+2684m^2-2288m+768)}{2^{10}\cdot 9 \cdot 11}.
\end{equation*}

\paragraph*{XXIV.}
 Now let us represent these values in such a way that the law of progression can be explored more easily:
 
 \begin{equation*}
       \renewcommand{\arraystretch}{2.5}
\setlength{\arraycolsep}{0mm}
\begin{array}{lll}
     P &~=~& \dfrac{1 m}{12} \\
     Q &~=~& \dfrac{1 \cdot 3 m}{12^2}\left(m-\dfrac{2}{5}\right) \\
     R &~=~& \dfrac{1 \cdot 3 \cdot 5m}{12^3}\left(m^2-\dfrac{6}{5}m+\dfrac{16}{35}\right) \\
     S &~=~& \dfrac{1 \cdot 3 \cdot 5 \cdot 7 m}{12^4}\left(m^3-\dfrac{12}{5}m+\dfrac{404}{175}m-\dfrac{144}{175}\right) \\
     T &~=~& \dfrac{1 \cdot 3 \cdot 5 \cdot 7 \cdot 9 m}{12^5}\left(m^4-\dfrac{20}{5}m^3+\dfrac{244}{35}m^2-\dfrac{208}{35}m+\dfrac{768}{385}\right)
\end{array}
 \end{equation*}
 and here in the first and second terms the law of progression is so manifest that the same  can safely be assigned for all following letters, but in the remaining terms one can still not observe any law.
 
 \paragraph*{XXV.}
 
 Therefore, to find the value of the letter $V$, let us set

 \begin{equation*}
     V= \dfrac{1 \cdot 3 \cdot 5 \cdot 7 \cdot 9 \cdot 11 m}{12^6}\left(m^5-\dfrac{30}{5}m^4+\alpha m^3-\beta m^2 +\gamma m -\delta\right).
 \end{equation*}
But from the general form
 
 \begin{equation*}
     V=\dfrac{1}{13 \cdot 14 \cdots (m+12)}\left(\left(\dfrac{m}{2}\right)^{m+12}-\left(\dfrac{m}{2}-1\right)^{m+12}+\dfrac{m(m-1)}{1\cdot 2}\left(\dfrac{m}{2}-2\right)^{m+12}-\text{etc.}\right)
 \end{equation*}
 we conclude that:
 
 \begin{equation*}
          \renewcommand{\arraystretch}{2.5}
\setlength{\arraycolsep}{0mm}
\begin{array}{lllccllrlrlrlrlrl}
     \text{if}  & & & & \text{it will be} \\
     m =1, \quad & V &~=~& \dfrac{1}{2^{12}\cdot 13} &~=~& \dfrac{5 \cdot 7 \cdot 11}{2^{12}\cdot 3^3}&(~-~&5 &~+~& \alpha &~-~& \beta &~+~& \gamma &~-~& \delta) \\
       m =2, \quad & V &~=~& \dfrac{1}{7\cdot 13} &~=~& \dfrac{5 \cdot 7 \cdot 11}{2^{11}\cdot 3^3}&(~-~&64 &~+~& 8\alpha &~-~& 4\beta &~+~& 2\gamma &~-~& \delta) \\
        m =3, \quad & V &~=~& \dfrac{597871}{2^{12}\cdot 5 \cdot 7 \cdot 13} &~=~& \dfrac{5 \cdot 7 \cdot 11}{2^{12}\cdot 3^2}&(~-~&243 &~+~& 27\alpha &~-~& 9\beta &~+~& 3\gamma &~-~& \delta) \\
         m =4, \quad & V &~=~& \dfrac{5461}{2^{2}\cdot 5 \cdot 7 \cdot 13} &~=~& \dfrac{5 \cdot 7 \cdot 11}{2^{10}\cdot 3^3}&(~-~&512 &~+~& 64\alpha &~-~& 16\beta &~+~& 4\gamma &~-~& \delta) \\
 m =5, \quad & V &~=~& \dfrac{5838647}{2^{12}\cdot 7 \cdot 13} &~=~& \dfrac{25 \cdot 7 \cdot 11}{2^{12}\cdot 3^3}&(~-~&625 &~+~& 125\alpha &~-~& 25\beta &~+~& 5\gamma &~-~& \delta) \\    
 m =6, \quad & V &~=~& \dfrac{63047}{2^{2}\cdot 3 \cdot 7 \cdot 13} &~=~& \dfrac{5 \cdot 7 \cdot 11}{2^{11}\cdot 3^2}&(~-~&0 &~+~& 216\alpha &~-~& 26\beta &~+~& 6\gamma &~-~& \delta) \\    
\end{array}
 \end{equation*}
 Therefore,  let us form the following equations:
 
 \begin{equation*}
            \renewcommand{\arraystretch}{2.5}
\setlength{\arraycolsep}{0mm}
\begin{array}{rlrlrlrlclr}
     \alpha &~-~& \beta &~+~& \gamma &~-~& \delta &~=~& \dfrac{27}{5 \cdot 7 \cdot 11 \cdot 13} &~+~& 5 \\
     8\alpha &~-~& 4\beta &~+~& 2\gamma &~-~& \delta &~=~& \dfrac{27 \cdot 2048}{5 \cdot 7^2 \cdot 11 \cdot 13} &~+~& 64 \\
     27\alpha &~-~& 9\beta &~+~& 3\gamma &~-~& \delta &~=~& \dfrac{9 \cdot 597871}{5^2 \cdot 7^2 \cdot 11 \cdot 13} &~+~& 243 \\
     64\alpha &~-~& 16\beta &~+~& 4\gamma &~-~& \delta &~=~& \dfrac{27 \cdot 256 \cdot 5461}{5^2 \cdot 7^2 \cdot 11 \cdot 13} &~+~& 512 \\
     125\alpha &~-~& 25\beta &~+~& 5\gamma &~-~& \delta &~=~& \dfrac{27 \cdot 5838647}{5^2 \cdot 7^2 \cdot 11 \cdot 13} &~+~& 625 \\
     216\alpha &~-~& 36\beta &~+~& 6\gamma &~-~& \delta &~=~& \dfrac{3 \cdot 512 \cdot 63047}{5 \cdot 7^2 \cdot 11 \cdot 13} &~+~& 0 \\
\end{array}
 \end{equation*}
 Now the first differences will look as follows:
 
 \begin{equation*}
       \renewcommand{\arraystretch}{2.5}
\setlength{\arraycolsep}{0mm}
\begin{array}{rlrlrlclr}
     7 \alpha &~-~& 3 \beta &~+~& \gamma &~=~& \dfrac{27 \cdot 157}{5 \cdot 7^2 \cdot 11} &~+~& 59 \\
      19 \alpha &~-~& 5 \beta &~+~& \gamma &~=~& \dfrac{9 \cdot 43627}{5^2 \cdot 7^2 \cdot 11} &~+~& 179 \\
       37 \alpha &~-~& 7 \beta &~+~& \gamma &~=~& \dfrac{9 \cdot 276629}{5^2 \cdot 7^2 \cdot 11} &~+~& 269 \\
        61 \alpha &~-~& 9 \beta &~+~& \gamma &~=~& \dfrac{27 \cdot 341587}{5^2 \cdot 7^2 \cdot 11} &~+~& 113 \\
         91 \alpha &~-~& 11 \beta &~+~& \gamma &~=~& \dfrac{3 \cdot 8373269}{5^2 \cdot 7^2 \cdot 11} &~-~& 625 \\
\end{array}
 \end{equation*}
 the second differences, divided by $2$, on the other hand give
 
 \begin{equation*}
          \renewcommand{\arraystretch}{2.5}
\setlength{\arraycolsep}{0mm}
\begin{array}{rlrlclr}
     6 \alpha &~-~& \beta &~=~& \dfrac{9 \cdot 268}{5^2 \cdot 7} &~+~& 60 \\
      9 \alpha &~-~& \beta &~=~& \dfrac{9 \cdot 1513}{5^2 \cdot 7} &~+~& 45 \\
       12 \alpha &~-~& \beta &~=~& \dfrac{9 \cdot 4858}{5^2 \cdot 7} &~-~& 78 \\
        15 \alpha &~-~& \beta &~=~& \dfrac{3 \cdot 34409}{5^2 \cdot 7} &~-~& 369 \\
\end{array}
 \end{equation*}
Finally, the third differences, divided by $3$, yield
 
 \begin{equation*}
     \alpha = \dfrac{3 \cdot 249}{5 \cdot 7}-5=\dfrac{3 \cdot 669}{5 \cdot 7}-41=\dfrac{3967}{5 \cdot 7}-97,
 \end{equation*}
 which three equations give the same value
 
 \begin{equation*}
     \alpha = \dfrac{572}{5 \cdot 7}= \dfrac{4 \cdot 11 \cdot 13}{5 \cdot 7},
 \end{equation*}
 from which value  the remaining ones are now defined as follows:
 
 \begin{equation*}
         \renewcommand{\arraystretch}{2.5}
\setlength{\arraycolsep}{0mm}
\begin{array}{lll}
     \beta &~=~& \dfrac{6 \cdot 572}{5 \cdot 7}-\dfrac{9 \cdot 268}{5^2 \cdot 7}-60= \dfrac{12 \cdot 1229}{5^2 \cdot 7}-60 =\dfrac{4248}{175}=\dfrac{8 \cdot 9 \cdot 59}{175} \\
     \gamma &~=~& 3 \beta - 7\alpha +\dfrac{27 \cdot 157}{5 \cdot 7^2 \cdot 11}+59 =\dfrac{255968}{5^2 \cdot 7^2 \cdot 11} \\
     \delta &~=~& \alpha -\beta +\gamma -\dfrac{27}{5 \cdot 7 \cdot 11 \cdot 13}-5 =\dfrac{1061376}{5^2 \cdot 7^2 \cdot 11 \cdot 13}.
\end{array}
 \end{equation*}
 
 \paragraph*{XXVI.}
 
 Thus, let us list up the values of the letters $P$, $Q$, $R$ etc. found up to this point all at once
 
 \begin{equation*}
         \renewcommand{\arraystretch}{2.5}
\setlength{\arraycolsep}{0mm}
\begin{array}{lll}
     P &~=~& \dfrac{1m}{12} \\
     Q &~=~& \dfrac{1 \cdot 3m}{12^2}\left(m-\dfrac{2}{5}\right) \\
     R &~=~& \dfrac{1 \cdot 3 \cdot 5 m}{12^3}\left(m^2-\dfrac{6}{5}m+\dfrac{16}{35}\right) \\
     S &~=~& \dfrac{1 \cdot 3 \cdot 5 \cdot 7 m}{12^4} \left(m^3-\dfrac{12}{5}m^2+\dfrac{404}{175}m-\dfrac{144}{175}\right) \\
      T &~=~& \dfrac{1 \cdot 3 \cdot 5 \cdot 7 \cdot 9 m}{12^5}\left(m^4-\dfrac{20}{5}m^3+\dfrac{244}{35}m^2-\dfrac{208}{35}m+\dfrac{768}{385}\right) \\
      V &~=~& \dfrac{1 \cdot 3 \cdot 5 \cdot 7 \cdot 9 \cdot 11m}{12^6}\left(m^5-\dfrac{30}{5}m^4+\dfrac{572}{35}m^3-\dfrac{4248}{175}m^2+\dfrac{255968}{13475}m-\dfrac{1061376}{175175}\right).
\end{array}
 \end{equation*}
 From the first terms I conclude that  powers occur here, after the separation of which it seems the structure can be recognised more clearly:

 \begin{equation*}
        \renewcommand{\arraystretch}{2.5}
\setlength{\arraycolsep}{0mm}
\begin{array}{lll}
     P &~=~& \dfrac{1m}{12} \\
     Q &~=~& \dfrac{1 \cdot 3m}{12^2}\left(m-\dfrac{2}{5}\right)^1\\
     R &~=~& \dfrac{1 \cdot 3 \cdot 5 m}{12^4}\left(\left(m-\dfrac{3}{5}\right)+\dfrac{17}{175}\right)\\
       S &~=~& \dfrac{1 \cdot 3 \cdot 5 \cdot 7m}{12^4}\left(\left(m-\dfrac{4}{5}\right)^3+\dfrac{4 \cdot 17}{175}m-\dfrac{16 \cdot 17}{5 \cdot 175 }\right)\\
         T &~=~& \dfrac{1 \cdot 3 \cdot 5 \cdot 7 \cdot 9 m}{12^5}\left(\left(m-\dfrac{5}{5}\right)^4+\dfrac{2 \cdot 17}{35}m^2-\dfrac{4 \cdot 17}{35 }m+\dfrac{383}{385}\right)\\
           V &~=~& \dfrac{1 \cdot 3 \cdot 5 \cdot 7 \cdot 9  \cdot 11 m}{12^6}\left(\left(m-\dfrac{6}{5}\right)^5+\dfrac{4 \cdot 17}{35}m^3-\dfrac{72 \cdot 17}{175 }m^2+\dfrac{581296}{5^3\cdot 7^2 \cdot 11}m-\dfrac{78185568}{5^5\cdot 7^2 \cdot 11 \cdot 13}\right),
\end{array}
 \end{equation*}
yes, it even seems that  the second terms can be approximately contracted in this way so that it results:

  \begin{center}
 \begin{equation*}
       \renewcommand{\arraystretch}{2.5}
\setlength{\arraycolsep}{0mm}
\begin{array}{lll}
     P &~=~& \dfrac{1m}{12}\cdot 1 \\
     Q &~=~& \dfrac{1 \cdot 3m}{12^2}\left(m-\dfrac{2}{5}\right) \\
     R &~=~& \dfrac{1 \cdot 3 \cdot 5 m}{12^3}\left(\left(m-\dfrac{3}{5}\right)^2+\dfrac{17}{175}\right) \\
     S &~=~& \dfrac{1 \cdot 3 \cdot 5 \cdot 7 m}{12^4}\left(\left(m-\dfrac{4}{5}\right)^3+\dfrac{4 \cdot 17}{175}\left(m-\dfrac{4}{5}\right)\right) \\
     T &~=~& \dfrac{1 \cdot 3 \cdot 5 \cdot 7 \cdot 9 m}{12^5}\left(\left(m-\dfrac{5}{5}\right)^4+\dfrac{10 \cdot 17}{175}\left(m-\dfrac{5}{5}\right)^2+\dfrac{9}{385}\right) \\
     V &~=~& \dfrac{1 \cdot 3 \cdot 5 \cdot 7 \cdot 9 \cdot 11m}{12^6}\left(\left(m-\dfrac{6}{5}\right)^5+\dfrac{20 \cdot 17}{175}\left(m-\dfrac{6}{5}\right)^3+\dfrac{15808}{5^3\cdot 7^2 \cdot 11}m-\dfrac{4672128}{5^5 \cdot 7^2 \cdot 11 \cdot 13}\right).
\end{array}
 \end{equation*}
 \end{center}
 If we had not found the last value, it would seem that all these expressions are reduced to powers of this kind, what we now have to admit not to happen. Therefore, one must  investigate the law of these letters from another source.
 
 \paragraph*{XXVIII.}
 
 Therefore, let us rather represent each term of these formulas in this way:

 \begin{equation*}
       \renewcommand{\arraystretch}{2.5}
\setlength{\arraycolsep}{0mm}
\begin{array}{lll}
     P &~=~& \dfrac{m}{4 \cdot 3} \\
     Q &~=~& \dfrac{mm}{16 \cdot 3} -\dfrac{m}{8 \cdot 3 \cdot 5} \\
     R &~=~& \dfrac{5m^3}{64 \cdot 9}-\dfrac{mm}{32 \cdot 3} +\dfrac{m}{4 \cdot 9 \cdot 7} \\
     S &~=~& \dfrac{5 \cdot 7 m^4}{256 \cdot 27}-\dfrac{7m^3}{64 \cdot 9}+\dfrac{101m^2}{64 \cdot 27 \cdot 5}-\dfrac{m}{16 \cdot 3 \cdot 5} \\
     T &~=~& \dfrac{5 \cdot 7 m^5}{1024 \cdot 9}-\dfrac{5 \cdot 7 m^4}{256 \cdot 9}+\dfrac{61m^3}{256 \cdot 9}-\dfrac{13m^2}{64 \cdot 9}+\dfrac{m}{4 \cdot 3 \cdot 1} \\
     V&~=~& \dfrac{5 \cdot 7 \cdot 11 m^6}{4096 \cdot 27}-\dfrac{5 \cdot 7 \cdot 11 m^5}{2048 \cdot 9}+\dfrac{1573 m^4}{1024 \cdot 27}-\dfrac{649m^3}{512 \cdot 3 \cdot 5}+\dfrac{7999m^2}{128 \cdot 27 \cdot 5 \cdot 7}-\dfrac{691 m}{8 \cdot 9 \cdot 5 \cdot 7 \cdot 13},
\end{array}
 \end{equation*}
 where we have already noticed the structure in the first and second terms,  but the last terms seemed to have no structure at all, until, having expanded the value of the letter $V$, the number $691$ provided us with a criterion  that the last terms contain the Bernoulli numbers.\\
 Therefore, let us denote the Bernoulli numbers by the letters $\alpha$, $\beta$, $\gamma$, $\delta$ etc. such that
 
 \begin{equation*}
     \alpha =\dfrac{1}{2}, \quad \beta =\dfrac{1}{6}, \quad \gamma =\dfrac{1}{6}, \quad \delta =\dfrac{3}{10}, \quad \varepsilon =\dfrac{5}{6}, \quad \zeta = \dfrac{691}{210} \quad \text{etc.}
 \end{equation*}
 and, concerning these, let us note this law of progression:
 
 \begin{equation*}
          \renewcommand{\arraystretch}{2.5}
\setlength{\arraycolsep}{0mm}
\begin{array}{lllllllllllllllll}
     \alpha &~=~& \dfrac{1}{2^1} \\
     \beta &~=~& \dfrac{5 \cdot 4 \alpha}{2^2 \cdot 1 \cdot 2 \cdot 3}&~-~&\dfrac{2}{2^3}\\
     \gamma &~=~& \dfrac{7 \cdot 6 \beta}{2^2 \cdot 1\cdot 2 \cdot 3}&~-~&\dfrac{7 \cdot 6 \cdot 5 \cdot 4 \alpha}{2^4 \cdot 1 \cdot 2 \cdots 5}&~+~&\dfrac{3}{2^5} \\
     \delta &~=~& \dfrac{9 \cdot 8 \gamma}{2^2 \cdot 1 \cdot 2 \cdot 3}&~-~&\dfrac{9 \cdot 8 \cdot 7 \cdot 6 \cdot \beta}{2^4 \cdot 1 \cdot 2 \cdots 5}&~+~& \dfrac{9 \cdot 4 \alpha}{2^6\cdot 1 \cdots 7}&~-~&\dfrac{4}{2^7} \\
     \varepsilon &~=~& \dfrac{10 \cdot 11 \delta}{2^2 \cdot 1 \cdot 2 \cdot 3}&~-~&\dfrac{11 \cdots 8 \gamma}{2^4 \cdot 1 \cdots 5}&~+~&\dfrac{11 \cdots 6 \beta}{2^6 \cdot 1 \cdots 7}&~-~&\dfrac{11 \cdots 4 \alpha}{2^8 \cdot 1 \cdots 9}&~+~& \dfrac{5}{2^9} \\
     \zeta &~=~& \dfrac{13 \cdot 12 \varepsilon}{2^2 \cdot 1 \cdot 2 \cdot 3}&~-~&\dfrac{13 \cdots 10 \delta}{2^4 \cdot 1 \cdots 5}&~+~&\dfrac{13 \cdots 8 \gamma}{2^6 \cdot 1 \cdots 7}&~-~&\dfrac{13 \cdots 6 \beta}{2^8 \cdot 1 \cdots 9}&~+~&\dfrac{13 \cdots 4 \alpha}{2^{10}\cdot 1 \cdot 11}-\dfrac{6}{2^{11}} \\
     &  & &  & \text{etc.}
\end{array}
 \end{equation*}
 And the last terms of the letters $P$, $Q$, $R$, $S$ etc. can represented in short form as  follows
 
 \begin{equation*}
     \dfrac{\alpha m}{2 \cdot 3}, \quad \dfrac{\beta m}{4 \cdot 5}, \quad \dfrac{\gamma m}{6 \cdot 7}, \quad \dfrac{\delta m}{8 \cdot 9}, \quad \dfrac{\varepsilon m}{10 \cdot 11}, \quad \dfrac{\zeta m}{12 \cdot 13}.
 \end{equation*}
 
 \paragraph*{XXIX.}
 
 But to investigate how these letters $P$, $Q$, $R$, $S$ etc. proceed, let us subtract  a multiple of it from the preceding one such that the first terms are cancelled, and since  the letter $O=1$ precedes those letters, we will have:
 
 \begin{equation*}
          \renewcommand{\arraystretch}{2.5}
\setlength{\arraycolsep}{0mm}
\begin{array}{llcllllllllllllll}
     P &~-~& \dfrac{m}{12}& O &~=~& 0 \\
     Q &~-~& \dfrac{3m}{12} &P&~=~& -\dfrac{\beta m}{4 \cdot 5} \\
     R &~-~& \dfrac{5m}{12} &Q&~=~& -\dfrac{mm}{16 \cdot 9}+\dfrac{\gamma m}{6 \cdot 7}= -\dfrac{m}{12}P+\dfrac{\gamma m}{6 \cdot 7} \\
     S &~-~& \dfrac{7m}{12}  &R& ~=~& - \dfrac{7m^3}{128 \cdot 9}+\dfrac{3m^2}{16 \cdot 5}-\dfrac{\delta m}{8 \cdot 9}\\
     T &~-~& \dfrac{9m}{12}&S &~=~& -\dfrac{7m^4}{128 \cdot 9}+\dfrac{17m^3}{64 \cdot 3 \cdot 5}-\dfrac{7m^2}{8 \cdot 9 \cdot 5}+\dfrac{\varepsilon m}{10 \cdot 11} \\
     V &~-~& \dfrac{11m}{2}&T &~=~& - \dfrac{5 \cdot 7 \cdot 11 m^5}{2048 \cdot 27}+\dfrac{451 m^4}{512 \cdot 27}- \dfrac{121 m^3}{2048 \cdot 27 \cdot 5}~+~ \dfrac{7159 m^2}{128 \cdot 27 \cdot 5 \cdot 7}~-~\dfrac{\zeta m}{12 \cdot 13}.
\end{array}
 \end{equation*}
 Therefore, if we now consider these forms with more attention and, for the sake of brevity, set:
 
 \begin{equation*}
     \dfrac{\alpha m}{2 \cdot 3}= \alpha^1, \quad \dfrac{\beta m}{4 \cdot 5}=\beta^1, \quad \dfrac{\gamma m}{6 \cdot 7}= \gamma^1, \quad \dfrac{\delta m}{8 \cdot 9}=\delta^1 \quad \text{etc.},
 \end{equation*}
 we will detect the following sufficiently simple law in our letters $P$, $Q$, $R$ etc:

 \begin{center}
 \begin{equation*}
          \renewcommand{\arraystretch}{2.5}
\setlength{\arraycolsep}{0mm}
\begin{array}{llcllcllcllcllcllcllcllc}
     P &~-~& \dfrac{1}{1}&\alpha^1 &~=~&0 \\
     Q &~-~& \dfrac{3}{1}& \alpha^1 P &~+~& \dfrac{3 \cdot 2 \cdot 1}{1 \cdot 2 \cdot 3}&\beta^1 &~=~& 0 \\
     R &~-~& \dfrac{5}{1}&\alpha^1 Q &~+~& \dfrac{5 \cdot 4 \cdot 3}{1 \cdot 2 \cdot 3}&\beta^1 P &~-~& \dfrac{5 \cdot 4 \cdot 3 \cdot 2 \cdot 1}{1 \cdot 2 \cdot 3 \cdot 4 \cdot 5}&\gamma^1 &~=~& 0\\
     S &~-~& \dfrac{7}{1}&\alpha^1 R &~+~& \dfrac{7 \cdot 6 \cdot 5}{1 \cdot 2 \cdot 3}&\beta^1 Q &~-~& \dfrac{7 \cdot 6 \cdots 3}{1 \cdot 2 \cdots 5}&\gamma^1P &~+~& \dfrac{7 \cdot 6 \cdots 1}{1 \cdot 2 \cdots 7}&\delta^1 &~=~& 0 \\
      T &~-~& \dfrac{9}{1}&\alpha^1 S &~+~& \dfrac{9  \cdots 5}{1  \cdots 3}&\beta^1 R &~-~& \dfrac{9  \cdots 5}{1  \cdots 5}&\gamma^1Q &~+~& \dfrac{9 \cdots 3}{1  \cdots 7}&\delta^1 P &~-~& \dfrac{9 \cdots 1}{1  \cdots 9}&\varepsilon^1 &~=~&0 \\
       V &~-~& \dfrac{11}{1}&\alpha^1 T &~+~& \dfrac{11  \cdots 9}{1  \cdots 3}&\beta^1 S &~-~& \dfrac{11  \cdots 7}{1  \cdots 5}&\gamma^1R &~+~& \dfrac{11 \cdots 5}{1  \cdots 7}&\delta^1 Q &~-~& \dfrac{11 \cdots 3}{1  \cdots 1}&\varepsilon^1P &~+~&\dfrac{11 \cdots 1}{1 \cdots 11}\zeta^1 =0 \\
       & &  & &  & & &  & &\text{etc.}
\end{array}
 \end{equation*}
  \end{center}
But these new letters $\alpha^1$, $\beta^1$, $\gamma^1$, $\delta^1$ etc. from the preceding ones follow this law:

 \begin{equation*}
             \renewcommand{\arraystretch}{2.5}
\setlength{\arraycolsep}{0mm}
\begin{array}{llcllcllcllcllcllc}
     \alpha^1 &~-~& \dfrac{m}{2^2 \cdot 3}& &~=~& 0 \\
     \beta^1  &~-~& \dfrac{3 \cdot 2}{2^2  \cdot 1 \cdot 2 \cdot 3}& \alpha^1 &~+~& \dfrac{m}{2^4\cdot 5}& &~=~ & 0 \\
     \gamma^1  &~-~& \dfrac{5 \cdot 4}{2^2  \cdot 1 \cdot 2 \cdot 3}& \beta^1 &~+~& \dfrac{5 \cdot 4 \cdot 3 \cdot 2}{2^4\cdot 1\cdots 5}&\alpha^1 &~-~ & \dfrac{m}{2^6 \cdot 7} & &~=~& 0 \\
      \delta^1  &~-~& \dfrac{7 \cdot 6}{2^2  \cdot 1 \cdots  3}& \gamma^1 &~+~& \dfrac{7 \cdots 4}{2^4\cdot 1\cdots 5}&\beta^1 &~-~ & \dfrac{7 \cdots 2}{2^6 \cdot 1 \cdots 7} &\alpha^1 &~+~& \dfrac{m}{2^8 \cdot 9}  & &~=~& 0 \\
        \varepsilon^1  &~-~& \dfrac{9 \cdot 8}{2^2  \cdot 1 \cdots  3}& \delta^1 &~+~& \dfrac{9 \cdots 6}{2^4\cdot 1\cdots 5}&\gamma^1 &~-~ & \dfrac{9 \cdots 4}{2^6 \cdot 1 \cdots 7} &\beta^1 &~+~& \dfrac{9 \cdots 2}{2^8 1 \cdots 9}  &\alpha^1 &~-~& \dfrac{m}{2^{10}\cdot 11}=0 \\
\end{array}
 \end{equation*}
Therefore, I now have to believe to have answered the question on that extraordinary series, that I have contemplated, completely, whence I will now  present the answer in short form here.
 
 \subsection*{Problem}
 
 Having propounded this indefinite progression:

 \begin{equation*}
     s= x^{m+\lambda}-\dfrac{m}{1}(x-1)^{m+\lambda}+\dfrac{m(m-1)}{1 \cdot 2}(x-2)^{m+\lambda}-\dfrac{m(m-1)(m-2)}{1 \cdot 2 \cdot 3}(x-3)^{m+\lambda}+\text{etc.}
 \end{equation*}
to assign its sum, if $\lambda$ was an arbitrary positive integer.
 
 \subsection*{Solution}
 
 Let the letters $\mathfrak{A}$, $\mathfrak{B}$, $\mathfrak{C}$, $\mathfrak{D}$ etc. denote the Bernoulli numbers such that:
 
 \begin{equation*}
               \renewcommand{\arraystretch}{2.5}
\setlength{\arraycolsep}{0mm}
\begin{array}{lll}
     \mathfrak{A} &~=~& \dfrac{1}{2}, \quad \mathfrak{B}=\dfrac{1}{6}, \quad \mathfrak{C}=\dfrac{1}{6}, \quad \mathfrak{D}=\dfrac{3}{10}, \quad \mathfrak{E}=\dfrac{5}{6}, \\
     \mathfrak{F} &~=~& \dfrac{691}{210}, \quad  \mathfrak{G}= \dfrac{35}{2}, \quad \mathfrak{H}=\dfrac{3617}{30}, \quad \mathfrak{I}= \dfrac{43867}{42}, \\
     \mathfrak{K}&~=~& \dfrac{1222277}{110}, \quad \mathfrak{L}=\dfrac{854513}{6}, \\
     \mathfrak{M}&~=~& \dfrac{1181820455}{546}, \quad \mathfrak{N}= \dfrac{76977927}{2}, \\
     \mathfrak{Q}&~=~& \dfrac{23749461029}{30}, \quad  \mathfrak{P}=\dfrac{8615841276005}{462}, \\
     \mathfrak{Q}&~=~&\dfrac{84802531453387}{170}, \quad \mathfrak{R}=\dfrac{90219075042845}{6} \\
     & &\text{etc.}
\end{array}
 \end{equation*}
 I observed that these numbers proceed in such a way that
 
 \begin{equation*}
               \renewcommand{\arraystretch}{2.5}
\setlength{\arraycolsep}{0mm}
\begin{array}{llcl}
     \mathfrak{A} &~=~& \dfrac{1}{2} \\
     \mathfrak{B} &~=~& \dfrac{4}{2} &\cdot \dfrac{\mathfrak{A}^2}{3} \\
     \mathfrak{C} &~=~& \dfrac{6}{2} & \cdot \dfrac{2 \mathfrak{AB}}{3} \\
     \mathfrak{D} &~=~& \dfrac{8}{2} & \cdot \dfrac{2\mathfrak{AC}}{3}+\dfrac{8 \cdot 7 \cdot 6}{2 \cdot 3 \cdot 4}\cdot \dfrac{\mathfrak{B}^2}{5} \\
     \mathfrak{E} &~=~& \dfrac{10}{2} & \cdot \dfrac{2\mathfrak{AD}}{3}+\dfrac{10 \cdot 9 \cdot 8}{2 \cdot 3 \cdot 4}\cdot \dfrac{2\mathfrak{BC}}{5} \\
     \mathfrak{F} &~=~& \dfrac{12}{2} &\cdot \dfrac{2 \mathfrak{AE}}{3}+ \dfrac{12 \cdot 11 \cdot 10}{2 \cdot 3 \cdot 4}\cdot \dfrac{2 \mathfrak{BD}}{5}+\dfrac{12 \cdot 11 \cdot 10 \cdot 9 \cdot 8}{2 \cdot 3 \cdot 4 \cdot 5 \cdot 6}\cdot \dfrac{\mathfrak{CC}}{7} \\
     \mathfrak{G}&~=~& \dfrac{14}{2}&\cdot \dfrac{2 \mathfrak{AF}}{3}+\dfrac{14 \cdot 13 \cdot 12}{2 \cdot 3 \cdot 4}\cdot \dfrac{2 \mathfrak{BE}}{5}+\dfrac{14 \cdot 13 \cdot 12 \cdot 11 \cdot 10}{2 \cdot 3 \cdot 4 \cdot 5 \cdot 6}\cdot \dfrac{2 \mathfrak{CD}}{7} \\
     & & &\text{etc.}
\end{array}
 \end{equation*}
 Hence now find numbers $P$, $Q$, $R$, $S$ etc. that
 
 \begin{equation*}
               \renewcommand{\arraystretch}{2.5}
\setlength{\arraycolsep}{0mm}
\begin{array}{llcllcllcllcllcllc}
     P &~=~& \dfrac{1 \mathfrak{A}m}{1 \cdot 2 \cdot 3} \\
     Q &~=~& \dfrac{3 \mathfrak{A}m}{1 \cdot 2 \cdot 3}&P&-~~& \dfrac{3 \cdot 2 \cdot 1 \mathfrak{B}m}{1 \cdot 2 \cdot 3 \cdot 4 \cdot 5} \\
     R &~=~& \dfrac{5 \mathfrak{A}m}{1 \cdot 2 \cdot 3} &Q&~-~& \dfrac{5 \cdot 4 \cdot 3 \cdot \mathfrak{B}m}{1 \cdot 2 \cdot 3 \cdot 4 \cdot 5} &P&~+~& \dfrac{5 \cdot 4 \cdots 1 \mathfrak{C}m}{1 \cdot 2 \cdots 7} \\
      S &~=~& \dfrac{7 \mathfrak{A}m}{1 \cdot 2 \cdot 3} &R&~-~& \dfrac{7 \cdot 6 \cdot 5 \cdot \mathfrak{B}m}{1 \cdot 2 \cdots 5} &Q&~+~& \dfrac{7 \cdots 3 \mathfrak{C}m}{1 \cdot 2 \cdots 7}&P& ~-~& \dfrac{7 \cdots 1 \mathfrak{D}}{1 \cdot 2 \cdots 9} \\
      & & & && \text{etc.}
\end{array}
 \end{equation*}
 where the law of progression also is perspicuous.\\
 Having found this series, the sum $s$ in question will be expressed in this way:
 
 \begin{equation*}
          \renewcommand{\arraystretch}{2.5}
\setlength{\arraycolsep}{0mm}
\begin{array}{rcllcll}
     \dfrac{s}{(\lambda +1)(\lambda +2)\cdots (\lambda +m)}&~=~& \left(x-\dfrac{m}{2}\right)^{\lambda}&~+~& \dfrac{\lambda (\lambda -1)}{1 \cdot 2}&P&\left(x-\dfrac{m}{2}\right)^{\lambda -2} \\
       & & &~+~& \dfrac{\lambda \cdots (\lambda -3)}{1 \cdots 4}&Q&\left(x-\dfrac{m}{2}\right)^{\lambda -4}\\
        & & &~+~& \dfrac{\lambda \cdots (\lambda -5)}{1 \cdots 6}&R&\left(x-\dfrac{m}{2}\right)^{\lambda -6}\\
         & & &~+~& \dfrac{\lambda \cdots (\lambda -7)}{1 \cdots 8}&S&\left(x-\dfrac{m}{2}\right)^{\lambda -8} \\
         & & &  &\text{etc.}
\end{array}
 \end{equation*}
 where one should note, if the number $m$ is not an integer, that the value of the product
 
 \begin{equation*}
     (\lambda +1)(\lambda +2)\cdots (\lambda +m)
 \end{equation*}
 can be defined through other artifices explained on another occasion.
 
 \subsection*{Corollary 1}
 
 If instead of the  Bernoulli numbers  we want to introduce the related ones, which I used  for the sums of the powers of the reciprocals, and denote them by the letters $A$, $B$, $C$, $D$ etc. that $A=\frac{1}{6}$, $B=\frac{1}{90}$, $C=\frac{1}{945}$, $D=\frac{1}{9450}$, $E=\frac{1}{93555}$, since these numbers depend on the first in such a way that
 
 \begin{equation*}
     \mathfrak{A}=\dfrac{1 \cdot 2 \cdot 3}{2^1}A, \quad \mathfrak{B}=\dfrac{1 \cdots 5}{2^3}B, \quad \mathfrak{C}=\dfrac{1 \cdots 7}{2^5}C \quad \text{etc.},
 \end{equation*}
 but are connected to each other in such a way that:
 
 \begin{equation*}
     5B=2A^2, \quad 7C=4AB, \quad 9D=4AC+2BB,
 \end{equation*}
 \begin{equation*}
     11E=4AD+4BC, \quad 13F=4AE+4BD+2CC \quad \text{etc.},
 \end{equation*}
 then from these numbers the letters $P$, $Q$, $R$, $S$ etc. will be determined as follows:
 
 \begin{equation*}
              \renewcommand{\arraystretch}{2.5}
\setlength{\arraycolsep}{0mm}
\begin{array}{llcllcllcllcllc}
     P &~=~& \dfrac{1Am}{2} \\
     Q &~=~& \dfrac{3Am}{2} &P &~-~& \dfrac{3 \cdot 2 \cdot 1 Bm}{2^3} \\
     R &~=~& \dfrac{5Am}{2} &Q &~-~& \dfrac{5 \cdot 4 \cdot 3 Bm}{2^3}&P&~+~& \dfrac{5 \cdot 4 \cdots 1Cm}{2^5} \\
      S &~=~& \dfrac{7Am}{2} &R &~-~& \dfrac{7 \cdot 6 \cdot 5 Bm}{2^3}&Q&~+~& \dfrac{7 \cdot 6 \cdots 3Cm}{2^5}&P&~-~&\dfrac{7 \cdot 6 \cdots 1 Dm}{2^7} \\
      & &  &  & &\text{etc.}
\end{array}
 \end{equation*}
 
 \subsection*{Corollary 2}
 
 If for the various values of the number $\lambda$ we indicate the sum of the propounded progression by the sign $\int (\lambda)$ and now for $\lambda$ successively write the numbers $0$, $1$, $2$, $3$, $4$ etc., for these cases the sums $\int (0)$, $\int (1)$, $\int (2)$, $\int (3)$ etc., for the sake of brevity having put $x-\frac{m}{2}=y$, will be expressed in the following way:
 
 \begin{equation*}
        \renewcommand{\arraystretch}{2.5}
\setlength{\arraycolsep}{0mm}
\begin{array}{rlllllllll}
     \dfrac{\int (0)}{1 \cdot 2 \cdots m} &~=~& 1 \\
     \dfrac{\int (1)}{2 \cdot 3 \cdots (m+1)} &~=~& y \\
     \dfrac{\int (2)}{3 \cdot 4 \cdots (m+2)} &~=~& y^2 &~+~& P \\
      \dfrac{\int (3)}{4 \cdot 5 \cdots (m+3)} &~=~& y^3 &~+~& 3Py \\
        \dfrac{\int (4)}{5 \cdot 6 \cdots (m+4)} &~=~& y^4 &~+~& 6Py^2 &~+~& Q \\
     \dfrac{\int (5)}{6 \cdot 7 \cdots (m+5)} &~=~& y^5 &~+~& 10Py^3 &~+~& 5Qy \\ 
     \dfrac{\int (6)}{7 \cdot 8 \cdots (m+6)} &~=~& y^6 &~+~& 15Py^4 &~+~& 15Qy^2 &~+~&R \\
     &  &  &  &  &\text{etc.}
\end{array}
 \end{equation*}
 
 \subsection*{Corollary 3}
 
 Therefore, these sums can be defined from the preceding ones as follows
 \begin{center}
 \begin{footnotesize}
 \begin{equation*}
         \renewcommand{\arraystretch}{2.5}
\setlength{\arraycolsep}{0mm}
\begin{array}{ccccccccccccccccccccccc}
     \int (1) &~=~& \dfrac{m+1}{1}&y\int(0) \\
      \int (2) &~=~& \dfrac{m+2}{2}&y\int(1)&~+~& \dfrac{(m+2)(m+1)}{2 \cdot 2}&mA\int (0) \\
        \int (3) &~=~& \dfrac{m+3}{3}&y\int(2)&~+~& \dfrac{(m+3)(m+2)}{2 \cdot 3}&mA\int (1) \\
  \int (4) &~=~& \dfrac{m+4}{4}&y\int(3)&~+~& \dfrac{(m+4)(m+3)}{2 \cdot 4}&mA\int (2)&~-~&\dfrac{(m+4)\cdots (m+1)}{2^3 \cdot 4}mB\int (0) \\      
  \int (5) &~=~& \dfrac{m+5}{5}&y\int(4)&~+~& \dfrac{(m+5)(m+4)}{2 \cdot 5}&mA\int (3)&~-~&\dfrac{(m+5)\cdots (m+2)}{2^3 \cdot 5}mB\int (1) \\      
   \int (6) &~=~& \dfrac{m+6}{6}&y\int(5)&~+~& \dfrac{(m+6)(m+5)}{2 \cdot 6}&mA\int (4)&~-~&\dfrac{(m+6)\cdots (m+3)}{2^3 \cdot 5}mB\int (2)&~+~&\dfrac{(m+6)\cdots(m+1)}{2^5 \cdot 6}mC\int (0) \\      
     \int (7) &~=~& \dfrac{m+7}{7}&y\int(6)&~+~& \dfrac{(m+7)(m+6)}{2 \cdot 7}&mA\int (5)&~-~&\dfrac{(m+7)\cdots (m+4)}{2^3 \cdot 7}mB\int (3)&~+~&\dfrac{(m+7)\cdots(m+2)}{2^5 \cdot 7}mC\int (1) \\      
     \int (8) &~=~& \dfrac{m+8}{8}&y\int(7)&~+~& \dfrac{(m+8)(m+7)}{2 \cdot 8}&mA\int (6)&~-~&\dfrac{(m+8)\cdots (m+5)}{2^3 \cdot 8}mB\int (4)&~+~&\dfrac{(m+8)\cdots(m+3)}{2^5 \cdot 8}mC\int (2) \\      
      & & & & & & & & &~-~&\dfrac{(m+8)\cdots(m+1)}{2^7 \cdot 8}mD\int (0),      
\end{array}
 \end{equation*}
 \end{footnotesize}
 \end{center}
  which law will become obvious soon to the attentive reader.
 
 \subsection*{Conclusion}
 
 Now there will be not much difficulty to generalise this task quite substantially such that, if $\varphi :x$ denotes an arbitrary function of $x$,  we can assign the sum of this series
 
 \begin{equation*}
     s= \varphi:x -m \varphi :(x-1)+\dfrac{m(m-1)}{1\cdot 2}\varphi:(x-2)-\dfrac{m(m-1)(m-2)}{1\cdot 2 \cdot 3}\varphi :(x-3).
 \end{equation*}
  For, it is perspicuous that this form exhibits the difference of order $m$ of this progression
  
  \begin{equation*}
      \varphi:x, \quad \varphi :(x-1), \quad \varphi:(x-2), \quad \varphi:(x-3) \quad \text{etc.}
  \end{equation*}
  For, from that, what I covered in \textit{Institutiones Calculi Differentialis} pag. 343 \footnote{p. 264 in the Opera Onmia Version, i.e. Volume 10 of Series 1}, if we set $\varphi:x =y$, one concludes that the differences of the respective orders are:
  
  \begin{equation*}
          \renewcommand{\arraystretch}{2.5}
\setlength{\arraycolsep}{0mm}
\begin{array}{lccccccccccccccccc}
     \Delta y &~=~& \dfrac{dy}{dx}&~-~& \dfrac{ddy}{2dx^2}&~+~& \dfrac{d^3y}{2 \cdot 3 dx^3}&~-~& \dfrac{d^4y}{2 \cdot 3 \cdot 4 dx^4}&~+~&\dfrac{d^5y}{2 \cdots 5dx^5}&~-~&\text{etc.} \\
     \Delta^2 y &~=~& \dfrac{d^2y}{dx^2}&~-~& \dfrac{3d^3y}{3dx^3}&~+~& \dfrac{7d^4y}{3\cdot 4 dx^4}&~-~&\dfrac{15d^5y}{3 \cdot 4 \cdot 5 dx^5}&~+~&\dfrac{31d^6 y}{3 \cdots 6dx^6}&~-~&\text{etc.} \\
     \Delta^3 y &~=~& \dfrac{d^3y}{dx^3}&~-~& \dfrac{6d^4y}{4dx^4}&~+~& \dfrac{25 d^5 y}{4 \cdot 5 dx^5}&~-~& \dfrac{90 d^6 y}{4 \cdot 5 \cdot 6 dx^6}&~+~& \dfrac{301 d^7 y}{4 \cdots 7 dx^7}&~-~&\text{etc.} \\
     \Delta^4y &~=~& \dfrac{d^4y}{dx^4}&~-~& \dfrac{10d^5y}{5dx^5}&~+~& \dfrac{65d^6y}{5\cdot 6 dx^6}&~-~& \dfrac{350 d^7 y}{5 \cdot 6 \cdot 7 dx^7}&~+~& \dfrac{1701 d^8 y}{5 \cdots 8 dx^8}&~-~&\text{etc.} \\
         &  &  &  & & &\text{etc.},
\end{array}
  \end{equation*}
  since which coefficients are those we had above in paragraph IV, in like manner, we will understand that the difference of order $m$ or $\Delta^m y$, i.e. the sum of the propounded series, will be
  
  \begin{equation*}
      s= \dfrac{d^m y}{dx^m}-\dfrac{A^1 d^{m+1}y}{(m+1)dx^{m+1}}+\dfrac{B^1 d^{m+2}y}{(m+1)(m+2)dx^{m+2}}-\dfrac{C^1 d^{m+3}y}{(m+1)\cdots (m+3)dx^{m+3}}+\text{etc.},
  \end{equation*}
  which coefficients $A^1$, $B^1$, $C^1$ etc. I determined above in paragraph XIII. Therefore, it will be
  
  \begin{small}
  \begin{equation*}
          \renewcommand{\arraystretch}{2.5}
\setlength{\arraycolsep}{0mm}
\begin{array}{lll}
     \dfrac{A^1}{m+1} =\dfrac{m}{2} \\
     \dfrac{B^1}{(m+1)(m+2)}=\dfrac{m}{1\cdot 2 \cdot 3}+\dfrac{3m(m-1)}{1\cdot 2 \cdot 3 \cdot 4} \\
     \dfrac{C^1}{(m+1)\cdots (m+3)}=\dfrac{m}{1\cdot 2 \cdot 3 \cdot 4}+\dfrac{10m(m-1)}{1\cdot 2 \cdots 5}+\dfrac{15m(m-1)(m-2)}{1\cdot 2 \cdots 6} \\
     \dfrac{D^1}{(m+1)\cdots (m+4)}=\dfrac{m}{1\cdot 2 \cdots 5}+\dfrac{25m(m-1)}{1 \cdot 2 \cdots 6}+\dfrac{105 m(m-1)(m-2)}{1\cdot 2 \cdots 7}+\dfrac{105m(m-1)(m-2)(m-3)}{1\cdot 2 \cdots 8} \\
     \text{etc.}
\end{array}
  \end{equation*}
  \end{small}Therefore, if we set $\varphi: \left(x-\dfrac{m}{2}\right)=v$, such that $v$ results from $y$, if one writes $x-\frac{m}{2}$ instead of $y$, it will obviously be
  
  \begin{equation*}
      \dfrac{d^m v}{dx^m}= \dfrac{d^m y}{dx^m}-\dfrac{md^{m+1}y}{2 dx^{m+1}}+\dfrac{m^2 d^{m+2}y}{2\cdot 4 dx^{m+2}}-\text{etc.};
  \end{equation*}
  if this equation is subtracted from that, one will have to do exactly the same calculations as above. Hence introducing the same letters $P$, $Q$, $R$, $S$ etc., which we defined above, we will obtain the following value of the sum $s$:

  \begin{equation*}
      s=\dfrac{d^mv}{dx^m}+\dfrac{Pd^{m+2}v}{1\cdot 2 dx^{m+2}}+\dfrac{Qd^{m+4}v}{1\cdot 2 \cdots 4 dx^{m+4}}+\dfrac{Rd^{m+6}v}{1\cdot 2 \cdots 6 dx^{m+6}}+\dfrac{Sd^{m+8}v}{1\cdot 2 \cdots 8 dx^{m+8}}+\text{etc.}
  \end{equation*}
  and hence, if one takes
  
  \begin{equation*}
      y= \varphi :x = x^{m+\lambda} \quad \text{and} \quad v = \left(x-\dfrac{m}{2}\right)^{m+\lambda},
  \end{equation*}
  manifestly the same summation we found before results, and thus the whole task reduces to the letters $P$, $Q$, $R$, $S$ etc., whose nature I derived from the Bernoulli numbers above.\\
  Hence it follows immediately , what has been less obvious before, that, if in the function $y$ or $v$ the number of dimensions was smaller than the exponent $m$, which number certainly has to be a positive integer, then all differentials of order $m$ and higher vanish and the sum $s$ will be $=0$.\\
  Further, hence  there is a clearer way to find the values of the letters $P$,$Q$, $R$, $S$ etc. For, since, having set
  
  \begin{equation*}
      s= \dfrac{d^my}{dx^m}-\dfrac{\alpha d^{m+1}y}{dx^{m+1}}+\dfrac{\beta d^{m+2}y}{dx^{m+2}}-\dfrac{\gamma d^{m+3}y}{dx^{m+3}}+\text{etc.},
  \end{equation*}
  we have
  
  \begin{equation*}
          \renewcommand{\arraystretch}{2.5}
\setlength{\arraycolsep}{0mm}
\begin{array}{lllllllllllllllllllllll}
     \alpha &~=~& \dfrac{m}{1\cdot 2} \\
     \beta &~=~& \dfrac{m}{1 \cdot 2 \cdot 3} &~+~& \dfrac{3m(m-1)}{1 \cdot 2 \cdot 3 \cdot 4} \\
     \gamma &~=~& \dfrac{m}{1\cdots 4}&~+~&\dfrac{10m(m-1)}{1\cdots 5}&~+~& \dfrac{15m \cdots (m-2)}{1 \cdots 6} \\
     \delta &~=~& \dfrac{m}{1\cdots 5}&~+~& \dfrac{25m(m-1)}{1\cdots 6}&~+~& \dfrac{105m \cdots (m-2)}{1\cdots 7}&~+~& \dfrac{105m \cdots (m-3)}{1\cdots 8} \\
           &    &  & &\text{etc.},
\end{array}
  \end{equation*}
  but the function $y$ results from the function $v=\varphi :\left(x-\frac{m}{2}\right)$, if in it instead of $x$ one writes $x+\frac{m}{2}$,  in general it will be
  
  \begin{equation*}
      \dfrac{d^ny}{dx^n}=\dfrac{d^nv}{dx^n}+\dfrac{m}{2}\cdot \dfrac{d^{n+1}v}{dx^{n+1}}+\dfrac{m^2}{2\cdot 4}\cdot \dfrac{d^{n+2}v}{dx^{n+2}}+\dfrac{m^3}{2\cdot 4 \cdot 6}\cdot \dfrac{d^{n+3}v}{dx^{n+3}}+\text{etc.},
  \end{equation*}
  whence, if one substitutes the differentials of $v$ for those of $y$, it will be

  \begin{equation*}
      s=\dfrac{d^nv}{dx^n}+\left(\dfrac{m}{2}-\alpha\right)\dfrac{d^{n+1}v}{dx^{n+1}}+\left(\dfrac{m^2}{2\cdot 4}-\dfrac{m}{2}\alpha+\beta\right)\dfrac{d^{n+2}v}{dx^{n+2}}+\left(\dfrac{m^3}{2\cdot 4 \cdot 6}-\dfrac{m^2}{2\cdot 4}\alpha +\dfrac{m}{2}\beta -\gamma\right)\dfrac{d^{n+3}v}{dx^{n+3}}+\text{etc.}
  \end{equation*}

  and so we will have:
  
  \begin{equation*}
          \renewcommand{\arraystretch}{2.5}
\setlength{\arraycolsep}{0mm}
\begin{array}{rrrrrrrrrrlrrr}
    \dfrac{m}{2} &~-~& \alpha &~=~& 0 \\
    \dfrac{m^2}{2 \cdot 4}&~-~& \dfrac{m}{2}\alpha &~+~& \beta &~=~& \dfrac{P}{1\cdot 2} \\
    \dfrac{m^3}{2 \cdot 4 \cdot 6} &~-~& \dfrac{m^2}{2\cdot 4}\alpha &~+~& \dfrac{m}{2}\beta &~-~& \gamma &~=~& 0 \\
    \dfrac{m^4}{2 \cdot 4 \cdot 6 \cdot 8} &~-~& \dfrac{m^3}{2 \cdot 4 \cdot 6} \alpha &~+~& \dfrac{m^2}{2 \cdot 4} \beta &~-~& \dfrac{m}{2}\gamma &~+~& \delta &~=~&\dfrac{Q}{1\cdot 2 \cdot 3 \cdot 4} \\
    \dfrac{m^5}{1 \cdot 2 \cdot 3 \cdot 4 \cdot 5} &~-~& \dfrac{m^4}{2\cdot 4 \cdot 6 \cdot 8}\alpha &~+~& \dfrac{m^3}{2 \cdot 4 \cdot 6}\beta &~-~& \dfrac{m^2}{2 \cdot 4}\gamma &~+~& \dfrac{m}{2}\delta &~-~& \varepsilon =0\\
     &  &  &   &  &\text{etc.};
\end{array}
  \end{equation*}
  for, one easily sees that these expressions must vanish alternately.

\lhead[\thepage]{}
\chead[Appendix]{Translation of E421}
\rhead[]{\thepage}
\chapter{Translation of E421 - Expansion of the integral  $\int x^{f-1}
dx (\log x)^\frac{m}{n}$ having extended the integration from the value $x=0$ to $x=1$}

\section*{Theorem 1}

\paragraph*{§1}

\textit{If $n$ denotes a positive integer and the integral}

\[
\int x^{f-1}dx(1-x^g)^n
\]
\textit{is extended from the value $x=0$ to $x=1$, the value of the integral will be}

\[
=\frac{g^n}{f}\cdot \frac{1 \cdot 2 \cdot 3 \cdots n}{(f+g)(f+2g)(f+3g) \cdots (f+ng)}.
\]

\subsection*{Proof}

It is known that  the integral $\int x^{f-1}dx(1-x^g)^m$   can in general be reduced to  this one $\int x^{f-1}dx(1-x^g)^{m-1}$, since it is possible to define constant quantities $A$ and $B$ in such a way that 

\[
\int x^{f-1}dx(1-x^g)^{m}=A\int x^{f-1}dx(1-x^g)^{m-1}+Bx^f(1-x^g)^m;
\]
by differentiating, this equation results

\[
 x^{f-1}dx(1-x^g)^{m}
\] 
\[
=Ax^{f-1}dx(1-x^g)^{m-1}+Bfx^{f-1}dx(1-x^g)^m-Bmgx^{f+g-1}dx(1-x^g)^{m-1},
\]
which divided by $x^{f-1}dx(1-x^g)^{m-1}$ gives

\[
1-x^g=A+Bf(1-x^g)-Bmgx^g
\]
or

\[
1-x^g=A-Bmg+B(f+mg)(1-x^g);
\]
in order for this equation to hold, it is necessary that

\[
1=B(f+mg) \quad \text{and} \quad A=Bmg,
\]
whence we conclude

\[
B=\frac{1}{f+mg} \quad \text{and} \quad A=\frac{mg}{f+mg}.
\]
Therefore, we will have the following general reduction

\[
\int x^{f-1}dx(1-x^g)^{m}= \frac{mg}{f+mg}\int x^{f-1}dx(1-x^g)^{m-1}+\frac{1}{f+mg}x^f(1-x^g)^m;
\]
because it vanishes for $x=0$, if $f>0$, of course, the addition of a constant is not necessary. Hence having extended both integrals to $x=1$, the last absolute part vanishes and for the case $x=1$ it will be

\[
\int x^{f-1}dx(1-x^g)^{m}= \frac{mg}{f+mg}\int x^{f-1}dx(1-x^g)^{m-1}.
\]
Since for $m=1$

\[
\int x^{f-1}dx(1-x^g)^{0}=\frac{1}{f}x^f=\frac{1}{f},
\]
having put $x=1$, we obtain the following values for the same case $x=1$

\begin{alignat*}{18}
&\int x^{f-1}dx(1-x^g)^{1}&&=\frac{g}{f}&&\cdot \frac{1}{f+g},\\
&\int x^{f-1}dx(1-x^g)^{2}&&=\frac{g^2}{f}&&\cdot \frac{1}{f+g} \cdot \frac{2}{f+2g},\\
&\int x^{f-1}dx(1-x^g)^{3}&&=\frac{g^3}{f}&&\cdot \frac{1}{f+g} \cdot \frac{2}{f+2g} \cdot \frac{3}{f+3g}
\end{alignat*}
and hence  we conclude that for any positive integer $n$ it will be

\[
\int x^{f-1}dx(1-x^g)^n= \frac{g^n}{f} \cdot \frac{1}{f+g} \cdot \frac{2}{f+2g} \cdot \frac{3}{f+3g} \cdots \frac{n}{f+ng},
\]
if only the numbers $f$ and $g$ are positive.

\subsection*{Corollary 1}

\paragraph*{§2}

 Vice versa, the value of a product of this kind, formed from an arbitrary amount of factors, can be expressed by an integral so that

\[
\frac{1 \cdot 2 \cdot3 \cdots n}{(f+g)(f+2g)(f+3g)\cdots(f+ng)}=\frac{f}{g^n}\int x^{f-1}dx(1-x^g)^n
\]
having extended this integral from the value $x=0$ to $x=1$.

\subsection*{Corollary 2}

\paragraph*{§3}

Therefore, if one considers a progression of this kind

\[
\frac{1}{f+g}, \quad \frac{1 \cdot 2}{(f+g)(f+2g)}, \quad \frac{1 \cdot 2 \cdot 3}{(f+g)(f+2g)(f+3g)}, \quad \frac{1 \cdot 2 \cdot 3 \cdot 4}{(f+g)(f+2g)(f+3g)(f+4g)} \quad \text{etc.},
\]
its general term corresponding to the indefinite index $n$ is conveniently represented by this integral $\frac{f}{g^n}\int x^{f-1}dx(1-x^g)^n$; and using this formula, the progression and its terms corresponding to fractional indices can be exhibited.

\subsection*{Corollary 3}

\paragraph*{§4}

If  we write $n-1$ instead of $n$, we will have
\[
\frac{1 \cdot 2 \cdot 3 \cdots (n-1)}{(f+g)(f+2g)(f+3g) \cdots (f+(n-1)g)}=\frac{f}{g^{n-1}}\int x^{f-1}dx(1-x^g)^{n-1};
\]
multiplication by $\frac{n}{f+ng}$ yields

\[
\frac{1 \cdot 2 \cdot 3 \cdots n}{(f+g)(f+2g)(f+3g) \cdots (f+ng)}= \frac{f \cdot ng}{g^n(f+ng)}\int x^{f-1}dx(1-x^g)^{n-1}.
\]

\subsection*{Scholium 1}

\paragraph*{§5}

It would have been possible to derive this last formula immediately from the preceding one, since we just proved that 

\[
\int x^{f-1}dx(1-x^g)^n= \frac{ng}{f+ng}\int x^{f-1}dx(1-x^g)^{n-1},
\]
if both integrals are extended from the value $x=0$ to $x=1$; this is to be kept in mind for all the integrals in everything what follows. Furthermore, it is to be noted that the quantities $f$ and $g$ are positive, a condition that was used in the proof, of course. Concerning the number $n$, if it denotes the index of a certain term of the progression (§3), that index can also be negative, because all terms, also those corresponding to negative indices, of the progression are considered to be exhibited by the given integral formula. Nevertheless, it is to be noted that this reduction

\[
\int x^{f-1}dx(1-x^g)^m=\frac{mg}{f+mg}\int x^{f-1}dx(1-x^g)^{m-1}
\]
is only true, if $m>0$, because otherwise the algebraic part $\frac{1}{f+mg}x^f(1-x^g)^m$  would not vanish for $x=1$.

\subsection*{Scholium 2}

\paragraph*{§6}

I already  studied series of this kind, which can be called transcendental, because the terms corresponding to fractional indices are transcendental quantities, in \textsc{Comment. acad. sc. Petrop., book 5} in more detail\footnote{Euler refers to his paper "'De progressionibus transcendentibus seu quarum termini generales algebraice dari nequeunt."' This is paper E19 in the Eneström-Index}; therefore, I will not investigate those progressions here again but focus on the remarkable comparisons of the integral formulas that can be derived from it. After I had shown that the value of the indefinite product $1 \cdot 2 \cdot 3 \cdots n$ is expressed by the integral formula $\int dx\big( \log \frac{1}{x}\big)^n$ extended from $x=0$ to $x=1$, which, if $n$ is a positive integer, is manifest by direct integration, I examined the cases, in which a fractional number is taken for $n$; in these cases it is indeed not obvious at all, to which kind of transcendental quantities these terms are to be referred. But, using a singular artifice, I reduced the same terms to more familiar quadratures; therefore, this seems to be worth of one's while to consider it with all eagerness.

\section*{Problem 1}

\paragraph*{§7}

\textit{Since it was demonstrated that}

\[
\frac{1 \cdot 2 \cdot 3 \cdots n}{(f+g)(f+2g)(f+3g) \cdots (f+ng)}=\frac{f}{g^n}\int x^{f-1}dx(1-x^g)^n,
\] 
\textit{having extended the integral from $x=0$ to $x=1$, to assign the value of the same product in the case $g=0$ by means of an integral.}

\subsection*{Solution}

Having put $g=0$ in the integral, the term $(1-x^g)^n$ vanishes, but at the same time also the denominator $g^n$ vanishes, whence the question reduces to the task to define the value of the fraction $\frac{(1-x^g)^n}{g^n}$  in the case $g=0$, in which both the numerator and the denominator vanishes. Therefore, let us consider $g$ as an infinitely small quantity, and because $x^g=e^{g \log x}$, it will be $x^g=1+g \log x$ and hence $(1-x^g)^n=g^n(-\log x)^n=g^n\big(\log \frac{1}{x}\big)^n$; hence  our integral  becomes $f\int x^{f-1}dx \big(\log \frac{1}{x}\big)^n$ for this case so that one now has this expression

\[
\frac{1 \cdot 2 \cdot 3 \cdots n}{f^n}=f \int x^{f-1}dx \bigg(\log \frac{1}{x}\bigg)^n
\]
or

\[
1 \cdot 2 \cdot 3 \cdots n= f^{n+1} \int x^{f-1}dx \bigg(\log\frac{1}{x}\bigg)^n.
\]

\subsection*{Corollary 1}

\paragraph*{§8}

If $n$ is a positive integer, the integration of the integral $\int  x^{f-1}dx \big(\log \frac{1}{x}\big)^n$ succeeds and, having extended it from $x=0$ to $x=1$, indeed the product we found to be equal to it results. But if fractional numbers are taken for $n$, the same formula can be applied to interpolate this hypergeometric progression

\[
1, \quad 1 \cdot 2, \quad 1 \cdot 2 \cdot 3, \quad 1 \cdot 2 \cdot 3 \cdot 4, \quad 1\cdot 2 \cdot 3 \cdot 4 \cdot 5 \quad \text{etc.}
\]
or

\[
1, \quad 2, \quad 6,  \quad 24, \quad 120, \quad 720, \quad 5040 \quad \text{etc.}
\]

\subsection*{Corollary 2}

\paragraph*{§9}

If the expression just found is divided by the principal one, a product  whose factors proceed in an arithmetic progression will emerge, namely

\[
(f+g)(f+2g)(f+3g)\cdots(f+ng)=f^ng^n \frac{\int x^{f-1}dx\left(\log \frac{1}{x}\right)^n}{\int x^{f-1}dx\left(1-x^g\right)^n},
\]
whose values can also be assigned, using the integral, if $n$ is a fractional number.

\subsection*{Corollary 3}

\paragraph*{§10}

Since
\[
\int x^{f-1}dx(1-x^g)^n=\frac{ng}{f+ng}\int x^{f-1}dx(1-x^g)^{n-1},
\]
 in like manner,  for the case $g=0$ it will  be

\[
\int x^{f-1}dx\left(\log \frac{1}{x}\right)^n=\frac{n}{f}\int x^{f-1}dx\left(\log \frac{1}{x}\right)^{n-1}
\]
and hence by those other integrals

\[
1 \cdot 2 \cdot 3 \cdots n=nf^n \int x^{f-1}dx\left(\log \frac{1}{x}\right)^{n-1}
\]
and

\[
(f+g)(f+2g)\cdots (f+ng)=f^{n-1}g^{n-1}(f+ng)\frac{\int x^{f-1}dx\left(\log \frac{1}{x}\right)^{n-1}}{\int x^{f-1}dx(1-x^g)^{n-1}}.
\]

\subsection*{Scholium}

\paragraph*{§11}

Because we found that

\[
1 \cdot 2 \cdot 3 \cdots n=f^{n+1}  \int x^{f-1}dx\left(\log \frac{1}{x}\right)^n,
\]
it is plain that this integral does not depend on the value of the quantity $f$, which is also easily seen by putting $x^f=y$, whence first we find

\[
fx^{f-1}dx=dy \quad \text{and} \quad \log \frac{1}{x}=-\log x=-\frac{1}{f}\log y =\frac{1}{f}\log \frac{1}{y}
\]
and therefore

\[
f^n \left(\log \frac{1}{x}\right)^n=\left(\log \frac{1}{y}\right)^n
\]
such that

\[
1 \cdot 2 \cdot 3 \cdots n=\int dy \left(\log \frac{1}{y}\right)^n,
\]
which expression results from the first by putting $f=1$. Therefore, for an interpolation of this kind the whole task is  reduced to the definition of the values of the integral $\int dx \left(\log \frac{1}{x}\right)^n$ for the cases, in which the exponent $n$ is a fractional number. For example, if $n=\frac{1}{2}$, one has to assign the value of the formula $\int dx\sqrt{\log \frac{1}{x}}$, which value I already once showed to be $=\frac{1}{2}\sqrt{\pi}$, while $\pi$ denotes the circumference of the circle whose diameter is $=1$; but for other fractional numbers I taught how to reduce its value to quadratures of algebraic curves of higher order. Because this reduction is by no means obvious and is only valid, if the integration of the formula $\int dx \left(\log \frac{1}{x}\right)^n$ is extended from the value $x=0$ to $x=1$, it  seems to be worth of one's attention. But even though I already treated this subject once\footnote{Euler considered this expression also in E19 mentioned already in the footnote above.},  nevertheless, because I was led to the results in a rather non straight-forward way, I decided take on this subject here again and explain everything in more detail.

\section*{Theorem 2}

\paragraph*{§12}

\textit{If the integrals are extended from the value $x=0$ to $x=1$ and $n$ denotes a positive integer, it will be}

\[
\frac{1 \cdot 2 \cdot 3 \cdots n}{(n+1)(n+2)(n+2) \cdots 2n}=\frac{1}{2}ng \int x^{f+ng-1}dx(1-x^g)^{n-1} \cdot \frac{\int x^{f-1}dx(1-x^g)^{n-1}}{\int x^{f-1}dx(1-x^g)^{2n-1}},
\]
\textit{whatever positive numbers are taken for $f$ and $g$.}

\subsection*{Proof}

Because above (§4) we showed that

\[
\frac{1 \cdot 2 \cdot 3 \cdots n}{(f+g)(f+2g) \cdots (f+ng)}=\frac{f \cdot ng}{g^n(f+ng)}\int x^{f-1} dx(1-x^g)^{n-1},
\]
if we write $2n$ instead of $n$,  we will have

\[
\frac{1 \cdot 2 \cdot 3 \cdots 2n}{(f+g)(f+2g) \cdots (f+2ng)}=\frac{f \cdot 2ng}{g^{2n}(f+2ng)}\int x^{f-1}dx(1-x^g)^{2n-1}.
\]
Now divide the first equation by the second one and this third one will result

\[
\frac{(f+(n+1)g)(f+(n+2)g)\cdots (f+2ng)}{(n+1)(n+2) \cdots 2n}= \frac{g^n(f+2ng)}{2(f+ng)} \cdot \frac{\int x^{f-1}dx(1-x^g)^{n-1}}{\int x^{f-1}dx(1-x^g)^{2n-1}}.
\]
But if  one writes $f+ng$ instead of $f$  in the first equation, this fourth equation will result

\[
\frac{1 \cdot 2 \cdot 3 \cdots n}{(f+(n+1)g)(f+(n+2)g) \cdots (f+2ng)}=\frac{(f+ng)ng}{g^n(f+2ng)}\int x^{f+ng-1}dx(1-x^g)^{n-1}.
\]
Multiply this fourth equation by the third and one will find the equation to be demonstrated, namely

\[
\frac{1 \cdot 2 \cdot 3 \cdots n}{(n+1)(n+2)(n+3) \cdots 2n}=\frac{1}{2}ng \int x^{f+ng-1}dx(1-x^g)^{n-1} \cdot \frac{\int x^{f-1}dx(1-x^g)^{n-1}}{\int x^{f-1}dx(1-x^g)^{2n-1}}.
\]

\subsection*{Corollary 1}

\paragraph*{§13}

If  one sets $f=n$ and $g=1$ in the first equation, the same product will result, of course

\[
\frac{1 \cdot 2 \cdot 3 \cdots n}{(n+1)(n+2) \cdots 2n}=\frac{1}{2}n \int x^{n-1}dx(1-x^g)^{n-1};
\]
having compared this equation to the one mentioned above we obtain

\[
\frac{\int x^{n-1}dx(1-x)^{n-1}}{g\int x^{f+ng-1}dx(1-x^g)^{n-1}}=\frac{\int x^{f-1}dx(1-x^g)^{n-1}}{\int x^{f-1}dx(1-x^g)^{2n-1}}.
\]

\subsection*{Corollary 2}

\paragraph*{§14}

If we write $x^g$ instead of $x$ in that equation, it will be

\[
\frac{1 \cdot 2 \cdot3 \cdots n}{(n+1)(n+2) \cdots 2n}=\frac{1}{2}ng \int x^{ng-1}dx(1-x^g)^{n-1}
\]
such that we find this comparison of the following integral formulas

\[
\int x^{ng-1}dx(1-x^g)^{n-1}=\int x^{f+ng-1}dx(1-x^g)^{n-1} \cdot \frac{\int x^{f-1}dx(1-x^g)^{n-1}}{\int x^{f-1}dx(1-x^g)^{2n-1}}.
\]

\subsection*{Corollary 3}

\paragraph*{§15}

If  we set $g=0$ in the equation of the theorem, because of $(1-x^g)^m=g^m\left(\log \frac{1}{x}\right)^m$, the powers of $g$ will cancel each other and this equation will result

\[
\frac{1 \cdot 2 \cdot 3 \cdots n}{(n+1)(n+2) \cdots 2n}=\frac{1}{2}n \int x^{f-1}dx \left(\log \frac{1}{x}\right)^{n-1} \cdot \frac{\int x^{f-1}dx \left(\log \frac{1}{x}\right)^{n-1}}{\int x^{f-1}dx \left(\log \frac{1}{x}\right)^{2n-1}},
\]
whence we conclude

\[
\frac{\left(\int x^{f-1}dx \left(\log \frac{1}{x}\right)^{n-1}\right)^2}{\int x^{f-1}dx \left(\log \frac{1}{x}\right)^{2n-1}}=g \int x^{ng-1}dx(1-x^g)^{n-1}
\]
or, because of

\[
\int x^{f-1}dx \left(\log \frac{1}{x}\right)^{n-1}= \frac{f}{n}\int x^{f-1}dx \left(\log \frac{1}{x}\right)^{n},
\]
this equality

\[
\frac{2f}{n}\cdot \frac{\left(\int x^{f-1}dx \left(\log \frac{1}{x}\right)^{n}\right)^2}{\int x^{f-1}dx \left(\log \frac{1}{x}\right)^{2n}}=g \int x^{ng-1}dx(1-x^g)^{n-1}.
\]

\subsection*{Corollary 4}

\paragraph*{§16}

Let us set $f=1$, $g=2$ and $n=\frac{m}{2}$ here so that $m$ is a positive integer, and, because of

\[
\int dx \left(\log \frac{1}{x}\right)^{m}=1 \cdot 2 \cdot3 \cdots m,
\]
it will be

\[
\frac{4}{m}\cdot \frac{\left(\int dx \left(\log \frac{1}{x}\right)^{\frac{m}{2}}\right)^2}{1 \cdot 2 \cdot 3 \cdots m}=2 \int x^{m-1}dx(1-x^2)^{\frac{m}{2}-1}
\]
and hence

\[
\int dx\left(\log \frac{1}{x}\right)^{\frac{m}{2}}=\sqrt{1 \cdot 2 \cdot 3 \cdots m \cdot \frac{m}{2}\int x^{m-1}dx(1-x^2)^{\frac{m}{2}-1}}
\]
and by taking $m=1$, because of

\[
\int \frac{dx}{\sqrt{1-xx}}=\frac{\pi}{2},
\]
one will have

\[
\int dx \sqrt{\log \frac{1}{x}}= \sqrt{\frac{1}{2}\int \frac{dx}{\sqrt{1-xx}}}=\frac{1}{2}\sqrt{\pi}.
\]

\subsection*{Scholium}

\paragraph*{§17}

So lo and behold this succinct proof of the theorem I proved some time ago\footnote{Euler proved this theorem in \cite{E19},}, which says that  $\int dx \sqrt{\log \frac{1}{x}}=\frac{1}{2}\sqrt{\pi}$; furthermore, note that I did not use an argument involving interpolations, which I had used back then.  Here, it was  deduced from this theorem I found here, which states that 

\[
\frac{\left(\int x^{f-1}dx \left(\log \frac{1}{x}\right)^{n-1}\right)^2}{\int x^{f-1}dx\left(\log \frac{1}{x}\right)^{2n-1}}=g \int x^{ng-1}dx(1-x^g)^{n-1}.
\]
But the principal theorem, from which this one is deduced, reads as follows

\[
g \frac{\int x^{f-1}dx(1-x^g)^{n-1} \cdot \int x^{f+ng-1}dx(1-x^g)^{n-1}}{\int x^{f-1}dx(1-x^g)^{2n-1}}=\int x^{n-1}dx(1-x)^{n-1};
\]
for, each side, if it is actually calculated by an integration extended from $x=0$ to $x=1$, is equal to this product

\[
\frac{1 \cdot 2 \cdot 3 \cdots (n-1)}{(n+1)(n+2)\cdots (2n-1)}.
\]
But if we want to give the one side a more general form involving a further-extending class of integrals, we can state the theorem in such a way that
\[
g \frac{\int x^{f-1}dx(1-x^g)^{n-1} \cdot \int x^{f+ng-1}dx(1-x^g)^{n-1}}{\int x^{f-1}dx(1-x^g)^{2n-1}}=k \int x^{nk-1}dx(1-x^k)^{n-1};
\]
and, if  one takes $g=0$ here, 

\[
\frac{\left(\int x^{f-1}dx \left(\log \frac{1}{x}\right)^{n-1}\right)^2}{\int x^{f-1}dx \left(\log \frac{1}{x}\right)^{2n-1}}=k \int x^{nk-1}dx(1-x^k)^{n-1}.
\]
Therefore, it has to be noted that this equality holds, whatever numbers are taken for $f$ and $g$; in the case $f=g$, this is indeed clear, since

\[
\int x^{g-1}dx(1-x^g)^{n-1}= \frac{1-(1-x^g)^n}{ng}=\frac{1}{ng};
\]
for, it will be

\[
2g\int x^{ng+g-1}dx(1-x^g)^{n-1}=k \int x^{nk-1}dx(1-x^k)^{n-1},
\]
and because

\[
\int x^{ng+g-1}dx(1-x^g)^{n-1}= \frac{1}{2}\int x^{ng-1}dx(1-x^g)^{n-1},
\]
the equality is obvious, because $k$ can be taken arbitrarily. But in the same way we arrived at this theorem, it is possible to get to other similar ones.

\section*{Theorem 3}

\paragraph*{§18}

\textit{If the following integrals are extended from the value $x=0$ to $x=1$ and $n$ denotes any positive integer, it will be}

\[
\frac{1 \cdot 2 \cdot 3 \cdots n}{(2n+1)(2n+2) \cdots 3n}= \frac{2}{3}ng \int x^{f+2ng-1}dx(1-x^g)^{n-1} \cdot \frac{\int x^{f-1}dx(1-x^g)^{2n-1}}{\int x^{f-1}dx(1-x^g)^{3n-1}},
\]
\textit{whatever positive numbers are taken for $f$ and $g$.}

\subsection*{Proof}

In the preceding theorem, we already saw that

\[
\frac{1 \cdot 2 \cdot 3 \cdots 2n}{(f+g)(f+2g)\cdots (f+2ng)}=\frac{f \cdot 2ng}{g^{2n}(f+2ng)}\int x^{f-1}dx(1-x^g)^{2n-1};
\]
if, in like manner, we write $3n$ instead of $n$ in the principal formula, we will have

\[
\frac{1 \cdot 2 \cdot 3 \cdots 3n}{(f+g)(f+2g)\cdots (f+3ng)}=\frac{f \cdot 3ng}{g^{3n}(f+3ng)}\int x^{f-1}dx(1-x^g)^{3n-1};
\]
hence, dividing this equation by the first one, we are led to

\[
\frac{(f+(2n+1)g)(f+(2n+2)g) \cdots (f+3ng)}{(2n+1)(2n+2) \cdots 3n}=\frac{2g^n(f+3ng)}{3(f+2ng)}\cdot \frac{\int x^{f-1}dx(1-x^g)^{2n-1}}{\int x^{f-1}dx(1-x^g)^{3n-1}}.
\]
But if we write $f+2gn$ instead of $f$ in the principal equation (§4), we obtain this equation

\[
\frac{1 \cdot 2 \cdot 3 \cdots n}{(f+(2n+1)g)(f+(2n+2)g)\cdots (f+3ng)}= \frac{(f+2ng)ng}{g^n(f+3ng)}\int x^{f+2ng-1}dx(1-x^g)^{n-1}.
\]
Now multiply this equation by the preceding and the equation to be proved will result

\[
\frac{1 \cdot 2 \cdot 3 \cdots n}{(2n+1)(2n+2) \cdots 3n}= \frac{2}{3}ng \int x^{f+2ng-1}dx(1-x^g)^{n-1} \cdot \frac{\int x^{f-1}dx(1-x^g)^{2n-1}}{\int x^{f-1}dx(1-x^g)^{3n-1}}.
\]

\subsection*{Corollary 1}

\paragraph*{§19}

We obtain the same value from the principal equation by putting $f=2n$ and $g=1$ so that

\[
\frac{1\cdot 2 \cdot 3 \cdots n}{(2n+1)(2n+2) \cdots 3n}= \frac{2}{3}n \int x^{2n-1}dx(1-x)^{n-1},
\]
which integral formula, writing $x^k$ instead of $x$, is transformed into this one

\[
\frac{2}{3}nk \int x^{2nk-1}dx(1-x^k)^{n-1}
\]
such that 

\[
g \int x^{f+2ng-1}dx(1-x^g)^{n-1} \cdot \frac{\int x^{f-1}dx \left(\log \frac{1}{x}\right)^{2n-1}}{\int x^{f-1}dx \left(\log \frac{1}{x}\right)^{3n-1}}=k \int x^{2nk-1}dx(1-x^k)^{n-1}.
\]

\subsection*{Corollary 2}

\paragraph*{§20}

If we set $g=0$ here, because of $1-x^g=g\log \frac{1}{x}$, we will have this equation

\[
\int x^{f-1}dx \left(\log \frac{1}{x}\right)^{n-1}\cdot \frac{\int x^{f-1}dx \left(\log \frac{1}{x}\right)^{2n-1}}{\int x^{f-1}dx \left(\log \frac{1}{x}\right)^{3n-1}}=k \int x^{2nk-1}dx(1-x^k)^{n-1};
\]
because we had found before that

\[
\frac{\left(\int x^{f-1}dx \left(\log \frac{1}{x}\right)^{n-1} \right)^2}{\int x^{f-1}dx \left(\log \frac{1}{x}\right)^{2n-1}}=k \int x^{nk-1}dx(1-x^k)^{n-1},
\]
by multiplying both expressions by each other we will have this equation

\[
\frac{\left(\int x^{f-1}dx \left(\log \frac{1}{x}\right)^{n-1}\right)^3}{\int x^{f-1}dx \left(\log \frac{1}{x}\right)^{3n-1}}=k^2 \int x^{nk-1}dx(1-x^k)^{n-1} \cdot \int x^{2nk-1}dx(1-x^k)^{n-1}.
\]

\subsection*{Corollary 3}

\paragraph*{§21}

Without any restriction one can put $f=1$ here; because,  then for $n=\frac{1}{3}$ and $k=3$ it will be

\[
\frac{\left(\int dx \left(\log \frac{1}{x}\right)^{-\frac{2}{3}}\right)^3}{\int dx \left(\log \frac{1}{x}\right)^{0}}=9 \int dx(1-x^3)^{-\frac{2}{3}} \cdot \int xdx(1-x^3)^{-\frac{2}{3}}
\]
and, because of

\[
\int dx \left(\log \frac{1}{x}\right)^{-\frac{2}{3}}=3 \int dx \left(\log \frac{1}{x}\right)^{\frac{1}{3}} \quad \text{and} \quad \int dx \left(\log \frac{1}{x}\right)^{0}=1,
\]
\[
\left(\int dx \left(\log \frac{1}{x}\right)^{\frac{1}{3}}\right)^3=\frac{1}{3} \int dx(1-x^3)^{-\frac{2}{3}} \cdot \int xdx(1-x^3)^{-\frac{2}{3}};
\]
but then for $n=\frac{2}{3}$ and $k=3$ it will be

\[
\frac{\left(\int dx \left(\log \frac{1}{x}\right)^{-\frac{1}{3}}\right)^3}{\int dx \log \frac{1}{x}}=9 \int xdx(1-x^3)^{-\frac{1}{3}} \cdot \int x^3dx(1-x^3)^{-\frac{1}{3}}
\]
or

\[
\left(\int dx \left(\log \frac{1}{x}\right)^{\frac{2}{3}}\right)^3=\frac{4}{3}\int xdx(1-x^3)^{-\frac{1}{3}} \cdot \int x^3dx(1-x^3)^{-\frac{1}{3}}.
\]

\section*{General Theorem}

\paragraph*{§22}

\textit{If the following integrals are extended from the value $x=0$ to $x=1$ and $n$ denotes a positive integer, it will be}

\[
\frac{1 \cdot 2 \cdot 3 \cdots n}{(\lambda n+1)(\lambda n+2)\cdots (\lambda +1)n}=\frac{\lambda}{\lambda +1}ng \int x^{f+ \lambda ng-1}dx(1-x^g)^{n-1} \cdot \frac{\int x^{f-1}dx(1-x^g)^{\lambda n-1}}{\int x^{f-1}dx(1-x^g)^{(\lambda +1)n-1}},
\]
\textit{whatever positive numbers are taken for the letters $f$ and $g$.}

\subsection*{Proof}

Since, as we showed above,

\[
\frac{1 \cdot 2 \cdots n}{(f+g)(f+2g) \cdots (f+ng)}= \frac{f \cdot ng}{g^n(f+ng)}\int x^{f-1}dx(1-x^g)^{ n-1},
\]
if we write $\lambda n$ instead of $n$ here at first, but then $(\lambda +1)n$ instead of $n$, we will obtain these two equations

\[
\frac{1 \cdot 2 \cdots \lambda n}{(f+g)(f+2g) \cdots(f+\lambda ng)}= \frac{f \cdot \lambda ng}{g^{\lambda n}(f+\lambda ng)}\int x^{f-1}dx(1-x^g)^{\lambda n-1},
\]
\[
\frac{1 \cdot 2 \cdots (\lambda +1) n}{(f+g)(f+2g) \cdots(f+(\lambda+1) ng)}= \frac{f \cdot (\lambda+1) ng}{g^{(\lambda +1) n}(f+(\lambda +1) ng)}\int x^{f-1}dx(1-x^g)^{(\lambda +1) n-1};
\]
dividing the first equation by this one gives

\[
\frac{(f+\lambda ng+g)(f+\lambda ng+2g) \cdots (f+\lambda ng+ng)}{(\lambda n+1)(\lambda n+2) \cdots (\lambda n+n)}=g^n \frac{\lambda(f+\lambda ng+ng)}{(\lambda +1)(f+\lambda ng)} \cdot \frac{\int x^{f-1}dx(1-x^g)^{\lambda n-1}}{\int x^{f-1}dx(1-x^g)^{(\lambda +1) n-1}}.
\]

But if we  write $f+\lambda ng$ instead of $f$ in the first equation, we will obtain

\[
\frac{1 \cdot 2 \cdots n}{(f+\lambda ng+g)(f+\lambda n g+2g) \cdots (f+\lambda ng +ng)}=\frac{(f+\lambda ng)ng}{g^n(f+ \lambda ng+ng)}\int x^{f+\lambda ng-1}dx(1-x^g)^{n-1},
\]
which two equations multiplied by each other produce the equation to be demonstrated 
\[
\frac{1 \cdot 2 \cdots n}{(\lambda n+1)(\lambda n+2) \cdots (\lambda n+n)}= \frac{\lambda n g}{\lambda +1}\int x^{f+\lambda ng-1}dx(1-x^g)^{n-1} \cdot \frac{\int x^{f-1}dx(1-x^g)^{\lambda n-1}}{\int x^{f-1}dx(1-x^g)^{(\lambda +1)n-1}}.
\]

\subsection*{Corollary 1}

\paragraph*{§23}

If we set $f=\lambda n$ and $g=1$ in the principal equation, we will also find

\[
\frac{1 \cdot 2 \cdots n}{(\lambda n+1)(\lambda n+2) \cdots (\lambda n +n)}=\frac{\lambda n}{\lambda +1}\int x^{\lambda n-1}dx(1-x)^{n-1},
\]
which form writing $x^k$ instead of $x$ changes into this one

\[
\frac{\lambda nk}{\lambda +1}\int x^{\lambda nk-1}dx(1-x^k)^{n-1}
\]
such that we have this very far-extending theorem

\[
g \int x^{f+\lambda ng-1}dx(1-x^g)^{n-1} \cdot \frac{\int x^{f-1}dx(1-x^g)^{\lambda n-1}}{\int x^{f-1}dx(1-x^g)^{\lambda n+n-1}}=k \int x^{\lambda nk-1}dx(1-x^k)^{n-1}.
\] 

\subsection*{Corollary 2}

\paragraph*{§24}

This theorem now holds, even if $n$ is not an integer; because the number $\lambda$ can be taken arbitrarily, let us even write $m$ instead of $\lambda n$ and we will find this theorem

\[
\frac{\int x^{f-1}dx(1-x^g)^{m-1}}{\int x^{f-1}dx(1-x^g)^{m+n-1}}=\frac{k \int x^{mk-1}dx(1-x^k)^{n-1}}{g\int x^{f+mg-1}dx(1-x^g)^{n-1}}.
\]

\subsection*{Corollary 3}

\paragraph*{§25}

If we set $g=0$, because of $1-x^g=g \log \frac{1}{x}$, that theorem will take  this form

\[
\frac{\int x^{f-1}dx \left(\log \frac{1}{x}\right)^{m-1}}{\int x^{f-1}dx \left(\log \frac{1}{x}\right)^{m+n-1}}=\frac{k \int x^{mk-1}dx(1-x^k)^{n-1}}{\int x^{f-1}dx \left(\log \frac{1}{x}\right)^{n-1}},
\]
which is more conveniently represented as follows

\[
\frac{\int x^{f-1}dx \left(\log \frac{1}{x}\right)^{n-1} \cdot \int x^{f-1}dx \left(\log \frac{1}{x}\right)^{m-1}}{\int x^{f-1}dx \left(\log \frac{1}{x}\right)^{m+n-1}}=k\int x^{mk-1}dx(1-x^k)^{n-1};
\]
here, it is evident  that the numbers $m$ and $n$ can be permuted.

\subsection*{Scholium}

\paragraph*{§26}

Thus, we found two ways, along which many comparisons and relations of integrals formulas can be derived; the one way, found in § 24, contains integrals of this kind

\[
\int x^{p-1}dx(1-x^g)^{q-1},
\]
which I already treated some time ago in my observations on the integrals\footnote{Euler again refers to his paper "'Observationes circa integralia formularum $\int x^{p-1}dx(1-x^n)^{\frac{q}{n}-1}$ posito post integrationem $x = 1$"'. This is paper E321 in the Eneström-Index.}

\[
\int x^{p-1}dx(1-x^n)^{\frac{q}{n}-1},
\]
extended from the value $x=0$ to $x=1$; there I showed at first that the letters $p$ and $q$ can be interchanged such that 

\[
\int x^{p-1}dx(1-x^n)^{\frac{q}{n}-1}=\int x^{q-1}dx(1-x^n)^{\frac{p}{n}-1},
\]
but then that

\[
\int \frac{x^{p-1}dx}{(1-x^n)^{\frac{p}{n}}}=\frac{\pi}{n\sin \frac{p \pi}{n}};
\]
But in particular, I demonstrated that 

\[
\int \frac{x^{p-1}dx}{\sqrt[n]{(1-x^n)^{n-q}}}\cdot \int \frac{x^{p+q-1}dx}{\sqrt[n]{(1-x^n)^{n-r}}}=\int \frac{x^{p-1}dx}{\sqrt[n]{(1-x^n)^{n-r}}} \cdot \int \frac{x^{p+r-1}dx}{\sqrt[n]{(1-x^n)^{n-q}}};
\]
 the comparison found in § 24 is already contained in this equation such that  nothing, which I have not already explained, can be deduced from this. Therefore, here I mainly attempt to follow  the other way explained in § 25; since without any restriction one can take $f=1$, our primary equation will be

\[
\frac{\int dx \left(\log \frac{1}{x}\right)^{n-1} \cdot \int dx \left(\log \frac{1}{x}\right)^{m-1}}{\int dx \left(\log \frac{1}{x}\right)^{m+n-1}}=k \int x^{mk-1}dx(1-x^k)^{n-1},
\]
by means of which the values of the integral formula $\int dx \left(\log \frac{1}{x}\right)^{\lambda}$, if $\lambda$ is not an integer, can be reduced to quadratures of algebraic curves; since, if $\lambda$ is an integer, the integrals can be solved explicitly, because

\[
\int dx \left(\log \frac{1}{x}\right)^{\lambda}=1 \cdot 2 \cdot 3 \cdots \lambda.
\]
But the question of greatest importance concerns the cases, in which $\lambda$ is a rational number. Therefore, I will define these here successively for some small denominators.

\section*{Problem 2}

\paragraph*{§27}

\textit{While $i$ denotes a positive integer, to define the value of the integral $\int dx \left(\log \frac{1}{x}\right)^{\frac{i}{2}}$, having extended the integration from $x=0$ to $x=1$.}

\subsection*{Solution}

Let us put $m=n$ in our general equation and it will be

\[
\frac{\left(\int dx \left(\log \frac{1}{x}\right)^{n-1}\right)^2}{\int dx \left(\log \frac{1}{x}\right)^{2n-1}}=k\int x^{nk-1}dx(1-x^k)^{n-1}.
\]
Now let $n-1=\frac{i}{2}$ and, because of $2n-1=i+1$, it will be

\[
\int dx \left(\log \frac{1}{x}\right)^{2n-1}=1 \cdot 2 \cdot 3 \cdots (i+1);
\]
now further take $k=2$ so that $nk-1=i+1$, and it will be

\[
\frac{\left(\int dx\sqrt{\left(\log \frac{1}{x}\right)^i}\right)^2}{1 \cdot 2 \cdot 3 \cdots (i+1)}=2\int x^{i+1}dx(1-x^2)^{\frac{i}{2}}
\]
and hence

\[
\frac{\int dx\sqrt{\left(\log \frac{1}{x}\right)^i}}{\sqrt{1 \cdot 2 \cdot 3 \cdots (i+1)}}=\sqrt{2\int x^{i+1}dx(1-x^2)^{\frac{i}{2}}},
\]
where it is evidently sufficient to take only odd numbers for $i$, because for the even ones the expansion is  immediately obvious.

\subsection*{Corollary 1}

\paragraph*{§28}

But all cases are easily reduced to $i=1$ or even to $i=-1$; for, if $i+1$ is not a negative number, the reduction we found holds. For this case it will therefore be

\[
\int \frac{dx}{\sqrt{\log \frac{1}{x}}}=\sqrt{2\int \frac{dx}{\sqrt{1-xx}}}=\sqrt{\pi},
\]
because of $\int \frac{dx}{\sqrt{1-xx}}=\frac{\pi}{2}.$

\subsection*{Corollary 2}

\paragraph*{§29}

But having covered these principal cases, because of

\[
\int dx \left(\log \frac{1}{x}\right)^n=n \int dx \left(\log \frac{1}{x}\right)^{n-1},
\]
we will have

\[
\int dx \sqrt{\log \frac{1}{x}}= \frac{1}{2}\sqrt{\pi}, \quad \int dx \left(\log \frac{1}{x}\right)^{\frac{3}{2}}=\frac{1 \cdot 3}{2 \cdot 2}\sqrt{\pi},
\]
and in general

\[
\int dx \left(\log \frac{1}{x}\right)^{\frac{2n+1}{2}}=\frac{1}{2}\cdot \frac{3}{2} \cdot  \frac{5}{2} \cdot \frac{7}{2} \cdots \frac{2n+1}{2}\sqrt{\pi}.
\]

\section*{Problem 3}

\paragraph*{§30}

\textit{While $i$ denotes a positive integer, to define the value of the integral $\int dx \left(\log \frac{1}{x}\right)^{\frac{i}{3}-1}$, having extended the integration from $x=0$ to $x=1$, of course.}

\subsection*{Solution}

Let us start from the equation of the preceding problem

\[
\frac{\left(\int dx \left(\log \frac{1}{x}\right)^{n-1}\right)^2}{\int dx \left(\log \frac{1}{x}\right)^{2n-1}}=k\int x^{nk-1}dx(1-x^k)^{n-1}
\]
and let us set $m=2n$ in the general formula such that one has

\[
\frac{\int dx \left(\log \frac{1}{x}\right)^{n-1} \cdot \int dx \left(\log \frac{1}{x}\right)^{2n-1}}{\int dx \left(\log \frac{1}{x}\right)^{3n-1}}= k \int x^{2nk-1}dx(1-x^k)^{n-1},
\]
and by multiplying these two equations we obtain

\[
\frac{\left(\int dx \left(\log \frac{1}{x}\right)^{n-1}\right)^3}{\int dx \left(\log \frac{1}{x}\right)^{3n-1}}=kk\int x^{nk-1}dx(1-x^k)^{n-1} \cdot \int x^{2nk-1}dx(1-x^k)^{n-1}.
\]
Now just set $n=\frac{i}{3}$ here such that

\[
\int dx \left(\log \frac{1}{x}\right)^{i-1}=1 \cdot 2 \cdot 3 \cdots (i-1),
\]
and take $k=3$ and this equation will result

\[
\frac{\left(\int dx \sqrt[3]{\left(\log \frac{1}{x}\right)^{i-3}}\right)^3}{1 \cdot 2 \cdot 3 \cdots (i-1)}=9 \int x^{i-1}dx \sqrt[3]{(1-x^3)^{i-3}} \cdot \int x^{2i-1}dx \sqrt[3]{(1-x^3)^{i-3}},
\]
whence we conclude

\[
\frac{\int dx \sqrt[3]{\left(\log \frac{1}{x}\right)^{i-3}}}{\sqrt{1 \cdot 2 \cdot 3 \cdots (i-1)}}=\sqrt[3]{9 \int \frac{x^{i-1}dx}{\sqrt[3]{(1-x^3)^{3-i}}} \cdot \int \frac{x^{2i-1}dx}{\sqrt[3]{(1-x^3)^{3-i}}}}.
\]

\subsection*{Corollary 1}

\paragraph*{§31}

Here, two principal cases occur, on which all remaining ones depend, namely the cases  $i=1$ and $i=2$; for these cases

\begin{alignat*}{18}
&\text{I.} \quad && \int \frac{dx}{\sqrt[3]{\left(\log \frac{1}{x}\right)^2}}&&= \sqrt[3]{9 \int \frac{dx}{\sqrt[3]{(1-x^3)^{2}}}\cdot \int \frac{xdx}{\sqrt[3]{(1-x^3)^{2}}} },\\
&\text{II.} \quad && \int \frac{dx}{\sqrt[3]{\log \frac{1}{x}}}&&= \sqrt[3]{9 \int \frac{dx}{\sqrt[3]{1-x^3}}\cdot \int \frac{x^3dx}{\sqrt[3]{1-x^3}}};
\end{alignat*}
the last formula, because of

\[
\int \frac{x^3dx}{\sqrt[3]{1-x^3}}=\frac{1}{3}\int \frac{dx}{\sqrt[3]{1-x^3}},
\]
can be transformed into this one

\[
\int \frac{dx}{\sqrt[3]{\log \frac{1}{x}}}= \sqrt[3]{\int \frac{dx}{\sqrt[3]{1-x^3}} \cdot \int \frac{xdx}{\sqrt[3]{1-x^3}}}
\]

\subsection*{Corollary 2}

\paragraph*{§32}

If, for the sake of brevity,  as in my observations mentioned before\footnote{Euler refers to his paper E321 again.} we set

\[
\int \frac{x^{p-1}dx}{\sqrt[3]{(1-x^3)^{3-q}}}=\left(\frac{p}{q}\right)
\]
and, as we did it there, for this class also set

\[
\left(\frac{2}{1}\right)=\frac{\pi}{3 \sin \frac{\pi}{3}}=\alpha,
\]
but then put

\[
\left(\frac{1}{1}\right)=\int \frac{dx}{\sqrt[3]{(1-x^3)^2}}=A,
\]
it will be

\begin{alignat*}{18}
& \text{I.} \quad && \int \frac{dx}{\sqrt[3]{\left(\log \frac{1}{x}\right)^2}}&&=\sqrt[3]{9\left(\frac{1}{1}\right)\left(\frac{2}{1}\right)}&&=\sqrt[3]{9\alpha A},\\
& \text{II.} \quad && \int \frac{dx}{\sqrt[3]{\left(\log \frac{1}{x}\right)^1}}&&=\sqrt[3]{3\left(\frac{1}{2}\right)\left(\frac{2}{2}\right)}&&=\sqrt[3]{\frac{3 \alpha \alpha}{A}}.
\end{alignat*}

\subsection*{Corollary 3}

\paragraph*{§33}

Therefore,  for the first case we will have

\[
\int dx \sqrt[3]{\left(\log \frac{1}{x}\right)^{-2}}=\sqrt[3]{9 \alpha A}, \quad \int dx \sqrt[3]{\log \frac{1}{x}}=\frac{1}{3}\sqrt[3]{9 \alpha A}
\]
and

\[
\int dx \sqrt[3]{\left(\log \frac{1}{x}\right)^{3n+1}}=\frac{1}{3}\cdot \frac{4}{3} \cdot \frac{7}{3}\cdots \frac{3n+1}{3}\sqrt[3]{9 \alpha A},
\] 
but for the other case, on the other hand, we will find

\[
\int dx \sqrt[3]{\left(\log \frac{1}{x}\right)^{-1}}=\sqrt[3]{\frac{3 \alpha \alpha}{A}}, \quad \int dx \sqrt[3]{\left(\log \frac{1}{x}\right)^{2}}=\frac{2}{3}\sqrt[3]{\frac{3 \alpha \alpha}{A}}
\]
and

\[
\int dx \sqrt[3]{\left(\log \frac{1}{x}\right)^{3n-1}}=\frac{2}{3}\cdot \frac{5}{3}\cdot \frac{8}{3} \cdots \frac{3n-1}{3} \sqrt[3]{\frac{3 \alpha \alpha}{A}}.
\]

\section*{Problem 4}

\paragraph*{§34}

\textit{While $i$ denotes a positive integer, to define the value of the integral $\int dx \left(\log \frac{1}{x}\right)^{\frac{i}{4}-1}$, having extended the integration from $x=0$ to $x=1$.}

\subsection*{Solution}

In the solution of the preceding problem, we were led to this equation

\[
\frac{\left(\int dx \left(\log \frac{1}{x}\right)^{n-1}\right)^3}{\int dx\left(\log \frac{1}{x}\right)^{3n-1}}= kk \int \frac{x^{nk-1}dx}{(1-x^k)^{1-n}} \cdot \int \frac{x^{2nk-1}dx}{(1-x^k)^{1-n}};
\]
but the general formula, setting $m=3n$ in it, yields

\[
\frac{\int dx\left(\log \frac{1}{x}\right)^{n-1} \cdot \int dx\left(\log \frac{1}{x}\right)^{3n-1}}{\int dx\left(\log \frac{1}{x}\right)^{4n-1}}=k\int \frac{x^{3nk-1}dx}{(1-x^k)^{1-n}};
\]
combining these formulas we obtain

\[
\frac{\left(\int dx\left(\log \frac{1}{x}\right)^{n-1}\right)^4}{\int dx\left(\log \frac{1}{x}\right)^{4n-1}}=k^3 \int \frac{x^{nk-1}dx}{(1-x^k)^{1-n}} \cdot \int \frac{x^{2nk-1}dx}{(1-x^k)^{1-n}} \cdot \int \frac{x^{3nk-1}dx}{(1-x^k)^{1-n}}.
\]
Let $n=\frac{i}{4}$ and take $k=4$ and it will be

\[
\frac{\int dx \left(\log \frac{1}{x}\right)^{\frac{i}{4}-1}}{\sqrt[4]{1 \cdot 2 \cdot 3 \cdots (i-1)}}=\sqrt[4]{4^3 \int \frac{x^{i-1}dx}{\sqrt[4]{(1-x^4)^{4-i}}} \cdot \int \frac{x^{2i-1}dx}{\sqrt[4]{(1-x^4)^{4-i}}} \cdot \int \frac{x^{3i-1}dx}{\sqrt[4]{(1-x^4)^{4-i}}}}.
\]

\subsection*{Corollary 1}

\paragraph*{§35}

So if $i=1$, we will have this equation

\[
\int dx\sqrt[4]{\left(\log \frac{1}{x}\right)^{-3}}= \sqrt[4]{4^3 \int \frac{dx}{\sqrt[4]{(1-x^4)^{3}}} \cdot \int \frac{xdx}{\sqrt[4]{(1-x^4)^{3}}}\cdot \int \frac{x^2dx}{\sqrt[4]{(1-x^4)^{3}}}};
\]
if this expression is denoted by the letter $P$, it will be in general

\[
\int dx \sqrt[4]{\left(\log \frac{1}{x}\right)^{4n-3}}= \frac{1}{4} \cdot \frac{5}{4} \cdot \frac{9}{4} \cdots \frac{4n-3}{4}P.
\]

\subsection*{Corollary 2}

\paragraph*{§36}

For the other principal case, let us take $i=3$ and it will be

\[
\int dx\sqrt[4]{\left(\log \frac{1}{x}\right)^{-1}}= \sqrt[4]{2 \cdot 4^3 \int \frac{x^2dx}{\sqrt[4]{1-x^4}} \cdot \int \frac{x^5dx}{\sqrt[4]{1-x^4}}\cdot \int \frac{x^8dx}{\sqrt[4]{1-x^4}}}
\]
or, after some simplification,

\[
\int dx\sqrt[4]{\left(\log \frac{1}{x}\right)^{-1}}= \sqrt[4]{8 \int \frac{xxdx}{\sqrt[4]{1-x^4}} \cdot \int \frac{xdx}{\sqrt[4]{1-x^4}}\cdot \int \frac{dx}{\sqrt[4]{1-x^4}}};
\]
if this expression is denoted by the letter $Q$, in general, it will be 

\[
\int dx \sqrt[4]{\left(\log \frac{1}{x}\right)^{4n-1}}=\frac{3}{4} \cdot \frac{7}{4} \cdot \frac{11}{4} \cdots \frac{4n-1}{4}Q.
\]

\subsection*{Scholium}

\paragraph*{§37}

If we indicate the integral formula $\int \frac{x^{p-1}dx}{\sqrt[4]{(1-x^4)^{4-q}}}$ by the sign $\left(\frac{p}{q}\right)$, in general, the solution  will be as follows

\[
\int dx \sqrt[4]{\log \left( \frac{1}{x}\right)^{i-4}}=\sqrt[4]{1 \cdot 2 \cdot 3 \cdots (i-1)4^3\left(\frac{i}{i}\right)\left(\frac{2i}{i}\right)\left(\frac{3i}{i}\right)}
\]
and for the two cases expanded before

\[
P=\sqrt[4]{4^3\left(\frac{1}{1}\right)\left(\frac{2}{1}\right)\left(\frac{3}{1}\right)} \quad \text{and} \quad Q=\sqrt[4]{8\left(\frac{3}{3}\right)\left(\frac{2}{3}\right)\left(\frac{1}{3}\right)}.
\]
Now for the formulas depending on the circle, let us set 
\[
\left(\frac{3}{1}\right)=\frac{\pi}{4 \sin \frac{\pi}{4}}=\alpha \quad \text{and} \quad \left(\frac{2}{2}\right)=\frac{\pi}{4 \sin \frac{2 \pi}{4}}=\beta,
\]
but for the transcendental ones of higher order let

\[
\left(\frac{2}{1}\right)=\int \frac{xdx}{\sqrt[4]{(1-x^4)^3}}=\int \frac{dx}{\sqrt[2]{1-x^4}}=A,
\]
on which all remaining ones depend; hence we will find

\[
P= \sqrt[4]{4^3 \frac{\alpha \alpha}{\beta}AA} \quad \text{and} \quad Q=\sqrt[4]{4 \alpha \alpha \beta \frac{1}{AA}},
\]
whence it is clear that 

\[
PQ=4\alpha = \frac{\pi}{\sin \frac{\pi}{4}}.
\]
But because  $\alpha =\frac{\pi}{2\sqrt{2}}$ and $\beta =\frac{\pi}{4}$, it will be

\[
P=\sqrt[4]{32 \pi AA} \quad \text{and} \quad Q=\sqrt[4]{\frac{\pi^3}{8AA}} \quad \text{and} \quad \frac{P}{Q}=\frac{4A}{\sqrt{\pi}}.
\]

\section*{Problem 5}

\paragraph*{§38}

\textit{While $i$ denotes a positive integer, to define the value of the integral $\int dx \sqrt[5]{\left(\log \frac{1}{x}\right)^{{i-5}}}$, having extended the integration from $x=0$ to $x=1$, of course.}

\subsection*{Solution}

From the preceding solutions it is already perspicuous that for this case one will  obtain this formula at the end

\[
\frac{\int dx \sqrt[5]{\left(\log \frac{1}{x}\right)^{i-5}}}{\sqrt[5]{1 \cdot 2 \cdot 3 \cdots (i-1)}}=\sqrt[5]{5^4 \int \frac{x^{i-1}dx}{\sqrt[5]{(1-x^5)^{5-i}}} \cdot \int \frac{x^{2i-1}dx}{\sqrt[5]{(1-x^5)^{5-i}}} \cdot \int \frac{x^{3i-1}dx}{\sqrt[5]{(1-x^5)^{5-i}}} \cdot \int \frac{x^{4i-1}dx}{\sqrt[5]{(1-x^5)^{5-i}}}},
\]
which integral formulas belong to the fifth class introduced in my dissertation mentioned above\footnote{Euler again refers to his paper "'Observationes circa integralia formularum $\int x^{p-1}dx(1-x^n)^{\frac{q}{n}-1}$ posito post integrationem $x = 1$"'. This is paper E321 in the Eneström-Index.}. Hence, if in the same way as it was done there the sign $\left(\frac{p}{q}\right)$ denotes this formula $\int \frac{x^{p-1}dx}{\sqrt[5]{(1-x^5)^{5-q}}}$, the value in question can be more conveniently expressed in such a way that

\[
\int dx \sqrt[5]{\left(\log \frac{1}{x}\right)^{i-5}}= \sqrt[5]{1 \cdot 2 \cdot 3 \cdots (i-1) 5^4 \left(\frac{i}{i}\right)\left(\frac{2i}{i}\right)\left(\frac{3i}{i}\right)\left(\frac{4i}{i}\right)};
\]
here it indeed suffices to have assigned values smaller than five to $i$; for, if the numerators exceed five just note that

\[
\left(\frac{5+m}{i}\right)=\frac{m}{m+i}\left(\frac{m}{i}\right),
\] 
but then further

\[
\left(\frac{10+m}{i}\right)=\frac{m}{m+i}\cdot \frac{m+5}{m+i+5} \left(\frac{m}{i}\right),
\]
\[
\left(\frac{15+m}{i}\right)=\frac{m}{m+i}\cdot \frac{m+5}{m+i+5} \cdot \frac{m+10}{m+i+10} \left(\frac{m}{i}\right).
\]
Furthermore, for this class  two formulas indeed involve the quadrature of the circle; these formulas are

\[
\left(\frac{4}{1}\right)=\frac{\pi}{5 \sin \frac{\pi}{5}}=\alpha \quad \text{and} \quad \left(\frac{3}{2}\right)= \frac{\pi}{5 \sin \frac{2 \pi}{5}}=\beta,
\]
but then two contain higher quadratures, which we want to put as

\[
\left(\frac{3}{1}\right)=\int \frac{xxdx}{\sqrt[5]{(1-x^5)^4}}=\int \frac{dx}{\sqrt[5]{(1-x^5)^2}}=A \quad \text{and} \quad \left(\frac{2}{2}\right)=\int \frac{xdx}{\sqrt[5]{(1-x^5)^3}}=B,
\]
and using these I assigned the values of all remaining formulas of this class\footnote{Euler took the following list out of E321.}, namely

\begin{alignat*}{18}
& \left(\frac{5}{1}\right)&&=1, \quad &&\left(\frac{5}{2}\right)&&=\frac{1}{2}, \quad &&\left(\frac{5}{3}\right)&&=\frac{1}{3}&&\quad \left(\frac{5}{4}\right)&&=\frac{1}{4}, \quad &&\left(\frac{5}{5}\right)=\frac{1}{5};\\[2mm] 
& \left(\frac{4}{1}\right)&&=\alpha, \quad &&\left(\frac{4}{2}\right)&&=\frac{\beta}{A}, \quad &&\left(\frac{4}{3}\right)&&=\frac{\beta}{2B}&&\quad \left(\frac{4}{4}\right)&&=\frac{\alpha}{3A};\\[2mm]
& \left(\frac{3}{1}\right)&&=A, \quad &&\left(\frac{3}{2}\right)&&=\beta, \quad &&\left(\frac{3}{3}\right)&&=\frac{\beta \beta}{\alpha B};\\[2mm]
& \left(\frac{2}{1}\right)&&=\frac{\alpha B}{\beta}, \quad &&\left(\frac{2}{2}\right)&&=B;\\[2mm]
& \left(\frac{1}{1}\right)&&=\frac{\alpha A}{\beta}.
\end{alignat*}

\subsection*{Corollary 1}

Having taken the exponent $i=1$, it will be

\[
\int dx \sqrt[5]{\left(\log \frac{1}{x}\right)^{-4}}=\sqrt[5]{5^4 \left(\frac{1}{1}\right)\left(\frac{2}{1}\right)\left(\frac{3}{1}\right)\left(\frac{4}{1}\right)}=\sqrt[5]{5^4 \frac{\alpha ^3}{\beta ^2}A^2B},
\]
whence we conclude in general, while $n$ denotes a positive integer, that

\[
\int dx \sqrt[5]{\left(\log \frac{1}{x}\right)^{5n-4}}=\frac{1}{5}\cdot \frac{6}{5} \cdot \frac{11}{5} \cdots \frac{5n-4}{5}\sqrt[5]{5^4 \frac{\alpha ^3}{\beta}A^2B}.
\]

\subsection*{Corollary 2}

\paragraph*{§40}

Now let $i=2$, and since then this equation results
\[
\int dx \sqrt[5]{\left(\log \frac{1}{x}\right)^{-3}}=\sqrt[5]{5^4 \left(\frac{2}{2}\right)\left(\frac{4}{2}\right)\left(\frac{6}{2}\right)\left(\frac{8}{2}\right)},
\]
because of

\[
\left(\frac{6}{2}\right)=\frac{1}{3}\left(\frac{1}{2}\right)=\frac{1}{3}\left(\frac{2}{1}\right) \quad \text{and} \quad  \left(\frac{8}{2}\right)=\frac{3}{3}\left(\frac{3}{2}\right),
\]
the left-hand side will be

\[
\sqrt[5]{5^3\left(\frac{2}{2}\right)\left(\frac{4}{2}\right)\left(\frac{2}{1}\right)\left(\frac{3}{2}\right)}=\sqrt[5]{5^3 \alpha \beta \frac{BB}{A}}
\]
and in general

\[
\int dx \sqrt[5]{\left(\log \frac{1}{x}\right)^{5n-3}}=\frac{2}{5}\cdot \frac{7}{5} \cdot  \frac{12}{5} \cdots \frac{5n-3}{5}\sqrt[5]{5^3 \alpha \beta \frac{BB}{A}}.
\]

\subsection*{Corollary 3}

\paragraph*{§41}

Let $i=3$ and the form found

\[
\int dx \sqrt[5]{\left(\log \frac{1}{x}\right)^{-2}}=\sqrt[5]{2 \cdot 5^4\left(\frac{3}{3}\right)\left(\frac{6}{3}\right)\left(\frac{9}{3}\right)\left(\frac{12}{3}\right)},
\]
because of

\[
\left(\frac{6}{3}\right)=\frac{1}{4}\left(\frac{3}{1}\right), \quad \left(\frac{9}{3}\right)=\frac{4}{7}\left(\frac{4}{3}\right), \quad \left(\frac{12}{3}\right)=\frac{2}{5}\cdot \frac{7}{10}\left(\frac{3}{2}\right),
\]
changes to

\[
\sqrt[5]{2 \cdot 5^2 \left(\frac{3}{3}\right)\left(\frac{3}{1}\right)\left(\frac{4}{3}\right)\left(\frac{3}{2}\right)}=\sqrt[5]{5^2 \frac{\beta^4}{\alpha}\cdot \frac{A}{BB}},
\]
whence it is concluded that  in general

\[
\int dx \sqrt[5]{\left(\log \frac{1}{x}\right)^{5n-2}}=\frac{3}{5}\cdot \frac{8}{5} \cdot \frac{13}{5} \cdots \frac{5n-2}{5}\sqrt[5]{5^2 \frac{\beta^4}{\alpha}\cdot \frac{A}{BB}}.
\]

\subsection*{Corollary 4}

\paragraph*{§42}

Finally, for $i=4$ our equation 

\[
\int dx \sqrt[5]{\left(\log \frac{1}{x}\right)^{-1}}=\sqrt[5]{6 \cdot 5^4 \left(\frac{4}{4}\right)\left(\frac{8}{4}\right)\left(\frac{12}{4}\right)\left(\frac{16}{4}\right)},
\]
because of

\[
\left(\frac{8}{4}\right)=\frac{3}{7}\left(\frac{4}{3}\right), \quad \left(\frac{12}{4}\right)=\frac{2}{6}\cdot \frac{7}{11}\left(\frac{4}{2}\right), \quad \left(\frac{16}{4}\right)=\frac{1}{5}\cdot \frac{6}{10} \cdot \frac{11}{15} \left(\frac{4}{1}\right),
\]
will be transformed into this form

\[
\sqrt[5]{6 \cdot 5\left(\frac{4}{4}\right)\left(\frac{4}{3}\right)\left(\frac{4}{2}\right)\left(\frac{4}{1}\right)}=\sqrt[5]{5 \frac{\alpha \alpha \beta \beta}{AAB}}
\]
such that in general

\[
\int dx \sqrt[5]{\left(\log \frac{1}{x}\right)^{5n-1}}=\frac{4}{5}\cdot \frac{9}{5} \cdot \frac{14}{5} \cdots \frac{5n-1}{5}\sqrt[5]{5 \alpha \alpha \beta \beta \frac{1}{AAB}}.
\]

\subsection*{Scholium}

\paragraph*{§43}

If we represent the value of the integral formula $\int dx \left(\log \frac{1}{x}\right)^{\lambda}$ by the sign $\left[\lambda\right]$, the cases expanded up to now yield

\begin{alignat*}{18}
&\left[-\frac{4}{5}\right]&&=\sqrt[5]{5^4 \frac{\alpha^3}{\beta^2}\cdot A^2B}, \quad &&\left[+\frac{1}{5}\right]&&=\frac{1}{5}\sqrt[5]{5^4 \frac{\alpha^3}{\beta^2}\cdot A^2B},\\[2mm]
&\left[-\frac{3}{5}\right]&&=\sqrt[5]{5^3 \alpha \beta \cdot \frac{BB}{A}}, \quad &&\left[+\frac{2}{5}\right]&&=\frac{2}{5}\sqrt[5]{5^3 \alpha \beta \frac{BB}{A}},\\[2mm]
&\left[-\frac{2}{5}\right]&&=\sqrt[5]{5^2 \frac{\beta^4}{\beta}\cdot \frac{A}{BB}}, \quad &&\left[+\frac{3}{5}\right]&&=\frac{3}{5}\sqrt[5]{5^2 \frac{\beta^4}{\alpha}\cdot \frac{A}{BB}},\\[2mm]
&\left[-\frac{1}{5}\right]&&=\sqrt[5]{5 \alpha^2 \beta^2 \cdot \frac{1}{AAB}}, \quad &&\left[+\frac{4}{5}\right]&&=\frac{4}{5}\sqrt[5]{5 \alpha^2 \beta^2 \cdot \frac{1}{AAB}},
\end{alignat*}
whence by combining two, whose indices add up to $0$, we conclude

\begin{alignat*}{18}
&\left[+\frac{1}{5}\right]&&\cdot \left[-\frac{1}{5}\right]&&=\alpha &&=\frac{\pi}{5\sin \frac{\pi}{5}},\\[2mm]
&\left[+\frac{2}{5}\right]&&\cdot \left[-\frac{2}{5}\right]&&=2 \beta &&=\frac{2\pi}{5\sin \frac{2\pi}{5}},\\[2mm]
&\left[+\frac{3}{5}\right]&&\cdot \left[-\frac{3}{5}\right]&&=3 \beta &&=\frac{3\pi}{5\sin \frac{3\pi}{5}},\\[2mm]
&\left[+\frac{4}{5}\right]&&\cdot \left[-\frac{4}{5}\right]&&=4\alpha &&=\frac{\pi}{5\sin \frac{4\pi}{5}}.
\end{alignat*}
But from the preceding problem, in like manner, we deduce:

\begin{alignat*}{18}
&\left[-\frac{3}{4}\right]&&=P&&=\sqrt[4]{4^3 \frac{\alpha \alpha}{\beta} \cdot AA}, \quad &&\left[+\frac{1}{4}\right]&&=\frac{1}{4}\sqrt[4]{4^3 \frac{\alpha \alpha}{\beta} \cdot AA},\\
&\left[-\frac{1}{4}\right]&&=Q&&=\sqrt[4]{4 \alpha \alpha \beta \cdot \frac{1}{AA}}, \quad &&\left[+\frac{3}{4}\right]&&=\frac{3}{4}\sqrt[4]{4 \alpha \alpha \beta \cdot \frac{1}{AA}}
\end{alignat*}
and hence

\begin{alignat*}{18}
&\left[+\frac{1}{4}\right]&&\cdot \left[-\frac{1}{4}\right]&&=\alpha &&=\frac{\pi}{4 \sin \frac{\pi}{4}},\\
&\left[+\frac{3}{4}\right]&&\cdot \left[-\frac{3}{4}\right]&&=3\alpha &&=\frac{3\pi}{4 \sin \frac{3\pi}{4}},
\end{alignat*}
whence, in general, we obtain this theorem that

\[
\left[\lambda\right] \cdot \left[-\lambda \right]=\frac{\lambda \pi}{\sin \lambda \pi};
\]
the reason for this can be given from the interpolation method explained some time ago\footnote{Euler explains the interpolation method he talks about here also in E19.} as follows. Since

\[
\left[\lambda\right]=\frac{1^{1-\lambda}\cdot 2^{\lambda}}{1+\lambda}\cdot \frac{2^{1-\lambda}\cdot 3^{\lambda}}{2+\lambda}\cdot \frac{3^{1-\lambda}\cdot 4^{\lambda}}{3+\lambda}\cdot \text{etc.},
\]
it will be

\[
\left[-\lambda\right]=\frac{1^{1+\lambda}\cdot 2^{-\lambda}}{1-\lambda}\cdot \frac{2^{1+\lambda}\cdot 3^{-\lambda}}{2-\lambda}\cdot \frac{3^{1+\lambda}\cdot 4^{-\lambda}}{3-\lambda}\cdot \text{etc.}
\]
and hence

\[
\left[\lambda\right]\cdot \left[-\lambda\right]=\frac{1 \cdot 1}{1 -\lambda \lambda} \cdot \frac{2 \cdot 2}{4 -\lambda \lambda} \cdot \frac{3 \cdot 3}{9 -\lambda \lambda}\cdot \text{etc.}=\frac{\lambda \pi}{\sin \lambda \pi},
\]
as I demonstrated elsewhere\footnote{Euler proved this relation in his paper "'Methodus facilis computandi angulorum sinus ac tangentes tam naturales quam artificiales"'. This is paper E128 in the Eneström-Index.}.

\section*{Problem 6 - General Problem}

\paragraph*{§44}

\textit{If the letters $i$ and $n$ denote positive integers, to define the value of the integral}

\[
\int dx \left(\log \frac{1}{x}\right)^{\frac{i-n}{n}} \quad \textit{or} \quad \int dx \sqrt[n]{\left(\log \frac{1}{x}\right)^{i-n}}
\]
\textit{having extended the integration from $x=0$ to $x=1$.}

\subsection*{Solution}

The  method explained up to this point will exhibit the value in question expressed via  quadratures of algebraic curves in the following way

\[
\frac{\int dx \sqrt[n]{\left(\log \frac{1}{x}\right)^{i-n}}}{\sqrt[n]{1 \cdot 2 \cdot3 \cdots (i-1)}}= \sqrt[n]{n^{n-1} \int \frac{x^{i-1}dx}{\sqrt[n]{(1-x^n)^{n-i}}} \cdot \int \frac{x^{2i-1}dx}{\sqrt[n]{(1-x^n)^{n-i}}}\cdots \int \frac{x^{(n-1)i-1}dx}{\sqrt[n]{(1-x^n)^{n-i}}}}.
\]
Hence, if, for the sake of brevity, we denote the integral formula $\int \frac{x^{p-1}dx}{\sqrt[n]{(1-x^n)^{n-q}}}$ by this character $\left(\frac{p}{q}\right)$, but on the other hand the formula $\int dx \sqrt[n]{\left(\log \frac{1}{x}\right)^m}$ by this character $\left[\frac{m}{n}\right]$ such that $\left[\frac{m}{n}\right]$ denotes the value of this indefinite product $1 \cdot 2 \cdot 3 \cdots z$, while $z=\frac{m}{n}$, the  value in question will be expressed more succinctly as follows

\[
\left[\frac{i-n}{n}\right]= \sqrt[n]{1 \cdot 2 \cdot 3 \cdots (i-1)n^{n-1}\left(\frac{i}{i}\right)\left(\frac{2i}{i}\right)\left(\frac{3i}{i}\right)\cdots \left(\frac{ni-i}{i}\right)},
\]
whence it is also concluded that

\[
\left[\frac{i}{n}\right]= \frac{i}{n}\sqrt[n]{1 \cdot 2 \cdot 3 \cdots (i-1)n^{n-1}\left(\frac{i}{i}\right)\left(\frac{2i}{i}\right)\left(\frac{3i}{i}\right)\cdots \left(\frac{ni-i}{i}\right)}.
\]
Here, it will always suffice to  take the number $i$ smaller than $n$, because it is known for larger numbers that 

\[
\left[\frac{i+n}{n}\right]=\frac{i+n}{n}\left[\frac{i}{n}\right], \quad \text{in the same way} \quad \left[\frac{i+2n}{n} \right]=\frac{i+n}{n}\cdot \frac{i+2n}{n}\left[\frac{i}{n}\right] \quad \text{etc.},
\]
and so the whole investigation is hence reduced to those cases in which the numerator $i$ of the fraction $\frac{i}{n}$ is smaller than the denominator $n$. In addition, it will be helpful to have noted the following properties of the integral formulas

\[
\int \frac{x^{p-1}dx}{\sqrt[n]{(1-x^n)^{n-q}}}=\left(\frac{p}{q}\right):
\]

I. The letters $p$ and $q$ are interchangeable so that

\[
\left(\frac{p}{q}\right)=\left(\frac{q}{p}\right).
\]

II. If one of the two numbers $p$ or $q$ is equal to the exponent $n$, the value of the integral formula will be algebraic, namely

\[
\left(\frac{n}{p}\right)=\left(\frac{p}{n}\right)=\frac{1}{p}\quad \text{or} \quad \left(\frac{n}{q}\right)=\left(\frac{q}{n}\right)=\frac{1}{q}.
\]

III. If the sum of the numbers $p+q$ is equal to the exponent $n$, the value of the integral formula $\left(\frac{p}{q}\right)$ can be exhibited by means of the quadrature of the circle, because 

\[
\left(\frac{p}{n-p}\right)=\left(\frac{n-p}{p}\right)=\frac{\pi}{n \sin \frac{p \pi}{n}} \quad \text{and} \quad \left(\frac{q}{n-q}\right)=\left(\frac{n-q}{q}\right)=\frac{\pi}{n \sin \frac{q \pi}{n}}.
\]

IV. If one of the numbers $p$ or $q$ is greater than the exponent $n$, the integral formula $\left(\frac{p}{q}\right)$ can be reduced to another one whose terms are smaller than $n$; this is achieved using this reduction

\[
\left(\frac{p+n}{q}\right)=\frac{p}{p+q}\left(\frac{p}{q}\right).
\]

V. There is a relation among many of these integral formulas of such a kind that

\[
\left(\frac{p}{q}\right)\left(\frac{p+q}{r}\right)=\left(\frac{p}{r}\right)\left(\frac{p+r}{q}\right)=\left(\frac{q}{r}\right)\left(\frac{q+r}{p}\right);
\]
by means of this relation all reductions I gave in my observations on these formulas\footnote{Euler is again referring to E321.} are found.

\subsection*{Corollary 1}

\paragraph*{§45}

If  we accommodate the formula that we found to each case in this way, by means of reduction IV, we will be able to exhibit them in the most simple way as follows. And for the case $n=2$, in which no further reduction is necessary, we will have

\[
\left[\frac{1}{2}\right]=\frac{1}{2}\sqrt[2]{2\left(\frac{1}{1}\right)}=\frac{1}{2}\sqrt[2]{\frac{\pi}{\sin \frac{\pi}{2}}}=\frac{1}{2}\sqrt{\pi}.
\]

\subsection*{Corollary 2}

\paragraph*{§46}

For the case $n=3$ we will have these reductions

\begin{alignat*}{18}
&\left[\frac{1}{3}\right]&&=\frac{1}{3}\sqrt[3]{3^2 \left(\frac{1}{1}\right)\left(\frac{2}{1}\right)}\\
&\left[\frac{2}{3}\right]&&=\frac{2}{3}\sqrt[3]{3 \cdot 1 \left(\frac{2}{2}\right)\left(\frac{1}{2}\right)}.
\end{alignat*}

\subsection*{Corollary 3}

\paragraph*{§47}

For the case $n=4$ one obtains these three reductions

\begin{alignat*}{18}
&\left[\frac{1}{4}\right]&&=\frac{1}{4}\sqrt[4]{4^3 \left(\frac{1}{1}\right)\left(\frac{2}{1}\right)\left(\frac{3}{1}\right)},\\
& \left[\frac{2}{4}\right]&&=\frac{2}{4}\sqrt[4]{4^2 \cdot 2 \left(\frac{2}{2}\right)^2\left(\frac{4}{2}\right)}=\frac{1}{2}\sqrt[2]{4\left(\frac{2}{2}\right)}
\intertext{because of $\left(\frac{4}{2}\right)=\frac{1}{2},$}
&\left[\frac{3}{4}\right]&&=\frac{3}{4}\sqrt{4 \cdot 1 \cdot 2 \left(\frac{3}{3}\right)\left(\frac{2}{3}\right)\left(\frac{1}{3}\right)};
\end{alignat*}
because in the second equation  $\left(\frac{2}{2}\right)=\left(\frac{4-2}{2}\right)=\frac{\pi}{4}$, it will, of course as before, be

\[
\left[\frac{2}{4}\right]=\left[\frac{1}{2}\right]= \frac{1}{2}\sqrt{\pi}.
\]

\subsection*{Corollary 4}

\paragraph*{§48}

Now let $n=5$ and these four reductions result

\begin{alignat*}{18}
&\left[\frac{1}{5}\right]&&=\frac{1}{5}\sqrt[5]{5^4\left(\frac{1}{1}\right)\left(\frac{2}{1}\right)\left(\frac{3}{1}\right)\left(\frac{4}{1}\right)},\\
&\left[\frac{2}{5}\right]&&=\frac{2}{5}\sqrt[5]{5^3 \cdot 1\left(\frac{2}{2}\right)\left(\frac{4}{2}\right)\left(\frac{1}{2}\right)\left(\frac{3}{2}\right)},\\
&\left[\frac{3}{5}\right]&&=\frac{3}{5}\sqrt[5]{5^2 \cdot 1 \cdot 2\left(\frac{3}{3}\right)\left(\frac{1}{3}\right)\left(\frac{4}{3}\right)\left(\frac{2}{3}\right)},\\
&\left[\frac{4}{5}\right]&&=\frac{4}{5}\sqrt[5]{5  \cdot 1 \cdot 2 \cdot 3\left(\frac{4}{4}\right)\left(\frac{3}{4}\right)\left(\frac{2}{4}\right)\left(\frac{1}{4}\right)}.
\end{alignat*}

\subsection*{Corollary 5}

\paragraph*{§49}

Let $n=6$ and we will have these reductions

\begin{alignat*}{18}
&\left[\frac{1}{6}\right]&&=\frac{1}{6}\sqrt[6]{6^5\left(\frac{1}{1}\right)\left(\frac{2}{1}\right)\left(\frac{3}{1}\right)\left(\frac{4}{1}\right)\left(\frac{5}{1}\right)},\\
&\left[\frac{2}{6}\right]&&=\frac{2}{6}\sqrt[6]{6^4 \cdot 2 \left(\frac{2}{2}\right)^2\left(\frac{4}{2}\right)^2\left(\frac{6}{2}\right)}=\frac{1}{3}\sqrt[3]{6^2\left(\frac{3}{2}\right)\left(\frac{4}{2}\right)},\\
&\left[\frac{3}{6}\right]&&=\frac{3}{6}\sqrt[6]{6^3 \cdot 3 \cdot 3 \left(\frac{3}{3}\right)^3\left(\frac{6}{3}\right)^2}=\frac{1}{2}\sqrt[2]{6\left(\frac{3}{3}\right)},\\
&\left[\frac{4}{6} \right]&&=\frac{4}{6}\sqrt[8]{6^2 \cdot 2 \cdot 4 \cdot 2 \left(\frac{4}{4}\right)^2\left(\frac{2}{4}\right)^2\left(\frac{6}{4}\right)}=\frac{2}{3}\sqrt[3]{6 \cdot 2 \left(\frac{4}{4}\right)\left(\frac{2}{4}\right)},\\
&\left[\frac{5}{6}\right]&&=\frac{5}{6}\sqrt[6]{6 \cdot 1 \cdot 2 \cdot 3 \cdot 4 \left(\frac{5}{5}\right)\left(\frac{4}{5}\right)\left(\frac{3}{5}\right)\left(\frac{2}{5}\right)\left(\frac{1}{5}\right)}.
\end{alignat*}

\subsection*{Corollary 6}

\paragraph*{§50}

For $n=7$ the following six equations result

\begin{alignat*}{18}
&\left[\frac{1}{7}\right]&&=\frac{1}{7}\sqrt[7]{7^6 \left(\frac{1}{1}\right)\left(\frac{2}{1}\right)\left(\frac{3}{1}\right)\left(\frac{4}{1}\right)\left(\frac{5}{1}\right)\left(\frac{6}{1}\right)},\\
&\left[\frac{2}{7}\right]&&=\frac{2}{7}\sqrt[7]{7^5 \cdot 1 \left(\frac{2}{2}\right)\left(\frac{4}{2}\right)\left(\frac{6}{2}\right)\left(\frac{1}{2}\right)\left(\frac{3}{2}\right)\left(\frac{5}{2}\right)},\\
&\left[\frac{3}{7}\right]&&=\frac{3}{7}\sqrt[7]{7^4 \cdot 1 \cdot 2 \left(\frac{3}{3}\right)\left(\frac{6}{3}\right)\left(\frac{2}{3}\right)\left(\frac{5}{3}\right)\left(\frac{1}{3}\right)\left(\frac{4}{3}\right)},\\
&\left[\frac{4}{7}\right]&&=\frac{4}{7}\sqrt[7]{7^3 \cdot 1 \cdot 2 \cdot 3 \left(\frac{4}{4}\right)\left(\frac{1}{4}\right)\left(\frac{5}{4}\right)\left(\frac{2}{4}\right)\left(\frac{6}{4}\right)\left(\frac{3}{4}\right)},\\
&\left[\frac{5}{7}\right]&&=\frac{5}{7}\sqrt[7]{7^2 \cdot 1 \cdot 2 \cdot 3 \cdot 4 \left(\frac{5}{5}\right)\left(\frac{3}{5}\right)\left(\frac{1}{5}\right)\left(\frac{6}{5}\right)\left(\frac{4}{5}\right)\left(\frac{2}{5}\right)},\\
&\left[\frac{6}{7}\right]&&=\frac{6}{7}\sqrt[7]{7 \cdot 1 \cdot 2 \cdot 3 \cdot 4 \cdot 5\left(\frac{6}{6}\right)\left(\frac{5}{6}\right)\left(\frac{4}{6}\right)\left(\frac{3}{6}\right)\left(\frac{2}{6}\right)\left(\frac{1}{6}\right)}.
\end{alignat*}

\subsection*{Corollary 7}

\paragraph*{§51}

Now let $n=8$ and one will get to these seven reductions

\begin{alignat*}{18}
&\left[\frac{1}{8}\right]&&=\frac{1}{8}\sqrt[8]{8^7 \left(\frac{1}{1}\right)\left(\frac{2}{1}\right)\left(\frac{3}{1}\right)\left(\frac{4}{1}\right)\left(\frac{5}{1}\right)\left(\frac{6}{1}\right)\left(\frac{7}{1}\right)},\\
&\left[\frac{2}{8}\right]&&=\frac{2}{8}\sqrt[8]{8^6 \cdot 2 \left(\frac{2}{2}\right)^2\left(\frac{4}{2}\right)^2 \left(\frac{6}{2}\right)^2\left(\frac{8}{2}\right)}=\frac{1}{4}\sqrt[4]{8^3\left(\frac{2}{2}\right)\left(\frac{4}{2}\right)\left(\frac{6}{2}\right)},\\
&\left[\frac{3}{8}\right]&&=\frac{3}{8}\sqrt[8]{8^5 \cdot 1 \cdot 2 \left(\frac{3}{3}\right)\left(\frac{6}{3}\right)\left(\frac{1}{3}\right)\left(\frac{4}{3}\right)\left(\frac{7}{3}\right)\left(\frac{2}{3}\right)\left(\frac{5}{3}\right)},\\
&\left[\frac{4}{8}\right]&&=\frac{4}{8}\sqrt[8]{8^4 \cdot 4 \cdot 4 \cdot 4 \left(\frac{4}{4}\right)^4\left(\frac{8}{4}\right)^3}=\frac{1}{2}\sqrt[2]{8\left(\frac{4}{4}\right)},\\
&\left[\frac{5}{8}\right]&&=\frac{5}{8}\sqrt[8]{8^3 \cdot 1 \cdot 2 \cdot 3 \cdot 4 \left(\frac{5}{5}\right)\left(\frac{2}{5}\right)\left(\frac{7}{5}\right)\left(\frac{4}{5}\right)\left(\frac{1}{5}\right)\left(\frac{6}{5}\right)\left(\frac{3}{5}\right)},\\
&\left[\frac{6}{8}\right]&&=\frac{6}{8}\sqrt[8]{8^2 \cdot 4 \cdot 2 \cdot 6 \cdot 4 \cdot 2\left(\frac{6}{6}\right)^2\left(\frac{4}{6}\right)^2\left(\frac{2}{6}\right)^2 \left(\frac{8}{6}\right)}=\frac{3}{4}\sqrt[4]{8 \cdot 2 \cdot 4 \left(\frac{6}{6}\right)\left(\frac{4}{6}\right)\left(\frac{2}{6}\right)},\\
& \left[\frac{7}{8}\right]&&= \frac{7}{8}\sqrt[8]{8 \cdot 1 \cdot 2 \cdot 3 \cdot 4 \cdot 5 \cdot 6\left(\frac{7}{7}\right)\left(\frac{6}{7}\right)\left(\frac{5}{7}\right)\left(\frac{4}{7}\right)\left(\frac{3}{7}\right)\left(\frac{2}{7}\right)\left(\frac{1}{7}\right)}.
\end{alignat*}

\subsection*{Scholium}

\paragraph*{§52}

It would be superfluous to expand these cases any further, because  the structure of these formulas is already seen very clearly from the ones we gave. If  the numbers $m$ and $n$ are coprime in the propounded formula $\left[\frac{m}{n}\right]$, the rule is obvious, because

\[
\left[\frac{m}{n}\right]=\frac{m}{n}\sqrt[n]{n^{n-m}\cdot 1 \cdot 2 \cdots (m-1)\left(\frac{1}{m}\right)\left(\frac{2}{m}\right)\left(\frac{3}{m}\right)\cdots \left(\frac{n-1}{m}\right)};
\]
but if these numbers $m$ and $n$ have a common divisor, it will indeed be useful to reduce this fraction $\frac{m}{n}$ to the smallest form and extract the value in question from the preceding cases; nevertheless, the operation can also be done as follows. Because the  expression in question certainly has this form

\[
\left[\frac{m}{n}\right]=\frac{m}{n}\sqrt[n]{n^{n-m}PQ},
\]
where $Q$ is the product of the $n-1$ integral formulas, $P$ on the other hand the product of some absolute numbers, in order to find that product $Q$ just continue this series of formulas $\left(\frac{m}{m}\right)\left(\frac{2m}{m}\right)\left(\frac{3m}{m}\right)$etc. until the numerator exceeds the exponent $n$, and instead of this numerator write its excess over $n$; if this excess is set $=\alpha$ such that our formula is $\left(\frac{\alpha}{m}\right)$, this numerator $\alpha$ will give a factor of a product $P$; then continue this series of formulas $\left(\frac{\alpha }{m}\right)\left(\frac{\alpha +m}{m}\right)\left(\frac{\alpha +2m}{m}\right)$etc. until one again gets to a numerator greater than the exponent $n$, and the formula $\left(\frac{n+\beta}{m}\right)$ emerges; instead of this formula one then has to write $\left(\frac{\beta}{m}\right)$, and hence the factor $\beta $ is introduced into the product and  one has to continue like this until $n-1$ formulas for $Q$ will have emerged.\\[2mm]
To understand these operations more easily, let us expand the case of the formula

\[
\left[\frac{9}{12}\right]=\frac{9}{12}\sqrt[12]{12^3PQ}
\]
in this way; the letters $P$ and $Q$ are found as follows:

\begin{alignat*}{38}
&\text{for} \quad &&Q \dots &&\left(\frac{9}{9}\right)&&\left(\frac{6}{9}\right)&&\left(\frac{3}{9}\right)&&\left(\frac{12}{9}\right)&&\left(\frac{9}{9}\right)&&\left(\frac{6}{9}\right)&&\left(\frac{3}{9}\right)&&\left(\frac{12}{9}\right)&&\left(\frac{9}{9}\right)&&\left(\frac{6}{9}\right)&&\left(\frac{3}{9}\right),\\
&\text{for} \quad &&P \dots && &&~~~~6  &&\cdot ~~3&& &&~~~~9  &&\cdot ~~6 &&\cdot ~~3&& && ~~~~9 && \cdot ~~6  && \cdot ~~3
\end{alignat*}
and so one finds

\[
Q=\left(\frac{9}{9}\right)^3\left(\frac{6}{9}\right)^3\left(\frac{3}{9}\right)^3\left(\frac{12}{9}\right)^2 \quad \text{and} \quad P=6^3 \cdot 3^3 \cdot 9^2.
\]
Because $\left(\frac{12}{9}\right)=\frac{1}{9}$,  $PQ=6^3\cdot 3^3 \left(\frac{9}{9}\right)^3 \left(\frac{6}{9}\right)^3 \left(\frac{3}{9}\right)^3$ and hence

\[
\left[\frac{9}{12}\right]=\frac{3}{4}\sqrt[4]{12 \cdot 6 \cdot 3 \left(\frac{9}{9}\right)\left(\frac{6}{9}\right)\left(\frac{3}{9}\right)}.
\]

\section*{Theorem}

\paragraph*{§53}

\textit{Whatever positive numbers are indicated by the letters $m$ and $n$, in the notation  introduced and explained before it will always be}

\[
\left[\frac{m}{n}\right]=\frac{m}{n}\sqrt[n]{n^{n-m} \cdot 1 \cdot 2 \cdot 3 \cdots (m-1)\left(\frac{1}{m}\right)\left(\frac{2}{m}\right)\left(\frac{3}{m}\right)\cdots \left(\frac{n-1}{m}\right)}.
\]

\subsection*{Proof}

For the cases in which $m$ and $n$ are coprime numbers, the validity of this theorem was shown in the preceding theorems; but the fact that it also holds, if those numbers $m$ and $n$ have a common divisor, is not evident from that theorem; but since the formula was already proved to be true in the cases in which $m$ and $n$ are mutually prime, it is natural to conclude that this theorem is true in general. I am  completely aware that this kind to deduce something is completely unusual and must  seem suspect to most people. In order to clear those doubts, because for the cases, in which the numbers $m$ and $n$ are composite, we obtained two expressions, it will be useful to have shown the agreement for the cases explained before. And the case $m=n$ is already a huge confirmation, in which case our formula obviously becomes $=1$.

\subsection*{Corollary 1}

\paragraph*{§54}

The first case requiring a demonstration of the agreement is that one, in which  $m=2$ and $n=4$, for which we found above (§47)

\[
\left[\frac{2}{4}\right]=\frac{2}{4}\sqrt[4]{4^2\left(\frac{2}{2}\right)^2};
\]
but now via the theorem 

\[
\left[\frac{2}{4}\right]=\frac{2}{4}\sqrt[4]{4^2 \cdot 1 \left(\frac{1}{2}\right)\left(\frac{2}{2}\right)\left(\frac{3}{2}\right)},
\]
where by comparison

\[
\left(\frac{2}{2}\right)=\left(\frac{1}{2}\right)\left(\frac{3}{2}\right),
\]
whose validity was confirmed in my observations mentioned\footnote{Euler again refers to his paper "'Observationes circa integralia formularum $\int x^{p-1}dx(1-x^n)^{\frac{q}{n}-1}$ posito post integrationem $x = 1$"'. This is paper E321 in the Eneström-Index.} above.

\subsection*{Corollary 2}

\paragraph*{§55}

If  $m=2$ and $n=6$, using the results derived above (§49) we have

\[
\left[\frac{2}{6}\right]=\frac{2}{6}\sqrt[6]{6^4 \left(\frac{2}{2}\right)^2 \left(\frac{4}{2}\right)^2};
\]
now on the other hand by means of the theorem

\[
\left[\frac{2}{6}\right]=\frac{2}{6}\sqrt[6]{6^4 \cdot 1 \left(\frac{1}{2}\right)\left(\frac{2}{2}\right)\left(\frac{3}{2}\right)\left(\frac{4}{2}\right)\left(\frac{5}{2}\right)}
\]
and therefore it has to be

\[
\left(\frac{2}{2}\right)\left(\frac{4}{2}\right)=\left(\frac{1}{2}\right)\left(\frac{3}{2}\right)\left(\frac{5}{2}\right),
\]
whose validity is clear for the same reasons.

\subsection*{Corollary 3}

\paragraph*{§56}

If $m=3$ and $n=6$, one arrives at this equation

\[
\left(\frac{3}{3}\right)^2=1 \cdot 2 \left(\frac{1}{3}\right)\left(\frac{2}{3}\right)\left(\frac{4}{3}\right)\left(\frac{5}{3}\right);
\]
but if $m=4$ and $n=6$,  in like manner,

\[
2^2\left(\frac{4}{4}\right)\left(\frac{2}{4}\right)=1 \cdot 2 \cdot 3 \left(\frac{1}{4}\right)\left(\frac{3}{4}\right)\left(\frac{5}{4}\right)
\]
or

\[
\left(\frac{4}{4}\right)\left(\frac{2}{4}\right)=\frac{3}{2}\left(\frac{1}{4}\right)\left(\frac{3}{4}\right)\left(\frac{5}{4}\right),
\]
which is also found to be true.

\subsection*{Corollary 4}

\paragraph*{§57}

The case $m=2$ and $n=8$ yields this equality

\[
\left(\frac{2}{2}\right)\left(\frac{4}{2}\right)\left(\frac{6}{2}\right)=\left(\frac{1}{2}\right)\left(\frac{3}{2}\right)\left(\frac{5}{2}\right)\left(\frac{7}{2}\right),
\]
but the case $m=4$ and $n=8$ this one

\[
\left(\frac{4}{4}\right)^3=1 \cdot 2 \cdot 3 \left(\frac{1}{4}\right)\left(\frac{2}{4}\right)\left(\frac{3}{4}\right)\left(\frac{5}{4}\right)\left(\frac{6}{4}\right)\left(\frac{7}{4}\right)
\]
and finally the case $m=6$ and $n=8$ gives this equation

\[
2 \cdot 4 \left(\frac{6}{6}\right)\left(\frac{4}{6}\right)\left(\frac{2}{6}\right)=1 \cdot 3 \cdot 5 \left(\frac{1}{6}\right)\left(\frac{3}{6}\right)\left(\frac{5}{6}\right)\left(\frac{7}{6}\right),
\]
which is also true.

\subsection*{Scholium}

\paragraph*{§58}

But if in general the numbers $m$ and $n$ have the common factor $2$ and the propounded formula is $\left[\frac{2m}{2n}\right]=\left[\frac{m}{n}\right]$, because

\[
\left[\frac{m}{n}\right]=\frac{m}{n}\sqrt[n]{n^{n-m} \cdot 1 \cdot 2 \cdot 3 \cdots (m-1)\left(\frac{1}{m}\right)\left(\frac{2}{m}\right)\left(\frac{3}{m}\right)\cdots \left(\frac{n-1}{m}\right)},
\]
after having reduced the same to the exponent $2n$ it will be

\[
\frac{m}{n}\sqrt[2n]{2n^{2n-2m} \cdot 2^2 \cdot 4^2 \cdot 6^2 \cdots (2m-2)^2 \left(\frac{2}{2m}\right)^2\left(\frac{4}{2m}\right)^2\left(\frac{6}{2m}\right)^2 \cdots \left(\frac{2n-2}{2m}\right)^2}.
\]
By the theorem, the same expression on the other hand becomes

\[
\frac{m}{n}\sqrt[2n]{2n^{2n-2m} \cdot 1 \cdot 2 \cdot 3 \cdots (2m-1) \left(\frac{1}{2m}\right)\left(\frac{2}{2m}\right)\left(\frac{3}{2m}\right) \cdots \left(\frac{2n-1}{2m}\right)},
\]
whence for the exponent $2n$ it will be

\[
2 \cdot 4 \cdot 6 \cdots (2m-2)\left(\frac{2}{2m}\right)\left(\frac{4}{2m}\right)\left(\frac{6}{2m}\right)\cdots \left(\frac{2n-2}{2m}\right)
\]
\[
=1 \cdot 3 \cdot 5 \cdots (2m-1)\left(\frac{1}{2m}\right)\left(\frac{3}{2m}\right)\left(\frac{5}{2m}\right)\cdots \left(\frac{2n-1}{2m}\right).
\]
If in the same way the common divisor is $3$,  for the exponent $3n$ one will find

\[
3^2 \cdot 6^2 \cdot 9^2 \cdots (3m-3)^2 \left(\frac{3}{3m}\right)^2\left(\frac{6}{3m}\right)^2\left(\frac{9}{3m}\right)^2\cdots \left(\frac{3n-3}{3m}\right)^2
\]
\[
=1 \cdot 2 \cdot 4 \cdot 5 \cdots (3m-2)(3m-1) \left(\frac{1}{3m}\right)\left(\frac{2}{3m}\right)\left(\frac{4}{3m}\right)\left(\frac{5}{3m}\right)\cdots \left(\frac{3n-1}{3m}\right),
\]
which equation can be more conveniently exhibited as follows

\[
\frac{1 \cdot 2 \cdot 4 \cdot 5 \cdot 7 \cdot 8 \cdot 10 \cdots (3m-2)(3m-1)}{3^2 \cdot 6^2 \cdot 9^2 \cdots (3m-3)^2}=\frac{\left(\frac{3}{3m}\right)^2\left(\frac{6}{3m}\right)^2\cdots \left(\frac{3n-3}{3m}\right)^2}{\left(\frac{1}{3m}\right)\left(\frac{2}{3m}\right)\left(\frac{4}{3m}\right)\left(\frac{5}{3m}\right)\left(\frac{7}{3m}\right)\cdots \left(\frac{3n-2}{3m}\right)\left(\frac{3n-1}{3m}\right)}.
\]
But if in general the common divisor is $d$ and the exponent $dn$, one will have

\[
\left(d \cdot 2d \cdot 3d \cdots (dm-d)\left(\frac{d}{dm}\right)\left(\frac{2d}{dm}\right)\left(\frac{3d}{dm}\right)\cdots \left(\frac{dn-d}{dm}\right)\right)^d
\]
\[
=1 \cdot 2 \cdot 3 \cdot 4 \cdots (dm-1)\left(\frac{1}{dm}\right)\left(\frac{2}{dm}\right)\left(\frac{3}{dm}\right)\cdots \left(\frac{dn-1}{dm}\right),
\]
which equation can easily be accommodated to any cases, whence the following theorem deserves  to be noted.

\section*{Theorem}

\paragraph*{§59}

\textit{If $\alpha$ is a common divisor of the numbers $m$ and $n$ and the formula $\left(\frac{p}{q}\right)$ denotes the value of the integral $\int \frac{x^{p-1}dx}{\sqrt[n]{(1-x^n)^{n-q}}}$ extended from $x=0$ to $x=1$, it will be}

\[
\left(\alpha \cdot 2 \alpha \cdot 3 \alpha \cdots (m-\alpha)\left(\frac{\alpha}{m}\right)\left(\frac{2\alpha }{m}\right)\left(\frac{3 \alpha}{m}\right)\cdots\left(\frac{n-\alpha}{m}\right)\right)^{\alpha}
\]
\[
=1 \cdot 2 \cdot 3 \cdots (m-1)\left(\frac{1}{m}\right)\left(\frac{2}{m}\right)\left(\frac{3}{m}\right)\cdots \left(\frac{n-1}{m}\right).
\]

\subsection*{Proof}

The validity of this theorem is already seen from the preceding Scholium; while the common divisor was $=d$ and the two propounded numbers were $dm$ and $dn$ there, here I just wrote $m$ and $n$ instead of them, but instead of their divisor $d$ I wrote the letter $\alpha$, the kind of divisor which the stated equality contains in such a way that one assumes the numbers $m$ and $n$ and hence also $m-\alpha$ and $n-\alpha$ to occur in the continued arithmetic progression $\alpha$, $2\alpha$, $3 \alpha$ etc. In addition, I am forced to confess that this proof is, of course, mainly based on induction and cannot be considered to be rigorous by any means; but because we are nevertheless convinced of its truth, this theorem seems to be worth of one's greater attention; nevertheless, there is no doubt that a further expansion of integral formulas of this kind will finally lead to a complete proof; but it is an extraordinary specimen of analytical investigation that it was possible for us to see its truth before we had the complete proof.

\subsection*{Corollary 1}

\paragraph*{§60}

Thus, if we  substitute the integrals  for the signs we introduced, our theorem will be as follows

\[
\alpha \cdot 2 \alpha \cdot 3 \alpha \cdots (m-\alpha) \int \frac{x^{\alpha -1}dx}{\sqrt[n]{(1-x^n)^{n-m}}} \cdot \int \frac{x^{2\alpha -1}dx}{\sqrt[n]{(1-x^n)^{n-m}}} \cdots \int \frac{x^{n-\alpha -1}dx}{\sqrt[n]{(1-x^n)^{n-m}}}
\]
\[
=\sqrt[\alpha]{1 \cdot 2 \cdot 3 \cdots (m-1)\int \frac{dx}{\sqrt[n]{(1-x^n)^{n-m}}}\cdot \int \frac{x^{}dx}{\sqrt[n]{(1-x^n)^{n-m}}}\cdots \int \frac{x^{n -2}dx}{\sqrt[n]{(1-x^n)^{n-m}}}}.
\]

\subsection*{Corollary 2}

\paragraph*{§61}

Or if, for the sake of brevity, we set $\sqrt[n]{(1-x^n)^{n-m}}=X$, it will be

\[
\alpha \cdot 2 \alpha \cdot 3 \alpha \cdots (m-\alpha) \int \frac{x^{\alpha -1}dx}{X} \cdot \int \frac{x^{2\alpha -1}dx}{X} \cdots \int \frac{x^{n-\alpha -1}dx}{X}
\]
\[
=\sqrt[\alpha]{1 \cdot 2 \cdot 3 \cdots (m-1)\int \frac{dx}{X} \cdot \int \frac{xdx}{X} \cdot \int \frac{x^{2}dx}{X} \cdots \int \frac{x^{n-2}dx}{X}}.
\]

\section*{General Theorem}

\paragraph*{§62}

\textit{If the divisors of the two numbers $m$ and $n$ are $\alpha$, $\beta$, $\gamma$ etc. and the formula $\left(\frac{p}{q}\right)$ denotes the value of the integral $\int \frac{x^{p-1}dx}{\sqrt[n]{(1-x^n)^{n-q}}}$ extended from $x=0$ to $x=1$, the following expressions consisting of integral formulas of this kind will be equal to each other}

\begin{alignat*}{18}
& &&\bigg(\alpha &&\cdot 2 \alpha && \cdot 3 \alpha &&\cdots (m&&-\alpha)&&\left(\frac{\alpha}{m}\right)&&\left(\frac{2\alpha}{m}\right)&&\left(\frac{3\alpha}{m}\right)&&\cdots \left(\frac{n-\alpha}{m}\right)\bigg)^{\alpha}\\
&=&&\bigg(\beta &&\cdot 2 \beta && \cdot 3 \beta &&\cdots (m&&-\beta)&&\left(\frac{\beta}{m}\right)&&\left(\frac{2\beta}{m}\right)&&\left(\frac{3\beta}{m}\right)&&\cdots \left(\frac{n-\beta}{m}\right)\bigg)^{\beta}\\
&=&&\bigg(\gamma &&\cdot 2 \gamma && \cdot 3 \gamma &&\cdots (m&&-\gamma)&&\left(\frac{\gamma}{m}\right)&&\left(\frac{2\gamma}{m}\right)&&\left(\frac{3\gamma}{m}\right)&&\cdots \left(\frac{n-\gamma}{m}\right)\bigg)^{\gamma}\\
& && && && && && && && \text{etc.}
\end{alignat*}

\subsection*{Proof}

The validity of this theorem obviously follows from the preceding theorem, because each of these expressions is equal to this one

\[
1 \cdot 2 \cdot 3 \cdots (m-1) \left(\frac{1}{m}\right)\left(\frac{2}{m}\right)\left(\frac{3}{m}\right)\cdots \left(\frac{n-1}{m}\right),
\]
which corresponds to the unity as smallest common divisor of the numbers $m$ and $n$. Therefore, so many expressions of this kind, i.e. all equal to each other, can be exhibited as there were common divisors of the two numbers $m$ and $n$.

\subsection*{Corollary 1}

\paragraph*{§63}

Because this formula $\left(\frac{n}{m}\right)$ is $=\frac{1}{m}$ and hence $m\left(\frac{n}{m}\right)=1$, our equal  expressions can be represented  more succinctly as follows

\begin{alignat*}{18}
& &&\bigg(\alpha &&\cdot 2 \alpha && \cdot 3 \alpha &&\cdots m&&\left(\frac{\alpha}{m}\right)&&\left(\frac{2\alpha}{m}\right)&&\left(\frac{3\alpha}{m}\right)&&\cdots \left(\frac{n}{m}\right)\bigg)^{\alpha}\\
&=&&\bigg(\beta &&\cdot 2 \beta && \cdot 3 \beta &&\cdots m&&\left(\frac{\beta}{m}\right)&&\left(\frac{2\beta}{m}\right)&&\left(\frac{3\beta}{m}\right)&&\cdots \left(\frac{n}{m}\right)\bigg)^{\beta}\\
&=&&\bigg(\gamma &&\cdot 2 \gamma && \cdot 3 \gamma &&\cdots m&&\left(\frac{\gamma}{m}\right)&&\left(\frac{2\gamma}{m}\right)&&\left(\frac{3\gamma}{m}\right)&&\cdots \left(\frac{n}{m}\right)\bigg)^{\gamma}.
\end{alignat*}
For, even if the number of factors was increased here,  the structure of these formulas is nevertheless easily seen.

\subsection*{Corollary 2}

\paragraph*{§64}

Thus, if $m=6$ and $n=12$, because of the common divisors of these numbers, $6$, $3$, $2$, $1$, one will have the following four forms that are all equal to each other

\[
=\left(6 \left(\frac{6}{6}\right)\left(\frac{12}{6}\right)\right)^6 =\left(3 \cdot 6 \left(\frac{3}{6}\right)\left(\frac{6}{6}\right)\left(\frac{9}{6}\right)\left(\frac{12}{6}\right)\right)^3
\]
\[
=\left(2 \cdot 4 \cdot 6 \left(\frac{2}{6}\right)\left(\frac{4}{6}\right)\left(\frac{6}{6}\right)\left(\frac{8}{6}\right)\left(\frac{10}{6}\right)\left(\frac{12}{6}\right)\right)^2
\]
\[
=1 \cdot 2 \cdot 3 \cdot 4 \cdot 5 \cdot 6 \left(\frac{1}{6}\right)\left(\frac{2}{6}\right)\left(\frac{3}{6}\right)\cdots \left(\frac{12}{6}\right).
\]

\subsection*{Corollary 3}

\paragraph*{§65}

If the last formula is combined with the penultimate, this equation will result

\[
\frac{1 \cdot 3 \cdot 5}{2 \cdot 4 \cdot 6}= \frac{\left(\frac{2}{6}\right)\left(\frac{4}{6}\right)\left(\frac{6}{6}\right)\left(\frac{8}{6}\right)\left(\frac{10}{6}\right)\left(\frac{12}{6}\right)}{\left(\frac{1}{6}\right)\left(\frac{3}{6}\right)\left(\frac{5}{6}\right)\left(\frac{7}{6}\right)\left(\frac{9}{6}\right)\left(\frac{11}{6}\right)},
\]
but the last compared to the second yields

\[
\frac{1 \cdot 2 \cdot 4 \cdot 5}{3 \cdot 3 \cdot 6 \cdot 6}=\frac{\left(\frac{3}{6}\right)\left(\frac{3}{6}\right)\left(\frac{6}{6}\right)\left(\frac{6}{6}\right)\left(\frac{9}{6}\right)\left(\frac{9}{6}\right)\left(\frac{12}{6}\right)\left(\frac{12}{6}\right)}{\left(\frac{1}{6}\right)\left(\frac{2}{6}\right)\left(\frac{4}{6}\right)\left(\frac{5}{6}\right)\left(\frac{7}{6}\right)\left(\frac{8}{6}\right)\left(\frac{10}{6}\right)\left(\frac{11}{6}\right)}.
\]

\subsection*{Scholium}

\paragraph*{§66}

Hence infinitely many relations among the integral formulas of the form 

\[
\int \frac{x^{p-1}dx}{\sqrt[n]{(1-x^n)^{n-q}}}=\left(\frac{p}{q}\right)
\]
follow, which are even more remarkable, because we were led to them by a completely singular method. And if anyone does not believe them to be true, he or she should consult my observations on these integral formulas\footnote{Euler again refers to his paper "Observationes circa integralia formularum $\int x^{p-1}dx(1-x^n)^{\frac{q}{n}-1}$ posito post integrationem $x = 1$". This is paper E321 in the Eneström-Index.} and will then  easily be convinced of their truth for any case. But even if this consideration provides some confirmation, the relations found here are nevertheless of even greater importance, because  a certain structure is noticed in them and they are easily generalised to all classes, whatever number was assumed for the exponent $n$, whereas in the first treatment the calculation for the higher classes becomes continuously more cumbersome and intricate. \newpage

\section*{Supplement containing the Proof of the theorem propounded in §53}

It is convenient to derive this proof from the results mentioned above; just take the equation given in §25, which for $f=1$, having changed the letters, reads as follows

\[
\frac{\int dx \left(\log \frac{1}{x}\right)^{\nu-1} \cdot \int dx \left(\log\frac{1}{x}\right)^{\mu-1}}{\int dx \left (\log \frac{1}{x}\right)^{\nu+\mu-1}}=\varkappa \int \frac{x^{\varkappa \mu -1}dx}{(1-x^{\varkappa})^{1-\nu}},
\]
and, using known reductions, represent it in this form

\[
\frac{\int dx \left(\log \frac{1}{x}\right)^{\nu} \cdot \int dx \left(\log\frac{1}{x}\right)^{\mu}}{\int dx \left (\log \frac{1}{x}\right)^{\nu+\mu}}=\frac{\varkappa \mu \nu}{\mu+\nu} \int \frac{x^{\varkappa \mu -1}dx}{(1-x^{\varkappa})^{1-\nu}}.
\]
Now set $\nu= \frac{m}{n}$ and $\mu=\frac{\lambda}{n}$, but then $\varkappa=n$ so that we have

\[
\frac{\int dx \left(\log \frac{1}{x}\right)^{\frac{m}{n}}\cdot \int dx \left(\log \frac{1}{x}\right)^{\frac{\lambda}{n}}}{\int dx \left(\log \frac{1}{x}\right)^{\frac{\lambda +m}{n}}}=\frac{\lambda m}{\lambda +m}\int \frac{x^{\lambda -1}dx}{\sqrt[\varkappa]{(1-x^n)^{n-m}}},
\]
which, for the sake of brevity having used notation introduced above, is more conveniently expressed as follows

\[
\frac{\left[\frac{m}{n}\right]\left[\frac{\lambda}{n}\right]}{\left[\frac{\lambda +m}{n}\right]}=\frac{\lambda m}{\lambda +m}\left(\frac{\lambda}{m}\right).
\]
Now  successively write the numbers $1$, $2$, $3$, $4 \dots n$ instead of $\lambda$ and multiply all these equations, whose number is $=n$, and the resulting equation will be

\[
\left[\frac{m}{n}\right]^n \frac{\left[\frac{1}{n}\right]\left[\frac{2}{n}\right]\left[\frac{3}{n}\right]\cdots \left[\frac{n}{n}\right]}{\left[\frac{m+1}{n}\right]\left[\frac{m+2}{n}\right]\left[\frac{m+3}{n}\right]\cdots \left[\frac{m+n}{n}\right]}
\]
\[
=m^n \frac{1}{m+1}\cdot \frac{2}{m+2}\cdot \frac{3}{m+3}\cdots \frac{n}{m+n}\left(\frac{1}{m}\right)\left(\frac{2}{m}\right)\left(\frac{3}{m}\right)\cdots \left(\frac{n}{m}\right)
\]
\[
=m^n \frac{1 \cdot 2 \cdot 3 \cdots m}{(n+1)(n+2)(n+3) \cdots (m+n)}\left(\frac{1}{m}\right)\left(\frac{2}{m}\right)\left(\frac{3}{m}\right)\cdots \left(\frac{n}{m}\right).
\]
But in  like manner, just transform the left-hand side such that 

\[
\left[\frac{m}{n}\right]^n \frac{\left[\frac{1}{n}\right]\left[\frac{2}{n}\right]\left[\frac{3}{n}\right]\cdots \left[\frac{m}{n}\right]}{\left[\frac{n+1}{n}\right]\left[\frac{n+2}{n}\right]\left[\frac{n+3}{n}\right]\cdots \left[\frac{n+m}{n}\right]},
\]
whose agreement with the preceding is revealed by cross multiplication. But because  from the nature of these formulas 
\[
\left[\frac{n+1}{n}\right]=\frac{n+1}{n}\left[\frac{1}{n}\right], \quad \left[\frac{n+2}{n}\right]=\frac{n+2}{2}\left[\frac{2}{n}\right], \quad \left[\frac{n+3}{n}\right]=\frac{n+3}{n}\left[\frac{3}{n}\right] \quad \text{etc.},
\]
and since we have $m$ of these formulas here, this left-hand side will become

\[
\left[\frac{m}{n}\right]^n \frac{n^m}{(n+1)(n+2)(n+3)\cdots (n+m)};
\]
because this one is equal to the other part exhibited before, namely

\[
m^n \frac{1 \cdot 2 \cdot 3 \cdots m}{(n+1)(n+2)(n+3)\cdots (n+m)}\left(\frac{1}{m}\right)\left(\frac{2}{m}\right)\left(\frac{3}{m}\right)\cdots \left(\frac{n}{m}\right),
\]
we obtain this equation

\[
\left[\frac{m}{n}\right]^n=\frac{m^n}{n^n}1 \cdot 2 \cdot 3 \cdots m \left(\frac{1}{m}\right)\left(\frac{2}{m}\right) \left(\frac{3}{m}\right)\cdots \left(\frac{n}{m}\right)
\]
such that

\[
\left[\frac{m}{n}\right]=m\sqrt[n]{\frac{1 \cdot 2 \cdot 3 \cdots m}{n^m}\left(\frac{1}{m}\right)\left(\frac{2}{m}\right)\left(\frac{3}{m}\right)\cdots \left(\frac{n}{m}\right)};
\]
 because  this equation $\left(\frac{n}{m}\right)=\frac{1}{m}$  agrees with the one propounded in §53, its truth is now indeed proved from most certain principles.

\section*{Proof of the theorem propounded in §59}

Also this theorem needs a more rigorous proof which I will give using the equation established before, i.e.

\[
\frac{\left[\frac{m}{n}\right]\left[\frac{\lambda}{n}\right]}{\left[\frac{\lambda+m}{n}\right]}=\frac{\lambda m}{\lambda +m}\left(\frac{\lambda}{m}\right),
\]
as follows. While $\alpha$ is a common divisor of the numbers $m$ and $n$, successively write the numbers $\alpha$, $2 \alpha$, $3 \alpha$ etc. up to $n$ instead of $\lambda$, whose total amount is $=\frac{n}{\alpha}$, and now multiply all equations resulting  in this way such that this equation emerges

\[
\left[\frac{m}{n}\right]^{\frac{n}{\alpha}}\frac{\left[\frac{\alpha}{n}\right]\left[\frac{2\alpha}{n}\right]\left[\frac{3\alpha}{n}\right]\cdots \left[\frac{n}{n}\right]}{\left[\frac{m+\alpha}{n}\right]\left[\frac{m+2\alpha}{n}\right]\left[\frac{m+3\alpha}{n}\right]\cdots \left[\frac{m+n}{n}\right]}
\]
\[
=m^{\frac{n}{\alpha}}\frac{\alpha}{m+\alpha}\cdot \frac{2\alpha}{m+2 \alpha}\cdot \frac{3\alpha}{m+3\alpha}\cdots \frac{n}{n+m}\left(\frac{\alpha}{m}\right)\left(\frac{2\alpha}{m}\right)\left(\frac{3\alpha}{m}\right)\cdots \left(\frac{n}{m}\right).
\]
Now transform the left-hand side into this one equal to it

\[
\left[\frac{m}{m}\right]^{\frac{n}{\alpha}}\frac{\left[\frac{\alpha}{n}\right]\left[\frac{2 \alpha}{n}\right]\left[\frac{3 \alpha}{n}\right]\cdots \left[\frac{m}{n}\right]}{\left[\frac{n+\alpha}{n}\right]\left[\frac{n+2\alpha}{n}\right]\left[\frac{n+3\alpha}{n}\right]\cdots\left[\frac{n+m}{n}\right]},
\]
which, because of $\left[\frac{n+\alpha}{n}\right]=\frac{n+\alpha}{n}\left[\frac{\alpha}{n}\right]$ and similarly for the remaining ones, is reduced to this one

\[
\left[\frac{m}{n}\right]^{\frac{n}{\alpha}} \frac{n}{n+\alpha}\cdot \frac{n}{n+2 \alpha} \cdot \frac{n}{n+3 \alpha}\cdots \frac{n}{n+m},
\]
In like manner, the right-hand side of the equation  is transformed into this one

\[
m^{\frac{n}{\alpha}}\frac{\alpha}{n+\alpha}\cdot \frac{2 \alpha}{n+2 \alpha} \cdot \frac{3 \alpha}{n+3 \alpha}\cdots \frac{m}{n+m}\left(\frac{\alpha}{m}\right)\left(\frac{2\alpha}{m}\right)\left(\frac{3\alpha}{m}\right)\cdots \left(\frac{n}{m}\right),
\]
whence this equation results

\[
\left[\frac{m}{n}\right]^{\frac{n}{\alpha}}n^{\frac{m}{\alpha}}=n^{\frac{n}{\alpha}} \alpha \cdot  2\alpha \cdot 3 \alpha \cdots m \left(\frac{\alpha}{m}\right)\left(\frac{2\alpha}{m}\right)\left(\frac{3\alpha}{m}\right)\cdots \left(\frac{n}{m}\right)
\]
and hence

\[
\left[\frac{m}{n}\right]=m \sqrt[n]{\frac{1}{n^m}\left(\alpha \cdot  2\alpha \cdot 3 \alpha \cdots m \left(\frac{\alpha}{m}\right)\left(\frac{2\alpha}{m}\right)\left(\frac{3\alpha}{m}\right)\cdots \left(\frac{n}{m}\right)\right)^{\alpha}}
\]
which equation compared to the preceding yields this equation

\[
\left(\alpha \cdot  2\alpha \cdot 3 \alpha \cdots m \left(\frac{\alpha}{m}\right)\left(\frac{2\alpha}{m}\right)\left(\frac{3\alpha}{m}\right)\cdots \left(\frac{n}{m}\right)\right)^{\alpha}=1 \cdot 2 \cdot 3 \cdots m \left(\frac{1}{m}\right)\left(\frac{2}{m}\right)\left(\frac{3}{m}\right)\cdots \left(\frac{n}{m}\right),
\]
which is to be understood for all common divisors of the two numbers $m$ and $n$.


\end{document}